\documentclass[12pt,reqno]{amsart}
\usepackage{etex}
\usepackage{lscape}
\usepackage{mathrsfs}
\usepackage{amsfonts}
\usepackage[sorted,compressed-cites,sorted-cites,initials]{amsrefs}
\usepackage[centertags]{amsmath}
\usepackage{amssymb,array}
\usepackage{amsthm,youngtab}
\usepackage{graphicx}
\usepackage{caption}
\usepackage{ytableau}
\usepackage{resizegather}
\usepackage{enumerate,enumitem}
\SetLabelAlign{center}{\hfill #1\hfill}
\usepackage[textwidth=16cm, hmarginratio=1:1]{geometry}
\usepackage{tikz-cd}
\usepackage{rotating}
\usepackage{mathdots}
\usepackage[all,cmtip]{xy}
\usepackage[titletoc,title]{appendix}
\usepackage{amscd}
\usepackage[colorlinks]{hyperref}
\usepackage{float}
\usepackage{appendix}
\usepackage{makecell}
\usepackage{listings} 
\hypersetup{
bookmarksnumbered,
pdfstartview={FitH},
breaklinks=true,
linkcolor=blue,
urlcolor=blue,
citecolor=blue,
bookmarksdepth=2
} 
\allowdisplaybreaks[4]

\theoremstyle{plain}
   \newtheorem{theorem}{Theorem}[section]
   \newtheorem{proposition}[theorem]{Proposition}
   \newtheorem{lemma}[theorem]{Lemma}
   \newtheorem{corollary}[theorem]{Corollary}
   \newtheorem{conjecture}[theorem]{Conjecture}
   
\theoremstyle{definition}
   \newtheorem{definition}[theorem]{Definition}
   \newtheorem{example}[theorem]{Example}

   \newtheorem{remark}[theorem]{Remark}

\numberwithin{equation}{section}

\newlist{thmlist}{enumerate}{1}

\begin{document}

\title[Real simple modules over simply-laced quantum affine algebras]{Real simple modules over simply-laced quantum affine algebras and categorifications of cluster algebras}
\author{Bing Duan}
\address[Bing Duan]{School of Mathematics and Statistics, Lanzhou University, Lanzhou, 730000 P. R. China}
\email{duanbing@lzu.edu.cn}
\author{Ralf Schiffler}
\address[Ralf Schiffler]{Department of Mathematics, University of Connecticut, Storrs, CT 06269--1009, USA}
\email{schiffler@math.uconn.edu}
\thanks{The first author was supported by the National Natural Science Foundation of China (No. 12001254, 12171213) and by Gansu Province Science Foundation for Youths (No. 22JR5RA534). The second author was supported by the NSF grant DMS-2054561.}

\date{}

\maketitle

\begin{abstract}
Let $\mathscr{C}$ be the category of finite-dimensional modules over a simply-laced quantum affine algebra $U_q(\widehat{\mathfrak{g}})$. For any height function $\xi$ and $\ell\in \mathbb{Z}_{\geq 1}$, we introduce certain subcategories $\mathscr{C}^{\leq \xi}_\ell$ of $\mathscr{C}$, and prove that the quantum Grothendieck ring $K_t(\mathscr{C}^{\leq \xi}_\ell)$ of $\mathscr{C}^{\leq \xi}_\ell$ admits a quantum cluster algebra structure. Using $F$-polynomials and monoidal categorifications of cluster algebras, we classify all real simple modules in $\mathscr{C}^{\leq \xi}_1$ in terms of their highest $l$-weight monomials, among them the families of type $D$ and type $E$ are new. For any $\ell$, inspired by Hernandez and Leclerc's work, we propose two conjectures for the study of real simple modules, and prove them for the subcategories $\mathscr{C}^{\leq \xi}_\ell$  whose Grothendieck rings are cluster algebras of finite type.

\hspace{0.15cm}


\hspace{0.15cm}

\end{abstract}

\tableofcontents

\section{Introduction}

Cluster algebras were introduced in \cite{FZ02} by Fomin and Zelevinsky as a tool for studying canonical bases and total positivity from Lie theory. Later, quantum cluster algebras, as quantum counterparts of cluster algebras of geometric type, were introduced by Berenstein and Zelevinsky \cite{BZ05}.
An additive categorification of cluster algebras, the cluster categories, was established in \cite{BMRRT06,CCS06,Amiot}. This provided a fruitful connection to the representation theory of finite-dimensional algebras. On the other hand, a monoidal categorification of cluster algebras was obtained in \cite{HL10} using representations of quantum affine algebras.

\smallskip

Let $U_q(\widehat{\mathfrak{g}})$ be the quantum affine algebra associated to a simply-laced simple complex Lie algebra $\mathfrak{g}$ with quantum parameter $q\in \mathbb{C}^*$ not a  root of unity. Denote by $\mathscr{C}$ the category of finite-dimensional $U_q(\widehat{\mathfrak{g}})$-modules. Every finite-dimensional simple $U_q(\widehat{\mathfrak{g}})$-module in $\mathscr{C}$ is a highest weight module and is parameterized by a dominant monomial \cite{CP94, CP95, FR98}. A simple module $M$ in $\mathscr{C}$ is said to be \textit{real} if $M\otimes M$ is simple \cite{Lec03}, and $M$ is said to be \textit{prime} if it cannot be written as a tensor product of non-trivial modules \cite{CP97}.

Hernandez and Leclerc \cite{HL10} discovered a connection between representations of quantum affine algebras and cluster algebras, where a monoidal categorification of a cluster algebra was given. The existence of a monoidal categorification implies the positivity of the cluster variables and the linear independence of the cluster monomials, see \cite[Proposition 2.2]{HL10}. Hernandez and Leclerc showed in \cite{HL16} that the Grothendieck rings of certain subcategories of $\mathscr{C}$ admit cluster algebra structures \cite[Theorem 5.1]{HL16}, and conjectured that the set of cluster monomials of each cluster algebra should be identified with the set of classes of real simple objects in the corresponding subcategory. One direction of the bijection (the modules corresponding to cluster monomials are real simple modules) is known \cite{Qin17,KKKO18,KKOP22}, but the other direction (the classes of real simple modules in a subcategory are cluster monomials in the corresponding cluster algebra) is still open except for Kirillov-Reshetikhin modules \cite{HL16}, snake modules of type $AB$ \cite{DLL19,DS20}, as well as Hernandez-Leclerc modules of type $A$ \cite{BC19,GDL22}.  

After Hernandez and Leclerc's work \cite{HL10,HL16}, Bittmann proved in \cite{Bit21a,Bit21} that the quantum Grothendieck rings of certain subcategories of $\mathscr{C}$ admit quantum cluster algebra structures. As a result of quantum cluster mutation, the $(q,t)$-characters of real simple modules in these subcategories have a close connection with quantum cluster monomials. 

Moreover, Hernandez and Leclerc \cite[Section 5.2.6]{HL21} pointed out that  
\begin{align*}
& ``\textit{Classification of real simple modules (in terms of their highest loop-weight) is a difficult} \\
&  \textit{open problem}."  
\end{align*}

In this paper, our main aim is to classify real simple modules in subcategories of $\mathscr{C}$.

Let $\xi: I \to \mathbb{Z}$ be a height function satisfying that $|\xi(i)-\xi(j)|=1$ if there is an edge $i\sim j$ in the Dynkin diagram of $\mathfrak{g}$. For any height function $\xi$ and $\ell\in \mathbb{Z}_{\geq 1}$, we introduce in Section \ref{subcategories} a series of  monoidal subcategories $\mathscr{C}^{\leq \xi}_\ell$ of $\mathscr{C}$ as follows:
\[
\mathscr{C}^{\leq \xi}_1 \subset \mathscr{C}^{\leq \xi}_2 \subset \cdots \subset \mathscr{C}^{\leq \xi}_\ell \subset \cdots \subset \mathscr{C}^{\leq \xi}_\infty:=\mathscr{C}^{\leq \xi}.
\]
The subcategories $\mathscr{C}^{\leq \xi}_\ell$ coincide with Hernandez-Leclerc subcategories $\mathscr{C}_\ell$ \cite{HL16,HL21} if we let $\xi$ be a sink-source function. The subcategory $\mathscr{C}^{\leq \xi}_1$ was introduced in \cite{HL13}. Very recently,  the subcategory $\mathscr{C}^{\leq \xi}$ was introduced and studied independently by Fujita, Hernandez, Oh and Oya \cite[Section 4.3]{FHOO23} for general Dynkin type, not only ADE type.  

It is natural to ask if there exists a direct relation, for example a pair of maps, between the additive and the monoidal categorification of the cluster algebra. Leclerc suggested the terminology \emph{exponential} and \emph{log} in  \cite{Lec08}, and  Nakajima proposed \emph{tropicalization} and \emph{de-tropicalization} in \cite{Nak11}. We give such a correspondence in the case where $\ell=1$, and conjecture a generalization for $\ell\geq 2$.

Using Bittmann's work \cite{Bit21}, we prove that the quantum Grothendieck ring $K_t(\mathscr{C}^{\leq \xi}_\ell)$ of $\mathscr{C}^{\leq \xi}_\ell$ admits a quantum cluster algebra structure, see Theorem \ref{quantum Grothendieck ring admits a quantum cluster algebra structure}. The crucial step is to prove that, for any choice of $\xi$, $(L^{\leq \xi}_\ell,B^{\leq \xi}_\ell)$ forms a compatible pair, see Proposition \ref{Compatible pair} and Section \ref{quantum cluster algebra structure on quantum Grothendieck ring of a subcategory} for undefined symbols.

In order to use the technique of additive categorification of cluster algebras, for  an arbitrary height function $\xi$ and $\ell\in \mathbb{Z}_{\geq 1}$,  we define a Jacobian algebra $A^{\leq \xi}_\ell$ in the sense of Derksen-Weyman-Zelevinsky \cite{DWZ08,DWZ10} in Section \ref{our Jacobian algebras}. When $\ell=1$, our Jacobian algebra $A^{\leq \xi}_1$ is nothing but a path algebra of a Dynkin quiver, with vertex set $I$. The associated cluster category $\mathcal{C}^{\leq \xi}_1$ of \cite{BMRRT06} is a triangulated 2-Calabi-Yau category with cluster-tilting objects. For each pair of non-split triangles $L\to M \to N \to L[1]$,  $N\to M' \to L \to N[1]$ in $\mathcal{C}^{\leq \xi}_1$, with $L$ and $N$ indecomposable and $\text{Ext}^1_{\mathcal{C}^{\leq \xi}_1}(L,N)=1$, it follows from \cite{CK06,FK10} that 
\begin{align}\label{cluster exchanges in a cluster category}
X_L X_{N}=X_M + \textbf{y}^{\alpha}X_{M'}
\end{align}
(we exchange $M$ and $M'$ if necessary such that $g(M)=g(L)+g(N)$, see Lemma \ref{g-vector equation in an exchange pair}), where $X_L$ is the cluster variable associated to $L$, $\textbf{y}^{\alpha}$ is a Laurent monomial in variables $y_i$, with $i\in I$, and the vector $\alpha\in \mathbb{Z}^I$ is the dimension vector of the image of the morphism $h: \tau^{-1} L \to N$ from the second exchange triangle.

We lift the formula (\ref{cluster exchanges in a cluster category}) to the following exchange relation (\ref{cluster exchanges in the subcategory}) in $K_0(\mathscr{C}^{\leq \xi}_1)$ in Theorem~\ref{main theorem2}:
\begin{align}\label{cluster exchanges in the subcategory}
[\Phi(L)] [\Phi(N)] = [\Phi(M)] \left(\prod_{i \in I} [L(f_i)]^{c_i} \right) + [\Phi(M')] \left(\prod_{j \in I} [L(f_j)]^{d_j} \right),
\end{align}
where $[L(f_i)]$, with $i\in I$, are frozen variables in the cluster algebra $K_0(\mathscr{C}^{\leq \xi}_1)$, $\Phi$ is a map from the set of indecomposable rigid objects in $\mathcal{C}^\xi_1$ to the set of real prime simple modules without $\{L(f_i) \mid i\in I \}$ in $\mathscr{C}^{\leq \xi}_1$, defined by a composition of the cluster character or the CC map and a bijection induced by the cluster algebra structure of $K_0(\mathscr{C}^{\leq \xi}_1)$, see Section~\ref{case l=1}, and the non-negative vectors $(c_i)_{i\in I}$ and $(d_j)_{j\in I}$ are defined by
\begin{itemize}
\item[(1)] $(c_i)_{i\in I} = (\max(n_i-m_i,0))_{i\in I}$, where $\text{soc}(M)= \oplus_{i\in I} S(i)^{m_i}$, $\text{soc}(N\oplus \tau_{\mathcal{C}}\,\text{Ker}(h))=\oplus_{i\in I} S(i)^{n_i}$ with $m_i, n_i \geq 0$, and $\text{soc}(P(i)[1])=0$ for $i\in I$. See Theorem \ref{equivalent descriptions of c}.  
\item[(2)] $(d_j)_{j\in I} = \underline{\text{dim}}(\text{soc}(L))+\underline{\text{dim}}(\text{soc}(N))+g(\text{Im}(h))-\underline{\text{dim}}(\text{soc}(M'))$, where $\underline{\text{dim}}(M)$ is the dimension vector of $M$, and $g(M)$ is the $g$-vector of $M$. See Theorem \ref{equivalent descriptions of d}.  
\end{itemize}
In particular,  in the case where $L=P(i)[1]$, $N=I(i)$ for an indecomposable projective module $P(i)$ and indecomposable injective module $I(i)$, our formula (\ref{cluster exchanges in the subcategory}) is an equation from a $T$-system, see Proposition \ref{images of indecomposable injective modules under Phi} and Remark \ref{one case of our Equation from T system}. If $N$ is an indecomposable non-projective $A^{\leq \xi}_1$-module and $L=\tau N$, then $M'=0$ and $\alpha=\underline{\text{dim}}(N)$. In this case, we obtain the following exchange relation (\ref{cluster exchanges for meshes in the subcategory}) in Corollary \ref{a corollary of main theorem2}.
\begin{align}\label{cluster exchanges for meshes in the subcategory}
[\Phi(\tau N)] [\Phi(N)] = [\Phi(M)] \left(\prod_{i \in I} [L(f_i)]^{c_i} \right) + \prod_{j \in I} [L(f_j)]^{d_j}.
\end{align}

Since \text{mod}\,$A^{\leq \xi}_1$ is a hereditary category, we can compute the highest $l$-weight monomials of all prime simple modules by starting with the images $\Phi(I(i))$ of the indecomposable injective modules $I(i)$ and apply Equation (\ref{cluster exchanges for meshes in the subcategory}) until we reach the images $\Phi(P(i))$ of the indecomposable projective modules $P(i)$.

Along the way, we obtain the highest $l$-weight monomials of Hernandez-Leclerc modules in Section \ref{The highest l-monomials of Hernandez-Leclerc modules}. The families of Hernandez-Leclerc modules of type $D$ and type $E$ are new, and the treatment of Hernandez-Leclerc modules of type $A$ is different from the one in \cite{BC19}.  An Hernandez-Leclerc module is a real prime simple module in $\mathscr{C}^{\leq \xi}_1$. They were studied by Hernandez and Leclerc \cite{HL10, HL13}, Brito and Chari \cite{BC19}, and Guo, Duan, and Luo \cite{GDL22}. 
The equivalence class of an Hernandez-Leclerc module is a cluster variable in $K_0(\mathscr{C}^{\leq \xi}_1)$, and the equivalence classes of tensor products of Hernandez-Leclerc modules from the same cluster are cluster monomials, and every cluster monomial is of this form. The set of all cluster monomials forms a canonical basis of $K_0(\mathscr{C}^{\leq \xi}_1)$.

For general $\ell\in \mathbb{Z}_{\geq 2}$, inspired by Hernandez and Leclerc's work \cite{HL16}, we provide a conjectural framework for the study of real simple modules in subcategories $\mathscr{C}^{\leq \xi}_\ell$, see Conjecture \ref{rigid objects equal real prime simple modules conjecture} on a bijection between real (prime) simple modules and (indecomposable) rigid objects and Conjecture \ref{character conjecture} on a realization of the $q$-character of a real simple module as the $F$-polynomial of a rigid object in the cluster category. If we take $\xi$ as a sink-source function and $\ell=\infty$, our Conjecture \ref{character conjecture} becomes a conjecture on Hernandez-Leclerc's geometric $q$-character formulas \cite[Conjecture 5.3]{HL16}. We prove the two conjectures for any height function $\xi$, in the cases where $\ell=1$ in type $ADE$, $\ell\leq 4$ in type $A_2$, and $\ell=2$ in types $A_3$ and $A_4$ in Section \ref{check our conjecture for small cases}.

This paper is organized as follows. In Section \ref{preliminaries}, we briefly collect basic material on quantum affine algebras, quantum cluster algebras, Auslander-Reiten quivers, geometric realizations of cluster categories, $F$-polynomials and $g$-vectors. In Section \ref{subcategories and cluster algebras}, we define monoidal subcategories $\mathscr{C}^{\leq \xi}_\ell$ of $\mathscr{C}$, Jacobian algebras $A^{\leq \xi}_\ell$, and a quantum cluster algebra associated to $(L^{\leq \xi}_\ell,B^{\leq \xi}_\ell)$ for any height function $\xi$. In Section \ref{quantum cluster algebras on subcategories}, we prove that the quantum Grothendieck ring $K_t(\mathscr{C}^{\leq \xi}_\ell)$ of $\mathscr{C}^{\leq \xi}_\ell$ is a quantum cluster algebra, see Theorem \ref{quantum Grothendieck ring admits a quantum cluster algebra structure}. In Section \ref{the classification of real simple modules in Cl}, we classify real simple modules in subcategories $\mathscr{C}^{\leq \xi}_1$ for any height function $\xi$, see Theorem~\ref{main theorem2}. Moreover, we provide a conjectural framework for the study of all the real simple modules. In Section \ref{check our conjecture for small cases}, we prove our conjectures for  an arbitrary height function $\xi$, $\ell=1$ in type $ADE$, $\ell\leq 4$ in type $A_2$, and $\ell=2$ in types $A_3$ and $A_4$. In Section \ref{The highest l-monomials of Hernandez-Leclerc modules}, we make use of Theorem \ref{main theorem2} and Corollary \ref{a corollary of main theorem2} to determine highest $l$-weight monomials of Hernandez-Leclerc modules. In Appendices \ref{Hernandez-Leclerc modules of type E6}, \ref{Hernandez-Leclerc modules of type E7}, and \ref{Hernandez-Leclerc modules of type E8}, we give highest $l$-weight monomials of Hernandez-Leclerc modules of types $E_6$, $E_7$, and $E_8$, respectively.

\section{Preliminaries} \label{preliminaries}

\subsection{Quantum affine algebras and their representations}

Let $C=(C_{ij})_{i,j\in I}$ be an indecomposable symmetric Cartan matrix with an index set $I$ and let $\mathfrak{g}$ be the simply-laced simple complex  Lie algebra associated to $C$. We will use the following convention:
\begin{itemize}
\item[(i)] $\gamma$ is the Dynkin diagram associated to $C$, where we use the same labeling as in \cite{Bou02}. 
\item[(ii)] $P$ is the weight lattice associated to $C$ with fundamental weights $\omega_1,\ldots,\omega_{n}$, where $n=|I|$.
\item[(iii)] $\Pi=\{\alpha_i\in P \mid i\in I\}$ is the set of simple roots associated to $C$, $\Phi$ is the corresponding root system with a decomposition $\Phi=\Phi^{+} \cup \Phi^{-}$, where $\Phi^+$ and $\Phi^-$ are the set of positive and negative roots respectively, and $\mathbb{Z}\Phi$ is the root lattice.
\item[(iv)] $W$ is the Weyl group of $C$ with the set of simple reflections $\{s_i \mid i\in I\}$, the action of $W$ on $\Phi$ is given by $s_i(\alpha_j)=\alpha_j-C_{ij}\alpha_i$.
\item[(v)] $\Phi_{\geq -1}=\Phi^{+} \cup (-\Pi)$ is the set of almost positive roots, equivalently, the union of all positive roots and negative simple roots.
\end{itemize}

Let $q\in \mathbb{C}^*$ be not a root of unity, and let $U_q(\mathfrak{g})$ be the quantum enveloping algebra associated to $\mathfrak{g}$, and $U_q(\widehat{\mathfrak{g}})$ the corresponding quantum affine algebra. The algebra $U_q(\widehat{\mathfrak{g}})$ has Drinfeld generators $x^{\pm}_{i,m}$ ($i\in I, m\in \mathbb{Z}$), $h_{i,m}$ ($i\in I, m\in \mathbb{Z}^*$), $k_i^{\pm1}$ ($i\in I$) and center elements $c^{\pm \frac{1}{2}}$ \cite{Dri87,Bec94}. We refer to \cite{FR98} for the relations among Drinfeld generators and the definition of $\phi^{\pm}_{i,\pm m} \in U_q(\widehat{\mathfrak{g}})$ ($i\in I, m\in \mathbb{Z}$). There is an embedding $U_q(\mathfrak{g}) \subset U_q(\widehat{\mathfrak{g}})$ and a Hopf algebra structure on $U_q(\widehat{\mathfrak{g}})$. Note that $h_{i,m}$ and $h_{i,-m}$ do not commute, but if the action of $c^{\pm 1/2}$ on a $U_q(\widehat{\mathfrak{g}})$-module $V$ is the identity, then the images of $h_{i,m}, k_i$ ($i\in I, m\in \mathbb{Z}^*$) in $\text{End}(V)$ form a commutative subalgebra. From now on, we assume that $V$ is of type 1, equivalently, $V$ is finite-dimensional, the action of $c^{\pm 1/2}$ on $V$ is the identity, and as a $U_q(\mathfrak{g})$-module $V$ has the following decomposition
\begin{align*} 
V=\bigoplus_{\lambda \in P} V_\lambda, \,\,\,  V_\lambda=\{ v\in V\mid k_i v = q^{(\alpha_i,\lambda)}v\}.
\end{align*}

Similar to the Cartan's highest weight classification of finite-dimensional representations of $\mathfrak{g}$ and $U_q(\mathfrak{g})$, a $U_q(\widehat{\mathfrak{g}})$-module $V$ can be refined into Jordan subspaces of $\phi^{\pm }_{i,\pm r}$ \cite{FR98}:
\begin{align*}
V=\bigoplus_{\eta} V_\eta, \,\,\,  \eta=(\eta^{\pm}_{i,\pm r})_{i\in I, r\in \mathbb{Z}_{\geq 0}}, \,\, \eta^{\pm}_{i,\pm r}\in \mathbb{C},
\end{align*}
where $V_\eta = \{ v\in V \mid \exists \, p\in \mathbb{N}, \forall\,i\in I, \forall\,m\in \mathbb{Z}, (\phi^{\pm}_{i,m}-\eta^{\pm}_{i,m})^p v=0\}$. If $\text{dim}(V_\eta)\neq 0$, $\eta$ is called the $l$-weight of $V$, and $V_\eta$ is called the $l$-weight space of $V$ with $l$-weight $\eta$. Following \cite{CP94,CP95,FR98}, the $l$-weight of every $U_q(\widehat{\mathfrak{g}})$-module is of the following form
\begin{align*}
\eta^{\pm}_i(u):=\sum_{r=0}^\infty \eta^{\pm}_{i,\pm r} u^{\pm r} = q^{\text{deg(Q)}_i-\text{deg(R)}_i} \cfrac{Q_i(uq^{-1}) R_i(uq)}{Q_i(uq) R_i(uq^{-1})},
\end{align*}
where $Q_i$ and $R_i$ are polynomials of the following form
\begin{align*}
Q_i(u) = \prod_{a\in \mathbb{C}^*} (1-ua)^{w_{i,a}}, \,\,\, R_i(u) = \prod_{a\in \mathbb{C}^*} (1-ua)^{x_{i,a}}, 
\end{align*}
where $w_{i,a}, x_{i,a} \geq 0$, with $i\in I$. 

Let $\mathcal{P}$ be a free abelian multiplicative group generated by infinitely many formal variables $(Y_{i,a})_{i\in I, a\in \mathbb{C}^*}$. There is a bijection between $\mathcal{P}$ and the set of $l$-weights, that is, 
\begin{align*}
m=\prod_{i\in I, a\in \mathbb{C}^*} Y^{w_{i,a}-x_{i,a}}_{i,a} \mapsto \eta.
\end{align*}
We identify the elements of $\mathcal{P}$ with  $l$-weights of $V$. A monomial $m$ is said to be \textit{dominant} if it does not contain variables $Y_{i,a}$ with negative exponents. Given a dominant monomial $m$, we have the corresponding simple module $L(m)$ \cite{CP94, FR98}. 

Let $V$ be a simple $U_q(\widehat{\mathfrak{g}})$-module. A non-trivial module $V$ is said to be \textit{prime} if it has no non-trivial tensor factorization \cite{CP97}, and it is said to be \textit{real} if $V \otimes V$ is simple \cite{Lec03}. For a real simple module $V$, $V^{\otimes k}$ is simple \cite{Her10,KKKO15} for any $k\in \mathbb{Z}_{>0}$.  

\subsection{The $q$-characters}

Let $\mathcal{Y}=\mathbb{Z}[Y^{\pm 1}_{i,a}\mid i\in I, a\in \mathbb{C}^*]$ be the Laurent polynomial ring with integer coefficients. We denote by $\mathscr{C}$ the category of finite-dimensional $U_q(\widehat{\mathfrak{g}})$-modules, by $K_0(\mathscr{C})$ the Grothendieck ring of $\mathscr{C}$, and by $[V] \in K_0(\mathscr{C})$ the equivalence class of $V\in \mathscr{C}$. It follows from \cite{FR98} that $K_0(\mathscr{C})$ is commutative and is generated by $[L(Y_{i,a})]$, with $i\in I, a\in \mathbb{C}^*$. The \emph{$q$-character} $\chi_q(?): K_0(\mathscr{C}) \to \mathcal{Y}$ is an injective ring homomorphism defined by 
\begin{align*}
\chi_q([V]) =\sum_{m\in \mathcal{P}} \text{dim}(V_m)m \,\, \text{ for $[V]\in K_0(\mathscr{C})$}.
\end{align*}

Following \cite{FR98}, for $i\in I, a\in \mathbb{C}^*$, one defines $A_{i,a} \in \mathcal{P}$ (called a quantum affine analog of a root, because we have $\text{wt}(A_{i,a})=\alpha_i$) by
\begin{align*} 
A_{i,a} = Y_{i,  aq^{-1}} Y_{i,  aq} \left(\prod_{j:C_{ij}=-1}Y_{j,a}^{-1}\right). 
\end{align*}  
Assume that $\mathcal{Q}$ is the subgroup of $\mathcal{P}$ generated by $A^{\pm 1}_{i,a} \, (i\in I,  a\in \mathbb{C}^*)$, $\mathcal{Q}^{+}$ (respectively, $\mathcal{Q}^{-}$) is the submonoid of $\mathcal{Q}$ generated by $A_{i,a}$ (respectively, $A^{-1}_{i,a}$), with $i\in I, a\in \mathbb{C}^*$. There exists a partial order $\leq$ on $\mathcal{P}$:
\begin{align}\label{Nakajima partial order}
m\leq m' \text{ if and only if } m'm^{-1}\in \mathcal{Q}^{+}.
\end{align}
The partial order is compatible with a partial order on the weight lattice. In other words, $m\leq m'$ implies that $\text{wt}(m) \leq \text{wt}(m')$. It was shown by Frenkel and Mukhin in \cite{FM01} that, for each dominant monomial $m$, 
\begin{align} \label{FM q-character foumula} 
\chi_q(L(m))=m(1+\sum_{p}M_p), 
\end{align}
where each $M_p$ is a monomial in variables $A^{-1}_{i,a}$, with $i\in I,a\in \mathbb{C}^*$. Clearly, the monomial $m$ is the highest $l$-weight monomial in $\chi_q(L(m))$ under (\ref{Nakajima partial order}). 

We shall frequently use the following theorem.

\begin{theorem}[{\cite[Proposition 5.3]{HL10}}] \label{dominant monomials determine q-characters}
Suppose that $V$ and $W$ are two $U_q(\widehat{\mathfrak{g}})$-modules. If $\chi_q([V)]$ and $\chi_q([W])$ have the same dominant monomials and multiplicities, then $\chi_q([V]) = \chi_q([W])$.
\end{theorem}

For a dominant monomial $m$, Frenkel and Mukhin \cite{FM01} developed an algorithm, now called Frenkel-Mukhin algorithm, which outputs a Laurent polynomial $\text{FM}(m)$. A simple $U_q(\widehat{\mathfrak{g}})$-module $L(m)$ is called \textit{minuscule} or \textit{special} if $m$ is the unique dominant monomial in $\chi_q(L(m))$. If $L(m)$ is minuscule, then $\chi_q(L(m))=\text{FM}(m)$.  A sufficient condition for having $FM(m) = \chi_q(L(m))$ is that $FM(m)$ contains all the dominant monomials of $\chi_q(L(m))$. Based on the Frenkel-Mukhin algorithm, a combinatorial method to compute $q$-characters of snake modules was introduced by Mukhin and Young \cite{MY12a,MY12b}. In some cases, but not always, the Frenkel-Mukhin algorithm gives the correct $q$-characters \cite{NN11}. 

The $q$-characters of standard modules of a simply-laced quantum affine algebra were calculated by Nakajima \cite{Nak01} in terms of homology of graded quiver varieties. Nakajima's algorithm was in fact built at the quantum level. Similar to the Kazhdan-Lusztig algorithm, $q$-characters of simple modules were calculated in terms of standard modules \cite{Nak04}. 

Let $m_1, m_2$ be two dominant monomials. Then $L(m_1m_2)$ is a simple quotient of $L(m_1) \otimes L(m_2)$. An interesting question is when the equation $L(m_1m_2)=L(m_1) \otimes L(m_2)$ holds. A consequence of Hernandez-Leclerc's cluster algebra approach is that if $[L(m_1)]$ and $[L(m_2)]$ are cluster variables in the same cluster, then $L(m_1m_2)=L(m_1)\otimes L(m_2)$. 

We frequently use two operators, \textit{renormalization} and \textit{truncation}, on $q$-characters. The former means that we divide the $q$-character of $L(m)$ by $m$, so that its leading term becomes $1$, denoted by $\widetilde{\chi}_{q}([L(m)])$, the latter means that we delete some monomials in $q$-characters which are beyond the scope of restrictions. Both renormalization and truncation are injective ring homomorphisms from $K_0(\mathscr{C})$ to $\mathcal{Y}$. 

\subsection{Quantum cluster algebras}
The notion of quantum cluster algebras was introduced by Berenstein and Zelevinsky \cite{BZ05} as a quantum counterpart of a Fomin-Zelevinsky's cluster algebra \cite{FZ02}.

Suppose that $m, n$ are two positive integers with $m\geq n$. Let $\widetilde{B}=(b_{ij})$ be an integer-valued $m\times n$ matrix such that its upper $n\times n$ submatrix $B=(b_{ij})$ is skew-symmetric, and $\Lambda=(\lambda_{ij})$ an $m\times m$ skew-symmetric matrix. A pair $(\widetilde{B},\Lambda)$ is called a \textit{compatible pair} if 
\[
\widetilde{B}^T\Lambda=(D \mid 0),
\] 
where $(D \mid 0)$ denotes an $n\times m$ block matrix with a diagonal submatrix $D=\text{diag}(d_1,\ldots,d_n)$ for integer coefficients $d_i$, $1 \leq i \leq n$, with constant signs.

To such a matrix $\widetilde{B}$, we have an associated quiver $Q_{\widetilde{B}}$ with vertex set $\{1,2,\ldots,m\}$ and with $b_{ij}$ arrows from $j$ to $i$ if $b_{ij}>0$, $|b_{ij}|$ arrows from $i$ to $j$ if $b_{ij}<0$, the vertices $n+1,n+2,\ldots,m$ are called frozen vertices. Since $B$ is skew-symmetric, it follows that $Q_{\widetilde{B}}$ does not have any loop, any 2-cycle, nor arrow between frozen vertices. Conversely, for any quiver without loops or 2-cycles, one can associate a matrix $\widetilde{B}$ by letting
\[
b_{ij}=(\text{the number of arrows from $j$ to $i$})-(\text{the number of arrows from $i$ to $j$}).
\]

The quantum torus $\mathcal{T}(\Lambda)$ associated to $\Lambda$  is an associative $\mathbb{Z}[q^{\pm 1/2}]$-algebra with generators $X^{\pm 1}_1, X^{\pm 1}_2, \ldots, X^{\pm 1}_m$ subject to the following relations
\[
X_iX^{\pm 1}_i=X^{\pm 1}_iX_i=1, \quad  X_i X_j=q^{\lambda_{ij}} X_j X_i, \,\,  1\leq i,j \leq m.
\]
Let $\mathcal{F}$ be the fractional skew-field of $\mathcal{T}(\Lambda)$.  The toric frame $M: \mathbb{Z}^m \to \mathcal{F}\setminus\{0\}$ is a map defined by 
\[
M(\sum_{i=1}^m c_ie_i)=q^{1/2\sum_{i<j}c_ic_j\lambda_{ji}} X^{c_1}_1 X^{c_2}_2 \cdots X^{c_m}_m,
\]
where $\{e_1,e_2,\ldots,e_m\}$ is an orthogonal basis of $\mathbb{Z}^m$. The set $\{ M(c) \mid c \in \mathbb{Z}^m \}$ is a $\mathbb{Z}[q^{\pm 1/2}]$-basis of $\mathcal{T}(\Lambda)$. For any $c,c'\in \mathbb{Z}^m$, 
\begin{align*}
M(c) M(c') &=  q^{1/2 \sum_{i>j} (c_ic'_j - c'_ic_j)\lambda_{ij}} M(c+c') \\
& = q^{\sum_{i>j} (c_ic'_j - c'_ic_j)\lambda_{ij}} M(c') M(c).
\end{align*}

An initial quantum seed in $\mathcal{F}$ is a triple $(\widetilde{\textbf{z}}, \widetilde{B}, \Lambda)$, where $\widetilde{\textbf{z}}=\{X_1,\ldots,X_m\}$ is called a \textit{quantum extended cluster},  $\textbf{z}=\{X_1,\ldots,X_n\} \subset \widetilde{\textbf{z}}$ is called a \textit{quantum cluster}, every variable in $\textbf{z}$ is called a \textit{quantum cluster variable}, and $(\widetilde{B}, \Lambda)$ is a compatible pair. 

For $k\in \{1,\ldots,n\}$, one defines a mutation $\mu_k$ of  $(\widetilde{\textbf{z}},\widetilde{B}, \Lambda)$ at $k$, that produces a new quantum seed $\mu_k(\widetilde{\textbf{z}},\widetilde{B}, \Lambda)=(\widetilde{\textbf{z}}', \widetilde{B}', \Lambda')$, where 
\begin{align*}
& X'_j = \begin{cases}
M(-e_k+\sum_{i: b_{ik}>0} b_{ik}e_i) + M(-e_k-\sum_{i: b_{ik}<0} b_{ik}e_i) & \text{if $j=k$}, \\
X_j  &  \text{otherwise},
\end{cases} \\
& \widetilde{B}'=E_\varepsilon \widetilde{B} F_\varepsilon, \\
& \Lambda'= E^T_\varepsilon \Lambda E_\varepsilon,
\end{align*}
where $\varepsilon\in \{\pm \}$, $E_\varepsilon=(E_{ij})$ and $F_\varepsilon=(F_{ij})$ are defined as follows:
\begin{align*}
E_{ij}=\begin{cases}
\delta_{ij}  & \text{if $j \neq k$}, \\
-1  & \text{if $i=j=k$}, \\
\max\{ \varepsilon b_{ik}, 0\} & \text{if $i \neq j=k$}.
\end{cases} \quad 
F_{ij}=\begin{cases}
\delta_{ij}  & \text{if $i \neq k$}, \\
-1  & \text{if $i=j=k$}, \\
\max\{-\varepsilon b_{kj}, 0\} & \text{if $i=k\neq j$}.
\end{cases}
\end{align*}

It was shown in \cite{BZ05} that the mutation of quantum seeds is an involution, the new quantum seed $(\widetilde{\textbf{z}}', \widetilde{B}', \Lambda')$ is still a quantum seed, and does not depend on the choice of $\varepsilon$.

A quantum cluster algebra $\mathcal{A}_{q^{1/2}}(\widetilde{\textbf{z}}, \widetilde{B}, \Lambda)$ with an initial quantum seed $(\widetilde{\textbf{z}}, \widetilde{B}, \Lambda)$ is a $\mathbb{Z}[q^{\pm 1/2}][X^{\pm}_{n+1},\ldots,X^{\pm}_m]$-subalgebra of $\mathcal{F}$ generated by all quantum cluster variables obtained from $(\widetilde{\textbf{z}}, \widetilde{B}, \Lambda)$ by mutations. 

The specialization of a quantum cluster algebra at $q=1$ is a cluster algebra of geometric type. Moreover, the quantum cluster algebras share some common properties with cluster algebras, for instance, Laurent phenomenon \cite{FZ02,BZ05}, finite-type classification \cite{FZ03,BZ05}, and positivity \cite{LS15,GHKK18,Dav18}. 

\subsection{Auslander-Reiten quivers} \label{Auslander-Reiten quivers} 

This section prepares some material on Auslander-Reiten quivers. We follow here the exposition in \cite{Bed99} for type $ADE$, other sources are the books \cite{ASimS06,Sch14}.

Assume that $Q$ is an orientation of the Dynkin diagram $\gamma$. Then $Q$ is called a Dynkin quiver.  Let $\sigma$ be the unique permutation of vertices of $\gamma$ such that $w_0(\alpha_i)=-\alpha_{\sigma(i)}$, where $w_0$ is the longest word in the Weyl group. The Auslander-Reiten quiver $\Gamma_Q$ of $Q$ is a model of the category $\text{mod}\,kQ$ of finite-dimensional $kQ$-modules, where $k$ is an algebraically closed field. Its vertices are the isoclasses of indecomposable modules of $Q$ and there is an arrow $[M]\to [N]$ between $[M]$ and $[N]$ if and only if there exists an irreducible morphism from $M$ to $N$. There is an important functor $\tau: \text{mod}\,kQ \to \text{mod}\,kQ$, called the Auslander-Reiten translation. 

Denote by $I(i),P(i),S(i)$ the indecomposable injective, projective, simple $kQ$-module respectively with socle $S(i)$ and top $S(i)$, respectively. Gabriel's theorem says that the dimension vector $\underline{\textit{dim}}$ gives a bijection from the set of isoclasses of indecomposable modules to the set of positive roots of $\gamma$. From now on, we identify an indecomposable module $M$ with its positive root $\underline{\textit{dim}}(M)$.  

Recall that $Q$ has vertex set $I$. For each $i\in I$, we denote by $s_i(Q)$ the quiver obtained from $Q$ by reversing all arrows at $i$. A sequence of reflections $s_{i_1} s_{i_2} \cdots s_{i_n}$ is adapted to $Q$ if $i_1$ is a source of $Q$, and $i_k$ is a source of $s_{i_{k-1}} \cdots s_{i_2} s_{i_1}(Q)$ for each $2\leq k \leq n$. 

Let $\nu$ be the number of positive roots, and $w_0=s_{i_1} s_{i_2} \cdots s_{i_{\nu}}$ a reduced expression of the longest element $w_0$ in the Weyl group $W$ of $Q$, which is adapted to $Q$. The sequence of roots
\begin{align}\label{order positive roots}
\beta_k=s_{i_1} s_{i_2} \cdots s_{i_{k-1}}(\alpha_{i_{k}}) \,\, \text{ for $1\leq k \leq \nu$}
\end{align}
contains every positive root exactly once. It induces a total order on positive roots by $\alpha<\beta$ if $\alpha$ appears in the sequence (\ref{order positive roots}) before $\beta$. 

The Auslander-Reiten quiver $\Gamma_Q$ can be constructed as follows: the vertices of $\Gamma_Q$ are indexed by (\ref{order positive roots}), and there is an arrow from $\beta_k$ to $\beta_j$ if $j^+>k>j>k^-$, and an edge $i_k \sim i_j$ in $\gamma$, where we denote by $k^+$ (respectively, $k^-$) the smallest (respectively, largest) index $\ell$ such that $k<\ell$ (respectively, $\ell<k$) and $i_\ell=i_k$. We have the convention $k^+=v+1$ (respectively, $k^-=0$)  if $i_\ell \neq i_k$ for any $\ell>k$ (respectively, $\ell<k$).

Assume without loss of generality that $(i_1, i_2, \cdots, i_n)$ is the first $n$ different indices of $w_0$, where $n=|I|$. Then $c=s_{i_1} s_{i_2} \cdots s_{i_n}\in W$ is a Coxeter element adapted to $Q$. For each $j\in I$,  let 
\[
\Omega_j=\left\{ c^{\ell} \theta_j  \mid 0\leq \ell \leq  \frac{h+a_j-b_j}{2}-1 \right\},
\] 
where $\theta_j=s_{i_1} s_{i_2} \cdots  s_{i_{j-1}} (\alpha_{i_j})$, $h$ is the Coxeter number, $a_j$ (respectively, $b_j$) is the number of arrows in the unoriented path from $j$ to $\sigma(j)$ which are directed towards $\sigma(j)$ (respectively, $j$). Each $\Omega_j$ is a $\tau$-orbit in $\Gamma_Q$, and $\Omega_j \cap \Omega_{k}=\emptyset$ for $j\neq k$. If $M$ is a non-projective indecomposable $kQ$-module, then the dimension vector $\underline{\textit{dim}}(\tau M)$ is equal to $c\,\underline{\textit{dim}}(M)$.

\subsection{A geometric realization of the cluster category} \label{a geometric realization of cluster category} 
In this section, we refer to \cite{BMRRT06,BMR07,CCS06,FZ03,FST08,Kel05,Sch08,Sch14}.

Following \cite{BMRRT06}, let $\mathcal{C}_Q$ be the cluster category associated to $Q$ defined as a quotient category of the bounded derived category $\mathcal{D}^b(kQ)$ of $\text{mod}\,kQ$ by $\tau^{-1}[1]$, where $[1]: \mathcal{D}^b(kQ) \to \mathcal{D}^b(kQ)$ is the shift functor, and $\tau: \mathcal{D}^b(kQ) \to \mathcal{D}^b(kQ)$ is the AR translation on the derived category. It acts on indecomposable objects by the rule $\tau(M[i])=(\tau M)[i]$ for a non-projective module $M\in \text{mod}\,kQ$, and $\tau(P(x)[i])=I(x)[i-1]$. It was shown in \cite{Kel05,BMRRT06} that $\mathcal{C}_Q$ is a triangulated 2-Calabi-Yau category, that is, we have bifunctorial isomorphisms
\[
\text{Ext}^1_{\mathcal{C}_Q}(N,L) \cong D\,\text{Ext}^1_{\mathcal{C}_Q}(L,N)
\] 
for any objects $M,N\in \mathcal{C}_Q$, where $D=\text{Hom}_k(-,k)$ is the $k$-duality.

An object $T$ in a triangulated 2-Calabi–Yau category $\mathcal{C}$ is \textit{rigid} if $\text{Ext}^1_{\mathcal{C}}(T,T)=0$. We denote by $\text{add}\,T$ the subcategory of $\mathcal{C}$ whose objects consist of the direct sums of direct summands of $T$. A rigid object $T\in \mathcal{C}$ is called a \textit{cluster-tilting object} if $\text{Ext}^1_{\mathcal{C}}(T,X)=0$ for some $X\in \mathcal{C}$ implies that $X\in \text{add}\,T$. Let $T$ be a cluster-tilting object corresponding to a choice of initial cluster variables. The functor $\text{Hom}_{\mathcal{C}_Q}(T,-): \mathcal{C}_Q \to \text{mod}\,kQ$ induces an equivalence of categories $\mathcal{C}_Q/\tau\,T \cong \text{mod}\,kQ$, see \cite{BMR07}.

Let $P_{n+3}$ be a regular polygon with $n+3$ vertices and boundary edges. A diagonal in $P_{n+3}$ is a line segment that joins two nonadjacent vertices and goes through the interior of the polygon, and a triangulation of $P_{n+3}$ is a maximal collection of non-crossing diagonals, labeled by $1,2,\ldots,n$. For a triangulation $T$ of $P_{n+3}$ with the property that every triangle in it has at least one edge on the boundary, we associate a quiver of type $\gamma=A_n$ as in \cite{Sch14}. To a diagonal $d$ not in $T$, we associate a representation $M_d=(M_i,\psi_\alpha)$ of $Q$ by setting
\begin{align*}
M_i=\begin{cases}
k & \text{if $d$ crosses the diagonal $i$}, \\
0 & \text{otherwise},
\end{cases}
\end{align*}
and letting $\psi_\alpha=1$ whenever the vector spaces corresponding to source and target of $\alpha$ are $k$, and $\psi_\alpha=0$ otherwise. 

For type $D_n$, instead of a regular polygon, we work with a triangulated punctured polygon $P^{\bullet}_{n}$ with $n$ boundary vertices and one interior punctured vertex. For each boundary vertex $a$, there are two arcs (plain or tagged to distinguish them) from $a$ to the puncture, denoted by $a$ and $a^{\bowtie}$. Next we define the crossing number between two arcs. If one of the two arcs has both endpoints on the boundary, the number of crossing should be intuitively clear. The number of crossings of $a$ and $a^{\bowtie}$, or $a$ and $b$ ($a\neq b$), or $a^{\bowtie}$ and $b^{\bowtie}$ ($a\neq b$) is 0, but the number of crossings of $a$ and $b^{\bowtie}$ ($a\neq b$) is 1.  
We say that two arcs cross if their crossing number is at least $1$, and a triangulation in $P^{\bullet}_{n}$ is a maximal set of non-crossing arcs. Similar to the case of type $A$, we associate a triangulation $T_Q$ of $P^{\bullet}_{n}$ to a quiver $Q$ (in the mutation class of a type $D_n$ quiver), but have to be more careful in a self-folded triangle, see Figure \ref{triangulation from a quiver of type Dn}. To an arc $d$ not in $T_Q$, we associate an indecomposable representation $M_d=(M_i,\psi_\alpha)$ of $Q$, where the dimension of $M_i$ is equal to the number of crossings between $d$ and the arc $i$. 

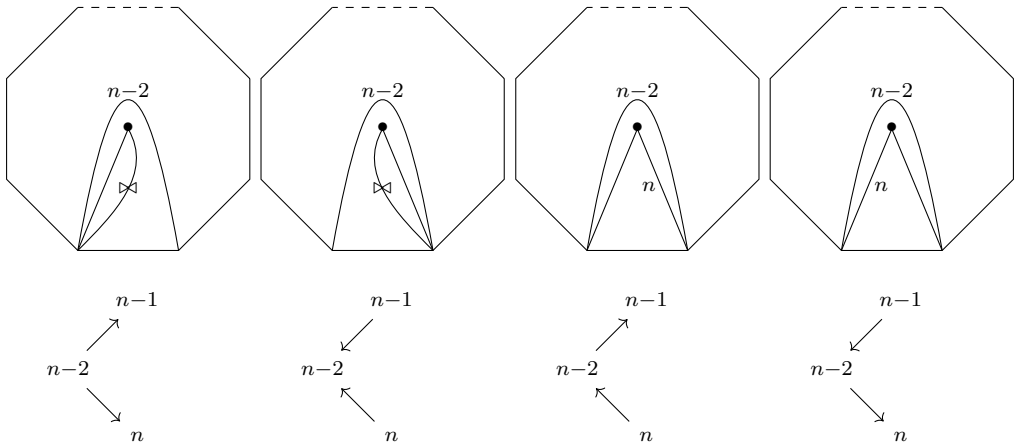
\begin{figure}[H]
\begin{tikzpicture}[node distance={13mm}] 
\draw (22.5:1.75cm) -- (67.5:1.75cm)  (112.5:1.75cm) -- (157.5:1.75cm) -- (-157.5:1.75cm) -- (-112.5:1.75cm) -- (-67.5:1.75cm) -- (-22.5:1.75cm) -- (22.5:1.75cm);
\draw[dashed] (67.5:1.75cm) -- (112.5:1.75cm);
\node[thick] (A) at (0,0) {$\substack{\bullet}$};
\draw (-112.5:1.75cm) to [out=80,in=100,looseness=5.2] (-67.5:1.75cm);
\draw (-112.5:1.75cm) -- (0,0);
\draw (-112.5:1.75cm) to [out=45,in=-60,looseness=1] (0,0);
\node at (-90:0.8cm) {$\substack{\bowtie}$};
\node at (0,0.5) {$\substack{n-2}$};
\node (1) at (-0.8,-3.2) {$\substack{n-2}$}; 
\node (2) [above right of=1] {$\substack{n-1}$}; 
\node (3) [below right of=1] {$\substack{n}$}; 
\draw[->] (1) -- (2); 
\draw[->] (1) -- (3); 
\end{tikzpicture} 
\begin{tikzpicture}[node distance={13mm}] 
\draw (22.5:1.75cm) -- (67.5:1.75cm)  (112.5:1.75cm) -- (157.5:1.75cm) -- (-157.5:1.75cm) -- (-112.5:1.75cm) -- (-67.5:1.75cm) -- (-22.5:1.75cm) -- (22.5:1.75cm);
\draw[dashed] (67.5:1.75cm) -- (112.5:1.75cm);
\node[thick] (A) at (0,0) {$\substack{\bullet}$};
\draw (-112.5:1.75cm) to [out=80,in=100,looseness=5.2] (-67.5:1.75cm);
\draw (-67.5:1.75cm) -- (0,0);
\draw (-67.5:1.75cm) to [out=135,in=-120,looseness=1] (0,0);
\node at (-90:0.8cm) {$\substack{\bowtie}$};
\node at (0,0.5) {$\substack{n-2}$};
\node (1) at (-0.8,-3.2) {$\substack{n-2}$}; 
\node (2) [above right of=1] {$\substack{n-1}$}; 
\node (3) [below right of=1] {$\substack{n}$}; 
\draw[->] (2) -- (1); 
\draw[->] (3) -- (1); 
\end{tikzpicture}
\begin{tikzpicture}[node distance={13mm}] 
\draw (22.5:1.75cm) -- (67.5:1.75cm)  (112.5:1.75cm) -- (157.5:1.75cm) -- (-157.5:1.75cm) -- (-112.5:1.75cm) -- (-67.5:1.75cm) -- (-22.5:1.75cm) -- (22.5:1.75cm);
\draw[dashed] (67.5:1.75cm) -- (112.5:1.75cm);
\node[thick] (A) at (0,0) {$\substack{\bullet}$};
\draw (-67.5:1.75cm) -- (0,0)  (-112.5:1.75cm) -- (0,0);
\draw (-112.5:1.75cm) to [out=80,in=100,looseness=5.2] (-67.5:1.75cm);
\node at (-79:0.8cm) {$\substack{n}$};
\node at (0,0.5) {$\substack{n-2}$};
\node (1) at (-0.8,-3.2) {$\substack{n-2}$}; 
\node (2) [above right of=1] {$\substack{n-1}$}; 
\node (3) [below right of=1] {$\substack{n}$}; 
\draw[->] (1) -- (2); 
\draw[->] (3) -- (1); 
\end{tikzpicture}
\begin{tikzpicture}[node distance={13mm}] 
\draw (22.5:1.75cm) -- (67.5:1.75cm)  (112.5:1.75cm) -- (157.5:1.75cm) -- (-157.5:1.75cm) -- (-112.5:1.75cm) -- (-67.5:1.75cm) -- (-22.5:1.75cm) -- (22.5:1.75cm);
\draw[dashed] (67.5:1.75cm) -- (112.5:1.75cm);
\node[thick] (A) at (0,0) {$\substack{\bullet}$};
\draw (-67.5:1.75cm) -- (0,0)  (-112.5:1.75cm) -- (0,0);
\draw (-112.5:1.75cm) to [out=80,in=100,looseness=5.2] (-67.5:1.75cm);
\node at (-100:0.8cm) {$\substack{n}$};
\node at (0,0.5) {$\substack{n-2}$};
\node (1) at (-0.8,-3.2) {$\substack{n-2}$}; 
\node (2) [above right of=1] {$\substack{n-1}$}; 
\node (3) [below right of=1] {$\substack{n}$}; 
\draw[->] (2) -- (1); 
\draw[->] (1) -- (3); 
\end{tikzpicture}
\caption{Arrows associated to a self-folded triangle, from \cite[Figure 3.12]{Sch14}} \label{triangulation from a quiver of type Dn}
\end{figure}

The map $d \to M_d$ gives a bijection from the set of non-initial diagonals/arcs and the set of isoclasses of indecomposable representations of $Q$. We extend the map $d \to M_d$ to include the initial diagonals/arcs in $T_Q$ by defining the diagonal/arc $i$ in $T_Q$ as $M_{-\alpha_i}$, which induces a bijection from the set of diagonals/arcs and the set of indecomposable objects in $\mathcal{C}_Q$. 
 
The Auslander–Reiten translation $\tau$ (respectively, $\tau^{-1}$) is explained as an elementary clockwise (anti-clockwise) rotation of $P_{n+3}$ or of $P^{\bullet}_{n}$ with simultaneous change of the tags at the puncture. The projective representation $P(i)$ is given by $\tau^{-1}$ of the diagonal $i$, and the injective representation $I(i)$ is given by $\tau$ of the diagonal $i$. The cluster category $\mathcal{C}_Q$ is equivalent to  
\begin{itemize}
\item the category of all diagonals in $P_{n+3}$ in type $A_n$, see \cite{CCS06},
\item the category of all (tagged) arcs in $P^{\bullet}_{n}$ in type $D_n$, see \cite{Sch08}.
\end{itemize}

The following proposition is very useful.

\begin{proposition}[{\cite[Corollary 4.4]{FZ03},\cite[Theorem 5.2]{CCS06},\cite[Theorem 7.5]{BMRRT06},\cite[Theorem 4.3]{Sch08}}] \label{characteristic of exchange pair}
Suppose that $M_d, \, M_{d'}$ are two indecomposable objects in $\mathcal{C}_Q$. The following statements hold. 
\begin{itemize}
\item[(1)] $dim(Ext^1_{\mathcal{C}_Q}(M_d, M_{d'}))$ is equal to the number of crossings between the diagonals or arcs $d$ and $d'$. 
\item[(2)] $M_d, \, M_{d'}$ are the summands of the same cluster-tilting object if and only if 
\[
dim(Ext^1_{\mathcal{C}_Q}(M_d, M_{d'}))= 0.
\]
\item[(3)] $M_d, \, M_{d'}$ form an exchange pair if and only if $dim(Ext^1_{\mathcal{C}_Q}(M_d, M_{d'}))=1$.
\end{itemize} 
\end{proposition} 

\subsection{Cluster characters}\label{cluster characters}

Let $(Q,W)$ be a quiver with potential in the sense of \cite{DWZ08}, and assume further that the vertex set of the quiver is equal to $I$. Let $\mathcal{I}$ be the  Jacobian ideal of the potential and  $A=kQ/\mathcal{I}$ the Jacobian algebra. Then $\text{mod}\,A$ is equivalent to the category $\text{rep}(Q,\mathcal{I})$ of finite-dimensional representations of the bound quiver $(Q,\mathcal{I})$, denoted by $\text{mod}\,A\cong \text{rep}(Q,\mathcal{I})$.

We denote by $\mathcal{C}$ the cluster category of $A$ as defined in \cite{Amiot}. Let $T=T_1\oplus \cdots \oplus T_{n}$ be a cluster-tilting object in $\mathcal{C},$ and $B=\text{End}_{\mathcal{C}}(T)$. Following \cite{DWZ10}, the $F$-polynomial $F^T_M$ of $M\in \mathcal{C}$ with respect to $T$ is defined as 
\begin{align} \label{F-polynomials formula}
F^T_M(y_1,\ldots,y_{n})=
\begin{cases}
1 & \text{if $M\in \text{add}\,T[1]$}, \\
\sum_{\underline{e} \in \mathbb{N}^{I}} \chi(\text{Gr}_{\underline{e}}(M)) \prod_{i\in I} y^{e_i}_i \in \mathbb{Z}[y_1,\ldots,y_{n}], & \text{if $M\in \text{mod}\,B$},
\end{cases}
\end{align}
where $\text{Gr}_{e}(M)$ is the Grassmannian of submodules of $M$ with dimension $\underline{e}$ and $\chi$ is the Euler-Poincar\'e characteristic. For two modules $M$ and $N$ in $\text{mod}\,B$, it was shown in \cite{CC06,DWZ10} that $F^T_{M\oplus N}=F^T_M F^T_N$. 

For any $M \in \mathcal{C}$, there exists a triangle  
\[
\xymatrix{
M \ar[r]^f& T^{0}_M[2] \ar[r]& T^{1}_M[2] \ar[r]& M[1]},
\]
where $T^{0}_M=\oplus_{i=1}^n T^{a_i}_i, T^{1}_M=\oplus_{i=1}^n T^{b_i}_i \in \text{add}\,T$, and $f$ is a minimal $\text{add}\, T$-approximation. Define the \textit{$g$-vector} of $M$ with respect to $T$ by 
\[
g^T(M) = (b_1-a_1,b_2-a_2,\ldots,b_n-a_n)\in \mathbb{Z}^n.
\]

In \cite{Pal08}, Palu proved the following proposition, which is very useful in the sequel.

\begin{proposition}[{\cite[Proposition 2.2]{Pal08}}] \label{Palu's g-vector formula}
Let $L \overset{f}{\to}  M \overset{g}{\to} N \overset{\varepsilon}{\to} L[1]$ be a triangle in $\mathcal{C}$. Take $K\in \mathcal{C}$ to be any lift of $\textup{Ker}(\textup{Hom}_{\mathcal{C}}(T,f))$. Then
\begin{align*}
g^T(M) = g^T(L) + g^T(N) - g^T(K) - g^T(K[1]).
\end{align*}
\end{proposition}

We have the following corollary.

\begin{corollary}\label{g-vector equation associated an almost split sequence}
Suppose that $L \overset{f}{\to}   M  \to N  \to L[1]$ is a non-split triangle in $\mathcal{C}$.
\begin{itemize}
\item[(1)] If the triangle is induced by a non-split short exact sequence $0 \to L \to M \to N \to 0$ in $\textup{mod}\,B$, then $g^T(M) = g^T(L) + g^T(N)$.
\item[(2)] If $L\in \textup{add}\,T[1]$, then $g^T(M) = g^T(L) + g^T(N)$.
\end{itemize}
\end{corollary}
\begin{proof}
(1) Since $\text{Hom}_{\mathcal{C}}(T,L)\cong L$ for any $L\in \textup{mod}\,B$, $\textup{Ker}(\textup{Hom}_{\mathcal{C}}(T,f))=0$. Our result follows from Proposition \ref{Palu's g-vector formula}.

(2) It follows from Proposition \ref{Palu's g-vector formula} and $\text{Hom}_{\mathcal{C}}(T,L)=0$.
\end{proof}

The \emph{cluster character} or the CC map of $T$ is a map 
\[
X^{T}_{?}: \mathcal{C} \to \mathbb{Z}[x^{\pm 1}_1,x^{\pm 1}_2,\ldots,x^{\pm 1}_n;y_1,y_2,\ldots,y_n]
\]
defined, for any $M \in \mathcal{C}$, by 
\[
X^{T}_{M} = {\bf x}^{g^T(M)} \sum_{\underline{e} \in \mathbb{N}^n} \chi(\text{Gr}_{\underline{e}}(\text{Hom}_{\mathcal{C}}(T,M))) {\bf x}^{a^{T}(\underline{e})} {\bf y}^{\underline{e}},
\]
where for any $d\in \mathbb{Z}^n$, ${\bf x}^{d}=\prod_{i=1}^n x^{d_i}_i$, ${\bf y}^{d}=\prod_{i=1}^n y^{d_i}_i$, and the $i$-th component of $a^{T}(\underline{e})\in \mathbb{Z}^n$ is defined by 
\[
(a^{T}(\underline{e}))_i = \sum_{j:j\to i} e_j - \sum_{j:i\to j} e_j. 
\] 
If $M$ is an indecomposable rigid reachable object in $\mathcal{C}$, then $X^{T}_{M}$ is a cluster variable in the cluster algebra $\mathcal{A}^T({\bf x},{\bf y},Q_B)$, where $Q_B$ is the quiver of $B$.

From now on, we ignore the superscript $T$ if $T$ is clear.

For two indecomposable rigid modules $L$ and $N$ with $\text{dim}(\text{Ext}^1_{\mathcal{C}}(N,L))=1$, it follows from \cite{Amiot,FK10,CK06} that  
\[
X_N X_L=X_M + \textbf{y}^{\alpha} X_{M'},
\] 
where $M, M'$ are the middle terms of the following non-split triangles 
\begin{align*}
L \to M \to N \to L[1],  \quad N \to M' \to L \to N[1],
\end{align*}
in $\mathcal{C}$, the vector $\alpha\in \mathbb{Z}_{\geq 0}^{I}$ is called a \textit{$c$-vector} of the exchange pair $(N,L)$ in the sense of Fei \cite[Definition 6.3]{Fei21}, defined as the dimension vector of the image of the nonzero morphism $h\colon\tau^{-1}_{\mathcal{C}}L \to N$ (denoted $L\to \tau_{\mathcal{C}} N$ in \cite{Fei21} due to different convention). 
Note that $h$ is unique up to scaling, because $\text{Hom}_{\mathcal{C}}(\tau^{-1}_{\mathcal{C}} L,N)\cong D\text{Ext}^1_{\mathcal{C}}(N,L)$ is one-dimensional by assumption. The vector $\alpha\in \mathbb{Z}_{\geq 0}^{I}$ satisfies that $-\mathcal{B}_{Q_B} \alpha=g(M)-g(M')$, where $\mathcal{B}_{Q_B}$ is the skew-symmetric matrix of $Q_B$, see \cite[Theorem 1.5]{ST16} (here $g$-vectors were interpreted as weights). In particular, for any almost-split exact sequence 
\[
0 \to \tau N \to M \to N \to 0,
\] 
where $N$ is a non-projective indecomposable module, we have
\begin{align*} 
X_{\tau N} X_N = X_{M} + \prod_{i\in I} y^{\text{dim}\,N_i}_i. 
\end{align*}

If $Q$ is a Dynkin quiver, then every indecomposable $kQ$-module is rigid. For any indecomposable module $M \in \text{mod}\,kQ$, there is an injective resolution of $M$ 
\[
0 \to M \to \bigoplus_{i\in I} I(i)^{a_i} \to  \bigoplus_{i\in I} I(i)^{b_i} \to 0, \text{ with $a_i, b_i \geq 0$}. 
\]
In this case, the $g$-vector $g(M)$ of $M$ is as follows:
\[
g(M)=(b_1-a_1,b_2-a_2,\ldots,b_{n}-a_{n}) \in \mathbb{Z}^{n}.
\]
Following \cite{FZ07}, we have
\begin{align}\label{separate formula for a cluster variable}
X_M=\textbf{z}^{g(M)} F_M(\widehat{y}_1,\ldots,\widehat{y}_{n}),
\end{align}
where $\widehat{y}_j=y_j(\prod_{i\in I} x^{b_{ij}}_i)$ for any $j\in I$.

\section{The subcategories $\mathscr{C}^{\leq \xi}_\ell$, Jacobian algebras $A^{\leq \xi}_\ell$, and quantum cluster algebras} \label{subcategories and cluster algebras}

Let $\gamma$ be a simply-laced Dynkin diagram with vertex set $I$, where we use the same labeling as in \cite{Bou02}.

\subsection{Subcategories {$\mathscr{C}^{\leq \xi}_\ell$}} \label{subcategories}

Let $\xi: I \to \mathbb{Z}$ be a height function such that $|\xi(i)-\xi(j)|=1$ if there is an edge $i\sim j$ in $\gamma$. Such $\xi$ exists, since $\gamma$ is finite and connected. It defines an orientation of the Dynkin diagram $\gamma$ via $i \to j$ if $\xi(i)=\xi(j)+1$.

We define an infinite quiver $\Gamma$ with vertex set $\widehat{I}=\{(i,p) \mid i\in I, p\in \xi(i)+2\mathbb{Z}\}$ and arrows given by $(i,r) \to (j,s)$ if and only if $C_{ij}\neq 0$, $s=r+C_{ij}$. The quiver $\Gamma$ was introduced by Hernandez and Leclerc \cite{HL16}, and it does not depend on the choice of $\xi$ up to isomorphism. An example of $\Gamma$ is shown in Figure \ref{quiver Gamma in type A3}. 

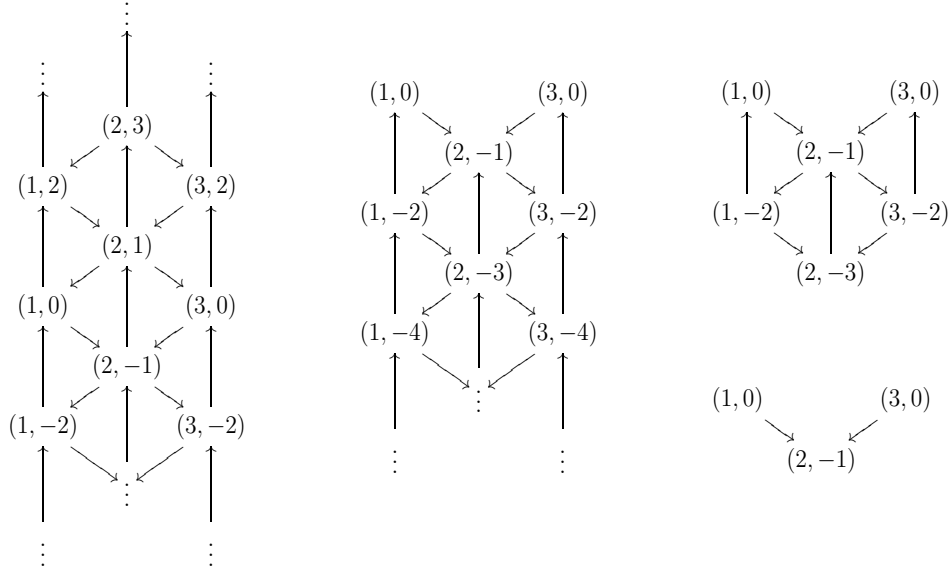
\begin{figure}
\centering
\resizebox{.8\textwidth}{.8\height}{
\begin{minipage}[c]{0.5\textwidth}
\begin{xy}
(40,60)*+{\vdots}="xx"; 
(25,70)*+{\vdots}="z";%
(10,60)*+{\vdots}="x"; 
(40,40)*+{(3,2)}="bb"; 
(25,50)*+{(2,3)}="a";%
(10,40)*+{(1,2)}="b";
(40,20)*+{(3,0)}="ee"; 
(25,30)*+{(2,1)}="d";%
(10,20)*+{(1,0)}="e";
(40,0)*+{(3,-2)}="hh";%
(25,10)*+{(2,-1)}="g";%
(10,0)*+{(1,-2)}="h";%
(10,-20)*+{\vdots}="j"; 
(25,-10)*+{\vdots}="k";
(40,-20)*+{\vdots}="jj";
{\ar "a";"b"};{\ar "a";"bb"};
{\ar "b";"d"};{\ar "bb";"d"};
{\ar "d";"a"};
{\ar "d";"e"};{\ar "d";"ee"};
{\ar "e";"b"};{\ar "ee";"bb"};
{\ar "e";"g"};{\ar "ee";"g"};%
{\ar "g";"d"};
{\ar "g";"h"};{\ar "g";"hh"};%
{\ar "h";"e"};{\ar "hh";"ee"};%
{\ar "j";"h"};{\ar "jj";"hh"};
{\ar "k";"g"};%
{\ar "h";"k"};{\ar "hh";"k"};%
{\ar "b";"x"};{\ar "bb";"xx"};
{\ar "a";"z"};
\end{xy}
\end{minipage}
\qquad \qquad
\begin{minipage}[c]{0.5\textwidth}
\begin{xy}
(40,40)*+{(3,0)}="bb"; 
(10,40)*+{(1,0)}="b";
(40,20)*+{(3,-2)}="ee"; 
(25,30)*+{(2,-1)}="d";%
(10,20)*+{(1,-2)}="e";
(40,0)*+{(3,-4)}="hh";%
(25,10)*+{(2,-3)}="g";%
(10,0)*+{(1,-4)}="h";%
(10,-20)*+{\vdots}="j"; 
(25,-10)*+{\vdots}="k";
(40,-20)*+{\vdots}="jj";
{\ar "b";"d"};{\ar "bb";"d"};
{\ar "d";"e"};{\ar "d";"ee"};
{\ar "e";"b"};{\ar "ee";"bb"};
{\ar "e";"g"};{\ar "ee";"g"};%
{\ar "g";"d"};
{\ar "g";"h"};{\ar "g";"hh"};%
{\ar "h";"e"};{\ar "hh";"ee"};%
{\ar "j";"h"};{\ar "jj";"hh"};
{\ar "k";"g"};%
{\ar "h";"k"};{\ar "hh";"k"};%
\end{xy}
\end{minipage}
\qquad \qquad
\begin{minipage}[c]{0.5\textwidth}
\begin{xy}
(40,40)*+{(3,0)}="bb"; 
(10,40)*+{(1,0)}="b";
(40,20)*+{(3,-2)}="ee"; 
(25,30)*+{(2,-1)}="d";%
(10,20)*+{(1,-2)}="e";
(25,10)*+{(2,-3)}="g";%
{\ar "b";"d"};{\ar "bb";"d"};
{\ar "d";"e"};{\ar "d";"ee"};
{\ar "e";"b"};{\ar "ee";"bb"};
{\ar "e";"g"};{\ar "ee";"g"};%
{\ar "g";"d"};
\end{xy}
\vspace{40pt}
\begin{xy}
(40,40)*+{(3,0)}="bb"; 
(10,40)*+{(1,0)}="b";
(25,30)*+{(2,-1)}="d";%
{\ar "b";"d"};{\ar "bb";"d"};
\end{xy}
\end{minipage}}
\caption{Quivers $\Gamma$ (left) and $\Gamma^{\leq \xi}$ (middle),  $\Gamma^{\leq \xi}_1$ (top right), and $Q$ (bottom right) in type $A_3$, where $\xi(1)=0$, $\xi(2)=-1$, and $\xi(3)=0$.} 
\label{quiver Gamma in type A3}
\end{figure}

Let $\Gamma^-$ and $\Gamma^{\leq \xi}$ be semi-infinite full subquivers of $\Gamma$ with vertex set $\widehat{I}^-=\widehat{I} \cap (I\times \mathbb{Z}_{\leq 0})$ and $\widehat{I}^{\leq \xi}=\{ (i,p) \mid i\in I, p\in \xi(i)+2\mathbb{Z}_{\leq 0}\}$, respectively. The two quivers $\Gamma^-$ and $\Gamma^{\leq \xi}$ are the same if we take an appropriate height function $\xi$. An example of $\Gamma^{\leq \xi}$ is shown in Figure \ref{quiver Gamma in type A3}. We define $k^{\xi}_{i,p}$, with $(i,p)\in \widehat{I}^{\leq \xi}$, to be the unique positive integer $k$ such that $(i,p)$ is the $k$-th vertex in its column, counting from the top $(i,\xi(i))$.

\begin{lemma}\label{mutation equivalent of Gamma}
For any choice of $\xi$, the quiver $\Gamma^{\leq \xi}$ is mutation equivalent to $\Gamma^-$ as quivers.
\end{lemma}
\begin{proof}
We identify each vertex $(i,p)$ in $\Gamma^{\leq \xi}$ with the vertex $(i,r)$ in $\Gamma^-$ such that $k_{i,r}=k^{\xi}_{i,p}$. Let $Q^{\leq \xi}_1$ be the subquiver of $\Gamma^{\leq \xi}$ with vertex set $I_\xi=\{(i,\xi(i)) \mid i\in I\}$, $Q$ the subquiver of $\Gamma^-$ with vertex set consisting of the top vertex in each column. By definition, $Q$ is a bipartite Dynkin quiver. It follows from \cite[Lemma 5.1.4]{Mar13} that there exists a sequence of source/sink mutations $\mathscr{S}$ that transforms $Q^{\leq \xi}_1$ into $Q$. We modify the sequence $\mathscr{S}$ by replacing each $(i,\xi(i))$ by the ordered vertices $(i,\xi(i)), (i,\xi(i)-2), \ldots$, up to infinity. The new sequence transforms $\Gamma^{\leq \xi}$ into $\Gamma^-$.
\end{proof}

For each $\ell\in \mathbb{Z}_{\geq 1}$, let $\Gamma^{\leq \xi}_\ell$ be the subquiver of $\Gamma^{\leq \xi}$ with vertex set $\widehat{I}^{\leq \xi}_\ell=\{(i,p) \mid i\in I, -2\ell+\xi(i) \leq p \leq \xi(i)\}$.

Fix a non-zero $a\in \mathbb{C}^*$, we simply write $Y_{i,r}$ for $Y_{i,aq^r}$. Let $\mathcal{P}$ (respectively, $\mathcal{P}^{\leq \xi}$, $\mathcal{P}^{\leq \xi}_\ell$) be the abelian group generated by $Y^{\pm 1}_{i,p}$, with $(i,p)\in \widehat{I}$ (respectively, $(i,p)\in \widehat{I}^{\leq \xi}$, $(i,p)\in \widehat{I}^{\leq \xi}_\ell$), and $\mathcal{P}_{+}$ (respectively, $\mathcal{P}^{\leq \xi}_{+}$, $\mathcal{P}^{\leq \xi}_{+,\ell}$) the submonoid of $\mathcal{P}$ (respectively, $\mathcal{P}^{\leq \xi}$, $\mathcal{P}^{\leq \xi}_\ell$) with generators $Y_{i,p}$, with $(i,r)\in \widehat{I}$ (respectively, $(i,r)\in \widehat{I}^{\leq \xi}$, $(i,r)\in \widehat{I}^{\leq \xi}_\ell$).  

Let $\mathscr{C}_{\mathbb{Z}}$ (respectively, $\mathscr{C}^{\leq \xi}$, $\mathscr{C}^{\leq \xi}_\ell$) be the monoidal subcategory of $\mathscr{C}$ whose objects have all their composition factors of the form $L(m)$, where $m\in \mathcal{P}_+$ (respectively, $m\in \mathcal{P}^{\leq \xi}_{+}$, $m\in  \mathcal{P}^{\leq \xi}_{+,\ell}$). This defines a series of subcategories of $\mathscr{C}$
\[
\mathscr{C}^{\leq \xi}_1 \subset \mathscr{C}^{\leq \xi}_2 \subset \cdots \subset \mathscr{C}^{\leq \xi}_\ell \subset \cdots \subset \mathscr{C}^{\leq \xi}_\infty:=\mathscr{C}^{\leq \xi} \subset \mathscr{C}_{\mathbb{Z}}.
\]

\begin{remark}
\begin{itemize}[]

\item[(1)] The subcategory $\mathscr{C}^{\leq \xi}_1$ was introduced in \cite{HL13}. For a linear height function $\xi$, there exists a subcategory $\mathscr{C}_Q$ of $\mathscr{C}^{\leq \xi}$ such that $\mathscr{C}_Q$ is a monoidal categorification of the quantum cluster algebra $\mathbb{C}_q[N]$, where $N$ is the maximal unipotent radical of $G=\text{Lie}(\mathfrak{g})$, see \cite{HL15}. Very recently, the subcategory $\mathscr{C}^{\leq \xi}$ was introduced and studied by Fujita, Hernandez, Oh and Oya \cite[Section 4.3]{FHOO23} for general Dynkin type, not only ADE type.

\item[(2)] If $\xi$ is a sink-source function, then the subcategory $\mathscr{C}^{\leq \xi}_\ell$ is the same as the Hernandez-Leclerc subcategories $\mathscr{C}_\ell$ \cite{HL16,HL21}.
 
\item[(3)] It is interesting to compare $\mathscr{C}^{\leq \xi}_\ell$ with Kashiwara-Kim-Oh-Park's  subcategory introduced in \cite[Definition 6.16]{KKOP23}.
\end{itemize} 
\end{remark}
 
The $q$-character of a simple object $L(m)\in \mathscr{C}^{\leq \xi}_\ell$ possibly contains  many monomials which do not belong to $\mathcal{P}^{\leq \xi}_{\ell}$. By discarding these monomials we obtain a truncated $q$-character $\chi_{q}([L(m)])_{\leq \xi}$. The Grothendieck ring $K_0(\mathscr{C}^{\leq \xi}_\ell)$ of $\mathscr{C}^{\leq \xi}_\ell$ is a polynomial ring in the truncated $q$-characters $\chi_{q}([L(Y_{i,p})])_{\leq \xi}$, with $(i,p)\in \widehat{I}^{\leq \xi}_\ell$.

One checks that all the dominant monomials (including multiplicities) occurring in $\chi_{q}([L(m)])$ appear in the truncated $q$-character $\chi_{q}([L(m)])_{\leq \xi}$. 

\begin{proposition}[{\cite[Proposition 2.2]{HL10}}] \label{dominant monomials determine q-characters}
If $\chi_{q}([M_1])_{\leq \xi}=\chi_{q}([M_2])_{\leq \xi}$ for any two simple objects $M_1, M_2\in \mathscr{C}^{\leq \xi}_\ell$, then $\chi_q([M_1])=\chi_q([M_2])$. 
\end{proposition}

Let $\mathcal{A}^{\leq \xi}_\ell$ be the cluster algebra with initial seed $(\mathbf{z}^{\leq \xi}_\ell,\Gamma^{\leq \xi}_\ell)$ and frozen vertices $\{ z^{\leq \xi}_{i,p} \mid (i,p) \in \widehat{I}^{\leq \xi}_\ell \setminus \widehat{I}^{\leq \xi}_{\ell-1}\}$, where
\[
\mathbf{z}^{\leq \xi}_\ell=\{ z^{\leq \xi}_{i,p} \mid (i,p) \in \widehat{I}^{\leq \xi}_\ell \},
\] 
and $z^{\leq \xi}_{i,p}$, with $(i,p) \in \widehat{I}^{\leq \xi}_\ell$, are the generators of a field of rational functions $\mathbb{Q}(\mathbf{z}^{\leq \xi}_\ell)$.

We will show in Section \ref{quantum cluster algebra structure on quantum Grothendieck ring of a subcategory} that \cite[Propostion 3.10]{HL16} and \cite[Theorem 5.1]{HL21} imply that $K_0(\mathscr{C}^{\leq \xi}_\ell)$ admits a cluster algebra structure $\mathcal{A}^{\leq \xi}_\ell$.

\subsection{Jacobian algebra $A^{\leq \xi}_\ell$} \label{our Jacobian algebras}

Let $Q^{\leq \xi}_\ell$ be the principal quiver of $\Gamma^{\leq \xi}_\ell$ obtained by deleting frozen vertices  $\{(i,p) \mid (i,p) \in \widehat{I}^{\leq \xi}_\ell \setminus \widehat{I}^{\leq \xi}_{\ell-1}\}$ and incident arrows. By definition, if $\ell>1$, the quiver $Q^{\leq \xi}_\ell$ has cycles of the form shown in Figure \ref{cycle type}.

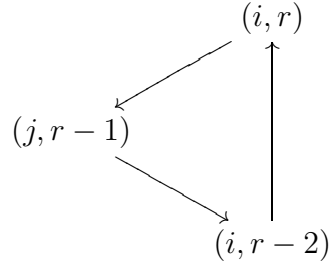
\begin{figure}[H]
\centerline{
\xymatrix{& (i,r) \ar[dl] \\ (j,r-1)\ar[dr] & \\ & (i,r-2) \ar[uu]}}
\caption{A cycle in $Q^{\leq \xi}_\ell$. Here the index $j$ is such that $C_{ij}=-1$.}\label{cycle type}
\end{figure}

Let $\mathfrak{P}$ be the sum of all the 3-cycles of the form shown in Figure \ref{cycle type}. Then $\mathfrak{P}$ is a potential in the sense of \cite[Section 3]{DWZ08}. Note that a given arrow of $Q^{\leq \xi}_\ell$ can only occur in a finite number of summands in  $\mathfrak{P}$. For an arrow $\alpha$ in $Q^{\leq \xi}_\ell$, the cyclic derivative $\partial_\alpha$ on a potential $\mathfrak{P}$ is the continuous $\mathbb{C}$-linear map which sends $\mathfrak{P}$ to the sum $\sum_{l=\beta\alpha\delta}\delta\beta$, where the sum takes over all possible cycles containing $\alpha$ in $\mathfrak{P}$.

Let 
\[
A^{\leq \xi}_\ell=\mathbb{C}Q^{\leq \xi}_\ell/J^{\leq \xi}_\ell,
\] 
where $J^{\leq \xi}_\ell$ is the truncated Jacobian ideal of $\mathbb{C}Q^{\leq \xi}_\ell$ generated by all cyclic derivatives in $Q^{\leq \xi}_\ell$. For each $\ell\in \mathbb{Z}_{\geq 1}$, the truncated Jacobian algebra $A^{\leq \xi}_\ell$ is a finite-dimensional $\mathbb{C}$-algebra, and the quiver with potential $(Q^{\leq \xi}_\ell,J^{\leq \xi}_\ell)$ is rigid, see \cite[Proposition 4.17]{HL16}.

Denote by $\text{mod}\,A^{\leq \xi}_\ell$ the category of finitely generated right $A^{\leq \xi}_\ell$-modules. It is well-known that $\text{mod}\,A^{\leq \xi}_\ell \cong \text{rep}(Q^{\leq \xi}_\ell,J^{\leq \xi}_\ell)$.

\subsection{Quantum Cartan matrix and quantum cluster algebra} \label{quantum cluster algebra structure on quantum Grothendieck ring of a subcategory}

Let $C$ be the Cartan matrix of $\mathfrak{g}$ and $C(z)=(C(z)_{ij})$ the corresponding quantum Cartan matrix defined by 
\[
C(z)_{ij} = \begin{cases}
z+z^{-1}  & \text{if $i=j$}, \\
-1  &  \text{if there is an edge $i\sim j$ in the Dynkin diagram of $\mathfrak{g}$}, \\
0 & \text{otherwise},
\end{cases}
\]
where $z$ is an indeterminate.

It follows from the above definition that $C(z)$ is an invertible symmetric matrix, and the specialization of $C(z)$ at $z=1$ is the ordinary Cartan matrix $C$. We denote by $\widetilde{C}(z)=(\widetilde{C}(z)_{ij})$ the inverse matrix of $C(z)$. The matrix $\widetilde{C}(z)$ is a symmetric matrix, one can write the entries of $\widetilde{C}(z)$ as power series in $z$:
\[
\widetilde{C}(z)_{ij}=\sum_{k=0}^{\infty} \widetilde{C}_{ij}(k)z^k \in \mathbb{Z}[[z]].
\]  
It was shown by Hernandez and Leclerc in \cite[Corollary 2.3]{HL15} that for $i,j\in I$ and $k\geq 1$, 
\[
\widetilde{C}_{ij}(k)=\widetilde{C}_{ij}(k+2h),
\] 
where $h$ is the Coxeter number of $C$.

\begin{example}
In type $A_n$, 
\begin{itemize}
\item[(1)] $|C(z)| = [n+1]_q$;
\item[(2)] for $1\leq i\leq j \leq n$,
\begin{align*}
\widetilde{C}(z)_{ij} = \left(\sum_{s=0}^{n-j} z^{j-i+2s+1}  -  \sum_{s=0}^{n-j} z^{j+i+2s+1} \right) \left( \sum_{k\geq 0} z^{2(n+1)k} \right).
\end{align*}
\end{itemize}
The case (2) appeared in \cite[Remark 3.1]{FH22}. Here we use a different expression of $\widetilde{C}(z)_{ij}$. 
\end{example}

The quantum Cartan matrix is closely related to the representations of a simply-laced quantum affine algebra \cite{Bit21}. Following \cite{HL15},  one can extend the definition of $\widetilde{C}_{ij}(k)$ to every $k\in \mathbb{Z}$ by setting $\widetilde{C}_{ij}(k) = 0$ if $k \leq 0$.

Let $t$ be an indeterminate. Let  $(\mathcal{Y}_t,*)$ (respectively, $(\mathcal{Y}^{\leq \xi}_t,*)$,  $(\mathcal{Y}^{\leq \xi}_{t,\ell},*)$) be the $\mathbb{C}[t^{\pm \frac{1}{2}}]$-algebra generated by variables $Y^{\pm 1}_{i,r}$, with $(i,r)\in \widehat{I}$ (respectively, $(i,r)\in \widehat{I}^{\leq \xi}$, $(i,r)\in \widehat{I}^{\leq \xi}_\ell$), subject to the following relations: for $(i,r), (j,s) \in \widehat{I}$ (respectively, $\widehat{I}^{\leq \xi}$, $\widehat{I}^{\leq \xi}_\ell$), 
\begin{align}\label{t-multiplication}
&  Y_{i,r} * Y^{-1}_{i,r} = Y^{-1}_{i,r} * Y_{i,r} =1, \notag \\
&  Y_{i,r} * Y_{j,s} = t^{\mathcal{N}_{ij}(s-r)} Y_{j,s} * Y_{i,r},  
\end{align}
where $\mathcal{N}_{ij}: \mathbb{Z} \to \mathbb{Z}$ is the odd function defined by
\[
\mathcal{N}_{ij}(k) = \widetilde{C}_{ij}(k+1) +  \widetilde{C}_{ij}(-k-1) -  \widetilde{C}_{ij}(k-1)-\widetilde{C}_{ij}(-k+1).
\]
Since $\widetilde{C}(z)$ is symmetric, we have $\mathcal{N}_{ij}(k)=-\mathcal{N}_{ji}(-k)$. 

Let $L^{\leq \xi}_\ell$ be the $\widehat{I}^{\leq \xi}_\ell \times \widehat{I}^{\leq \xi}_\ell$ skew-symmetric matrix whose $((i,r),(j,s))$-entry (for $\xi(i)-2\ell \leq r \leq \xi(i)$, $\xi(j)-2\ell \leq s \leq \xi(j)$) is defined by

\begin{gather}\label{Lambda matrix entries}
\begin{aligned}
& L^{\leq \xi}_\ell((i,r),(j,s)) = \sum_{\substack{k\geq 0 \\ r+2k\leq \xi(i)}} \sum_{\substack{l\geq 0\\ s+2l \leq \xi(j)}} \mathcal{N}_{ij}(s+2l-r-2k) \\
%& = \sum_{\substack{k\geq 0 \\ r+2k\leq \xi(i)}} \left( \mathcal{N}_{ij}(s-r-2k) + \mathcal{N}_{ij}(s+2-r-2k) + \cdots +   \mathcal{N}_{ij}(\xi(j)-r-2k) \right) \\
%& = \sum_{\substack{k\geq 0 \\ r+2k\leq \xi(i)}} \left( \widetilde{C}_{ij}(s-r-2k+1) + \widetilde{C}_{ij}(-s+r+2k-1)  - \widetilde{C}_{ij}(s-r-2k-1) - \widetilde{C}_{ij}(-s+r+2k+1) \right. \\
%& \quad + \widetilde{C}_{ij}(s-r-2k+3) + \widetilde{C}_{ij}(-s+r+2k-3) - \widetilde{C}_{ij}(s-r-2k+1) - \widetilde{C}_{ij}(-s+r+2k-1)   \\
%& \quad  + \cdots + \\
%& \quad  \left. \widetilde{C}_{ij}(\xi(j)-r-2k+1) + \widetilde{C}_{ij}(-\xi(j)+r+2k-1) - \widetilde{C}_{ij}(\xi(j)-r-2k-1) - \widetilde{C}_{ij}(-\xi(j)+r+2k+1)  \right) \\
& = \sum_{\substack{k\geq 0 \\ r+2k\leq \xi(i)}} \left(  \widetilde{C}_{ij}(\xi(j)-r-2k+1) + \widetilde{C}_{ij}(-\xi(j)+r+2k-1) - \widetilde{C}_{ij}(s-r-2k-1) - \widetilde{C}_{ij}(-s+r+2k+1)    \right).
\end{aligned}
\end{gather}

With this notation, the $t$-commutation relation (\ref{t-multiplication}) yields the following equation
\begin{gather}\label{t-commutative relation in Yt}
\left(\prod_{k\geq 0,\atop r+2k\leq \xi(i)} Y_{i,r+2k}\right) * \left(\prod_{l\geq 0,\atop s+2l\leq \xi(j)} Y_{j,s+2l}\right)= t^{L^{\leq \xi}_\ell((i,r),(j,s))} \left(\prod_{l\geq 0,\atop s+2l\leq \xi(j)} Y_{j,s+2l}\right) * \left(\prod_{k\geq 0,\atop r+2k\leq \xi(i)} Y_{i,r+2k}\right).
\end{gather}

We need the following proposition.

\begin{proposition}[{\cite{Bit21a,Bit21,HL15}}] \label{quantum Cartan matrix relations}
For all $(i,j) \in I \times I$, 
\begin{align*}
& \widetilde{C}_{ij}(m-1) +  \widetilde{C}_{ij}(m+1) - \sum_{k \sim j} \widetilde{C}_{ik}(m) = 0, \, \forall \,  m \geq 1, \\
& \widetilde{C}_{ij}(m-1) +  \widetilde{C}_{ij}(m+1) - \sum_{k \sim i} \widetilde{C}_{kj}(m) = 0, \, \forall \,  m \geq 1, \\
& \widetilde{C}_{ij}(1) =  \delta_{i,j}.
\end{align*}
\end{proposition}

For any choice of $\xi$,  we have the following analog of Bittmann's Proposition 5.1.1 in \cite{Bit21} (where the sink-source height function is used) and Proposition 6.2.4 in \cite{Bit21a}.
\begin{proposition} \label{Compatible pair}
For any choice of $\xi$, $(L^{\leq \xi}_\ell,B^{\leq \xi}_\ell)$ forms a compatible pair. 
\end{proposition}

\begin{proof}
Suppose that $((i,r),(j,s))\in \widehat{I}^{\leq \xi}_{\ell-1}\times \widehat{I}^{\leq \xi}_\ell$ for $r\leq \xi(i)-2$. We compute the position $((i,r), (j,s))$ of the matrix $(B^{\leq \xi}_\ell)^T L^{\leq \xi}_\ell$ as follows:
\begin{gather*}
\begin{aligned}
& (B^{\leq \xi}_\ell)^T L^{\leq \xi}_\ell ((i,r), (j,s))  = L^{\leq \xi}_\ell ((i,r+2), (j,s))-L^{\leq \xi}_\ell ((i,r-2), (j,s)) \\
&  \qquad \qquad \qquad \qquad \qquad - \sum_{u \sim i} \left( L^{\leq \xi}_\ell ((u,r+1), (j,s))-L^{\leq \xi}_\ell ((u,r-1), (j,s)) \right) \\
& = \sum_{ k\geq 0, \atop r+2k+2\leq \xi(i)} \left( \widetilde{C}_{ij}(-r-2k+\xi(j)-1)+\widetilde{C}_{ij}(r+2k-\xi(j)+1) - \widetilde{C}_{ij}(-r-2k+s-3) -\widetilde{C}_{ij}(r+2k-s+3)  \right) \\
& - \sum_{ k\geq 0, \atop r+2k-2\leq \xi(i)} \left( \widetilde{C}_{ij}(-r-2k+\xi(j)+3)+\widetilde{C}_{ij}(r+2k-\xi(j)-3)-\widetilde{C}_{ij}(-r-2k+s+1) -\widetilde{C}_{ij}(r+2k-s-1) \right)  \\
& - \sum_{u \sim i}  \sum_{k\geq 0,\atop r+2k+1\leq \xi(u)} \left( \widetilde{C}_{uj}(-r-2k+\xi(j))+ \widetilde{C}_{uj}(r+2k-\xi(j)) -\widetilde{C}_{uj}(-r-2k+s-2) -\widetilde{C}_{uj}(r+2k-s+2) \right) \\
& + \sum_{u \sim i} \sum_{k\geq 0,\atop r+2k-1\leq \xi(u)} \left( \widetilde{C}_{uj}(-r-2k+\xi(j)+2) + \widetilde{C}_{uj}(r+2k-\xi(j)-2) -\widetilde{C}_{uj}(-r-2k+s) - \widetilde{C}_{uj}(r+2k-s) \right)  \\
& = -\left( \widetilde{C}_{ij}(-r+\xi(j)+3)+\widetilde{C}_{ij}(r-\xi(j)-3)-\widetilde{C}_{ij}(-r+s+1)-\widetilde{C}_{ij}(r-s-1) \right)  \\
& - \left( \widetilde{C}_{ij}(-r+\xi(j)+1) + \widetilde{C}_{ij}(r-\xi(j)-1) - \widetilde{C}_{ij}(-r+s-1) - \widetilde{C}_{ij}(r-s+1) \right) \\
& + \sum_{u \sim i} \left( \widetilde{C}_{uj}(-r+\xi(j)+2) + \widetilde{C}_{uj}(r-\xi(j)-2) - \widetilde{C}_{uj}(-r+s) - \widetilde{C}_{uj}(r-s) \right).
\end{aligned}
\end{gather*}
Here the second equality uses Equation (\ref{Lambda matrix entries}), the last equality is obtained by combining terms of the same kind.

Rearranging the terms, applying Proposition \ref{quantum Cartan matrix relations}, and $\widetilde{C}_{ij}(m) = 0$ for $m\leq 0$, we obtain
\begin{gather}\label{step 1 formula}
\begin{aligned}
(B^{\leq \xi}_\ell)^T L^{\leq \xi}_\ell ((i,r), (j,s)) & =  -\left( \widetilde{C}_{ij}(r-\xi(j)-3) + \widetilde{C}_{ij}(r-\xi(j)-1)- \sum_{u \sim i} \widetilde{C}_{uj}(r-\xi(j)-2) \right) \\
& - \left( \widetilde{C}_{ij}(-r+\xi(j)+3) + \widetilde{C}_{ij}(-r+\xi(j)+1)- \sum_{u \sim i} \widetilde{C}_{uj}(-r+\xi(j)+2) \right)  \\
& + \left( \widetilde{C}_{ij}(r-s-1) + \widetilde{C}_{ij}(r-s+1) - \sum_{u \sim i} \widetilde{C}_{uj}(r-s)  \right) \\
& + \left( \widetilde{C}_{ij}(-r+s+1) + \widetilde{C}_{ij}(-r+s-1) - \sum_{u \sim i}  \widetilde{C}_{uj}(-r+s)  \right).
\end{aligned}
\end{gather}

In the following, we discuss (\ref{step 1 formula}) by a case-by-case analysis. 

{\bf Case 1.} Suppose that $(i,r)\neq (j,s)$. 

We start with the discussion of the first two terms on the right hand side of Equation~(\ref{step 1 formula}).

{\bf Case 1.1} If $r-\xi(j)-1\leq 0$, then (we agree that $\widetilde{C}_{ij}(m) = 0$ for $m\leq 0$)
\[
\widetilde{C}_{ij}(r-\xi(j)-3) + \widetilde{C}_{ij}(r-\xi(j)-1)- \sum_{u \sim i} \widetilde{C}_{uj}(r-\xi(j)-2)=0.
\]
In this case, $-(r-\xi(j)-1) \geq 0$, so by Proposition \ref{quantum Cartan matrix relations}
\[
\widetilde{C}_{ij}(-r+\xi(j)+3) + \widetilde{C}_{ij}(-r+\xi(j)+1)- \sum_{u \sim i} \widetilde{C}_{uj}(-r+\xi(j)+2)=0.
\]

{\bf Case 1.2} If $r-\xi(j)-1=1$, then ($r=\xi(j)+2$)
\begin{align*}
& \widetilde{C}_{ij}(r-\xi(j)-3) + \widetilde{C}_{ij}(r-\xi(j)-1) - \sum_{u \sim i} \widetilde{C}_{uj}(r-\xi(j)-2)=\widetilde{C}_{ij}(1), \\
& \widetilde{C}_{ij}(-r+\xi(j)+3) + \widetilde{C}_{ij}(-r+\xi(j)+1)- \sum_{u \sim i} \widetilde{C}_{uj}(-r+\xi(j)+2)=\widetilde{C}_{ij}(1) .
\end{align*}
Note that $r\leq \xi(i)-2$, which implies that $i\neq j$, otherwise if $i=j$, $\xi(j)+2=\xi(i)+2<\xi(i)-2$, a contradiction. By Proposition \ref{quantum Cartan matrix relations}, we have $\widetilde{C}_{ij}(1)=0$.

{\bf Case 1.3} If $r-\xi(j)-1\geq 2$, then by Proposition \ref{quantum Cartan matrix relations} and $\widetilde{C}_{ij}(m) = 0$ for $m\leq 0$
\begin{align*}
\widetilde{C}_{ij}(r-\xi(j)-3) + \widetilde{C}_{ij}(r-\xi(j)-1) - \sum_{u \sim i} \widetilde{C}_{uj}(r-\xi(j)-2)=0,\\
\widetilde{C}_{ij}(-r+\xi(j)+3) + \widetilde{C}_{ij}(-r+\xi(j)+1)- \sum_{u \sim i} \widetilde{C}_{uj}(-r+\xi(j)+2)=0.
\end{align*}
Thus in all three cases the sum of the first two terms on the right hand side of Equation~(\ref{step 1 formula}) is zero.

Next we discuss the last two summands on the right hand side of Equation (\ref{step 1 formula}).
 
{\bf Case 1.1'}  If $r-s+1\leq 0$, then
\begin{align*}
& \widetilde{C}_{ij}(r-s-1) + \widetilde{C}_{ij}(r-s+1) - \sum_{u \sim i} \widetilde{C}_{uj}(r-s)=0,\\
& \widetilde{C}_{ij}(-r+s+1) + \widetilde{C}_{ij}(-r+s-1) - \sum_{u \sim i}  \widetilde{C}_{uj}(-r+s)=0.
\end{align*}

{\bf Case 1.2'}  If $r-s+1=1$, then $r=s$ (implies $i\neq j$, otherwise $(i,s)=(j,t)$),
\begin{align*}
& \widetilde{C}_{ij}(r-s-1) + \widetilde{C}_{ij}(r-s+1) - \sum_{u \sim i} \widetilde{C}_{uj}(r-s)=\widetilde{C}_{ij}(1)=0, \\
& \widetilde{C}_{ij}(-r+s+1) + \widetilde{C}_{ij}(-r+s-1) - \sum_{u \sim i}  \widetilde{C}_{uj}(-r+s)=\widetilde{C}_{ij}(1)=0.
\end{align*}

{\bf Case 1.3'}  If $r-s+1\geq 2$, then 
\begin{align*}
& \widetilde{C}_{ij}(r-s-1) + \widetilde{C}_{ij}(r-s+1) - \sum_{u \sim i} \widetilde{C}_{uj}(r-s)=\widetilde{C}_{ij}(1)=0,\\
& \widetilde{C}_{ij}(-r+s+1) + \widetilde{C}_{ij}(-r+s-1) - \sum_{u \sim i}  \widetilde{C}_{uj}(-r+s)=\widetilde{C}_{ij}(1)=0.
\end{align*}
Therefore, the right hand side of Equation (\ref{step 1 formula}) is zero in case 1.

{\bf Case 2.} Suppose that $(i,r)=(j,s)$. Note that in this case 
\[
r-\xi(i)-1 \leq \xi(i)-2-\xi(i)-1=-3,
\] 
and we agree that $\widetilde{C}_{ij}(m) = 0$ for $m\leq 0$. We have 
\begin{align*}
& \widetilde{C}_{ij}(r-\xi(j)-3) + \widetilde{C}_{ij}(r-\xi(j)-1)- \sum_{u \sim i} \widetilde{C}_{uj}(r-\xi(j)-2)=0, \\
& \widetilde{C}_{ij}(-r+\xi(j)+3) + \widetilde{C}_{ij}(-r+\xi(j)+1)- \sum_{u \sim i} \widetilde{C}_{uj}(-r+\xi(j)+2)=0, \\
& \widetilde{C}_{ij}(r-s-1) + \widetilde{C}_{ij}(r-s+1) - \sum_{u \sim i} \widetilde{C}_{uj}(r-s)=\widetilde{C}_{ij}(1)=1, \quad (i=j) \\
& \widetilde{C}_{ij}(-r+s+1) + \widetilde{C}_{ij}(-r+s-1) - \sum_{u \sim i}  \widetilde{C}_{uj}(-r+s)=\widetilde{C}_{ij}(1)=1, \quad (i=j)
\end{align*}
where the last three equations use Proposition \ref{quantum Cartan matrix relations}.

In conclusion, 
\begin{gather*}
(B^{\leq \xi}_\ell)^T L^{\leq \xi}_\ell ((i,r), (j,s)) = \begin{cases}
2 & \text{if } (i,r)=(j,s), \\
0  & \text{otherwise}.
\end{cases}
\end{gather*}

It remains the case where $r=\xi(i)$. Suppose that $((i,\xi(i)),(j,s))\in \widehat{I}^{\leq \xi}_{\ell-1}\times \widehat{I}^{\leq \xi}_\ell$. If there is an arrow $(u,\xi(i)+1) \to (i,\xi(i))$, then there exists an arrow $(i,\xi(i)) \to (u,\xi(i)-1)$. In this case, 
\begin{gather*}
\begin{aligned}
L^{\leq \xi}_\ell ((u,\xi(i)+1), (j,s))  - & L^{\leq \xi}_\ell ((u,\xi(i)-1), (j,s))  =  \widetilde{C}_{uj}(\xi(j)-\xi(i)+2) \\
& \qquad   + \widetilde{C}_{uj}(\xi(i) - \xi(j)-2)  - \widetilde{C}_{uj}(s-\xi(i)) - \widetilde{C}_{uj}(\xi(i)-s).
\end{aligned}
\end{gather*}

By Equation (\ref{Lambda matrix entries}) and Proposition \ref{quantum Cartan matrix relations} again, we have the following equation: 

\begin{gather*}
\begin{aligned}
& (B^{\leq \xi}_\ell)^T L^{\leq \xi}_\ell ((i,\xi(i)), (j,s)) = - L^{\leq \xi}_\ell ((i,\xi(i)-2), (j,s))  \\
& +  \sum_{u \sim i}  \left( \widetilde{C}_{uj}(\xi(i)-\xi(j)-2) + \widetilde{C}_{uj}(-\xi(i)+\xi(j)+2)- \widetilde{C}_{uj}(\xi(i)-s) - \widetilde{C}_{uj}(-\xi(i)+s) \right)   \\
& = \left( -\widetilde{C}_{ij}(\xi(i)-\xi(j)-1) - \widetilde{C}_{ij}(-\xi(i)+\xi(j)+1) + \widetilde{C}_{ij}(\xi(i)-s+1) + \widetilde{C}_{ij}(-\xi(i)+s-1) \right)  \\
& + \left(-\widetilde{C}_{ij}(\xi(i)-\xi(j)-3) - \widetilde{C}_{ij}(-\xi(i)+\xi(j)+3) + \widetilde{C}_{ij}(\xi(i)-s-1) + \widetilde{C}_{ij}(-\xi(i)+s+1) \right)  \\
& + \sum_{u \sim i}  \left( \widetilde{C}_{uj}(\xi(i)-\xi(j)-2) + \widetilde{C}_{uj}(-\xi(i)+\xi(j)+2)- \widetilde{C}_{uj}(\xi(i)-s) - \widetilde{C}_{uj}(-\xi(i)+s) \right) \\
& = \left( -\widetilde{C}_{ij}(\xi(i)-\xi(j)-1) - \widetilde{C}_{ij}(\xi(i)-\xi(j)-3) + \sum_{u\sim i}  \widetilde{C}_{uj} (\xi(i)-\xi(j)-2)  \right) \\
& - \left( \widetilde{C}_{ij}(-\xi(i)+\xi(j)+1) + \widetilde{C}_{ij}(-\xi(i)+\xi(j)+3) - \sum_{u\sim i}  \widetilde{C}_{uj} (-\xi(i)+\xi(j)+2)  \right) \\
& + \left( \widetilde{C}_{ij}(\xi(i)-s-1) + \widetilde{C}_{ij}(\xi(i)-s+1) - \sum_{u\sim i}  \widetilde{C}_{uj} (\xi(i)-s) \right) \\
& + \left( \widetilde{C}_{ij}(-\xi(i)+s-1) + \widetilde{C}_{ij}(-\xi(i)+s+1) - \sum_{u\sim i}  \widetilde{C}_{uj} (-\xi(i)+s) \right) \\
& = \begin{cases}
2 & \text{if } (i,\xi(i))=(j,s), \\
0  & \text{otherwise},
\end{cases}
\end{aligned}
\end{gather*}
where the last equation is obtained by a similar case-by-case analysis as we did before.
\end{proof}

We are now ready to define the quantum cluster algebra. Let $\mathcal{A}^{\leq \xi}_{t,\ell}$ be the quantum cluster algebra with initial quantum seed $(\textbf{z}^{\leq \xi}_{t,\ell},L^{\leq \xi}_\ell,B^{\leq \xi}_\ell)$ and frozen vertices $\{ z^{\leq \xi}_{t,(i,p)} \mid (i,p) \in \widehat{I}^{\leq \xi}_\ell \setminus \widehat{I}^{\leq \xi}_{\ell-1}\}$, where
\[
\textbf{z}^{\leq \xi}_{t,\ell}=\{ z^{\leq \xi}_{t,(i,p)} \mid (i,p)\in \widehat{I}^{\leq \xi}_\ell \},
\] 
and $z^{\leq \xi}_{t,(i,p)}$, as well as its inverse, with $(i,p) \in \widehat{I}^{\leq \xi}_\ell$, are the generators of a quantum torus $\mathcal{T}^{\leq \xi}_{\ell}$, subject to 
\begin{align} \label{t-commutative relation in Tt}
z^{\leq \xi}_{t,(i,r)} *  z^{\leq \xi}_{t,(j,s)} = t^{L^{\leq \xi}_\ell((i,r),(j,s))}  z^{\leq \xi}_{t,(j,s)} *  z^{\leq \xi}_{t,(i,r)}.
\end{align}

The map
\begin{align*}
\eta: \mathcal{T}^{\leq \xi}_{\ell} & \longrightarrow \mathcal{Y}^{\leq \xi}_{t,\ell}, \\
z^{\leq \xi}_{t,(i,r)} & \longmapsto \prod_{k\geq 0, \atop  r+2k\leq \xi(i)} Y_{i,r+2k}, \text{ for } (i,r)\in \widehat{I}^{\leq \xi}_\ell,
\end{align*}
gives an isomorphism of quantum tori, because these generators have the same quasi-commutation relations (Equation (\ref{t-commutative relation in Yt}) and Equation (\ref{t-commutative relation in Tt})), and the map $\eta$ is invertible, with inverse:
\begin{align*}
\eta^{-1}: \mathcal{Y}^{\leq \xi}_{t,\ell}  & \longrightarrow \mathcal{T}^{\leq \xi}_{\ell}, \\
Y_{i,r} & \longmapsto \begin{cases}
z^{\leq \xi}_{t,(i,r)} (z^{\leq \xi}_{t,(i,r+2)})^{-1}  & \text{if $(i,r+2)\in \widehat{I}^{\leq \xi}_\ell$}, \\
z^{\leq \xi}_{t,(i,r)} & \text{otherwise}.
\end{cases} 
\end{align*}

\section{Quantum cluster algebra structures on quantum Grothendieck rings of subcategories}\label{quantum cluster algebras on subcategories}

In this section, we shall prove that there is a quantum cluster algebra structure on the quantum Grothendieck ring of $\mathscr{C}^{\leq \xi}_\ell$.

\subsection{Quantum Grothendieck ring $K_t(\mathscr{C}_{\mathbb{Z}})$}

We recall the definition of the $t$-deformation $K_t(\mathscr{C}_{\mathbb{Z}})$ (also called the quantum Grothendieck ring) of $K_0(\mathscr{C}_{\mathbb{Z}})$ \cite{HL15,Bit21}. We define a toric frame $M: \mathbb{Z}^{\widehat{I}}  \to \mathcal{Y}_{t}$ by 
\begin{align}\label{Our toric frame}
M\left(\sum_{(i,r)\in \widehat{I}} a_{i,r}e_{i,r}\right)=t^{\frac{1}{2}\sum_{(i,r)<(j,s)} a_{i,r} a_{j,s} \mathcal{N}_{ji}(s-r)} \overset{\to}{*}_{(i,r)\in \widehat{I}} Y^{a_{i,r}}_{i,r},
\end{align}
where $e_{i,r}$ is the standard basis vector of $\mathbb{Z}^{\widehat{I}}$, with 1 at $(i,r)$-position and 0 elsewhere, the product $ \overset{\to}{*}$ is ordered by a choice of an order on $\widehat{I}$ (notice that the result does not depend on the choice of the order).

Following Bittmann's notation \cite[Section 4.2]{Bit21}, we use the commutative monomial $\prod_{(i,r)\in \widehat{I}} Y^{a_{i,r}}_{i,r}$ to denote the expression in the left hand side of Equation (\ref{Our toric frame}). For each $(i,r)\in \widehat{I}$, let 
\[
A_{i,r-1}=Y_{i,r-2}Y_{i,r} \left(\prod_{j\sim i} Y_{j,r-1}\right)^{-1} \in \mathcal{Y}_{t}.
\]

For all $i\in I$, one can define $K_{i,t}$ as the subring of $\mathcal{Y}_{t}$ generated by ($(i,r), (j,s)\in \widehat{I}$)
\begin{align*}
Y_{i,r}(1+A^{-1}_{i,r+1}), \quad  Y_{j,s} \,\,(j\neq i). 
\end{align*}

Following \cite{Her03,Bit21}, for all $i\in I$, one defines a free $\mathcal{Y}_{t}$-modules 
\[
\mathcal{Y}^{l}_{t,i}:=\bigoplus_{(i,r)\in \widehat{I}} \mathcal{Y}_{t} \cdot S_{i,r},
\]
$\mathcal{Y}^{l}_{t,i}$ is a direct sum of $\widehat{I}$ copies of $\mathcal{Y}_{t}$, whose basis elements are denoted by $S_{i,r}$. Let $\mathcal{Y}_{t,i}$ be the quotient of $\mathcal{Y}^{l}_{t,i}$ by the left $\mathcal{Y}_{t}$-module generated by 
\[
Q_{i,r}:=A^{-1}_{i,r+1} S_{i,r+2} - t^2 S_{i,r}, \quad \text{for any } (i,r)\in \widehat{I}.
\]
It was shown in \cite[Lemma 4.3.1]{Bit21} that the module $\mathcal{Y}_{t,i}$ is free. Following \cite{Her03}, for all $i\in I$, there exists a $\mathbb{Z}[t^{\pm\frac{1}{2}}]$-linear map $S_{i,t}:\mathcal{Y}_{t} \to \mathcal{Y}_{t,i}$, which is a derivation and such that $K_{i,t}=\text{Ker}(S_{i,t})$. Let 
\[
K_t(\mathscr{C}_{\mathbb{Z}}):=\bigcap_{i\in I} K_{i,t}.
\]

Following \cite{Her04}, for a dominant monomial $m$, there exists a unique element $F_t(m)\in K_t(\mathscr{C}_{\mathbb{Z}})$ such that $m$ has multiplicity 1 and no other dominant monomial in $F_t(m)$. The quantum Grothendieck ring $K_t(\mathscr{C}_{\mathbb{Z}})$ has three distinguished bases: the classes of all simple modules, the classes of all standard modules, and all $F_t(m)$. Moreover, $K_t(\mathscr{C}_{\mathbb{Z}})$ is algebraically generated by the classes of all fundamental modules.

\subsection{The $(q,t)$-characters}

Given a dominant monomial $m\in \mathcal{P}_+$, write it as a commutative monomial in $\mathcal{Y}_t$: 
\[
m=\prod_{(i,r)\in \widehat{I}} Y^{a_{i,r}(m)}_{i,r}.
\]
Let $M(m)$ be the standard module associated to $m$, which is defined as a tensor product of fundamental modules $\bigotimes_{(i,r)\in \widehat{I}} (L(Y_{i,r}))^{\otimes a_{i,r}(m)}$. Following \cite{Bit21}, the $(q,t)$-character of $M(m)$, denoted by $[M(m)]_t$, is defined as follows:
\begin{align*}
[M(m)]_t = t^{\alpha(m)} \overset{\leftarrow}{*}_{r\in \mathbb{Z}} F_t \left(\prod_{i\in I} Y^{a_{i,r}(m)}_{i,r}\right) \in K_t(\mathscr{C}_{\mathbb{Z}}),
\end{align*}
where $\alpha(m)\in \frac{1}{2}\mathbb{Z}$ is such that $m$ occurs with multiplicity one in the expansion of $[M(m)]_t$ on the basis of the commutative monomials of $\mathcal{Y}_t$, and the product $ \overset{\leftarrow}{*}$ is ordered by decreasing $r\in \mathbb{Z}$. In particular, 
\[
[L(Y_{i,r})]_t=[M(Y_{i,r})]_t=F_t(Y_{i,r}),
\] 
and $[M(m)]_t$ reduces to $\chi_q(M(m))$ if $t=1$. The evaluation at $t=1$ of $[M(m)]_t$ recovers its $q$-character. 

As in \cite{Nak04}, there is a $\mathbb{Z}$-algebra anti-automorphism $\overline{\cdot}$ of $\mathcal{Y}_t$ defined by: for any $(i,r)\in \widehat{I}$, 
\[
\overline{t^{1/2}}=t^{-1/2}, \quad  \overline{Y_{i,r}}=Y_{i,r}.
\]
The anti-automorphism is called the bar-involution.

Nakajima proved the following theorem.
\begin{theorem}[{\cite{Nak04}}]
There exists a unique family $\{[L(m)]_t\}_{m\in \mathcal{P}_+}$ of elements of $K_t(\mathscr{C}_{\mathbb{Z}})$ such that for all $m\in \mathcal{P}_+$,
\begin{itemize}
\item[(1)] $\overline{[L(m)]_t}=[L(m)]_t$, 
\item[(2)] $[L(m)]_t = [M(m)]_t + \sum_{m'<m} a(m')[M(m')]_t$, where $a(m')\in t^{-1}\mathbb{Z}[t^{-1}]$, the order $m'<m$ is defined in (\ref{Nakajima partial order}).
\end{itemize} 
The evaluation at $t=1$ of the $(q,t)$-characters recovers the $q$-characters.
\end{theorem}

\subsection{Truncated quantum $T$-system}

For any $m\in \mathcal{P}^{\leq \xi}_{+,\ell}$, we denote by $[L(m)]^{\leq \xi}_t$ the Laurent polynomial obtained from $[L(m)]_t$ by deleting monomials involving variables $Y_{i,r}$, with $(i,r)\not\in I^{\leq\xi}_\ell$.

The quantum Grothendieck ring $K_t(\mathscr{C}^{\leq \xi}_\ell)$ of $\mathscr{C}^{\leq \xi}_\ell$ is defined as a subalgebra of $K_t(\mathscr{C}_{\mathbb{Z}})$, which is algebraically generated by the truncated $(q,t)$-characters $[L(Y_{i,r})]^{\leq \xi}_t$ of the fundamental modules $L(Y_{i,r})$, with $(i,r)\in I^{\leq\xi}_\ell$.

For any $(i,r)\in \widehat{I}$, and $k \in \mathbb{Z}_{\geq 1}$, let 
\[
W^{(i)}_{k,r}=L\left(\prod_{\ell=0}^{k-1}Y_{i,r+2\ell}\right)
\] 
be a Kirillov-Reshetikhin module.

We recall the quantum $T$-system \cite{Nak03,HL15}, which is a $t$-deformation of the $T$-system \cite{KNS94,Nak03,Her06}, as follows.

\begin{proposition}[{\cite{Bit21,Nak03,HL15}}]  \label{quantum T-system}
For any $(i,r)\in \widehat{I}$, and $k \in \mathbb{Z}_{\geq 1}$, the following quantum $T$-system equation holds in $K_t(\mathscr{C}_{\mathbb{Z}})$:
\begin{gather*}
[W^{(i)}_{k,r}]_t * [W^{(i)}_{k,r+2}]_t = t^{\alpha(i,k)} [W^{(i)}_{k-1,r+2}]_t * [W^{(i)}_{k+1,r}]_t + t^{\gamma(i,k)} \underset{i\sim j}{*} [W^{(j)}_{k,r+1}]_t,
\end{gather*}
where $\alpha(i,k)=-1+\frac{1}{2}(\widetilde{C}_{ii}(2k-1)+\widetilde{C}_{ii}(2k+1))$, $\gamma(i,k)=\alpha(i,k)+1$.
\end{proposition}

Since the truncated $(q,t)$-character map is an injective ring homomorphism, we obtain the following equation in $K_t(\mathscr{C}^{\leq \xi}_\ell)$. For any $(i,r)\in \widehat{I}^{\leq \xi}_\ell$, and $k \in \mathbb{Z}_{\geq 1}$,

\begin{gather}\label{truncated quantum T-system}
[W^{(i)}_{k,r}]^{\leq \xi}_t * [W^{(i)}_{k,r+2}]^{\leq \xi}_t = t^{\alpha(i,k)} [W^{(i)}_{k-1,r+2}]^{\leq \xi}_t * [W^{(i)}_{k+1,r}]^{\leq \xi}_t + t^{\gamma(i,k)} \underset{i\sim j}{*} [W^{(j)}_{k,r+1}]^{\leq \xi}_t,
\end{gather}
where $\alpha(i,k)$ and $\gamma(i,k)$ are defined as the ones in Proposition \ref{quantum T-system}.

\subsection{Quantum cluster algebra structure on $K_t(\mathscr{C}^{\leq \xi}_\ell)$}

We recall Hernandez and Leclerc's sequence of vertices \cite[Section 2.2.3]{HL16} in the quiver $\Gamma^{\leq \xi}$, where $\xi$ is a sink-source height function. Assume that $(i_1,i_2,\ldots,i_n)$ is an order on the columns of $\Gamma^{\leq \xi}$ such that $\xi(i_k)\geq \xi(i_l)$ if $k\leq l$. Hernandez and Leclerc's sequence $\mathscr{S}_{HL}$ is defined by reading each column, from top to bottom, in this order. At each step of this sequence, if applied to the cluster algebra (respectively, quantum cluster algebra) with initial quiver $\Gamma^{\leq \xi}$, the exchange relation (respectively, quantum exchange relation) is a $T$-system equation \cite[Section 3.2.3]{HL16} (respectively, quantum $T$-system equation \cite[Section 6.3]{Bit21}).

After identifying $z^{\leq \xi}_{t,(i,r)}$ with 
\begin{align}\label{identify z and Y}
\prod_{k\geq 0, \atop  r+2k\leq \xi(i)} Y_{i,r+2k},
\end{align} 
we obtain the following theorem.

\begin{theorem}\label{quantum Grothendieck ring admits a quantum cluster algebra structure} 
For  an arbitrary height function $\xi$, the truncated $(q,t)$-character map $[?]^{\leq \xi}_t: K_t(\mathscr{C}^{\leq \xi}_\ell) \to \mathcal{A}^{\leq \xi}_{t,\ell}$ induces a ring isomorphism from $K_t(\mathscr{C}^{\leq \xi}_\ell)$ to $\mathcal{A}^{\leq \xi}_{t,\ell}$, sending $[L(\prod_{k\geq 0, \atop  r+2k\leq \xi(i)} Y_{i,r+2k})]^{\leq \xi}_t$ to $z^{\leq \xi}_{t,(i,r)}$. In particular, $ K_t(\mathscr{C}^{\leq \xi}_\ell) $ is a quantum cluster algebra.
\end{theorem}

\begin{proof}
It follows from \cite[Theorem 5.2.4]{Bit21} that it is true for a sink-source height function $\xi$. Even though in our setting $\xi$ is an arbitrary height function, the   proof is essentially the same as Bittmann's proof \cite{Bit21}. We sketch it as follows. 

The inclusion $K_t(\mathscr{C}^{\leq \xi}_\ell)\subset \mathcal{A}^{\leq \xi}_{t,\ell}$ follows from the fact that for each vertex $(i,p)\in \widehat{I}^{\leq \xi}_\ell$, the truncated $(q,t)$-character $[L(Y_{i,p})]^{\leq \xi}_{t}$ is a quantum cluster variable in $\mathcal{A}^{\leq \xi}_{t,\ell}$ via first using the source sequence defined in Lemme \ref{mutation equivalent of Gamma}, and then using the sequence $\mathscr{S}_{HL}\cap \widehat{I}^{\leq \xi}_\ell$, as well as the truncated quantum $T$-system equation (\ref{truncated quantum T-system}). The effect of the mutation sequence on two fixed columns of $\Gamma^{\leq \xi}_\ell$ is the same as the effect of an iteration of the mutation sequence on the corresponding quiver of rank 2. Bittmann \cite[Proposition 6.3.1]{Bit21} proved that the truncated $(q,t)$-characters $[L(Y_{i,r})]^{\leq \xi}_t$ of the fundamental module $L(Y_{i,r})$, for each $(i,r)\in I^{\leq\xi}_\ell$, is in $\mathcal{A}^{\leq \xi}_{t,\ell}$.

The reverse inclusion uses the interpretation of the $(q,t)$-character homomorphism as an intersection of kernels of $t$-deformed screening operators, and the induction on quantum cluster exchange relations, which was done by Bittmann \cite[Proposition 6.4.2]{Bit21}.
\end{proof}

\begin{remark}
If $\ell=\infty$, then Theorem \ref{quantum Grothendieck ring admits a quantum cluster algebra structure} has been proved recently by Fujita, Hernandez, Oh and Oya \cite[Theorem 5.16]{FHOO23} for general type, not only ADE type. 
\end{remark}

In order to be consistent with the notation appearing in the below of Proposition \ref{dominant monomials determine q-characters}, we denote by $K_0(\mathscr{C}^{\leq \xi}_\ell)$ and $\mathcal{A}^{\leq \xi}_{\ell}$ the specializations of $K_t(\mathscr{C}^{\leq \xi}_\ell)$ and $\mathcal{A}^{\leq \xi}_{t,\ell}$ at $t=1$, respectively.

In the following, we give an example of the quantum Grothendieck ring of a subcategory in type $A_2$.

\begin{example}
Let $\xi(1,2)=(-1,0)\in \mathbb{Z}^2$ in type $A_2$. The quantum Grothendieck ring $K_t(\mathscr{C}^{\leq \xi}_2)$ is a quantum cluster algebra with initial data shown in Figure \ref{The initial quantum cluster}. Its skew-symmetric matrix and compatible matrix are computed as follows.

\begin{align*}
& B^{\leq \xi}_2 = \begin{pmatrix}
0  & 1  & 1 & -1    \\ 
-1 & 0 &  0  & 1     \\
-1  & 0 &  0  & 1    \\
1  & -1 &  -1  & 0     \\
0  & 1 &  0  & -1   \\
0  & -1 &  0  & 0 
\end{pmatrix}, \quad 
L^{\leq \xi}_2 = \begin{pmatrix}
0 &  1 & 1  & 0  & 2 & 2  \\
-1  & 0 &  -1 & -1  & 0 & 2   \\
-1  & 1  & 0 & 1   & 2 & 0   \\ 
0 & 1 &  -1  & 0  & 2 & 2   \\
-2  & 0 &  -2  & -2  & 0 & 0   \\
-2  & -2 &  0  & -2 & 0 & 0 
\end{pmatrix}.
\end{align*}
Here the rows of $\widetilde{B}$, as well as the rows and columns of $\widetilde{L}$, are indexed by 
\begin{gather*}
\{ [L(Y_{1,-1})]_t, [L(Y_{1,-3}Y_{1,-1})]_t, [L(Y_{2,0})]_t, [L(Y_{2,-2}Y_{2,0})]_t, [L(Y_{2,-4}Y_{2,-2}Y_{2,0})]_t,  [L(Y_{1,-5}Y_{1,-3}Y_{1,-1})]_t\},
\end{gather*}
and the columns of $\widetilde{B}$ are indexed by 
\[
\{ [L(Y_{1,-1})]_t, [L(Y_{1,-3}Y_{1,-1})]_t, [L(Y_{2,0})]_t, [L(Y_{2,-2}Y_{2,0})]_t \}.
\]

\begin{figure}[H]
\begin{minipage}{0.5\textwidth}
\begin{xy}
(30,50)*+{[L(Y_{2,0})]_t}="a";
(0,40)*+{[L(Y_{1,-1})]_t}="b"; 
(30,30)*+{[L(Y_{2,-2}Y_{2,0})]_t}="c";
(0,20)*+{[L(Y_{1,-3}Y_{1,-1})]_t}="d"; 
(30,10)*+{\fbox{$[L(Y_{2,-4}Y_{2,-2}Y_{2,0})]_t$}}="e";
(0,0)*+{\fbox{$[L(Y_{1,-5}Y_{1,-3}Y_{1,-1})]_t$}}="f"; 
{\ar "a";"b"};
{\ar "b";"c"};
{\ar "c";"d"};
{\ar "d";"e"};
{\ar "f";"d"};
{\ar "d";"b"};
{\ar "e";"c"};
{\ar "c";"a"};
\end{xy}
\end{minipage}
\caption{An initial quantum extended cluster of $K_t(\mathscr{C}^{\leq \xi}_2)$ in type $A_2$.} \label{The initial quantum cluster}
\end{figure}
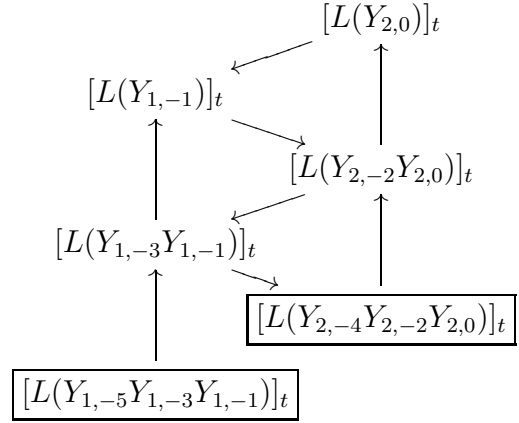
\end{example}

\section{On the classification of real simple modules in $\mathscr{C}^{\leq \xi}_\ell$} \label{the classification of real simple modules in Cl}
In this section, we establish connections between the additive categorifications of cluster algebras and the monoidal categorifications of cluster algebras. Any exchange relation appearing in additive categorifications of cluster algebras of type ADE is lifted to an exchange relation in monoidal categorifications of cluster algebras $\mathcal{A}^{\leq \xi}_1$, and, as a result, we completely  classify the real simple modules in $\mathscr{C}^{\leq \xi}_1$. Inspired by Hernandez and Leclerc's work \cite{HL10,HL16}, we propose two conjectures for the study of all the real simple modules in $\mathscr{C}^{\leq \xi}_\ell$.

In order to classify and study real simple modules in $\mathscr{C}^{\leq \xi}_\ell$, it is enough to study their $q$-characters. 

\subsection{Monoidal categorifications of cluster algebras} 

We recall the notion of a monoidal categorification  of a cluster algebra introduced by Hernandez and Leclerc \cite{HL10}.

\begin{definition}[{\cite[Definition 2.1]{HL10}}]
Let $\mathcal{A}$ be a cluster algebra and let $\mathscr{C}$ be an abelian monoidal category. The category $\mathscr{C}$ is a monoidal categorification of $\mathcal{A}$ if the Grothendieck ring of $\mathscr{C}$ is isomorphic to $\mathcal{A}$, and if 
\begin{itemize}
\item[(1)] the cluster monomials of $\mathcal{A}$ are the classes of all the real simple objects of $\mathscr{C}$;
\item[(2)] the cluster variables of $\mathcal{A}$ (including the frozen variables) are the classes of all the real prime simple objects of $\mathscr{C}$.
\end{itemize}
\end{definition}

The existence of a monoidal categorification of a cluster algebra implies the positivity of the cluster algebra and the linear independence of cluster monomials of the cluster algebra, see \cite[Proposition 2.2]{HL10}.

\begin{theorem}\label{monoidal categorification C1}
 The subcategory $\mathscr{C}^{\leq \xi}_1$ is a monoidal categorification of the cluster algebra $\mathcal{A}^{\leq \xi}_1$.
\end{theorem}

\begin{proof}
It follows from Theorem \ref{quantum Grothendieck ring admits a quantum cluster algebra structure} that the Grothendieck ring $K_0(\mathscr{C}^{\leq \xi}_1)$ of $\mathscr{C}^{\leq \xi}_1$ is isomorphic to the cluster algebra $\mathcal{A}^{\leq \xi}_1$. Following \cite{Qin17} or independently \cite{KKKO18},  under the isomorphism $\mathcal{A}^{\leq \xi}_1 \to K_0(\mathscr{C}^{\leq \xi}_1)$, the cluster monomials in $\mathcal{A}^{\leq \xi}_1$ are classes of real simple modules in $\mathscr{C}^{\leq \xi}_1$. Since $K_0(\mathscr{C}^{\leq \xi}_1)$ is isomorphic to a polynomial ring, applying \cite{GLS13}, the cluster variables in $\mathcal{A}^{\leq \xi}_1$ (including the frozen variables) are classes of real prime simple modules in $\mathscr{C}^{\leq \xi}_1$. 

In fact, the set of cluster (variables) monomials in $\mathcal{A}^{\leq \xi}_1$ is in bijection with the set of the classes of the (real prime) real simple modules in $\mathscr{C}^{\leq \xi}_1$. Indeed, since $\mathcal{A}^{\leq \xi}_1$ is a cluster algebra of finite type, the set of cluster monomials in $\mathcal{A}^{\leq \xi}_1$ forms a basis of $\mathcal{A}^{\leq \xi}_1$. Therefore the set of the images of cluster monomials in $\mathcal{A}^{\leq \xi}_1$ forms a basis of $K_0(\mathscr{C}^{\leq \xi}_1)$. On the other hand, the set of the classes of simple modules in $\mathscr{C}^{\leq \xi}_1$ gives a basis of $K_0(\mathscr{C}^{\leq \xi}_1)$. Hence 
\begin{align*}
\{ \text{the images of cluster monomials in $\mathcal{A}^{\leq \xi}_1$} \} &=\{ \text{the classes of real simple modules in $\mathscr{C}^{\leq \xi}_1$} \}  \\
& = \{ \text{the classes of simple modules in $\mathscr{C}^{\leq \xi}_1$} \}. 
\end{align*}
In particular, every simple module in $\mathscr{C}^{\leq \xi}_1$ is real. Suppose that $L(m)\in \mathscr{C}^{\leq \xi}_1$ is a real prime simple module. Then there is a unique cluster monomial $M$ in $\mathcal{A}^{\leq \xi}_1$ such that its image is $[L(m)]$. If $M$ is not a cluster variable in $\mathcal{A}^{\leq \xi}_1$, then a decomposition of $M$ as a cluster monomial implies a tensor decomposition of $[L(m)]$, which contradicts the primeness of $L(m)$. 
\end{proof}

\begin{remark} For  sink-source height functions $\xi$, Theorem \ref{monoidal categorification C1} was proved by Hernandez and Leclerc \cite{HL10} in type $A$ and type $D_4$, and  by Nakajima \cite{Nak03} in  types $A,D$ and $E$.
For linear height functions $\xi$, the theorem was proved by Hernandez and Leclerc \cite{HL13} 
in types $A$ and $D$.
For arbitrary height functions $\xi$, the theorem was proved by Brito and Chari \cite{BC19} in type $A$. Our proof for types $D$ and $E$  is new.
\end{remark}

Every real prime simple module in $\mathscr{C}^{\leq \xi}_1$ is called an \textit{Hernandez-Leclerc module} in honor of Hernandez and Leclerc's discovery on the connection between quantum affine algebras and cluster algebras. The terminology first appeared in \cite{BC19} for type $A$, and these modules of type $D$ and type $E$ are new.

It follows from Theorem \ref{monoidal categorification C1} that the set of cluster monomials in $\mathcal{A}^{\leq \xi}_1$ is in bijection with the set of  classes of  real simple modules in $\mathscr{C}^{\leq \xi}_1$. Every real simple module in $\mathscr{C}^{\leq \xi}_1$ is a tensor product of Hernandez-Leclerc modules from the same cluster in $K_0(\mathscr{C}^{\leq \xi}_1)$. Conversely, any tensor product of Hernandez-Leclerc modules from the same cluster in $K_0(\mathscr{C}^{\leq \xi}_1)$ is a real simple module. 

Moreover, Hernandez and Leclerc \cite{HL10,HL16} proposed a conjecture, which describes the simple objects of subcategories of $\mathscr{C}_{\mathbb{Z}}$ whose classes in the Grothendieck ring are cluster monomials. Inspired by this, we propose the following conjecture.

\begin{conjecture}\label{real simple objects are cluster monomials}
The classes of real simple modules in $\mathscr{C}^{\leq \xi}_\ell$ are cluster monomials in $K_0(\mathscr{C}^{\leq \xi}_\ell)$.
\end{conjecture} 

It follows from the proof of Theorem \ref{monoidal categorification C1} that Conjecture \ref{real simple objects are cluster monomials} holds for $\mathscr{C}^{\leq \xi}_1$.  By the same argument as in the proof of the case $\ell=1$, Conjecture \ref{real simple objects are cluster monomials} holds for 
arbitrary height functions $\xi$ in type $A_2$ if $\ell\le 4$, and in types $A_3$ and $A_4$ if $\ell=2$, because in these cases the Grothendieck rings of the subcategories are cluster algebras of finite type. Conjecture~ \ref{real simple objects are cluster monomials} is open for arbitrary height functions $\xi$ and other $\ell$.

For any simple module $L(m)\in \mathscr{C}^{\leq \xi}_\ell$, we denote by $\chi_q([L(m)])_{\leq \xi}$ the truncated $q$-character of $L(m)$, which is obtained from $\chi_q([L(m)])$ by keeping only the monomials in the variables $Y_{i,r}$, with $(i,r)\in \widehat{I}^{\leq \xi}_\ell$. It follows from \cite{FM01,HL16} that $\chi_q([L(m)])_{\leq \xi}$ can be written as
\begin{align} \label{truncated q-character foumula}
\chi_q([L(m)])_{\leq \xi}= m P_m,
\end{align}
where $P_m$ is a polynomial in the variables $\{A^{-1}_{i,r-1}|(i,r)\in \widehat{I}^{\leq \xi}_{\ell-1}\}$ with integer coefficients and constant term 1.

\subsection{The classification of real simple modules in $\mathscr{C}^{\leq \xi}_1$} \label{case l=1}
In this section, we  classify all the real simple modules in $\mathscr{C}^{\leq \xi}_1$ by establishing connections between the additive categorifications of cluster algebras of type ADE and the monoidal categorifications of cluster algebras $\mathcal{A}^{\leq \xi}_1$. 
 
Let $\xi$ be a fixed height function. Recall that $Q^{\leq \xi}_1$ is (isomorphic to) the principal quiver of $\Gamma^{\leq \xi}_1$ with vertex set $I$.
 Thus $Q^{\leq \xi}_1$ is a Dynkin quiver of type $\gamma$, and $\Gamma^{\leq \xi}_1$ contains two copies of $Q^{\leq \xi}_1$ as well as connecting arrows. See Figure~\ref{quiver Gamma in type A3} for an example.
For simplicity of notation, we denote by $Q$ (respectively, $\Gamma$, $\mathscr{C}$) the quiver $Q^{\leq \xi}_1$ (respectively, the quiver $\Gamma^{\leq \xi}_1$, the subcategory $\mathscr{C}^{\leq \xi}_1$), and denote by $\widetilde{\bf z}={\bf z} \cup {\bf z}'$ the set $\mathbf{z}^{\leq \xi}_1$ introduced in the  paragraph following Proposition \ref{dominant monomials determine q-characters}, where ${\bf z}=\{z_i \mid i\in I\}$ is the set of mutable variables and ${\bf z}'=\{z'_i \mid i\in I\}$ is the set of frozen variables. Under our new notation, when $\ell=1$, the identification (\ref{identify z and Y}) becomes $z_{i} \mapsto  Y_{i,\xi(i)}$, with $i\in I$.

Let $V \in \mathscr{C}$ be a simple module, and let $\textbf{hw}(V)$ denote the highest $l$-weight monomial in $\chi_q([V])$. It was shown in \cite{FM01} that, for any two simple modules $V_1$ and  $V_2$, we have 
\[
\textbf{hw}(V_1 \otimes V_2)=\textbf{hw}(V_1) \textbf{hw}(V_2).
\] 

Let $\mathcal{A}(Q)$ be the cluster algebra with initial quiver $Q$, and $\mathcal{C}$ the associated cluster category of \cite{BMRRT06}. For $i\in I$, let $f_i=Y_{i,\xi(i)-2}Y_{i,\xi(i)}$.
 The modules
\[
L(f_i)=L(Y_{i,\xi(i)-2}Y_{i,\xi(i)}), \, \text{with $i\in I$},
\]
are special Kirillov-Reshetikhin modules. It is well-known that Kirillov-Reshetikhin modules are real prime modules, see \cite{HL16}. 

We denote by $\mathcal{A}(\Gamma)$ (respectively, $K_0(\mathscr{C})$) the cluster algebra $\mathcal{A}(\Gamma)$ with initial extended cluster $(\widetilde{\bf z},\Gamma)$ (respectively, the Grothendieck ring of $\mathscr{C}$)  with the frozen vertices as those in $\Gamma\backslash Q$ (respectively, with the frozen vertices $\{[L(f_i)] \mid i\in I\}$ as a cluster algebra). The two cluster algebras have the same number of cluster variables, since they have the same principal quiver. It follows from \cite{BMRRT06,CCS06} that the number of cluster variables in $\mathcal{A}(Q)$ is the same as the number of indecomposable objects in $\mathcal{C}$, as well as the number of real prime simple modules excluding $\{ L(f_i) \mid i\in I\}$, because these $L(f_i)$ correspond to the frozen vertices,
in $\mathscr{C}$.

Suppose that $L(m)\in \mathscr{C}$ is a real prime simple module such that $\chi_q(L(m))_{\leq \xi}$ corresponds to an indecomposable $kQ$-module $M$. By Equations (\ref{separate formula for a cluster variable}) and (\ref{truncated q-character foumula}), as well as Fomin-Zelevinsky's separation of addition formula \cite{FZ07}, we have
\begin{align*}
\textbf{z}^{g(M)} \frac{F_M(\widehat{y}_1,\ldots,\widehat{y}_{n})|_{\mathcal{F}}}{F_M(y_1,\ldots,y_{n})|_{\mathbb{P}}}=mP_m,
\end{align*}
where $\mathbb{P}$ is the tropical semifield generated by $\{f_i \mid i \in I\}$, with tropical addition $\oplus$, and for any $j\in I$,
\[
y_i=f^{-1}_i\left( \prod_{j: j \to i} f_j \right), \text{ and } \widehat{y}_j=A^{-1}_{j,\xi(j)-1}. 
\] 

By the same argument as in \cite[Lemma 7.3]{HL10} and \cite[Proposition 4.16]{HL16}, we have
\begin{align}\label{g-vector and height weight and F-polynomial} 
\frac{\textbf{z}^{g(M)}}{F_M(y_1,\ldots,y_{n})|_{\mathbb{P}}}=m, \quad  F_M(\widehat{y}_1,\ldots,\widehat{y}_{n})|_{\mathcal{F}}=P_m.
\end{align}
In the sequel, we will use $F$-polynomials to compute renormalized truncated $q$-characters of simple modules.

We define a bijection $\Phi$ from the set of indecomposable (rigid) objects in $\mathcal{C}$ to the set of (real) prime simple modules in $\mathscr{C}$ as a composition of the  following two bijections:
\begin{gather} \label{a bijection Phi}
\xymatrix{ \left\{ \makecell{\text{Indecomposable (rigid)} \\ \text{objects in $\mathcal{C}$}} \right\}  \ar[rr]^{\Phi} \ar[dr]^-{\iota} && \left\{ \makecell{\text{(Real) prime simple} \\ \text{modules in $\mathscr{C}$}} \right\} \Big\backslash  \{ L(f_i) \mid i\in I \}  \\
 &  \left\{  \makecell{ \text{Cluster variables in the} \\ \text{cluster algebra $\mathcal{A}(\Gamma)$}} \right\} \ar[ur]^-{\widetilde{\Phi}} &} 
\end{gather}
where the map $\iota$ is defined by the cluster character or the CC map and Fomin-Zelevinsky's separation of addition  formula, and $\widetilde{\Phi}$ is the bijection induced by the isomorphism $\mathcal{A}(\Gamma) \cong K_0(\mathscr{C})$. Both the CC map and $\widetilde{\Phi}$ preserve clusters and cluster exchange relations in cluster algebras. We extend $\Phi$ to arbitrary rigid objects in $\mathcal{C}$ by 
\[
\Phi(M_1\oplus M_2)=\Phi(M_1) \otimes \Phi( M_2).
\]
Under the map $\Phi$, for any $i\in I$, $\Phi(P(i)[1])=L(Y_{i,\xi(i)})$.

Recall from Section \ref{cluster characters} that  an exchange pair in the cluster category $\mathcal{C}$ consists of  two indecomposable rigid objects $L$ and $N$ with $\text{dim}(\text{Ext}^1_{\mathcal{C}}(N,L))=1$. In this situation, there exists a unique pair of non-split exchange triangles 
\begin{align} \label{non-split triangles}
L\to M \to N\to L[1],  \quad  N \to M' \to L \to N[1]
\end{align}
in $\mathcal{C}$. 

\begin{lemma}\label{g-vector equation in an exchange pair}
 Let $(L,N)$ be an exchange pair in the cluster category of an acyclic quiver $Q$ with exchange triangles (\ref{non-split triangles}). Then either $g(M)=g(L)+g(N)$ or $g(M')=g(L)+g(N)$, and only one of two equations holds.
\end{lemma}
\begin{proof}
Clearly only one of the equations can hold, because $M\ncong M'$.
We prove this lemma in two cases.

{\bf Case 1.} Let $L,N\in \text{mod}\,kQ$. 
Following \cite{BMRRT06}, we have 
\begin{align*}
D\text{Ext}^1_{\mathcal{C}}(N,L) \cong \text{Hom}_{\mathcal{C}}(L,N[1]) = \text{Hom}_{\mathcal{D}}(L,N[1]) \oplus \text{Hom}_{\mathcal{D}}(L,\tau N).
\end{align*}
If $\text{dim}_k(\text{Hom}_{\mathcal{D}}(L,N[1]))=1$, then the second triangle in (\ref{non-split triangles}) is induced by  a non-split short exact sequence  $0 \to N \to M' \to L \to 0$ in $\text{mod}\,kQ$. Corollary \ref{g-vector equation associated an almost split sequence} implies that $g(M)=g(L)+g(N)$.

Otherwise, $\text{dim}_k(\text{Hom}_{\mathcal{D}}(L,\tau N))=1$, and by the Serre duality formula $\text{Hom}_{\mathcal{D}}(L,\tau N) \cong D\text{Hom}_{\mathcal{D}}(N,L[1])$, we have 
\[
\text{dim}_k(\text{Hom}_{\mathcal{D}}(N,L[1]))=1.
\]
By the same argument as before, we obtain $g(M)=g(L)+g(N)$. 

{\bf Case 2.} We assume without loss of generality that $L=P(i)[1]$ for some $i\in I$, $N\in \text{mod}\,kQ$. We have 
\begin{align*}
P(i)[1] \to M \to N \to P(i)[2].
\end{align*} 
Corollary \ref{g-vector equation associated an almost split sequence} implies $g(M)=g(L)+g(N)$, as desired.
\end{proof}

\begin{remark} \label{remark 5.6}
In fact, Lemma~\ref{g-vector equation in an exchange pair} holds for any skew-symmetrizable not necessarily acyclic cluster algebra. More precisely, it follows from \cite[Corollary 5.5]{GHKK18} and \cite[Proposition~1.3]{NZ12} that either $g(M)=g(L)+g(N)$ or $g(M')=g(L)+g(N)$, where $x_L x_N=x_{M} + x_{M'}$ is  an exchange relation in the cluster algebra.
\end{remark}

We now assume without loss of generality that $g(M)=g(L)+g(N)$ (exchanging $L$ and $N$ if necessary) in the non-split exchange triangles (\ref{non-split triangles}). By the proof of Lemma \ref{g-vector equation in an exchange pair} this implies that $N\in \text{mod}\,kQ$.

The following result was shown by Hubery in an unpublished paper \cite[Lemma 13]{Hub06}. For convenience we include his proof here.

\begin{lemma} \label{lem Hubery}
Let $(L,N)$ be an exchange pair in $\mathcal{C}$ with $N\in \textup{mod}\,kQ$ and with  exchange triangles (\ref{non-split triangles}). Let $h\colon\tau^{-1}_{\mathcal{C}} L= L[-1] \to N$ be the morphism in the second triangle in (\ref{non-split triangles}).  Then
\begin{align}\label{Hubery's a result}
M'\cong \tau_{\mathcal{C}}\,\text{Ker}(h) \oplus \text{Coker}(h).
\end{align} 
\end{lemma}

\begin{proof}
Let $(L,N)$ be an exchange pair in $\mathcal{C}$ with $N\in \textup{mod}\,kQ$ and  $L$ arbitrary indecomposable. Thus either $L \in \text{mod}\,kQ$ or $L=P(i)[1]$ for some $i\in I$. In the former case, the first exchange triangle in (\ref{non-split triangles}) is induced by a short exact sequence $0\to L\to M\to N\to 0$  in $\textup{mod}\,kQ$. In particular,  $L$ is not injective and hence $\tau^{-1}L\in \textup{mod}\,kQ$. On the other hand, if $L=P(i)[1]$ for some $i\in I$, then $\tau^{-1}_{\mathcal{C}} L=P(i)$. Thus in both cases, we have that $\tau^{-1}L$ or $\tau^{-1}_{\mathcal{C}} L$ and the morphism $h$ are in $ \textup{mod}\,kQ$.

Let $K=\ker h$ and $C=\text{Coker}\,h$. Then there is an exact sequence
\[\xymatrix{0\ar[r]&K\ar[r]^i& \tau^{-1}L\ar[r]^h&N\ar[r]^p&C\ar[r]&0}.
\]
This yields two short exact sequences
\begin{align*}
 \eta\colon& \xymatrix{0\ar[r]&K\ar[r]^i&\tau^{-1}L\ar[r]^-{h_1}& X\ar[r]&0}\\
  \eta'\colon& \xymatrix@C28pt{0\ar[r]&X\ar[r]^{h_2}&N\ar[r]^{p}& C\ar[r]&0}
\end{align*}
with $h=h_2\circ h_1$.
In the derived category $\mathcal{D}$. We therefore have the following triangles
\begin{align*}
&  \xymatrix{\tau^{-1}L\ar[r]^-{h_1}& X\ar[r]^\eta&K[1]\ar[r]&\tau^{-1}L[1]}\\
 &\xymatrix@C39pt{X\ar[r]^{h_2}&N\ar[r]^{p}& C\ar[r]^{\eta'}&X[1]}\\
 &\xymatrix@C35pt {\tau^{-1}L\ar[r]^-{h}&N\ar[r] & M'\ar[r]&  L},
\end{align*}
where the first two are induced by $\eta$ and $\eta'$ and the third is the second triangle in (\ref{non-split triangles}). 
Applying the octahedral axiom of triangulated categories, we obtain the following triangle
in $\mathcal{D}$
\[\xymatrix@C35pt{K[1]\ar[r]&M'\ar[r]&C\ar[r]^-{\eta[1]\,\circ\,\eta'}&K[2].}\]
Observe that ${\eta[1]\,\circ\,\eta'}=0$, because $C,K\in \textup{mod}\,kQ$ and $kQ $ is hereditary. Therefore the last triangle splits, and we obtain $M'\cong K[1]\oplus C$ in $\mathcal{D}$. Thus in the cluster category, we have $M'\cong \tau_{\mathcal{C}} K\oplus C$, as desired.
\end{proof}

It follows from \cite{CK06,FK10} that  
\begin{align} \label{exchange relation associated to two non-split triangles}
X_L X_{N}=X_M + \textbf{y}^{\alpha} X_{M'},
\end{align} 
where $\alpha\in \mathbb{Z}_{\geq 0}^{I}$ is the dimension vector of the image of $h\colon \tau^{-1}_{\mathcal{C}} L \to N$, and $\textbf{y}^{\alpha}$ is a monomial in principal coefficient variables.

\begin{theorem}\label{main theorem2}
 Let $(L,N)$ be an exchange pair in $\mathcal{C}$ with $N\in \textup{mod}\,kQ$ and with exchange triangles (\ref{non-split triangles}). Then there is an exact sequence in $\mathscr{C}$
\begin{align}\label{Gamma1 exact sequence a}
0 \to \Phi(M') \bigotimes   \left(\otimes_{j\in I}L(f_j)^{\otimes d_j}\right)  \to \Phi(L) \otimes \Phi(N) \to \Phi(M) \bigotimes \left(\otimes_{i\in I} L(f_i)^{\otimes c_i} \right)  \to 0,
\end{align}
or the same sequence with arrows reversed, where all $c_i$ and $d_j$ satisfy the following conditions:
\begin{itemize}
\item[(1)] The vector $(c_i)_{i\in I}$ is such that 
\begin{align} \label{determine variables fi associated a general exchange pairs}  
\prod_{i \in I} f_i^{c_i} = \frac{F_{M}(y_1,\ldots,y_{n})|_{\mathbb{P}} }{ F_{M}(y_1,\ldots,y_{n})|_{\mathbb{P}} \oplus \left( \left( \prod_{i\in I} y^{\alpha_i}_i \right) F_{M'}(y_1,\ldots,y_{n})\right)|_{\mathbb{P}} },
\end{align}
where $\mathbb{P}$ is the tropical semifield generated by $f_i$, with $i \in I$, and $F_{M'}$ is a polynomial in variables $y_i|_{\mathbb{P}}=f^{-1}_i\left( \prod_{j:j \to i} f_j \right)$ $(i\in I)$.

\item[(2)] The vector $(d_i)_{i\in I}$ is such that  
\begin{align}\label{determine vectors dj associated to a general exchange pair} 
\prod_{j \in I} f_j^{d_j}= \frac{\textbf{hw}(\Phi(L))\textbf{hw}(\Phi(N))}{\textbf{hw}(\Phi(M'))}  \prod_{i \in I} A^{-\alpha_i}_{i,\xi(i)-1}.
\end{align}
\end{itemize}
In particular, we have 
\begin{align*} 
\textbf{hw}(\Phi(L))\textbf{hw}(\Phi(N))=\textbf{hw}(\Phi(M)) \left(\prod_{i \in I} f^{c_i}_i\right).
\end{align*}
\end{theorem}

\begin{proof}
 We will first show that the existence of the exact sequence (\ref{Gamma1 exact sequence a}) or the same sequence with arrows reversed follows from the following conditions.
\begin{itemize}
\item[(i)] In $K_0(\mathscr{C})$, there is an equation:
\begin{align}\label{Gamma1 equation2} 
[\Phi(L)] [\Phi(N)] = [\Phi(M)] \left(\prod_{i \in I} [L(f_i)]^{c_i} \right) + [\Phi(M')] \left(\prod_{j \in I} [L(f_j)]^{d_j} \right),
\end{align}
where $\Phi(M)=\prod_{i\in I} \Phi(E_i)^{b_i}$ if $M=\bigoplus_{i\in I} E^{\oplus b_i}_i$ with $E_i$ indecomposable and $b_i\geq 0$.

\item[(ii)] The modules $\Phi(M') \bigotimes   \left(\otimes_{j\in I}L(f_j)^{\otimes d_j}\right)$ and $\Phi(M) \bigotimes \left(\otimes_{i\in I} L(f_i)^{\otimes c_i}\right)$ are simple.

\item[(iii)] The highest $l$-weight monomial of $\Phi(M) \bigotimes \left(\otimes_{i\in I} L(f_i)^{\otimes c_i}\right)$ is the same with the highest $l$-weight monomial of $\Phi(L) \Phi(N)$, that is,  
\begin{align}\label{proof formula3}
\textbf{hw}(\Phi(L)) \textbf{hw}(\Phi(N))=\textbf{hw}(\Phi(M)) \prod_{i\in I} f_i^{c_i}.
\end{align}
\end{itemize}

Assume conditions (i)-(iii) hold. We want to show that there is an exact sequence as in the theorem. By definition of $\Phi$, the modules $\Phi(L)$ and $\Phi(N)$ are real prime simple modules. Because of conditions (i) and (ii),  the composition factors of $\Phi(L) \otimes \Phi(N)$ are $\Phi(M') \bigotimes  \left(\otimes_{j\in I}L(f_j)^{\otimes d_j}\right)$ and $\Phi(M) \bigotimes \left(\otimes_{i\in I} L(f_i)^{\otimes c_i} \right)$. So $\Phi(M) \bigotimes \left(\otimes_{i\in I} L(f_i)^{\otimes c_i} \right)$ is the unique simple top or submodule of $\Phi(L) \otimes \Phi(N)$ with the same highest $l$-weight monomial as $\Phi(L) \otimes \Phi(N)$ by condition (iii). Condition (i) implies that the remainder term $\mathcal{M}$ in the exact sequence must have the $q$-character of $\Phi(M') \bigotimes \left(\otimes_{j\in I}L(f_j)^{\otimes d_j}\right)$. Since $\Phi(M')\bigotimes \left(\otimes_{j\in I}L(f_j)^{\otimes d_j}\right)$ is simple by condition (ii), we have 
\begin{align*}
\mathcal{M} \cong \Phi(M') \prod_{j \in I} L(f_j)^{d_j}.
\end{align*}

So it suffices to show conditions (i)-(iii). Because of Theorem \ref{monoidal categorification C1}, the modules corresponding to cluster monomials in $\mathcal{A}(\Gamma)$ are real simple objects in $\mathscr{C}$. Hence $\Phi(L)$, $\Phi(M)$, $\Phi(M')$ and $\Phi(N)$ are real simple modules. Because $\{[L(f_i)] \mid i\in I\}$ is the set of frozen variables in $K_0(\mathscr{C})$, which belongs to each extended cluster, by Theorem \ref{monoidal categorification C1}, $\Phi(M')\bigotimes \left(\otimes_{j\in I}L(f_j)^{\otimes d_j}\right)$ and $\Phi(M) \bigotimes \left(\otimes_{i\in I} L(f_i)^{\otimes c_i}\right)$ are simple.

Next we show conditions (i) and (iii).

Every  $L(f_i)$ is a certain Kirillov-Reshetikhin module, so by the Frenkel-Mukhin algorithm \cite{FM01}, 
\[
\chi_{q}([L(f_i)])_{\leq \xi}=f_i=\textbf{hw}(L(f_i)).
\] 
Equation (\ref{Gamma1 equation2}) holds if and only if (by Proposition \ref{dominant monomials determine q-characters})
\[
\chi_{q}([\Phi(L)])_{\leq \xi} \chi_{q}([\Phi(N)])_{\leq \xi} = \chi_{q}([\Phi({M})])_{\leq \xi} \left(\prod_{i \in I} f^{c_i}_i\right) + \chi_{q}([\Phi({M'})])_{\leq \xi} \prod_{j \in I} f_j^{d_j}
\]
if and only if (by Equation (\ref{g-vector and height weight and F-polynomial}))
\[
\textbf{hw}(\Phi(L)) \textbf{hw}(\Phi(N))  F_{L} F_{N}= \textbf{hw}(\Phi({M})) F_M \left(\prod_{i \in I} f^{c_i}_i\right)  +  \textbf{hw}(\Phi({M'})) F_{M'} \left(\prod_{j \in I} f^{d_j}_j\right)
\]
if and only if 
\begin{align}\label{proof formula2}
F_{L} F_{N}=\frac{  \textbf{hw}(\Phi({M})) \left(\prod_{i \in I} f^{c_i}_i\right) }{ \textbf{hw}(\Phi(L)) \textbf{hw}(\Phi(N))} F_M   + \frac{ \textbf{hw}(\Phi({M'})) \left(\prod_{j \in I} f_j^{d_j}\right) }{\textbf{hw}(\Phi(L)) \textbf{hw}(\Phi(N))} F_{M'},
\end{align}
where each $F_M$ is a polynomial in variables $\widehat{y}_i:=A^{-1}_{i,\xi(i)-1}$ $(i\in I)$ with constant term $1$, and with positive integer coefficients, as well as a unique maximum degree term $\prod_{i\in I} \widehat{y}^{\text{dim}\,M_i}_i$ with coefficient 1. 
 The coefficient of $F_{M'}$ on the right hand side of Equation~(\ref{proof formula2}) is equal to $\mathbf{\widehat{y}}^\alpha$, by condition (2) of the theorem.

 By Equation (\ref{exchange relation associated to two non-split triangles}),
 we have 
\begin{align} \label{F-polynomials associated almost-split exact sequence2}
F_{L} F_N = F_M +  \left( \prod_{i\in I} \widehat{y}^{\alpha_i}_i \right) F_{M'}.
\end{align}
So we only need to show that the coefficient of $F_M$ in Equation (\ref{proof formula2}) is equal to 1. This is equivalent to proving Equation~(\ref{proof formula3}).

By Equation (\ref{g-vector and height weight and F-polynomial}), we have 
\begin{align*} 
& \textbf{hw}(\Phi(L))=\frac{\textbf{z}^{g(L)}}{F_{L}(y_1,\ldots,y_{n})|_{\mathbb{P}}}, \\
& \textbf{hw}(\Phi(N))=\frac{\textbf{z}^{g(N)}}{F_{N}(y_1,\ldots,y_{n})|_{\mathbb{P}}}, \\
& \textbf{hw}(\Phi({M})) = \frac{\textbf{z}^{g( M)}}{F_{M}(y_1,\ldots,y_{n})|_{\mathbb{P}}}, 
\end{align*}
where $\mathbb{P}$ is the tropical semifield generated by $y_i|_{\mathbb{P}}=f^{-1}_i\left( \prod_{j: j \to i} f_j \right)$, with $i \in I$. Hence 
 the left hand side of (\ref{proof formula3}) is equal to
\begin{align*}
\textbf{hw}(\Phi(L)) \textbf{hw}(\Phi(N)) & = \frac{\textbf{z}^{g(L)}\textbf{z}^{g(N)}}{ F_{L}(y_1,\ldots,y_{n})|_{\mathbb{P}} F_{N}(y_1,\ldots,y_{n})|_{\mathbb{P}}} \\
& = \frac{ \textbf{z}^{g(M)}}{F_{M}(y_1,\ldots,y_{n})|_{\mathbb{P}} \oplus \left( \left( \prod_{i\in I} y^{\alpha_i}_i \right) F_{M'}(y_1,\ldots,y_{n})\right)|_{\mathbb{P}} },
\end{align*}
where the last equation is obtained from Equation (\ref{F-polynomials associated almost-split exact sequence2}), and the equation of $g$-vectors 
\[
g(L)+g(N)=g(M).
\]

On the other hand, the right hand side of Equation (\ref{proof formula3}) is
\[
\textbf{hw}(\Phi(M)) \left(\prod_{i \in I} f^{c_i}_i\right)  = \frac{\mathbf{z}^{g(M)}}{F_M(y_1,\ldots,y_{n})|_{\mathbb{P}}}  \left( \prod_{i\in I} f_i^{c_i} \right).
\]
Thus Equation (\ref{proof formula3}) holds by condition (1) of the theorem. This shows Equations (\ref{Gamma1 equation2}) and (\ref{proof formula3}), and the proof is complete.\end{proof}

Equation (\ref{Gamma1 equation2}) gives all the exchange relations in $K_0(\mathscr{C})$, since $(L,N)$ is an exchange pair if and only if 
\[
\text {dim(Ext}^1_{\mathcal{C}}(L,N))=\text{dim(Ext}^1_{\mathcal{C}}(N,L))=1,
\] 
see Proposition \ref{characteristic of exchange pair}.

\subsubsection{ {The vector $c$}  }
In this section, we study a representation theoretic interpretation of the vector $(c_i)_{i\in I}$ in Equation (\ref{Gamma1 equation2}). We need a preparatory lemma.

\begin{lemma}\label{tropical F-polynomials}
Let $M$ be a $kQ$-module. Then
\[
F_M(y_1,\ldots,y_{n})|_{\mathbb{P}}={\bf f}^{-\underline{\text{dim}}(\text{soc}(M))},
\]
where $\text{soc}(M)$ is the socle of $M$ generated by all simple submodules of $M$, and  ${\bf f}^{d}=\prod_{i\in I}f^{d_i}_i$ for any $d\in \mathbb{Z}^I$.
\end{lemma}
\begin{proof}

Recall that $y_i|_{\mathbb{P}}=f^{-1}_i\left( \prod_{j: j \to i} f_j \right)$. If $\text{soc}(M)= S(i_1) \oplus \cdots \oplus S(i_k)$, then 
\[
F_M(y_1,\ldots,y_{n})= 1+ y_{i_1}+ \cdots + y_{i_k} + \text{higher order terms}.
\]
so $F_M(y_1,\ldots,y_{n})|_{\mathbb{P}} = f_{i_1}^{-1} \cdots f_{i_k}^{-1} \oplus 1 \oplus (\text{higher order terms})|_{\mathbb{P}}$. This shows that ${\bf f}^{-\underline{\text{dim}}(\text{soc}(M))}$ is a factor of $F_M(y_1,\ldots,y_{n})|_{\mathbb{P}}$.

Suppose first that $M=I(x)$ is an indecomposable injective $kQ$-module. For every vertex $j$ in $Q$, the dimension of the vector space $I(x)_j$ is equal to the number of paths from $j$ to $x$ in $Q$. Since $Q$ is a Dynkin quiver, there is at most one such path, hence $\underline{\dim}(I(x)_j) \le 1$. 

Assume now there exists a simple module $S(j)$ that is not a summand of $\text{soc}(I(x))$, such that the $F$-polynomial $F_{I(x)}$ contains a monomial $m$ that has a nontrivial factor $y_j^{a_j}$. By definition of the $F$-polynomial, there exists a submodule $M'\subset I(x)$ such that $\underline{\dim} (M'_j)=a_j$ and $M'_j\subset I(x)_j$. Thus  $a_j=1$ and there exists a path $j\to j_1\to\ldots\to j_t=x$. Since $M'$ is a submodule, it follows that $M'_{j_1}$ is also nonzero, and hence of dimension one. In particular, the monomial $m$ contains the factor $y_{j_1}$ as well. 
Then the specialization
\[
y_j y_{j_1}|_{\mathbb{P}} = f_j^{-1} \left( \prod_{h: h\to j}f_h \right) f_{j_1}^{-1} f_j \left( \prod_{h: h\to j_1,h\ne j}f_h \right)
\]
does not contain $f_j$ as a factor.
So in this case, our result is true.

Now suppose that $M$ is an indecomposable non-injective $kQ$-module. Then there exists a unique injective envelope $I(M)$ of $M$ such that
\[
F_{I(M)}(y_1,\ldots,y_n)|_{\mathbb{P}}={\bf f}^{-\underline{\text{dim}}(\text{soc}(I(M)))}={\bf f}^{-\underline{\text{dim}}(\text{soc}(M))}.
\] 
Since every submodule of $M$ is also a submodule of $I(m)$, we have that $F_M(y_1,\ldots,y_n)|_{\mathbb{P}}$ is a factor of ${\bf f}^{-\underline{\text{dim}}(\text{soc}(M))}$. 

Every $kQ$-module is a direct sum of certain indecomposable $kQ$-modules. Since $F_{M\oplus N}=F_{M}F_{N}$ and $\text{soc}(M\oplus N)=\text{soc}(M)\oplus \text{soc}(N)$ for two indecomposable modules $M$ and $N$, we have 
\[
F_{M\oplus N}(y_1,\ldots,y_n)|_{\mathbb{P}}=F_{M}(y_1,\ldots,y_n)|_{\mathbb{P}} F_{N}(y_1,\ldots,y_n)|_{\mathbb{P}} = {\bf f}^{-\underline{\text{dim}}(\text{soc}(M\oplus N))}.
\]
This completes the proof.
\end{proof}

\begin{corollary}\label{a corollary of tropical F-polynomials}
Let $M$ be a $kQ$-module. Then
\[
{\bf f}^{-\underline{\text{dim}}(\text{soc}(M))} = (1+ {\bf y}^{\underline{\text{dim}}(M)})|_{\mathbb{P}}.
\]
\end{corollary}
\begin{proof}
Suppose that $\text{soc}(M)= S(i_1) \oplus \cdots \oplus S(i_k)$. Since $S(i_l)$, with $1\leq l \leq k$, is a simple submodule of $M$, there are no morphisms outgoing from the index $i_l$ contributed by $S(i_l)$. This implies that the factor $f^{-1}_{i_l}$ in $1+ \textbf{y}^{\underline{\text{dim}}(M)}$ cannot be cancelled. Hence ${\bf f}^{-\underline{\text{dim}}(\text{soc}(M))}$ is a factor in $(1+ \textbf{y}^{\underline{\text{dim}}(M)})|_{\mathbb{P}}$. On the other hand, $(1+ \textbf{y}^{\underline{\text{dim}}(M)})|_{\mathbb{P}}$ is a factor in $F_M(y_1,\ldots,y_{n})|_{\mathbb{P}}={\bf f}^{-\underline{\text{dim}}(\text{soc}(M))}$. Our result follows.
\end{proof}

Let $M,N \in \mathcal{C}$. We use the  convention that $\text{soc}(P(j)[1])=0$ for $j\in I$. We define non-negative integers $n_j,m_j$ by $\text{soc}(M)= \oplus_{j\in I} S(j)^{m_j}$ and $\text{soc}(N)=\oplus_{j\in I} S(j)^{n_j}$ and let $\underline\delta (N,M)
=(\max(n_j-m_j,0))_{j\in I}.$

As a direct consequence of Lemma \ref{tropical F-polynomials}, we obtain a formula for the vector $(c_i)_{i\in I}$.

\begin{theorem}\label{equivalent descriptions of c}
In Equation (\ref{Gamma1 equation2}), the vector  $(c_i)_{i\in I}$ is 
\begin{align}\label{a formula of the vector ci in general case}
(c_i)_{i\in I} = \underline{\delta}(N\oplus \tau_{\mathcal{C}} \,\text{Ker}(h),M) \geq 0.
\end{align} 
In particular, if $L=\tau\,N$, 
\begin{align*}
(c_i)_{i\in I}=\underline{\delta}(N,M) \geq 0.
\end{align*}
 Moreover the following are equivalent:
\begin{itemize}
\item [(1)] The vector $(c_i)_{i\in I}$ is 0. 

\item [(2)] $F_{M}(y_1,\ldots,y_{n})|_{\mathbb{P}}=F_{M}(y_1,\ldots,y_{n})|_{\mathbb{P}} \oplus \left( {\bf y}^{\alpha} F_{M'}(y_1,\ldots,y_{n})\right)|_{\mathbb{P}}.$

\item [(3)] $\text{soc}(N\oplus \tau_{\mathcal{C}}\,\text{Ker}(h)) \text{ is a summand of } \text{soc}(M).$

\item [(4)] $\text{soc}(M) = \text{soc}(N) \oplus \text{soc}(L)$.

\item [(5)] $\textbf{hw}(\Phi(N)) \text{ is a factor in } \textbf{hw}(\Phi(M)).$
\end{itemize}
\end{theorem}

\begin{proof}
It follows from Equation (\ref{Hubery's a result}) and the vector $\alpha$ in Equation (\ref{exchange relation associated to two non-split triangles}) that 
\begin{align*}
\textbf{y}^{\alpha} F_{M'}(y_1,\ldots,y_{n}) = \textbf{y}^{\alpha} F_{\text{Coker}(h)}(y_1,\ldots,y_{n}) F_{\tau_{\mathcal{C}}\,\text{Ker}(h)}(y_1,\ldots,y_{n}), 
\end{align*}
and $\textbf{y}^{\alpha} F_{\text{Coker}(h)}(y_1,\ldots,y_{n})$ has a unique maximum degree term $\textbf{y}^{\underline{\text{dim}}(N)}$.

The set of submodules of $\text{Coker}(h)$ is bijection with the set of submodules of $N$ containing $\text{Im}(h)$. Every term in $1+\textbf{y}^{\alpha} F_{M'}(y_1,\ldots,y_{n})$ appears as a term in $F_{N \oplus \tau_{\mathcal{C}}\,\text{Ker}(h)}(y_1,\ldots,y_{n})$. The tropicalization does not depend on the integer coefficients, and so 
\[
(1+\textbf{y}^{\alpha} F_{M'}(y_1,\ldots,y_{n}))|_{\mathbb{P}}
\] 
is a factor of $F_{N \oplus \tau_{\mathcal{C}}\,\text{Ker}(h)}|_{\mathbb{P}}(y_1,\ldots,y_n)={\bf f}^{-\underline{\text{dim}}(\text{soc}(N \oplus \tau_{\mathcal{C}}\,\text{Ker}(h)))}$. On the other hand, by Corollary \ref{a corollary of tropical F-polynomials}, 
\begin{align*}
(1+ \textbf{y}^{\underline{\text{dim}}(N\oplus \tau_{\mathcal{C}}\,\text{Ker}(h))})|_{\mathbb{P}}={\bf f}^{-\underline{\text{dim}}(\text{soc}(N \oplus \tau_{\mathcal{C}}\,\text{Ker}(h)))}
\end{align*}
is a factor of $(1+\textbf{y}^{\alpha} F_{M'}(y_1,\ldots,y_{n}))|_{\mathbb{P}}$. Hence 
\[
(1+\textbf{y}^{\alpha} F_{M'}(y_1,\ldots,y_{n}))|_{\mathbb{P}}={\bf f}^{-\underline{\text{dim}}(\text{soc}(N \oplus \tau_{\mathcal{C}}\,\text{Ker}(h)))}.
\] 
Note that the $F$-polynomial $F_{M}$ has a constant term 1. We have 
\begin{align}\label{the equivalence of (b) and (c)}
\begin{split}
& F_{M}(y_1,\ldots,y_{n})|_{\mathbb{P}} \oplus \left( \textbf{y}^{\alpha} F_{M'}(y_1,\ldots,y_{n})\right)|_{\mathbb{P}}   \\
& = F_{M}(y_1,\ldots,y_{n})|_{\mathbb{P}} \oplus \left(1\oplus  \textbf{y}^{\alpha} F_{M'}(y_1,\ldots,y_{n})\right)|_{\mathbb{P}} \\
& ={\bf f}^{\min\{-\underline{\text{dim}}(\text{soc}(M)), -\underline{\text{dim}}(\text{soc}(N \oplus \tau_{\mathcal{C}}\,\text{Ker}(h)))\}}.
\end{split}
\end{align}
Here $\min$ is taken in terms of components of vectors.

By Equation (\ref{determine variables fi associated a general exchange pairs}), and Lemma \ref{tropical F-polynomials}, we have
\begin{align}\label{the equivalence of (a) and (c)}
\begin{split}
\prod_{i\in I} f_i^{c_i} & = \frac{ F_{M}(y_1,\ldots,y_{n})|_{\mathbb{P}} }{ F_{M}(y_1,\ldots,y_{n})|_{\mathbb{P}} \oplus \left( \left( \prod_{i\in I} y^{\alpha_i}_i \right) F_{M'}(y_1,\ldots,y_{n})\right)|_{\mathbb{P}} } \\
& = \frac{{\bf f}^{-\underline{\text{dim}}(\text{soc}(M))}}{{\bf f}^{\min\{-\underline{\text{dim}}(\text{soc}(M)),-\underline{\text{dim}}(\text{soc}(N \oplus \tau_{\mathcal{C}}\,\text{Ker}(h)))\}}} \\
& = {\bf f}^{\max\{\underline{\text{dim}}(\text{soc}(M)),\underline{\text{dim}}(\text{soc}(N\oplus \tau_{\mathcal{C}}\,\text{Ker}(h)))\}-\underline{\text{dim}}(\text{soc}(M))} \\
& ={\bf f}^{\underline{{\delta}}(N \oplus \tau_{\mathcal{C}}\,\text{Ker}(h),M)},
\end{split}
\end{align}
where we agree that $\text{soc}(P(i)[1])=0$ for any $i\in I$. So Equation (\ref{a formula of the vector ci in general case}) follows.

The remainder of the proof is devoted to proving the equivalence of statements (1)--(5). The equivalence of the items (1) and (3) follows from Equation (\ref{the equivalence of (a) and (c)}). The item (2) holds if and only if (by Equation (\ref{the equivalence of (b) and (c)}) and Lemma \ref{tropical F-polynomials})
\[
{\bf f}^{-\underline{\text{dim}}(\text{soc}(M))} = {\bf f}^{\min\{-\underline{\text{dim}}(\text{soc}(M)), -\underline{\text{dim}}(\text{soc}(N \oplus \tau_{\mathcal{C}}\,\text{Ker}(h)))\}},
\]
if and only if
\[
\text{soc}(N\oplus \tau_{\mathcal{C}}\,\text{Ker}(h)) \text{ is a summand of } \text{soc}(M).
\]  

By Equation (\ref{g-vector and height weight and F-polynomial}) and Lemma \ref{tropical F-polynomials}, we have 
\begin{align}\label{hight weight equation of M, N, and L}
\begin{split}
\textbf{hw}(\Phi(M)) & = \frac{\textbf{z}^{g(M)}}{F_{M}(y_1,\ldots,y_{n})|_{\mathbb{P}}} =  \textbf{z}^{g(L)+g(N)} {\bf f}^{\underline{\text{dim}}(\text{soc}(M))}, \\
\textbf{hw}(\Phi(N)) & = \frac{\textbf{z}^{g(N)}}{F_{N}(y_1,\ldots,y_{n})|_{\mathbb{P}}}  =  \textbf{z}^{g(N)} {\bf f}^{\underline{\text{dim}}(\text{soc}(N))},  \\
\textbf{hw}(\Phi(L)) & = \frac{\textbf{z}^{g(L)}}{F_{L}(y_1,\ldots,y_{n})|_{\mathbb{P}}} = \textbf{z}^{g(L)} {\bf f}^{\underline{\text{dim}}(\text{soc}(L))}.
\end{split}
\end{align} 
It follows from Equation (\ref{proof formula3}) that the item (1) holds if and only if 
\[
\textbf{hw}(\Phi(L)) \textbf{hw}(\Phi(N))=\textbf{hw}(\Phi(M)),
\]
if and only if (by Equation (\ref{hight weight equation of M, N, and L}))
\[
\text{soc}(M) = \text{soc}(N) \oplus \text{soc}(L).
\]
So the items (1) and (4) are equivalent.

Now we prove that $(1) \Rightarrow (5) \Rightarrow (3)$. Assume that the item (1) holds. Then the item (5) follows from Equation (\ref{proof formula3}). Assume that the item (5) holds. By Equation (\ref{hight weight equation of M, N, and L}) and our assumption, 
\begin{align*}
\frac{\textbf{hw}(\Phi(M))}{\textbf{hw}(\Phi(N))} & = \textbf{z}^{g(M)-g(N)} {\bf f}^{\underline{\text{dim}}(\text{soc}(M)) - \underline{\text{dim}}(\text{soc}(N))} \\
& = \textbf{z}^{g(L)} {\bf f}^{\underline{\text{dim}}(\text{soc}(M)) - \underline{\text{dim}}(\text{soc}(N))} \\
& = \textbf{z}^{g(L)^+} \textbf{z}^{g(L)^-} {\bf f}^{\underline{\text{dim}}(\text{soc}(M)) - \underline{\text{dim}}(\text{soc}(N))},
\end{align*}
where $g(L)^+$ (respectively, $g(L)^-$) is the non-negative (respectively, non-positive) component of $g(L)$. Moreover, $g(L)^- = -\underline{\text{dim}}(\text{soc}(L))$, since the injective modules defining the non-positive vectors $g(L)^-$ are determined by $\text{soc}(L)$.

Recall that $f_i=Y_{i,\xi(i)}Y_{i,\xi(i)-2}$ for $i\in I$ and $\textbf{z}=\{ Y_{i,\xi(i)} \mid i\in I \}$. By assumption, $\textbf{z}^{g(L)^-} {\bf f}^{\underline{\text{dim}}(\text{soc}(M)) - \underline{\text{dim}}(\text{soc}(N))}$ is a monomial in the $Y_{i,a}$ with positive exponents. It follows that
\[
\underline{\text{dim}}(\text{soc}(M)) - \underline{\text{dim}}(\text{soc}(N)) -\underline{\text{dim}}(\text{soc}(L))  \geq 0,
\]
that is, $\text{soc}(N) \oplus \text{soc}(L)$ is a summand of $\text{soc}(M)$. Since $\text{Ker}(h)$ is a submodule of $\tau^{-1}L$ and $\tau_{\mathcal{C}}$ is an equivalence of $\mathcal{C}$, we have 
\[
0 \to \tau_{\mathcal{C}}\,\text{Ker}(h) \to L,
\] 
which implies that $\text{soc}(\tau_{\mathcal{C}}\,\text{Ker}(h))$ is a summand of $\text{soc}(L)$. So we have  
\[
\text{soc}(N \oplus \tau_{\mathcal{C}}\,\text{Ker}(h)) \text{ is a summand of } \text{soc}(M),
\]
and this shows (3).
\end{proof}

\subsubsection{The vector $d$}
In this subsection, we study a representation theoretic interpretation of the vector $(d_j)_{j\in I}$ in Equation (\ref{Gamma1 equation2}).

Let $\langle -,- \rangle: \mathbb{Z}^I\times \mathbb{Z}^I \to \mathbb{Z}^I$ be the Euler form of $kQ$. For any two finite-dimensional $kQ$-modules $L$ and $M$, we have
\[
\langle  \underline{ \textup{dim}}\,L, \underline{ \textup{dim}}\,M \rangle = \text{dim}\,\text{Hom}(L,M)-\text{dim}\,\text{Ext}^1(L,M).
\]

%\begin{lemma}\label{a g-vectors equation}
%Let $M$ be an indecomposable non-projective $kQ$-module. For the almost split sequence 
%\begin{align*}
%0 \longrightarrow \tau M \longrightarrow E  \longrightarrow M \longrightarrow 0,
%\end{align*}
%we have 
%\begin{align*}
%-g(\tau M)-g(M)=a(M),
%\end{align*}
%where $a(M)\in \mathbb{Z}^I$ is defined by 
%\begin{align*}
%a(M)_i=\sum_{j:j\to i} \underline{\text{dim}}(M_j) - \sum_{j:i\to j} \underline{\text{dim}}(M_j).
%\end{align*}
%\end{lemma}

By definition, $\tau_{\mathcal{C}} M=\tau_{kQ} M$ if $M$ has no projective summands, and $\tau_{\mathcal{C}} P(i)=P(i)[1]$ for $i\in I$.

\begin{lemma}\label{a g-vectors equation}
Let $M$ be an indecomposable $kQ$-module. Then
\begin{align*}
-g(\tau_{\mathcal{C}} M)-g(M)=a(M),
\end{align*}
where $a(M)\in \mathbb{Z}^I$ is defined by 
\begin{align*}
a(M)_i=\sum_{j:j\to i} \underline{\text{dim}}(M_j) - \sum_{j:i\to j} \underline{\text{dim}}(M_j).
\end{align*}
\end{lemma}

\begin{proof}
The proof follows from Lemma 2.1~(2) and Lemma 2.3 in \cite{Pal08}. We explain it as follows. Note that the notation $-\text{coind}(M)$ in \cite{Pal08} corresponds to our $g(M)$. By Lemma~2.1~(2) and Lemma 2.3 in \cite{Pal08}, for any $i\in I$, 
\begin{align*}
-g(\tau_{\mathcal{C}} M)_i= -\langle M, S(i) \rangle = -\underline{\text{dim}}(M_i) + \sum_{j:j\to i} \underline{\text{dim}}(M_j). 
\end{align*}
Again by Lemma 2.3 in \cite{Pal08}, for any $i\in I$,
\begin{align*}
-g(M)_i= \langle S(i), M \rangle = \underline{\text{dim}}(M_i) - \sum_{j:i\to j} \underline{\text{dim}}(M_j). 
\end{align*}
Adding the two equations above, our result follows.
\end{proof}

In the following, we give a representation theoretic interpretation of $\prod_{i \in I} A^{-\underline{\dim}(M_i)}_{i,\xi(i)-1}$.

\begin{lemma}\label{lemma on representation theoretic interpretations of affine roots}
Let $M$ be a $kQ$-module. For any height function $\xi$ and $i\in I$, we have
\begin{align}\label{representation theoretic interpretations of affine roots}
\prod_{i \in I} A^{-\underline{\dim}(M_i)}_{i,\xi(i)-1} =  {\bf z}^{a(M)}  {\bf f}^{g(M)},
\end{align} 
where the vector $a(M)$ is defined in Lemma \ref{a g-vectors equation}.
\end{lemma}
\begin{proof}
Recall that $A^{-1}_{i,r} = Y^{-1}_{i,r-1} Y^{-1}_{i,r+1} \left(\prod_{j:C_{ij}=-1} Y_{j,r} \right)$, the identification $z_{i}= Y_{i,\xi(i)}$, and $f_i=Y_{i,\xi(i)}Y_{i,\xi(i)-2}$  for $i\in I$, $r\in \mathbb{Z}$. For any $i,j\in I$ with $C_{ij}=-1$,
\begin{align*}
Y_{j,\xi(i)-1} = \begin{cases}
z_{j}  & \text{if $i \to j$ in $Q$},  \\
\frac{f_j}{z_{j}}  & \text{if $j \to i$ in $Q$}.
\end{cases}
\end{align*}
Each $A^{-1}_{i,\xi(i)-1}$ contributes one $f^{-1}_i$, $z_{j}=Y_{j,\xi(i)-1}$ for $i \to j$, and $\frac{f_j}{z_{j}}=Y_{j,\xi(i)-1}$ for $j \to i$. In the left hand side of Equation (\ref{representation theoretic interpretations of affine roots}), it contributes $(\sum_{j:j\to i} \underline{\text{dim}}(M_j) - \sum_{j:i\to j} \underline{\text{dim}}(M_j))$  $z_{i}$, and $(\underline{\text{dim}}(M_i)-\sum_{j:i\to j} \underline{\text{dim}}(M_j))$ $f^{-1}_i$. Our result follows.
\end{proof}

\begin{lemma} \label{a g-vectors equation associated the other triangle}
For the almost split triangle 
\begin{align}\label{non-split triangles NL}
& N \overset{u}{\to} M' \to L \to N[1], 
\end{align}
we have 
\[
g(M')=g(L)+g(N)+a(\text{Im}(h)).
\]
\end{lemma}
\begin{proof}
Recall that $h: \tau^{-1} L \to N$. By Proposition \ref{Palu's g-vector formula}, 
\[
g(M')=g(N)+g(L)-g(\text{Ker}(\text{Hom}_{\mathcal{C}}(kQ,u)))-g(\tau\,\text{Ker}(\text{Hom}_{\mathcal{C}}(kQ,u))), 
\] 
where $\text{Hom}_{\mathcal{C}}(kQ,-):\mathcal{C} \to \text{mod}\,kQ$ is a functor from $\mathcal{C}$ to $\text{mod}\,kQ$, which induces an equivalence $\mathcal{C}/kQ[1] \cong \text{mod}\,kQ$. It follows from \cite[Corollary 4.9]{Fei23} that $\text{Ker}(\text{Hom}_{\mathcal{C}}(kQ,u))=\text{Im}(h)$.   Lemma \ref{a g-vectors equation} yields $-g( \text{Im} \,h)-g(\tau\,\text{Im} \,h)=a( \text{Im} \,h)$, and the proof is complete.
\end{proof}

Denote by $g(N)^+$ (respectively, $g(N)^-$) the non-negative (respectively, non-positive) component of $g(N)$ for $N \in \text{mod}\,kQ$.

\begin{lemma}\label{inequations of positive and negative g-vectors} 
Let $0\to L \to M \to N\to 0$ be an exact sequence in $\text{mod}\,kQ$. Then 
\begin{align*}
g(M)^- \leq g(L)^-, \quad  g(M)^+ \geq g(N)^+.
\end{align*}
\end{lemma}
\begin{proof}
It follows from Proposition \ref{Palu's g-vector formula} that $g(M)=g(L)+g(N)$, equivalently, 
\begin{align} \label{the g-vector formula is the need}
g(M)^- + g(M)^+  = g(L)^+ + g(L)^- +  g(N)^+ +  g(N)^-.
\end{align}
The injective modules defining the non-positive vectors $g(L)^-$ and $g(M)^-$ are determined by $\text{soc}(L)$ and $\text{soc}(M)$ respectively. Since $\text{soc}(L) \subset \text{soc}(M)$, we have $g(M)^- \leq g(L)^-$. By Equation (\ref{the g-vector formula is the need}), we obtain the following inequality
\[
g(M)^+  \geq  g(N)^+ + g(L)^+ +  g(N)^-.
\]
By definition of $g(N)^+$ and $g(N)^-$, there is no cancellation between $g(N)^+$ and $g(N)^-$. In other words, we assume that the positive coordinates in $g(N)^+$ are $\{i_1, \ldots, i_k\}$, and the negative coordinates in $g(N)^-$ are $\{j_1, \ldots, j_r\}$. Then $\{i_1, \ldots, i_k\} \cap \{j_1, \ldots, j_r\} =\emptyset$. All possible negative coordinates in $g(L)^+ + g(N)^-$ must be from $\{j_1, \ldots, j_r\}$, while the values of $g(N)^+$ at coordinates $\{j_1, \ldots, j_r\}$ are 0.  So
\[
g(M)^+  \geq   g(N)^+  + \max\{g(L)^+ + g(N)^-,0\} \geq  g(N)^+. \qedhere
\] 
\end{proof}

Combining the lemmas above, we obtain a formula of the vector $(d_j)_{j\in I}$.

\begin{theorem} \label{equivalent descriptions of d}
In Equation (\ref{Gamma1 equation2}), the vector $(d_j)_{j\in I}$ is 
\begin{align*}
(d_j)_{j\in I} = \underline{\text{dim}}(\text{soc}(L))+\underline{\text{dim}}(\text{soc}(N))+g(\text{Im}(h))-\underline{\text{dim}}(\text{soc}(M')).
\end{align*}
Moreover
\begin{align} \label{inequation of the exponents in general case}
0 \leq \underline{\text{dim}}(\text{soc}(L)) - \underline{\text{dim}}(\text{soc}(\tau_{\mathcal{C}}\,Ker(h))) \leq (d_j)_{j\in I}.
\end{align}
In particular, if  $L=\tau_{\mathcal{C}}\,N$, then $(d_j)_{j\in I} = \underline{\text{dim}}(\text{soc}(\tau_{\mathcal{C}}\,N))+\underline{\text{dim}}(\text{soc}(N))+g(N)$.
\end{theorem}

\begin{proof}
By  definition of the vector $(d_i)_{i\in I}$,  
\begin{gather}\label{eq 519a}
\begin{aligned}
\prod_{j \in I} f_j^{d_j} & = \frac{\textbf{hw}(\Phi(L))\textbf{hw}(\Phi(N))}{\textbf{hw}(\Phi(M'))}  \prod_{i \in I}  A^{-\alpha_i}_{i,\xi(i)-1} \\
& = \frac{\textbf{z}^{g(L)+g(N)-g(M')+a(\text{Im}(h))}{\bf f}^{g(\text{Im}(h))} F_{M'}(y_1,\ldots,y_{n})|_{\mathbb{P}}}{ F_{L}(y_1,\ldots,y_{n})|_{\mathbb{P}} F_{N}(y_1,\ldots,y_{n})|_{\mathbb{P}}} \\
& = {\bf f}^{\underline{\text{dim}}(\text{soc}(L))+\underline{\text{dim}}(\text{soc}(N))+g(\text{Im}(h))-\underline{\text{dim}}(\text{soc}(M'))}, 
\end{aligned}
\end{gather} 
where the second equality follows from Equation (\ref{g-vector and height weight and F-polynomial}) and Lemma \ref{lemma on representation theoretic interpretations of affine roots}, the last equality follows from Lemma \ref{tropical F-polynomials} and Lemma \ref{a g-vectors equation associated the other triangle}.

It follows from Equation (\ref{Hubery's a result}) that 
\begin{align}\label{eq 519}
\begin{split}
& \underline{\text{dim}}(\text{soc}(N))+g(\text{Im}(h))-\underline{\text{dim}}(\text{soc}(M')) \\
& = \underline{\text{dim}}(\text{soc}(N))+g(\text{Im}(h))-\underline{\text{dim}}(\text{soc}(\tau_{\mathcal{C}}\,\text{Ker}(h))) - \underline{\text{dim}}(\text{soc}(\text{Coker}(h))).
\end{split}
\end{align}

Applying Lemma \ref{inequations of positive and negative g-vectors} to the following exact sequence 
\[
0 \to \text{Im}(h) \to N \to \text{Coker}(h) \to 0
\]
in $\text{mod}\,kQ$, we have $g(N)^+ \geq g(\text{Coker}(h))^+$.  Substituting this in Equation (\ref{eq 519}) yields
\begin{align*}
& \underline{\text{dim}}(\text{soc}(N))+g(\text{Im}(h))-\underline{\text{dim}}(\text{soc}(M')) \\
& = (\underline{\text{dim}}(\text{soc}(N))+g(N)) - (g(\text{Coker}(h)) + \underline{\text{dim}}(\text{soc}(\text{Coker}(h)))) - \underline{\text{dim}}(\text{soc}(\tau_{\mathcal{C}}\,\text{Ker}(h))) \\
& = g(N)^+  - g(\text{Coker}(h))^+ - \underline{\text{dim}}(\text{soc}(\tau_{\mathcal{C}}\,\text{Ker}(h)) \\
& \geq - \underline{\text{dim}}(\text{soc}(\tau_{\mathcal{C}}\,\text{Ker}(h)),
\end{align*}
where the second equality holds because $\underline{\text{dim}}(\text{soc}(N))$ is equal to the absolute value of the negative part of $g(N)$, for $N\in \text{mod}\,kQ$.

Hence
\begin{gather*}
\begin{aligned}
\underline{\text{dim}}(\text{soc}(L)) - \underline{\text{dim}}(\text{soc}(\tau_{\mathcal{C}}\,\text{Ker}(h)) & \leq \underline{\text{dim}}(\text{soc}(L))+\underline{\text{dim}}(\text{soc}(N))+g(\text{Im}(h))-\underline{\text{dim}}(\text{soc}(M')) \\
& = (d_j)_{j\in I},
\end{aligned}
\end{gather*}
where the last equality follows from Equation (\ref{eq 519a}). Since $Ker(h)$ is a submodule of $\tau^{-1}L$ and $\tau_{\mathcal{C}}$ is an equivalent of $\mathcal{C}$, we have $0\to \tau_{\mathcal{C}}\,Ker(h)\to L$, and so $\text{soc}(\tau_{\mathcal{C}}\,Ker(h)) \subseteq \text{soc}(L)$. This implies that $(d_j)_{j\in I} \geq 0$.
\end{proof}

\subsubsection{ An algorithm for the highest $l$-weight monomials of Hernandez-Leclerc modules. }

In this section, similar to the knitting algorithm of Auslander-Reiten quivers, we give an algorithm to compute the highest $l$-weight monomials of Hernandez-Leclerc modules. 

Consider the following non-split exchange triangles
\begin{align*}                                    
& P(i)[1] \to  0  \to  I(i) \to P(i)[2]=I(i), \\
& I(i) \to \oplus_{j:i\to j} P(j)[1] \bigoplus \oplus_{j:j\to i} I(j) \to P(i)[1], 
\end{align*}
where $h$ is the morphism $\tau^{-1} P(i)[1] \cong P(i) \to I(i)$. Then $\text{Im}\,h=S(i)$, and $\text{Ker}\,h=\oplus_{j:i\to j} P(j)$. By Theorems \ref{equivalent descriptions of c} and \ref{equivalent descriptions of d},
\begin{gather}
\begin{align*}
(c_i)_{i\in I} & = \underline{\delta}(I(i) \oplus \tau_{\mathcal{C}}\,\text{Ker}(h),0)=\underline{\text{dim}}(S(i)), \\
(d_j)_{j\in I} & = \underline{\text{dim}}(\text{soc}(P(i)[1]))+\underline{\text{dim}}(\text{soc}(I(i)))+g(S(i))-\underline{\text{dim}}(\text{soc}(\oplus_{j:i\to j} P(j)[1] \bigoplus \oplus_{j:j\to i} I(j))) \\
& = \underline{\text{dim}}(\text{soc}(I(i)))+g(S(i))-\underline{\text{dim}}(\oplus_{j:j\to i} \text{soc}(I(j))) \\
& = 0,
\end{align*}
\end{gather}
where the last equation follows from $0\to S(i) \to I(i) \to \oplus_{j:j\to i}I(j) \to 0$. With these values for $c_i,d_j$, our Theorem \ref{main theorem2} yields
\begin{align}\label{$T$-system equation in our setting}
[\Phi(P(i)[1])] [\Phi(I(i))] & = [L(f_i)] + \prod_{j:i\to j} [\Phi(P(j)[1])] \prod_{j:j\to i} [\Phi(I(j))].
\end{align}

Recall that $\Phi(P(i)[1])=L(Y_{i,\xi(i)})$ for $i\in I$. Moreover, Equation (\ref{proof formula3}) yields
\[
\textbf{hw}(\Phi(P(i)[1])) \textbf{hw}(\Phi(I(i))) = f_i
\] 
and hence $\textbf{hw}(\Phi(I(i)))=Y_{i,\xi(i)-2}$. We have shown the following.
\begin{proposition}\label{images of indecomposable injective modules under Phi}
 For any indecomposable injective $I(i)$,
\begin{align}\label{the real prime simple modules associated to indecomposable injective modules}
\Phi(I(i))=L(Y_{i,\xi(i)-2}). 
\end{align}\qed
\end{proposition} 
\begin{remark}\label{one case of our Equation from T system}
 It follows from the definition of $\xi$ in Section \ref{subcategories} that $\Phi(P(j)[1])=L(Y_{j,\xi(i)-1})$ for $i \to j$, and $\Phi(I(j))=L(Y_{j,\xi(i)-1})$ for $j \to i$. Thus our Equation (\ref{$T$-system equation in our setting}) is nothing but an equation from $T$-system:
\begin{align*}
[L(Y_{i,\xi(i)-2})] [L(Y_{i,\xi(i)})] =  [L(f_i)] + \prod_{j:j\sim i} [L(Y_{j,\xi(i)-1})].
\end{align*}
\end{remark}

Applying Theorem \ref{main theorem2} to the case that $L=\tau\,N$, and using the fact that in this case $\alpha=\underline{\text{dim}}(N)$ and $M'=0$, we obtain the following corollary.

\begin{corollary}\label{a corollary of main theorem2}
Let $N$ be an indecomposable non-projective $kQ$-module. For an almost split sequence 
\[
0 \longrightarrow \tau N \longrightarrow M  \longrightarrow N \longrightarrow 0,
\]
we have an exact sequence in $\mathscr{C}$:
\begin{align}\label{Gamma1 exact sequence}
0 \to \bigotimes_{j\in I} L(f_j)^{\otimes d_j}  \to \Phi(\tau N) \otimes \Phi(N) \to  \Phi(M)  \otimes   \left(\bigotimes_{i\in I} L(f_i)^{\otimes c_i} \right) \to 0,
\end{align}
or the same sequence with arrows reversed, where all $c_i$ and $d_j$ are non-negative integers defined by:  
% $\Phi(\tau N) \circ \Phi(N)$ is a cyclic tensor product of $\Phi(\tau N)$ and $\Phi(N)$, and 
\begin{itemize}
\item[(1)] $(c_i)_{i\in I}=\underline{\delta}(N,M)$,
\item[(2)] $(d_j)_{j\in I} = \underline{\text{dim}}(\text{soc}(\tau N))+\underline{\text{dim}}(\text{soc}(N))+g(N)$.
\end{itemize}
 Moreover the following are equivalent:
\begin{itemize}
\item [(1)] The vector $(c_i)_{i\in I}$ is 0. 

\item [(2)] $F_{M}(y_1,\ldots,y_{n})|_{\mathbb{P}}=F_{M}(y_1,\ldots,y_{n})|_{\mathbb{P}} \oplus \left( {\bf y}^{\underline{\text{dim}}(N)} \right)|_{\mathbb{P}}.$

\item [(3)] $\text{soc}(N) \text{ is a summand of } \text{soc}(M).$

\item [(4)] $\text{soc}(M) = \text{soc}(N) \oplus \text{soc}(\tau N)$.

\item [(5)] $\textbf{hw}(\Phi(N)) \text{ is a factor in } \textbf{hw}(\Phi(M)).$
\end{itemize}
\end{corollary}

The exact sequence (\ref{Gamma1 exact sequence}) implies that the following Equation (\ref{Gamma1 equation}) holds in $K_0(\mathscr{C})$:
\begin{align} \label{Gamma1 equation}
[\Phi(\tau N)] [\Phi(N)] = [\Phi(M)] \prod_{i\in I} [L(f_i)]^{c_i} + \prod_{j \in I} [L(f_j)]^{d_j}.
\end{align}

Proposition \ref{images of indecomposable injective modules under Phi} and Corollary \ref{a corollary of main theorem2} give in fact an algorithm to compute the highest $l$-weight monomials of Hernandez-Leclerc modules. We start with the images $\Phi(I(i))$ of the indecomposable injective modules $I(i)$ and apply Corollary~\ref{a corollary of main theorem2} until we reach the images $\Phi(P(i))$ of the indecomposable projective modules $P(i)$. As a consequence, Hernandez-Leclerc modules (excluding initial and frozen Hernandez-Leclerc modules) in $\mathscr{C}$ can be read off from the Auslander-Reiten quiver of $\text{mod}\,kQ$. We give an example of type $D_4$ to explain it.

\begin{example}
Let $\xi(1,2,3,4)=(4,3,2,2)\in \mathbb{Z}^4$. By definition, our quiver $Q$ is the following quiver 
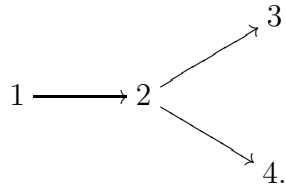
\[
\xymatrix@C35pt@R15pt{ && 3 \\ 1 \ar [r] & 2 \ar[ur] \ar[dr] &  \\  && 4.}
\] 

Following \cite{Sch08,Sch14}, a geometric model of the Auslander-Reiten quiver of the cluster category of $Q$ can be realized by arcs in  the punctured square $P^{\bullet}_4$, see Figure \ref{D4clustercategory} and Section \ref{a geometric realization of cluster category}. The Auslander-Reiten quiver of the cluster category of $Q$ is shown in Figure~\ref{the Auslander-Reiten quiver of linear D4}. Following \cite{BMR07}, the Auslander-Reiten quiver of $\text{mod}\,kQ$ can be obtained from Figure~\ref{the Auslander-Reiten quiver of linear D4} by deleting $P(i)[1]$, with $i=1,2,3,4$.

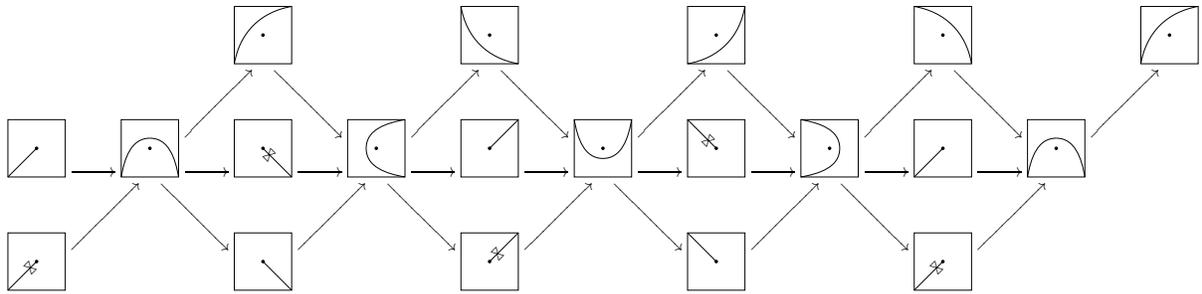
\begin{figure}[H]
\center
\begin{gather*}
\xymatrix{
&& \begin{tikzpicture}
\fill circle (1pt);
\draw  (45:0.8cm) \foreach \x in {135,225,315}{
 -- (\x:0.8cm) } -- cycle (135:0.8cm);
\draw (225:0.8) to [out=80,in=190] (45:0.8); 
\end{tikzpicture} \ar[rd]  &&  
\begin{tikzpicture}
\fill circle (1pt);
\draw  (45:0.8cm) \foreach \x in {135,225,315}{
 -- (\x:0.8cm) } -- cycle (135:0.8cm);
\draw (135:0.8) to [out=-80,in=170] (315:0.8); 
\end{tikzpicture} \ar[rd]  &&  
\begin{tikzpicture}
\fill circle (1pt);
\draw  (45:0.8cm) \foreach \x in {135,225,315}{
 -- (\x:0.8cm) } -- cycle (135:0.8cm);
\draw (45:0.8) to [out=-100,in=10] (225:0.8); 
\end{tikzpicture} \ar[rd] && 
\begin{tikzpicture}
\fill circle (1pt);
\draw  (45:0.8cm) \foreach \x in {135,225,315}{
 -- (\x:0.8cm) } -- cycle (135:0.8cm);
\draw (315:0.8) to [out=100,in=-10] (135:0.8);  
\end{tikzpicture} \ar[rd] && \begin{tikzpicture}
\fill circle (1pt);
\draw  (45:0.8cm) \foreach \x in {135,225,315}{
 -- (\x:0.8cm) } -- cycle (135:0.8cm);
\draw (225:0.8) to [out=80,in=190] (45:0.8); 
\end{tikzpicture}  \\
\begin{tikzpicture}
\fill circle (1pt);
\draw  (45:0.8cm) \foreach \x in {135,225,315}{
 -- (\x:0.8cm) } -- cycle (135:0.8cm);
\path (225:0.8) edge (225:0); 
\end{tikzpicture}  \ar[r] &  \begin{tikzpicture}
\fill circle (1pt);
\draw  (45:0.8cm) \foreach \x in {135,225,315}{
 -- (\x:0.8cm) } -- cycle (135:0.8cm);
\draw (225:0.8) to [out=80,in=180] (90:0.2); 
\draw (315:0.8) to [out=100,in=0] (90:0.2);   
\end{tikzpicture}  \ar[rd] \ar[ru] \ar[r] & \begin{tikzpicture}
\fill circle (1pt);
\draw  (45:0.8cm) \foreach \x in {135,225,315}{
 -- (\x:0.8cm) } -- cycle (135:0.8cm);
\path (315:0.8) edge (315:0); 
\node[rotate=45]  at (315:0.2) {$\substack{\bowtie}$};
\end{tikzpicture} \ar [r] & \begin{tikzpicture}
\fill circle (1pt);
\draw  (45:0.8cm) \foreach \x in {135,225,315}{
 -- (\x:0.8cm) } -- cycle (135:0.8cm);
\draw (315:0.8) to [out=170,in=-90] (180:0.2); 
\draw (45:0.8) to [out=190,in=90] (180:0.2);   
\end{tikzpicture} \ar[rd] \ar[ru] \ar[r]  & \begin{tikzpicture}
\fill circle (1pt);
\draw  (45:0.8cm) \foreach \x in {135,225,315}{
 -- (\x:0.8cm) } -- cycle (135:0.8cm);
\path (45:0.8) edge (45:0); 
\end{tikzpicture} \ar[r]  & \begin{tikzpicture}
\fill circle (1pt);
\draw  (45:0.8cm) \foreach \x in {135,225,315}{
 -- (\x:0.8cm) } -- cycle (135:0.8cm);
\draw (135:0.8) to [out=-80,in=180] (-90:0.2); 
\draw (45:0.8) to [out=-100,in=0] (-90:0.2);  
\end{tikzpicture} \ar[rd] \ar[ru] \ar[r]  &   \begin{tikzpicture}
\fill circle (1pt);
\draw  (45:0.8cm) \foreach \x in {135,225,315}{
 -- (\x:0.8cm) } -- cycle (135:0.8cm);
\path (135:0.8) edge (135:0); 
\node[rotate=45]  at (135:0.2) {$\substack{\bowtie}$};
\end{tikzpicture}  \ar[r]  & \begin{tikzpicture}
\fill circle (1pt);
\draw  (45:0.8cm) \foreach \x in {135,225,315}{
 -- (\x:0.8cm) } -- cycle (135:0.8cm);
\draw (225:0.8) to [out=10,in=-90] (0:0.2); 
\draw (135:0.8) to [out=-10,in=90] (0:0.2);  
\end{tikzpicture} \ar[rd] \ar[ru] \ar[r]  & 
\begin{tikzpicture}
\fill circle (1pt);
\draw  (45:0.8cm) \foreach \x in {135,225,315}{
 -- (\x:0.8cm) } -- cycle (135:0.8cm);
\path (225:0.8) edge (225:0); 
\end{tikzpicture}  \ar[r] &  \begin{tikzpicture}
\fill circle (1pt);
\draw  (45:0.8cm) \foreach \x in {135,225,315}{
 -- (\x:0.8cm) } -- cycle (135:0.8cm);
\draw (225:0.8) to [out=80,in=180] (90:0.2); 
\draw (315:0.8) to [out=100,in=0] (90:0.2);   
\end{tikzpicture} \ar[ru] &  \\
\begin{tikzpicture}
\fill circle (1pt);
\draw  (45:0.8cm) \foreach \x in {135,225,315}{
 -- (\x:0.8cm) } -- cycle (135:0.8cm);
\path (225:0.8) edge (225:0); 
\node[rotate=-45] at (225:0.2) {$\substack{\bowtie}$};
\end{tikzpicture} \ar[ru]  && 
\begin{tikzpicture}
\fill circle (1pt);
\draw  (45:0.8cm) \foreach \x in {135,225,315}{
 -- (\x:0.8cm) } -- cycle (135:0.8cm);
\path (315:0.8) edge (315:0); 
\end{tikzpicture} \ar[ru]  
&& \begin{tikzpicture}
\fill circle (1pt);
\draw  (45:0.8cm) \foreach \x in {135,225,315}{
 -- (\x:0.8cm) } -- cycle (135:0.8cm);
\path (45:0.8) edge (45:0); 
\node[rotate=-45] at (45:0.2) {$\substack{\bowtie}$};
\end{tikzpicture} \ar[ru]   
&& \begin{tikzpicture}
\fill circle (1pt);
\draw  (45:0.8cm) \foreach \x in {135,225,315}{
 -- (\x:0.8cm) } -- cycle (135:0.8cm);
\path (135:0.8) edge (135:0); 
\end{tikzpicture} \ar[ru]  
&& \begin{tikzpicture}
\fill circle (1pt);
\draw  (45:0.8cm) \foreach \x in {135,225,315}{
 -- (\x:0.8cm) } -- cycle (135:0.8cm);
\path (225:0.8) edge (225:0); 
\node[rotate=-45] at (225:0.2) {$\substack{\bowtie}$};
\end{tikzpicture} \ar[ru]  
&& }
\end{gather*}
\caption{A geometric model of the Auslander-Reiten quiver of the cluster category of $Q$. The vertices on the far left are to be identified with the vertices on the far right, and they correspond to the initial clusters.}\label{D4clustercategory}
\end{figure}

\begin{figure}
\center
\begin{tikzcd}[column sep=small]
&& \substack{\,1 \, \\ \,2\, \\ 34} \substack{[1]} \ar[dr]  && \substack{\,1 \, \\ \,2\, \\ 34} \ar[dr] && \substack{2} \ar[dr]  &&  \substack{1} \ar[dr] && \substack{\,1 \, \\ \,2\, \\ 34} \substack{[1]} \\
\substack{3}\substack{[1]} \ar[r]  & \substack{\,2\, \\ 34}\substack{[1]} \ar[ur] \ar[dr]\ar[r] & \substack{3} \ar[r] & \substack{\,2\, \\ 34} \ar[ur]  \ar[r] \ar[dr] & \substack{2 \\ 4} \ar[r]  &  \substack{\,1\, \\ \,22\, \\ 34} \ar[ur]  \ar[r] \ar[dr]  &  \substack{1 \\ 2 \\ 3}  \ar[r] &  \substack{1 \\ 2} \ar[ur] \ar[r] \ar[dr] & \substack{3}\substack{[1]} \ar[r]  & \substack{\,2\, \\ 34}\substack{[1]} \ar[ur] & \\
\substack{4}\substack{[1]}  \ar[ur] && \substack{4} \ar[ur] && \substack{2\\3} \ar[ur]  &&  \substack{1 \\ 2 \\ 4} \ar[ur] && \substack{4}\substack{[1]} \ar[ur] && 
\end{tikzcd}
\caption{The Auslander-Reiten quiver of the cluster category of $Q$.}\label{the Auslander-Reiten quiver of linear D4}
\end{figure}
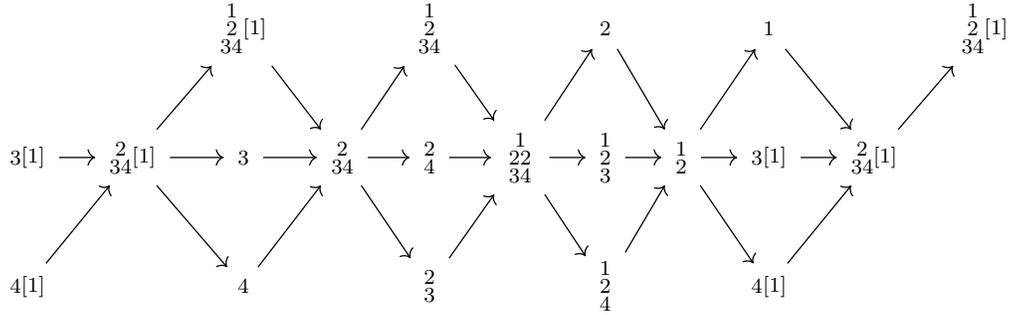

By Theorem \ref{equivalent descriptions of c} and Theorem \ref{equivalent descriptions of d}, for the almost split sequence in $\text{mod}\,kQ$
\[
0 \to \substack{2} \to  \substack{1\\2} \to \substack{1} \to 0,
\] 
we have 
\begin{align*}
& (c_i)_{i\in I}= \underline{\delta}(\substack{1},\substack{1\\2})=(1,0,0,0)^T, \\
& (d_j)_{j\in I} = \underline{\text{dim}}(\text{soc}(\substack{2}))+\underline{\text{dim}}(\text{soc}(\substack{1}))+g(\substack{1})=(0,1,0,0)^T. 
\end{align*}
The indecomposable modules $\substack{1}$ and $\substack{1\\2}$ are injective modules, thus by Equation (\ref{the real prime simple modules associated to indecomposable injective modules}) 
\[
\Phi(\substack{1})=L(Y_{1,2}) \text{ and } \Phi(\substack{1\\2})=L(Y_{2,1}). 
\]
Equation (\ref{proof formula3}) yields
\[
\textbf{hw}(\Phi(\substack{2})) \textbf{hw}(\Phi(\substack{1}))=\textbf{hw}(\Phi(\substack{1\\2})) \prod_{i\in I} f_i^{c_i}
\]
and thus
$\textbf{hw}(\Phi(\substack{2}))
Y_{1,2}
=Y_{2,1} f_1 .$ By definition of $f_1$, we have $f_1=Y_{1,2}Y_{1,4}$ and thus $\textbf{hw}(\Phi(\substack{2})) =Y_{2,1}Y_{1,2}Y_{1,4}Y_{1,2}^{-1}=Y_{1,4}Y_{2,1}$.  Therefore 
 $\Phi(\substack{2})=L(Y_{1,4}Y_{2,1})$, and the corresponding  Equation (\ref{Gamma1 equation}) in $K_0(\mathscr{C})$ is,
\[
[L(Y_{1,4}Y_{2,1})] [L(Y_{1,2})] = [L(Y_{2,1})] [L(f_1)] + [L(f_2)].
\]
In a similar way, we can use  the other meshes in the Auslander-Reiten quiver of mod\,$kQ$ in order to compute the prime simple modules in $\mathscr{C}$. In total, there are 20 real prime simple modules in $\mathscr{C}$ including the four frozen modules $\{L(f_1), L(f_2), L(f_3), L(f_4)\}$. The 16 unfrozen modules, including the four initial modules 
$
L(Y_{1,4}), L(Y_{2,3}), L(Y_{3,2}), L(Y_{4,2}),
$
are listed in Figure \ref{Real prime simple modules from the Auslander-Reiten quiver of D4}. Every mesh in the sense of the Auslander-Reiten quiver corresponds to an exact sequence in $\mathscr{C}$ and to an equation in $K_0(\mathscr{C})$.

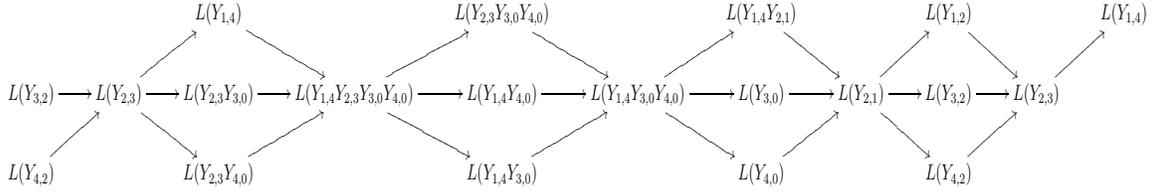
\begin{figure}
\center
\resizebox{.5\width}{.7\height}{
\xymatrix{
&& L(Y_{1,4}) \ar[dr]  &&  L(Y_{2,3}Y_{3,0}Y_{4,0}) \ar[dr] &&   L(Y_{1,4}Y_{2,1})  \ar[dr]  &&  L(Y_{1,2}) \ar[dr] &&  L(Y_{1,4}) \\
L(Y_{3,2}) \ar[r] & L(Y_{2,3}) \ar[ur]  \ar[r] \ar[dr] & L(Y_{2,3}Y_{3,0}) \ar[r] &  L(Y_{1,4}Y_{2,3}Y_{3,0}Y_{4,0}) \ar[ur]  \ar[r] \ar[dr] & L(Y_{1,4}Y_{4,0}) \ar[r]  &  L(Y_{1,4}Y_{3,0}Y_{4,0}) \ar[ur]  \ar[r] \ar[dr]  & L(Y_{3,0})  \ar[r] &  L(Y_{2,1}) \ar[ur]  \ar[r] \ar[dr] & L(Y_{3,2}) \ar[r] & L(Y_{2,3}) \ar[ur] & \\
L(Y_{4,2}) \ar[ur] && L(Y_{2,3}Y_{4,0}) \ar[ur] &&  L(Y_{1,4}Y_{3,0}) \ar[ur]  &&  L(Y_{4,0}) \ar[ur] && L(Y_{4,2}) \ar[ur] && }} 
\caption{The 16 real prime simple modules in $\mathscr{C}$.}\label{Real prime simple modules from the Auslander-Reiten quiver of D4}
\end{figure}

The complete list of the remaining equations of type (\ref{Gamma1 equation}) are as follows.
\begin{gather*}
\begin{aligned}[]
[L(Y_{1,4}Y_{3,0}Y_{4,0})] [L(Y_{2,1})] & = [L(Y_{1,4}Y_{2,1})] [L(Y_{3,0})] [L(Y_{4,0})] + [L(f_3)] [L(f_4)], \\
[L(Y_{2,3}Y_{3,0}Y_{4,0})] [L(Y_{1,4}Y_{2,1})] & = [L(Y_{1,4}Y_{3,0}Y_{4,0})] [L(f_2)] + [L(f_1)][L(f_3)][L(f_4)], \\
[L(Y_{1,4}Y_{4,0})]  [L(Y_{3,0})] & =  [L(Y_{1,4}Y_{3,0}Y_{4,0})] + [L(f_4)], \\
[L(Y_{1,4}Y_{3,0})]  [L(Y_{4,0})] &  =  [L(Y_{1,4}Y_{3,0}Y_{4,0})] + [L(f_3)], \\
[L(Y_{1,4}Y_{2,3}Y_{3,0}Y_{4,0})]  [L(Y_{1,4}Y_{3,0}Y_{4,0})] & = [L(Y_{2,3}Y_{3,0}Y_{4,0})]  [L(Y_{1,4}Y_{4,0})]  [L(Y_{1,4}Y_{3,0})] + [L(f_1)][L(f_3)][L(f_4)], \\
[L(Y_{2,3}Y_{3,0})]  [L(Y_{1,4}Y_{4,0})] &  = [L(Y_{1,4}Y_{2,3}Y_{3,0}Y_{4,0})] + [L(f_1)][L(f_3)], \\
[L(Y_{2,3}Y_{4,0})]  [L(Y_{1,4}Y_{3,0})] & = [L(Y_{1,4}Y_{2,3}Y_{3,0}Y_{4,0})] + [L(f_1)][L(f_4)]. 
\end{aligned}
\end{gather*}
\end{example}

\subsection{Two conjectures for the case $\ell>1$}

Let $\xi$ be an arbitrary height function and $\ell\ge 2$. Then $K_0(\mathscr{C}^{\leq \xi}_\ell)$ is the cluster algebra with initial quiver $\Gamma^{\leq \xi}_\ell$. Recall that in Section \ref{our Jacobian algebras}, $Q^{\leq \xi}_\ell$ is the principal quiver of $\Gamma^{\leq \xi}_\ell$ with vertex set $\widehat{I}^{\leq \xi}_{\ell-1}$, $J^{\leq \xi}_\ell$ is the Jacobian ideal of $\mathbb{C}Q^{\leq \xi}_\ell$, and  $A^{\leq \xi}_\ell=\mathbb{C}Q^{\leq \xi}_\ell/J^{\leq \xi}_\ell$ is the truncated Jacobian algebra. The category $\text{mod}\,A^{\leq \xi}_\ell$ is equivalent to the category of finite-dimensional representations of the bound quiver $(Q^{\leq \xi}_\ell,J^{\leq \xi}_\ell)$. We denote by $\mathcal{C}^{\leq \xi}_\ell$ the associated cluster category of $Q^{\leq \xi}_\ell$.

The $F$-polynomial $F_M$ of $M\in \mathcal{C}^{\leq \xi}_\ell$ is defined as in (\ref{F-polynomials formula}). For a vertex $(i,r)$ of $Q^{\leq \xi}_\ell$, we define $\widehat{y}_{i,r}=A^{-1}_{i,r-1}$. For any $(i,r)\in \widehat{I}^{\leq \xi}_\ell \setminus \widehat{I}^{\leq \xi}_{\ell-1}$, we define 
\[
f_{i.r} = Y_{i,r}Y_{i,r+2}\cdots Y_{i,\xi(i)}.
\]

Inspired by Hernandez and Leclerc's work, we propose the following conjectures.

\begin{conjecture}\label{rigid objects equal real prime simple modules conjecture}
\hfill
\begin{itemize}
\item[(1)] There is a bijection between the set of indecomposable rigid objects in the cluster category $\mathcal{C}^{\leq \xi}_\ell$ and the set of real prime simple modules in $\mathscr{C}^{\leq \xi}_\ell$ that are not of the form $L(f_{i,r})$. 

\item[(2)] There is a bijection between the set of rigid objects in the cluster category $\mathcal{C}^{\leq \xi}_\ell$ and the set of real simple modules in $\mathscr{C}^{\leq \xi}_\ell$ that have no tensor factors of the form $L(f_{i,r})$. 
\end{itemize} 
\end{conjecture}

If Conjecture~\ref{rigid objects equal real prime simple modules conjecture} is true then the bijection is completely determined by $F$-polynomials  and $g$-vectors. More precisely, let $L(m)\in \mathscr{C}^{\leq \xi}_\ell$ be a real prime simple module corresponding to a cluster variable $X_M$, with $M\in \mathcal{C}^{\leq \xi}_\ell$. We have 
\[
\textbf{z}_\ell^{g(M)} \frac{F_M((\widehat{y}_{i,r})_{(i,r)\in \widehat{I}^{\leq \xi}_{\ell-1}})|_{\mathcal{F}}}{F_M((y_{i,r})_{(i,r)\in \widehat{I}^{\leq \xi}_{\ell-1}})|_{\mathbb{P}}}=mP_m,
\]
where $\textbf{z}_\ell$ is the specialization of quantum cluster variables in $\textbf{z}^{\leq \xi}_{t,\ell}$ at $t=1$ (here we ignore the superscript $^{\leq \xi}$ for a fixed $\xi$), 
\begin{align*} 
y_{i,r} = \begin{cases}
1 & \text{if $(i,r)\in \widehat{I}^{\leq \xi}_{\ell-2}$}, \\
f^{-1}_{i,r-2} \prod_{(j,s)\to (i,r)} f_{j,s-2} & \text{if $(i,r)\in \widehat{I}^{\leq \xi}_{\ell-1} \setminus \widehat{I}^{\leq \xi}_{\ell-2}$}, \\
\end{cases}
\end{align*}
and $\mathbb{P}$ is the tropical semifield generated by $f_{i,r}$, with $(i,r)\in \widehat{I}^{\leq \xi}_\ell \setminus \widehat{I}^{\leq \xi}_{\ell-1}$, and with tropical addition $\oplus$. It is clear that $y_{i,r} \in \mathbb{P}$. By the same argument with (\ref{g-vector and height weight and F-polynomial}), we have
\begin{align*} 
\frac{\textbf{z}_\ell^{g(M)}}{F_M((y_{i,r})_{(i,r)\in \widehat{I}^{\leq \xi}_{\ell-1}})|_{\mathbb{P}}}=m, \quad  P_m=F_M((\widehat{y}_{i,r})_{(i,r)\in \widehat{I}^{\leq \xi}_{\ell-1}})|_{\mathcal{F}}.
\end{align*}

After identifying $z_{i,r}$ with $\prod_{k\geq 0, r+2k\leq \xi(i)} Y_{i,r+2k}$, the dominant monomial $m$ can be expressed as a Laurent monomial in $z_{i,r}$ with $(i,r)\in \widehat{I}^{\leq \xi}_\ell$, that is, 
\[
m=\prod_{(i,r)\in \widehat{I}^{\leq \xi}_\ell} z^{d_{i,r}}_{i,r}
\] 
for some integers $d_{i,r}$. 

Define an $A^{\leq \xi}_\ell$-module $K^{\leq \xi}_\ell(m)$ as a general kernel from $\bigoplus_{(i,r)\in \widehat{I}^{\leq \xi}_{\ell-1}} (I^{\leq \xi}_{(i,r),\ell})^{\max\{-d_{i,r},0\}}$ to $\bigoplus_{(i,r)\in \widehat{I}^{\leq \xi}_{\ell-1}} (I^{\leq \xi}_{(i,r),\ell})^{\max\{d_{i,r},0\}}$, thus
\[
0 \to K^{\leq \xi}_\ell(m) \to \bigoplus_{(i,r)\in \widehat{I}^{\leq \xi}_{\ell-1}} (I^{\leq \xi}_{(i,r),\ell})^{\max\{-d_{i,r},0\}} \to  \bigoplus_{(i,r)\in \widehat{I}^{\leq \xi}_{\ell-1}} (I^{\leq \xi}_{(i,r),\ell})^{\max\{d_{i,r},0\}},
\]
where $I^{\leq \xi}_{(i,r),\ell}$ is the injective $A^{\leq \xi}_\ell$-module with simple socle $S_{(i,r)}$.

\begin{conjecture} \label{character conjecture}  
Let $L(m)\in \mathscr{C}^{\leq \xi}_\ell$ be a real simple module (without tensor factor $L(f_{i,r})$ for any $(i,r)\in \widehat{I}^{\leq \xi}_\ell \setminus \widehat{I}^{\leq \xi}_{\ell-1}$). Then
\[
\widetilde{\chi}_{q}(L(m))_{\leq \xi}=F_{K^{\leq \xi}_\ell(m)}((\widehat{y}_{i,r})_{(i,r)\in \widehat{I}^{\leq \xi}_{\ell-1}}).
\]
\end{conjecture}
When $\xi$ is a sink-source function and $\ell=\infty$, Conjecture \ref{character conjecture} becomes Hernandez-Leclerc's geometric $q$-character formulas conjecture \cite[Conjecture 5.3]{HL16}. In particular, if $L(m)\in \mathscr{C}^{\leq \xi}_\ell$ is a real simple module corresponding to a cluster monomial, then Conjecture \ref{character conjecture} holds. For any height function $\xi$ and $\ell=1$, Conjecture \ref{rigid objects equal real prime simple modules conjecture} and Conjecture \ref{character conjecture} follow from Theorem \ref{monoidal categorification C1}. 

%since each indecomposable objects in $\mathcal{C}^{\leq \xi}_1$ is rigid. Conjecture \ref{character conjecture} follows from Theorem \ref{main theorem2}, since these Grothendieck rings are cluster algebras of finite type \cite{FZ03}, and the set of real simple modules (without tensor factor $f_{i,r}$ for any $(i,r)\in \widehat{I}^{\leq \xi}_\ell \setminus \widehat{I}^{\leq \xi}_{\ell-1}$) in $\mathscr{C}^{\leq \xi}_1$ are in one-to-one correspondence with the set of cluster monomials (without frozen variables) in $\mathcal{A}(\Gamma)$. 

Using Palu's work on cluster characters \cite{Pal08}, for each pair of non-split triangles 
\[
L\to M \to N\to L[1], \quad N \to M' \to L \to N[1]
\] 
in $\mathcal{C}^{\leq \xi}_\ell$, we have 
\begin{align*}
X_L X_N= X_M + \textbf{y}^{\alpha} X_{M'},  
\end{align*}
where $\alpha\in \mathbb{Z}_{\geq 0}^{\widehat{I}^{\leq \xi}_{\ell-1}}$ is the rank vector of the image of the morphism $\tau^{-1}L \to N$, and $\textbf{y}^{\alpha}$ is  a monomial in principal coefficient variables $y_i$, with $i\in I$.

Assume that Conjecture~\ref{rigid objects equal real prime simple modules conjecture} holds. By Remark~\ref{remark 5.6} we may also assume $g(M)=g(L)+g(N)$. Let $L, N$ be two indecomposable rigid objects in $\mathcal{C}^{\leq \xi}_\ell$ with $\text{dim}(\text{Ext}^1_{\mathcal{C}^{\leq \xi}_\ell}(L,N))=1$. By the same argument as in Theorem \ref{main theorem2}, we have a relation in $K_0(\mathscr{C}^{\leq \xi}_\ell)$: 
\begin{align}\label{exchange equation for any subcategory}
[\Phi(L)] [\Phi(N)] = [\Phi(M)] \left( \prod_{(i,r)\in \widehat{I}^{\leq \xi}_\ell \setminus \widehat{I}^{\leq \xi}_{\ell-1}} [L(f_{i,r})]^{c_{i,r}} \right) + [\Phi(M')] \left( \prod_{(i,r)\in \widehat{I}^{\leq \xi}_\ell \setminus \widehat{I}^{\leq \xi}_{\ell-1}} [L(f_{i,r})]^{d_{i,r}} \right),
\end{align}
where all $c_{i,r}$ and $d_{i,r}$ are some non-negative integers such that  
 
\begin{align*}
& \prod_{(i,r)\in \widehat{I}^{\leq \xi}_\ell \setminus \widehat{I}^{\leq \xi}_{\ell-1}} f_{i,r}^{c_{i,r}} = \frac{ F_{M}((y_{i,r})_{(i,r)\in \widehat{I}^{\leq \xi}_{\ell-1}})|_{\mathbb{P}}}{F_{M}((y_{i,r})_{(i,r)\in \widehat{I}^{\leq \xi}_{\ell-1}})|_{\mathbb{P}} \oplus \left( (\prod_{(i,r)\in \widehat{I}^{\leq \xi}_{\ell-1}} y^{\text{dim}\,M_{i,r}}_{i,r}) F_{M'}((y_{i,r})_{(i,r)\in \widehat{I}^{\leq \xi}_{\ell-1}}) \right) |_{\mathbb{P}}}, \\
& \prod_{(i,r)\in \widehat{I}^{\leq \xi}_\ell \setminus \widehat{I}^{\leq \xi}_{\ell-1}}  f_{i,r}^{d_{i,r}} = \frac{\textbf{hw}(\Phi(L)) \textbf{hw}(\Phi(N))}{\textbf{hw}(\Phi(M'))}  \prod_{(i,r)\in \widehat{I}^{\leq \xi}_{\ell-1}} A^{-\alpha_{i,r}}_{i,r-1}.
\end{align*}
In particular,  
\begin{align*}
\textbf{hw}(\Phi(L)) \textbf{hw}(\Phi(N))=\textbf{hw}(\Phi(M)) \prod_{(i,r)\in \widehat{I}^{\leq \xi}_\ell \setminus \widehat{I}^{\leq \xi}_{\ell-1}} f_{i,r}^{c_{i,r}}.
\end{align*}

As a conclusion, if Conjecture \ref{rigid objects equal real prime simple modules conjecture} holds, the highest $l$-weight monomial of the new real prime simple module can be recursively computed by Equation (\ref{exchange equation for any subcategory}) starting from initial Kirillov-Reshetikhin modules. Conjecture \ref{character conjecture} predicts an explicit correspondence 
\[
L(m) \mapsto K^{\leq \xi}_\ell(m)
\] 
for any real simple module $L(m) \in \mathscr{C}^{\leq \xi}_\ell$.

\section{Conjecture \ref{rigid objects equal real prime simple modules conjecture} and Conjecture \ref{character conjecture} for $\mathscr{C}^{\leq \xi}_4$ in type $A_2$ and $\mathscr{C}^{\leq \xi}_2$ in type $A_4$}\label{check our conjecture for small cases}

In this section, we prove Conjecture \ref{rigid objects equal real prime simple modules conjecture} and Conjecture \ref{character conjecture} for any height function $\xi$, $\ell\leq 4$ of type $A_2$, and $\ell=2$ of type $A_3$ and type $A_4$.

For $\ell \leq 4$ of type $A_2$, it is enough to prove our conjectures for any height function $\xi$ and $\ell=4$, because subcategories $\mathscr{C}^{\leq \xi}_\ell$ with $\ell<4$ are contained in the subcategory $\mathscr{C}^{\leq \xi}_4$. For $\ell=2$ of type $A_3$ and type $A_4$, by the same reasoning, it is enough to prove our conjectures for any height function $\xi$ and $\ell=2$ of type $A_4$. In both cases, the truncated Jacobi algebras are cluster-tilted algebras of type $E_8$ \cite{Sch14}. 

For the simplicity of notation, for any $i\in I$, $r\in \mathbb{Z}$, we denote by $i_r$ the variable $Y_{i,r}$,  and denote by $m$ the simple module $L(m)$ in this section.

\subsection{The subcategory $\mathscr{C}^{\leq \xi}_4$ in type $A_2$}

Let $\xi(1,2)=(0,-1)$ (up to switching $1$ and $2$ and up to parameter shift). The cluster algebra $K_0(\mathscr{C}^{\leq \xi}_4)$ has the initial quiver shown in Figure~\ref{the initial quiver (left) and its principal quiver (right) A4}, $Q^{\leq \xi}_4$ is its principal quiver, see Figure~\ref{the initial quiver (left) and its principal quiver (right) A4} (note that we use labelings $1,2,3,4,5,6,7,8$ instead of $(1,0)$, $(2,-1)$, $(1,-2)$, $(2,-3)$, $(1,-4)$, $(2,-5)$, $(1,-6)$, $(2,-7)$ respectively). Then the truncated Jacobian algebra 
\[
A^{\leq \xi}_4 = \mathbb{C}Q^{\leq \xi}_4/\langle de, fg, hi, ea, ad+fk, ke+gb, bf+hl, lg+ic, ch+jm, mi, ef, gh, ij \rangle
\] 
is a cluster-tilted algebra of type $E_8$. The indecomposable rigid objects in $\mathcal{C}^{\leq \xi}_4$ are shown in Figure \ref{the cluster category C4 in A2}. The real prime simple modules (excluding the frozen modules) in $\mathscr{C}^{\leq \xi}_4$ are shown in Figure \ref{all the real prime simple modules C4 in A2}.  The two sets are in bijection under the map $L(m) \to K^{\leq \xi}_\ell(m)$.

\begin{figure}
\resizebox{.8\width}{.8\height}{ 
\begin{minipage}{0.5\textwidth}
\begin{xy}
(-25,50)*+{(1,0)}="a";
(0,50)*+{(2,-1)}="b"; 
(-25,30)*+{(1,-2)}="c";
(0,30)*+{(2,-3)}="d"; 
(-25,10)*+{(1,-4)}="e";
(0,10)*+{(2,-5)}="f";
(-25,-10)*+{(1,-6)}="g";
(0,-10)*+{(2,-7)}="h";
(-25,-30)*+{\fbox{(1,-8)}}="i";
(0,-30)*+{\fbox{(2,-9)}}="j";
{\ar "a";"b"};
{\ar "b";"c"};
{\ar "c";"d"};
{\ar "d";"e"};
{\ar "f";"d"};
{\ar "d";"b"};
{\ar "e";"c"};
{\ar "c";"a"};
{\ar "g";"e"};
{\ar "e";"f"};
{\ar "h";"f"};
{\ar "f";"g"};
{\ar "g";"h"};
{\ar "i";"g"};
{\ar "j";"h"};
{\ar "h";"i"};
\end{xy}
\end{minipage} 
\qquad 
\begin{minipage}{0.5\textwidth}
\begin{xy}
(-25,50)*+{1}="a";
(0,50)*+{2}="b"; 
(-25,30)*+{3}="c";
(0,30)*+{4}="d"; 
(-25,10)*+{5}="e";
(0,10)*+{6}="f";
(-25,-10)*+{7}="g";
(0,-10)*+{8}="h";
(-27,40)*+{a}="gg";
(-27,20)*+{b}="g1";
(-27,0)*+{c}="gg1";
(-12.5,52.5)*+{d}="hh";
(-12.5,37)*+{e}="i";
(-12.5,27.5)*+{f}="j";
(-12.5,17)*+{g}="i1";
(-12.5,7)*+{h}="i2";
(-12.5,-3)*+{i}="ii";
(-12.5,-13)*+{j}="i3";
(2,40)*+{k}="ij";
(2,20)*+{l}="jj";
(2.2,0)*+{m}="k";
(-25,-30)*+{}="i";
(0,-30)*+{}="j";
{\ar "a";"b"};
{\ar "b";"c"};
{\ar "c";"d"};
{\ar "d";"e"};
{\ar "f";"d"};
{\ar "d";"b"};
{\ar "e";"c"};
{\ar "c";"a"};
{\ar "e";"f"};
{\ar "g";"e"};
{\ar "f";"g"};
{\ar "h";"f"};
{\ar "g";"h"};
\end{xy}
\end{minipage}} 
\caption{The initial quiver (left) of $K_0(\mathscr{C}^{\leq \xi}_4)$ and its principal quiver $Q^{\leq \xi}_4$ (right).} \label{the initial quiver (left) and its principal quiver (right) A4}
\end{figure}
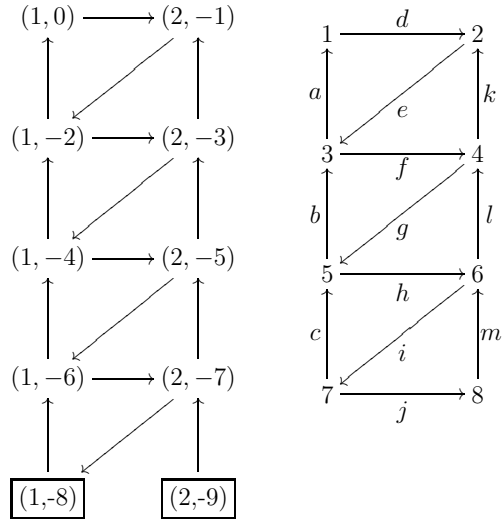

\begin{figure}\huge
\centerline{
\resizebox{.25\width}{.25\height}{ 
\xymatrix@R10pt@C10pt{
 &&  \makecell[c]{00\\00\\0-1\\00}  \ar[rdd] && \makecell[c]{01\\11\\11\\10} \ar[rdd]\ar[rdd] && \makecell[c]{00\\-10\\00\\00} \ar[rdd] 
 &&  \makecell[c]{11\\11\\00\\00}  \ar[rdd]  && \makecell[c]{00\\01\\11\\10} \ar[rdd]  &&  \makecell[c]{00\\00\\-10\\00} \ar[rdd]  && \makecell[c]{11\\11\\11\\00} \ar[rdd] && \makecell[c]{00\\01\\01\\01} \ar[rdd] && \makecell[c]{10\\10\\00\\00} \ar[rdd] && \makecell[c]{00\\11\\11\\00} \ar[rdd] && \makecell[c]{00\\00\\01\\01} \ar[rdd] && \makecell[c]{10\\10\\10\\00}\ar[rdd] && \makecell[c]{00\\11\\11\\11} \ar[rdd] &&  \makecell[c]{00\\0-1\\00\\00} \ar[rdd] && \makecell[c]{01\\11\\10\\00} \ar[rdd] && \makecell[c]{00\\00\\11\\11} \ar[rdd] && \makecell[c]{00\\00\\0-1\\00} && \\
 \\
& \makecell[c]{00\\00\\10\\11} \ar[rdd] \ar[ruu] &&  \makecell[c]{01\\11\\10\\10} \ar[rdd]\ar[ruu] &&  \makecell[c]{01\\01\\11\\10} \ar[rdd] \ar[ruu] &&  \makecell[c]{11\\01\\00\\00}  \ar[ruu]\ar[rdd] && \makecell[c]{11\\12\\11\\10} \ar[ruu]\ar[rdd] && \makecell[c]{00\\01\\01\\10} \ar[ruu]\ar[rdd] && \makecell[c]{11\\11\\01\\00} \ar[ruu]\ar[rdd] && \makecell[c]{11\\12\\12\\01} \ar[ruu]\ar[rdd] && \makecell[c]{10\\11\\01\\01} \ar[ruu]\ar[rdd] && \makecell[c]{10\\21\\11\\00} \ar[ruu]\ar[rdd] && \makecell[c]{00\\11\\12\\01} \ar[ruu]\ar[rdd] && \makecell[c]{10\\10\\11\\01} \ar[ruu]\ar[rdd] &&  \makecell[c]{10\\21\\21\\11} \ar[ruu]\ar[rdd] &&  \makecell[c]{00\\10\\11\\11} \ar[ruu]\ar[rdd] && \makecell[c]{01\\10\\10\\00} \ar[ruu]\ar[rdd] && \makecell[c]{01\\11\\21\\11} \ar[ruu]\ar[rdd] && \makecell[c]{00\\00\\10\\11} \ar[ruu] & && \\
 \\
\makecell[c]{01\\11\\20\\11} \ar[r] \ar[ruu] \ar[rdd] & \makecell[c]{01\\11\\10\\11} \ar[r] & \ar[r] \makecell[c]{01\\11\\20\\21} \ar[ruu]\ar[rdd] &  \ar[r] \makecell[c]{00\\00\\10\\10} & \ar[r] \makecell[c]{01\\01\\10\\10} \ar[ruu] \ar[rdd] &  \makecell[c]{01\\01\\00\\00} \ar[r] & \ar[r]  \makecell[c]{12\\02\\11\\10} \ar[ruu] \ar[rdd] &  \ar[r] \makecell[c]{11\\01\\11\\10} & \ar[r]  \makecell[c]{11\\02\\11\\10} \ar[rdd] \ar[ruu] &  \ar[r]  \makecell[c]{00\\01\\00\\00}  & \ar[r] \makecell[c]{11\\12\\01\\10} \ar[rdd]\ar[ruu] &  \ar[r] \makecell[c]{11\\11\\01\\10} & \ar[r]  \makecell[c]{11\\12\\02\\10} \ar[rdd]\ar[ruu]&  \ar[r] \makecell[c]{00\\01\\01\\00} & \ar[r] \makecell[c]{11\\12\\02\\01} \ar[rdd]\ar[ruu] &  \ar[r] \makecell[c]{11\\11\\01\\01} & \ar[r] \makecell[c]{21\\22\\12\\01} \ar[rdd]\ar[ruu]&  \ar[r] \makecell[c]{10\\11\\11\\00} & \ar[r] \makecell[c]{10\\22\\12\\01} \ar[rdd]\ar[ruu]&  \ar[r] \makecell[c]{00\\11\\01\\01} & \ar[r] \makecell[c]{10\\21\\12\\01} \ar[rdd]\ar[ruu] & \ar[r] \makecell[c]{10\\10\\11\\00}  &\ar[r] \makecell[c]{10\\21\\22\\01} \ar[rdd]\ar[ruu]&  \ar[r] \makecell[c]{00\\11\\11\\01} & \ar[r]  \makecell[c]{10\\21\\22\\12} \ar[rdd]\ar[ruu]&  \ar[r] \makecell[c]{10\\10\\11\\11} & \ar[r]  \makecell[c]{10\\20\\21\\11} \ar[rdd]\ar[ruu]&  \ar[r] \makecell[c]{00\\10\\10\\00} & \ar[r]  \makecell[c]{01\\20\\21\\11} \ar[rdd]\ar[ruu] &  \ar[r] \makecell[c]{01\\10\\11\\11} &\ar[r] \makecell[c]{01\\10\\21\\11} \ar[rdd]\ar[ruu] &  \ar[r] \makecell[c]{00\\00\\10\\00} &\ar[r] \makecell[c]{01\\11\\20\\11} \ar[rdd]\ar[ruu]  &  \makecell[c]{01\\11\\10\\11} &&& \\
 \\
& \makecell[c]{01\\11\\20\\10} \ar[ruu] \ar[rdd] && \makecell[c]{01\\01\\10\\11} \ar[ruu] \ar[rdd] &&  \makecell[c]{11\\01\\10\\10} \ar[ruu] \ar[rdd] && \makecell[c]{01\\02\\11\\10} \ar[ruu] \ar[rdd] &&  \makecell[c]{11\\01\\01\\10}  \ar[ruu] \ar[rdd] &&  \makecell[c]{11\\12\\01\\00} \ar[ruu] \ar[rdd] &&  \makecell[c]{11\\12\\02\\11} \ar[ruu] \ar[rdd] && \makecell[c]{10\\11\\01\\00} \ar[ruu] \ar[rdd] && \makecell[c]{11\\22\\12\\01} \ar[ruu] \ar[rdd] && \makecell[c]{10\\11\\12\\01} \ar[ruu] \ar[rdd] && \makecell[c]{10\\21\\11\\01} \ar[ruu] \ar[rdd]  && \makecell[c]{10\\21\\22\\11} \ar[ruu] \ar[rdd] && \makecell[c]{00\\10\\11\\01} \ar[ruu] \ar[rdd] && \makecell[c]{11\\20\\21\\11} \ar[ruu] \ar[rdd] && \makecell[c]{00\\10\\21\\11} \ar[ruu] \ar[rdd] && \makecell[c]{01\\10\\10\\11} \ar[ruu] \ar[rdd] && \makecell[c]{01\\11\\20\\10} \ar[rdd] &&& \\
 \\
&& \makecell[c]{01\\01\\10\\00} \ar[ruu] \ar[rdd] &&  \makecell[c]{11\\01\\10\\11} \ar[ruu] \ar[rdd] &&  \makecell[c]{00\\01\\10\\10} \ar[ruu] \ar[rdd] && \makecell[c]{01\\01\\01\\10} \ar[ruu] \ar[rdd] &&  \makecell[c]{11\\01\\01\\00} \ar[ruu] \ar[rdd] &&  \makecell[c]{11\\12\\01\\01} \ar[ruu] \ar[rdd] && \makecell[c]{10\\11\\01\\10} \ar[ruu] \ar[rdd] &&  \makecell[c]{00\\11\\01\\00} \ar[ruu] \ar[rdd] && \makecell[c]{11\\11\\12\\01} \ar[ruu] \ar[rdd] && \makecell[c]{10\\11\\11\\01} \ar[ruu] \ar[rdd] && \makecell[c]{10\\21\\11\\11} \ar[ruu] \ar[rdd] && \makecell[c]{00\\10\\11\\00} \ar[ruu] \ar[rdd] && \makecell[c]{01\\10\\11\\01} \ar[ruu] \ar[rdd] && \makecell[c]{10\\10\\21\\11} \ar[ruu] \ar[rdd] && \makecell[c]{00\\10\\10\\11} \ar[ruu] \ar[rdd] && \makecell[c]{01\\10\\10\\10} \ar[ruu]  \ar[rdd] && \makecell[c]{01\\01\\10\\00} \ar[rdd] && \\
\\
&&& \makecell[c]{11\\01\\10\\00} \ar[ruu] \ar[rdd] && \makecell[c]{00\\01\\10\\11}  \ar[ruu] \ar[rdd] && \makecell[c]{00\\00\\00\\10} \ar[ruu] \ar[rdd] && \makecell[c]{01\\01\\01\\00} \ar[ruu] \ar[rdd] && \makecell[c]{11\\01\\01\\01} \ar[ruu] \ar[rdd] &&  \makecell[c]{10\\11\\00\\00} \ar[ruu] \ar[rdd] && \makecell[c]{00\\11\\01\\10} \ar[ruu] \ar[rdd] && \makecell[c]{00\\00\\01\\00} \ar[ruu] \ar[rdd] &&  \makecell[c]{11\\11\\11\\01} \ar[ruu] \ar[rdd] && \makecell[c]{10\\11\\11\\11} \ar[ruu] \ar[rdd] && \makecell[c]{00\\10\\00\\00} \ar[ruu] \ar[rdd] &&  \makecell[c]{01\\10\\11\\00} \ar[ruu] \ar[rdd] && \makecell[c]{00\\00\\11\\01} \ar[ruu] \ar[rdd] && \makecell[c]{10\\10\\10\\11} \ar[ruu] \ar[rdd] &&  \makecell[c]{00\\10\\10\\10} \ar[ruu] \ar[rdd] && \makecell[c]{01\\00\\00\\00} \ar[ruu] \ar[rdd] && \makecell[c]{11\\01\\10\\00}  \ar[rdd] & \\
 \\
&&&& \makecell[c]{00\\01\\10\\00} \ar[ruu] && \makecell[c]{00\\00\\00\\11} \ar[ruu] && \makecell[c]{00\\00\\00\\0-1} \ar[ruu] && \makecell[c]{01\\01\\01\\01} \ar[ruu] && \makecell[c]{10\\00\\00\\00} \ar[ruu] &&  \makecell[c]{00\\11\\00\\00} \ar[ruu]  && \makecell[c]{00\\00\\01\\10} \ar[ruu] && \makecell[c]{00\\00\\00\\-10} \ar[ruu] &&  \makecell[c]{11\\11\\11\\11} \ar[ruu]  && \makecell[c]{0-1\\00\\00\\00} \ar[ruu]   &&  \makecell[c]{01\\10\\00\\00} \ar[ruu]  &&  \makecell[c]{00\\00\\11\\00} \ar[ruu] &&  \makecell[c]{00\\00\\00\\01} \ar[ruu]  &&  \makecell[c]{10\\10\\10\\10} \ar[ruu]  && \makecell[c]{-10\\00\\00\\00} \ar[ruu]   && \makecell[c]{11\\00\\00\\00} \ar[ruu]   && \makecell[c]{00\\01\\10\\00}}}}
\caption{The cluster category $\mathcal{C}^{\leq \xi}_4$ of $kQ^{\leq \xi}_4$, and the objects on the far left are to be identified with the objects on the far right.} \label{the cluster category C4 in A2}
\end{figure}

\pagebreak
%\begin{landscape}
\begin{figure}
\Large
\centerline{
\resizebox{.085\width}{.3\height}{
\xymatrix@C1pt@R35pt{
 &&  \makecell[c]{2_{-5}2_{-3}2_{-1}} \ar[rdd] &&   \makecell[c]{1_{-4}1_{-6}1_{-8}} \ar[rdd]\ar[rdd] &&   \makecell[c]{1_{-2}1_{0}} \ar[rdd] && \makecell[c]{2_{-3}2_{-5}} \ar[rdd] && \makecell[c]{1_{-6}1_{-8}} \ar[rdd] &&  \makecell[c]{1_{0}1_{-2}1_{-4}} \ar[rdd] && \makecell[c]{2_{-3}2_{-5}2_{-7}} \ar[rdd] &&  \makecell[c]{1_{0}1_{-2}2_{-5}2_{-7}2_{-9}}  \ar[rdd] && \makecell[c]{1_{-4}1_{-2}} \ar[rdd] && \makecell[c]{2_{-5}2_{-7}} \ar[rdd]&& \makecell[c]{1_{0}1_{-2}1_{-4}2_{-7}2_{-9}} \ar[rdd] && \makecell[c]{1_{-2}1_{-4}1_{-6}} \ar[rdd] && \makecell[c]{2_{-5}2_{-7}2_{-9}} \ar[rdd]&& \makecell[c]{2_{-3}2_{-1}} \ar[rdd] && \makecell[c]{1_{-4}1_{-6}} \ar[rdd] && \makecell[c]{2_{-7}2_{-9}} \ar[rdd] && \makecell[c]{2_{-5}2_{-3}2_{-1}} && \\
\\
&  \makecell[c]{2_{-1}2_{-3}1_{-6}2_{-9}}  \ar[rdd] \ar[ruu] &&  \makecell[c]{2_{-1}2_{-3}1_{-4}2_{-5}1_{-6}1_{-8}} \ar[rdd]\ar[ruu] && \makecell[c]{1_{0}2_{-3}1_{-6}1_{-8}} \ar[rdd] \ar[ruu] && \makecell[c]{1_{0}1_{-2}2_{-3}2_{-5}} \ar[ruu]\ar[rdd] && \makecell[c]{2_{-3}2_{-5}1_{-6}1_{-8}} \ar[ruu]\ar[rdd] &&  \makecell[c]{1_{0}1_{-2}2_{-5}1_{-8}} \ar[ruu]\ar[rdd] && \makecell[c]{1_{0}1_{-2}2_{-3}1_{-4}2_{-5}2_{-7}} \ar[ruu] \ar[rdd] && \makecell[c]{1_{0}1_{-2}2_{-3}2_{-5}^{2}2_{-7}^{2}2_{-9}} \ar[ruu]\ar[rdd]  &&  \makecell[c]{1_{0}1_{-2}^{2}1_{-4}2_{-5}2_{-7}2_{-9}} \ar[ruu]\ar[rdd] &&  \makecell[c]{1_{-2}1_{-4}2_{-5}2_{-7}} \ar[ruu]\ar[rdd] &&  \makecell[c]{1_{0}1_{-2}1_{-4}2_{-5}2_{-7}^{2}2_{-9}} \ar[ruu]\ar[rdd] &&  \makecell[c]{1_{0}1_{-2}^{2}1_{-4}^{2}1_{-6}2_{-7}2_{-9}} \ar[ruu]\ar[rdd] && \makecell[c]{1_{-2}1_{-4}2_{-5}1_{-6}2_{-7}2_{-9}} \ar[ruu]\ar[rdd] && \makecell[c]{2_{-1}1_{-4}2_{-7}2_{-9}} \ar[ruu]\ar[rdd] && \makecell[c]{2_{-1}2_{-3}1_{-4}1_{-6}} \ar[ruu]\ar[rdd] && \makecell[c]{1_{-4}1_{-6}2_{-7}2_{-9}} \ar[ruu]\ar[rdd]  &&  \makecell[c]{2_{-1}2_{-3}1_{-6}2_{-9}} \ar[ruu] & && \\
 \\
\makecell[c]{2_{-1}2_{-3}1_{-4}1_{-6}^{2}2_{-9}}  \ar[r] \ar[ruu] \ar[rdd] & \makecell[c]{1_{-4}1_{-6}2_{-9}} \ar[r] & \ar[r] \makecell[c]{2_{-1}2_{-3}1_{-4}1_{-6}^{2}1_{-8}2_{-9}} \ar[ruu]\ar[rdd] &  \ar[r] \makecell[c]{2_{-1}2_{-3}1_{-6}1_{-8}} & \ar[r]  \makecell[c]{1_{0}2_{-1}2_{-3}^{2}2_{-5}1_{-6}1_{-8}}  \ar[ruu] \ar[rdd] &  \ar[r]  \makecell[c]{1_{0}2_{-3}2_{-5}}  & \ar[r] \makecell[c]{1_{0}2_{-3}^{2}2_{-5}1_{-6}1_{-8}} \ar[ruu] \ar[rdd]&  \ar[r] \makecell[c]{2_{-3}1_{-6}1_{-8}} & \ar[r] \makecell[c]{1_{0}1_{-2}2_{-3}2_{-5}1_{-6}1_{-8}} \ar[rdd] \ar[ruu]&  \ar[r] \makecell[c]{1_{0}1_{-2}2_{-5}} & \ar[r] \makecell[c]{1_{0}1_{-2}2_{-3}2_{-5}^{2}1_{-8}} \ar[rdd]\ar[ruu]&  \ar[r] \makecell[c]{2_{-3}2_{-5}1_{-8}} & \ar[r] \makecell[c]{1_{0}1_{-2}2_{-3}2_{-5}^{2}2_{-7}1_{-8}} \ar[rdd]\ar[ruu] &  \ar[r] \makecell[c]{1_{0}1_{-2}2_{-5}2_{-7}} & \ar[r] \makecell[c]{1_{0}^{2}1_{-2}^{2}2_{-3}1_{-4}2_{-5}^{2}2_{-7}^{2}2_{-9}} \ar[rdd]\ar[ruu] &  \ar[r] \makecell[c]{1_{0}1_{-2}2_{-3}1_{-4}2_{-5}2_{-7}2_{-9}} & \ar[r]\makecell[c]{1_01_{-2}^{2}2_{-3}1_{-4}2_{-5}^{2}2_{-7}^{2}2_{-9}} \ar[rdd]\ar[ruu] &  \ar[r] \makecell[c]{1_{-2}2_{-5}2_{-7}}  & \ar[r] \makecell[c]{1_{0}1_{-2}^{2}1_{-4}2_{-5}^{2}2_{-7}^{2}2_{-9}} \ar[rdd]\ar[ruu] &  \ar[r] \makecell[c]{1_{0}1_{-2}1_{-4}2_{-5}2_{-7}2_{-9}} & \ar[r] \makecell[c]{1_{0}1_{-2}^{2}1_{-4}^{2}2_{-5}2_{-7}^{2}2_{-9}} \ar[rdd]\ar[ruu] &  \ar[r] \makecell[c]{1_{-2}1_{-4}2_{-7}} &\ar[r] \makecell[c]{1_{0}1_{-2}^{2}1_{-4}^{2}2_{-5}1_{-6}2_{-7}^{2}2_{-9}} \ar[rdd]\ar[ruu]&  \ar[r] \makecell[c]{1_{0}1_{-2}1_{-4}2_{-5}1_{-6}2_{-7}2_{-9}} &\ar[r]   \makecell[c]{1_{0}1_{-2}^{2}1_{-4}^{2}2_{-5}1_{-6}2_{-7}^{2}2_{-9}^{2}} \ar[rdd]\ar[ruu]&  \ar[r]  \makecell[c]{1_{-2}1_{-4}2_{-7}2_{-9}} &\ar[r] \makecell[c]{2_{-1}1_{-2}1_{-4}^{2}1_{-6}2_{-7}2_{-9}} \ar[rdd]\ar[ruu] &  \ar[r] \makecell[c]{2_{-1}1_{-4}1_{-6}} &\ar[r] \makecell[c]{2_{-1}1_{-4}^{2}1_{-6}2_{-7}2_{-9}} \ar[rdd]\ar[ruu] &  \ar[r] \makecell[c]{1_{-4}2_{-7}2_{-9}} &\ar[r] \makecell[c]{2_{-1}2_{-3}1_{-4}1_{-6}2_{-7}2_{-9}} \ar[rdd]\ar[ruu] &  \ar[r] \makecell[c]{2_{-1}2_{-3}1_{-6}} & \ar[r] \makecell[c]{2_{-1}2_{-3}1_{-4}1_{-6}^{2}2_{-9}} \ar[rdd]\ar[ruu]  &  \makecell[c]{1_{-4}1_{-6}2_{-9}} &&& \\
 \\
& \makecell[c]{2_{-1}2_{-3}1_{-4}1_{-6}^{2}1_{-8}} \ar[ruu] \ar[rdd] && \makecell[c]{1_{0}2_{-3}1_{-6}2_{-9}} \ar[ruu] \ar[rdd] && \makecell[c]{2_{-1}2_{-3}^{2}2_{-5}1_{-6}1_{-8}} \ar[ruu] \ar[rdd] && \makecell[c]{1_{0}2_{-3}2_{-5}1_{-6}1_{-8}} \ar[ruu] \ar[rdd] && \makecell[c]{1_{0}1_{-2}2_{-3}2_{-5}1_{-8}} \ar[ruu] \ar[rdd] && \makecell[c]{1_{0}1_{-2}2_{-3}2^{2}_{-5}2_{-7}} \ar[ruu] \ar[rdd] &&  \makecell[c]{1_{0}1_{-2}2_{-3}2_{-5}^{2}2_{-7}1_{-8}2_{-9}} \ar[ruu] \ar[rdd] &&  \makecell[c]{1_{0}1_{-2}^{2}1_{-4}2_{-5}2_{-7}} \ar[ruu] \ar[rdd] &&  \makecell[c]{1_{0}1_{-2}2_{-3}1_{-4}2_{-5}^{2}2_{-7}^{2}2_{-9}} \ar[ruu] \ar[rdd] &&  \makecell[c]{1_{0}1_{-2}^{2}1_{-4}2_{-5}2_{-7}^{2}2_{-9}} \ar[ruu] \ar[rdd] && \makecell[c]{1_{0}1_{-2}^{2}1_{-4}^{2}2_{-5}1_{-6}2_{-7}2_{-9}} \ar[ruu] \ar[rdd] &&  \makecell[c]{1_{-2}1_{-4}2_{-5}2_{-7}^{2}2_{-9}}  \ar[ruu] \ar[rdd] &&   \makecell[c]{1_02_{-1}1_{-2}1_{-4}^{2}1_{-6}2_{-7}2_{-9}} \ar[ruu] \ar[rdd] && \makecell[c]{1_{-2}1_{-4}^{2}1_{-6}2_{-7}2_{-9}} \ar[ruu] \ar[rdd] && \makecell[c]{2_{-1}1_{-4}1_{-6}2_{-7}2_{-9}} \ar[ruu] \ar[rdd] && \makecell[c]{2_{-1}2_{-3}1_{-4}1_{-6}2_{-9}} \ar[ruu] \ar[rdd] && \makecell[c]{2_{-1}2_{-3}1_{-4}1_{-6}^{2}1_{-8}} \ar[rdd] &&& \\
 \\
&& \makecell[c]{1_{0}2_{-3}1_{-6}} \ar[ruu] \ar[rdd] && \makecell[c]{2_{-3}1_{-6}2_{-9}} \ar[ruu] \ar[rdd]  &&  \makecell[c]{2_{-1}2_{-3}2_{-5}1_{-6}1_{-8}}  \ar[ruu] \ar[rdd] &&  \makecell[c]{1_{0}2_{-3}2_{-5}1_{-8}} \ar[ruu] \ar[rdd] && \makecell[c]{1_{0}1_{-2}2_{-3}2_{-5}2_{-7}} \ar[ruu] \ar[rdd] && \makecell[c]{1_{0}1_{-2}2_{-3}2^{2}_{-5}2_{-7}2_{-9}} \ar[ruu] \ar[rdd] && \makecell[c]{1_{-8}2_{-5}1_{-2}} \ar[ruu] \ar[rdd] && \makecell[c]{1_{0}1_{-2}1_{-4}2_{-5}2_{-7}} \ar[ruu] \ar[rdd] && \makecell[c]{1_{0}1_{-2}2_{-3}1_{-4}2_{-5}2_{-7}^{2}2_{-9}}  \ar[ruu] \ar[rdd] && \makecell[c]{1_{0}1_{-2}^{2}1_{-4}2_{-5}1_{-6}2_{-7}2_{-9}} \ar[ruu] \ar[rdd] && \makecell[c]{1_{-2}1_{-4}2_{-5}2_{-7}2_{-9}} \ar[ruu] \ar[rdd] && \makecell[c]{2_{-1}1_{-4}2_{-7}} \ar[ruu] \ar[rdd] && \makecell[c]{1_{0}1_{-2}1_{-4}^{2}1_{-6}2_{-7}2_{-9}} \ar[ruu] \ar[rdd] && \makecell[c]{1_{-2}1_{-4}1_{-6}2_{-7}2_{-9}} \ar[ruu] \ar[rdd] &&  \makecell[c]{2_{-1}1_{-4}1_{-6}2_{-9}} \ar[ruu] \ar[rdd] && \makecell[c]{2_{-1}2_{-3}1_{-4}1_{-6}1_{-8}} \ar[ruu]  \ar[rdd] && \makecell[c]{1_{0}2_{-3}1_{-6}} \ar[rdd] && \\
\\
&&& \makecell[c]{2_{-3}1_{-6}} \ar[ruu] \ar[rdd] &&  \makecell[c]{1_{-6}2_{-9}}  \ar[ruu] \ar[rdd] &&  \makecell[c]{2_{-1}2_{-3}2_{-5}1_{-8}} \ar[ruu] \ar[rdd] && \makecell[c]{1_{0}2_{-3}2_{-5}2_{-7}} \ar[ruu] \ar[rdd] && \makecell[c]{1_{0}1_{-2}2_{-3}2_{-5}2_{-7}2_{-9}} \ar[ruu] \ar[rdd] && \makecell[c]{2_{-5}1_{-2}} \ar[ruu] \ar[rdd] && \makecell[c]{1_{-8}2_{-5}}  \ar[ruu] \ar[rdd] &&  \makecell[c]{1_{0}1_{-2}1_{-4}2_{-7}} \ar[ruu] \ar[rdd] &&  \makecell[c]{1_{0}1_{-2}2_{-3}1_{-4}2_{-5}1_{-6}2_{-7}2_{-9}} \ar[ruu] \ar[rdd] && \makecell[c]{1_{-2}2_{-5}2_{-7}2_{-9}} \ar[ruu] \ar[rdd] && \makecell[c]{2_{-1}1_{-4}} \ar[ruu] \ar[rdd] && \makecell[c]{1_{-4}2_{-7}} \ar[ruu] \ar[rdd] && \makecell[c]{1_{0}1_{-2}1_{-4}1_{-6}2_{-7}2_{-9}} \ar[ruu] \ar[rdd] && \makecell[c]{1_{-2}1_{-4}1_{-6}2_{9}} \ar[ruu] \ar[rdd] && \makecell[c]{2_{-1}1_{-4}1_{-6}1_{-8}} \ar[ruu] \ar[rdd] && \makecell[c]{1_{0}2_{-3}} \ar[ruu] \ar[rdd] &&  \makecell[c]{2_{-3}1_{-6}} \ar[rdd] & \\
 \\
&&&&  \makecell[c]{1_{-6}}  \ar[ruu] && \makecell[c]{2_{-9}} \ar[ruu] && \makecell[c]{2_{-1}2_{-3}2_{-5}2_{-7}} \ar[ruu] && \makecell[c]{1_{0}2_{-3}2_{-5}2_{-7}2_{-9}} \ar[ruu] && \makecell[c]{1_{-2}} \ar[ruu] && \makecell[c]{2_{-5}} \ar[ruu]  && \makecell[c]{1_{-8}} \ar[ruu] && \makecell[c]{1_{-6}1_{-4}1_{-2}1_{0}}  \ar[ruu] &&  \makecell[c]{2_{-3}2_{-5}2_{-7}2_{-9}} \ar[ruu]  && \makecell[c]{2_{-1}} \ar[ruu]   && \makecell[c]{1_{-4}} \ar[ruu]  && \makecell[c]{2_{-7}} \ar[ruu] &&  \makecell[c]{1_{0}1_{-2}1_{-4}1_{-6}2_{-9}} \ar[ruu]  && \makecell[c]{1_{-2}1_{-4}1_{-6}1_{-8}} \ar[ruu]  && \makecell[c]{1_{0}} \ar[ruu]  && \makecell[c]{2_{-3}} \ar[ruu]   && \makecell[c]{1_{-6}}}}}
\caption{All the real prime simple modules in $\mathscr{C}^{\leq \xi}_4$, excluding the frozen real prime simple modules $1_{-8}1_{-6}1_{-4}1_{-2}1_0$ and $2_{-9}2_{-7}2_{-5}2_{-3}2_{-1}$.} \label{all the real prime simple modules C4 in A2}
\end{figure}
%\end{landscape}
\pagebreak

\subsection{The subcategory $\mathscr{C}^{\leq \xi}_2$ in type $A_4$}

The Dynkin diagrams $A_n,D_n$ and $E_6$ admit a nontrivial involution $\sigma$. For example, in type $A_n$, the involution acts as a reflection in the centre point of the Dynkin diagram (which may or may not be a vertex). It induces an automorphism of $\mathbb{Z}\mathcal{P}$ such that $Y_{i,r} \to Y_{\sigma(i),r}$.

In type $A_4$, there are $2^3=8$ height functions $\xi$ up to parameter shift. Under the canonical involution, it is enough to prove Conjecture \ref{rigid objects equal real prime simple modules conjecture} and Conjecture \ref{character conjecture} for $\xi(1,2,3,4)=(-2,-1,0,-1)$, $\xi(1,2,3,4)=(0,-1,-2,-3)$, $\xi(1,2,3,4)=(0,-1,-2,-1)$, and $\xi(1,2,3,4)=(0,-1,0,-1)$.

Let $\xi(1,2,3,4)=(-2,-1,0,-1)$. The initial quiver and the principal quiver of the cluster algebra $K_0(\mathscr{C}^{\leq \xi}_2)$ are shown in Figure \ref{the initial quiver (left) and its principal quiver (right) 1}. The indecomposable rigid objects in $\mathcal{C}^{\leq \xi}_2$ are shown in Figure \ref{the cluster category C21 in A4}. The real prime simple modules (excluding the frozen modules) in $\mathscr{C}^{\leq \xi}_2$ are shown in Figure \ref{all the real prime simple modules C21 in A4}. The two sets are in bijection under the map $L(m) \to K^{\leq \xi}_\ell(m)$.

\begin{figure}
\resizebox{.8\width}{.8\height}{ 
\begin{minipage}{0.5\textwidth}
\begin{xy}
(-25,50)*+{(1,-2)}="a";
(0,50)*+{(2,-1)}="b"; 
(-25,30)*+{(1,-4)}="c";
(0,30)*+{(2,-3)}="d"; 
(-25,10)*+{\fbox{(1,-6)}}="e";
(0,10)*+{\fbox{(2,-5)}}="f";
(25,50)*+{(3,0)}="g";
(50,50)*+{(4,-1)}="h"; 
(25,30)*+{(3,-2)}="i";
(50,30)*+{(4,-3)}="j"; 
(25,10)*+{\fbox{(3,-4)}}="k";
(50,10)*+{\fbox{(4,-5)}}="l";
{\ar "b";"a"}; {\ar "g";"b"}; {\ar "g";"h"};
{\ar "a";"d"};
{\ar "d";"c"};{\ar "i";"d"};{\ar "i";"j"};
{\ar "c";"f"};{\ar "f";"d"};
{\ar "d";"b"};
{\ar "e";"c"};
{\ar "c";"a"};
{\ar "k";"i"};{\ar "i";"g"};
{\ar "l";"j"};{\ar "j";"h"};
{\ar "b";"i"};{\ar "h";"i"};
{\ar "d";"k"};{\ar "j";"k"};
\end{xy}
\end{minipage} 
\qquad 
\begin{minipage}{0.5\textwidth}
\begin{xy}
(-25,50)*+{1}="a";
(0,50)*+{2}="b"; 
(-25,30)*+{5}="c";
(0,30)*+{6}="d"; 
(25,50)*+{3}="g";
(50,50)*+{4}="h"; 
(25,30)*+{7}="i";
(50,30)*+{8}="j"; 
{\ar "b";"a"}; {\ar "g";"b"}; {\ar "g";"h"};
{\ar "a";"d"};
{\ar "d";"c"};{\ar "i";"d"};{\ar "i";"j"};
{\ar "d";"b"};
{\ar "c";"a"};
{\ar "i";"g"};
{\ar "j";"h"};
{\ar "b";"i"};
{\ar "h";"i"};
\end{xy}
\end{minipage}} 
\caption{The initial quiver (left) of $K_0(\mathscr{C}^{\leq \xi}_2)$ and its principal quiver (right).} \label{the initial quiver (left) and its principal quiver (right) 1}
\end{figure}
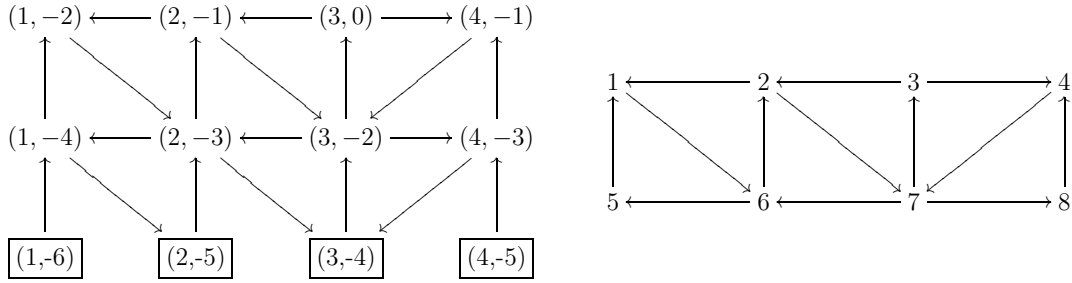

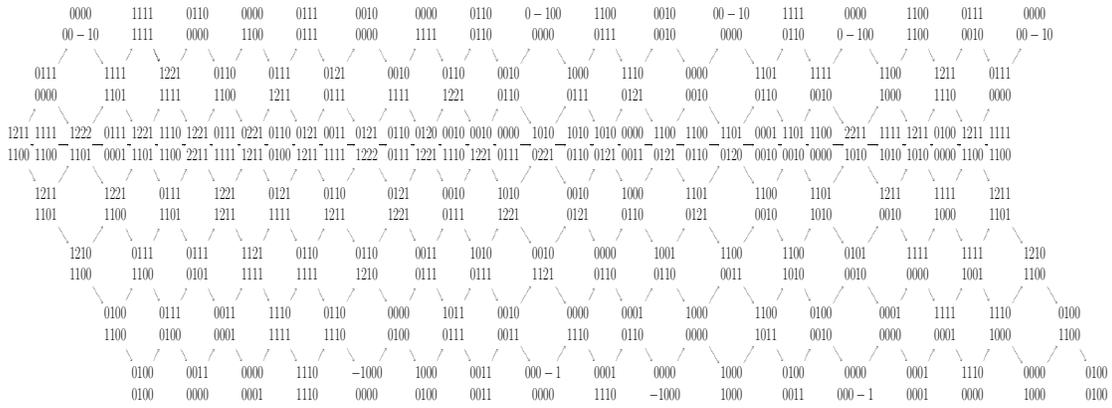
\begin{figure}\huge
\centerline{
\resizebox{.18\width}{.27\height}{ 
\xymatrix@C5pt@R5pt{
&& \makecell[c]{0000\\00-10} \ar[rdd] & & \makecell[c]{1111\\1111}  \ar[rdd]\ar[rdd] && \makecell[c]{0110\\0000} \ar[rdd] &&  \makecell[c]{0000\\1100} \ar[rdd]&&   \makecell[c]{0111\\0111}  \ar[rdd]&& \makecell[c]{0010\\0000} \ar[rdd]&&  \makecell[c]{0000\\1111} \ar[rdd]   && \makecell[c]{0110\\0110} \ar[rdd] && \makecell[c]{0-100\\0000} \ar[rdd] && \makecell[c]{1100\\0111} \ar[rdd] && \makecell[c]{0010\\0010} \ar[rdd]  && \makecell[c]{00-10\\0000} \ar[rdd] && \makecell[c]{1111\\0110} \ar[rdd] &&  \makecell[c]{0000\\0-100} \ar[rdd]&& \makecell[c]{1100\\1100}  \ar[rdd] && \makecell[c]{0111\\0010}  \ar[rdd] && \makecell[c]{0000\\00-10}  && \\
 \\
& \makecell[c]{0111\\0000}   \ar[rdd] \ar[ruu] && \makecell[c]{1111\\1101} \ar[rdd]\ar[ruu] && \makecell[c]{1221\\1111} \ar[rdd] \ar[ruu] &&  \makecell[c]{0110\\1100} \ar[ruu]\ar[rdd] &&  \makecell[c]{0111\\1211}  \ar[ruu]\ar[rdd] &&  \makecell[c]{0121\\0111}  \ar[ruu]\ar[rdd]&&  \makecell[c]{0010\\1111} \ar[ruu]\ar[rdd]&&  \makecell[c]{0110\\1221} \ar[ruu]\ar[rdd]&& \makecell[c]{0010\\0110} \ar[ruu]\ar[rdd]&& \makecell[c]{1000\\0111} \ar[ruu]\ar[rdd] && \makecell[c]{1110\\0121} \ar[ruu]\ar[rdd] && \makecell[c]{0000\\0010} \ar[ruu]\ar[rdd] && \makecell[c]{1101\\0110}\ar[ruu]\ar[rdd] &&  \makecell[c]{1111\\0010} \ar[ruu]\ar[rdd] &&  \makecell[c]{1100\\1000}  \ar[ruu]\ar[rdd] && \makecell[c]{1211\\1110}  \ar[ruu]\ar[rdd]  && \makecell[c]{0111\\0000} \ar[ruu] & && \\
 \\
\makecell[c]{1211\\1100} \ar[r] \ar[ruu] \ar[rdd] & \makecell[c]{1111\\1100}  \ar[r] & \ar[r] \makecell[c]{1222\\1101} \ar[ruu]\ar[rdd] &  \ar[r]  \makecell[c]{0111\\0001} & \ar[r]  \makecell[c]{1221\\1101} \ar[ruu] \ar[rdd] &  \ar[r]  \makecell[c]{1110\\1100} &\ar[r]  \makecell[c]{1221\\2211} \ar[ruu] \ar[rdd] &  \ar[r]  \makecell[c]{0111\\1111} & \ar[r]  \makecell[c]{0221\\1211} \ar[rdd]\ar[ruu]&  \ar[r] \makecell[c]{0110\\0100} & \ar[r] \makecell[c]{0121\\1211} \ar[rdd]\ar[ruu]&  \ar[r] \makecell[c]{0011\\1111} &\ar[r] \makecell[c]{0121\\1222} \ar[rdd]\ar[ruu]&  \ar[r] \makecell[c]{0110\\0111} &\ar[r] \makecell[c]{0120\\1221}  \ar[rdd]\ar[ruu]&  \ar[r] \makecell[c]{0010\\1110}  & \ar[r] \makecell[c]{0010\\1221} \ar[rdd]\ar[ruu] &  \ar[r] \makecell[c]{0000\\0111} &\ar[r] \makecell[c]{1010\\0221} \ar[rdd]\ar[ruu]&  \ar[r] \makecell[c]{1010\\0110} & \ar[r] \makecell[c]{1010\\0121} \ar[rdd]\ar[ruu]&  \ar[r] \makecell[c]{0000\\0011} &\ar[r] \makecell[c]{1100\\0121} \ar[rdd]\ar[ruu]&  \ar[r] \makecell[c]{1100\\0110} &\ar[r] \makecell[c]{1101\\0120} \ar[rdd]\ar[ruu]&  \ar[r] \makecell[c]{0001\\0010} &\ar[r]  \makecell[c]{1101\\0010} \ar[rdd]\ar[ruu]&  \ar[r] \makecell[c]{1100\\0000} &\ar[r]  \makecell[c]{2211\\1010} \ar[rdd]\ar[ruu] &  \ar[r] \makecell[c]{1111\\1010}  &\ar[r] \makecell[c]{1211\\1010}  \ar[rdd]\ar[ruu] &  \ar[r] \makecell[c]{0100\\0000}  &\ar[r] \makecell[c]{1211\\1100} \ar[rdd]\ar[ruu]  &  \makecell[c]{1111\\1100} &&& \\
 \\
&  \makecell[c]{1211\\1101} \ar[ruu] \ar[rdd] &&  \makecell[c]{1221\\1100} \ar[ruu] \ar[rdd]&&  \makecell[c]{0111\\1101} \ar[ruu] \ar[rdd] && \makecell[c]{1221\\1211} \ar[ruu] \ar[rdd]&& \makecell[c]{0121\\1111} \ar[ruu] \ar[rdd] && \makecell[c]{0110\\1211} \ar[ruu] \ar[rdd]&& \makecell[c]{0121\\1221}  \ar[ruu] \ar[rdd] &&  \makecell[c]{0010\\0111} \ar[ruu] \ar[rdd] &&  \makecell[c]{1010\\1221} \ar[ruu] \ar[rdd] && \makecell[c]{0010\\0121} \ar[ruu] \ar[rdd] && \makecell[c]{1000\\0110} \ar[ruu] \ar[rdd] && \makecell[c]{1101\\0121} \ar[ruu] \ar[rdd] && \makecell[c]{1100\\0010} \ar[ruu] \ar[rdd] && \makecell[c]{1101\\1010} \ar[ruu] \ar[rdd] &&   \makecell[c]{1211\\0010} \ar[ruu] \ar[rdd]&& \makecell[c]{1111\\1000} \ar[ruu] \ar[rdd] && \makecell[c]{1211\\1101}  \ar[rdd] &&& \\
 \\
&&  \makecell[c]{1210\\1100} \ar[ruu] \ar[rdd] &&  \makecell[c]{0111\\1100} \ar[ruu] \ar[rdd] &&  \makecell[c]{0111\\0101} \ar[ruu] \ar[rdd] && \makecell[c]{1121\\1111}\ar[ruu] \ar[rdd] && \makecell[c]{0110\\1111} \ar[ruu] \ar[rdd] && \makecell[c]{0110\\1210} \ar[ruu] \ar[rdd] && \makecell[c]{0011\\0111} \ar[ruu] \ar[rdd] && \makecell[c]{1010\\0111} \ar[ruu] \ar[rdd] && \makecell[c]{0010\\1121} \ar[ruu] \ar[rdd] &&  \makecell[c]{0000\\0110} \ar[ruu] \ar[rdd] &&  \makecell[c]{1001\\0110} \ar[ruu] \ar[rdd] && \makecell[c]{1100\\0011} \ar[ruu] \ar[rdd] && \makecell[c]{1100\\1010} \ar[ruu] \ar[rdd] &&  \makecell[c]{0101\\0010} \ar[ruu] \ar[rdd] &&  \makecell[c]{1111\\0000} \ar[ruu] \ar[rdd] && \makecell[c]{1111\\1001}  \ar[ruu]  \ar[rdd] && \makecell[c]{1210\\1100} \ar[rdd] && \\
\\
&&&  \makecell[c]{0100\\1100} \ar[ruu] \ar[rdd]&&  \makecell[c]{0111\\0100} \ar[ruu] \ar[rdd]&&  \makecell[c]{0011\\0001} \ar[ruu] \ar[rdd]&& \makecell[c]{1110\\1111} \ar[ruu] \ar[rdd] && \makecell[c]{0110\\1110} \ar[ruu] \ar[rdd]  && \makecell[c]{0000\\0100}  \ar[ruu] \ar[rdd] && \makecell[c]{1011\\0111}  \ar[ruu] \ar[rdd]&& \makecell[c]{0010\\0011} \ar[ruu] \ar[rdd] && \makecell[c]{0000\\1110} \ar[ruu] \ar[rdd] && \makecell[c]{0001\\0110} \ar[ruu] \ar[rdd] && \makecell[c]{1000\\0000} \ar[ruu] \ar[rdd] && \makecell[c]{1100\\1011} \ar[ruu] \ar[rdd] && \makecell[c]{0100\\0010} \ar[ruu] \ar[rdd] &&  \makecell[c]{0001\\0000} \ar[ruu] \ar[rdd] &&  \makecell[c]{1111\\0001}  \ar[ruu] \ar[rdd]&& \makecell[c]{1110\\1000} \ar[ruu] \ar[rdd] && \makecell[c]{0100\\1100} \ar[rdd] & \\
 \\
&&&&  \makecell[c]{0100\\0100} \ar[ruu] &&  \makecell[c]{0011\\0000} \ar[ruu] && \makecell[c]{0000\\0001} \ar[ruu] && \makecell[c]{1110\\1110} \ar[ruu] && \makecell[c]{-1000\\0000}\ar[ruu] && \makecell[c]{1000\\0100}\ar[ruu]  &&  \makecell[c]{0011\\0011} \ar[ruu] &&  \makecell[c]{000-1\\0000} \ar[ruu] && \makecell[c]{0001\\1110} \ar[ruu]  && \makecell[c]{0000\\-1000} \ar[ruu]   && \makecell[c]{1000\\1000} \ar[ruu]  && \makecell[c]{0100\\0011} \ar[ruu] && \makecell[c]{0000\\000-1} \ar[ruu]  &&  \makecell[c]{0001\\0001} \ar[ruu]  &&  \makecell[c]{1110\\0000}  \ar[ruu]   && \makecell[c]{0000\\1000} \ar[ruu]   && \makecell[c]{0100\\0100}
}}}
\caption{The cluster category $\mathcal{C}^{\leq \xi}_2$ of $kQ^{\leq \xi}_2$, and the objects on the far left are to be identified with the objects on the far right.} \label{the cluster category C21 in A4}
\end{figure}

\begin{figure} \large
\resizebox{.1\width}{.4\height}{ 
\xymatrix@C5pt@R15pt{
&&  \makecell[c]{3_{-2}3_0} \ar[rdd] && \makecell[c]{3_{0}1_{-4}4_{-3}1_{-6}4_{-5}}  \ar[rdd]\ar[rdd] && \makecell[c]{2_{-3}} \ar[rdd]  &&  \makecell[c]{3_03_{-2}1_{-6}} \ar[rdd]  &&   \makecell[c]{3_0 2_{-3}4_{-3}2_{-5}4_{-5}}  \ar[rdd] && \makecell[c]{3_{-2}} \ar[rdd]  &&  \makecell[c]{3_03_{-2}1_{-6}4_{-5}} \ar[rdd]   && \makecell[c]{2_{-3}2_{-5}} \ar[rdd] && \makecell[c]{2_{-1}} \ar[rdd] && \makecell[c]{3_0 3_{-2} 2_{-5} 4_{-5}} \ar[rdd] && \makecell[c]{3_{-4}3_{-2}} \ar[rdd]  && \makecell[c]{3_0} \ar[rdd] && \makecell[c]{2_{-5}} \ar[rdd] &&  \makecell[c]{2_{-3}2_{-1}} \ar[rdd]&& \makecell[c]{1_{-6}1_{-4}3_0}  \ar[rdd] && \makecell[c]{3_{-4}}  \ar[rdd] && \makecell[c]{3_{-2}3_0}  && \\
 \\
& \makecell[c]{2_{-3}3_04_{-3}} \ar[rdd] \ar[ruu]  &&  \makecell[c]{3^2_{0}1_{-4}3_{-2}4_{-3}1_{-6}4_{-5}} \ar[rdd]\ar[ruu] &&   \makecell[c]{3_{0}1_{-4}2_{-3}4_{-3}1_{-6}4_{-5}} \ar[rdd] \ar[ruu]  &&  \makecell[c]{2_{-3}3_03_{-2}1_{-6}}  \ar[ruu]\ar[rdd] &&  \makecell[c]{3^{2}_0 3_{-2} 2_{-3}4_{-3}1_{-6}2_{-5}4_{-5}}  \ar[ruu]\ar[rdd]  &&  \makecell[c]{3_03_{-2}2_{-3}4_{-3}2_{-5}4_{-5}}  \ar[ruu]\ar[rdd]&&  \makecell[c]{3_03^{2}_{-2}1_{-6}4_{-5}} \ar[ruu]\ar[rdd]&&  \makecell[c]{3_03_{-2}2_{-3}2_{-5}1_{-6}4_{-5}} \ar[ruu]\ar[rdd]&& \makecell[c]{1_{-2}2_{-5}3_{-2}} \ar[ruu]\ar[rdd]&& \makecell[c]{2_{-1} 3_0 3_{-2} 2_{-5} 4_{-5}} \ar[ruu]\ar[rdd] && \makecell[c]{2_{-5} 3_{-2} 4_{-5}} \ar[ruu]\ar[rdd] && \makecell[c]{2_{-1}3_{-4}4_{-1}} \ar[ruu]\ar[rdd] && \makecell[c]{2_{-5}3_0}\ar[ruu]\ar[rdd] &&  \makecell[c]{1_{-4}2_{-1}3_{-4}} \ar[ruu]\ar[rdd] &&  \makecell[c]{1_{-6}1_{-4}2_{-3}2_{-1}3_0}  \ar[ruu]\ar[rdd] && \makecell[c]{1_{-6}1_{-4}3_03_{-4}}  \ar[ruu]\ar[rdd]  && \makecell[c]{2_{-3}3_04_{-3}} \ar[ruu] & && \\
 \\
 \makecell[c]{3_0^{2} 1_{-4} 2_{-3} 4_{-3} 1_{-6}}  \ar[r] \ar[ruu] \ar[rdd] & \makecell[c]{3_{0}1_{-4}4_{-3}1_{-6}}  \ar[r] & \ar[r] \makecell[c]{3_{0}^{2}1_{-4}2_{-3}4_{-3}^{2}1_{-6}4_{-5}} \ar[ruu]\ar[rdd] &  \ar[r] \makecell[c]{3_{0}2_{-3}4_{-3}4_{-5}} & \ar[r]  \makecell[c]{3^{2}_{0}1_{-4}3_{-2}2_{-3}4_{-3}1_{-6}4_{-5}} \ar[ruu] \ar[rdd] &   \ar[r]  \makecell[c]{3_{0}1_{-4}3_{-2}1_{-6}} & \ar[r]  \makecell[c]{3^2_01_{-4}2_{-3}3_{-2}4_{-3}1^2_{-6}4_{-5}} \ar[ruu] \ar[rdd] &  \ar[r]  \makecell[c]{3_02_{-3}4_{-3}1_{-6}4_{-5}} & \ar[r]  \makecell[c]{3^{2}_0 2^{2}_{-3} 3_{-2}4_{-3}2_{-5}1_{-6}4_{-5}} \ar[rdd]\ar[ruu]&  \ar[r] \makecell[c]{3_0 2_{-3} 3_{-2}2_{-5}} & \ar[r] \makecell[c]{3^{2}_03^{2}_{-2}2_{-3}2_{-5}4_{-3}1_{-6}4_{-5}} \ar[rdd]\ar[ruu] & \ar[r] \makecell[c]{3_03_{-2}4_{-3}1_{-6}4_{-5}} &\ar[r] \makecell[c]{3^{2}_03^{2}_{-2}2_{-3}4_{-3}1_{-6}2_{-5}4^2_{-5}} \ar[rdd]\ar[ruu]&  \ar[r] \makecell[c]{3_03_{-2}2_{-3}2_{-5}4_{-5}} &\ar[r] \makecell[c]{3_03^{2}_{-2}2_{-3}1_{-6}2_{-5}4_{-5}}  \ar[rdd]\ar[ruu]&  \ar[r] \makecell[c]{3_{-2}1_{-6}}  & \ar[r] \makecell[c]{1_{-2}3_03^{2}_{-2}1_{-6}2_{-5}4_{-5}} \ar[rdd]\ar[ruu] &  \ar[r] \makecell[c]{1_{-2}3_03_{-2}2_{-5}4_{-5}} &\ar[r] \makecell[c]{1_{-2} 3_0 3^{2}_{-2} 2^{2}_{-5} 4_{-5}} \ar[rdd]\ar[ruu]&  \ar[r] \makecell[c]{3_{-2} 2_{-5}} & \ar[r] \makecell[c]{2_{-1}2_{-5}3_{-2}4_{-5}} \ar[rdd]\ar[ruu]&  \ar[r] \makecell[c]{4_{-5}2_{-1}} &\ar[r] \makecell[c]{ 2_{-1}2_{-5}4_{-1}4_{-5}} \ar[rdd]\ar[ruu]&  \ar[r] \makecell[c]{2_{-5}4_{-1}} &\ar[r] \makecell[c]{2_{-1}3_{-4}4_{-1}2_{-5}} \ar[rdd]\ar[ruu]&  \ar[r] \makecell[c]{3_{-4}2_{-1}} &\ar[r]  \makecell[c]{1_{-4}2_{-1}3_{-4}3_0} \ar[rdd]\ar[ruu]&  \ar[r] \makecell[c]{1_{-4}3_0} &\ar[r]  \makecell[c]{1_{-6}1^{2}_{-4}2_{-1}3_{-4}3_0} \ar[rdd]\ar[ruu] &  \ar[r] \makecell[c]{1_{-6}1_{-4}2_{-1}3_{-4}}  &\ar[r] \makecell[c]{1_{-6}1_{-4}2_{-3}2_{-1}3_{-4}3_0}  \ar[rdd]\ar[ruu] &  \ar[r] \makecell[c]{2_{-3}3_0}  &\ar[r] \makecell[c]{1_{-6}1_{-4}2_{-3}3^{2}_04_{-3}} \ar[rdd]\ar[ruu]  &  \makecell[c]{1_{-6}1_{-4}3_04_{-3}} &&& \\
 \\
& \makecell[c]{1_{-6}1_{-4}2_{-3}3^{2}_{0}4_{-3}4_{-5}} \ar[ruu] \ar[rdd] &&  \makecell[c]{3_{0}1_{-4}2_{-3}4_{-3}1_{-6}}  \ar[ruu] \ar[rdd] &&  \makecell[c]{2_{-3}3^2_{0}3_{-2}4_{-3}4_{-5}1_{-6}} \ar[ruu] \ar[rdd] && \makecell[c]{3^2_{0}1_{-4}2_{-3}3_{-2}4_{-3}1_{-6}2_{-5}4_{-5}} \ar[ruu] \ar[rdd] && \makecell[c]{3_03_{-2}2_{-3}4_{-3}1_{-6}4_{-5}} \ar[ruu] \ar[rdd] && \makecell[c]{3^2_03^2_{-2}2_{-3}1_{-6}2_{-5}4_{-5}} \ar[ruu] \ar[rdd] && \makecell[c]{3_03_{-2}2_{-3}2_{-5}4_{-3}4_{-5}1_{-6}}  \ar[ruu] \ar[rdd] &&  \makecell[c]{1_{-2}3_03^{2}_{-2}2_{-5}4_{-5}} \ar[ruu] \ar[rdd] &&  \makecell[c]{3_03^{2}_{-2}1_{-6}2_{-5}4_{-5}} \ar[ruu] \ar[rdd] && \makecell[c]{1_{-2} 2_{-5} 3_{-2} 4_{-5}} \ar[ruu] \ar[rdd] && \makecell[c]{2_{-1}2_{-5}4_{-1}} \ar[ruu] \ar[rdd] && \makecell[c]{2_{-1}2_{-5}4_{-5}} \ar[ruu] \ar[rdd] && \makecell[c]{1_{-4}2_{-1}3_{-4}4_{-1}} \ar[ruu] \ar[rdd] && \makecell[c]{1_{-6}1_{-4}2_{-1}3_{-4}3_0} \ar[ruu] \ar[rdd] &&  \makecell[c]{1_{-4}3_03_{-4}} \ar[ruu] \ar[rdd]&& \makecell[c]{1_{-6}1_{-4}2_{-3}2_{-1}3_04_{-3}} \ar[ruu] \ar[rdd] && \makecell[c]{1_{-6}1_{-4}2_{-3}3^{2}_04_{-3}4_{-5}}  \ar[rdd] &&& \\
 \\
&&  \makecell[c]{3_{0}1_{-4}2_{-3}1_{-6}} \ar[ruu] \ar[rdd] && \makecell[c]{3_{0}2_{-3}4_{-3}1_{-6}}   \ar[ruu] \ar[rdd]  &&  \makecell[c]{3^2_{0}2_{-3}3_{-2}4_{-3}2_{-5}4_{-5}} \ar[ruu] \ar[rdd] &&  \makecell[c]{3_0 3_{-2}4_{-3}1_{-4}1_{-6}4_{-5}} \ar[ruu] \ar[rdd]  && \makecell[c]{3_0 3_{-2}2_{-3}1_{-6}4_{-5}} \ar[ruu] \ar[rdd] && \makecell[c]{2_{-3}3_03_{-2}2_{-5}1_{-6}} \ar[ruu] \ar[rdd] && \makecell[c]{1_{-2}3_03_{-2}4_{-3}2_{-5}4_{-5}} \ar[ruu] \ar[rdd] && \makecell[c]{3_03^{2}_{-2}2_{-5}4_{-5}} \ar[ruu] \ar[rdd] && \makecell[c]{1_{-6}3_{-2}4_{-5}} \ar[ruu] \ar[rdd] &&  \makecell[c]{1_{-2} 2_{-5} 4_{-1}} \ar[ruu] \ar[rdd] &&  \makecell[c]{2_{-1}2_{-5}} \ar[ruu] \ar[rdd] && \makecell[c]{1_{-4}2_{-1}4_{-5}} \ar[ruu] \ar[rdd] && \makecell[c]{1_{-6}1_{-4}2_{-1}3_{-4}4_{-1}} \ar[ruu] \ar[rdd] &&  \makecell[c]{3_{-4}3_0} \ar[ruu] \ar[rdd] &&  \makecell[c]{1_{-4}3_04_{-3}} \ar[ruu] \ar[rdd] && \makecell[c]{1_{-6}1_{-4}2_{-3}2_{-1}4_{-5}4_{-3}3_0}  \ar[ruu]  \ar[rdd] && \makecell[c]{1_{-6}1_{-4}2_{-3}3_0} \ar[rdd] && \\
\\
&&&  \makecell[c]{3_02_{-3}1_{-6}} \ar[ruu] \ar[rdd]&&   \makecell[c]{3_{0}2_{-3}4_{-3}2_{-5}} \ar[ruu] \ar[rdd]  &&  \makecell[c]{3_{0}3_{-2}4_{-3}4_{-5}} \ar[ruu] \ar[rdd]&& \makecell[c]{3_{0}3_{-2}4_{-5}1_{-4}1_{-6}} \ar[ruu] \ar[rdd]  && \makecell[c]{2_{-3}1_{-6}} \ar[ruu] \ar[rdd]  && \makecell[c]{1_{-2}3_03_{-2}2_{-5}}  \ar[ruu] \ar[rdd] && \makecell[c]{3_03_{-2}4_{-3}2_{-5}4_{-5}}  \ar[ruu] \ar[rdd]&& \makecell[c]{3_{-2}4_{-5}} \ar[ruu] \ar[rdd] && \makecell[c]{1_{-6}4_{-1}} \ar[ruu] \ar[rdd] && \makecell[c]{1_{-2}2_{-5}} \ar[ruu] \ar[rdd] && \makecell[c]{1_{-4}2_{-1}} \ar[ruu] \ar[rdd] && \makecell[c]{1_{-6}1_{-4}2_{-1}4_{-5}} \ar[ruu] \ar[rdd] && \makecell[c]{3_{-4}4_{-1}} \ar[ruu] \ar[rdd] &&  \makecell[c]{4_{-3}3_0} \ar[ruu] \ar[rdd] &&  \makecell[c]{4_{-5}4_{-3}3_01_{-4}}  \ar[ruu] \ar[rdd]&& \makecell[c]{1_{-6}1_{-4}2_{-3}2_{-1}} \ar[ruu] \ar[rdd] && \makecell[c]{1_{-6}2_{-3}3_0} \ar[rdd] & \\
 \\
&&&&  \makecell[c]{2_{-5}2_{-3}3_{0}} \ar[ruu] &&  \makecell[c]{4_{-3}} \ar[ruu]  && \makecell[c]{3_{0}3_{-2}4_{-5}} \ar[ruu] && \makecell[c]{1_{-4}1_{-6}} \ar[ruu] && \makecell[c]{1_{-2}}  \ar[ruu]   && \makecell[c]{3_03_{-2}2_{-5}}\ar[ruu]  &&  \makecell[c]{4_{-3}4_{-5}} \ar[ruu] &&  \makecell[c]{4_{-1}} \ar[ruu] && \makecell[c]{1_{-6}} \ar[ruu]  && \makecell[c]{1_{-4}1_{-2}} \ar[ruu]   && \makecell[c]{1_{-6}1_{-4}2_{-1}} \ar[ruu]  && \makecell[c]{4_{-5}} \ar[ruu] && \makecell[c]{4_{-3}4_{-1}} \ar[ruu]  &&  \makecell[c]{4_{-5}4_{-3}3_0} \ar[ruu]  &&  \makecell[c]{1_{-4}}  \ar[ruu]   && \makecell[c]{1_{-6}2_{-3}2_{-1}} \ar[ruu]   && \makecell[c]{2_{-5}2_{-3}3_0}
}}
\caption{All the real prime simple modules in $\mathscr{C}^{\leq \xi}_2$, excluding the frozen real prime simple modules $1_{-6}1_{-4}1_{-2}$,  $2_{-5}2_{-3}2_{-1}$, $3_{-4}3_{-2}3_{0}$, and $4_{-5}4_{-3}4_{-1}$.} \label{all the real prime simple modules C21 in A4}
\end{figure}
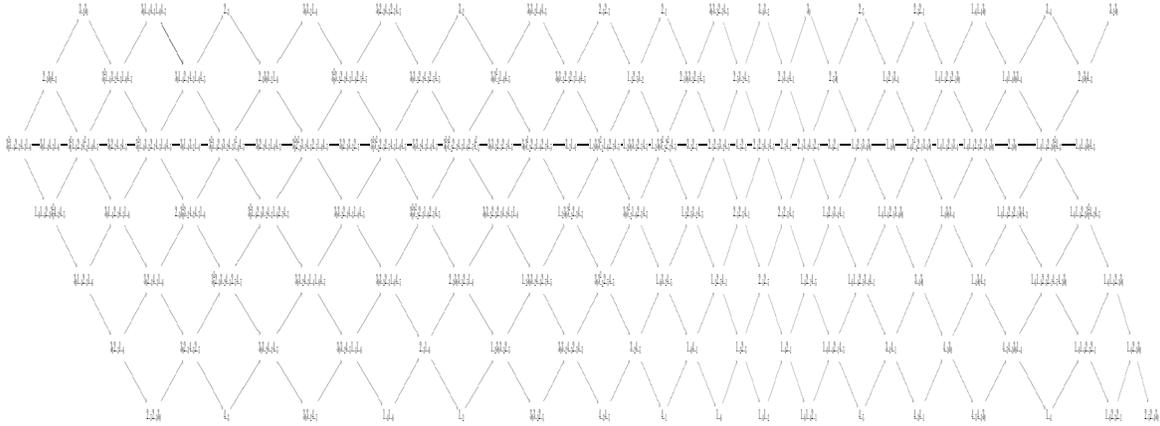

Let $\xi(1,2,3,4)=(0,-1,-2,-3)$. The initial quiver and the principal quiver of the cluster algebra $K_0(\mathscr{C}^{\leq \xi}_2)$ are shown in Figure \ref{the initial quiver (left) and its principal quiver (right) 2}. The indecomposable rigid objects in $\mathcal{C}^{\leq \xi}_2$ are shown in Figure \ref{the cluster category C22 in A4}. The real prime simple modules (excluding the frozen modules) in $\mathscr{C}^{\leq \xi}_2$ are shown in Figure \ref{all the real prime simple modules C22 in A4}.  The two sets are in bijection under the map $L(m) \to K^{\leq \xi}_\ell(m)$.

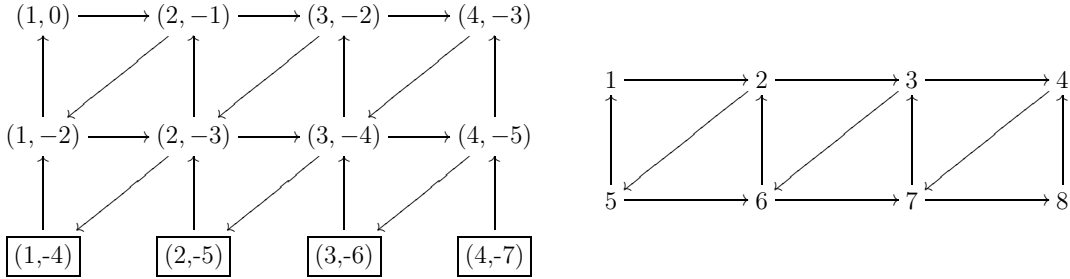
\begin{figure}
\resizebox{.8\width}{.8\height}{ 
\begin{minipage}{0.5\textwidth}
\begin{xy}
(-25,50)*+{(1,0)}="a";
(0,50)*+{(2,-1)}="b"; 
(-25,30)*+{(1,-2)}="c";
(0,30)*+{(2,-3)}="d"; 
(-25,10)*+{\fbox{(1,-4)}}="e";
(0,10)*+{\fbox{(2,-5)}}="f";
(25,50)*+{(3,-2)}="g";
(50,50)*+{(4,-3)}="h"; 
(25,30)*+{(3,-4)}="i";
(50,30)*+{(4,-5)}="j"; 
(25,10)*+{\fbox{(3,-6)}}="k";
(50,10)*+{\fbox{(4,-7)}}="l";
{\ar "a";"b"}; {\ar "b";"g"}; {\ar "g";"h"};
{\ar "b";"c"};
{\ar "c";"d"};{\ar "d";"i"};{\ar "i";"j"};
{\ar "d";"e"};{\ar "f";"d"};
{\ar "d";"b"};
{\ar "e";"c"};
{\ar "c";"a"};
{\ar "k";"i"};{\ar "i";"g"};
{\ar "l";"j"};{\ar "j";"h"};
{\ar "g";"d"};{\ar "h";"i"};
{\ar "i";"f"};{\ar "j";"k"};
\end{xy}
\end{minipage} 
\qquad 
\begin{minipage}{0.5\textwidth}
\begin{xy}
(-25,50)*+{1}="a";
(0,50)*+{2}="b"; 
(-25,30)*+{5}="c";
(0,30)*+{6}="d"; 
(25,50)*+{3}="g";
(50,50)*+{4}="h"; 
(25,30)*+{7}="i";
(50,30)*+{8}="j"; 
{\ar "a";"b"}; {\ar "b";"g"}; {\ar "g";"h"};
{\ar "b";"c"};
{\ar "c";"d"};{\ar "d";"i"};{\ar "i";"j"};
{\ar "d";"b"};
{\ar "c";"a"};
{\ar "i";"g"};
{\ar "j";"h"};
{\ar "g";"d"};
{\ar "h";"i"};
\end{xy}
\end{minipage}} 
\caption{The initial quiver (left) of $K_0(\mathscr{C}^{\leq \xi}_2)$ and its principal quiver (right).} \label{the initial quiver (left) and its principal quiver (right) 2}
\end{figure}

\begin{figure}\Large
\centerline{
\resizebox{.2\width}{.3\height}{ 
\xymatrix@C5pt@R5pt{ 
 &&  \makecell[c]{0000\\00-10} \ar[rdd] && \makecell[c]{0011\\0011} \ar[rdd]\ar[rdd] && \makecell[c]{0110\\1100} \ar[rdd] && \makecell[c]{0000\\0-100} \ar[rdd] &&  \makecell[c]{0111\\0111} \ar[rdd] &&  \makecell[c]{1110\\0000} \ar[rdd] && \makecell[c]{0000\\0011} \ar[rdd] && \makecell[c]{0110\\0110}\ar[rdd] && \makecell[c]{1100\\0000} \ar[rdd] && \ar[rdd]  \makecell[c]{0000\\0111} &&  \makecell[c]{1110\\1110} \ar[rdd] && \makecell[c]{00-10\\0000}  \ar[rdd] &&  \makecell[c]{0011\\0110} \ar[rdd] && \makecell[c]{1100\\1100} \ar[rdd] && \makecell[c]{0-100\\0000} \ar[rdd] && \makecell[c]{0111\\1110} \ar[rdd] && \makecell[c]{0000\\00-10} \\
 \\
 & \makecell[c]{0111\\1100} \ar[rdd] \ar[ruu] && \makecell[c]{0011\\0001} \ar[rdd]\ar[ruu] && \makecell[c]{0121\\1111} \ar[rdd] \ar[ruu] && \makecell[c]{0110\\1000} \ar[ruu]\ar[rdd] && \makecell[c]{0111\\0011} \ar[ruu]\ar[rdd] && \makecell[c]{1221\\0111} \ar[ruu]\ar[rdd] && \makecell[c]{1110\\0011}  \ar[ruu]\ar[rdd] && \makecell[c]{0110\\0121}  \ar[ruu]\ar[rdd] && \makecell[c]{1210\\0110}\ar[ruu]\ar[rdd] && \makecell[c]{1100\\0111}  \ar[ruu]\ar[rdd] &&  \makecell[c]{1110\\1221}  \ar[ruu]\ar[rdd] &&  \makecell[c]{1100\\1110}  \ar[ruu]\ar[rdd] && \makecell[c]{0001\\0110} \ar[ruu]\ar[rdd] && \makecell[c]{1111\\1210} \ar[ruu]\ar[rdd] && \makecell[c]{1000\\1100}\ar[ruu]\ar[rdd] && \makecell[c]{0011\\1110} \ar[ruu]\ar[rdd]  && \makecell[c]{0111\\1100} \ar[ruu] \\
 \\
\makecell[c]{0011\\1100} \ar[r] \ar[ruu] \ar[rdd] & \ar[r]  \makecell[c]{0011\\0000}  & \ar[r]  \makecell[c]{0122\\1101}  \ar[ruu]\ar[rdd] &  \ar[r] \makecell[c]{0111\\1101} & \ar[r] \makecell[c]{0121\\1101} \ar[ruu] \ar[rdd]  &  \ar[r] \makecell[c]{0010\\0000} & \ar[r] \makecell[c]{0121\\1011} \ar[ruu] \ar[rdd]&  \ar[r] \makecell[c]{0111\\1011}  & \ar[r] \makecell[c]{0221\\1011} \ar[rdd]\ar[ruu]&  \ar[r]  \makecell[c]{0110\\0000}  & \ar[r]  \makecell[c]{1221\\0011}  \ar[rdd]\ar[ruu]&  \ar[r]  \makecell[c]{1111\\0011} & \ar[r] \makecell[c]{1221\\0122} \ar[rdd]\ar[ruu] &  \ar[r] \makecell[c]{0110\\0111} &\ar[r] \makecell[c]{1220\\0121} \ar[rdd]\ar[ruu] &  \ar[r] \makecell[c]{1110\\0010} & \ar[r]\makecell[c]{1210\\0121} \ar[rdd]\ar[ruu]&  \ar[r] \makecell[c]{0100\\0111} & \ar[r] \makecell[c]{1210\\0221} \ar[rdd]\ar[ruu] &  \ar[r]  \makecell[c]{1110\\0110} &\ar[r] \makecell[c]{2210\\1221} \ar[rdd]\ar[ruu]  &  \ar[r] \makecell[c]{1100\\1111}  & \ar[r]  \makecell[c]{1100\\1221}  \ar[rdd]\ar[ruu] &  \ar[r] \makecell[c]{0000\\0110} & \ar[r] \makecell[c]{1101\\1220} \ar[rdd]\ar[ruu] &  \ar[r] \makecell[c]{1101\\1110} & \ar[r] \makecell[c]{1101\\1210} \ar[rdd]\ar[ruu] &  \ar[r] \makecell[c]{0000\\0100} & \ar[r] \makecell[c]{1011\\1210} \ar[rdd]\ar[ruu] &  \ar[r] \makecell[c]{1011\\1110} & \ar[r] \makecell[c]{1011\\2210} \ar[rdd]\ar[ruu] &  \ar[r] \makecell[c]{0000\\1100} & \ar[r] \makecell[c]{0011\\1100} \ar[rdd]\ar[ruu]  &  \makecell[c]{0011\\0000} \\
 \\
& \makecell[c]{0011\\1101}  \ar[ruu] \ar[rdd] && \makecell[c]{0121\\1100} \ar[ruu] \ar[rdd] && \makecell[c]{0111\\1001} \ar[ruu] \ar[rdd] && \makecell[c]{0121\\0011}  \ar[ruu] \ar[rdd] && \makecell[c]{1221\\1011} \ar[ruu] \ar[rdd] && \makecell[c]{0110\\0011} \ar[ruu] \ar[rdd] && \makecell[c]{1221\\0121} \ar[ruu] \ar[rdd]  && \makecell[c]{1210\\0111} \ar[ruu] \ar[rdd] && \makecell[c]{1110\\0121} \ar[ruu] \ar[rdd] &&  \makecell[c]{1210\\1221} \ar[ruu] \ar[rdd] && \makecell[c]{1100\\0110} \ar[ruu] \ar[rdd] && \makecell[c]{1101\\1221} \ar[ruu] \ar[rdd] && \makecell[c]{1100\\1210} \ar[ruu] \ar[rdd] && \makecell[c]{1001\\1110} \ar[ruu] \ar[rdd] && \makecell[c]{0011\\1210} \ar[ruu] \ar[rdd] && \makecell[c]{1011\\1100}\ar[ruu] \ar[rdd] && \makecell[c]{0011\\1101} \ar[rdd] \\
 \\
&& \makecell[c]{0010\\1100} \ar[ruu] \ar[rdd] && \makecell[c]{0111\\1000} \ar[ruu] \ar[rdd] &&  \makecell[c]{0111\\0001} \ar[ruu] \ar[rdd] && \makecell[c]{1121\\0011} \ar[ruu] \ar[rdd] && \makecell[c]{0110\\1011}  \ar[ruu] \ar[rdd] &&  \makecell[c]{0110\\0010} \ar[ruu] \ar[rdd] && \makecell[c]{1211\\0111} \ar[ruu] \ar[rdd] && \makecell[c]{1110\\0111} \ar[ruu] \ar[rdd] && \makecell[c]{1110\\1121} \ar[ruu] \ar[rdd] &&  \makecell[c]{0100\\0110} \ar[ruu] \ar[rdd] &&  \makecell[c]{1101\\0110} \ar[ruu] \ar[rdd] &&  \makecell[c]{1100\\1211} \ar[ruu] \ar[rdd] && \makecell[c]{1000\\1110} \ar[ruu] \ar[rdd] && \makecell[c]{0001\\1110} \ar[ruu] \ar[rdd] && \makecell[c]{0011\\0100} \ar[ruu] \ar[rdd] && \makecell[c]{1011\\1101} \ar[ruu]  \ar[rdd] && \makecell[c]{0010\\1100} \ar[rdd] \\
\\
&&& \makecell[c]{0000\\1000} \ar[ruu] \ar[rdd] &&  \makecell[c]{0111\\0000} \ar[ruu] \ar[rdd]  && \makecell[c]{1111\\0001} \ar[ruu] \ar[rdd] &&  \makecell[c]{0010\\0011} \ar[ruu] \ar[rdd] && \makecell[c]{0110\\1010} \ar[ruu] \ar[rdd] && \makecell[c]{0100\\0000} \ar[ruu] \ar[rdd] && \makecell[c]{1111\\0111} \ar[ruu] \ar[rdd] && \makecell[c]{1110\\1111} \ar[ruu] \ar[rdd] && \makecell[c]{0000\\0010} \ar[ruu] \ar[rdd] &&  \makecell[c]{0101\\0110} \ar[ruu] \ar[rdd] &&  \makecell[c]{1100\\0100}  \ar[ruu] \ar[rdd] &&  \makecell[c]{1000\\1111} \ar[ruu] \ar[rdd] && \makecell[c]{0000\\1110}\ar[ruu] \ar[rdd] && \makecell[c]{0001\\0000} \ar[ruu] \ar[rdd] && \makecell[c]{0011\\0101} \ar[ruu] \ar[rdd] && \makecell[c]{1010\\1100} \ar[ruu] \ar[rdd] && \makecell[c]{0000\\1000} \ar[rdd] \\
 \\
&&&& \makecell[c]{-1000\\0000} \ar[ruu] && \makecell[c]{1111\\0000} \ar[ruu] && \makecell[c]{0000\\0001} \ar[ruu] && \makecell[c]{0010\\0010} \ar[ruu] &&  \makecell[c]{0100\\1000} \ar[ruu] &&  \makecell[c]{0000\\-1000} \ar[ruu]  && \makecell[c]{1111\\1111} \ar[ruu] && \makecell[c]{000-1\\0000} \ar[ruu] && \makecell[c]{0001\\0010}  \ar[ruu]  &&  \makecell[c]{0100\\0100} \ar[ruu]  &&   \makecell[c]{1000\\0000} \ar[ruu]  && \makecell[c]{0000\\1111} \ar[ruu] && \makecell[c]{0000\\000-1}  \ar[ruu]  && \makecell[c]{0001\\0001} \ar[ruu]  &&  \makecell[c]{0010\\0100} \ar[ruu]  &&  \makecell[c]{1000\\1000} \ar[ruu]   && \makecell[c]{-1000\\0000}
}}}
\caption{The cluster category $\mathcal{C}^{\leq \xi}_2$ of $kQ^{\leq \xi}_2$, and the objects on the far left are to be identified with the objects on the far right.} \label{the cluster category C22 in A4}
\end{figure}

\begin{figure}
\centerline{
\resizebox{.1\width}{.4\height}{ 
\xymatrix{ 
 &&  \makecell[c]{3_{-4}3_{-2}} \ar[rdd] && \makecell[c]{2_{-1}4_{-5}4_{-7}} \ar[rdd]\ar[rdd] &&  \makecell[c]{2_{-5}} \ar[rdd] &&  \makecell[c]{2_{-3}2_{-1}} \ar[rdd] &&  \makecell[c]{1_{0}4_{-5}4_{-7}} \ar[rdd] &&  \makecell[c]{3_{-4}}  \ar[rdd]  &&  \makecell[c]{2_{-1}2_{-3}4_{-7}}  \ar[rdd] &&  \makecell[c]{1_{0}Y3_{-4}3_{-6}}  \ar[rdd]  && \makecell[c]{2_{-3}} \ar[rdd] &&  \makecell[c]{1_{0}1_{-2}4_{-7}}  \ar[rdd] &&  \makecell[c]{3_{-4}3_{-6}} \ar[rdd] && \makecell[c]{3_{-2}} \ar[rdd] && \makecell[c]{1_{0}1_{-2}3_{-6}} \ar[rdd] && \makecell[c]{2_{-3}2_{-5}} \ar[rdd] && \makecell[c]{2_{-1}} \ar[rdd] && \makecell[c]{3_{-6}}  \ar[rdd] &&  \makecell[c]{3_{-4}3_{-2}} \\
 \\
 & \makecell[c]{3_{-2}4_{-5}2_{-5}}  \ar[rdd] \ar[ruu] && \makecell[c]{2_{-1}3_{-2}3_{-4}4_{-5}4_{-7}} \ar[rdd]\ar[ruu] && \makecell[c]{2_{-1}4_{-5}2_{-5}4_{-7}} \ar[rdd] \ar[ruu] &&  \makecell[c]{2_{-1}3_{-4}1_{-4}} \ar[ruu]\ar[rdd] &&  \makecell[c]{1_{0}2_{-1}2_{-3}4_{-5}4_{-7}}  \ar[ruu]\ar[rdd] &&  \makecell[c]{1_{0}3_{-4}4_{-5}4_{-7}} \ar[ruu]\ar[rdd] &&  \makecell[c]{2_{-1}2_{-3}3_{-4}4_{-7}} \ar[ruu]\ar[rdd] &&  \makecell[c]{1_{0}2_{-1}2_{-3}3_{-4}3_{-6}4_{-7}} \ar[ruu]\ar[rdd] &&  \makecell[c]{1_{0}2_{-3}3_{-4}3_{-6}}  \ar[ruu]\ar[rdd] &&  \makecell[c]{1_{0}1_{-2}2_{-3}4_{-7}}  \ar[ruu]\ar[rdd] &&  \makecell[c]{1_{0}1_{-2}3_{-4}3_{-6}4_{-7}}  \ar[ruu]\ar[rdd] &&  \makecell[c]{4_{-3}2_{-3}3_{-6}}  \ar[ruu]\ar[rdd] && \makecell[c]{1_0 3_{-2}1_{-2}3_{-6}} \ar[ruu]\ar[rdd] && \makecell[c]{1_{0}1_{-2}2_{-3}2_{-5}3_{-6}} \ar[ruu]\ar[rdd] && \makecell[c]{3_{-2}1_{-2}2_{-5}} \ar[ruu]\ar[rdd]  &&  \makecell[c]{2_{-1}3_{-6}} \ar[ruu]\ar[rdd]  &&  \makecell[c]{3_{-2}4_{-5}2_{-5}}  \ar[ruu] \\
 \\
\makecell[c]{2_{-1}3_{-2}4_{-5}2_{-5}} \ar[r] \ar[ruu] \ar[rdd] & \makecell[c]{2_{-1}4_{-5}}  \ar[r] & \ar[r] \makecell[c]{2_{-1}3_{-2}4_{-5}^{2}2_{-5}4_{-7}} \ar[ruu]\ar[rdd] &  \ar[r] \makecell[c]{3_{-2}4_{-5}2_{-5}4_{-7}}  &\ar[r] \makecell[c]{2_{-1}3_{-2}3_{-4}4_{-5}2_{-5}4_{-7}} \ar[ruu] \ar[rdd]  &  \ar[r] \makecell[c]{2_{-1}3_{-4}}  &\ar[r]  \makecell[c]{2_{-1}^{2}3_{-4}4_{-5}1_{-4}4_{-7}}  \ar[ruu] \ar[rdd] &  \ar[r] \makecell[c]{2_{-1}4_{-5}1_{-4}4_{-7}}  & \ar[r]  \makecell[c]{1_{0}2_{-1}3_{-4}4_{-5}1_{-4}4_{-7}}  \ar[rdd]\ar[ruu]&  \ar[r] \makecell[c]{1_{0}3_{-4}}  & \ar[r]  \makecell[c]{1_{0}2_{-1}2_{-3}3_{-4}4_{-5}4_{-7}} \ar[rdd]\ar[ruu] &  \ar[r] \makecell[c]{2_{-1}2_{-3}4_{-5}4_{-7}}  & \ar[r] \makecell[c]{1_{0}2_{-1}2_{-3}3_{-4}4_{-5}4_{-7}^{2}}  \ar[rdd]\ar[ruu] &  \ar[r] \makecell[c]{1_{0}2_{-1}2_{-3}3_{-4}^{2}3_{-6}4_{-7}}  &  \ar[r]  \makecell[c]{1_{0}3_{-4}4_{-7}} \ar[rdd]\ar[ruu] &  \ar[r]  \makecell[c]{2_{-1}2_{-3}3_{-4}3_{-6}} & \ar[r]\makecell[c]{1_{0}2_{-1}2_{-3}^{2}3_{-4}3_{-6}4_{-7}}  \ar[rdd]\ar[ruu] &  \ar[r]  \makecell[c]{1_{0}2_{-3}4_{-7}}  & \ar[r] \makecell[c]{1_{0}^{2}1_{-2}2_{-3}3_{-4}3_{-6}4_{-7}} \ar[rdd]\ar[ruu] &  \ar[r] \makecell[c]{1_{0}1_{-2}3_{-4}3_{-6}}  & \ar[r] \makecell[c]{1_{0}1_{-2}2_{-3}3_{-4}3_{-6}4_{-7}}   \ar[rdd]\ar[ruu] &  \ar[r]  \makecell[c]{2_{-3}4_{-7}}  & \ar[r] \makecell[c]{1_{0}4_{-3}1_{-2}2_{-3}3_{-6}4_{-7}}  \ar[rdd]\ar[ruu] &  \ar[r] \makecell[c]{1_{0}4_{-3}1_{-2}3_{-6}} & \ar[r] \makecell[c]{1_{0}4_{-3}1_{-2}2_{-3}3_{-6}^{2}}  \ar[rdd]\ar[ruu] &  \ar[r] \makecell[c]{2_{-3}3_{-6}} & \ar[r] \makecell[c]{1_{0}3_{-2}1_{-2}2_{-3}2_{-5}3_{-6}}  \ar[rdd]\ar[ruu] &  \ar[r] \makecell[c]{1_{0}3_{-2}1_{-2}2_{-5}}  & \ar[r] \makecell[c]{1_{0}3_{-2}1_{-2}^{2}2_{-5}3_{-6}}  \ar[rdd]\ar[ruu] &  \ar[r] \makecell[c]{3_{-2}1_{-2}2_{-5}3_{-6}}  &\ar[r]  \makecell[c]{3_{-2}1_{-2}2_{-5}3_{-6}} \ar[rdd]\ar[ruu] &  \ar[r] \makecell[c]{3_{-2}2_{-5}} &\ar[r] \makecell[c]{2_{-1}3_{-2}4_{-5}2_{-5}} \ar[rdd]\ar[ruu]  & \makecell[c]{2_{-1}4_{-5}} \\
 \\
 & \makecell[c]{2_{-1}3_{-2}4_{-5}2_{-5}4_{-7}} \ar[ruu] \ar[rdd] && \makecell[c]{2_{-1}4_{-5}2_{-5}} \ar[ruu] \ar[rdd] && \makecell[c]{2_{-1}3_{-2}3_{-4}4_{-5}1_{-4}4_{-7}} \ar[ruu] \ar[rdd] &&  \makecell[c]{1_{0}2_{-1}3_{-4}4_{-5}4_{-7}} \ar[ruu] \ar[rdd] && \makecell[c]{2_{-1}3_{-4}4_{-5}1_{-4}4_{-7}} \ar[ruu] \ar[rdd] &&  \makecell[c]{1_{0}2_{-1}2_{-3}3_{-4}4_{-7}}  \ar[ruu] \ar[rdd] &&  \makecell[c]{1_{0}2_{-1}2_{-3}3_{-4}4_{-5}3_{-6}4_{-7}} \ar[ruu] \ar[rdd] && \makecell[c]{1_{0}2_{-3}3_{-4}4_{-7}} \ar[ruu] \ar[rdd] &&  \makecell[c]{1_{0}2_{-1}1_{-2}2_{-3}3_{-4}3_{-6}4_{-7}} \ar[ruu] \ar[rdd] &&  \makecell[c]{1_{0}2_{-3}3_{-4}3_{-6}4_{-7}} \ar[ruu] \ar[rdd] &&  \makecell[c]{1_{0}4_{-3}1_{-2}2_{-3}3_{-6}} \ar[ruu] \ar[rdd] &&  \makecell[c]{1_{0}1_{-2}2_{-3}3_{-6}4_{-7}} \ar[ruu] \ar[rdd] &&  \makecell[c]{1_{0}4_{-3}1_{-2}2_{-3}2_{-5}3_{-6}} \ar[ruu] \ar[rdd] && \makecell[c]{3_{-2}1_{-2}3_{-6}} \ar[ruu] \ar[rdd] &&  \makecell[c]{1_{0}3_{-2}1_{-2}2_{-5}3_{-6}}  \ar[ruu] \ar[rdd] &&  \makecell[c]{3_{-2}1_{-2}4_{-5}2_{-5}} \ar[ruu] \ar[rdd] && \makecell[c]{2_{-1}3_{-2}4_{-5}2_{-5}4_{-7}} \ar[rdd] \\
 \\
&& \makecell[c]{2_{-5}2_{-1}} \ar[ruu] \ar[rdd] &&  \makecell[c]{2_{-1}4_{-5}1_{-4}}  \ar[ruu] \ar[rdd] &&  \makecell[c]{1_{0}3_{-2}3_{-4}4_{-5}4_{-7}} \ar[ruu] \ar[rdd] && \makecell[c]{2_{-1}3_{-4}4_{-5}4_{-7}} \ar[ruu] \ar[rdd] && \makecell[c]{2_{-1}3_{-4}1_{-4}4_{-7}} \ar[ruu] \ar[rdd] && \makecell[c]{1_{0}2_{-1}2_{-3}3_{-4}3_{-6}} \ar[ruu] \ar[rdd] && \makecell[c]{1_{0}2_{-3}4_{-5}4_{-7}} \ar[ruu] \ar[rdd] && \makecell[c]{1_{0}1_{-2}3_{-4}4_{-7}} \ar[ruu] \ar[rdd] &&  \makecell[c]{2_{-1}2_{-3}3_{-4}3_{-6}4_{-7}} \ar[ruu] \ar[rdd] && \makecell[c]{1_{0}4_{-3}2_{-3}3_{-6}} \ar[ruu] \ar[rdd] && \makecell[c]{1_{0}1_{-2}2_{-3}3_{-6}} \ar[ruu] \ar[rdd] && \makecell[c]{1_{0}1_{-2}2_{-3}2_{-5}4_{-7}} \ar[ruu] \ar[rdd] && \makecell[c]{4_{-3}1_{-2}3_{-6}} \ar[ruu] \ar[rdd] && \makecell[c]{3_{-2}3_{-6}} \ar[ruu] \ar[rdd] && \makecell[c]{1_{0}3_{-2}1_{-2}4_{-5}2_{-5}} \ar[ruu] \ar[rdd] && \makecell[c]{3_{-2}1_{-2}4_{-5}2_{-5}4_{-7}} \ar[ruu]  \ar[rdd] && \makecell[c]{2_{-5}2_{-1}} \ar[rdd] \\
\\
&&& \makecell[c]{2_{-1}1_{-4}}  \ar[ruu] \ar[rdd] && \makecell[c]{1_{0}4_{-5}} \ar[ruu] \ar[rdd] && \makecell[c]{3_{-2}3_{-4}4_{-5}4_{-7}} \ar[ruu] \ar[rdd] && \makecell[c]{2_{-1}3_{-4}4_{-7}} \ar[ruu] \ar[rdd] &&  \makecell[c]{2_{-1}3_{-4}1_{-4}3_{-6}} \ar[ruu] \ar[rdd] &&  \makecell[c]{1_{0}2_{-3}} \ar[ruu] \ar[rdd] &&  \makecell[c]{1_{0}1_{-2}4_{-5}4_{-7}} \ar[ruu] \ar[rdd] &&  \makecell[c]{3_{-4}4_{-7}} \ar[ruu] \ar[rdd] && \makecell[c]{2_{-1}4_{-3}2_{-3}3_{-6}} \ar[ruu] \ar[rdd] &&  \makecell[c]{1_{0}2_{-3}3_{-6}} \ar[ruu] \ar[rdd] &&  \makecell[c]{1_{0}1_{-2}2_{-3}2_{-5}} \ar[ruu] \ar[rdd] && \makecell[c]{1_{-2}4_{-7}} \ar[ruu] \ar[rdd] && \makecell[c]{4_{-3}3_{-6}} \ar[ruu] \ar[rdd] && \makecell[c]{ 3_{-2}4_{-5}} \ar[ruu] \ar[rdd] && \makecell[c]{1_{0}3_{-2}1_{-2}4_{-5}2_{-5}4_{-7}} \ar[ruu] \ar[rdd] && \makecell[c]{1_{-2}2_{-5}} \ar[ruu] \ar[rdd] && \makecell[c]{2_{-1}1_{-4}} \ar[rdd] \\
 \\
&&&& \makecell[c]{1_{0}}  \ar[ruu] && \makecell[c]{4_{-5}} \ar[ruu] && \makecell[c]{3_{-2}3_{-4}4_{-7}} \ar[ruu] && \makecell[c]{2_{-1}3_{-4}3_{-6}} \ar[ruu] && \makecell[c]{1_{-4}} \ar[ruu] && \makecell[c]{1_{-2}1_{0}}\ar[ruu]  && \makecell[c]{4_{-7}4_{-5}} \ar[ruu] && \makecell[c]{4_{-3}} \ar[ruu] && \makecell[c]{2_{-1}2_{-3}3_{-6}} \ar[ruu]  && \makecell[c]{1_{0}2_{-3}2_{-5}} \ar[ruu]   &&   \makecell[c]{1_{-2}} \ar[ruu]  && \makecell[c]{4_{-7}} \ar[ruu] && \makecell[c]{4_{-5}4_{-3}} \ar[ruu]  &&  \makecell[c]{3_{-2}4_{-5}4_{-7}} \ar[ruu]  && \makecell[c]{1_{0}1_{-2}2_{-5}} \ar[ruu]  && \makecell[c]{1_{-4}1_{-2}} \ar[ruu]  &&  \makecell[c]{1_{0}}
}}}
\caption{All the real prime simple modules in $\mathscr{C}^{\leq \xi}_2$, excluding the frozen real prime simple modules $1_{-4}1_{-2}1_0$,  $2_{-5}2_{-3}2_{-1}$, $3_{-6}3_{-4}3_{-2}$, and $4_{-7}4_{-5}4_{-3}$.} \label{all the real prime simple modules C22 in A4}
\end{figure}

Let $\xi(1,2,3,4)=(0,-1,-2,-1)$. The initial quiver and the principal quiver of the cluster algebra $K_0(\mathscr{C}^{\leq \xi}_2)$ are shown in Figure \ref{the initial quiver (left) and its principal quiver (right) 3}. The indecomposable rigid objects in $\mathcal{C}^{\leq \xi}_2$ are shown in Figure \ref{the cluster category C23 in A4}. The real prime simple modules (excluding the frozen modules) in $\mathscr{C}^{\leq \xi}_2$ are shown in Figure \ref{all the real prime simple modules C23 in A4}.  The two sets are in bijection under the map $L(m) \to K^{\leq \xi}_\ell(m)$.

\begin{figure}
\resizebox{.8\width}{.8\height}{ 
\begin{minipage}{0.5\textwidth}
\begin{xy}
(-25,50)*+{(1,0)}="a";
(0,50)*+{(2,-1)}="b"; 
(-25,30)*+{(1,-2)}="c";
(0,30)*+{(2,-3)}="d"; 
(-25,10)*+{\fbox{(1,-4)}}="e";
(0,10)*+{\fbox{(2,-5)}}="f";
(25,50)*+{(3,-2)}="g";
(50,50)*+{(4,-1)}="h"; 
(25,30)*+{(3,-4)}="i";
(50,30)*+{(4,-3)}="j"; 
(25,10)*+{\fbox{(3,-6)}}="k";
(50,10)*+{\fbox{(4,-5)}}="l";
{\ar "a";"b"}; {\ar "b";"g"}; {\ar "h";"g"};
{\ar "b";"c"};
{\ar "c";"d"};{\ar "d";"i"};{\ar "j";"i"};
{\ar "d";"e"};{\ar "f";"d"};
{\ar "d";"b"};
{\ar "e";"c"};
{\ar "c";"a"};
{\ar "k";"i"};{\ar "i";"g"};
{\ar "l";"j"};{\ar "j";"h"};
{\ar "g";"d"};{\ar "g";"j"};
{\ar "i";"f"};{\ar "i";"l"};
\end{xy}
\end{minipage} 
\qquad 
\begin{minipage}{0.5\textwidth}
\begin{xy}
(-25,50)*+{1}="a";
(0,50)*+{2}="b"; 
(-25,30)*+{5}="c";
(0,30)*+{6}="d"; 
(25,50)*+{3}="g";
(50,50)*+{4}="h"; 
(25,30)*+{7}="i";
(50,30)*+{8}="j"; 
{\ar "a";"b"}; {\ar "b";"g"}; {\ar "h";"g"};
{\ar "b";"c"};
{\ar "c";"d"};{\ar "d";"i"};{\ar "j";"i"};
{\ar "d";"b"};
{\ar "c";"a"};
{\ar "i";"g"};
{\ar "j";"h"};
{\ar "g";"d"};
{\ar "g";"j"};
\end{xy}
\end{minipage}} 
\caption{The initial quiver (left) of $K_0(\mathscr{C}^{\leq \xi}_2)$ and its principal quiver (right).} \label{the initial quiver (left) and its principal quiver (right) 3}
\end{figure}

\begin{figure}
\centerline{
\resizebox{.2\width}{.3\height}{ 
\xymatrix{ 
 && \makecell[c]{0000\\0010} \ar[rdd]  && \makecell[c]{0111\\0111} \ar[rdd]\ar[rdd] && \makecell[c]{1100\\0000} \ar[rdd] &&  \makecell[c]{0000\\0110} \ar[rdd] && \makecell[c]{1111\\1111} \ar[rdd]&& \makecell[c]{00-10\\0000} \ar[rdd] && \makecell[c]{0010\\0111}  \ar[rdd] && \makecell[c]{1100\\1100} \ar[rdd] && \makecell[c]{0-100\\0000} \ar[rdd] && \makecell[c]{0110\\1111} \ar[rdd] && \makecell[c]{0000\\00-10} \ar[rdd] && \makecell[c]{0010\\0010}  \ar[rdd] && \makecell[c]{0111\\1100} \ar[rdd] && \makecell[c]{0000\\0-100} \ar[rdd] && \makecell[c]{0110\\0110}\ar[rdd] && \makecell[c]{1111\\0000}  \ar[rdd] && \makecell[c]{0000\\0010} \\
 \\
&  \makecell[c]{1111\\0010}  \ar[rdd] \ar[ruu] && \makecell[c]{0111\\0121} \ar[rdd]\ar[ruu] && \makecell[c]{1211\\0111} \ar[rdd] \ar[ruu] && \makecell[c]{1100\\0110} \ar[ruu]\ar[rdd] && \makecell[c]{1111\\1221} \ar[ruu]\ar[rdd] &&  \makecell[c]{1101\\1111} \ar[ruu]\ar[rdd] &&  \makecell[c]{0000\\0111} \ar[ruu]\ar[rdd] && \makecell[c]{1110\\1211} \ar[ruu]\ar[rdd] && \makecell[c]{1000\\1100} \ar[ruu]\ar[rdd] && \makecell[c]{0010\\1111} \ar[ruu]\ar[rdd] && \makecell[c]{0110\\1101} \ar[ruu]\ar[rdd] && \makecell[c]{0010\\0000}  \ar[ruu]\ar[rdd] && \makecell[c]{0121\\1110} \ar[ruu]\ar[rdd] && \makecell[c]{0111\\1000} \ar[ruu]\ar[rdd] && \makecell[c]{0110\\0010} \ar[ruu]\ar[rdd] && \makecell[c]{1221\\0110} \ar[ruu]\ar[rdd]  &&  \makecell[c]{1111\\0010}  \ar[ruu] \\
 \\
\makecell[c]{1221\\0120} \ar[r] \ar[ruu] \ar[rdd] & \makecell[c]{0111\\0110} \ar[r] & \ar[r] \makecell[c]{1222\\0121} \ar[ruu]\ar[rdd] &   \makecell[c]{1111\\0011} \ar[r]   &\ar[r] \makecell[c]{1211\\0121} \ar[ruu] \ar[rdd]  &  \ar[r] \makecell[c]{0100\\0110} & \ar[r] \makecell[c]{1211\\0221} \ar[ruu] \ar[rdd] &  \ar[r] \makecell[c]{1111\\0111}  & \ar[r] \makecell[c]{2211\\1221} \ar[rdd]\ar[ruu]&  \ar[r] \makecell[c]{1100\\1110} & \ar[r] \makecell[c]{1101\\1221} \ar[rdd]\ar[ruu]&  \ar[r] \makecell[c]{0001\\0111} & \ar[r] \makecell[c]{1101\\1222} \ar[rdd]\ar[ruu] &  \ar[r] \makecell[c]{1100\\1111} & \ar[r] \makecell[c]{1100\\1211} \ar[rdd] \ar[ruu] &  \ar[r] \makecell[c]{0000\\0100} & \ar[r]\makecell[c]{1010\\1211} \ar[rdd]\ar[ruu] &  \ar[r] \makecell[c]{1010\\1111} & \ar[r] \makecell[c]{1010\\2211} \ar[rdd]\ar[ruu] &  \ar[r] \makecell[c]{0000\\1100} & \ar[r] \makecell[c]{0010\\1101} \ar[rdd]\ar[ruu]&  \ar[r] \makecell[c]{0010\\0001} &\ar[r] \makecell[c]{0120\\1101} \ar[rdd]\ar[ruu] & \ar[r] \makecell[c]{0110\\1100} &\ar[r] \makecell[c]{0121\\1100} \ar[rdd]\ar[ruu] &  \ar[r] \makecell[c]{0011\\0000} & \ar[r] \makecell[c]{0121\\1010}  \ar[rdd]\ar[ruu] &  \ar[r] \makecell[c]{0110\\1010}  & \ar[r] \makecell[c]{0221\\1010} \ar[rdd]\ar[ruu] &  \ar[r] \makecell[c]{0111\\0000} &\ar[r] \makecell[c]{1221\\0010}  \ar[rdd]\ar[ruu] &  \ar[r] \makecell[c]{1110\\0010} &\ar[r] \makecell[c]{1221\\0120} \ar[rdd]\ar[ruu]  &   \makecell[c]{0111\\0110} \\
 \\
& \makecell[c]{1221\\0121} \ar[ruu] \ar[rdd] && \makecell[c]{1211\\0110} \ar[ruu] \ar[rdd] && \makecell[c]{1111\\0121} \ar[ruu] \ar[rdd] &&  \makecell[c]{1211\\1221} \ar[ruu] \ar[rdd] && \makecell[c]{1101\\0111} \ar[ruu] \ar[rdd] &&  \makecell[c]{1100\\1221} \ar[ruu] \ar[rdd] && \makecell[c]{1101\\1211}  \ar[ruu] \ar[rdd] &&  \ar[ruu] \makecell[c]{1000\\1111} \ar[rdd] &&  \makecell[c]{0010\\1211} \ar[ruu] \ar[rdd] && \makecell[c]{1010\\1101} \ar[ruu] \ar[rdd] && \makecell[c]{0010\\1100} \ar[ruu] \ar[rdd] && \makecell[c]{0121\\1101}  \ar[ruu] \ar[rdd] && \makecell[c]{0110\\1000} \ar[ruu] \ar[rdd] && \makecell[c]{0121\\0010} \ar[ruu] \ar[rdd] && \makecell[c]{1221\\1010} \ar[ruu] \ar[rdd] && \makecell[c]{0111\\0010} \ar[ruu] \ar[rdd] && \makecell[c]{1221\\0121} \ar[rdd] \\
 \\
&&  \makecell[c]{1210\\0110}  \ar[ruu] \ar[rdd] && \makecell[c]{1111\\0110} \ar[ruu] \ar[rdd] &&  \makecell[c]{1111\\1121}  \ar[ruu] \ar[rdd] && \makecell[c]{0101\\0111} \ar[ruu] \ar[rdd] && \makecell[c]{1100\\0111} \ar[ruu] \ar[rdd] && \makecell[c]{1100\\1210} \ar[ruu] \ar[rdd] && \makecell[c]{1001\\1111} \ar[ruu] \ar[rdd] && \makecell[c]{0000\\1111} \ar[ruu] \ar[rdd] &&   \makecell[c]{0010\\0101} \ar[ruu] \ar[rdd] && \makecell[c]{1010\\1100} \ar[ruu] \ar[rdd] && \makecell[c]{0011\\1100} \ar[ruu] \ar[rdd] && \makecell[c]{0110\\1001} \ar[ruu] \ar[rdd] && \makecell[c]{0110\\0000}\ar[ruu] \ar[rdd]  && \makecell[c]{1121\\0010} \ar[ruu] \ar[rdd] && \makecell[c]{0111\\1010} \ar[ruu] \ar[rdd] && \makecell[c]{0111\\0011} \ar[ruu]  \ar[rdd] && \makecell[c]{1210\\0110}\ar[rdd] \\
\\
&&& \makecell[c]{1110\\0110} \ar[ruu] \ar[rdd] && \makecell[c]{1111\\1110} \ar[ruu] \ar[rdd] &&  \makecell[c]{0001\\0011} \ar[ruu] \ar[rdd] && \makecell[c]{0100\\0111} \ar[ruu] \ar[rdd] && \makecell[c]{1100\\0100}  \ar[ruu] \ar[rdd] && \makecell[c]{1000\\1110} \ar[ruu] \ar[rdd] && \makecell[c]{0001\\1111}  \ar[ruu] \ar[rdd]&& \makecell[c]{0000\\0001} \ar[ruu] \ar[rdd] &&  \makecell[c]{0010\\0100} \ar[ruu] \ar[rdd] &&  \makecell[c]{1011\\1100} \ar[ruu] \ar[rdd] && \makecell[c]{0000\\1000} \ar[ruu] \ar[rdd] && \makecell[c]{0110\\0001} \ar[ruu] \ar[rdd] && \makecell[c]{1110\\0000} \ar[ruu] \ar[rdd] &&  \makecell[c]{0011\\0010} \ar[ruu] \ar[rdd]   &&  \makecell[c]{0111\\1011} \ar[ruu] \ar[rdd]   &&  \makecell[c]{0100\\0000} \ar[ruu] \ar[rdd] && \makecell[c]{1110\\0110} \ar[rdd] \\
 \\
&&&& \makecell[c]{1110\\1110} \ar[ruu] && \makecell[c]{0001\\0000} \ar[ruu] &&  \makecell[c]{0000\\0011} \ar[ruu] && \makecell[c]{0100\\0100} \ar[ruu] && \makecell[c]{1000\\0000} \ar[ruu] && \makecell[c]{0000\\1110} \ar[ruu]  && \makecell[c]{0001\\0001} \ar[ruu] && \makecell[c]{000-1\\0000} \ar[ruu] && \makecell[c]{0011\\0100} \ar[ruu]  && \makecell[c]{1000\\1000} \ar[ruu]   && \makecell[c]{-1000\\0000} \ar[ruu]  && \makecell[c]{1110\\0001} \ar[ruu] &&  \makecell[c]{0000\\000-1}  \ar[ruu]  && \makecell[c]{0011\\0011} \ar[ruu]  && \makecell[c]{0100\\1000} \ar[ruu]   && \makecell[c]{0000\\-1000} \ar[ruu]  &&  \makecell[c]{1110\\1110}
}}}
\caption{The cluster category $\mathcal{C}^{\leq \xi}_2$ of $kQ^{\leq \xi}_2$, and the objects on the far left are to be identified with the objects on the far right.} \label{the cluster category C23 in A4}
\end{figure}

\begin{figure}\large
\centerline{
\resizebox{.086\width}{.3\height}{ 
\xymatrix{ 
&& \makecell[c]{2_{-1}4_{-1}2_{-3}4_{-3}3_{-6}} \ar[rdd]  &&  \makecell[c]{1_{0}3_{-4}3_{-6}} \ar[rdd] \ar[rdd] &&  \makecell[c]{2_{-3}} \ar[rdd] 
&& \makecell[c]{1_0 4_{-1} 1_{-2} 4_{-3} 3_{-6}} \ar[rdd] &&  \makecell[c]{3_{-4}3_{-6}} \ar[rdd]  && \makecell[c]{3_{-2}} \ar[rdd] &&  \makecell[c]{1_{0}1_{-2}3_{-6}} \ar[rdd] &&  \makecell[c]{2_{-5}2_{-3}} \ar[rdd] && \makecell[c]{2_{-1}} \ar[rdd] && \makecell[c]{3_{-6}} \ar[rdd] && \makecell[c]{3_{-4}3_{-2}} \ar[rdd] &&  \makecell[c]{2_{-1}4_{-1}3_{-4}3_{-6}} \ar[rdd] &&  \makecell[c]{2_{-5}}  \ar[rdd] && \makecell[c]{2_{-3}2_{-1}} \ar[rdd] && \makecell[c]{1_{0}4_{-1}3_{-4}3_{-6}} \ar[rdd] && \makecell[c]{3_{-4}} \ar[rdd] && \makecell[c]{2_{-1}4_{-1}2_{-3}4_{-3}3_{-6}}  \\
 \\
 & \makecell[c]{2_{-1}4_{-1}2_{-3}3_{-4}4_{-3}3_{-6}}  \ar[rdd] \ar[ruu] && \makecell[c]{1_{0}2_{-1}4_{-1}2_{-3}3_{-4}4_{-3}3_{-6}^{2}} \ar[rdd]\ar[ruu] &&  \makecell[c]{1_{0}2_{-3}3_{-4}3_{-6}} \ar[rdd] \ar[ruu] &&  \makecell[c]{1_{0}4_{-1}1_{-2}2_{-3}4_{-3}3_{-6}}  \ar[ruu]\ar[rdd] && \makecell[c]{1_{0}4_{-1}1_{-2}3_{-4}4_{-3}3_{-6}^{2}}  \ar[ruu]\ar[rdd] &&  \makecell[c]{2_{-3}3_{-6}4_{-3}}  \ar[ruu]\ar[rdd] &&  \makecell[c]{1_{0} 3_{-2}1_{-2}3_{-6}} \ar[ruu]\ar[rdd] &&   \makecell[c]{1_0 1_{-2}2_{-3}2_{-5}3_{-6}} \ar[ruu]\ar[rdd] && \makecell[c]{1_{-2}2_{-5}3_{-2}} \ar[ruu]\ar[rdd] && \makecell[c]{2_{-1}3_{-6}} \ar[ruu]\ar[rdd] && \makecell[c]{3_{-2}2_{-5}4_{-5}} \ar[ruu]\ar[rdd] && \makecell[c]{2_{-1}3_{-4}4_{-1}} \ar[ruu]\ar[rdd] && \makecell[c]{2_{-1}4_{-1}3_{-4}2_{-5}3_{-6}} \ar[ruu]\ar[rdd] && \makecell[c]{2_{-1}3_{-4}1_{-4}} \ar[ruu]\ar[rdd] &&  \makecell[c]{1_{0}2_{-1}4_{-1}2_{-3}3_{-4}3_{-6}} \ar[ruu]\ar[rdd] && \makecell[c]{1_{0}4_{-1}3_{-4}^{2}3_{-6}} \ar[ruu]\ar[rdd]  && \makecell[c]{2_{-1}4_{-1}2_{-3}3_{-4}4_{-3}3_{-6}}\ar[ruu] \\
 \\
\makecell[c]{1_{0}2_{-1}4_{-1}^{2}2_{-3}3_{-4}^{2}4_{-3}3_{-6}^{2}} \ar[r] \ar[ruu] \ar[rdd] & \makecell[c]{1_{0}4_{-1}3_{-4}4_{-3} 3_{-6}} \ar[r] & \ar[r] \makecell[c]{1_{0}2_{-1}4_{-1}2_{-3}3_{-4}^{2}4_{-3}3_{-6}^{2}}  \ar[ruu]\ar[rdd]  &  \ar[r] \makecell[c]{2_{-1}2_{-3}3_{-4} 3_{-6}} & \ar[r] \makecell[c]{1_{0}2_{-1}4_{-1}2_{-3}^{2}3_{-4}4_{-3}3_{-6}^{2}}  \ar[ruu] \ar[rdd]  &  \ar[r] \makecell[c]{1_{0}4_{-1}2_{-3}4_{-3} 3_{-6}} & \ar[r] \makecell[c]{1_{0}^{2}4_{-1}1_{-2}2_{-3}3_{-4}4_{-3}3_{-6}^{2}}  \ar[ruu] \ar[rdd]&  \ar[r] \makecell[c]{1_{0}1_{-2}3_{-4}3_{-6}} & \ar[r] \makecell[c]{1_{0}4_{-1}1_{-2}2_{-3}3_{-4}4_{-3}3_{-6}^{2}} \ar[rdd]\ar[ruu]  &  \ar[r] \makecell[c]{4_{-1}2_{-3}4_{-3}3_{-6}}  & \ar[r] \makecell[c]{1_{0}4_{-1}1_{-2}2_{-3}4_{-3}^{2}3_{-6}^{2}} \ar[rdd]\ar[ruu ] &  \ar[r] \makecell[c]{1_{0}1_{-2}3_{-6}4_{-3}}  &\ar[r] \makecell[c]{1_{0}1_{-2}2_{-3}4_{-3}3_{-6}^{2}} \ar[rdd]\ar[ruu] &  \ar[r] \makecell[c]{2_{-3}3_{-6}} & \ar[r] \makecell[c]{1_{0}3_{-2}1_{-2}2_{-3}2_{-5}3_{-6}} \ar[rdd]\ar[ruu] &  \ar[r] \makecell[c]{1_{0}1_{-2}2_{-5}3_{-2}} & \ar[r]  \makecell[c]{1_{0}3_{-2}1_{-2}^{2}2_{-5}3_{-6}} \ar[rdd]\ar[ruu] &  \ar[r] \makecell[c]{1_{-2}3_{-6}} & \ar[r] \makecell[c]{1_{-2}2_{-5}3_{-2}3_{-6}} \ar[rdd]\ar[ruu] &  \ar[r] \makecell[c]{3_{-2}2_{-5}} & \ar[r] \makecell[c]{2_{-1}3_{-2}2_{-5}4_{-5}} \ar[rdd]\ar[ruu] &  \ar[r] \makecell[c]{2_{-1}4_{-5}}  & \ar[r] \makecell[c]{2_{-1}4_{-1}2_{-5}4_{-5}} \ar[rdd]\ar[ruu] &  \ar[r] \makecell[c]{4_{-1}2_{-5}}  &\ar[r] \makecell[c]{2_{-1}4_{-1}3_{-4}2_{-5}}  \ar[rdd]\ar[ruu] &  \ar[r] \makecell[c]{2_{-1}3_{-4}} & \ar[r] \makecell[c]{2_{-1}^{2}4_{-1}3_{-4}^{2}1_{-4}3_{-6}}  \ar[rdd] \ar[ruu] &  \ar[r] \makecell[c]{2_{-1}4_{-1}3_{-4}1_{-4}3_{-6}} & \ar[r] \makecell[c]{1_{0}2_{-1}4_{-1}3_{-4}^{2}1_{-4}3_{-6}} \ar[rdd]\ar[ruu] &  \ar[r] \makecell[c]{1_{0}3_{-4}}  &\ar[r] \makecell[c]{1_{0}2_{-1}4_{-1}2_{-3}3_{-4}^{2}3_{-6}} \ar[rdd]\ar[ruu] &  \ar[r] \makecell[c]{2_{-1}4_{-1}2_{-3}3_{-4}3_{-6}} &\ar[r] \makecell[c]{1_{0}2_{-1}4_{-1}^{2}2_{-3}3_{-4}^{2}4_{-3}3_{-6}^{2}} \ar[rdd]\ar[ruu]  & \makecell[c]{1_{0}4_{-1}3_{-4}4_{-3} 3_{-6}} \\
 \\
& \makecell[c]{1_{0}2_{-1}4_{-1}2_{-3}3_{-4}^{2}3_{-6}^{2}} \ar[ruu] \ar[rdd] && \makecell[c]{1_{0}4_{-1}2_{-3}3_{-4}4_{-3}3_{-6}} \ar[ruu] \ar[rdd] &&  \makecell[c]{1_{0}2_{-1}4_{-1}1_{-2}2_{-3}3_{-4}4_{-3}3_{-6}^{2}} \ar[ruu] \ar[rdd] &&  \makecell[c]{1_{0}4_{-1}2_{-3}3_{-4}4_{-3}3_{-6}^{2}} \ar[ruu] \ar[rdd] && \makecell[c]{1_{0}1_{-2}2_{-3}4_{-3}3_{-6}} \ar[ruu] \ar[rdd] && \makecell[c]{1_{0}4_{-1}1_{-2}2_{-3}4_{-3}3_{-6}^{2}} \ar[ruu] \ar[rdd] &&  \makecell[c]{1_{0}1_{-2}2_{-3}4_{-3}2_{-5}3_{-6}} \ar[ruu] \ar[rdd] && \makecell[c]{3_{-2}1_{-2}3_{-6}} \ar[ruu] \ar[rdd] && \makecell[c]{1_{0}1_{-2}2_{-5}3_{-2}3_{-6}} \ar[ruu] \ar[rdd] && \makecell[c]{1_{-2}2_{-5}3_{-2}4_{-5}} \ar[ruu] \ar[rdd] &&  \makecell[c]{2_{-1} 4_{-1}2_{-5}} \ar[ruu] \ar[rdd] && \makecell[c]{2_{-1}2_{-5}4_{-5}} \ar[ruu] \ar[rdd] && \makecell[c]{2_{-1}4_{-1}3_{-4}1_{-4}} \ar[ruu] \ar[rdd] && \makecell[c]{1_{0}2_{-1}4_{-1}3_{-4}^{2}3_{-6}} \ar[ruu] \ar[rdd] && \makecell[c]{2_{-1}4_{-1}3_{-4}^{2}1_{-4}3_{-6}}  \ar[ruu] \ar[rdd] && \makecell[c]{1_{0}2_{-1}4_{-1}2_{-3}3_{-4}4_{-3}3_{-6}} \ar[ruu] \ar[rdd] && \makecell[c]{1_{0}2_{-1}4_{-1}2_{-3}3_{-4}^{2}3_{-6}^{2}} \ar[rdd] \\
 \\
&& \makecell[c]{1_{0}4_{-1}2_{-3}3_{-4}3_{-6}} \ar[ruu] \ar[rdd] &&  \makecell[c]{1_{0}4_{-1}1_{-2}3_{-4}4_{-3}3_{-6}} \ar[ruu] \ar[rdd] && \makecell[c]{2_{-1}4_{-1}2_{-3}3_{-4}4_{-3}3_{-6}^{2}} \ar[ruu] \ar[rdd] && \makecell[c]{1_{0}2_{-3}3_{-6}4_{-3}} \ar[ruu] \ar[rdd] && \makecell[c]{1_{0}1_{-2}2_{-3}3_{-6}} \ar[ruu] \ar[rdd] && \makecell[c]{1_{0}4_{-1}1_{-2}2_{-3}4_{-3}2_{-5}3_{-6}} \ar[ruu] \ar[rdd] &&  \makecell[c]{1_{-2} 3_{-6}4_{-3}} \ar[ruu] \ar[rdd] && \makecell[c]{3_{-2}1_{-2}3_{-6}} \ar[ruu] \ar[rdd] &&  \makecell[c]{1_{0}1_{-2}2_{-5}3_{-2}4_{-5}} \ar[ruu] \ar[rdd] &&  \makecell[c]{4_{-1}1_{-2}2_{-5}} \ar[ruu] \ar[rdd] && \makecell[c]{2_{-1}2_{-5}} \ar[ruu] \ar[rdd] && \makecell[c]{2_{-1}1_{-4}4_{-5}} \ar[ruu] \ar[rdd] && \makecell[c]{1_{0}3_{-4}4_{-1}} \ar[ruu] \ar[rdd] &&  \makecell[c]{2_{-1}4_{-1}3_{-4}^{2}3_{-6}} \ar[ruu] \ar[rdd] &&  \makecell[c]{2_{-1}4_{-1}3_{-4}4_{-3}1_{-4}3_{-6}} \ar[ruu] \ar[rdd] && \makecell[c]{1_{0}2_{-1}2_{-3}3_{-4}3_{-6}} \ar[ruu]  \ar[rdd] &&  \makecell[c]{1_{0}4_{-1}2_{-3}3_{-4}3_{-6}} \ar[rdd] \\
\\
&&& \makecell[c]{1_{0}4_{-1}1_{-2}3_{-4}3_{-6}} \ar[ruu] \ar[rdd] && \makecell[c]{4_{-1}3_{-4}4_{-3}3_{-6}} \ar[ruu] \ar[rdd] && \makecell[c]{2_{-1} 2_{-3}4_{-3}3_{-6}} \ar[ruu] \ar[rdd] && \makecell[c]{1_{0}2_{-3}3_{-6}}  \ar[ruu] \ar[rdd] && \makecell[c]{1_{0}1_{-2}2_{-3}2_{-5}} \ar[ruu] \ar[rdd] && \makecell[c]{4_{-1}1_{-2}4_{-3}3_{-6}} \ar[ruu] \ar[rdd] && \makecell[c]{4_{-3}3_{-6}} \ar[ruu] \ar[rdd] && \makecell[c]{3_{-2}4_{-5}} \ar[ruu] \ar[rdd] && \makecell[c]{1_{0}1_{-2}2_{-5}4_{-1}} \ar[ruu] \ar[rdd] && \makecell[c]{1_{-2}2_{-5}} \ar[ruu] \ar[rdd] && \makecell[c]{2_{-1}1_{-4}} \ar[ruu] \ar[rdd] &&  \makecell[c]{1_{0}4_{-5}} \ar[ruu] \ar[rdd] && \makecell[c]{4_{-1}3_{-4}} \ar[ruu] \ar[rdd] && \makecell[c]{2_{-1}4_{-1}3_{-4}4_{-3}3_{-6}} \ar[ruu] \ar[rdd] &&  \makecell[c]{2_{-1}3_{-4}1_{-4}3_{-6}} \ar[ruu] \ar[rdd] && \makecell[c]{1_{0}2_{-3}} \ar[ruu] \ar[rdd] && \makecell[c]{1_{0}4_{-1}1_{-2}3_{-4}3_{-6}} \ar[rdd] \\
 \\
&&&& \makecell[c]{4_{-1}3_{-4}3_{-6}} \ar[ruu] && \makecell[c]{4_{-3}} \ar[ruu] && \makecell[c]{2_{-1} 2_{-3}3_{-6}} \ar[ruu] && \makecell[c]{1_{0}2_{-3}2_{-5}} \ar[ruu] && \makecell[c]{1_{-2}} \ar[ruu] && \makecell[c]{4_{-1}4_{-3}3_{-6}} \ar[ruu]  &&  \makecell[c]{4_{-3}4_{-5}} \ar[ruu] &&  
\makecell[c]{4_{-1}}\ar[ruu] && \makecell[c]{1_{0}1_{-2}2_{-5}} \ar[ruu]  && \makecell[c]{1_{-4}1_{-2}} \ar[ruu]   && \makecell[c]{1_0} \ar[ruu]  && \makecell[c]{4_{-5}} \ar[ruu] && \makecell[c]{4_{-3}4_{-1}} \ar[ruu]  && \makecell[c]{2_{-1}3_{-4}3_{-6}}  \ar[ruu]  && \makecell[c]{1_{-4}} \ar[ruu]   && \makecell[c]{1_{-2}1_{0}} \ar[ruu]   &&  \makecell[c]{4_{-1}3_{-4}3_{-6}} 
}}}
\caption{All the real prime simple modules in $\mathscr{C}^{\leq \xi}_2$, excluding the frozen real prime simple modules $1_{-4}1_{-2}1_0$,  $2_{-5}2_{-3}2_{-1}$, $3_{-6}3_{-4}3_{-2}$, and $4_{-5}4_{-3}4_{-1}$.} \label{all the real prime simple modules C23 in A4}
\end{figure}

Let $\xi(1,2,3,4)=(0,-1,0,-1)$. The initial quiver and the principal quiver of the cluster algebra $K_0(\mathscr{C}^{\leq \xi}_2)$ are shown in Figure \ref{the initial quiver (left) and its principal quiver (right) 4}. The indecomposable rigid objects in $\mathcal{C}^{\leq \xi}_2$ are shown in Figure \ref{the cluster category C24 in A4}. The real prime simple modules (excluding the frozen modules) in $\mathscr{C}^{\leq \xi}_2$ are shown in Figure \ref{all the real prime simple modules C24 in A4}.  The two set has a bijection under the map $L(m) \to K^{\leq \xi}_\ell(m)$.

\begin{figure}
\resizebox{.8\width}{.8\height}{ 
\begin{minipage}{0.5\textwidth}
\begin{xy}
(-25,50)*+{(1,0)}="a";
(0,50)*+{(2,-1)}="b"; 
(-25,30)*+{(1,-2)}="c";
(0,30)*+{(2,-3)}="d"; 
(-25,10)*+{\fbox{(1,-4)}}="e";
(0,10)*+{\fbox{(2,-5)}}="f";
(25,50)*+{(3,0)}="g";
(50,50)*+{(4,-1)}="h"; 
(25,30)*+{(3,-2)}="i";
(50,30)*+{(4,-3)}="j"; 
(25,10)*+{\fbox{(3,-4)}}="k";
(50,10)*+{\fbox{(4,-5)}}="l";
{\ar "a";"b"}; {\ar "g";"b"}; {\ar "g";"h"};
{\ar "b";"c"};
{\ar "c";"d"};{\ar "i";"d"};{\ar "i";"j"};
{\ar "d";"e"};{\ar "f";"d"};
{\ar "d";"b"};
{\ar "e";"c"};
{\ar "c";"a"};
{\ar "k";"i"};{\ar "i";"g"};
{\ar "l";"j"};{\ar "j";"h"};
{\ar "b";"i"};{\ar "h";"i"};
{\ar "d";"k"};{\ar "j";"k"};
\end{xy}
\end{minipage} 
\qquad 
\begin{minipage}{0.5\textwidth}
\begin{xy}
(-25,50)*+{1}="a";
(0,50)*+{2}="b"; 
(-25,30)*+{5}="c";
(0,30)*+{6}="d"; 
(25,50)*+{3}="g";
(50,50)*+{4}="h"; 
(25,30)*+{7}="i";
(50,30)*+{8}="j"; 
{\ar "a";"b"}; {\ar "g";"b"}; {\ar "g";"h"};
{\ar "b";"c"};
{\ar "c";"d"};{\ar "i";"d"};{\ar "i";"j"};
{\ar "d";"b"};
{\ar "c";"a"};
{\ar "i";"g"};
{\ar "j";"h"};
{\ar "b";"i"};
{\ar "h";"i"};
\end{xy}
\end{minipage}} 
\caption{The initial quiver (left) of $K_0(\mathscr{C}^{\leq \xi}_2)$ and its principal quiver (right).} \label{the initial quiver (left) and its principal quiver (right) 4}
\end{figure}

\begin{figure}\Large
\centerline{
\resizebox{.22\width}{.3\height}{ 
\xymatrix@C5pt@C5pt{ 
 && \makecell[c]{0100\\1111} \ar[rdd]  && \makecell[c]{0010\\0010}  \ar[rdd]\ar[rdd] && \makecell[c]{00-10\\0000} \ar[rdd] &&  \makecell[c]{0111\\1110}  \ar[rdd] && \makecell[c]{0000\\0-100}\ar[rdd] && \makecell[c]{0100\\0100}  \ar[rdd] &&  \makecell[c]{1111\\0010} \ar[rdd] && \makecell[c]{0000\\00-10} \ar[rdd] && \makecell[c]{0111\\0111}  \ar[rdd] && \makecell[c]{1110\\0000}  \ar[rdd] && \makecell[c]{0000\\0100} \ar[rdd] && \makecell[c]{1111\\1111}  \ar[rdd] && \makecell[c]{0010\\0000}  \ar[rdd] && \makecell[c]{0000\\0111} \ar[rdd] &&  \makecell[c]{1110\\1110} \ar[rdd]  && \makecell[c]{0-100\\0000} \ar[rdd] && \makecell[c]{0100\\1111}  \\
\\
 & \makecell[c]{0000\\1111}  \ar[rdd] \ar[ruu] & & \makecell[c]{0110\\1121} \ar[rdd]\ar[ruu]  && \makecell[c]{0000\\0010}  \ar[rdd] \ar[ruu] && \makecell[c]{0101\\1110}  \ar[ruu]\ar[rdd]  &&  \makecell[c]{0111\\1010}  \ar[ruu]\ar[rdd] && \makecell[c]{0100\\0000}  \ar[ruu]\ar[rdd] &&  \makecell[c]{1211\\0110}  \ar[ruu]\ar[rdd] && \makecell[c]{1111\\0000} \ar[ruu]\ar[rdd] && \makecell[c]{0111\\0101} \ar[ruu]\ar[rdd] && \makecell[c]{1221\\0111} \ar[ruu]\ar[rdd] && \makecell[c]{1110\\0100} \ar[ruu]\ar[rdd] && \makecell[c]{1111\\1211} \ar[ruu] \ar[rdd] && \makecell[c]{1121\\1111}  \ar[ruu]\ar[rdd] && \makecell[c]{0010\\0111} \ar[ruu]\ar[rdd] && \makecell[c]{1110\\1221}  \ar[ruu]\ar[rdd] && \makecell[c]{1010\\1110} \ar[ruu]\ar[rdd]  && \makecell[c]{0000\\1111}  \ar[ruu] \\
 \\
\makecell[c]{1010\\2221} \ar[r] \ar[ruu] \ar[rdd] & \makecell[c]{0010\\1110} \ar[r] & \ar[r] \makecell[c]{0010\\1121}  \ar[ruu]\ar[rdd] &  \ar[r]  \makecell[c]{0000\\0011} &\ar[r]  \makecell[c]{0100\\1121}  \ar[ruu] \ar[rdd]  &  \ar[r] \makecell[c]{0100\\1110}  &\ar[r] \makecell[c]{0101\\1120}  \ar[ruu] \ar[rdd]  &  \ar[r] \makecell[c]{0001\\0010}  & \ar[r] \makecell[c]{0101\\1010}  \ar[rdd]\ar[ruu] &  \ar[r] \makecell[c]{0100\\1000}  & \ar[r] \makecell[c]{0211\\1010} \ar[rdd]\ar[ruu]&  \ar[r] \makecell[c]{0111\\0010}  & \ar[r] \makecell[c]{1211\\0010} \ar[rdd]\ar[ruu]  &  \ar[r] \makecell[c]{1100\\0000}  &\ar[r] \ar[r] \makecell[c]{1211\\0100} \ar[rdd]\ar[ruu] &  \ar[r] \makecell[c]{0111\\0100} & \ar[r] \makecell[c]{1222\\0101}  \ar[rdd]\ar[ruu] &  \ar[r] \makecell[c]{1111\\0001}  &\ar[r] \makecell[c]{1221\\0101} \ar[rdd]\ar[ruu] &  \ar[r] \makecell[c]{0110\\0100} &\ar[r]  \makecell[c]{1221\\0211} \ar[rdd]\ar[ruu]&  \ar[r]  \makecell[c]{1111\\0111}  & \ar[r]  \makecell[c]{2221\\1211}  \ar[rdd]\ar[ruu]&   \ar[r] \makecell[c]{1110\\1100}  &\ar[r]  \makecell[c]{1121\\1211}  \ar[rdd] \ar[ruu] &  \ar[r] \makecell[c]{0011\\0111} & \ar[r]  \makecell[c]{1121\\1222} \ar[rdd]\ar[ruu] &  \ar[r] \makecell[c]{1110\\1111}  &\ar[r] \makecell[c]{1120\\1221}  \ar[rdd]\ar[ruu] &  \ar[r] \makecell[c]{0010\\0110}  & \ar[r]  \makecell[c]{1010\\1221}  \ar[rdd]\ar[ruu] &  \ar[r] \makecell[c]{1000\\1111} & \ar[r]  \makecell[c]{1010\\2221}  \ar[rdd]\ar[ruu]  &  \makecell[c]{0010\\1110}  \\
 \\
 & \makecell[c]{1010\\1121}  \ar[ruu] \ar[rdd] &&  \makecell[c]{0000\\1110} \ar[ruu] \ar[rdd] &&  \makecell[c]{0101\\1121} \ar[ruu] \ar[rdd] && \makecell[c]{0100\\1010} \ar[ruu] \ar[rdd]  &&  \makecell[c]{0101\\0010}  \ar[ruu] \ar[rdd] && \makecell[c]{1211\\1010} \ar[ruu] \ar[rdd] &&  \makecell[c]{0111\\0000}  \ar[ruu] \ar[rdd] && \makecell[c]{1211\\0101} \ar[ruu] \ar[rdd] && \makecell[c]{1221\\0100}  \ar[ruu] \ar[rdd] &&  \makecell[c]{1111\\0101}  \ar[ruu] \ar[rdd] && \makecell[c]{1221\\1211}  \ar[ruu] \ar[rdd] && \makecell[c]{1121\\0111} \ar[ruu] \ar[rdd] && \makecell[c]{1110\\1211}  \ar[ruu] \ar[rdd] && \makecell[c]{1121\\1221}  \ar[ruu] \ar[rdd] && \makecell[c]{1010\\1111} \ar[ruu] \ar[rdd] && \makecell[c]{0010\\1221} \ar[ruu] \ar[rdd] && \makecell[c]{1010\\1121} \ar[rdd] \\
 \\
&& \makecell[c]{1000\\1110}  \ar[ruu] \ar[rdd] && \makecell[c]{0001\\1110} \ar[ruu] \ar[rdd] && \makecell[c]{0100\\1011} \ar[ruu] \ar[rdd] &&  \makecell[c]{0100\\0010}  \ar[ruu] \ar[rdd] &&  \makecell[c]{1101\\0010} \ar[ruu] \ar[rdd] &&  \makecell[c]{0111\\1000}  \ar[ruu] \ar[rdd] && \makecell[c]{0111\\0001} \ar[ruu] \ar[rdd] && \makecell[c]{1210\\0100} \ar[ruu] \ar[rdd] &&  \makecell[c]{1111\\0100}  \ar[ruu] \ar[rdd] && \makecell[c]{1111\\1101} \ar[ruu] \ar[rdd] && \makecell[c]{0121\\0111}  \ar[ruu] \ar[rdd] && \makecell[c]{1110\\0111} \ar[ruu] \ar[rdd] && \makecell[c]{1110\\1210}  \ar[ruu] \ar[rdd] && \makecell[c]{1011\\1111}  \ar[ruu] \ar[rdd] && \makecell[c]{0010\\1111}  \ar[ruu] \ar[rdd] && \makecell[c]{0010\\0121}  \ar[ruu]  \ar[rdd] && \makecell[c]{1000\\1110}  \ar[rdd] \\
\\
&&& \makecell[c]{1001\\1110} \ar[ruu] \ar[rdd] &&  \makecell[c]{0000\\1000} \ar[ruu] \ar[rdd] && \makecell[c]{0100\\0011}  \ar[ruu] \ar[rdd] && \makecell[c]{1100\\0010}  \ar[ruu] \ar[rdd] && \makecell[c]{0001\\0000} \ar[ruu] \ar[rdd] &&  \makecell[c]{0111\\1001} \ar[ruu] \ar[rdd] && \makecell[c]{0110\\0000}  \ar[ruu] \ar[rdd] && \makecell[c]{1100\\0100}  \ar[ruu] \ar[rdd] && \makecell[c]{1111\\1100}  \ar[ruu] \ar[rdd] && \makecell[c]{0011\\0001} \ar[ruu] \ar[rdd] && \makecell[c]{0110\\0111}  \ar[ruu] \ar[rdd] && \makecell[c]{1110\\0110} \ar[ruu] \ar[rdd] && \makecell[c]{1000\\1100}  \ar[ruu] \ar[rdd] && \makecell[c]{0011\\1111}  \ar[ruu] \ar[rdd] &&  \makecell[c]{0010\\0011}  \ar[ruu] \ar[rdd] && \makecell[c]{0000\\0110}  \ar[ruu] \ar[rdd] && \makecell[c]{1001\\1110}  \ar[rdd] \\
 \\
&&&& \makecell[c]{1000\\1000} \ar[ruu] &&  \makecell[c]{-1000\\0000} \ar[ruu] &&  \makecell[c]{1100\\0011}  \ar[ruu] && \makecell[c]{0000\\000-1}  \ar[ruu] && \makecell[c]{0001\\0001} \ar[ruu] && \makecell[c]{0110\\1000}  \ar[ruu]  &&  \makecell[c]{0000\\-1000}  \ar[ruu] && \makecell[c]{1100\\1100}  \ar[ruu] &&  \makecell[c]{0011\\0000}  \ar[ruu]  && \makecell[c]{0000\\0001} \ar[ruu]   && \makecell[c]{0110\\0110} \ar[ruu]  && \makecell[c]{1000\\0000} \ar[ruu] &&  \makecell[c]{0000\\1100} \ar[ruu]  && \makecell[c]{0011\\0011}  \ar[ruu]  && \makecell[c]{000-1\\0000} \ar[ruu]   && \makecell[c]{0001\\0110}  \ar[ruu]   && \makecell[c]{1000\\1000}
}}}
\caption{The cluster category $\mathcal{C}^{\leq \xi}_2$ of $kQ^{\leq \xi}_2$, and the objects on the far left are to be identified with the objects on the far right.} \label{the cluster category C24 in A4}
\end{figure}

\begin{figure}
\centerline{
\resizebox{.1\width}{.3\height}{ 
\xymatrix{ 
&& \makecell[c]{3_{0}3_{-2}2_{-5}4_{-5}} \ar[rdd]  && \makecell[c]{3_{-4}3_{-2}}   \ar[rdd]\ar[rdd] &&  \makecell[c]{3_{0}}  \ar[rdd] && \makecell[c]{2_{-5}} \ar[rdd] && \makecell[c]{2_{-3}2_{-1}} \ar[rdd] && \makecell[c]{1_{0}3_{0}2_{-3}2_{-5}}  \ar[rdd] &&  \makecell[c]{3_{-4}}  \ar[rdd] && \makecell[c]{3_{-2}3_0} \ar[rdd] && \makecell[c]{1_{0}3_{0}2_{-3}4_{-3}2_{-5}4_{-5}} \ar[rdd] &&  \makecell[c]{2_{-3}} \ar[rdd] && \makecell[c]{1_{0}3_{0}1_{-2}3_{-2}2_{-5}}  \ar[rdd] && \makecell[c]{3_{0}2_{-3}4_{-3}2_{-5}4_{-5}} \ar[rdd] && \makecell[c]{3_{-2}}  \ar[rdd] && \makecell[c]{1_{0}3_{0}1_{-2}3_{-2}2_{-5}4_{-5}}  \ar[rdd] && \makecell[c]{2_{-5}2_{-3}}  \ar[rdd] &&  \makecell[c]{2_{-1}}  \ar[rdd] && \makecell[c]{3_{0}3_{-2}2_{-5}4_{-5}}   \\
 \\
 & \makecell[c]{2_{-1}3_{0}3_{-2}2_{-5}4_{-5}}  \ar[rdd] \ar[ruu] &&  \makecell[c]{3_{-2}2_{-5}4_{-5}}  \ar[rdd]\ar[ruu]  && \makecell[c]{2_{-1}3_{-4}4_{-1}} \ar[rdd] \ar[ruu] &&  \makecell[c]{2_{-5}3_{0}} \ar[ruu]\ar[rdd]  &&  \makecell[c]{1_{-4}2_{-1}3_{-4}} \ar[ruu]\ar[rdd] &&  \makecell[c]{1_{0}2_{-3}3_{0}} \ar[ruu]\ar[rdd] && \makecell[c]{1_{0}3_{0}2_{-3}2_{-5}3_{-4}}  \ar[ruu]\ar[rdd] &&  \makecell[c]{2_{-3}3_{0}4_{-3}}  \ar[ruu]\ar[rdd] && \makecell[c]{1_{0}3_{0}^{2}2_{-3}3_{-2}4_{-3}2_{-5}4_{-5}} \ar[ruu]\ar[rdd] &&  \makecell[c]{1_{0}3_{0}2_{-3}^{2}4_{-3}2_{-5}4_{-5}} \ar[ruu]\ar[rdd] &&  \makecell[c]{1_{0}3_{0}1_{-2}2_{-3}3_{-2}2_{-5}} \ar[ruu]\ar[rdd] && \makecell[c]{1_{0}3_{0}^{2}1_{-2}2_{-3}3_{-2}4_{-3}2_{-5}^{2}4_{-5}}  \ar[ruu]\ar[rdd] &&  \makecell[c]{3_{0}2_{-3}3_{-2}4_{-3}2_{-5}4_{-5}} \ar[ruu]\ar[rdd] &&  \makecell[c]{1_{0}3_{0}1_{-2}3_{-2}^{2}2_{-5}4_{-5}} \ar[ruu]\ar[rdd] && \makecell[c]{1_{0}3_{0}1_{-2}2_{-3}3_{-2}2_{-5}^{2}4_{-5}} \ar[ruu]\ar[rdd] &&  \makecell[c]{1_{-2}2_{-5}3_{-2}} \ar[ruu]\ar[rdd]  && \makecell[c]{2_{-1}3_{0}3_{-2}2_{-5}4_{-5}}  \ar[ruu] \\
 \\
\makecell[c]{3_{0}1_{-2}3_{-2}^{2}2_{-5}^{2}4_{-5}} \ar[r] \ar[ruu] \ar[rdd] &  \makecell[c]{3_{-2}2_{-5}} \ar[r] & \ar[r]  \makecell[c]{2_{-1}3_{-2}2_{-5}4_{-5}} \ar[ruu]\ar[rdd]  &  \ar[r] \makecell[c]{2_{-1}4_{-5}}  &\ar[r]  \makecell[c]{2_{-1}4_{-1}2_{-5}4_{-5}} \ar[ruu] \ar[rdd]  &  \ar[r]  \makecell[c]{4_{-1}2_{-5}} &\ar[r] \makecell[c]{2_{-1}4_{-1}2_{-5}3_{-4}} \ar[ruu] \ar[rdd] &  \ar[r] \makecell[c]{2_{-1}3_{-4}}  & \ar[r] \makecell[c]{2_{-1}3_{0}1_{-4}3_{-4}} \ar[rdd]\ar[ruu]&  \ar[r] \makecell[c]{3_{0}1_{-4}}  & \ar[r] \makecell[c]{1_{0}3_{0}1_{-4}3_{-4}} \ar[rdd]\ar[ruu] &  \ar[r] \makecell[c]{1_{0}3_{-4}} &\ar[r] \makecell[c]{1_{0}3_{0}2_{-3}3_{-4}} \ar[rdd]\ar[ruu]&  \ar[r] \makecell[c]{3_{0}2_{-3}} & \ar[r] \makecell[c]{1_{0}3_{0}^{2}2_{-3}^{2}4_{-3}2_{-5}} \ar[rdd] \ar[ruu]  &  \ar[r] \makecell[c]{1_{0}3_{0}2_{-3}4_{-3}2_{-5}}  & \ar[r]  \makecell[c]{1_{0}3_{0}^{2}2_{-3}^{2}4_{-3}^{2}2_{-5}4_{-5}} \ar[rdd]\ar[ruu] &  \ar[r] \makecell[c]{3_{0}2_{-3}4_{-3}4_{-5}} &\ar[r] \makecell[c]{1_{0}3_{0}^{2}2_{-3}^{2}3_{-2}4_{-3}2_{-5}4_{-5}}  \ar[rdd]\ar[ruu] &  \ar[r] \makecell[c]{1_{0}3_{0}2_{-3}3_{-2}2_{-5}}  & \ar[r]  \makecell[c]{1_{0}^{2}3_{0}^{2}1_{-2}2_{-3}^{2}3_{-2}4_{-3}2_{-5}^{2}4_{-5}} \ar[rdd]\ar[ruu] &  \ar[r]  \makecell[c]{1_{0}3_{0}1_{-2}2_{-3}4_{-3}2_{-5}4_{-5}}  & \ar[r]  \makecell[c]{1_{0}3_{0}^{2}1_{-2}2_{-3}^{2}3_{-2}4_{-3}2_{-5}^{2}4_{-5}} \ar[rdd]\ar[ruu]&  \ar[r] \makecell[c]{3_{0}2_{-3}3_{-2}2_{-5}}  &\ar[r] \makecell[c]{1_{0}3_{0}^{2}1_{-2}2_{-3}3_{-2}^{2}4_{-3}2_{-5}^{2}4_{-5}} \ar[rdd]\ar[ruu]&  \ar[r]  \makecell[c]{1_{0}3_{0}1_{-2}3_{-2}4_{-3}2_{-5}4_{-5}} & \ar[r]  \makecell[c]{1_{0}3_{0}^{2}1_{-2}2_{-3}3_{-2}^{2}4_{-3}2_{-5}^{2}4_{-5}^{2}} \ar[rdd]\ar[ruu]&  \ar[r] \makecell[c]{3_{0}2_{-3}3_{-2}2_{-5}4_{-5}}  &\ar[r] \makecell[c]{1_{0}3_{0}1_{-2}2_{-3}3_{-2}^{2}2_{-5}^{2}4_{-5}} \ar[rdd]\ar[ruu] &  \ar[r] \makecell[c]{1_{0}1_{-2}3_{-2}2_{-5}} &\ar[r]  \makecell[c]{1_{0}3_{0}1_{-2}^{2}3_{-2}^{2}2_{-5}^{2}4_{-5}} \ar[rdd]\ar[ruu] &  \ar[r] \makecell[c]{3_{0}1_{-2}3_{-2}2_{-5}4_{-5}}  &\ar[r] \makecell[c]{3_{0}1_{-2}3_{-2}^{2}2_{-5}^{2}4_{-5}} \ar[rdd]\ar[ruu]  &  \makecell[c]{3_{-2}2_{-5}}  \\
 \\
& \makecell[c]{1_{-2}3_{-2}2_{-5}4_{-5}}  \ar[ruu] \ar[rdd] && \makecell[c]{2_{-1}4_{-1}2_{-5}}  \ar[ruu] \ar[rdd] && \makecell[c]{2_{-1}2_{-5}4_{-5}} \ar[ruu] \ar[rdd] && \makecell[c]{2_{-1}4_{-1}1_{-4}3_{-4}} \ar[ruu] \ar[rdd] &&  \makecell[c]{1_{0}3_{0}3_{-4}}  \ar[ruu] \ar[rdd] && \makecell[c]{1_{-4}3_{0}3_{-4}} \ar[ruu] \ar[rdd] && \makecell[c]{1_{0}3_{0}2_{-3}4_{-3}} \ar[ruu] \ar[rdd] && \makecell[c]{1_{0}3_{0}^{2}2_{-3}^{2}4_{-3}2_{-5}4_{-5}} \ar[ruu] \ar[rdd] && \makecell[c]{1_{0}3_{0}2_{-3}^{2}4_{-3}2_{-5}} \ar[ruu] \ar[rdd] && \makecell[c]{1_{0}3_{0}^{2}1_{-2}2_{-3}3_{-2}4_{-3}2_{-5}4_{-5}} \ar[ruu] \ar[rdd] && \makecell[c]{1_{0}3_{0}^{2}2_{-3}^{2}3_{-2}4_{-3}2_{-5}^{2}4_{-5}} \ar[ruu] \ar[rdd] && \makecell[c]{1_{0}3_{0}1_{-2}2_{-3}3_{-2}4_{-3}2_{-5}4_{-5}} \ar[ruu] \ar[rdd] && \makecell[c]{1_{0}3_{0}^{2}1_{-2}2_{-3}3_{-2}^{2}2_{-5}^{2}4_{-5}} \ar[ruu] \ar[rdd] && \makecell[c]{1_{0}3_{0}1_{-2}2_{-3}3_{-2}4_{-3}2_{-5}^{2}4_{-5}} \ar[ruu] \ar[rdd] && \makecell[c]{3_{0}1_{-2}3_{-2}^{2}2_{-5}4_{-5}} \ar[ruu] \ar[rdd] && \makecell[c]{1_{0}3_{0}1_{-2}3_{-2}^{2}2_{-5}^{2}4_{-5}} \ar[ruu] \ar[rdd] && \makecell[c]{1_{-2}3_{-2}2_{-5}4_{-5}} \ar[rdd] \\
 \\
&&  \makecell[c]{4_{-1}1_{-2}2_{-5}}  \ar[ruu] \ar[rdd] && \makecell[c]{2_{-1}2_{-5}} \ar[ruu] \ar[rdd] && \makecell[c]{2_{-1}1_{-4}4_{-5}} \ar[ruu] \ar[rdd] && \makecell[c]{1_{0}4_{-1}3_{-4}} \ar[ruu] \ar[rdd] &&  \makecell[c]{3_{0}3_{-4}}  \ar[ruu] \ar[rdd] && \makecell[c]{3_{0}4_{-3}1_{-4}} \ar[ruu] \ar[rdd] && \makecell[c]{1_{0}3_{0}2_{-3}4_{-3}4_{-5}}  \ar[ruu] \ar[rdd] &&  \makecell[c]{1_{0}3_{0}2_{-3}^{2}2_{-5}} \ar[ruu] \ar[rdd] &&  \makecell[c]{1_{0}3_{0}1_{-2}2_{-3}4_{-3}2_{-5}} \ar[ruu] \ar[rdd] && \makecell[c]{3_{0}^{2}2_{-3}3_{-2}4_{-3}2_{-5}4_{-5}} \ar[ruu] \ar[rdd] &&  \makecell[c]{1_{0}3_{0}2_{-3}3_{-2}4_{-3}2_{-5}4_{-5}} \ar[ruu] \ar[rdd]  && \makecell[c]{1_{0}3_{0}1_{-2}2_{-3}3_{-2}2_{-5}4_{-5}} \ar[ruu] \ar[rdd] &&  \makecell[c]{1_{0}3_{0}1_{-2}2_{-3}3_{-2}2_{-5}^{2}} \ar[ruu] \ar[rdd] && \makecell[c]{3_{0}1_{-2}3_{-2}4_{-3}2_{-5}4_{-5}} \ar[ruu] \ar[rdd] && \makecell[c]{3_{0}3_{-2}^{2}2_{-5}4_{-5}} \ar[ruu] \ar[rdd] && \makecell[c]{1_{0}1_{-2}3_{-2}2_{-5}4_{-5}} \ar[ruu]  \ar[rdd] && \makecell[c]{4_{-1}1_{-2}2_{-5}} \ar[rdd] \\
\\
&&& \makecell[c]{1_{-2}2_{-5}} \ar[ruu] \ar[rdd] && \makecell[c]{2_{-1}1_{-4}} \ar[ruu] \ar[rdd] && \makecell[c]{1_{0}4_{-5}} \ar[ruu] \ar[rdd] && \makecell[c]{4_{-1}3_{-4}} \ar[ruu] \ar[rdd] && \makecell[c]{3_{0}4_{-3}}  \ar[ruu] \ar[rdd] && \makecell[c]{3_{0}4_{-3}1_{-4}4_{-5}} \ar[ruu] \ar[rdd] && \makecell[c]{1_{0}2_{-3}} \ar[ruu] \ar[rdd] &&  \makecell[c]{1_{0}3_{0}1_{-2}2_{-3}2_{-5}} \ar[ruu] \ar[rdd] && \makecell[c]{3_{0}2_{-3}4_{-3}2_{-5}}  \ar[ruu] \ar[rdd] && \makecell[c]{3_{0}3_{-2}4_{-3}4_{-5}} \ar[ruu] \ar[rdd] &&  \makecell[c]{1_{0}3_{0}2_{-3}3_{-2}2_{-5}4_{-5}} \ar[ruu] \ar[rdd] && \makecell[c]{1_{0}1_{-2}2_{-3}2_{-5}} \ar[ruu] \ar[rdd] &&  \makecell[c]{3_{0}1_{-2}3_{-2}2_{-5}} \ar[ruu] \ar[rdd] && \makecell[c]{3_{0}3_{-2}4_{-3}2_{-5}4_{-5}} \ar[ruu] \ar[rdd] && \makecell[c]{3_{-2}4_{-5}} \ar[ruu] \ar[rdd] && \makecell[c]{1_{0}4_{-1}1_{-2}2_{-5}} \ar[ruu] \ar[rdd] && \makecell[c]{1_{-2}2_{-5}}  \ar[rdd] \\
 \\
&&&&  \makecell[c]{1_{-2}1_{-4}}  \ar[ruu] &&  \makecell[c]{1_{0}}  \ar[ruu] && \makecell[c]{4_{-5}} \ar[ruu] && \makecell[c]{4_{-3}4_{-1}} \ar[ruu] && \makecell[c]{3_{0}4_{-3}4_{-5}} \ar[ruu] && \makecell[c]{1_{-4}} \ar[ruu]  &&  \makecell[c]{1_{-2}1_0} \ar[ruu] && \makecell[c]{3_02_{-3}2_{-5}}  \ar[ruu] && \makecell[c]{4_{-3}} \ar[ruu]  && \makecell[c]{3_03_{-2}4_{-5}} \ar[ruu]   && \makecell[c]{1_02_{-3}2_{-5}} \ar[ruu]  && \makecell[c]{1_{-2}} \ar[ruu] && \makecell[c]{3_03_{-2}2_{-5}} \ar[ruu]  && \makecell[c]{4_{-3}4_{-5}}  \ar[ruu]  && \makecell[c]{4_{-1}} \ar[ruu]   && \makecell[c]{1_{0}1_{-2}2_{-5}}  \ar[ruu]   &&   \makecell[c]{1_{-2}1_{-4}}
}}}
\caption{All the real prime simple modules in $\mathscr{C}^{\leq \xi}_2$, excluding the frozen real prime simple modules $1_{-4}1_{-2}1_0$,  $2_{-5}2_{-3}2_{-1}$, $3_{-4}3_{-2}3_0$, and $4_{-5}4_{-3}4_{-1}$.} \label{all the real prime simple modules C24 in A4}
\end{figure}

\section{The highest $l$-weight monomials of Hernandez-Leclerc modules} \label{The highest l-monomials of Hernandez-Leclerc modules}
In this section, we use Theorem \ref{main theorem2} and Equation (\ref{Gamma1 equation2}) to obtain the highest $l$-weight monomials of unfrozen Hernandez-Leclerc modules in type ADE. 

We use the following notation throughout. Let $Q$ ($=Q^{\leq \xi}_1$ for some $\xi$ in the previous section) be an orientation of $\gamma$, and $\mathcal{C}_Q$ the cluster category associated to $Q$. For $j\in I$, let $d_j=\delta_{j,j_\diamond}$,
\[
j_\diamond=\min\{ k \geq j \mid \xi(k-1)=\xi(k+1)\}, \,\, j_\bullet=\max\{k<j \mid \xi(k-1)=\xi(k+1)\}.
\] 
We agree that $j_\diamond=n$ if there  is no $ k \geq j$ such that $\xi(k-1)=\xi(k+1)$, and  $j_\bullet=1$ if there is no $k < j$ such that $\xi(k-1)=\xi(k+1)$.

\subsection{Type $A_n$}

Let $1\leq i \leq j \leq n$. Every positive root in type $A_n$ has the form $\beta=\sum_{i\leq k \leq j} \alpha_k$, so we identify $\beta\in \Phi^+$ with a segment $[i,j]$. We denote by $[-i]$ the negative simple root $-\alpha_i$. 

Then for all $i\le j$, the two cluster variables $x([i,j])$ and $x([-j])$ form an exchange pair. It follows from the geometric realization of $\mathcal{C}_Q$ in Section \ref{a geometric realization of cluster category} that the corresponding indecomposable rigid objects in $\mathcal{C}_Q$ satisfy the following triangles: for $i=j\in I$, 
\begin{align*}
& [j] \to  [-(j-1)] \oplus [-(j+1)]^{d_j} \to [-j] \to [j][\textbf{1}], \\ 
& [-j] \to [-(j+1)]^{1-d_j} \to [j] \to [-j][\textbf{1}],
\end{align*}
or the triangles obtained by switching the roles of $[j]$ and $[-j]$; and for $j>i$, 
\begin{align*}
&  [i,j] \to  [i,j-1] \oplus   [-(j+1)]^{1-d_j}  \to [-j] \to  [i,j] [\textbf{1}], \\
&  [-j] \to [i,\max\{i-1,j_\bullet-1\}]^{1-\delta_{i,j_\bullet}} \oplus   [-(j+1)]^{d_j}  \to [i,j] \to  [-j][\textbf{1}]
\end{align*}
or the triangles obtained by switching the roles of  $[i,j]$ and $[-j]$, where we use the convention $[i,i-1]=[-(i-1)]$, and $[\textbf{1}]: \mathcal{C}_Q \to \mathcal{C}_Q$ is the shift functor. Following \cite{CK06}, we have the following exchange relations: for $j\in I$, 
\begin{align}\label{exchange relation (1) in An}
x([j]) x([-j])=x([-(j-1)]) x([-(j+1)])^{d_j}+x([-(j+1)])^{1-d_j},
\end{align}
and for $j>i$, 
\begin{align}\label{exchange relation (2) in An}
\begin{split}
x([i,j]) x([-j]) = & x([i,\max\{i-1,j_\bullet-1\}])^{1-\delta_{i,j_\bullet}} x([-(j+1)])^{d_j} \\
&  + x([i,j-1]) x([-(j+1)])^{1-d_j}.
\end{split}
\end{align}

In particular, if $Q$ is an alternating quiver, our formulas reduce to these formulas considered in \cite{CFZ02,FZ03}. By Equation (\ref{Gamma1 equation2}), the exchange relations (\ref{exchange relation (1) in An}) and (\ref{exchange relation (2) in An}) can be lifted into the following relations in $K_0(\mathscr{C}^{\leq \xi}_1)$: for $j\in I$,
\begin{gather} 
\Phi([j]) \Phi([-j])= \begin{cases}
\Phi([-(j-1)]) \Phi([-(j+1)])^{d_j} f_j  + \text{other terms}  &  \text{if $\xi(j-1)>\xi(j)$},  \\
\Phi([-(j+1)])^{1-d_j} f_j + \text{other terms}  & \text{if $\xi(j-1)<\xi(j)$}, 
\end{cases} \label{AC1 initial and non-initial objects}
\end{gather}
and for $j>i$,
\begin{gather}
\Phi([i,j]) \Phi([-j])=\begin{cases}
\Phi([i,\max\{i-1,j_\bullet-1\}])^{1-\delta_{i,j_\bullet}} \Phi([-(j+1)])^{d_j}f_j +\text{other terms} & \text{if $\xi(j-1)>\xi(j)$},  \\
\Phi([i,j-1]) \Phi([-(j+1)])^{1-d_j} + \text{other terms} & \text{if $\xi(j-1)<\xi(j)$}.
\end{cases} \label{AC1 other objects}
\end{gather}

Assume that $i_1<i_2<\cdots<i_{k}$ is an ordered enumeration of the following set 
\begin{align} \label{index enumeration}
\{ p: i<p< j, \text{ $p$ is a source or a sink}\}.
\end{align}
Define 
\[
Y(i,j)=Y_{i_1,a_1} Y_{i_2,a_2} \cdots Y_{i_k,a_k}
\] 
for open intervals $(i,j)$, where $a_\ell=\xi(i_\ell)$ if $\xi(i_\ell)=\xi(i_\ell+1)+1$, and $a_\ell=\xi(i_\ell)-2$ if $\xi(i_\ell)=\xi(i_\ell+1)-1$ for $1 \leq \ell \leq k$. Similarly, one can define it for arbitrary intervals. If the set (\ref{index enumeration}) is empty, we set $Y(i,j)=1$. 

\begin{proposition}\label{monoidal categorification AC1}
 The bijection defined in (\ref{a bijection Phi}) is given explicitly as follows: for any $j\in I$, 
\begin{gather}
\begin{align*}
& \Phi([-j])=L(Y_{j,\xi(j)}), \\ 
& \Phi([j])=\begin{cases}
L(Y_{j-1,\xi(j-1)} Y_{j,\xi(j)-2} Y^{d_j}_{j+1,\xi(j+1)})  & \text{if $\xi(j-1)>\xi(j)$}, \\
L(Y_{j,\xi(j)-2}Y^{1-d_j}_{j+1,\xi(j+1)})  & \text{if $\xi(j-1)<\xi(j)$}, 
\end{cases}
 \end{align*}
\end{gather}
and for $1 \leq i<j \leq n$, 
\begin{gather}
\begin{align*}
& \Phi([i,j])=\begin{cases}
L(Y^{(1-d_i)}_{i-1,\xi(i-1)} Y(i,j) Y_{j,\xi(j)-2} Y^{d_j}_{j+1,\xi(j+1)}) & \text{if $\xi(i)>\xi(i+1)$ and $\xi(j-1)>\xi(j)$}, \\
L(Y^{(1-d_i)}_{i-1,\xi(i-1)} Y(i,j) Y^{(1-d_j)}_{j+1,\xi(j+1)}) & \text{if $\xi(i)>\xi(i+1)$ and $\xi(j-1)<\xi(j)$}, \\
L(Y^{d_i}_{i-1,\xi(i-1)} Y_{i,\xi(i)-2} Y(i,j) Y_{j,\xi(j)-2} Y^{d_j}_{j+1,\xi(j+1)})  & \text{if $\xi(i)<\xi(i+1)$ and $\xi(j-1)>\xi(j)$}, \\
L(Y^{d_i}_{i-1,\xi(i-1)} Y_{i,\xi(i)-2} Y(i,j) Y^{(1-d_j)}_{j+1,\xi(j+1)})  & \text{if $\xi(i)<\xi(i+1)$ and $\xi(j-1)<\xi(j)$}.
\end{cases}
\end{align*}
\end{gather}
Here $Y_{i,p}=1$ for $i\not\in I$, $p\in \mathbb{C}^*$.
\end{proposition}

\begin{proof}
For any $j\in I$, the images $\Phi([-j]), \Phi([j])$ follow from Equation (\ref{AC1 initial and non-initial objects}). If $j>i>j_\bullet$, then by Equation (\ref{AC1 other objects}) 
\begin{align*}
\Phi([i,j]) = \begin{cases}
[L(Y_{i-1,\xi(i-1)}Y_{j,\xi(j)-2} Y_{j+1,\xi(j+1)}^{d_j})] & \text{if $\xi(j-1)>\xi(j)$}, \\
L(Y_{i,\xi(i)-2} Y^{(1-d_j)}_{j+1,\xi(j+1)})  &  \text{if $\xi(j-1)<\xi(j)$}.
\end{cases}
\end{align*}
For $j>i=j_\bullet$, by Equation (\ref{AC1 other objects})
\begin{align*}
\Phi([i,j]) = \begin{cases}
[L(Y_{j,\xi(j)-2} Y_{j+1,\xi(j+1)}^{d_j})] & \text{if $\xi(j-1)>\xi(j)$}, \\
L(Y_{i-1,\xi(i-1)} Y_{i,\xi(i)-2} Y^{(1-d_j)}_{j+1,\xi(j+1)})  &  \text{if $\xi(j-1)<\xi(j)$}.
\end{cases}
\end{align*}
Here the condition $\xi(j-1)>\xi(j)$ (respectively, $\xi(j-1)<\xi(j)$) implies $\xi(i)>\xi(i+1)$ (respectively, $\xi(i)<\xi(i+1)$).

For $i<j_\bullet$ and $\xi(j-1)>\xi(j)$, we assume by induction on the length of $j-i$ that
\begin{gather*}
\Phi([i,j_\bullet-1])=\begin{cases}
L(Y^{(1-d_i)}_{i-1,\xi(i-1)} Y(i,j_\bullet-1) Y_{j_\bullet-1,\xi(j_\bullet-1)-2} Y^{d_{j_\bullet-1}}_{j_\bullet,\xi(j_\bullet)}) & \text{if $\xi(i)>\xi(i+1)$ and $\xi(j_\bullet-2)>\xi(j_\bullet-1)$}, \\
L(Y^{(1-d_i)}_{i-1,\xi(i-1)} Y(i,j_\bullet-1) Y^{(1-d_{j_\bullet-1})}_{j_\bullet,\xi(j_\bullet)}) & \text{if $\xi(i)>\xi(i+1)$ and $\xi(j_\bullet-2)<\xi(j_\bullet-1)$}, \\
L(Y^{d_i}_{i-1,\xi(i-1)} Y_{i,\xi(i)-2} Y(i,j_\bullet-1) Y_{j_\bullet-1,\xi(j_\bullet-1)-2} Y^{d_{j_\bullet-1}}_{j_\bullet,\xi(j_\bullet)})  & \text{if $\xi(i)<\xi(i+1)$ and $\xi(j_\bullet-2)>\xi(j_\bullet-1)$}, \\
L(Y^{d_i}_{i-1,\xi(i-1)} Y_{i,\xi(i)-2} Y(i,j_\bullet-1) Y^{(1-d_{j_\bullet-1})}_{j_\bullet,\xi(j_\bullet)})  & \text{if $\xi(i)<\xi(i+1)$ and $\xi(j_\bullet-2)<\xi(j_\bullet-1)$}.
\end{cases}
\end{gather*}
It follows from Equation (\ref{AC1 other objects}) that 
\begin{gather*}
\Phi([i,j])=\begin{cases}
L(Y^{(1-d_i)}_{i-1,\xi(i-1)} Y(i,j) Y_{j,\xi(j)-2} Y_{j+1,\xi(j+1)}^{d_j}) & \text{if $\xi(i)>\xi(i+1)$ and $\xi(j-1)>\xi(j)$}, \\
L(Y^{d_i}_{i-1,\xi(i-1)} Y_{i,\xi(i)-2} Y(i,j) Y_{j,\xi(j)-2} Y_{j+1,\xi(j+1)}^{d_j})  & \text{if $\xi(i)<\xi(i+1)$ and $\xi(j-1)>\xi(j)$}.
\end{cases}
\end{gather*}

For $i<j_\bullet$ and $\xi(j-1)<\xi(j)$,  we assume by induction on the length of $j-i$ that 
\begin{gather*}
\Phi([i,j-1])=\begin{cases}
L(Y^{(1-d_i)}_{i-1,\xi(i-1)} Y(i,j-1) Y_{j-1,\xi(j-1)-2} Y^{d_{j-1}}_{j,\xi(j)}) & \text{if $\xi(i)>\xi(i+1)$ and $\xi(j-2)>\xi(j-1)$}, \\
L(Y^{(1-d_i)}_{i-1,\xi(i-1)} Y(i,j-1) Y^{(1-d_{j-1})}_{j,\xi(j)}) & \text{if $\xi(i)>\xi(i+1)$ and $\xi(j-2)<\xi(j-1)$}, \\
L(Y^{d_i}_{i-1,\xi(i-1)} Y_{i,\xi(i)-2} Y(i,j-1) Y_{j-1,\xi(j-1)-2} Y^{d_{j-1}}_{j,\xi(j)})  & \text{if $\xi(i)<\xi(i+1)$ and $\xi(j-2)>\xi(j-1)$}, \\
L(Y^{d_i}_{i-1,\xi(i-1)} Y_{i,\xi(i)-2} Y(i,j-1) Y^{(1-d_{j-1})}_{j,\xi(j)})  & \text{if $\xi(i)<\xi(i+1)$ and $\xi(j-2)<\xi(j-1)$}.
\end{cases}
\end{gather*}
Using Equation (\ref{AC1 other objects}) again, we obtain
\begin{gather*}
\Phi([i,j])=\begin{cases}
L(Y^{(1-d_i)}_{i-1,\xi(i-1)} Y(i,j) Y_{j+1,\xi(j+1)}^{(1-d_j)}) & \text{if $\xi(i)>\xi(i+1)$ and $\xi(j-1)<\xi(j)$}, \\
L(Y^{d_i}_{i-1,\xi(i-1)} Y_{i,\xi(i)-2} Y(i,j) Y^{(1-d_j)}_{j+1,\xi(j+1)})  & \text{if $\xi(i)<\xi(i+1)$ and $\xi(j-1)<\xi(j)$}.
\end{cases}
\end{gather*}
\end{proof}

It follows from Equation (\ref{Gamma1 equation}) that all Hernandez-Leclerc modules can be obtained by an induction on the distance between the module $[i,j]$ and a certain injective module in an Auslander-Reiten quiver. It is well-known that every mesh in an Auslander-Reiten quiver is one of the three possibilities shown in Figure \ref{three forms of meshes in AR quiver An}.

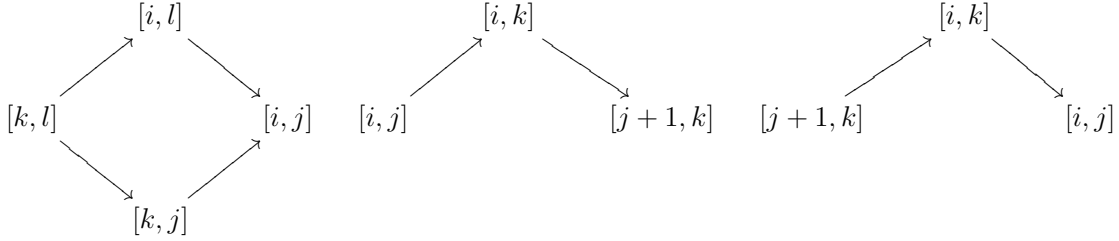
\begin{figure}
\resizebox{.9\width}{.9\height}{ 
\xymatrix{
& [i,l]  \ar[dr] &  \\
[k,l] \ar[ur] \ar[dr] &  & [i,j] \\
& [k,j] \ar[ur]  &} \quad
\xymatrix{
& [i,k] \ar[dr] &  \\
[i,j] \ar[ur] && [j+1,k]} \quad
\xymatrix{
& [i,k] \ar[dr] & \\
[j+1,k] \ar[ur] &  & [i,j]}}
\caption{All possible meshes in the Auslander-Reiten quivers of type $A_n$: integers $i,k \leq j,l$ for the left mesh, and  integers $i<j<k$ for the remainder two meshes.}\label{three forms of meshes in AR quiver An}
\end{figure} 

\begin{theorem} \label{Hernandez-Leclerc module of type A}
An Hernandez-Leclerc module of type $A_n$ is exactly the simple $U_{q}(\widehat{\mathfrak{g}})$-module with the highest $l$-weight monomial equal to
\begin{align} \label{typeAHL}
Y_{i_1,a_1}Y_{i_2,a_2}\ldots Y_{i_k,a_k},
\end{align}
where $k\in\mathbb{Z}_{\geq 1}$, $i_j\in \{1,2,\ldots,n\}$, $a_j\in \mathbb{Z}$ for $j=1,2,\ldots,k$, and
\begin{itemize}
\item[(1)] $i_1<i_2<\cdots<i_k$,
\item[(2)] $(a_{j}-a_{j-1})(a_{j+1}-a_{j})<0$ for $2\leq j\leq k-1$,
\item[(3)] $|a_{j}-a_{j-1}|=i_{j}-i_{j-1}+2$ for $2\leq j\leq k$.
\end{itemize}
\end{theorem}
\begin{proof}
It follows from Proposition \ref{monoidal categorification AC1} that for a fixed height function $\xi$, every Hernandez-Leclerc module has the required highest $l$-weight monomial. 

In the following, we show that given a monomial $m$ of type (\ref{typeAHL}), there exists an Hernandez-Leclerc module with the highest $l$-weight monomial $m$.

If $k=1$, we take the height function $\xi$ such that $\xi(i_1)=a_1$. Then by Proposition \ref{monoidal categorification AC1}, $\Phi([-i_1])=L(Y_{i_1,a_1})$. 

For $k=2$, if $a_1>a_2$, then we choose a height function $\xi$ such that $\xi$ is linear in intervals $[1,i_1], [i_1,n]$, where $\xi(i_1)=a_{1}>\xi(i_2)=a_{1}-i_2+i_1$. It follows from Proposition \ref{monoidal categorification AC1} that $\Phi([i_1+1,i_2])=L(Y_{i_1,a_1}Y_{i_2,a_2})$.

For $k=2$, if $a_1<a_2$, then we choose a height function $\xi$ such that $\xi$ is linear in intervals $[1,n]$, where $\xi(i_1)=a_{2}-i_2+i_1<\xi(i_2)=a_{2}$. It follows from Proposition \ref{monoidal categorification AC1} that $\Phi([i_1,i_2-1])=L(Y_{i_1,a_1}Y_{i_2,a_2})$.  

Assume now that $k\geq 3$. If $a_1>a_2$ and $a_{k-1}>a_k$, then we choose a height function $\xi$ such that $\xi$ is linear in intervals $[1,i_1], [i_1,i_2], \ldots, [i_{k-1},n]$, where $\xi(i_1)=a_{1}>\xi(i_2)=a_{1}-i_2+i_1$. It follows from Proposition \ref{monoidal categorification AC1} that $
  \Phi([i_1+1,i_k]) $ is equal to
\begin{gather}
\begin{align*}
\begin{cases}
L(Y^{(1-d_{i_1+1})}_{i_1,\xi(i_1)} Y(i_1+1,i_k) Y_{i_k,\xi(i_k)-2} Y^{d_{i_k}}_{i_k+1,\xi(i_k+1)}) & \text{if $\xi(i_1+1)>\xi(i_1+2)$ and $\xi(i_k-1)>\xi(i_k)$}, \\
L(Y^{d_{i_1+1}}_{i_1,\xi(i_1)} Y_{i_1+1,\xi(i_1+1)-2} Y(i_1+1,i_k) Y_{i_k,\xi(i_k)-2} Y^{d_{i_k}}_{i_k+1,\xi(i_k+1)})  & \text{if $\xi(i_1+1)=\xi(i_2)<\xi(i_1+2)$ and $\xi(i_k-1)>\xi(i_k)$}.
\end{cases}
\end{align*}
\end{gather}
Note that $i_1+1$ is not source or sink in the first case, $i_1+1$ is sink in the second case, and $i_k$ is not source or sink in both cases. So $\Phi([i_1+1,i_k])=L(Y_{i_1,a_1}Y_{i_2,a_2}\ldots Y_{i_k,a_k})$.

If $a_1>a_2$ and $a_{k-1}<a_k$, then we choose a height function $\xi$ such that $\xi$ is linear in intervals $[1,i_1], [i_1,i_2], \ldots, [i_{k-1},n]$, where $\xi(i_1)=a_{1}>\xi(i_2)=a_{1}-i_2+i_1$. It follows from Proposition \ref{monoidal categorification AC1} that $ \Phi([i_1+1,i_k-1])$ is equal to 

\begin{gather}
\begin{align*}
&\begin{cases}
L(Y^{(1-d_{i_1+1})}_{i_1,\xi(i_1)} Y(i_1+1,i_k-1) Y_{i_k-1,\xi(i_k-1)-2} Y^{d_{i_k-1}}_{i_k,\xi(i_k)}) & \text{if $\xi(i_1+1)>\xi(i_1+2)$ and $\xi(i_k-2)>\xi(i_k-1)=\xi(i_{k-1})$}, \\
L(Y^{(1-d_{i_1+1})}_{i_1,\xi(i_1)} Y(i_1+1,i_k-1) Y^{(1-d_{i_k-1})}_{i_k,\xi(i_k)}) & \text{if $\xi(i_1+1)>\xi(i_1+2)$ and $\xi(i_k-2)<\xi(i_k-1)$}, \\
L(Y^{d_{i_1+1}}_{i_1,\xi(i_1)} Y_{i_1+1,\xi(i_1+1)-2} Y(i_1+1,i_k-1) Y_{i_k-1,\xi(i_k-1)-2} Y^{d_{i_k-1}}_{i_k,\xi(i_k)})  & \text{if $\xi(i_1+1)=\xi(i_2)<\xi(i_1+2)$ and $\xi(i_k-2)>\xi(i_k-1)=\xi(i_{k-1})$}, \\
L(Y^{d_{i_1+1}}_{i_1,\xi(i_1)} Y_{i_1+1,\xi(i_1+1)-2} Y(i_1+1,i_k-1) Y^{(1-d_{i_k-1})}_{i_k,\xi(i_k)})  & \text{if $\xi(i_1+1)=\xi(i_2)<\xi(i_1+2)$ and $\xi(i_k-2)<\xi(i_k-1)$}.
\end{cases}
\end{align*}
\end{gather}
Note that $\xi(i_k-2)<\xi(i_k-1)$ implies that $\xi(i_k-1) \neq \xi(i_{k-1})$.  In each case, we obtain $\Phi([i_1+1,i_k-1])=L(Y_{i_1,a_1}Y_{i_2,a_2}\ldots Y_{i_k,a_k})$.

If $a_1<a_2$ and $a_{k-1}>a_k$, then we choose a height function $\xi$ such that $\xi$ is linear in intervals $[1,i_2], \ldots, [i_{k-1},n]$, where $\xi(i_1)=a_{2}-i_2+i_1<\xi(i_2)=a_{2}$. It follows from Proposition \ref{monoidal categorification AC1} that 
\[
\Phi([i_1,i_k])=L(Y^{d_{i_1}}_{i_1-1,\xi(i_1-1)} Y_{i_1,\xi(i_1)-2} Y(i_1,i_k) Y_{i_k,\xi(i_k)-2} Y^{d_{i_k}}_{i_k+1,\xi(i_k+1)}).
\]
Note that $i_1$ and $i_k$ are not sources or sinks, and $a_{k-1}>a_k$ implies that $\xi(i_k-1)>\xi(i_k)$. So $\Phi([i_1,i_k])=L(Y_{i_1,a_1}Y_{i_2,a_2}\ldots Y_{i_k,a_k})$.  

If $a_1<a_2$ and $a_{k-1}<a_k$, then we choose a height function $\xi$ such that $\xi$ is linear in intervals $[1,i_2], \ldots, [i_{k-1},n]$, where $\xi(i_1)=a_{2}-i_2+i_1<\xi(i_2)=a_{2}$. It follows from Proposition \ref{monoidal categorification AC1} that $\Phi([i_1,i_k-1])$ is equal to 
\begin{gather}
\begin{align*}
\begin{cases}
L(Y^{d_{i_1}}_{i_1-1,\xi(i_1-1)} Y_{i_1,\xi(i_1)-2} Y(i_1,i_k-1) Y_{i_k-1,\xi(i_k-1)-2} Y^{d_{i_k-1}}_{i_k,\xi(i_k)})  & \text{if $\xi(i_1)<\xi(i_1+1)$ and $\xi(i_k-2)>\xi(i_k-1)=\xi(i_{k-1})$}, \\
L(Y^{d_{i_1}}_{i_1-1,\xi(i_1-1)} Y_{i_1,\xi(i_1)-2} Y(i_1,i_k-1) Y^{(1-d_{i_k-1})}_{i_k,\xi(i_k)})  & \text{if $\xi(i_1)<\xi(i_1+1)$ and $\xi(i_k-2)<\xi(i_k-1)$}.
\end{cases}
\end{align*}
\end{gather}
Note that $\xi(i_k-2)<\xi(i_k-1)$ implies that $\xi(i_k-1) \neq \xi(i_{k-1})$. In each case, we obtain $\Phi([i_1,i_k-1])=L(Y_{i_1,a_1}Y_{i_2,a_2}\ldots Y_{i_k,a_k})$. 
\end{proof}

\begin{remark}
An example on Hernandez-Leclerc modules of type $A_n$ appeared in \cite{HL10}. These modules were then named and studied by Brito and Chari \cite{BC19}, using an approach that is different from ours.
\end{remark}

\subsection{Type $D_n$}
For any height function $\xi$, we have $|\xi(n)-\xi(n-1)|=2 \text{ or } 0$. We will use the following labeling of positive roots of type $D_n$ in \cite{Bou02}. Let $V=\mathbb{R}^n$ be the $n$-dimensional real space with canonical basis $\{ \varepsilon_i \mid 1\leq i\leq n\}$. We identify a positive root $\beta\in \Phi^+$ with $\{i,\pm j\}$ if
\begin{align*}
\beta=\begin{cases}
\varepsilon_i - \varepsilon_j=\sum_{i\leq k<j} \alpha_k  &  (1\leq i < j \leq n), \\
\varepsilon_i + \varepsilon_n=\sum_{i\leq k\leq n-2} \alpha_k +\alpha_n &  (1\leq i < n), \\
\varepsilon_i + \varepsilon_j=\sum_{i\leq k<j} \alpha_k + 2\sum_{j\leq k\leq n-2}\alpha_k + \alpha_{n-1}+\alpha_n & (1\leq i < j < n).
\end{cases}
\end{align*}
For simplicity of notation, we write $\alpha_{i,j}$ for $\{i,-(j+1)\}$ if $1\leq i \leq j \leq n-1$, and $\alpha_{i,n}=\{i,n\}$. 
Note that   $\alpha_{i,j}$ and $\alpha_{i,n}$ together describe all the roots that are of type $A$,  whereas $\{i,j\}$ describes all roots that are supported on every vertex of the Dynkin diagram.
It follows from the geometric realization of $\mathcal{C}_Q$ in Section \ref{a geometric realization of cluster category} and \cite{CK06} that the following exchange relations hold.
\begin{align}
 x[\alpha_j] x[-\alpha_j] & =  x[-\alpha_{j-1}] x[-\alpha_{j+1}]^{d_j} + x[-\alpha_{j+1}]^{1-d_j}\quad \text{ for $1\leq j \leq n-3$},   \label{exchange relation (1) Dn} \\
 x[\alpha_{n-2}] x[-\alpha_{n-2}] & = x[-\alpha_{n-3}] x[-\alpha_{n-1}]^{\delta_{\xi(n-3),\xi(n-1)}} x[-\alpha_{n}]^{\delta_{\xi(n-3),\xi(n)}} \nonumber \\ 
& + x[-\alpha_{n-1}]^{1-\delta_{\xi(n-3),\xi(n-1)}} x[-\alpha_{n}]^{1-\delta_{\xi(n-3),\xi(n)}},    \label{exchange relation (2) Dn} \\
x[\alpha_{n-1}] x[-\alpha_{n-1}] &=  x[-\alpha_{n-2}] + 1,    \label{exchange relation (3) Dn} \\
x[\alpha_{n}] x[-\alpha_{n}] &=  x[-\alpha_{n-2}] + 1.    \label{exchange relation (4) Dn}
\end{align}

For $1 \leq i<j \leq n-3$, we have
\begin{align}\label{exchange relation (5) Dn}
x[\alpha_{i,j}] x[-\alpha_j]=x[\alpha_{i,\max\{i-1,j_\bullet-1\}}]^{1-\delta_{i,j_\bullet}} x[-\alpha_{j+1}]^{d_j} + x[\alpha_{i,j-1}] x[-\alpha_{j+1}]^{1-d_j}.
\end{align}

For $1 \leq i<n-2$, $\xi(n-1)=\xi(n)$, we have
\begin{gather}
\begin{aligned}\label{exchange relation (6) Dn}
x[\alpha_{i,n-2}] x[-\alpha_{n-2}]=x[\alpha_{i,\max\{i-1,(n-2)_\bullet-1\}}]^{1-\delta_{i,(n-2)_\bullet}} x[-\alpha_{n-1}]^{d_{n-2}}  x[-\alpha_{n}]^{d_{n-2}} + x[\alpha_{i,n-3}] x[-\alpha_{n-1}]^{1-d_{n-2}} x[-\alpha_{n}]^{1-d_{n-2}}. 
\end{aligned}
\end{gather}

For $1 \leq i<n-2$, $|\xi(n-1)-\xi(n)|=2$,  we have
\begin{gather}
\begin{aligned} \label{exchange relation (7) Dn}
x[\alpha_{i,n-2}] x[-\alpha_{n-2}]=\begin{cases}
x[\alpha_{i,\max\{i-1,(n-2)_\bullet-1\}}]^{1-\delta_{i,(n-2)_\bullet}} x[-\alpha_{n}] + x[\alpha_{i,n-3}] x[-\alpha_{n-1}]  &  \text{if $\xi(n-3)=\xi(n)$}, \\
x[\alpha_{i,\max\{i-1,(n-2)_\bullet-1\}}]^{1-\delta_{i,(n-2)_\bullet}} x[-\alpha_{n-1}] + x[\alpha_{i,n-3}] x[-\alpha_n]  &  \text{if $\xi(n-1)=\xi(n-3)$}.
\end{cases}
\end{aligned}
\end{gather}

For $i\leq n-2$,  we have
\begin{gather}
\begin{aligned}\label{exchange relation (8) Dn}
& x[\alpha_{i,n-1}] x[-\alpha_{n-1}] = \begin{cases}
x[\alpha_{i,n-3}]^{1-\delta_{i,(n-2)}} + x[\alpha_{i,n-2}]  & \text{if $\xi(n)=\xi(n-1)>\xi(n-2)<\xi(n-3)$}, \\
x[\alpha_{i,\max\{i-1,(n-2)_\bullet-1\}}]^{1-\delta_{i,(n-2)_\bullet}} + x[\alpha_{i,n-2}]  & \text{if $\xi(n)=\xi(n-1)>\xi(n-2)>\xi(n-3)$}, \\
x[\alpha_{i,n-3}]^{1-\delta_{i,(n-2)}} + x[\alpha_{i,n-2}]  & \text{if $\xi(n)=\xi(n-1)<\xi(n-2)>\xi(n-3)$}, \\
x[\alpha_{i,\max\{i-1,(n-2)_\bullet-1\}}]^{1-\delta_{i,(n-2)_\bullet}} + x[\alpha_{i,n-2}]  & \text{if $\xi(n)=\xi(n-1)<\xi(n-2)<\xi(n-3)$}, \\
x[\alpha_{i,\max\{i-1,(n-2)_\bullet-1\}}]^{1-\delta_{i,(n-2)_\bullet}} x[-\alpha_{n}]  +  x[\alpha_{i,n-2}]  & \text{if $\xi(n-3)=\xi(n)>\xi(n-2)>\xi(n-1)$}, \\
x[\alpha_{i,n-3}]^{1-\delta_{i,(n-2)}} x[-\alpha_{n}]  +  x[\alpha_{i,n-2}]  & \text{if $\xi(n-3)=\xi(n-1)<\xi(n-2)<\xi(n)$}, \\
x[\alpha_{i,\max\{i-1,(n-2)_\bullet-1\}}]^{1-\delta_{i,(n-2)_\bullet}} x[-\alpha_{n}] + x[\alpha_{i,n-2}] & \text{if $\xi(n-3)=\xi(n)<\xi(n-2)<\xi(n-1)$}, \\
x[\alpha_{i,n-3}]^{1-\delta_{i,(n-2)}} x[-\alpha_{n}] + x[\alpha_{i,n-2}]  & \text{if $\xi(n-3)=\xi(n-1)>\xi(n-2)>\xi(n)$}.
\end{cases} 
\end{aligned}
\end{gather}

\begin{gather}
\begin{aligned}\label{exchange relation (9) Dn}
& x[\alpha_{i,n}] x[-\alpha_n] = \begin{cases}
x[\alpha_{i,n-3}]^{1-\delta_{i,(n-2)}} + x[\alpha_{i,n-2}]  & \text{if $\xi(n)=\xi(n-1)>\xi(n-2)<\xi(n-3)$}, \\
x[\alpha_{i,\max\{i-1,(n-2)_\bullet-1\}}]^{1-\delta_{i,(n-2)_\bullet}} + x[\alpha_{i,n-2}]  & \text{if $\xi(n)=\xi(n-1)>\xi(n-2)>\xi(n-3)$}, \\
x[\alpha_{i,n-3}]^{1-\delta_{i,(n-2)}} + x[\alpha_{i,n-2}]  & \text{if $\xi(n)=\xi(n-1)<\xi(n-2)>\xi(n-3)$}, \\
x[\alpha_{i,\max\{i-1,(n-2)_\bullet-1\}}]^{1-\delta_{i,(n-2)_\bullet}} + x[\alpha_{i,n-2}]  & \text{if $\xi(n)=\xi(n-1)<\xi(n-2)<\xi(n-3)$}, \\
x[\alpha_{i,\max\{i-1,(n-2)_\bullet-1\}}]^{1-\delta_{i,(n-2)_\bullet}} x[-\alpha_{n-1}] +  x[\alpha_{i,n-2}]  & \text{if $\xi(n-3)=\xi(n-1)>\xi(n-2)>\xi(n)$}, \\
x[\alpha_{i,n-3}]^{1-\delta_{i,(n-2)}} x[-\alpha_{n-1}]  +  x[\alpha_{i,n-2}]  & \text{if $\xi(n-3)=\xi(n)<\xi(n-2)<\xi(n-1)$}, \\
x[\alpha_{i,\max\{i-1,(n-2)_\bullet-1\}}]^{1-\delta_{i,(n-2)_\bullet}} x[-\alpha_{n-1}] + x[\alpha_{i,n-2}] & \text{if $\xi(n-3)=\xi(n-1)<\xi(n-2)<\xi(n)$}, \\
x[\alpha_{i,n-3}]^{1-\delta_{i,(n-2)}} x[-\alpha_{n-1}] + x[\alpha_{i,n-2}]  & \text{if $\xi(n-3)=\xi(n)>\xi(n-2)>\xi(n-1)$}.
\end{cases} 
\end{aligned}
\end{gather}

Recall that $\{i,n-1\}$ denotes the positive root $\alpha_{i}+\alpha_{i+1}+\ldots+\alpha_{n-2}+\alpha_{n-1}+\alpha_{n}$. It follows from the geometric  realization of $\mathcal{C}_Q$ (for instance, see Figure \ref{compatible of $[i,n-1]$ and $[i,n-2]$}) that $\{i,n-1\}$ and $\alpha_{i,n-2}$ form an exchange pair if $\xi(n-1)=\xi(n)$, but not so for $|\xi(n)-\xi(n-1)|=2$. For $1\leq i < n-1$, $\xi(n-1)=\xi(n)$,  we have
\begin{gather}
\begin{aligned}\label{exchange relation (10) Dn}
x[\{i,n-1\}] x[\alpha_{i,n-2}] =x[\alpha_{i,n-1}] x[\alpha_{i,n}]+x[\alpha_{i,\max\{i-1,(n-2)_\bullet-1\}}]^{1-\delta_{i,(n-2)_\bullet}} x[\alpha_{i,n-3}]^{1-\delta_{i,(n-2)}}.
\end{aligned}
\end{gather}

\begin{figure}
\begin{tikzpicture}[node distance={13mm}] 
\draw (22.5:1.75cm) -- (67.5:1.75cm)  (112.5:1.75cm) -- (157.5:1.75cm) -- (-157.5:1.75cm) -- (-112.5:1.75cm) -- (-67.5:1.75cm) -- (-22.5:1.75cm) -- (22.5:1.75cm);
\draw[dashed] (67.5:1.75cm) -- (112.5:1.75cm);
\node[thick] (A) at (0,0) {$\substack{\bullet}$};
\draw (-112.5:1.75cm) to [out=80,in=100,looseness=4.8] (-67.5:1.75cm);
\draw (-112.5:1.75cm) -- (0,0);
\draw (-112.5:1.75cm) to [out=45,in=-60,looseness=1] (0,0);
\node at (-90:0.8cm) {$\substack{\bowtie}$};
\draw (-157.5:1.75cm) to [out=45,in=90,looseness=2.4] (-67.5:1.75cm);
\draw (-157.5:1.75cm) to [out=55,in=120,looseness=1.5] (-22.5:1.75cm);
\draw (157.5:1.75cm) to [out=15,in=110,looseness=1.2] (-22.5:1.75cm);
\draw (157.5:1.75cm) to [out=30,in=150,looseness=1.2] (22.5:1.75cm);
\draw (157.5:1.75cm) to [out=40,in=-170,looseness=1] (67.5:1.75cm);
\draw (67.5:1.75cm) to [out=210,in=160,looseness=1.3] (-67.5:1.75cm);
\draw[red] (67.5:1.75cm) to [out=-60,in=15,looseness=1] (-112.5:1.75cm);
\end{tikzpicture} 
\begin{tikzpicture}[node distance={13mm}] 
\draw (22.5:1.75cm) -- (67.5:1.75cm)  (112.5:1.75cm) -- (157.5:1.75cm) -- (-157.5:1.75cm) -- (-112.5:1.75cm) -- (-67.5:1.75cm) -- (-22.5:1.75cm) -- (22.5:1.75cm);
\draw[dashed] (67.5:1.75cm) -- (112.5:1.75cm);
\node[thick] (A) at (0,0) {$\substack{\bullet}$};
\draw (-112.5:1.75cm) to [out=80,in=100,looseness=4.8] (-67.5:1.75cm);
\draw (-67.5:1.75cm) -- (0,0);
\draw (-67.5:1.75cm) to [out=135,in=-120,looseness=1] (0,0);
\node at (-90:0.8cm) {$\substack{\bowtie}$};
\draw (-157.5:1.75cm) to [out=45,in=90,looseness=2.4] (-67.5:1.75cm);
\draw (-157.5:1.75cm) to [out=55,in=120,looseness=1.5] (-22.5:1.75cm);
\draw (157.5:1.75cm) to [out=15,in=110,looseness=1.2] (-22.5:1.75cm);
\draw (157.5:1.75cm) to [out=30,in=150,looseness=1.2] (22.5:1.75cm);
\draw (157.5:1.75cm) to [out=40,in=-170,looseness=1] (67.5:1.75cm);
\draw (67.5:1.75cm) to [out=-60,in=15,looseness=1] (-112.5:1.75cm);
\draw[red] (67.5:1.75cm) to [out=210,in=160,looseness=1.3] (-67.5:1.75cm);
\end{tikzpicture}
\begin{tikzpicture}[node distance={13mm}] 
\draw (22.5:1.75cm) -- (67.5:1.75cm)  (112.5:1.75cm) -- (157.5:1.75cm) -- (-157.5:1.75cm) -- (-112.5:1.75cm) -- (-67.5:1.75cm) -- (-22.5:1.75cm) -- (22.5:1.75cm);
\draw[dashed] (67.5:1.75cm) -- (112.5:1.75cm);
\node[thick] (A) at (0,0) {$\substack{\bullet}$};
\draw (-67.5:1.75cm) -- (0,0)  (-112.5:1.75cm) -- (0,0);
\draw (-112.5:1.75cm) to [out=80,in=100,looseness=4.8] (-67.5:1.75cm);
\draw (-157.5:1.75cm) to [out=45,in=90,looseness=2.4] (-67.5:1.75cm);
\draw (-157.5:1.75cm) to [out=55,in=120,looseness=1.5] (-22.5:1.75cm);
\draw (157.5:1.75cm) to [out=15,in=110,looseness=1.2] (-22.5:1.75cm);
\draw (157.5:1.75cm) to [out=30,in=150,looseness=1.2] (22.5:1.75cm);
\draw (157.5:1.75cm) to [out=40,in=-170,looseness=1] (67.5:1.75cm);
\node at (-79:0.8cm) {$\substack{n}$};
\draw[red] (67.5:1.75cm) to [out=-150,in=120,looseness=1] (0,0);
\draw (67.5:1.75cm) to [out=-70,in=25,looseness=1] (0,0);
\node at (50:1.1cm) {$\substack{\bowtie}$};
\end{tikzpicture}
\begin{tikzpicture}[node distance={13mm}] 
\draw (22.5:1.75cm) -- (67.5:1.75cm)  (112.5:1.75cm) -- (157.5:1.75cm) -- (-157.5:1.75cm) -- (-112.5:1.75cm) -- (-67.5:1.75cm) -- (-22.5:1.75cm) -- (22.5:1.75cm);
\draw[dashed] (67.5:1.75cm) -- (112.5:1.75cm);
\node[thick] (A) at (0,0) {$\substack{\bullet}$};
\draw (-67.5:1.75cm) -- (0,0)  (-112.5:1.75cm) -- (0,0);
\draw (-112.5:1.75cm) to [out=80,in=100,looseness=4.8] (-67.5:1.75cm);
\draw (-157.5:1.75cm) to [out=45,in=90,looseness=2.4] (-67.5:1.75cm);
\draw (-157.5:1.75cm) to [out=55,in=120,looseness=1.5] (-22.5:1.75cm);
\draw (157.5:1.75cm) to [out=15,in=110,looseness=1.2] (-22.5:1.75cm);
\draw (157.5:1.75cm) to [out=30,in=150,looseness=1.2] (22.5:1.75cm);
\draw (157.5:1.75cm) to [out=40,in=-170,looseness=1] (67.5:1.75cm);
\node at (-100:0.8cm) {$\substack{n}$};
\draw[red] (67.5:1.75cm) to [out=-150,in=120,looseness=1] (0,0);
\draw (67.5:1.75cm) to [out=-70,in=25,looseness=1] (0,0);
\node at (50:1.1cm) {$\substack{\bowtie}$};
\end{tikzpicture}
\caption{The number of crossings of $\{i,n-1\}$ and $\alpha_{i,n-2}$. Here we choose a triangulation $T$ of $P^{\bullet}_n$, the black arc not in $T$ denotes $\{i,n-1\}$, and the red arc denotes $\alpha_{i,n-2}$.} \label{compatible of $[i,n-1]$ and $[i,n-2]$}
\end{figure}
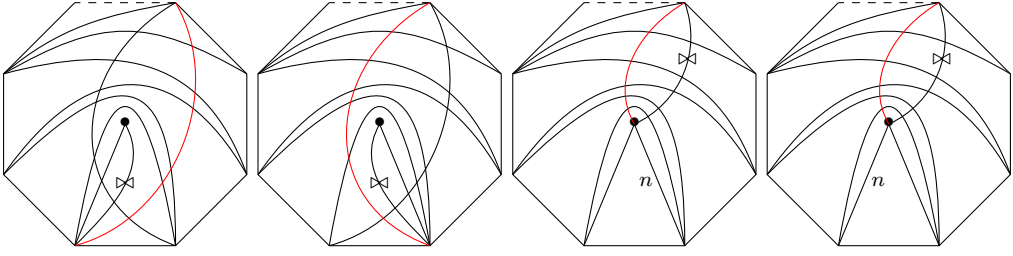

For $1\leq i \leq n-2$, $|\xi(n)-\xi(n-1)|=2$, 
\begin{gather}
\begin{aligned} \label{exchange relation (11) Dn}
& x[\{i,n-1\}] x[-\alpha_{n-1}] = x[\alpha_{i,n}] + \begin{cases} 
x[\alpha_{i,n-3}]^{1-\delta_{i,(n-2)}} & \text{if $\xi(n-1)=\xi(n-3)$},  \\
x[\alpha_{i,\max\{i-1,(n-2)_\bullet-1\}}]^{1-\delta_{i,(n-2)_\bullet}} & \text{if $\xi(n)=\xi(n-3)$}, 
\end{cases} \\
& x[\{i,n-1\}] x[-\alpha_{n}] =  x[\alpha_{i,n-1}] + \begin{cases} 
x[\alpha_{i,\max\{i-1,(n-2)_\bullet-1\}}]^{1-\delta_{i,(n-2)_\bullet}} & \text{if $\xi(n-1)=\xi(n-3)$},  \\
x[\alpha_{i,n-3}]^{1-\delta_{i,(n-2)}} & \text{if $\xi(n)=\xi(n-3)$}.
\end{cases}
\end{aligned}
\end{gather}

Next we deal with the image $\Phi(\{i,j\})$ for $1\leq i < j \leq n-2$. 
Recall that $\{i,j\}$ is the positive root $\alpha_{i}+\alpha_{i+1}+\ldots+\alpha_{j-1}+2\alpha_j+\ldots+2\alpha_{n-2}+\alpha_{n-1}+\alpha_{n}$. It follows from the geometric realization of $\mathcal{C}_Q$ (for instance, see Figure \ref{compatible of $[i,j]$ and $[i,-j]$}) that $\{i,j\}$ and $\{i,-j\}$ form an exchange pair. For $\xi(n-1)=\xi(n)$ and $1\leq i < j \leq n-2$, 
\begin{gather}
\begin{aligned} \label{exchange relation (12) Dn}
x[\{i,j\}] x[\{i,-j\}]= x[\{i,-n\}] x[\{i,n\}] + \text{other terms}.
\end{aligned}
\end{gather}
For $|\xi(n-1)-\xi(n)|=2$ and $1\leq i < j \leq n-2$,
\begin{gather}
\begin{aligned}  \label{exchange relation (13) Dn}
x[\{i,j\}] x[\{i,-j\}]=x[\{i,n-1\}] x[\{i,-(n-1)\}] + \text{other terms}.
\end{aligned}
\end{gather}

\begin{figure}
\begin{tikzpicture}[node distance={13mm}] 
\draw (22.5:1.75cm) -- (67.5:1.75cm)  (112.5:1.75cm) -- (157.5:1.75cm) -- (-157.5:1.75cm) -- (-112.5:1.75cm) -- (-67.5:1.75cm) -- (-22.5:1.75cm) -- (22.5:1.75cm);
\draw[dashed] (67.5:1.75cm) -- (112.5:1.75cm);
\node[thick] (A) at (0,0) {$\substack{\bullet}$};
\draw (-112.5:1.75cm) to [out=80,in=100,looseness=4.8] (-67.5:1.75cm);
\draw (-112.5:1.75cm) -- (0,0);
\draw (-112.5:1.75cm) to [out=45,in=-60,looseness=1] (0,0);
\node at (-90:0.8cm) {$\substack{\bowtie}$};
\draw (-157.5:1.75cm) to [out=45,in=90,looseness=2.4] (-67.5:1.75cm);
\draw (-157.5:1.75cm) to [out=55,in=120,looseness=1.5] (-22.5:1.75cm);
\draw (157.5:1.75cm) to [out=15,in=110,looseness=1.2] (-22.5:1.75cm);
\draw (157.5:1.75cm) to [out=30,in=150,looseness=1.2] (22.5:1.75cm);
\draw (157.5:1.75cm) to [out=40,in=-170,looseness=1] (67.5:1.75cm);
\draw (67.5:1.75cm) to [out=225,in=180,looseness=2.5] (-22.5:1.75cm);
\draw[red] (67.5:1.75cm) to [out=-60,in=60,looseness=1] (-67.5:1.75cm);
\end{tikzpicture} 
\begin{tikzpicture}[node distance={13mm}] 
\draw (22.5:1.75cm) -- (67.5:1.75cm)  (112.5:1.75cm) -- (157.5:1.75cm) -- (-157.5:1.75cm) -- (-112.5:1.75cm) -- (-67.5:1.75cm) -- (-22.5:1.75cm) -- (22.5:1.75cm);
\draw[dashed] (67.5:1.75cm) -- (112.5:1.75cm);
\node[thick] (A) at (0,0) {$\substack{\bullet}$};
\draw (-112.5:1.75cm) to [out=80,in=100,looseness=4.8] (-67.5:1.75cm);
\draw (-67.5:1.75cm) -- (0,0);
\draw (-67.5:1.75cm) to [out=135,in=-120,looseness=1] (0,0);
\node at (-90:0.8cm) {$\substack{\bowtie}$};
\draw (-157.5:1.75cm) to [out=45,in=90,looseness=2.4] (-67.5:1.75cm);
\draw (-157.5:1.75cm) to [out=55,in=120,looseness=1.5] (-22.5:1.75cm);
\draw (157.5:1.75cm) to [out=15,in=110,looseness=1.2] (-22.5:1.75cm);
\draw (157.5:1.75cm) to [out=30,in=150,looseness=1.2] (22.5:1.75cm);
\draw (157.5:1.75cm) to [out=40,in=-170,looseness=1] (67.5:1.75cm);
\draw (67.5:1.75cm) to [out=225,in=180,looseness=2.5] (-22.5:1.75cm);
\draw[red] (67.5:1.75cm) to [out=-60,in=60,looseness=1] (-67.5:1.75cm);
\end{tikzpicture} 
\begin{tikzpicture}[node distance={13mm}] 
\draw (22.5:1.75cm) -- (67.5:1.75cm)  (112.5:1.75cm) -- (157.5:1.75cm) -- (-157.5:1.75cm) -- (-112.5:1.75cm) -- (-67.5:1.75cm) -- (-22.5:1.75cm) -- (22.5:1.75cm);
\draw[dashed] (67.5:1.75cm) -- (112.5:1.75cm);
\node[thick] (A) at (0,0) {$\substack{\bullet}$};
\draw (-67.5:1.75cm) -- (0,0)  (-112.5:1.75cm) -- (0,0);
\draw (-112.5:1.75cm) to [out=80,in=100,looseness=4.8] (-67.5:1.75cm);
\draw (-157.5:1.75cm) to [out=45,in=90,looseness=2.4] (-67.5:1.75cm);
\draw (-157.5:1.75cm) to [out=55,in=120,looseness=1.5] (-22.5:1.75cm);
\draw (157.5:1.75cm) to [out=15,in=110,looseness=1.2] (-22.5:1.75cm);
\draw (157.5:1.75cm) to [out=30,in=150,looseness=1.2] (22.5:1.75cm);
\draw (157.5:1.75cm) to [out=40,in=-170,looseness=1] (67.5:1.75cm);
\node at (-79:0.8cm) {$\substack{n}$};
\draw (67.5:1.75cm) to [out=225,in=180,looseness=2.5] (-22.5:1.75cm);
\draw[red] (67.5:1.75cm) to [out=-60,in=60,looseness=1] (-67.5:1.75cm);
\end{tikzpicture}
\begin{tikzpicture}[node distance={13mm}] 
\draw (22.5:1.75cm) -- (67.5:1.75cm)  (112.5:1.75cm) -- (157.5:1.75cm) -- (-157.5:1.75cm) -- (-112.5:1.75cm) -- (-67.5:1.75cm) -- (-22.5:1.75cm) -- (22.5:1.75cm);
\draw[dashed] (67.5:1.75cm) -- (112.5:1.75cm);
\node[thick] (A) at (0,0) {$\substack{\bullet}$};
\draw (-67.5:1.75cm) -- (0,0)  (-112.5:1.75cm) -- (0,0);
\draw (-112.5:1.75cm) to [out=80,in=100,looseness=4.8] (-67.5:1.75cm);
\draw (-157.5:1.75cm) to [out=45,in=90,looseness=2.4] (-67.5:1.75cm);
\draw (-157.5:1.75cm) to [out=55,in=120,looseness=1.5] (-22.5:1.75cm);
\draw (157.5:1.75cm) to [out=15,in=110,looseness=1.2] (-22.5:1.75cm);
\draw (157.5:1.75cm) to [out=30,in=150,looseness=1.2] (22.5:1.75cm);
\draw (157.5:1.75cm) to [out=40,in=-170,looseness=1] (67.5:1.75cm);
\node at (-100:0.8cm) {$\substack{n}$};
\draw (67.5:1.75cm) to [out=225,in=180,looseness=2.5] (-22.5:1.75cm);
\draw[red] (67.5:1.75cm) to [out=-60,in=60,looseness=1] (-67.5:1.75cm);
\end{tikzpicture}
\caption{The number of crossings of $\{i,n-3\}$ and $\{i,-(n-3)\}$ is equal to 1. Here we choose a triangulation $T$ of $P^{\bullet}_n$, the black arc not in $T$ denotes $\{i,n-3\}$, and the red arc denotes $\{i,-(n-3)\}$.} \label{compatible of $[i,j]$ and $[i,-j]$}
\end{figure}
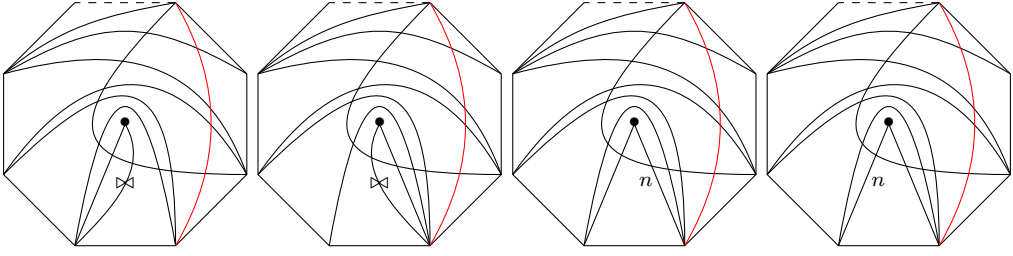

By Equation (\ref{Gamma1 equation2}), the formulas (\ref{exchange relation (1) Dn})--(\ref{exchange relation (4) Dn}) can be lifted to equations in $K_0(\mathscr{C}^{\leq \xi}_1)$. In particular, we have

\begin{gather}
\begin{align*}
& \Phi(\{j,-(j+1)\})=\begin{cases}
L(Y_{j-1,\xi(j-1)} Y_{j,\xi(j)-2} Y^{d_j}_{j+1,\xi(j+1)})  & \text{if $1\leq j \leq n-3$, $\xi(j-1)>\xi(j)$,} \\
L(Y_{j,\xi(j)-2} Y^{1-d_j}_{j+1,\xi(j+1)})  & \text{if $1\leq j \leq n-3$, $\xi(j-1)<\xi(j)$,} 
\end{cases}\\
& \Phi(\{n-2,-(n-1)\}) = L(Y^{\delta_{\xi(n-3),\xi(n-2)+1}}_{n-3,\xi(n-3)}Y_{n-2,\xi(n-2)-2}Y^{\delta_{\xi(n-1),\xi(n-2)+1}}_{n-1,\xi(n-1)}Y^{\delta_{\xi(n),\xi(n-2)+1}}_{n,\xi(n)}),  \\
& \Phi(\{n-1,-n\})=L(Y^{\delta_{\xi(n-2),\xi(n-1)+1}}_{n-2,\xi(n-2)}Y_{n-1,\xi(n-1)-2}), \\
& \Phi(\{n-1,n\})=L(Y^{\delta_{\xi(n-2),\xi(n)+1}}_{n-2,\xi(n-2)}Y_{n,\xi(n)-2}).
\end{align*}
\end{gather}

The root $\{i,-(j+1)\}$ for $1 \leq i<j \leq n-3$ is in fact a type $A$ positive root, and Equation (\ref{exchange relation (5) Dn}) is the same as Equation (\ref{exchange relation (2) in An}). Applying Proposition \ref{monoidal categorification AC1},  we have 
\begin{gather}
\begin{align*}
\Phi(\{i,-(j+1)\})=\begin{cases}
L(Y^{(1-d_i)}_{i-1,\xi(i-1)} Y(i,j) Y_{j,\xi(j)-2} Y^{d_j}_{j+1,\xi(j+1)}) & \text{if $\xi(i)>\xi(i+1)$ and $\xi(j-1)>\xi(j)$}, \\
L(Y^{(1-d_i)}_{i-1,\xi(i-1)} Y(i,j) Y^{(1-d_j)}_{j+1,\xi(j+1)}) & \text{if $\xi(i)>\xi(i+1)$ and $\xi(j-1)<\xi(j)$}, \\
L(Y^{d_i}_{i-1,\xi(i-1)} Y_{i,\xi(i)-2} Y(i,j) Y_{j,\xi(j)-2} Y^{d_j}_{j+1,\xi(j+1)})  & \text{if $\xi(i)<\xi(i+1)$ and $\xi(j-1)>\xi(j)$}, \\
L(Y^{d_i}_{i-1,\xi(i-1)} Y_{i,\xi(i)-2} Y(i,j) Y^{(1-d_j)}_{j+1,\xi(j+1)})  & \text{if $\xi(i)<\xi(i+1)$ and $\xi(j-1)<\xi(j)$}.
\end{cases}
\end{align*}
\end{gather} 

If $\xi(n-1)=\xi(n)$, we set $\overline{n-1}=\{n-1,n\}$, $Y_{\overline{n-1},\xi(\overline{n-1})}=Y_{n-1,\xi(n-1)}Y_{n,\xi(n)}$, and $d'_{n-2}=d_{n-2}$. If $|\xi(n)-\xi(n-1)|=2$, we set $\overline{n-1}=n-1$ if $\xi(n-1)>\xi(n-2)$, $\overline{n-1}=n$ if $\xi(n)>\xi(n-2)$, and $d'_{n-2}=1$. In this setting, Equations (\ref{exchange relation (6) Dn}) and (\ref{exchange relation (7) Dn}) are the same as Equation (\ref{exchange relation (2) in An}). By Proposition \ref{monoidal categorification AC1} again, for $1 \leq i \leq n-3$, 
\begin{gather}
\begin{align*}
& \Phi(\{i,-(n-1)\})=\begin{cases}
L(Y^{(1-d_i)}_{i-1,\xi(i-1)} Y(i,n-2) Y_{n-2,\xi(n-2)-2} Y^{d'_{n-2}}_{\overline{n-1},\xi(\overline{n-1})}) & \text{if $\xi(i)>\xi(i+1)$ and $\xi(n-3)>\xi(n-2)$}, \\
L(Y^{(1-d_i)}_{i-1,\xi(i-1)} Y(i,n-2) Y^{(1-d_{n-2})}_{\overline{n-1},\xi(\overline{n-1})}) & \text{if $\xi(i)>\xi(i+1)$ and $\xi(n-3)<\xi(n-2)$}, \\
L(Y^{d_i}_{i-1,\xi(i-1)} Y_{i,\xi(i)-2} Y(i,n-2) Y_{n-2,\xi(n-2)-2} Y^{d'_{n-2}}_{\overline{n-1},\xi(\overline{n-1})}) & \text{if $\xi(i)<\xi(i+1)$ and $\xi(n-3)>\xi(n-2)$}, \\
L(Y^{d_i}_{i-1,\xi(i-1)} Y_{i,\xi(i)-2} Y(i,n-2) Y^{(1-d_{n-2})}_{\overline{n-1},\xi(\overline{n-1})})  & \text{if $\xi(i)<\xi(i+1)$ and $\xi(n-3)<\xi(n-2)$}.
\end{cases}
\end{align*}
\end{gather}
Substitute $Y_{\overline{n-1},\xi(\overline{n-1})}$ and $d'_{n-2}$ with the relevant definitions, 
\begin{gather*}
\begin{aligned}
& \Phi(\{i,-(n-1)\})=\begin{cases}
L(Y^{(1-d_i)}_{i-1,\xi(i-1)} Y(i,n-2) Y^{1-d_{n-2}}_{n-2,\xi(n-2)-2})  & \text{if $\xi(i)>\xi(i+1)$ and $\xi(n-2)>\xi(n-1)=\xi(n)$}, \\
L(Y^{d_i}_{i-1,\xi(i-1)} Y_{i,\xi(i)-2} Y(i,n-2) Y^{1-d_{n-2}}_{n-2,\xi(n-2)-2})  & \text{if $\xi(i)<\xi(i+1)$ and $\xi(n-2)>\xi(n-1)=\xi(n)$}, \\
L(Y^{(1-d_i)}_{i-1,\xi(i-1)} Y(i,n) Y_{n,\xi(n)}) & \text{if $\xi(i)>\xi(i+1)$ and $\xi(n-2)<\xi(n-1)=\xi(n)$}, \\
L(Y^{d_i}_{i-1,\xi(i-1)} Y_{i,\xi(i)-2} Y(i,n) Y_{n,\xi(n)})  & \text{if $\xi(i)<\xi(i+1)$ and $\xi(n-2)<\xi(n-1)=\xi(n)$}, \\
L(Y^{(1-d_i)}_{i-1,\xi(i-1)} Y(i,n-2) Y_{n-2,\xi(n-2)-2} Y_{n,\xi(n)}) & \text{if $\xi(i)>\xi(i+1)$ and $\xi(n-3)=\xi(n)>\xi(n-2)>\xi(n-1)$},  \\
L(Y^{(1-d_i)}_{i-1,\xi(i-1)} Y(i,n-1) Y_{n,\xi(n)})  & \text{if $\xi(i)>\xi(i+1)$ and  $\xi(n)>\xi(n-2)>\xi(n-1)=\xi(n-3)$},  \\
L(Y^{d_i}_{i-1,\xi(i-1)} Y_{i,\xi(i)-2} Y(i,n-2) Y_{n-2,\xi(n-2)-2} Y_{n,\xi(n)})  & \text{if $\xi(i)<\xi(i+1)$ and $\xi(n-3)=\xi(n)>\xi(n-2)>\xi(n-1)$},  \\
L(Y^{d_i}_{i-1,\xi(i-1)} Y_{i,\xi(i)-2} Y(i,n-1) Y_{n,\xi(n)})  & \text{if $\xi(i)<\xi(i+1)$ and  $\xi(n)>\xi(n-2)>\xi(n-1)=\xi(n-3)$},  \\
L(Y^{(1-d_i)}_{i-1,\xi(i-1)} Y(i,n-2) Y_{n-2,\xi(n-2)-2} Y_{n-1,\xi(n-1)}) & \text{if $\xi(i)>\xi(i+1)$ and $\xi(n-3)=\xi(n-1)>\xi(n-2)>\xi(n)$}, \\
L(Y^{(1-d_i)}_{i-1,\xi(i-1)} Y(i,n-2) Y_{n-1,\xi(n-1)})  & \text{if $\xi(i)>\xi(i+1)$ and $\xi(n-1)>\xi(n-2)>\xi(n-3)=\xi(n)$},  \\
L(Y^{d_i}_{i-1,\xi(i-1)} Y_{i,\xi(i)-2} Y(i,n-2) Y_{n-2,\xi(n-2)-2} Y_{n-1,\xi(n-1)})  & \text{if $\xi(i)<\xi(i+1)$ and $\xi(n-3)=\xi(n-1)>\xi(n-2)>\xi(n)$},  \\
[L(Y^{d_i}_{i-1,\xi(i-1)} Y_{i,\xi(i)-2} Y(i,n-2) Y_{n-1,\xi(n-1)})  & \text{if $\xi(i)<\xi(i+1)$ and $\xi(n-1)>\xi(n-2)>\xi(n-3)=\xi(n)$}. 
\end{cases}
\end{aligned}
\end{gather*}

By Equation (\ref{Gamma1 equation2}), Equation (\ref{exchange relation (8) Dn}) can be lifted into the following form: if $i\leq n-2$, $\xi(n-2)>\xi(n-1)$, 
then $
\Phi(\{i,-n\}) [L(Y_{n-1,\xi(n-1)})] $ is equal to 
\begin{gather}
\begin{align*}\begin{cases}
\Phi(\{i,-(n-2)\}) [L(Y_{n,\xi(n)})]^{\delta_{\xi(n-1),\xi(n)-2}} f_{n-1}+\text{other terms} & \text{if $\xi(n-3)=\xi(n-1)$, $i\leq n-2$}, \\
 [L(Y^{(1-d_i)}_{i-1,\xi(i-1)})]  [L(Y^{\delta_{\xi(n-1),\xi(n)-2}}_{n,\xi(n)})] f_{n-1}+\text{other terms} & \text{if $\xi(n-3)\neq\xi(n-1)$, $(n-1)_\bullet\leq i\leq n-2$}, \\
\Phi(\{i,-(n-1)_\bullet\}) [L(Y^{\delta_{\xi(n-1),\xi(n)-2}}_{n,\xi(n)})] f_{n-1}+\text{other terms} & \text{if $\xi(n-3)\neq\xi(n-1)$, $i<(n-1)_\bullet$}, 
\end{cases}
\end{align*}
\end{gather}
and if $i\leq n-2$,  $\xi(n-2)<\xi(n-1)$, then
\begin{align*}
\Phi(\{i,-n\}) [L(Y_{n-1,\xi(n-1)})] = \Phi(\{i,-(n-1)\}) [L(Y^{\delta_{\xi(n-1),\xi(n)-2}}_{n,\xi(n)})] +\text{other terms}.
\end{align*}
The terms in the right hand side of the above formulas have been defined earlier.  By Equation (\ref{Gamma1 equation2}), for $1\leq i \leq n-2$, we obtain that 
$\Phi(\{i,-n\})$ is equal to 
\begin{gather*}
\begin{cases}
L(Y^{(1-d_i)}_{i-1,\xi(i-1)} Y(i,n-1) Y_{n-1,\xi(n-1)-2}) & \text{if $\xi(i)>\xi(i+1)$ and $\xi(n-2)>\xi(n-1)=\xi(n)$}, \\
L(Y^{d_i}_{i-1,\xi(i-1)} Y_{i,\xi(i)-2} Y(i,n-1) Y_{n-1,\xi(n-1)-2})  & \text{if $\xi(i)<\xi(i+1)$ and $\xi(n-2)>\xi(n-1)=\xi(n)$}, \\
L(Y^{(1-d_i)}_{i-1,\xi(i-1)} Y(i,n-1) Y_{n,\xi(n)}) & \text{if $\xi(i)>\xi(i+1)$ and $\xi(n-2)<\xi(n-1)=\xi(n)$}, \\
L(Y^{d_i}_{i-1,\xi(i-1)} Y_{i,\xi(i)-2} Y(i,n-1) Y_{n,\xi(n)})  & \text{if $\xi(i)<\xi(i+1)$ and $\xi(n-2)<\xi(n-1)=\xi(n)$},  \\
L(Y^{(1-d_i)}_{i-1,\xi(i-1)} Y(i,n-1) Y_{n-1,\xi(n-1)-2} Y_{n,\xi(n)}) & \text{if $\xi(i)>\xi(i+1)$ and $\xi(n-3)=\xi(n)>\xi(n-2)>\xi(n-1)$},  \\
L(Y^{(1-d_i)}_{i-1,\xi(i-1)} Y(i,n-2) Y_{n-2,\xi(n-2)} Y_{n-1,\xi(n-1)-2} Y_{n,\xi(n)})  & \text{if $\xi(i)>\xi(i+1)$ and  $\xi(n)>\xi(n-2)>\xi(n-1)=\xi(n-3)$},  \\
L(Y^{d_i}_{i-1,\xi(i-1)} Y_{i,\xi(i)-2} Y(i,n-1) Y_{n-1,\xi(n-1)-2} Y_{n,\xi(n)})  & \text{if $\xi(i)<\xi(i+1)$ and $\xi(n-3)=\xi(n)>\xi(n-2)>\xi(n-1)$}, \\
L(Y^{d_i}_{i-1,\xi(i-1)} Y_{i,\xi(i)-2} Y(i,n-2) Y_{n-2,\xi(n-2)} Y_{n-1,\xi(n-1)-2} Y_{n,\xi(n)})  & \text{if $\xi(i)<\xi(i+1)$ and  $\xi(n)>\xi(n-2)>\xi(n-1)=\xi(n-3)$}, \\
L(Y^{(1-d_i)}_{i-1,\xi(i-1)} Y(i,n-2) Y_{n-2,\xi(n-2)-2}) & \text{if $\xi(i)>\xi(i+1)$ and $\xi(n-3)=\xi(n-1)>\xi(n-2)>\xi(n)$}, \\
L(Y^{(1-d_i)}_{i-1,\xi(i-1)} Y(i,n-2))  & \text{if $\xi(i)>\xi(i+1)$ and $\xi(n-1)>\xi(n-2)>\xi(n-3)=\xi(n)$}, \\
L(Y^{d_i}_{i-1,\xi(i-1)} Y_{i,\xi(i)-2} Y(i,n-2) Y_{n-2,\xi(n-2)-2})  & \text{if $\xi(i)<\xi(i+1)$ and $\xi(n-3)=\xi(n-1)>\xi(n-2)>\xi(n)$}, \\
L(Y^{d_i}_{i-1,\xi(i-1)} Y_{i,\xi(i)-2} Y(i,n-2))   & \text{if $\xi(i)<\xi(i+1)$ and $\xi(n-1)>\xi(n-2)>\xi(n-3)=\xi(n)$}.
\end{cases} 
\end{gather*} 

Equations (\ref{exchange relation (8) Dn}) and (\ref{exchange relation (9) Dn}) are obtained from each other by exchanging $\xi(n)$ with $\xi(n-1)$, and $x[-\alpha_{n}]$ with $x[-\alpha_{n-1}]$. As a result, by exchanging $\xi(n)$ with $\xi(n-1)$, and $Y_{n-1,a}$ with $Y_{n,a}$ for some $a\in \mathbb{C}^*$, we obtain for $1\leq i \leq n-2$, 
that $\Phi(\{i,n\})$ is equal to 
\begin{gather*}
\begin{cases}
L(Y^{(1-d_i)}_{i-1,\xi(i-1)} Y(i,n-1) Y_{n,\xi(n)-2}) & \text{if $\xi(i)>\xi(i+1)$ and $\xi(n-2)>\xi(n-1)=\xi(n)$}, \\
L(Y^{d_i}_{i-1,\xi(i-1)} Y_{i,\xi(i)-2} Y(i,n-1) Y_{n,\xi(n)-2})  & \text{if $\xi(i)<\xi(i+1)$ and $\xi(n-2)>\xi(n-1)=\xi(n)$}, \\
L(Y^{(1-d_i)}_{i-1,\xi(i-1)} Y(i,n-1) Y_{n-1,\xi(n-1)}) & \text{if $\xi(i)>\xi(i+1)$ and $\xi(n-2)<\xi(n-1)=\xi(n)$}, \\
L(Y^{d_i}_{i-1,\xi(i-1)} Y_{i,\xi(i)-2} Y(i,n-1) Y_{n-1,\xi(n-1)})  & \text{if $\xi(i)<\xi(i+1)$ and $\xi(n-2)<\xi(n-1)=\xi(n)$},  \\
L(Y^{(1-d_i)}_{i-1,\xi(i-1)} Y(i,n-2) Y_{n-2,\xi(n-2)-2}) & \text{if $\xi(i)>\xi(i+1)$ and $\xi(n-3)=\xi(n)>\xi(n-2)>\xi(n-1)$}, \\
L(Y^{(1-d_i)}_{i-1,\xi(i-1)} Y(i,n-2))  & \text{if $\xi(i)>\xi(i+1)$ and  $\xi(n)>\xi(n-2)>\xi(n-1)=\xi(n-3)$}, \\
L(Y^{d_i}_{i-1,\xi(i-1)} Y_{i,\xi(i)-2} Y(i,n-2) Y_{n-2,\xi(n-2)-2})  & \text{if $\xi(i)<\xi(i+1)$ and $\xi(n-3)=\xi(n)>\xi(n-2)>\xi(n-1)$}, \\
L(Y^{d_i}_{i-1,\xi(i-1)} Y_{i,\xi(i)-2} Y(i,n-2))  & \text{if $\xi(i)<\xi(i+1)$ and  $\xi(n)>\xi(n-2)>\xi(n-1)=\xi(n-3)$}, \\
L(Y^{(1-d_i)}_{i-1,\xi(i-1)} Y(i,n-2) Y_{n-1,\xi(n-1)} Y_{n,\xi(n)-2}) & \text{if $\xi(i)>\xi(i+1)$ and $\xi(n-3)=\xi(n-1)>\xi(n-2)>\xi(n)$}, \\
L(Y^{(1-d_i)}_{i-1,\xi(i-1)} Y(i,n-2) Y_{n-2,\xi(n-2)} Y_{n-1,\xi(n-1)} Y_{n,\xi(n)-2})  & \text{if $\xi(i)>\xi(i+1)$ and $\xi(n-1)>\xi(n-2)>\xi(n-3)=\xi(n)$}, \\
L(Y^{d_i}_{i-1,\xi(i-1)} Y_{i,\xi(i)-2} Y(i,n-2) Y_{n-1,\xi(n-1)} Y_{n,\xi(n)-2})  & \text{if $\xi(i)<\xi(i+1)$ and $\xi(n-3)=\xi(n-1)>\xi(n-2)>\xi(n)$}, \\
L(Y^{d_i}_{i-1,\xi(i-1)} Y_{i,\xi(i)-2} Y(i,n-2) Y_{n-2,\xi(n-2)} Y_{n-1,\xi(n-1)} Y_{n,\xi(n)-2})  & \text{if $\xi(i)<\xi(i+1)$ and $\xi(n-1)>\xi(n-2)>\xi(n-3)=\xi(n)$}.
\end{cases}
\end{gather*}

For $1\leq i < n-1$, by  Equation (\ref{Gamma1 equation2}), Equations (\ref{exchange relation (10) Dn}) and (\ref{exchange relation (11) Dn}), we have the following formulas:
\begin{gather}
\begin{align*}
& \Phi(\{i,n-1\}) \Phi(\{i,-(n-1)\})=\begin{cases}
\Phi(\{i,-n\}) \Phi(\{i,n\}) f_{n-2} + \text{other terms} & \text{if $\xi(n-3)>\xi(n-2)>\xi(n-1)=\xi(n)$ or } \\
& \text{if $\xi(n-3)<\xi(n-2)>\xi(n-1)=\xi(n)$, $i=n-2$}, \\
\Phi(\{i,-n\}) \Phi(\{i,n\}) + \text{other terms} & \text{if $\xi(n-3)<\xi(n-2)>\xi(n-1)=\xi(n)$, $i\leq n-3$ or} \\
& \text{if $\xi(n-2)<\xi(n-1)=\xi(n)$}.
\end{cases} \\
& \Phi(\{i,n-1\}) [L(Y_{n-1,\xi(n-1)})]= \Phi(\{i,n\}) + \text{other terms} \quad \text{ if $\xi(n-1)-\xi(n)=2$}, \\
& \Phi(\{i,n-1\}) [L(Y_{n,\xi(n)})]= \Phi(\{i,-n\}) + \text{other terms} \quad \text{ if $\xi(n)-\xi(n-1)=2$}.
\end{align*}
\end{gather}
By the previous argument, the images $\Phi(\{i,-n\})$ and $\Phi(\{i,n\})$ have been defined. As a conclusion, we have $\Phi(\{i,n-1\})$ is equal to
\begin{gather}\label{the map Phi(i,n-1)}
\begin{cases}
L(Y^{(1-d_i)}_{i-1,\xi(i-1)} Y(i,n) Y_{n-2,\xi(n-2)} Y_{n,\xi(n)-2})  & \text{if $\xi(i)>\xi(i+1)$ and $\xi(n-2)>\xi(n-1)=\xi(n)$}, \\
L(Y^{d_i}_{i-1,\xi(i-1)} Y_{i,\xi(i)-2} Y(i,n) Y_{n-2,\xi(n-2)} Y_{n,\xi(n)-2})  & \text{if $\xi(i)<\xi(i+1)$ and  $\xi(n-2)>\xi(n-1)=\xi(n)$},  \\
L(Y^{(1-d_i)}_{i-1,\xi(i-1)} Y(i,n-2) Y^{d_{n-2}}_{n-2,\xi(n-2)-2}) & \text{if $\xi(i)>\xi(i+1)$ and $\xi(n-2)<\xi(n-1)=\xi(n)$},  \\
L(Y^{d_i}_{i-1,\xi(i-1)} Y_{i,\xi(i)-2} Y(i,n-2) Y^{d_{n-2}}_{n-2,\xi(n-2)-2})  & \text{if $\xi(i)<\xi(i+1)$ and $\xi(n-2)<\xi(n-1)=\xi(n)$}, \\
L(Y^{(1-d_i)}_{i-1,\xi(i-1)} Y(i,n-2) Y_{n-1,\xi(n-1)-2}) & \text{if $\xi(i)>\xi(i+1)$ and $\xi(n-3)=\xi(n)>\xi(n-2)>\xi(n-1)$}, \\
L(Y^{(1-d_i)}_{i-1,\xi(i-1)} Y(i,n-2) Y_{n-2,\xi(n-2)} Y_{n-1,\xi(n-1)-2})  & \text{if $\xi(i)>\xi(i+1)$ and  $\xi(n)>\xi(n-2)>\xi(n-1)=\xi(n-3)$},  \\
L(Y^{d_i}_{i-1,\xi(i-1)} Y_{i,\xi(i)-2} Y(i,n-2) Y_{n-1,\xi(n-1)-2})  & \text{if $\xi(i)<\xi(i+1)$ and $\xi(n-3)=\xi(n)>\xi(n-2)>\xi(n-1)$}, \\
L(Y^{d_i}_{i-1,\xi(i-1)} Y_{i,\xi(i)-2} Y(i,n-2) Y_{n-2,\xi(n-2)} Y_{n-1,\xi(n-1)-2})  & \text{if $\xi(i)<\xi(i+1)$ and  $\xi(n)>\xi(n-2)>\xi(n-1)=\xi(n-3)$}, \\
L(Y^{(1-d_i)}_{i-1,\xi(i-1)} Y(i,n-2) Y_{n,\xi(n)-2}) & \text{if $\xi(i)>\xi(i+1)$ and $\xi(n-3)=\xi(n-1)>\xi(n-2)>\xi(n)$},  \\
L(Y^{(1-d_i)}_{i-1,\xi(i-1)} Y(i,n-2) Y_{n-2,\xi(n-2)} Y_{n,\xi(n)-2}) & \text{if $\xi(i)>\xi(i+1)$ and $\xi(n-1)>\xi(n-2)>\xi(n-3)=\xi(n)$}, \\
L(Y^{d_i}_{i-1,\xi(i-1)} Y_{i,\xi(i)-2} Y(i,n-2) Y_{n,\xi(n)-2}) & \text{if $\xi(i)<\xi(i+1)$ and $\xi(n-3)=\xi(n-1)>\xi(n-2)>\xi(n)$},  \\
L(Y^{d_i}_{i-1,\xi(i-1)} Y_{i,\xi(i)-2} Y(i,n-2) Y_{n-2,\xi(n-2)} Y_{n,\xi(n)-2}) & \text{if $\xi(i)<\xi(i+1)$ and $\xi(n-1)>\xi(n-2)>\xi(n-3)=\xi(n)$}.
\end{cases}
\end{gather}

By Equation (\ref{Gamma1 equation2}), and Equations (\ref{exchange relation (12) Dn}) and (\ref{exchange relation (13) Dn}), we have, for  $\xi(n-1)=\xi(n)$, $1\leq i < j \leq n-2$, 
\begin{gather}
\begin{align*}
& \Phi(\{i,j\}) \Phi(\{i,-j\})=\begin{cases}
f^{1-d_{j-1}}_{j-1} \Phi(\{i,-n\}) \Phi(\{i,n\}) + \text{other terms} & \text{if $\xi(j-1)>\xi(j)$}, \\
f^{1-d_j}_{j} \Phi(\{i,-n\}) \Phi(\{i,n\}) + \text{other terms} & \text{if $\xi(j-1)<\xi(j)$},
\end{cases}
\end{align*}
\end{gather}
and for $|\xi(n-1)-\xi(n)|=2$, $1\leq i < j \leq n-2$, 
\begin{gather}
\begin{align*}
& \Phi(\{i,j\}) \Phi(\{i,-j\})=\begin{cases}
f^{1-d_{j-1}}_{j-1} \Phi(\{i,-(n-1)\}) \Phi(\{i,n-1\})  + \text{other terms} & \text{if $\xi(j-1)>\xi(j)$}, \\
f^{1-d_j}_{j} \Phi(\{i,-(n-1)\}) \Phi(\{i,n-1\}) + \text{other terms} & \text{if $\xi(j-1)<\xi(j)$}.
\end{cases}
\end{align*}
\end{gather}
The images $\Phi(\{i,-n\})$, $\Phi(\{i,n\})$, $\Phi(\{i,-(n-1)\})$, $\Phi(\{i,n-1\})$, and $\Phi(\{i,-j\})$ have been defined. So we have 
$\Phi(\{i,j\})$ is equal to 
\begin{gather}\label{the map Phi(i,j)}
\begin{cases}
L(Y^{(1-d_i)}_{i-1,\xi(i-1)} Y(i,n] Y[j,n-2] Y_{j-1,\xi(j-1)})  & \text{if $\xi(i)>\xi(i+1)$, $\xi(j-1)>\xi(j)$, and $\xi(n-1)=\xi(n)$}, \\
L(Y^{(1-d_i)}_{i-1,\xi(i-1)} Y(i,n] Y(j,n-2] Y^{1-d_j}_{j,\xi(j)-2})  & \text{if $\xi(i)>\xi(i+1)$, $\xi(j-1)<\xi(j)$, and $\xi(n-1)=\xi(n)$}, \\
L(Y^{d_i}_{i-1,\xi(i-1)} Y_{i,\xi(i)-2} Y(i,n]  Y[j,n-2] Y_{j-1,\xi(j-1)})  & \text{if $\xi(i)<\xi(i+1)$, $\xi(j-1)>\xi(j)$, and $\xi(n-1)=\xi(n)$},  \\
L(Y^{d_i}_{i-1,\xi(i-1)} Y_{i,\xi(i)-2} Y(i,n] Y(j,n-2] Y^{1-d_j}_{j,\xi(j)-2})  & \text{if $\xi(i)<\xi(i+1)$, $\xi(j-1)<\xi(j)$, and $\xi(n-1)=\xi(n)$}, \\
L(Y^{(1-d_i)}_{i-1,\xi(i-1)} Y(i,n] Y_{n-2,\xi(n-2)-2} Y[j,n-2) Y_{j-1,\xi(j-1)})  & \text{if $\xi(i)>\xi(i+1)$, $\xi(j-1)>\xi(j)$, $\xi(n-1) \neq \xi(n)$, and $\xi(n-3)>\xi(n-2)$}, \\
L(Y^{(1-d_i)}_{i-1,\xi(i-1)} Y(i,n] Y_{n-2,\xi(n-2)-2} Y(j,n-2) Y^{1-d_j}_{j,\xi(j)-2})  & \text{if $\xi(i)>\xi(i+1)$, $\xi(j-1)<\xi(j)$, $\xi(n-1)\neq\xi(n)$, and $\xi(n-3)>\xi(n-2)$}, \\
L(Y^{(1-d_i)}_{i-1,\xi(i-1)} Y(i,n] Y_{n-2,\xi(n-2)} Y[j,n-2) Y_{j-1,\xi(j-1)})  & \text{if $\xi(i)>\xi(i+1)$, $\xi(j-1)>\xi(j)$, $\xi(n-1) \neq \xi(n)$, and $\xi(n-3)<\xi(n-2)$}, \\
L(Y^{(1-d_i)}_{i-1,\xi(i-1)} Y(i,n] Y_{n-2,\xi(n-2)} Y(j,n-2) Y^{1-d_j}_{j,\xi(j)-2}) & \text{if $\xi(i)>\xi(i+1)$, $\xi(j-1)<\xi(j)$, $\xi(n-1)\neq\xi(n)$, and $\xi(n-3)<\xi(n-2)$}, \\
L(Y^{d_i}_{i-1,\xi(i-1)} Y_{i,\xi(i)-2} Y(i,n] Y_{n-2,\xi(n-2)-2} Y[j,n-2) Y_{j-1,\xi(j-1)})  & \text{if $\xi(i)<\xi(i+1)$, $\xi(j-1)>\xi(j)$, $\xi(n-1)\neq\xi(n)$, and $\xi(n-3)>\xi(n-2)$},  \\
L(Y^{d_i}_{i-1,\xi(i-1)} Y_{i,\xi(i)-2} Y(i,n] Y_{n-2,\xi(n-2)-2} Y(j,n-2) Y^{1-d_j}_{j,\xi(j)-2})  & \text{if $\xi(i)<\xi(i+1)$, $\xi(j-1)<\xi(j)$, $\xi(n-1)\neq\xi(n)$, and $\xi(n-3)>\xi(n-2)$}, \\
L(Y^{d_i}_{i-1,\xi(i-1)} Y_{i,\xi(i)-2} Y(i,n] Y_{n-2,\xi(n-2)} Y[j,n-2) Y_{j-1,\xi(j-1)})  & \text{if $\xi(i)<\xi(i+1)$, $\xi(j-1)>\xi(j)$, $\xi(n-1)\neq\xi(n)$, and $\xi(n-3)<\xi(n-2)$},  \\
L(Y^{d_i}_{i-1,\xi(i-1)} Y_{i,\xi(i)-2} Y(i,n] Y_{n-2,\xi(n-2)} Y(j,n-2) Y^{1-d_j}_{j,\xi(j)-2})  & \text{if $\xi(i)<\xi(i+1)$, $\xi(j-1)<\xi(j)$, $\xi(n-1)\neq\xi(n)$, and $\xi(n-3)<\xi(n-2)$}.
\end{cases}
\end{gather}
Here $Y_{i,p}=1$ for $i\not\in I$.

Analogous to the case of type $A$, we can use the Auslander-Reiten quivers of type $D$ to compute the Hernandez-Leclerc modules starting with the images of the injective modules $\Phi(I(i))$ and applying Corollary \ref{a corollary of main theorem2} until we reach the images of the projective modules $\Phi(P(i))$. Note however, that in type $D$, the meshes are more complicated than ones of type $A$.

\begin{theorem} \label{monoidal categorification DC1}
An Hernandez-Leclerc module of type $D_n$ is exactly the simple $U_{q}(\widehat{\mathfrak{g}})$-module with the highest $l$-weight monomial equal to 
one of the following forms:
\begin{align}  
& \begin{cases} 
Y_{i_1,a_1}Y_{i_2,a_2}\ldots Y_{i_k,a_k}, \\
Y_{i_1,a_1}Y_{i_2,a_2}\ldots Y_{i_k,a_k} Y_{n-1,a_{k+1}},  \\
Y_{i_1,a_1}Y_{i_2,a_2}\ldots Y_{i_k,a_k} Y_{n,a_{k+1}},  \\
Y_{i_1,a_1}Y_{i_2,a_2}\ldots Y_{i_k,a_k} Y_{n-1,a_{k+1}} Y_{n,a_{k+1}},  
\end{cases}  \label{Dnfom1}  \\
& \begin{cases} 
Y_{i_1,a_1} \ldots Y_{i_\ell,a_\ell} Y^2_{i_{\ell+1},a_{\ell+1}} \ldots Y^2_{i_k,a_k}  Y_{n-1,a_{k+1}} Y_{n,a_{k+1}}, \\
Y_{i_1,a_1} \ldots Y_{i_\ell,a_\ell} Y^2_{i_{\ell+1},a_{\ell+1}} \ldots Y^2_{i_k,a_k}  Y_{n-1,a_{k+1}} Y_{n,a_{k+1}} Y_{i_{\ell'},a_{\ell'}},  \\
\end{cases} \label{Dnfom2}  \\
& \begin{cases}
Y_{i_1,a_1} \ldots Y_{i_\ell,a_\ell} Y^2_{i_{\ell+1},a_{\ell+1}} \ldots Y^2_{i_k,a_k} Y_{n-2,a_{k+1}} Y_{n-1,a_{k+1}\pm 3} Y_{n,a_{k+1} \mp 1}, \\
Y_{i_1,a_1} \ldots Y_{i_\ell,a_\ell} Y^2_{i_{\ell+1},a_{\ell+1}} \ldots Y^2_{i_k,a_k} Y_{n-2,a_{k+1}} Y_{n-1,a_{k+1}\pm 3} Y_{n,a_{k+1} \mp 1} Y_{i_{\ell'},a_{\ell'}}, \\
Y_{i_1,a_1} \ldots Y_{i_\ell,a_\ell} Y^2_{i_{\ell+1},a_{\ell+1}} \ldots Y^2_{i_k,a_k} Y_{n-2,a_{k+1}} Y_{n,a_{k+1}\pm3} Y_{n-1,a_{k}\mp1}, \\
Y_{i_1,a_1} \ldots Y_{i_\ell,a_\ell} Y^2_{i_{\ell+1},a_{\ell+1}} \ldots Y^2_{i_k,a_k} Y_{n-2,a_{k+1}} Y_{n,a_{k+1}\pm3} Y_{n-1,a_{k+1}\mp1} Y_{i_{\ell'},a_{\ell'}}, 
\end{cases}   \label{Dnfom3}
\end{align}
where $k\in\mathbb{Z}_{\geq 1}$, $1\leq i_1<i_2<\cdots<i_k \leq n-2$ in (\ref{Dnfom1}) and (\ref{Dnfom2}), and $1\leq i_1<i_2<\cdots<i_k < n-2$ in (\ref{Dnfom3}), $a_j\in \mathbb{Z}$ for $j=1,2,\ldots,k$ satisfy 
\begin{itemize}
\item[(1)] $(a_{j}-a_{j-1})(a_{j+1}-a_{j})<0$ for $2\leq j\leq k$,
\item[(2)] $|a_{j}-a_{j-1}|=i_{j}-i_{j-1}+2$ for $2\leq j\leq k$,
\item[(3)] $|a_{k+1}-a_k|=n-i_k+1$ in (\ref{Dnfom1}) and (\ref{Dnfom2}), $|a_{k+1}-a_k|=n-i_k$ in (\ref{Dnfom3}), 
\item[(4)] there exists $k \geq \ell\geq 1$ such that $i_\ell < i_{\ell'} < i_{\ell+1}$ (we have the convention $i_{k+1}=n-1$), $|a_{\ell'}-a_\ell|=i_{\ell'}-i_\ell$, and $|a_{\ell+1}-a_{\ell'}|=i_{\ell+1}-i_{\ell'}+2$ in (\ref{Dnfom2}), 
\item[(5)] there exists $k \geq \ell\geq 1$ such that $i_\ell < i_{\ell'} < i_{\ell+1}$ (we have the convention $i_{k+1}=n-2$), $|a_{\ell'}-a_\ell|=i_{\ell'}-i_\ell$, and $|a_{\ell+1}-a_{\ell'}|=i_{\ell+1}-i_{\ell'}+2$ in (\ref{Dnfom3}).
\end{itemize}
\end{theorem}
\begin{proof}
It follows from the definition of our $\Phi(\{i,j\})$ that for a fixed height function $\xi$, every Hernandez-Leclerc module has the required highest $l$-weight monomial. 

In the following, we show that given a monomial $m$ of types (\ref{Dnfom1})--(\ref{Dnfom3}), there exists an Hernandez-Leclerc module with the highest $l$-weight monomial $m$. 

The case for (\ref{Dnfom1}) follows from the treatment of type $A$, and additional technical requirements: the second and third monomials need to let $|\xi(n)-\xi(n-1)|=2$, and the fourth monomial needs to let $\xi(n)=\xi(n-1)$.

Next, we deal with the case (\ref{Dnfom2}). If $a_1>a_2$, $a_{\ell}>a_{\ell+1}$, and $\xi(n-1)=\xi(n)$, we choose a height function $\xi$ such that $\xi$ is linear in intervals $[1,i_2], [i_2,i_3], \ldots, [i_k,n]$, where $\xi(i_1)=a_{1}>\xi(i_2)=a_{1}-i_2+i_1$. For $i_1+1<i_2$ and $i_{\ell'}+1\leq n-2$, by formula (\ref{the map Phi(i,j)}), we have 
\begin{gather*}
\begin{aligned}
\Phi(\{i_1+1,i_\ell \}) & = L(Y^{(1-d_{i_1+1})}_{i_1,\xi(i_1)} Y(i_1+1,n] Y(i_\ell,n-2] Y^{1-d_{i_\ell}}_{i_\ell,\xi(i_\ell)-2}) \\
& = L(Y_{i_1,a_1} \ldots Y_{i_\ell,a_\ell} Y^2_{i_{\ell+1},a_{\ell+1}} \ldots Y^2_{i_k,a_k}  Y_{n-1,a_{k+1}} Y_{n,a_{k+1}}), \\
\Phi(\{i_1+1,i_{\ell'}+1\}) & =  L(Y^{(1-d_{i_1+1})}_{i_1,\xi(i_1)} Y(i_1+1,n] Y[i_{\ell'}+1,n-2] Y_{i_{\ell'},\xi(i_{\ell'})}) \\
& =  L(Y_{i_1,a_1} \ldots Y_{i_\ell,a_\ell} Y^2_{i_{\ell+1},a_{\ell+1}} \ldots Y^2_{i_k,a_k}  Y_{n-1,a_{k+1}} Y_{n,a_{k+1}} Y_{i_{\ell'},a_{\ell'}}).
\end{aligned}
\end{gather*}
If $i_1+1<i_2$ and $i_{\ell'}+1=n-1$ (implying $\ell=k$), by formula (\ref{the map Phi(i,n-1)}), we have
\begin{gather*}
\begin{aligned}
\Phi(\{i_1+1,n-1) & =  L(Y^{(1-d_{i_1+1})}_{i_1,\xi(i_1)} Y(i_1+1,n) Y_{n,\xi(n)-2} Y_{n-2,\xi(n-2)}) \\
& =  L(Y_{i_1,a_1} \ldots Y_{i_k,a_k} Y_{n-1,a_{k+1}} Y_{n,a_{k+1}} Y_{n-2,a_{\ell'}}).
\end{aligned}
\end{gather*}

For $i_1+1=i_2$ and $i_{\ell'}+1<n-2$, by formula (\ref{the map Phi(i,j)}), we have 
\begin{gather*}
\begin{aligned}
\Phi(\{i_2,i_\ell\}) & = L(Y^{d_{i_2}}_{i_1,\xi(i_1)} Y_{i_2,\xi(i_2)-2} Y(i_2,n] Y(i_\ell,n-2] Y^{1-d_{i_\ell}}_{i_\ell,\xi(i_\ell)-2})  \\
& = L(Y_{i_1,a_1} \ldots Y_{i_\ell,a_\ell} Y^2_{i_{\ell+1},a_{\ell+1}} \ldots Y^2_{i_k,a_k}  Y_{n-1,a_{k+1}} Y_{n,a_{k+1}}), \\ 
\Phi(\{i_2,i_{\ell'}+1\}) & = L(Y^{d_{i_2}}_{i_1,\xi(i_1)} Y_{i_2,\xi(i_2)-2} Y(i_2,n]  Y[i_{\ell'}+1,n-2] Y_{i_{\ell'},\xi(i_{\ell'})}) \\
& = L(Y_{i_1,a_1} \ldots Y_{i_\ell,a_\ell} Y^2_{i_{\ell+1},a_{\ell+1}} \ldots Y^2_{i_k,a_k}  Y_{n-1,a_{k+1}} Y_{n,a_{k+1}} Y_{i_{\ell'},a_{\ell'}}).
\end{aligned}
\end{gather*}
If $i_1+1=i_2$ and $i_{\ell'}+1=n-1$ (implying $\ell=k$), by formula (\ref{the map Phi(i,n-1)}), we have
\begin{gather*}
\begin{aligned}
\Phi(\{i_1+1,n-1) & =  L(Y^{d_{i_1+1}}_{i_1,\xi(i_1)} Y_{i_1+1,\xi(i_1+1)-2} Y(i_1+1,n) Y_{n,\xi(n)-2} Y_{n-2,\xi(n-2)}) \\
& =  L(Y_{i_1,a_1} \ldots Y_{i_k,a_k} Y_{n-1,a_{k+1}} Y_{n,a_{k+1}} Y_{n-2,a_{\ell'}}).
\end{aligned}
\end{gather*}

If $a_1>a_2$, $a_{\ell}<a_{\ell+1}$, and $\xi(n-1)=\xi(n)$, we choose a height function $\xi$ such that $\xi$ is linear in intervals $[1,i_2], [i_2,i_3], \ldots, [i_k,n]$, where $\xi(i_1)=a_{1}>\xi(i_2)=a_{1}-i_2+i_1$. For $i_1+1<i_2$ and $i_{\ell+1}+1\leq n-2$, by formula (\ref{the map Phi(i,j)}), we have    
\begin{gather*}
\begin{aligned}
\Phi(\{i_1+1,i_{\ell+1}+1\}) & = L(Y^{(1-d_{i_1+1})}_{i_1,\xi(i_1)} Y(i_1+1,n] Y[i_{\ell+1}+1,n-2] Y_{i_{\ell+1},\xi(i_{\ell+1})})  \\
& = L(Y_{i_1,a_1} \ldots Y_{i_\ell,a_\ell} Y^2_{i_{\ell+1},a_{\ell+1}} \ldots Y^2_{i_k,a_k}  Y_{n-1,a_{k+1}} Y_{n,a_{k+1}}), \\
\Phi(\{i_1+1,i_{\ell'}\}) & = L(Y^{(1-d_{i_1+1})}_{i_1,\xi(i_1)} Y(i_1+1,n] Y(i_{\ell'},n-2] Y^{1-d_{i_{\ell'}}}_{i_{\ell'},\xi(i_{\ell'})-2}) \\
& = L(Y_{i_1,a_1} \ldots Y_{i_\ell,a_\ell} Y^2_{i_{\ell+1},a_{\ell+1}} \ldots Y^2_{i_k,a_k}  Y_{n-1,a_{k+1}} Y_{n,a_{k+1}} Y_{i_{\ell'},a_{\ell'}}). 
\end{aligned}
\end{gather*}
For $i_1+1=i_2$ and $i_{\ell+1}+1\leq n-2$, by formula (\ref{the map Phi(i,j)}), we have    
\begin{gather*}
\begin{aligned}
\Phi(\{i_2,i_{\ell+1}+1\}) & = L(Y^{d_{i_2}}_{i_1,\xi(i_1)} Y_{i_2,\xi(i_2)-2} Y(i_2,n]  Y[i_{\ell+1}+1,n-2] Y_{i_{\ell+1},\xi(i_{\ell+1})}) \\
& = L(Y_{i_1,a_1} \ldots Y_{i_\ell,a_\ell} Y^2_{i_{\ell+1},a_{\ell+1}} \ldots Y^2_{i_k,a_k}  Y_{n-1,a_{k+1}} Y_{n,a_{k+1}}), \\
\Phi(\{i_2,i_{\ell'}\}) & = L(Y^{d_{i_2}}_{i_1,\xi(i_1)} Y_{i_2,\xi(i_2)-2} Y(i_2,n] Y(i_{\ell'},n-2] Y^{1-d_{i_{\ell'}}}_{i_{\ell'},\xi(i_{\ell'})-2}) \\
& = L(Y_{i_1,a_1} \ldots Y_{i_\ell,a_\ell} Y^2_{i_{\ell+1},a_{\ell+1}} \ldots Y^2_{i_k,a_k}  Y_{n-1,a_{k+1}} Y_{n,a_{k+1}} Y_{i_{\ell'},a_{\ell'}}).
\end{aligned}
\end{gather*}
For $i_1+1<i_2$ or $i_1+1=i_2$ and $i_{\ell+1}+1=n$ (implying $\ell=k$), $\Phi(\{i_1+1,n\})$ and $\Phi(\{i_2,n\})$ are the cases of type (\ref{Dnfom1}). For $i_1+1<i_2$ or $i_1+1=i_2$ and $i_{\ell+1}+1=n-1$ (implying $\ell+1=k$), by formula (\ref{the map Phi(i,n-1)}), 
\begin{gather*}
\begin{aligned}
\Phi(\{i_1+1,n-1\}) & = L(Y^{(1-d_{i_1+1})}_{i_1,\xi(i_1)} Y(i_1+1,n) Y_{n-2,\xi(n-2)} Y_{n,\xi(n)-2}) \\
& = Y_{i_1,a_1} \ldots Y_{i_\ell,a_\ell} Y^2_{n-2,\xi(n-2)} Y_{n-1,a_{k+1}} Y_{n,a_{k+1}}, \\
\Phi(\{i_2,n-1\}) & =  L(Y^{d_{i_2}}_{i_1,\xi(i_1)} Y_{i_2,\xi(i_2)-2} Y(i_2,n) Y_{n-2,\xi(n-2)} Y_{n,\xi(n)-2}) \\
& =  Y_{i_1,a_1} \ldots Y_{i_\ell,a_\ell} Y^2_{n-2,\xi(n-2)} Y_{n-1,a_{k+1}} Y_{n,a_{k+1}}.
\end{aligned}
\end{gather*}

If $a_1<a_2$, $a_{\ell}>a_{\ell+1}$, and $\xi(n-1)=\xi(n)$, then we choose a height function $\xi$ such that $\xi$ is linear in intervals $[1,i_2], [i_2,i_3], \ldots, [i_k,n]$, where $\xi(i_1)=a_{2}-i_2+i_1<\xi(i_2)=a_{2}$. For $i_{\ell+1}+1\leq n-2$, by formula (\ref{the map Phi(i,j)}), we have

\begin{gather*}
\begin{aligned}
\Phi(\{i_1,i_{\ell}\}) & = L(Y^{d_{i_1}}_{i_1-1,\xi(i_1-1)} Y_{i_1,\xi(i_1)-2} Y(i_1,n] Y(i_{\ell},n-2] Y^{1-d_{i_{\ell}}}_{i_{\ell},\xi(i_{\ell})-2}) \\
& = L(Y_{i_1,a_1} \ldots Y_{i_\ell,a_\ell} Y^2_{i_{\ell+1},a_{\ell+1}} \ldots Y^2_{i_k,a_k}  Y_{n-1,a_{k+1}} Y_{n,a_{k+1}}), \\
\Phi(\{i_1,i_{\ell'}+1\}) & = L(Y^{d_{i_1}}_{i_1-1,\xi(i_1-1)} Y_{i_1,\xi(i_1)-2} Y(i_1,n]  Y[i_{\ell'}+1,n-2] Y_{i_{\ell'},\xi(i_{\ell'})}) \\
& = L(Y_{i_1,a_1} \ldots Y_{i_\ell,a_\ell} Y^2_{i_{\ell+1},a_{\ell+1}} \ldots Y^2_{i_k,a_k}  Y_{n-1,a_{k+1}} Y_{n,a_{k+1}} Y_{i_{\ell'},a_{\ell'}}).
\end{aligned}
\end{gather*}
If $i_{\ell'}+1=n-1$ (implying $\ell=k$), by formula (\ref{the map Phi(i,n-1)}), we have
\begin{gather*}
\begin{aligned}
\Phi(\{i_1,n-1) & =  L(Y^{d_{i_1}}_{i_1-1,\xi(i_1-1)} Y_{i_1,\xi(i_1)-2} Y(i_1,n) Y_{n-2,\xi(n-2)} Y_{n,\xi(n)-2})  \\
& =  L(Y_{i_1,a_1} \ldots Y_{i_k,a_k} Y_{n-1,a_{k+1}} Y_{n,a_{k+1}} Y_{n-2,\xi(n-2)}).
\end{aligned}
\end{gather*}

If $a_1<a_2$, $a_{\ell}<a_{\ell+1}$, and $\xi(n-1)=\xi(n)$, then we choose a height function $\xi$ such that $\xi$ is linear in intervals $[1,i_2], \ldots, [i_k,n]$, where $\xi(i_1)=a_{2}-i_2+i_1<\xi(i_2)=a_{2}$. For $i_{\ell+1}+1\leq n-2$, by formula (\ref{the map Phi(i,j)}), we have  
\begin{gather*}
\begin{aligned}
\Phi(\{i_1,i_{\ell+1}+1\}) & = L(Y^{d_{i_1}}_{i_1-1,\xi(i_1-1)} Y_{i_1,\xi(i_1)-2} Y(i_1,n]  Y[i_{\ell+1}+1,n-2] Y_{i_{\ell+1},\xi(i_{\ell+1})}) \\
& = L(Y_{i_1,a_1} \ldots Y_{i_\ell,a_\ell} Y^2_{i_{\ell+1},a_{\ell+1}} \ldots Y^2_{i_k,a_k}  Y_{n-1,a_{k+1}} Y_{n,a_{k+1}}), \\
\Phi(\{i_1,i_{\ell'}\}) & = L(Y^{d_{i_1}}_{i_1-1,\xi(i_1-1)} Y_{i_1,\xi(i_1)-2} Y(i_1,n] Y(i_{\ell'},n-2] Y^{1-d_{i_{\ell'}}}_{i_{\ell'},\xi(i_{\ell'})-2}) \\
& = L(Y_{i_1,a_1} \ldots Y_{i_\ell,a_\ell} Y^2_{i_{\ell+1},a_{\ell+1}} \ldots Y^2_{i_k,a_k}  Y_{n-1,a_{k+1}} Y_{n,a_{k+1}} Y_{i_{\ell'},a_{\ell'}}).
\end{aligned}
\end{gather*}
For $i_{\ell+1}+1=n$ (implying $\ell=k$), $\Phi(\{i_1,n\})$ is the case of type (\ref{Dnfom1}). For $i_{\ell+1}+1=n-1$ (implying $\ell+1=k$), by formula (\ref{the map Phi(i,n-1)}), 
\begin{gather*}
\begin{aligned}
\Phi(\{i_1,n-1\}) & = L(Y^{d_{i_1}}_{i_1-1,\xi(i_1-1)} Y_{i_1,\xi(i_1)-2} Y(i_1,n) Y_{n-2,\xi(n-2)} Y_{n,\xi(n)-2})  \\
& =  Y_{i_1,a_1} \ldots Y_{i_\ell,a_\ell} Y^2_{n-2,\xi(n-2)} Y_{n-1,a_{k+1}} Y_{n,a_{k+1}}.
\end{aligned}
\end{gather*}

The differences between the monomials in (\ref{Dnfom2}) and the monomials in (\ref{Dnfom3}) are $|\xi(n)-\xi(n-1)|=0$ and $|\xi(n)-\xi(n-1)|=2$. In the proof of the existence of $\xi$ for (\ref{Dnfom2}), we let $\xi(n-1) \neq \xi(n)$, and choose the corresponding formula in (\ref{the map Phi(i,n-1)}) or (\ref{the map Phi(i,j)}). Then for a monomial $m$ in (\ref{Dnfom3}), we find a height function $\xi$ such that an Hernandez-Leclerc module has the highest $l$-weight monomial $m$.
\end{proof}
 
\begin{remark}
For a linear height function $\xi$ in type $D_n$, the Hernandez-Leclerc modules of the form
\[
L(Y_{i,\xi(i)}Y_{j,\xi(j)}Y_{n-1,\xi(n-1)-2}Y_{n,\xi(n)-2}), \,\,\,  \text{with $0\leq i < j \leq n-2$},
\] 
were introduced and studied in \cite{HL13}. They belong to the second case of (\ref{Dnfom2}) in Theorem~\ref{monoidal categorification DC1} if we let $k=1$.
\end{remark}

\subsection{Types $E_6$, $E_7$, and $E_8$}\label{sect 73}
The Dynkin diagram of type $E_6$ contains a copy of the Dynkin diagram of type $D_4$ and two copies of the Dynkin diagram of type $D_5$. There are $42=(36+6)$ cluster variables, among them there are 7 cluster variables which are not  of type $A$ or $D$. These are in one-to-one correspondence with the following positive roots: 
\begin{align} \label{E6positiveroot}
(\substack{ 1 \\ 11111}), 
(\substack{ 1 \\ 11211}),  
(\substack{ 1 \\ 12211}), 
(\substack{ 1 \\ 11221}), 
(\substack{ 1 \\ 12221}), 
(\substack{ 1 \\ 12321}),
(\substack{ 2 \\ 12321}).
\end{align}
An arbitrary orientation of $\gamma$ of type $E_6$ gives a type $E_6$ quiver. So the number of Dynkin quivers of type $E_6$ is $2^5=32$. Up to spectral parameter shift, the total number of the highest $l$-weight monomials of Hernandez-Leclerc modules of type $E_6$ which are not of type $A$ or $D$ is therefore equal to  $2^5\times 7=224$.
These are listed in Appendix~\ref{Hernandez-Leclerc modules of type E6}.

The Dynkin diagram of type $E_7$ contains some copies of the Dynkin diagram of type $D_4, D_5, D_6$ or $E_6$. There are $70=(63+7)$ cluster variables, among them there are 16 cluster variables which are not ones of the type $A$, $D$, or $E_6$. These are in one-to-one correspondence with the following positive roots:  
\begin{align} \label{E7positiveroot}
\begin{split}
& (\substack{\hskip -0.2cm 1 \\ 111111}),  (\substack{\hskip -0.2cm 1 \\ 112111}),  (\substack{\hskip -0.2cm 1 \\ 112211}), (\substack{\hskip -0.2cm 1 \\ 122111}),   (\substack{\hskip -0.2cm 1 \\ 112221}), (\substack{\hskip -0.2cm 1 \\ 122211}), (\substack{\hskip -0.2cm 1 \\ 122221}), (\substack{\hskip -0.2cm 1  \\ 123211}),   \\
& (\substack{\hskip -0.2cm 1 \\ 123221}), (\substack{\hskip -0.2cm 2 \\ 123211}), (\substack{\hskip -0.2cm 1 \\ 123321}), (\substack{\hskip -0.2cm 2 \\ 123221}), 
(\substack{\hskip -0.2cm 2 \\ 123321}), (\substack{\hskip -0.2cm 2 \\ 124321}),  (\substack{\hskip -0.2cm 2 \\ 134321}), (\substack{\hskip -0.2cm 2 \\ 234321}).
\end{split}
\end{align}
Up to spectral parameter shift, the total number of the highest $l$-weight monomials of Hernandez-Leclerc modules of type $E_7$ which are not ones of type $A$, $D$ or $E_6$ is $2^6\times 16=1024$. These are listed in Appendix~\ref{Hernandez-Leclerc modules of type E7}.

The Dynkin diagram of type $E_8$ contains some copies of the Dynkin diagram of type $D_4, D_5, D_6, D_7$ or $E_7$. There are $128=(120+8)$ cluster variables, among them there are 44 cluster variables which are not ones of the type $A$, $D$, or $E_7$. These are in one-to-one correspondence with the following positive roots: 
\begin{align} \label{E8positiveroot}
\begin{split}
& (\substack{\hskip -0.3cm 1 \\ 1111111}),  (\substack{\hskip -0.3cm 1 \\ 1121111}),  (\substack{\hskip -0.3cm 1 \\ 1221111}), (\substack{\hskip -0.3cm 1 \\ 1122111}),   (\substack{\hskip -0.3cm 1 \\ 1222111}), (\substack{\hskip -0.3cm 1 \\ 1122211}), (\substack{\hskip -0.3cm 1  \\ 1232111}),  (\substack{\hskip -0.3cm 1 \\ 1222211}),  (\substack{\hskip -0.3cm 1 \\ 1122221}), \\
& (\substack{\hskip -0.3cm 2 \\ 1232111}),  (\substack{\hskip -0.3cm 1 \\ 1232211}), (\substack{\hskip -0.3cm 1 \\ 1222221}), (\substack{\hskip -0.3cm 2 \\ 1232211}), (\substack{\hskip -0.3cm 1 \\ 1233211}),  (\substack{\hskip -0.3cm 1 \\ 1232221}), (\substack{\hskip -0.3cm 2 \\ 1233211}),  (\substack{\hskip -0.3cm 2 \\ 1232221}), (\substack{\hskip -0.3cm 1 \\ 1233221}),  \\
& (\substack{\hskip -0.3cm 2 \\ 1243211}),  (\substack{\hskip -0.3cm 2 \\ 1233221}),  (\substack{\hskip -0.3cm 1 \\ 1233321}),  (\substack{\hskip -0.3cm 2 \\ 1343211}), (\substack{\hskip -0.3cm 2 \\ 1243221}),  (\substack{\hskip -0.3cm 2 \\ 1233321}),  (\substack{\hskip -0.3cm 2 \\ 2343211}), (\substack{\hskip -0.3cm 2 \\ 1343221}), (\substack{\hskip -0.3cm 2 \\ 1243321}),   \\
& (\substack{\hskip -0.3cm 2 \\ 2343221}), (\substack{\hskip -0.3cm 2 \\ 1343321}), (\substack{\hskip -0.3cm 2\\ 1244321}),  (\substack{\hskip -0.3cm 2 \\ 2343321}), (\substack{\hskip -0.3cm 2 \\ 1344321}), (\substack{\hskip -0.3cm 2 \\ 1354321}),  (\substack{\hskip -0.3cm 2 \\ 2344321}),  (\substack{\hskip -0.3cm 3 \\ 1354321}),   (\substack{\hskip -0.3cm 2 \\ 2354321}),   \\
&  (\substack{\hskip -0.3cm 3 \\ 2354321}),  (\substack{\hskip -0.3cm 2 \\ 2454321}),  (\substack{\hskip -0.3cm 3 \\ 2454321}),  (\substack{\hskip -0.3cm 3 \\ 2464321}), (\substack{\hskip -0.3cm 3 \\ 2465321}),  (\substack{\hskip -0.3cm 3 \\ 2465421}), (\substack{\hskip -0.3cm 3 \\ 2465431}), (\substack{\hskip -0.3cm 3 \\ 2465432}).
\end{split}
\end{align}
Up to spectral parameter shift, the total number of the highest $l$-weight monomials of Hernandez-Leclerc modules of type $E_8$ which are not of type $A$, $D$ or $E_7$ is $2^7\times 44=5632$. These are listed in Appendix~\ref{Hernandez-Leclerc modules of type E8}.

For a quiver $Q$, its Auslander-Reiten quiver can be constructed according to the content introduced in Section \ref{Auslander-Reiten quivers}. We make use of Corollary \ref{a corollary of main theorem2} to obtain the highest $l$-weight monomials of Hernandez-Leclerc modules of type $E$, starting with the images $\Phi(I(i))$ of the indecomposable injective modules $I(i)$ and apply Corollary \ref{a corollary of main theorem2} until we reach the images $\Phi(P(i))$ of the indecomposable projective modules $P(i)$, see Appendices \ref{Hernandez-Leclerc modules of type E6}, \ref{Hernandez-Leclerc modules of type E7}, and \ref{Hernandez-Leclerc modules of type E8}.

In the process of the computation of these highest $l$-weight monomials, to avoid mistakes, we refer to Keller's quiver mutation applet \cite{Kel}. For any height function $\xi$, we create a Dynkin quiver $Q$ (in the java, we input one more vertex than what we actually need), choose the menu ``Repetition--Highest weights", and then input ``2" in the pop-up windows. These extra vertices are frozen vertices.  Suppose that $(i_1, i_2, \ldots i_n)$ is a source sequence adapted to $Q$ such that $c=s_{i_1} s_{i_2} \cdots s_{i_n}$ is the Coxeter element. We mutate the sequence $(i_1, i_2, \ldots i_n)$ 7 times for type $E_6$, 10 times for type $E_7$, and 16 times for type $E_8$. Our Corollary \ref{a corollary of main theorem2} guarantees that the highest $l$-weight monomials of Hernandez-Leclerc modules can be inductively computed according to the formula (\ref{Gamma1 equation}). 

When $Q$ is a sink-source quiver, the highest $l$-weight monomials of Hernandez-Leclerc modules of type $E_6$ which are not of type $A$ or $D$ (respectively, of type $E_7$ which are not ones of type $A$, $D$ or $E_6$, of type $E_8$ which are not of type $A$, $D$, or $E_7$) have a beautiful correspondence with a subset of almost positive roots (\ref{E6positiveroot}) (respectively, (\ref{E7positiveroot}), (\ref{E8positiveroot})), that is 
\begin{align}\label{E6 correspsondence}
\alpha=\sum_{i\in I} a_i \alpha_i \mapsto \Phi(\alpha)=[L(\prod_{i\in I_{\text{source}}} Y^{a_i}_{i,\xi(i)} \prod_{j\in I_{\text{sink}}} Y^{a_j}_{j,\xi(j)-2})],
\end{align}
where $I_{\text{source}}$ is the set of source vertices in $Q$, $I_{\text{sink}}$ is the set of sink vertices in $Q$. Note that the correspondence (\ref{E6 correspsondence}) does not mean that the denominator of $\Phi(\alpha)$ as the Laurent expression of a cluster variable is parameterized by $\alpha$. For a general quiver $Q$, we hope that there is a bijection given by an explicit formula.

\section*{Acknowledgements}

The authors thank Professor Jiarui Fei for valuable discussions on $c$-vectors of an exchange pair in a cluster category and for pointing out the reference \cite{Hub06} to us. We thank Professor Changjian Fu for valuable discussions and a proof of Lemma \ref{g-vector equation in an exchange pair} and Professor David Hernandez for his encouragement on the computation of the highest $l$-weight monomials of Hernandez-Leclerc modules of type $E_8$ and pointing out the reference \cite{KKKO15} to us. We thank Professor Fan Qin for telling us after our paper was uploaded to Arxiv that our Conjecture \ref{rigid objects equal real prime simple modules conjecture} is equivalent to an equivalence between the additive reachability conjecture and the multiplicative reachability conjecture, see Remark 5.9 in \cite{Qin20}.

\clearpage

\begin{appendix} 

\setcounter{table}{0}  
\setcounter{figure}{0}

\section{Type $E_{6}$} \label{Hernandez-Leclerc modules of type E6}

In this appendix, we list the highest $l$-weight monomials of the Hernandez-Leclerc modules of type $E_6$  which are not  of type $A$ or $D$ (up to spectral parameter shift). We have seen in Section~\ref{sect 73} that there are $2^5\times 7$ of them, but we can reduce this number to $2^4\times 7=16\times 7$, because of the duality between a quiver $Q$ and its opposite quiver $Q^{\text op}$. Each of the rows below contains 7 different highest weights.
For simplicity,  we denote by $i_r$ the variable $Y_{i,r}$ for any $i\in I$, $r\in \mathbb{Z}$.

\begin{gather}
\begin{align*}
&(1)  1_{0}2_{\pm5}6_{0}, 1_{0}2_{\pm5}4_{\pm4}6_{0}, 1_{0}2_{\pm5}4_{\pm4}5_{\pm1}6_{0}, 1_{0}2_{\pm5}3_{\pm1}4_{\pm4}6_{0}, 1_{0}2_{\pm5}3_{\pm1}4_{\pm4}5_{\pm1}6_{0},  1_{0}2_{\pm5}3_{\pm1}4^2_{\pm4}5_{\pm1}6_{0},  1_{0}2^2_{\pm5}3_{\pm1}4_{\pm4}5_{\pm1}6_{0}, \\
&(2)  1_0 2_{\pm 1} 6_{\pm 6},  1_{0} 2_{\pm 1} 4_{\pm4} 6_{\pm6}, 1_{0} 2_{\pm1} 5_{\pm5} 6_{\pm6},  1_0 2_{\pm1} 3_{\pm1} 4_{\pm4} 6_{\pm6},  1_{0} 2_{\pm1} 3_{\pm1} 5_{\pm5} 6_{\pm6}, 1_{0} 2_{\pm1} 3_{\pm1} 4_{\pm4} 5_{\pm5} 6_{\pm6},  1_{0} 2^2_{\pm1} 3_{\pm1} 4_{\pm4} 5_{\pm5} 6_{\pm6}, \\
&(3)  1_{0}2_{\pm5}6_{\pm6}, 1_{0}2_{\pm5}3_{\pm1}6_{\pm6}, 1_{0}2_{\pm5}4_{\pm2}6_{\pm6},  1_{0}2_{\pm5}3_{\pm1}5_{\pm5}6_{\pm6},  1_{0}2_{\pm5}4_{\pm2}5_{\pm5}6_{\pm6},  1_{0}2_{\pm5}3_{\pm1}4_{\pm2}5_{\pm5}6_{\pm6}, 1_{0}2^2_{\pm5}3_{\pm1}4_{\pm2}5_{\pm5}6_{\pm6}, \\
&(4)  1_{0}2_{\pm1}4_{\pm4}6_{0}, 1_{0}2_{\pm1}4^2_{\pm4}6_{0}, 1_{0}2_{\pm1}3_{\pm1}4^2_{\pm4}6_{0},  1_{0}2_{\pm1}4^2_{\pm4}5_{\pm1}6_{0}, 1_{0}2_{\pm1}3_{\pm1}4^2_{\pm4}5_{\pm1}6_{0}, 1_{0}2_{\pm1}3_{\pm1}4^3_{\pm4}5_{\pm1}6_{0}, 1_{0}2^{2}_{\pm1}3_{\pm1}4^3_{\pm4}5_{\pm1}6_{0}, \\
&(5)  1_{\pm5}2_{\pm4}5_06_{\pm3}, 1_{\pm5}2_{\pm4}4_{\pm1}5_06_{\pm3}, 1_{\pm5}2_{\pm4}5^2_06_{\pm3}, 1_{\pm5}2_{\pm4}3_{\pm4}5^2_06_{\pm3}, 1_{\pm5}2_{\pm4}3_{\pm4}4_{\pm 1}5_06_{\pm3}, \\
& \,\,\quad  1_{\pm5}2_{\pm4}3_{\pm4}4_{\pm1}5^2_06_{\pm3},  1_{\pm5}2^2_{\pm4}3_{\pm4}4_{\pm1}5^2_06_{\pm3}, \\
&(6) 1_{\pm2}2_{\pm1}3_{\pm5}6_{0}, 1_{\pm2}2_{\pm1}3_{\pm5}4_{\pm4}6_{0},  1_{\pm2}2_{\pm1}3_{\pm5}4_{\pm4}5_{\pm1}6_{0}, 1_{\pm2}2_{\pm1}3^2_{\pm5}6_{0}, 1_{\pm2}2_{\pm1}3^2_{\pm5}5_{\pm1}6_{0}, \\
 & \,\,\quad  1_{\pm2}2_{\pm1}3^2_{\pm5}4_{\pm4}5_{\pm1}6_{0}, 1_{\pm2}2^2_{\pm1}3^2_{\pm5}4_{\pm4}5_{\pm1}6_{0}, \\
& (7)  1_{\pm2}2_{\pm5}3_{\pm5}6_{0},  1_{\pm2}2_{\pm5}3_{\pm5}4_{\pm2}6_{0},  1_{\pm2}2_{\pm5}3_{\pm5}5_{\pm1}6_{0}, 1_{\pm2}2_{\pm5}3^2_{\pm5}4_{\pm2}6_{0},  1_{\pm2}2_{\pm5}3^2_{\pm5}5_{\pm1}6_{0}, \\
& \,\,\quad  1_{\pm2}2_{\pm5}3^2_{\pm5}4_{\pm2}5_{\pm1}6_{0},  1_{\pm2}2^2_{\pm5}3^2_{\pm5}4_{\pm2}5_{\pm1}6_{0}, \\
&(8)  1_{0}2_{\pm1}5_{\pm5}6_{\pm2},  1_{0}2_{\pm1}4_{\pm4}5_{\pm5}6_{\pm2}, 1_{0}2_{\pm1}5^2_{\pm5}6_{\pm2},  1_{0}2_{\pm1}3_{\pm1}4_{\pm4}5_{\pm5}6_{\pm2},  1_{0}2_{\pm1}3_{\pm1}5^2_{\pm5}6_{\pm2}, \\
& \,\,\quad  1_{0}2_{\pm1}3_{\pm1}4_{\pm4}5^2_{\pm5}6_{\pm2}, 1_{0}2^2_{\pm1}3_{\pm1}4_{\pm4}5^2_{\pm5}6_{\pm2}, \\
&(9)  1_{0}2_{\pm5}5_{\pm5}6_{\pm2}, 1_{0}2_{\pm5}4_{\pm2}5_{\pm5}6_{\pm2}, 1_{0}2_{\pm5}4_{\pm2}5^2_{\pm5}6_{\pm2}, 1_{0}2_{\pm5}3_{\pm1}5_{\pm5}6_{\pm2}, 1_{0}2_{\pm5}3_{\pm1}5^2_{\pm5}6_{\pm2},  \\
& \,\,\quad   1_{0}2_{\pm5}3_{\pm1}4_{\pm2}5^2_{\pm5}6_{\pm2}, 1_{0}2^2_{\pm5}3_{\pm1}4_{\pm2}5^2_{\pm5}6_{\pm2}, \\
&(10)  1_{\pm3}2_03_05_{\pm4}6_{\pm1}, 1_{\pm3} 2_0 3_0 4_{\pm3} 5_{\pm4} 6_{\pm1}, 1_{\pm3} 2_0 3^2_0 4_{\pm3} 5_{\pm4} 6_{\pm1},  1_{\pm3} 2_0 3_0 5^2_{\pm4} 6_{\pm1},  1_{\pm3}2_{0}3^2_{0}5^2_{\pm4}6_{\pm1}, \\
& \,\,\,\,\,  \quad  1_{\pm3} 2_0 3^2_0 4_{\pm3} 5^2_{\pm4} 6_{\pm1},  1_{\pm3} 2^2_0 3^2_0 4_{\pm3} 5^2_{\pm4} 6_{\pm1}, \\
& (11)  1_{\pm4}2_{\pm3}4_{0}5_{\pm3}6_{0},  1_{\pm4}2_{\pm3}4^2_{0}5_{\pm3}6_{0}, 1_{\pm4}2_{\pm3}4^2_{0}5^2_{\pm3}6_{0}, 1_{\pm4}2_{\pm3}3_{\pm3}4^2_{0}5_{\pm3}6_{0}, 1_{\pm4}2_{\pm3}3_{\pm3}4^2_{0}5^2_{\pm3}6_{0}, \\
& \,\,\,\,\, \quad  1_{\pm4}2_{\pm3}3_{\pm3}4^3_{0}5^2_{\pm3}6_{0},  1_{\pm4}2^2_{\pm3}3_{\pm3}4^3_{0}5^2_{\pm3}6_{0}, \\
& (12)  1_{\pm 3} 2_0 3_{0} 4_{\pm 3} 5_0 6_{\pm3},  1_{\pm3} 2_0 3_{0} 4_{\pm 3}^2 5_{0} 6_{\pm 3}, 1_{\pm3} 2_0 3^2_0 4^2_{\pm3} 5_{0} 6_{\pm3},  1_{\pm3} 2_0 3_{0}  4^2_{\pm3} 5^2_{0} 6_{\pm3},  1_{\pm3} 2_0 3^2_{0} 4^2_{\pm3} 5^2_{0} 6_{\pm3}, \\
& \,\,\,\,\, \quad  1_{\pm3} 2_0 3^2_{0} 4^3_{\pm3} 5^2_{0} 6_{\pm3},  1_{\pm3} 2^2_{0} 3^2_{0} 4^3_{\pm3} 5^2_{0} 6_{\pm3}, \\
& (13) 1_{0}2_{\pm5}5_{\pm1}6_{\pm4}, 1_{0}2_{\pm5}4_{\pm4}5_{\pm1}6_{\pm4}, 1_{0}2_{\pm5}4_{\pm4}5^2_{\pm1}6_{\pm4}, 1_{0}2_{\pm5}3_{\pm1}4_{\pm4}5_{\pm1}6_{\pm4}, 1_{0}2_{\pm5}3_{\pm1}4_{\pm4}5^2_{\pm1}6_{\pm4}, \\
& \,\,\,\,\, \quad  1_{0}2_{\pm5}3_{\pm1}4^2_{\pm4}5^2_{\pm1}6_{\pm4},  1_{0}2^2_{\pm5}3_{\pm1}4_{\pm4}5^2_{\pm1}6_{\pm4}, \\
& (14)  1_{0}2_{\pm 1}4_{\pm 4}5_{\pm 1}6_{\pm 4},  1_{0}2_{\pm 1}4^2_{\pm 4}5_{\pm 1}6_{\pm 4},  1_{0}2_{\pm 1}4^2_{\pm 4}5^2_{\pm 1}6_{\pm 4},  1_{0}2_{\pm 1}3_{\pm 1}4^2_{\pm 4}5_{\pm 1}6_{\pm 4}, 1_{0}2_{\pm 1}3_{\pm 1}4^2_{\pm 4}5^2_{\pm 1}6_{\pm 4}, \\
& \,\,\,\,\, \quad  1_{0}2_{\pm 1}3_{\pm 1}4^3_{\pm 4}5^2_{\pm 1}6_{\pm 4},   1_{0}2^2_{\pm 1}3_{\pm 1}4^3_{\pm 4}5^2_{\pm 1}6_{\pm 4}, \\
& (15)  1_{\pm4}2_{\pm1}3_{\pm1}4_{\pm4}6_{0}, 1_{\pm4}2_{\pm1}3_{\pm1}4^2_{\pm4}6_{0},  1_{\pm4} 2_{\pm1} 3_{\pm1}^{2} 4_{\pm4}^{2} 6_{0}, 1_{\pm4}2_{\pm1}3_{\pm1}4^2_{\pm4}5_{\pm1}6_{0}, 1_{\pm4}2_{\pm1}3^2_{\pm1}4^2_{\pm4}5_{\pm1}6_{0}, \\
& \,\,\,\,\, \quad  1_{\pm4}2_{\pm1}3^2_{\pm1}4^3_{\pm4}5_{\pm1}6_{0}, 1_{\pm4}2^2_{\pm1}3^2_{\pm1}4^3_{\pm4}5_{\pm1}6_{0}, \\
& (16) 1_{\pm3}2_{\pm4}3_{0}5_{\pm4}6_{\pm1}, 1_{\pm3}2_{\pm4}3_{0}4_{\pm1}5_{\pm4}6_{\pm1}, 1_{\pm3}2_{\pm4}3^2_{0}5_{\pm4}6_{\pm1},  1_{\pm3}2_{\pm4}3_{0}4_{\pm1}5^2_{\pm4}6_{\pm1}, 1_{\pm3}2_{\pm4}3^2_{0}5^2_{\pm4}6_{\pm1},  \\
& \,\,\,\,\, \quad  1_{\pm3}2_{\pm4}3^2_{0}4_{\pm1}5^2_{\pm4}6_{\pm1}, 1_{\pm3}2^2_{\pm4}3^2_{0}4_{\pm1}5^2_{\pm4}6_{\pm1}.
\end{align*}
\end{gather}

\clearpage

\section{Type $E_{7}$} \label{Hernandez-Leclerc modules of type E7}

The highest $l$-weight monomials of the Hernandez-Leclerc modules of type $E_7$ which are not of type $A$ $D$ or $E_6$ (up to spectral parameter shift) are  as follows. Each paragraph contains 16 different highest weights.
\begin{gather}
\begin{align*}
& (1) 1_{\pm3}2_03_04_{\pm3}5_07_{\pm4}, 1_{\pm3}2_03_04^2_{\pm3}5_07_{\pm4}, 1_{\pm3}2_03^2_04^2_{\pm3}5_07_{\pm4},  1_{\pm3}2_03_04^2_{\pm3}5^2_07_{\pm4}, 1_{\pm3}2_03^2_04^2_{\pm3}5^2_07_{\pm4}, 1_{\pm3}2_03_04^2_{\pm3}5^2_06_{\pm3}7_{\pm4}, \\
& \quad 1_{\pm3}2_03^2_04^3_{\pm3}5^2_07_{\pm4}, 1_{\pm3}2_03^2_04^2_{\pm3}5^2_06_{\pm3}7_{\pm4}, 1_{\pm3}2^2_03^2_04^3_{\pm3}5^2_07_{\pm4}, 1_{\pm3}2_03^2_04^3_{\pm3}5^2_06_{\pm3}7_{\pm4},  1_{\pm3}2^2_03^2_04^3_{\pm3}5^2_06_{\pm3}7_{\pm4}, \\
& \quad 1_{\pm3}2_03^2_04^3_{\pm3}5^3_06_{\pm3}7_{\pm4},  1_{\pm3}2^2_03^2_04^3_{\pm3}5^3_06_{\pm3}7_{\pm4}, 1_{\pm3}2^2_03^2_04^4_{\pm3}5^3_06_{\pm3}7_{\pm4}, 1_{\pm3}2^2_03^3_04^4_{\pm3}5^3_06_{\pm3}7_{\pm4}, 1^2_{\pm3}2^2_03^3_04^4_{\pm3}5^3_06_{\pm3}7_{\pm4}, \\
& (2)  1_{\pm5}2_{\pm4}5_{0}6_{\pm3}7_{0}, 1_{\pm5}2_{\pm4}4_{\pm1}5_{0}6_{\pm3}7_{0}, 1_{\pm5}2_{\pm4}3_{\pm4}4_{\pm1}5_{0}6_{\pm3}7_{0}, 1_{\pm5}2_{\pm4}5^2_{0}6_{\pm3}7_{0}, 1_{\pm5}2_{\pm4}3_{\pm4}5^2_{0}6_{\pm3}7_{0}, 1_{\pm5}2_{\pm4}5^2_{0}6^2_{\pm3}7_{0},   \\ 
& \quad 1_{\pm5}2_{\pm4}3_{\pm4}4_{\pm1}5^2_{0}6_{\pm3}7_{0}, 1_{\pm5}2_{\pm4}3_{\pm4}5^2_{0}6^2_{\pm3}7_{0}, 1_{\pm5}2^2_{\pm4}3_{\pm4}4_{\pm1}5^2_{0}6_{\pm3}7_{0}, 1_{\pm5}2_{\pm4}3_{\pm4}4_{\pm1}5^2_{0}6^2_{\pm3}7_{0}, 1_{\pm5}2^2_{\pm4}3_{\pm4}4_{\pm1}5^2_{0}6^2_{\pm3}7_{0},   \\
& \quad 1_{\pm5}2_{\pm4}3_{\pm4}5^2_{3}6^2_{\pm3}7_{0}, 1_{\pm5}2^2_{\pm4}3_{\pm4}5^2_{3}6^2_{\pm3}7_{0}, 1_{\pm5}2^2_{\pm4}3_{\pm4}4_{\pm1}5^2_{3}6^2_{\pm3}7_{0}, 1_{\pm5}2^2_{\pm4}3^2_{\pm4}4_{\pm1}5^2_{3}6^2_{\pm3}7_{0},  1^2_{\pm5}2^2_{\pm4}3_{\pm4}4_{\pm1}5^2_{3}6^2_{\pm3}7_{0}, \\
& (3) 1_{\pm3}2_{\pm4}3_{0}5_{0}7_{\pm4}, 1_{\pm3}2_{\pm4}3_{0}4_{\pm3}5_{0}7_{\pm4}, 1_{\pm3}2_{\pm4}3^2_{0}4_{\pm3}5_{0}7_{\pm4}, 1_{\pm3}2_{\pm4}3_{0}4_{\pm3}5^2_{0}7_{\pm4}, 1_{\pm3}2_{\pm4}3^2_{0}4_{\pm3}5^2_{0}7_{\pm4}, 1_{\pm3}2_{\pm4}3_{0}4_{\pm3}5^2_{0}6_{\pm3}7_{\pm4},   \\   
& \quad 1_{\pm3}2_{\pm4}3^2_{0}4^2_{\pm3}5^2_{0}7_{\pm4}, 1_{\pm3}2_{\pm4}3^2_{0}4_{\pm3}5^2_{0}6_{\pm3}7_{\pm4}, 1_{\pm3}2^2_{\pm4}3^2_{0}4_{\pm3}5^2_{0}7_{\pm4}, 1_{\pm3}2_{\pm4}3^2_{0}4^2_{\pm3}5^2_{0}6_{\pm3}7_{\pm4}, 1_{\pm3}2^2_{\pm4}3^2_{0}4_{\pm3}5^2_{0}6_{\pm3}7_{\pm4},   \\  
& \quad 1_{\pm3}2_{\pm4}3^2_{0}4^2_{\pm3}5^3_{0}6_{\pm3}7_{\pm4}, 1_{\pm3}2^2_{\pm4}3^2_{0}4_{\pm3}5^3_{0}6_{\pm3}7_{\pm4}, 1_{\pm3}2^2_{\pm4}3^2_{0}4^2_{\pm3}5^3_{0}6_{\pm3}7_{\pm4}, 1_{\pm3}2^2_{\pm4}3^3_{0}4^2_{\pm3}5^3_{0}6_{\pm3}7_{\pm4}, 1^2_{\pm3}2^2_{\pm4}3^3_{0}4^2_{\pm3}5^3_{0}6_{\pm3}7_{\pm4}, \\
& (4) 1_{\pm3}2_{\pm4}3_{0}5_{\pm4}7_{0}, 1_{\pm3}2_{\pm4}3_{0}4_{\pm1}5_{\pm4}7_{0}, 1_{\pm3}2_{\pm4}3^2_{0}5_{\pm4}7_{0}, 1_{\pm3}2_{\pm4}3_{0}4_{\pm1}5^2_{\pm4}7_{0}, 1_{\pm3}2_{\pm4}3^2_{0}5^2_{\pm4}7_{0}, 1_{\pm3}2_{\pm4}3_{0}4_{\pm1}5^2_{\pm4}6_{\pm1}7_{0},   \\  
& \quad 1_{\pm3}2_{\pm4}3^2_{0}4_{\pm1}5^2_{\pm4}7_{0}, 1_{\pm3}2_{\pm4}3^2_{0}5^2_{\pm4}6_{\pm1}7_{0}, 1_{\pm3}2^2_{\pm4}3^2_{0}4_{\pm1}5^2_{\pm4}7_{0}, 1_{\pm3}2_{\pm4}3^2_{0}4_{\pm1}5^2_{\pm4}6_{\pm1}7_{0}, 1_{\pm3}2^2_{\pm4}3^2_{0}4_{\pm1}5^2_{\pm4}6_{\pm1}7_{0},   \\ 
& \quad 1_{\pm3}2_{\pm4}3^2_{0}4_{\pm1}5^3_{\pm4}6_{\pm1}7_{0}, 1_{\pm3}2^2_{\pm4}3^2_{0}4_{\pm1}5^3_{\pm4}6_{\pm1}7_{0}, 1_{\pm3}2^2_{\pm4}3^2_{0}4^2_{\pm1}5^3_{\pm4}6_{\pm1}7_{0}, 1_{\pm3}2^2_{\pm4}3^3_{0}4_{\pm1}5^3_{\pm4}6_{\pm1}7_{0}, 1^2_{\pm3}2^2_{\pm4}3^3_{0}4_{\pm1}5^3_{\pm4}6_{\pm1}7_{0}, \\
& (5) 1_{\pm5}2_{0}6_{\pm5}7_{\pm2}, 1_{\pm5}2_{0}4_{\pm1}6_{\pm5}7_{\pm2},  1_{\pm5}2_{0}3_{\pm4}4_{\pm1}6_{\pm5}7_{\pm2}, 1_{\pm5}2_{0}4_{\pm1}5_{\pm4}6_{\pm5}7_{\pm2},  1_{\pm5}2_{0}4_{\pm1}6^2_{\pm5}7_{\pm2}, 1_{\pm5}2_{0}3_{\pm4}4_{\pm1}6^2_{\pm5}7_{\pm2}, \\
& \quad 1_{\pm5} 2_{0} 3_{\pm4} 4_{\pm1} 5_{\pm4} 6_{\pm5} 7_{\pm2}, 1_{\pm5} 2_{0} 3_{\pm4} 4^2_{\pm1} 5_{\pm4} 6_{\pm5} 7_{\pm2}, 1_{\pm5} 2^2_{0} 3_{\pm4}4_{\pm1} 5_{\pm4} 6_{\pm5} 7_{\pm2}, 1_{\pm5} 2_{0} 3_{\pm4} 4^2_{\pm1} 6^2_{\pm5} 7_{\pm2}, 1_{\pm5} 2_{0} 3_{\pm4} 4^2_{\pm1} 5_{\pm4} 6^2_{\pm5} 7_{\pm2},\\
& \quad  1_{\pm5}2^2_{0}3_{\pm4}4_{\pm1}6^2_{\pm5}7_{\pm2}, 1_{\pm5}2^2_{0}3_{\pm4}4_{\pm1}5_{\pm4}6^2_{\pm5}7_{\pm2}, 1_{\pm5}2^2_{0}3_{\pm4}4^2_{\pm1}5_{\pm4}6^2_{\pm5}7_{\pm2}, 1_{\pm5}2^2_{0}3^2_{\pm4}4^2_{\pm1}5_{\pm4}6^2_{\pm5}7_{\pm2}, 1^2_{\pm5}2^2_{0}3_{\pm4}4^2_{\pm1}5_{\pm4}6^2_{\pm5}7_{\pm2},
\\
& (6) 1_{\pm5}2_{0}7_{\pm6}, 1_{\pm5}2_{0}4_{\pm1}7_{\pm6},   1_{\pm5}2_{0}3_{\pm4}4_{\pm1}7_{\pm6}, 1_{\pm5}2_{0}4_{\pm1}5_{\pm4}7_{\pm6}, 1_{\pm5}2_{0}4_{\pm1}6_{\pm5}7_{\pm6},  1_{\pm5}2_{0}3_{\pm4}4_{\pm1}5_{\pm4}7_{\pm6},  1_{\pm5}2_{0}3_{\pm4}4_{\pm1}6_{\pm5}7_{\pm6}, \\ 
& \quad 1_{\pm5}2_{0}3_{\pm4}4^2_{\pm1}5_{\pm4}7_{\pm6},  1_{\pm5}2^2_{0}3_{\pm4}4_{\pm1}5_{\pm4}7_{\pm6},  1_{\pm5}2_{0}3_{\pm4}4^2_{\pm1}6_{\pm5}7_{\pm6}, 1_{\pm5}2^2_{0}3_{\pm4}4_{\pm1}6_{\pm5}7_{\pm6},  1_{\pm5}2_{0}3_{\pm4}4^2_{\pm1}5_{\pm4}6_{\pm5}7_{\pm6},  \\  
& \quad 1_{\pm5}2^2_{0}3_{\pm4}4_{\pm1}5_{\pm4}6_{\pm5}7_{\pm6}, 1_{\pm5}2^2_{0}3_{\pm4}4^2_{\pm1}5_{\pm4}6_{\pm5}7_{\pm6},  1_{\pm5}2^2_{0}3^2_{\pm4}4^2_{\pm1}5_{\pm4}6_{\pm5}7_{\pm6}, 1^2_{\pm5}2^2_{0}3_{\pm4}4^2_{\pm1}5_{\pm4}6_{\pm5}7_{\pm6}, \\
& (7) 1_{\pm5}2_{\pm4}5_{0}7_{\pm4},  1_{\pm5}2_{\pm4}4_{\pm1}5_{0}7_{\pm4},  1_{\pm5}2_{\pm4}3_{\pm4}4_{\pm1}5_{0}7_{\pm4},  1_{\pm5}2_{\pm4}5^2_{0}7_{\pm4}, 1_{\pm5}2_{\pm4}3_{\pm4}5^2_{0}7_{\pm4},   1_{\pm5}2_{\pm4}3_{\pm4}4_{\pm1}5^2_{0}7_{\pm4}, \\  
& \quad  1_{\pm5}2_{\pm4}5^2_{0}6_{\pm3}7_{\pm4}, 1_{\pm5}2_{\pm4}3_{\pm4}5^2_{0}6_{\pm3}7_{\pm4}, 1_{\pm5}2^2_{\pm4}3_{\pm4}4_{\pm1}5^2_{0}7_{\pm4}, 1_{\pm5}2_{\pm4}3_{\pm4}4_{\pm1}5^2_{0}6_{\pm3}7_{\pm4}, 1_{\pm5}2^2_{\pm4}3_{\pm4}4_{\pm1}5^2_{0}6_{\pm3}7_{\pm4},   \\
& \quad 1_{\pm5}2_{\pm4}3_{\pm4}5^3_{0}6_{\pm3}7_{\pm4}, 1_{\pm5}2^2_{\pm4}3_{\pm4}5^3_{0}6_{\pm3}7_{\pm4}, 1_{\pm5}2^2_{\pm4}3_{\pm4}4_{\pm1}5^3_{0}6_{\pm3}7_{\pm4}, 1_{\pm5}2^2_{\pm4}3^2_{\pm4}4_{\pm1}5^3_{0}6_{\pm3}7_{\pm4}, 1^2_{\pm5}2^2_{\pm4}3_{\pm4}4_{\pm1}5^3_{0}6_{\pm3}7_{\pm4}, \\
& (8) 1_{\pm5}2_{0}5_{\pm4}7_{0},  1_{\pm5}2_{0}4_{\pm1}5_{\pm4}7_{0},  1_{\pm5}2_{0}3_{\pm4}4_{\pm1}5_{\pm4}7_{0}, 1_{\pm5}2_{0}4_{\pm1}5^2_{\pm4}7_{0},  1_{\pm5}2_{0}3_{\pm4}4_{\pm1}5^2_{\pm4}7_{0},  1_{\pm5}2_{0}4_{\pm1}5^2_{\pm4}6_{\pm1}7_{0},   \\ 
& \quad 1_{\pm5}2_{0}3_{\pm4}4^2_{\pm1}5^2_{\pm4}7_{0}, 1_{\pm5}2_{0}3_{\pm4}4_{\pm1}5^2_{\pm4}6_{\pm1}7_{0},  1_{\pm5}2^2_{0}3_{\pm4}4_{\pm1}5^2_{\pm4}7_{0}, 1_{\pm5}2_{0}3_{\pm4}4^2_{\pm1}5^2_{\pm4}6_{\pm1}7_{0}, 1_{\pm5}2^2_{0}3_{\pm4}4_{\pm1}5^2_{\pm4}6_{\pm1}7_{0},   \\
& \quad 1_{\pm5}2_{0}3_{\pm4}4^2_{\pm1}5^3_{\pm4}6_{\pm1}7_{0},  1_{\pm5}2^2_{0}3_{\pm4}4_{\pm1}5^3_{\pm4}6_{\pm1}7_{0}, 1_{\pm5}2^2_{0}3_{\pm4}4^2_{\pm1}5^3_{\pm4}6_{\pm1}7_{0}, 1_{\pm5}2^2_{0}3^2_{\pm4}4^2_{\pm1}5^3_{\pm4}6_{\pm1}7_{0}, 1^2_{\pm5}2^2_{0}3_{\pm4}4^2_{\pm1}5^3_{\pm4}6_{\pm1}7_{0}, \\
& (9) 1_{\pm7}2_{\pm2}7_{0},   1_{\pm7}2_{\pm2}4_{\pm5}7_{0},   1_{\pm7}2_{\pm2}3_{\pm6}7_{0},  1_{\pm7}2_{\pm2}4_{\pm5}5_{\pm2}7_{0}, 1_{\pm7}2_{\pm2}3_{\pm6}5_{\pm2}7_{0}, 1_{\pm7}2_{\pm2}4_{\pm5}6_{\pm1}7_{0},  1_{\pm7}2_{\pm2}3_{\pm6}4_{\pm5}5_{\pm2}7_{0},   \\ 
& \quad 1_{\pm7}2_{\pm2}3_{\pm6}6_{\pm1}7_{0}, 1_{\pm7}2^2_{\pm2}3_{\pm6}4_{\pm5}5_{\pm2}7_{0},  1_{\pm7}2_{\pm2}3_{\pm6}4_{\pm5}6_{\pm1}7_{0},  1_{\pm7}2^2_{\pm2}3_{\pm6}4_{\pm5}6_{\pm1}7_{0},  1_{\pm7}2_{\pm2}3_{\pm6}4_{\pm5}5_{\pm2}6_{\pm1}7_{0},  \\ 
& \quad 1_{\pm7}2^2_{\pm2}3_{\pm6}4_{\pm5}5_{\pm2}6_{\pm1}7_{0}, 1_{\pm7}2^2_{\pm2}3_{\pm6}4^2_{\pm5}5_{\pm2}6_{\pm1}7_{0}, 1_{\pm7}2^2_{\pm2}3^2_{\pm6}4_{\pm5}5_{\pm2}6_{\pm1}7_{0}, 1^2_{\pm7}2^2_{\pm2}3_{\pm6}4_{\pm5}5_{\pm2}6_{\pm1}7_{0}, \\
& (10)  1_{\pm3}2_{0}3_{0}7_{\pm6}, 1_{\pm3}2_{0}3_{0}4_{\pm3}7_{\pm6}, 1_{\pm3}2_{0}3^2_{0}4_{\pm3}7_{\pm6}, 1_{\pm3}2_{0}3_{0}5_{\pm4}7_{\pm6}, 1_{\pm3}2_{0}3^2_{0}5_{\pm4}7_{\pm6}, 1_{\pm3}2_{0}3_{0}6_{\pm5}7_{\pm6},  1_{\pm3}2_{0}3^2_{0}6_{\pm5}7_{\pm6},  \\ 
& \quad  1_{\pm3}2_{0}3^2_{0}4_{\pm3}5_{\pm4}7_{\pm6},  1_{\pm3}2^2_{0}3^2_{0}4_{\pm3}5_{\pm4}7_{\pm6}, 1_{\pm3}2_{0}3^2_{0}4_{\pm3}6_{\pm5}7_{\pm6}, 1_{\pm3}2^2_{0}3^2_{0}4_{\pm3}6_{\pm5}7_{\pm6},  1_{\pm3}2_{0}3^2_{0}5_{\pm4}6_{\pm5}7_{\pm6},  \\ 
& \quad 1_{\pm3}2^2_{0}3^2_{0}5_{\pm4}6_{\pm5}7_{\pm6}, 1_{\pm3}2^2_{0}3^2_{0}4_{\pm3}5_{\pm4}6_{\pm5}7_{\pm6}, 1_{\pm3}2^2_{0}3^3_{0}4_{\pm3}5_{\pm4}6_{\pm5}7_{\pm6}, 1^2_{\pm3}2^2_{0}3^3_{0}4_{\pm3}5_{\pm4}6_{\pm5}7_{\pm6}, \\
& (11) 1_{\pm3}2_{0}3_{0}5_{\pm4}7_{0}, 1_{\pm3}2_{0}3_{0}4_{\pm3}5_{\pm4}7_{0}, 1_{\pm3}2_{0}3^2_{0}4_{\pm3}5_{\pm4}7_{0}, 1_{\pm3}2_{0}3_{0}5^2_{\pm4}7_{0}, 1_{\pm3}2_{0}3^2_{0}5^2_{\pm4}7_{0}, 1_{\pm3}2_{0}3_{0}5^2_{\pm4}6_{\pm1}7_{0},   \\ 
& \quad 1_{\pm3}2_{0}3^2_{0}4_{\pm3}5^2_{\pm4}7_{0}, 1_{\pm3}2_{0}3^2_{0}5^2_{\pm4}6_{\pm1}7_{0}, 1_{\pm3}2^2_{0}3^2_{0}4_{\pm3}5^2_{\pm4}7_{0}, 1_{\pm3}2_{0}3^2_{0}4_{\pm3}5^2_{\pm4}6_{\pm1}7_{0}, 1_{\pm3}2^2_{0}3^2_{0}4_{\pm3}5^2_{\pm4}6_{\pm1}7_{0},   \\ 
& \quad 1_{\pm3}2_{0}3^2_{0}5^3_{\pm4}6_{\pm1}7_{0}, 1_{\pm3}2^2_{0}3^2_{0}5^3_{\pm4}6_{\pm1}7_{0},  1_{\pm3}2^2_{0}3^2_{0}4_{\pm3}5^3_{\pm4}6_{\pm1}7_{0},  1_{\pm3}2^2_{0}3^3_{0}4_{\pm3}5^3_{\pm4}6_{\pm1}7_{0}, 1^2_{\pm3}2^2_{0}3^3_{0}4_{\pm3}5^3_{\pm4}6_{\pm1}7_{0},
\end{align*}
\end{gather}

\begin{gather}
\begin{align*}
& (12)  1_{\pm3}2_{0}3_{0}6_{\pm5}7_{\pm2},   1_{\pm3}2_{0}3_{0}4_{\pm3}6_{\pm5}7_{\pm2},  1_{\pm3}2_{0}3^2_{0}4_{\pm3}6_{\pm5}7_{\pm2}, 1_{\pm3}2_{0}3_{0}5_{\pm4}6_{\pm5}7_{\pm2}, 1_{\pm3}2_{0}3^2_{0}5_{\pm4}6_{\pm5}7_{\pm2}, 1_{\pm3}2_{0}3_{0}6^2_{\pm5}7_{\pm2},   \\
& \quad 1_{\pm3}2_{0}3^2_{0}4_{\pm3}5_{\pm4}6_{\pm5}7_{\pm2},  1_{\pm3}2_{0}3^2_{0}6^2_{\pm5}7_{\pm2}, 1_{\pm3}2^2_{0}3^2_{0}4_{\pm3}5_{\pm4}6_{\pm5}7_{\pm2}, 1_{\pm3}2_{0}3^2_{0}4_{\pm3}6^2_{\pm5}7_{\pm2},  1_{\pm3}2^2_{0}3^2_{0}4_{\pm3}6^2_{\pm5}7_{\pm2},   \\
& \quad 1_{\pm3}2_{0}3^2_{0}5_{\pm4}6^2_{\pm5}7_{\pm2}, 1_{\pm3}2^2_{0}3^2_{0}5_{\pm4}6^2_{\pm5}7_{\pm2}, 1_{\pm3}2^2_{0}3^2_{0}4_{\pm3}5_{\pm4}6^2_{\pm5}7_{\pm2}, 1_{\pm3}2^2_{0}3^3_{0}4_{\pm3}5_{\pm4}6^2_{\pm5}7_{\pm2},  1^2_{\pm3}2^2_{0}3^3_{0}4_{\pm3}5_{\pm4}6^2_{\pm5}7_{\pm2}, \\
& (13) 1_{\pm7}2_{\pm6}7_{0}, 1_{\pm7}2_{\pm6}4_{\pm3}7_{0},  1_{\pm7}2_{\pm6}3_{\pm6}4_{\pm3}7_{0}, 1_{\pm7}2_{\pm6}5_{\pm2}7_{0}, 1_{\pm7}2_{\pm6}3_{\pm6}5_{\pm2}7_{0}, 1_{\pm7}2_{\pm6}6_{\pm1}7_{0},  1_{\pm7}2_{\pm6}3_{\pm6}4_{\pm3}5_{\pm2}7_{0},  \\
& \quad 1_{\pm7}2_{\pm6}3_{\pm6}6_{\pm1}7_{0}, 1_{\pm7}2^2_{\pm6}3_{\pm6}4_{\pm3}5_{\pm2}7_{0}, 1_{\pm7}2_{\pm6}3_{\pm6}4_{\pm3}6_{\pm1}7_{0}, 1_{\pm7}2^2_{\pm6}3_{\pm6}4_{\pm3}6_{\pm1}7_{0}, 1_{\pm7}2_{\pm6}3_{\pm6}5_{\pm2}6_{\pm1}7_{0}, \\ 
& \quad 1_{\pm7}2^2_{\pm6}3_{\pm6}5_{\pm2}6_{\pm1}7_{0}, 1_{\pm7}2^2_{\pm6}3_{\pm6}4_{\pm3}5_{\pm2}6_{\pm1}7_{0},  1_{\pm7}2^2_{\pm6}3^2_{\pm6}4_{\pm3}5_{\pm2}6_{\pm1}7_{0},  1^2_{\pm7}2^2_{\pm6}3_{\pm6}4_{\pm3}5_{\pm2}6_{\pm1}7_{0}, \\
& (14) 1_{\pm3}2_{\pm4}3_{0}7_{\pm6}, 1_{\pm3}2_{\pm4}3_{0}4_{\pm1}7_{\pm6}, 1_{\pm3}2_{\pm4}3_{0}^27_{\pm6},  1_{\pm3}2_{\pm4}3_{0}4_{\pm1}5_{\pm4}7_{\pm6}, 1_{\pm3}2_{\pm4}3_{0}^25_{\pm4}7_{\pm6}, 1_{\pm3}2_{\pm4}3_{0}4_{\pm1}6_{\pm5}7_{\pm6},  1_{\pm3}2_{\pm4}3_{0}^26_{\pm5}7_{\pm6}, \\  
& \quad  1_{\pm3}2_{\pm4}3_{0}^24_{\pm1}5_{\pm4}7_{\pm6}, 1_{\pm3}2^2_{\pm4}3_{0}^24_{\pm1}5_{\pm4}7_{\pm6}, 1_{\pm3}2_{\pm4}3_{0}^24_{\pm1}6_{\pm5}7_{\pm6}, 1_{\pm3}2^2_{\pm4}3_{0}^24_{\pm1}6_{\pm5}7_{\pm6},  1_{\pm3}2_{\pm4}3_{0}^24_{\pm1}5_{\pm4}6_{\pm5}7_{\pm6},  \\
& \quad 1_{\pm3}2^2_{\pm4}3_{0}^24_{\pm1}5_{\pm4}6_{\pm5}7_{\pm6}, 1_{\pm3}2^2_{\pm4}3_{0}^24_{\pm1}^25_{\pm4}6_{\pm5}7_{\pm6}, 1_{\pm3}2^2_{\pm4}3_{0}^34_{\pm1}5_{\pm4}6_{\pm5}7_{\pm6}, 1^2_{\pm3}2^2_{\pm4}3_{0}^34_{\pm1}5_{\pm4}6_{\pm5}7_{\pm6}, \\
& (15)  1_{\pm5}2_{\pm6}3_{\pm2}7_{0}, 1_{\pm5}2_{\pm6}3_{\pm2}4_{\pm5}7_{0}, 1_{\pm5}2_{\pm6}3^2_{\pm2}4_{\pm5}7_{0}, 1_{\pm5}2_{\pm6}3_{\pm2}4_{\pm5}5_{\pm2}7_{0}, 1_{\pm5}2_{\pm6}3^2_{\pm2}4_{\pm5}5_{\pm2}7_{0}, 1_{\pm5}2_{\pm6}3_{\pm2}4_{\pm5}6_{\pm1}7_{0},   \\ 
& \quad 1_{\pm5}2_{\pm6}3^2_{\pm2}4^2_{\pm5}5_{\pm2}7_{0}, 1_{\pm5}2_{\pm6}3^2_{\pm2}4_{\pm5}6_{\pm1}7_{0}, 1_{\pm5}2^2_{\pm6}3^2_{\pm2}4_{\pm5}5_{\pm2}7_{0}, 1_{\pm5}2_{\pm6}3^2_{\pm2}4^2_{\pm5}6_{\pm1}7_{0}, 1_{\pm5}2^2_{\pm6}3^2_{\pm2}4_{\pm5}6_{\pm1}7_{0},   \\
& \quad 1_{\pm5}2_{\pm6}3^2_{\pm2}4^2_{\pm5}5_{\pm2}6_{\pm1}7_{0}, 1_{\pm5}2^2_{\pm6}3^2_{\pm2}4_{\pm5}5_{\pm2}6_{\pm1}7_{0}, 1_{\pm5}2^2_{\pm6}3^2_{\pm2}4^2_{\pm5}5_{\pm2}6_{\pm1}7_{0}, 1_{\pm5}2^2_{\pm6}3^3_{\pm2}4^2_{\pm5}5_{\pm2}6_{\pm1}7_{0}, 1^2_{\pm5}2^2_{\pm6}3^3_{\pm2}4^2_{\pm5}5_{\pm2}6_{\pm1}7_{0}, \\
& (16) 1_{0}2_{\pm5}5_{\pm5}6_{\pm2}7_{\pm5},  1_{0}2_{\pm5}4_{\pm2}5_{\pm5}6_{\pm2}7_{\pm2}, 1_{0}2_{\pm5}3_{\pm1}5_{\pm5}6_{\pm2}7_{\pm5}, 1_{0}2_{\pm5}4_{\pm2}5^2_{\pm5}6_{\pm2}7_{\pm5}, 1_{0}2_{\pm5}3_{\pm1}5^2_{\pm5}6_{\pm2}7_{\pm5}, 1_{0}2_{\pm5}4_{\pm2}5^2_{\pm5}6^2_{\pm2}7_{\pm5},   \\  
& \quad 1_{0}2_{\pm5}3_{\pm1}4_{\pm2}5^2_{\pm5}6_{\pm2}7_{\pm5}, 1_{0}2_{\pm5}3_{\pm1}5^2_{\pm5}6^2_{\pm2}7_{\pm5}, 1_{0}2^2_{\pm5}3_{\pm1}4_{\pm2}5^2_{\pm5}6_{\pm2}7_{\pm5}, 1_{0}2_{\pm5}3_{\pm1}4_{\pm2}5^2_{\pm5}6^2_{\pm2}7_{\pm5}, 1_{0}2^2_{\pm5}3_{\pm1}4_{\pm2}5^2_{\pm5}6^2_{\pm2}7_{\pm5},   \\
& \quad 1_{0}2_{\pm5}3_{\pm1}4_{\pm2}5^3_{\pm5}6^2_{\pm2}7_{\pm5}, 1_{0}2^2_{\pm5}3_{\pm1}4_{\pm2}5^3_{\pm5}6^2_{\pm2}7_{\pm5}, 1_{0}2^2_{\pm5}3_{\pm1}4^2_{\pm2}5^3_{\pm5}6^2_{\pm2}7_{\pm5}, 1_{0}2^2_{\pm5}3^2_{\pm1}4_{\pm2}5^3_{\pm5}6^2_{\pm2}7_{\pm5}, 1^2_{0}2^2_{\pm5}3_{\pm1}4_{\pm2}5^3_{\pm5}6^2_{\pm2}7_{\pm5}, \\
& (17) 1_{\pm3}2_{0}3_{0}5_{\pm4}6_{\pm1}7_{\pm4}, 1_{\pm3}2_{0}3_{0}4_{\pm3}5_{\pm4}6_{\pm1}7_{\pm4}, 1_{\pm3}2_{0}3^2_{0}4_{\pm3}5_{\pm4}6_{\pm1}7_{\pm4}, 1_{\pm3}2_{0}3_{0}5^2_{\pm4}6_{\pm1}7_{\pm4}, 1_{\pm3}2_{0}3^2_{0}5^2_{\pm4}6_{\pm1}7_{\pm4}, 1_{\pm3}2_{0}3_{0}5^2_{\pm4}6^2_{\pm1}7_{\pm4},   \\ 
& \quad 1_{\pm3}2_{0}3^2_{0}4_{\pm3}5^2_{\pm4}6_{\pm1}7_{\pm4}, 1_{\pm3}2_{0}3^2_{0}5^2_{\pm4}6^2_{\pm1}7_{\pm4},  1_{\pm3}2^2_{0}3^2_{0}4_{\pm3}5^2_{\pm4}6_{\pm1}7_{\pm4}, 1_{\pm3}2_{0}3^2_{0}4_{\pm3}5^2_{\pm4}6^2_{\pm1}7_{\pm4}, 1_{\pm3}2^2_{0}3^2_{0}4_{\pm3}5^2_{\pm4}6^2_{\pm1}7_{\pm4},   \\ 
& \quad 1_{\pm3}2_{0}3^2_{0}5^3_{\pm4}6^2_{\pm1}7_{\pm4}, 1_{\pm3}2^2_{0}3^2_{0}5^3_{\pm4}6^2_{\pm1}7_{\pm4}, 1_{\pm3}2^2_{0}3^2_{0}4_{\pm3}5^3_{\pm4}6^2_{\pm1}7_{\pm4}, 1_{\pm3}2^2_{0}3^3_{0}4_{\pm3}5^3_{\pm4}6^2_{\pm1}7_{\pm4}, 1^2_{\pm3}2^2_{0}3^3_{0}4_{\pm3}5^3_{\pm4}6^2_{\pm1}7_{\pm4}, \\
& (18) 1_{\pm6}2_{\pm1}6_{0}7_{\pm3}, 1_{\pm6}2_{\pm1}4_{\pm4}6_{0}7_{\pm3}, 1_{\pm6}2_{\pm1}3_{\pm5}6_{0}7_{\pm3}, 1_{\pm6}2_{\pm1}4_{\pm4}5_{\pm1}6_{0}7_{\pm3}, 1_{\pm6}2_{\pm1}3_{\pm5}5_{\pm1}6_{0}7_{\pm3}, 1_{\pm6}2_{\pm1}4_{\pm4}6^2_{0}7_{\pm3},   \\  
& \quad 1_{\pm6}2_{\pm1}3_{\pm5}4_{\pm4}5_{\pm1}6_{0}7_{\pm3}, 1_{\pm6}2_{\pm1}3_{\pm5}6^2_{0}7_{\pm3}, 1_{\pm6}2^2_{\pm1}3_{\pm5}4_{\pm4}5_{\pm1}6_{0}7_{\pm3}, 1_{\pm6}2_{\pm1}3_{\pm5}4_{\pm4}6^2_{0}7_{\pm3},  1_{\pm6}2^2_{\pm1}3_{\pm5}4_{\pm4}6^2_{0}7_{\pm3},   \\
& \quad 1_{\pm6}2_{\pm1}3_{\pm5}4_{\pm4}5_{\pm1}6^2_{0}7_{\pm3}, 1_{\pm6}2^2_{\pm1}3_{\pm5}4_{\pm4}5_{\pm1}6^2_{0}7_{\pm3}, 1_{\pm6}2^2_{\pm1}3_{\pm5}4^2_{\pm4}5_{\pm1}6^2_{0}7_{\pm3}, 1_{\pm6}2^2_{\pm1}3^2_{\pm5}4_{\pm4}5_{\pm1}6^2_{0}7_{\pm3}, 1^2_{\pm6}2^2_{\pm1}3_{\pm5}4_{\pm4}5_{\pm1}6^2_{0}7_{\pm3}, \\
& (19) 1_{\pm3}2_{\pm4}3_{0}5_{0}6_{\pm3}7_{0},  1_{\pm3}2_{\pm4}3_{0}4_{\pm3}5_{0}6_{\pm3}7_{0}, 1_{\pm3}2_{\pm4}3^2_{0}4_{\pm3}5_{0}6_{\pm3}7_{0},  1_{\pm3}2_{\pm4}3_{0}4_{\pm3}5^2_{0}6_{\pm3}7_{0},  1_{\pm3}2_{\pm4}3^2_{0}4_{\pm3}5^2_{0}6_{\pm3}7_{0},  1_{\pm3}2_{\pm4}3_{0}4_{\pm3}5^2_{0}6^2_{\pm3}7_{0},   \\
& \quad 1_{\pm3}2_{\pm4}3^2_{0}4^2_{\pm3}5^2_{0}6_{\pm3}7_{0}, 1_{\pm3}2_{\pm4}3^2_{0}4_{\pm3}5^2_{0}6^2_{\pm3}7_{0},  1_{\pm3}2^2_{\pm4}3^2_{0}4_{\pm3}5^2_{0}6_{\pm3}7_{0},  1_{\pm3}2_{\pm4}3^2_{0}4^2_{\pm3}5^2_{0}6^2_{\pm3}7_{0},   1_{\pm3}2^2_{\pm4}3^2_{0}4_{\pm3}5^2_{0}6^2_{\pm3}7_{0}, \\ 
& \quad 1_{\pm3}2_{\pm4}3^2_{0}4^2_{\pm3}5^3_{0}6^2_{\pm3}7_{0}, 1_{\pm3}2^2_{\pm4}3^2_{0}4_{\pm3}5^3_{0}6^2_{\pm3}7_{0}, 1_{\pm3}2^2_{\pm4}3^2_{0}4^2_{\pm3}5^3_{0}6^2_{\pm3}7_{0}, 1_{\pm3}2^2_{\pm4}3^3_{0}4^2_{\pm3}5^3_{0}6^2_{\pm3}7_{0}, 1^2_{\pm3}2^2_{\pm4}3^3_{0}4^2_{\pm3}5^3_{0}6^2_{\pm3}7_{0}, \\
& (20)  1_{\pm 3}2_03_04_{\pm3}5_06_{\pm3}7_0,  1_{\pm3}2_{0}3_{0}4_{\pm3}^{2} 5_{0} 6_{\pm3}7_{0}, 1_{\pm3}2_{0}3_{0}4_{\pm3}^{2}5_{0}^{2}6_{\pm3}7_{0}, 1_{\pm3}2_{0}3_{0}^{2}4_{\pm3}^{2}5_{0}6_{\pm3}7_{0}, 1_{\pm3} 2_{0} 3_{0}^{2}4_{\pm3}^{2}5_{0}^{2}6_{\pm3}7_{0}, 1_{\pm3}2_{0}3_{0}4_{\pm3}^{2}5_{0}^{2}6_{\pm3}^{2}7_{0}, \\
& \quad 1_{\pm3} 2_{0} 3_{0}^{2} 4_{\pm3}^{3} 5_{0}^{2} 6_{\pm3} 7_{0}, 1_{\pm3}2_{0}3_{0}^{2}4_{\pm3}^{3}5_{0}^{2}6_{\pm3}^{2}7_{0}, 1_{\pm3} 2_{0}^{2} 3_{0}^{2} 4_{\pm3}^{3} 5_{0}^{2}  6_{\pm3} 7_{0}, 1_{\pm3}2_{0}3_{0}^{2}4_{\pm3}^{2} 5_{0}^{2} 6_{\pm3}^{2}7_{0},  1_{\pm3}2_{0} 3_{0}^{2}4_{\pm3}^{3}5_{0}^{3}6_{\pm3}^{2}7_{0}, \\
& \quad 1_{\pm3}2_{0}^{2}3_{0}^{2}4_{\pm3}^{3}5_{0}^{2}6_{\pm3}^{2}7_{0}, 1_{\pm3}2_{0}^{2}3_{0}^{2}4_{\pm3}^{3}5_{0}^{3}6_{\pm3}^{2}7_{0},  1_{\pm3}2_{0}^{2}3_{0}^{2}4_{\pm3}^{4}5_{0}^{3}6_{\pm3}^{2}7_{0}, 1_{\pm3}2_{0}^{2}3_{0}^{3}4_{\pm3}^{4}5_{0}^{3}6_{\pm3}^{2}7_{0}, 1_{\pm3}^{2}2_{0}^{2}3_{0}^{3}4_{\pm3}^{4}5_{0}^{3}6_{\pm3}^{2}7_{0}, \\
& (21) 1_{\pm4}2_{\pm3}4_{0}7_{\pm5}, 1_{\pm4}2_{\pm3}4^2_{0}7_{\pm5}, 1_{\pm4}2_{\pm3}3_{\pm3}4^2_{0}7_{\pm5}, 1_{\pm4}2_{\pm3}4^2_{0}5_{\pm3}7_{\pm5}, 1_{\pm4}2_{\pm3}3_{\pm3}4^2_{0}5_{\pm3}7_{\pm5}, 1_{\pm4}2_{\pm3}4^2_{0}6_{\pm4}7_{\pm5},   \\ 
& \quad 1_{\pm4}2_{\pm3}3_{\pm3}4^3_{0}5_{\pm3}7_{\pm5}, 1_{\pm4}2_{\pm3}3_{\pm3}4^2_{0}6_{\pm4}7_{\pm5}, 1_{\pm4}2^2_{\pm3}3_{\pm3}4^3_{0}5_{\pm3}7_{\pm5}, 1_{\pm4}2_{\pm3}3_{\pm3}4^3_{0}6_{\pm4}7_{\pm5}, 1_{\pm4}2^2_{\pm3}3_{\pm3}4^3_{0}6_{\pm4}7_{\pm5},   \\ 
& \quad 1_{\pm4}2_{\pm3}3_{\pm3}4^3_{0}5_{\pm3}6_{\pm4}7_{\pm5}, 1_{\pm4}2^2_{\pm3}3_{\pm3}4^3_{0}5_{\pm3}6_{\pm4}7_{\pm5}, 1_{\pm4}2^2_{\pm3}3_{\pm3}4^4_{0}5_{\pm3}6_{\pm4}7_{\pm5}, 1_{\pm4}2^2_{\pm3}3^2_{\pm3}4^4_{0}5_{\pm3}6_{\pm4}7_{\pm5}, 1^2_{\pm4}2^2_{\pm3}3_{\pm3}4^4_{0}5_{\pm3}6_{\pm4}7_{\pm5}, \\
& (22) 1_{\pm5}2_{\pm4}4_{\pm1}5_{\pm4}7_{0},  1_{\pm5}2_{\pm4}4^2_{\pm1}5_{\pm4}7_{0}, 1_{\pm5}2_{\pm4}3_{\pm4}4^2_{\pm1}5_{\pm4}7_{0},   1_{\pm5}2_{\pm4}4^2_{\pm1}5^2_{\pm4}7_{0}, 1_{\pm5}2_{\pm4}3_{\pm4}4^2_{\pm1}5^2_{\pm4}7_{0}, 1_{\pm5}2_{\pm4}4^2_{\pm1}5^2_{\pm4}6_{\pm1}7_{0},   \\
& \quad 1_{\pm5}2_{\pm4}3_{\pm4}4^3_{\pm1}5^2_{\pm4}7_{0},  1_{\pm5}2_{\pm4}3_{\pm4}4^2_{\pm1}5^2_{\pm4}6_{\pm1}7_{0},  1_{\pm5}2^2_{\pm4}3_{\pm4}4^3_{\pm1}5^2_{\pm4}7_{0}, 1_{\pm5}2_{\pm4}3_{\pm4}4^3_{\pm1}5^2_{\pm4}6_{\pm1}7_{0}, 1_{\pm5}2^2_{\pm4}3_{\pm4}4^3_{\pm1}5^2_{\pm4}6_{\pm1}7_{0},   \\  
& \quad 1_{\pm5}2_{\pm4}3_{\pm4}4^3_{\pm1}5^3_{\pm4}6_{\pm1}7_{0},  1_{\pm5}2^2_{\pm4}3_{\pm4}4^3_{\pm1}5^3_{\pm4}6_{\pm1}7_{0},  1_{\pm5}2^2_{\pm4}3_{\pm4}4^4_{\pm1}5^3_{\pm4}6_{\pm1}7_{0}, 1_{\pm5}2^2_{\pm4}3^2_{\pm4}4^4_{\pm1}5^3_{\pm4}6_{\pm1}7_{0}, 1^2_{\pm5}2^2_{\pm4}3_{\pm4}4^4_{\pm1}5^3_{\pm4}6_{\pm1}7_{0}, \\
& (23) 1_{\pm3}2_{\pm4}3_06_{\pm5}7_{\pm2},  1_{\pm3}2_{\pm4}3_04_{\pm1}6_{\pm5}7_{\pm2}, 1_{\pm3}2_{\pm4}3^2_06_{\pm5}7_{\pm2}, 1_{\pm3}2_{\pm4}3_04_{\pm1}5_{\pm4}6_{\pm5}7_{\pm2}, 1_{\pm3}2_{\pm4}3^2_05_{\pm4}6_{\pm5}7_{\pm2}, 1_{\pm3}2_{\pm4}3_04_{\pm1}6^2_{\pm5}7_{\pm2},   \\
& \quad 1_{\pm3}2_{\pm4}3^2_04_{\pm1}5_{\pm4}6_{\pm5}7_{\pm2}, 1_{\pm3}2_{\pm4}3^2_06^2_{\pm5}7_{\pm2}, 1_{\pm3}2^2_{\pm4}3^2_04_{\pm1}5_{\pm4}6_{\pm5}7_{\pm2}, 1_{\pm3}2_{\pm4}3^2_04_{\pm1}6^2_{\pm5}7_{\pm2}, 1_{\pm3}2^2_{\pm4}3^2_04_{\pm1}6^2_{\pm5}7_{\pm2},   \\
& \quad 1_{\pm3}2_{\pm4}3^2_04_{\pm1}5_{\pm4}6^2_{\pm5}7_{\pm2}, 1_{\pm3}2^2_{\pm4}3^2_04_{\pm1}5_{\pm4}6^2_{\pm5}7_{\pm2}, 1_{\pm3}2^2_{\pm4}3^2_04^2_{\pm1}5_{\pm4}6^2_{\pm5}7_{\pm2}, 1_{\pm3}2^2_{\pm4}3^3_04_{\pm1}5_{\pm4}6^2_{\pm5}7_{\pm2}, 1^2_{\pm3}2^2_{\pm4}3^3_04_{\pm1}5_{\pm4}6^2_{\pm5}7_{\pm2}, \\
& (24) 1_{\pm6}2_{\pm5}6_{0}7_{\pm3}, 1_{\pm6}2_{\pm5}4_{\pm2}6_{0}7_{\pm3},  1_{\pm6}2_{\pm5}3_{\pm5}4_{\pm2}6_{0}7_{\pm3},  1_{\pm6}2_{\pm5}5_{\pm1}6_{0}7_{\pm3}, 1_{\pm6}2_{\pm5}3_{\pm5}5_{\pm1}6_{0}7_{\pm3}, 1_{\pm6}2_{\pm5}6^2_{0}7_{\pm3},   \\
& \quad 1_{\pm6}2_{\pm5}3_{\pm5}4_{\pm2}5_{\pm1}6_{0}7_{\pm3},  1_{\pm6}2_{\pm5}3_{\pm5}6^2_{0}7_{\pm3}, 1_{\pm6}2^2_{\pm5}3_{\pm5}4_{\pm2}5_{\pm1}6_{0}7_{\pm3}, 1_{\pm6}2_{\pm5}3_{\pm5}4_{\pm2}6^2_{0}7_{\pm3}, 1_{\pm6}2^2_{\pm5}3_{\pm5}4_{\pm2}6^2_{0}7_{\pm3},   \\ 
& \quad 1_{\pm6}2_{\pm5}3_{\pm5}5_{\pm1}6^2_{0}7_{\pm3}, 1_{\pm6}2^2_{\pm5}3_{\pm5}5_{\pm1}6^2_{0}7_{\pm3}, 1_{\pm6}2^2_{\pm5}3_{\pm5}4_{\pm2}5_{\pm1}6^2_{0}7_{\pm3}, 1_{\pm6}2^2_{\pm5}3^2_{\pm5}4_{\pm2}5_{\pm1}6^2_{0}7_{\pm3}, 1^2_{\pm6}2^2_{\pm5}3_{\pm5}4_{\pm2}5_{\pm1}6^2_{0}7_{\pm3},
\end{align*}
\end{gather}

\begin{gather}
\begin{align*}
& (25) 1_{\pm4}2_{\pm3}4_{\pm0}6_{\pm4}7_{\pm1}, 1_{\pm4}2_{\pm3}4^2_{\pm0}6_{\pm4}7_{\pm1}, 1_{\pm4}2_{\pm3}3_{\pm3}4^2_{\pm0}6_{\pm4}7_{\pm1}, 1_{\pm4}2_{\pm3}4^2_{\pm0}5_{\pm3}6_{\pm4}7_{\pm1},  1_{\pm4}2_{\pm3}3_{\pm3}4^2_{\pm0}5_{\pm3}6_{\pm4}7_{\pm1}, 1_{\pm4}2_{\pm3}4^2_{\pm0}6^2_{\pm4}7_{\pm1},   \\ 
& \quad 1_{\pm4}2_{\pm3}3_{\pm3}4^3_{\pm0}5_{\pm3}6_{\pm4}7_{\pm1}, 1_{\pm4}2_{\pm3}3_{\pm3}4^2_{\pm0}6^2_{\pm4}7_{\pm1}, 1_{\pm4}2^2_{\pm3}3_{\pm3}4^3_{\pm0}5_{\pm3}6_{\pm4}7_{\pm1}, 1_{\pm4}2_{\pm3}3_{\pm3}4^3_{\pm0}6^2_{\pm4}7_{\pm1},  1_{\pm4}2^2_{\pm3}3_{\pm3}4^3_{\pm0}6^2_{\pm4}7_{\pm1},   \\  
& \quad 1_{\pm4}2_{\pm3}3_{\pm3}4^3_{\pm0}6^2_{\pm4}7_{\pm1}, 1_{\pm4}2^2_{\pm3}3_{\pm3}4^3_{\pm0}5_{\pm3}6^2_{\pm4}7_{\pm1},  1_{\pm4}2^2_{\pm3}3_{\pm3}4^4_{\pm0}5_{\pm3}6^2_{\pm4}7_{\pm1}, 1_{\pm4}2^2_{\pm3}3^2_{\pm3}4^4_{\pm0}5_{\pm3}6^2_{\pm4}7_{\pm1}, 1^2_{\pm4}2^2_{\pm3}3_{\pm3}4^4_{\pm0}5_{\pm3}6^2_{\pm4}7_{\pm1}, \\
& (26)  1_{\pm5}2_{\pm2}3_{\pm2}4_{\pm5}7_{0},  1_{\pm5}2_{\pm2}3_{\pm2}4^2_{\pm5}7_{0}, 1_{\pm5}2_{\pm2}3^2_{\pm2}4^2_{\pm5}7_{0}, 1_{\pm5}2_{\pm2}3_{\pm2}4^2_{\pm5}5_{\pm2}7_{0}, 1_{\pm5}2_{\pm2}3^2_{\pm2}4^2_{\pm5}5_{\pm2}7_{0}, 1_{\pm5}2_{\pm2}3_{\pm2}4^2_{\pm5}6_{\pm1}7_{0},   \\
& \quad 1_{\pm5}2_{\pm2}3^2_{\pm2}4^3_{\pm5}5_{\pm2}7_{0},  1_{\pm5}2_{\pm2}3^2_{\pm2}4^2_{\pm5}6_{\pm1}7_{0}, 1_{\pm5}2^2_{\pm2}3^2_{\pm2}4^3_{\pm5}5_{\pm2}7_{0}, 1_{\pm5}2_{\pm2}3^2_{\pm2}4^3_{\pm5}6_{\pm1}7_{0}, 1_{\pm5}2^2_{\pm2}3^2_{\pm2}4^3_{\pm5}6_{\pm1}7_{0},   \\
& \quad 1_{\pm5}2_{\pm2}3^2_{\pm2}4^3_{\pm5}5_{\pm2}6_{\pm1}7_{0}, 1_{\pm5}2^2_{\pm2}3^2_{\pm2}4^3_{\pm5}5_{\pm2}6_{\pm1}7_{0}, 1_{\pm5}2^2_{\pm2}3^2_{\pm2}4^4_{\pm5}5_{\pm2}6_{\pm1}7_{0}, 1_{\pm5}2^2_{\pm2}3^3_{\pm2}4^4_{\pm5}5_{\pm2}6_{\pm1}7_{0}, 1^2_{\pm5}2^2_{\pm2}3^3_{\pm2}4^4_{\pm5}5_{\pm2}6_{\pm1}7_{0}, \\
& (27) 1_{\pm5}2_{0}5_{0}7_{\pm4},  1_{\pm5}2_{0}4_{\pm3}5_{0}7_{\pm4},  1_{\pm5}2_{0}3_{\pm4}5_{0}7_{\pm4}, 1_{\pm5}2_{0}4_{\pm3}5^2_{0}7_{\pm4}, 1_{\pm5}2_{0}3_{\pm4}5^2_{0}7_{\pm4}, 1_{\pm5}2_{0}4_{\pm3}5^2_{0}6_{\pm3}7_{\pm4}, 1_{\pm5}2_{0}3_{\pm4}4_{\pm3}5^2_{0}7_{\pm4},  \\ 
& \quad 1_{\pm5}2_{0}3_{\pm4}5^2_{0}6_{\pm3}7_{\pm4}, 1_{\pm5}2^2_{0}3_{\pm4}4_{\pm3}5^2_{0}7_{\pm4}, 1_{\pm5}2_{0}3_{\pm4}4_{\pm3}5^2_{0}6_{\pm3}7_{\pm4}, 1_{\pm5}2^2_{0}3_{\pm4}4_{\pm3}5^2_{0}6_{\pm3}7_{\pm4},  1_{\pm5}2_{0}3_{\pm4}4_{\pm3}5^3_{0}6_{\pm3}7_{\pm4},  \\ 
& \quad 1_{\pm5}2^2_{0}3_{\pm4}4_{\pm3}5^3_{0}6_{\pm3}7_{\pm4}, 1_{\pm5}2^2_{0}3_{\pm4}4^2_{\pm3}5^3_{0}6_{\pm3}7_{\pm4}, 1_{\pm5}2^2_{0}3^2_{\pm4}4_{\pm3}5^3_{0}6_{\pm3}7_{\pm4}, 1^2_{\pm5}2^2_{0}3_{\pm4}4_{\pm3}5^3_{0}6_{\pm3}7_{\pm4}, \\
& (28)  1_{\pm 3}2_{\pm4}3_05_{\pm4}6_{\pm1}7_{\pm4},  1_{\pm3}2_{\pm4}3_04_{\pm1}5_{\pm4}6_{\pm1}7_{\pm4}, 1_{\pm3}2_{\pm4}3^2_05_{\pm4}6_{\pm1}7_{\pm4}, 1_{\pm3}2_{\pm4}3_04_{\pm1}5^2_{\pm4}6_{\pm1}7_{\pm4}, 1_{\pm3}2_{\pm4}3^2_05^2_{\pm4}6_{\pm1}7_{\pm4}, 1_{\pm3}2_{\pm4}3_04_{\pm1}5^2_{\pm4}6^{2}_{\pm1}7_{\pm4}, \\
& \quad 1_{\pm3}2_{\pm4}3^2_04_{\pm1}5^2_{\pm4}6_{\pm1}7_{\pm4}, 1_{\pm3}2_{\pm4}3^2_05^2_{\pm4}6^{2}_{\pm1}7_{\pm4}, 1_{\pm3}2^2_{\pm4}3^2_04_{\pm1}5^2_{\pm4}6_{\pm1}7_{\pm4}, 1_{\pm3}2_{\pm4}3^2_04_{\pm1}5^2_{\pm4}6^{2}_{\pm1}7_{\pm4}, 1_{\pm3}2^2_{\pm4}3^2_04_{\pm1}5^2_{\pm4}6^{2}_{\pm1}7_{\pm4}, \\
& \quad 1_{\pm3}2_{\pm4}3^2_04_{\pm1}5^3_{\pm4}6^{2}_{\pm1}7_{\pm4}, 1_{\pm3}2^2_{\pm4}3^2_04_{\pm1}5^3_{\pm4}6^{2}_{\pm1}7_{\pm4}, 1_{\pm3}2^2_{\pm4}3^2_04^2_{\pm1}5^3_{\pm4}6^{2}_{\pm1}7_{\pm4}, 1_{\pm3}2^2_{\pm4}3^3_04_{\pm1}5^3_{\pm4}6^{2}_{\pm1}7_{\pm4}, 1_{\pm3}^{2}2^2_{\pm4}3^3_04_{\pm1}5^3_{\pm4}6^{2}_{\pm1}7_{\pm4}, \\
& (29) 1_{\pm4}2_{\pm5}3_{\pm1}6_{0}7_{\pm3}, 1_{\pm4}2_{\pm5}3_{\pm1}4_{\pm4}6_{0}7_{\pm3},  1_{\pm4}2_{\pm5}3^2_{\pm1}4_{\pm4}6_{0}7_{\pm3}, 1_{\pm4}2_{\pm5}3_{\pm1}4_{\pm4}5_{\pm1}6_{0}7_{\pm3}, 1_{\pm4}2_{\pm5}3^2_{\pm1}4_{\pm4}5_{\pm1}6_{0}7_{\pm3}, 1_{\pm4}2_{\pm5}3_{\pm1}4_{\pm4}6^2_{0}7_{\pm3},   \\  
& \quad 1_{\pm4}2_{\pm5}3^2_{\pm1}4^2_{\pm4}5_{\pm1}6_{0}7_{\pm3}, 1_{\pm4}2_{\pm5}3^2_{\pm1}4_{\pm4}6^2_{0}7_{\pm3},  1_{\pm4}2^2_{\pm5}3^2_{\pm1}4_{\pm4}5_{\pm1}6_{0}7_{\pm3}, 1_{\pm4}2_{\pm5}3^2_{\pm1}4^2_{\pm4}6^2_{0}7_{\pm3}, 1_{\pm4}2^2_{\pm5}3^2_{\pm1}4_{\pm4}6^2_{0}7_{\pm3},   \\ 
& \quad 1_{\pm4}2_{\pm5}3^2_{\pm1}4^2_{\pm4}5_{\pm1}6^2_{0}7_{\pm3}, 1_{\pm4}2^2_{\pm5}3^2_{\pm1}4_{\pm4}5_{\pm1}6^2_{0}7_{\pm3}, 1_{\pm4}2^2_{\pm5}3^2_{\pm1}4^2_{\pm4}5_{\pm1}6^2_{0}7_{\pm3}, 1_{\pm4}2^2_{\pm5}3^3_{\pm1}4^2_{\pm4}5_{\pm1}6^2_{0}7_{\pm3}, 1^2_{\pm4}2^2_{\pm5}3^3_{\pm1}4^2_{\pm4}5_{\pm1}6^2_{0}7_{\pm3}, \\
& (30) 1_{\pm5}2_{0}5_{\pm4}6_{\pm1}7_{\pm4},  1_{\pm5}2_{0}4_{\pm1}5_{\pm4}6_{\pm1}7_{\pm4}, 1_{\pm5}2_{0}3_{\pm4}4_{\pm1}5_{\pm4}6_{\pm1}7_{\pm4}, 1_{\pm5}2_{0}4_{\pm1}5^2_{\pm4}6_{\pm1}7_{\pm4}, 1_{\pm5}2_{0}3_{\pm4}4_{\pm1}5^2_{\pm4}6_{\pm1}7_{\pm4}, 1_{\pm5}2_{0}4_{\pm1}5^2_{\pm4}6^2_{\pm1}7_{\pm4},   \\
& \quad 1_{\pm5}2_{0}3_{\pm4}4^2_{\pm1}5^2_{\pm4}6_{\pm1}7_{\pm4}, 1_{\pm5}2_{0}3_{\pm4}4_{\pm1}5^2_{\pm4}6^2_{\pm1}7_{\pm4}, 1_{\pm5}2^2_{0}3_{\pm4}4_{\pm1}5^2_{\pm4}6_{\pm1}7_{\pm4}, 1_{\pm5}2_{0}3_{\pm4}4^2_{\pm1}5^2_{\pm4}6^2_{\pm1}7_{\pm4}, 1_{\pm5}2^2_{0}3_{\pm4}4_{\pm1}5^2_{\pm4}6^2_{\pm1}7_{\pm4},   \\ 
& \quad 1_{\pm5}2_{0}3_{\pm4}4^2_{\pm1}5^3_{\pm4}6^2_{\pm1}7_{\pm4}, 1_{\pm5}2^2_{0}3_{\pm4}4_{\pm1}5^3_{\pm4}6^2_{\pm1}7_{\pm4}, 1_{\pm5}2^2_{0}3_{\pm4}4^2_{\pm1}5^3_{\pm4}6^2_{\pm1}7_{\pm4}, 1_{\pm5}2^2_{0}3^2_{\pm4}4^2_{\pm1}5^3_{\pm4}6^2_{\pm1}7_{\pm4}, 1^2_{\pm5}2^2_{0}3_{\pm4}4^2_{\pm1}5^3_{\pm4}6^2_{\pm1}7_{\pm4},  \\
& (31) 1_{\pm4}2_{\pm3}4_{0}5_{\pm3}6_{0}7_{\pm3}, 1_{\pm4}2_{\pm3}4^2_{0}5_{\pm3}6_{0}7_{\pm3},  1_{\pm4}2_{\pm3}3_{\pm3}4^2_{0}5_{\pm3}6_{0}7_{\pm3}, 1_{\pm4}2_{\pm3}4^2_{0}5^2_{\pm3}6_{0}7_{\pm3}, 1_{\pm4}2_{\pm3}3_{\pm3}4^2_{0}5^2_{\pm3}6_{0}7_{\pm3}, 1_{\pm4}2_{\pm3}4^2_{0}5^2_{\pm3}6^2_{0}7_{\pm3},   \\ 
& \quad  1_{\pm4}2_{\pm3}3_{\pm3}4^3_{0}5^2_{\pm3}6_{0}7_{\pm3},  1_{\pm4}2_{\pm3}3_{\pm3}4^2_{0}5^2_{\pm3}6^2_{0}7_{\pm3}, 1_{\pm4}2^2_{\pm3}3_{\pm3}4^3_{0}5^2_{\pm3}6_{0}7_{\pm3}, 1_{\pm4}2_{\pm3}3_{\pm3}4^3_{0}5^2_{\pm3}6^2_{0}7_{\pm3}, 1_{\pm4}2^2_{\pm3}3_{\pm3}4^3_{0}5^2_{\pm3}6^2_{0}7_{\pm3},   \\ 
& \quad  1_{\pm4}2_{\pm3}3_{\pm3}4^3_{0}5^3_{\pm3}6^2_{0}7_{\pm3},  1_{\pm4}2^2_{\pm3}3_{\pm3}4^3_{0}5^3_{\pm3}6^2_{0}7_{\pm3}, 1_{\pm4}2^2_{\pm3}3_{\pm3}4^4_{0}5^3_{\pm3}6^2_{0}7_{\pm3}, 1_{\pm4}2^2_{\pm3}3^2_{\pm3}4^4_{0}5^3_{\pm3}6^2_{0}7_{\pm3}, 1^2_{\pm4}2^2_{\pm3}3_{\pm3}4^4_{0}5^3_{\pm3}6^2_{0}7_{\pm3}, \\
& (32)  1_{\pm4}2_{\pm1}3_{\pm1}4_{\pm4}6_{0}7_{\pm3},  1_{\pm4}2_{\pm1}3_{\pm1}4^2_{\pm4}6_{0}7_{\pm3}, 1_{\pm4}2_{\pm1}3^2_{\pm1}4^2_{\pm4}6_{0}7_{\pm3},  1_{\pm4}2_{\pm1}3_{\pm1}4^2_{\pm4}5_{\pm1}6_{0}7_{\pm3},  1_{\pm4}2_{\pm1}3^2_{\pm1}4^2_{\pm4}5_{\pm1}6_{0}7_{\pm3}, 1_{\pm4}2_{\pm1}3_{\pm1}4^2_{\pm4}6^2_{0}7_{\pm3},   \\
& \quad 1_{\pm4}2_{\pm1}3^2_{\pm1}4^3_{\pm4}5_{\pm1}6_{0}7_{\pm3}, 1_{\pm4}2_{\pm1}3^2_{\pm1}4^2_{\pm4}6^2_{0}7_{\pm3},  1_{\pm4}2^2_{\pm1}3^2_{\pm1}4^3_{\pm4}5_{\pm1}6_{0}7_{\pm3}, 1_{\pm4}2_{\pm1}3^2_{\pm1}4^3_{\pm4}6^2_{0}7_{\pm3}, 1_{\pm4}2^2_{\pm1}3^2_{\pm1}4^3_{\pm4}6^2_{0}7_{\pm3},   \\
& \quad 1_{\pm4}2_{\pm1}3^2_{\pm1}4^3_{\pm4}5_{\pm1}6^2_{0}7_{\pm3},  1_{\pm4}2^2_{\pm1}3^2_{\pm1}4^3_{\pm4}5_{\pm1}6^2_{0}7_{\pm3}, 1_{\pm4}2^2_{\pm1}3^2_{\pm1}4^4_{\pm4}5_{\pm1}6^2_{0}7_{\pm3}, 1_{\pm4}2^2_{\pm1}3^3_{\pm1}4^4_{\pm4}5_{\pm1}6^2_{0}7_{\pm3},  1^2_{\pm4}2^2_{\pm1}3^3_{\pm1}4^4_{\pm4}5_{\pm1}6^2_{0}7_{\pm3}.
\end{align*}
\end{gather}

\clearpage

\section{Type $E_{8}$} \label{Hernandez-Leclerc modules of type E8}

The highest $l$-weight monomials of the Hernandez-Leclerc modules of type $E_8$ which are not of type $A$, $D$ or $E_7$ (up to spectral parameter shift) are as follows.  Each paragraph contains 44 different highest weights.
\begin{gather}
\begin{align*}
& (1) 1_0 2_{\pm5} 8_{\pm8}, 
1_0 3_{\pm1} 2_{\pm5} 8_{\pm8},  
1_0 4_{\pm2} 2_{\pm5} 8_{\pm8}, 
1_0 4_{\pm2} 2_{\pm5} 5_{\pm5} 8_{\pm8}, 
1_0 3_{\pm1} 2_{\pm5} 5_{\pm5} 8_{\pm8}, 
1_0 4_{\pm2} 2_{\pm5} 6_{\pm6} 8_{\pm8}, 
1_0 4_{\pm2} 2_{\pm5} 7_{\pm7} 8_{\pm8}, \\
& 1_0 3_{\pm1} 2_{\pm5} 7_{\pm7} 8_{\pm8}, 
1_0 3_{\pm1} 4_{\pm2} 2_{\pm5} 5_{\pm5} 8_{\pm8}, 
1_0 3_{\pm1} 4_{\pm2} 2_{\pm5} 7_{\pm7} 8_{\pm8}, 
1_0 3_{\pm1} 4_{\pm2} 2_{\pm5} 6_{\pm6} 8_{\pm8}, 
1_0 3_{\pm1} 2_{\pm5} 6_{\pm6} 8_{\pm8}, \\  
& 1_0 3_{\pm1} 4_{\pm2} 2_{\pm5} 6_{\pm6} 7_{\pm7} 8_{\pm8}, 
1_0 3_{\pm1} 4_{\pm2} 2_{\pm5}^{2} 5_{\pm5} 6_{\pm6} 8_{\pm8}, 
1_0 3_{\pm1} 4_{\pm2}^{2} 2_{\pm5}^{2} 5_{\pm5} 6_{\pm6} 8_{\pm8},
1_0 3_{\pm1}^{2} 4_{\pm2} 2_{\pm5}^{2} 5_{\pm5} 6_{\pm6} 8_{\pm8}, \\
& 1_0 3_{\pm1} 4_{\pm2} 2_{\pm5}^{2} 5_{\pm5} 7_{\pm7} 8_{\pm8}, 
1_0 3_{\pm1}^{2} 4_{\pm2} 2_{\pm5}^{2} 5_{\pm5} 7_{\pm7} 8_{\pm8},
1_0 3_{\pm1} 4_{\pm2}^{2} 2_{\pm5}^{2} 6_{\pm6} 7_{\pm7} 8_{\pm8},
1_0^{2} 3_{\pm1} 4_{\pm2} 2_{\pm5}^{2} 5_{\pm5} 7_{\pm7} 8_{\pm8}, \\
&  1_0 3_{\pm1}^{2} 4_{\pm2} 2_{\pm5}^{2} 6_{\pm6} 7_{\pm7} 8_{\pm8},
1_0^{2} 3_{\pm1} 4_{\pm2} 2_{\pm5}^{2} 5_{\pm5} 6_{\pm6} 8_{\pm8},
1_0 3_{\pm1} 4_{\pm2} 2_{\pm5}^{2} 6_{\pm6} 7_{\pm7} 8_{\pm8}, 
1_0^{2} 3_{\pm1} 4_{\pm2} 2_{\pm5}^{2} 6_{\pm6} 7_{\pm7} 8_{\pm8}, \\
& 1_0 3_{\pm1} 4_{\pm2} 2_{\pm5}^{2} 7_{\pm7} 8_{\pm8},
1_0 3_{\pm1} 4_{\pm2} 2_{\pm5}^{2} 6_{\pm6} 8_{\pm8},
1_0 3_{\pm1} 4_{\pm2} 2_{\pm5}^{2} 5_{\pm5} 8_{\pm8},
1_0 3_{\pm1} 4_{\pm2} 2_{\pm5} 5_{\pm5} 6_{\pm6} 8_{\pm8},
1_0 3_{\pm1} 4_{\pm2} 2_{\pm5} 5_{\pm5} 7_{\pm7} 8_{\pm8},\\
&  1_0 3_{\pm1} 4_{\pm2}^{2} 2_{\pm5}^{2} 5_{\pm5} 6_{\pm6} 7_{\pm7} 8_{\pm8}, 
1_0^{2} 3_{\pm1} 4_{\pm2}^{2} 2_{\pm5}^{2} 5_{\pm5} 6_{\pm6} 7_{\pm7} 8_{\pm8},
1_0^{2} 3_{\pm1}^{2} 4_{\pm2} 2_{\pm5}^{2} 5_{\pm5} 6_{\pm6} 7_{\pm7} 8_{\pm8},  
1_0 3_{\pm1}^{2} 4_{\pm2}^{2} 2_{\pm5}^{2} 5_{\pm5} 6_{\pm6} 7_{\pm7} 8_{\pm8}, \\ 
& 1_0^{2} 3_{\pm1}^{2} 4_{\pm2}^{2} 2_{\pm5}^{3} 5_{\pm5} 6_{\pm6} 7_{\pm7} 8_{\pm8}, 
1_0 3_{\pm1}^{2} 4_{\pm2} 2_{\pm5}^{2} 5_{\pm5} 6_{\pm6} 7_{\pm7} 8_{\pm8},  
1_0 3_{\pm1}^{2} 4_{\pm2}^{2} 2_{\pm5}^{3} 5_{\pm5} 6_{\pm6} 7_{\pm7} 8_{\pm8}, 
1_0^{2} 3_{\pm1} 4_{\pm2} 2_{\pm5}^{2} 5_{\pm5} 6_{\pm6} 7_{\pm7} 8_{\pm8}, \\
& 1_0^{2} 3_{\pm1} 4_{\pm2}^{2} 2_{\pm5}^{3} 5_{\pm5} 6_{\pm6} 7_{\pm7} 8_{\pm8},  
1_0^{2} 3_{\pm1}^{2} 4_{\pm2} 2_{\pm5}^{3} 5_{\pm5} 6_{\pm6} 7_{\pm7} 8_{\pm8},
1_0^{2} 3_{\pm1}^{2} 4_{\pm2}^{2} 2_{\pm5}^{3} 5_{\pm5}^{2}6_{\pm6} 7_{\pm7} 8_{\pm8}, 
1_0^{2} 3_{\pm1}^{2} 4_{\pm2}^{2} 2_{\pm5}^{3} 5_{\pm5} 6_{\pm6}^{2} 7_{\pm7} 8_{\pm8}, \\
& 1_0 3_{\pm1} 4_{\pm2}^{2} 2_{\pm5}^{2} 5_{\pm5} 7_{\pm7} 8_{\pm8},
1_0^{2} 3_{\pm1}^{2} 4_{\pm2}^{2} 2_{\pm5}^{3} 5_{\pm5} 6_{\pm6} 7_{\pm7}^{2} 8_{\pm8},
1_0^{2} 3_{\pm1}^{2} 4_{\pm2}^{2} 2_{\pm5}^{3} 5_{\pm5} 6_{\pm6} 7_{\pm7} 8_{\pm8}^{2}. \\
% \end{align*}
% \end{gather}
& \quad \\
%%
%%
%%Let $\xi=(0,-3,-1,-2,-3,-4,-5,-4)$. 
%%\begin{align*}
%%\xymatrix{
%%& & 2 &  &  \\
%%1  \ar[r] & 3 \ar[r]  & 4  \ar[u]  \ar[r] & 5  \ar[r] &  6 \ar[r] &  7 & 8   \ar[l] }
%%\end{align*}
%%The highest $l$-weight monomials of Hernandez-Leclerc modules of type $E_8$ that are not of type $A$, $D$ or $E_7$ are 
%\begin{gather}
%\begin{align*}
& (2) 1_0 8_{\pm4} 2_{\pm5} 7_{\pm7}, 
1_0 3_{\pm1} 8_{\pm4} 2_{\pm5} 5_{\pm5} 7_{\pm7}, 
1_{0} 3_{\pm1} 4_{\pm2} 8_{\pm4} 2_{\pm5}^{2} 5_{\pm5} 6_{\pm6} 7_{\pm7}, 
1_0^{2} 3_{\pm1} 4_{\pm2} 8_{\pm4} 2_{\pm5}^{2} 5_{\pm5} 6_{\pm6} 7_{\pm7}^{2},
1_{0} 3_{\pm1} 8_{\pm4} 2_{\pm5} 6_{\pm6} 7_{\pm7},  \\
& 1_0 3_{\pm1} 4_{\pm2} 8_{\pm4} 2_{\pm5}^{2} 5_{\pm5} 7_{\pm7}^{2},
1_0^{2} 3_{\pm1}^{2} 4_{\pm2} 8_{\pm4} 2_{\pm5}^{3} 5_{\pm5} 6_{\pm6} 7_{\pm7}^{2},
1_0 3_{\pm1}^{2} 4_{\pm2} 8_{\pm4} 2_{\pm5}^{2} 5_{\pm5} 6_{\pm6} 7_{\pm7}^{2}, 
1_0 3_{\pm1} 4_{\pm2} 8_{\pm4} 2_{\pm5}^{2} 6_{\pm6} 7_{\pm7}^{2}, \\
&  1_0 3_{\pm1} 8_{\pm4} 2_{\pm5} 7_{\pm7}, 
1_0 3_{\pm1} 8_{\pm4} 2_{\pm5} 7_{\pm7}^{2},
1_0 4_{\pm2} 8_{\pm4} 2_{\pm5} 5_{\pm5} 7_{\pm7},
1_0^{2} 3_{\pm1} 4_{\pm2} 8_{\pm4} 2_{\pm5}^{2} 5_{\pm5} 6_{\pm6} 7_{\pm7}, 
1_0^{2} 3_{\pm1}^{2} 4_{\pm2}^{2} 8_{\pm4} 2_{\pm5}^{3} 5_{\pm5}^{2} 6_{\pm6} 7_{\pm7}^{2}, \\
& 1_0 3_{\pm1} 4_{\pm2} 8_{\pm4} 2_{\pm5} 5_{\pm5} 6_{\pm6} 7_{\pm7},
1_0^{2} 3_{\pm1} 4_{\pm2}^{2} 8_{\pm4} 2_{\pm5}^{3} 5_{\pm5} 6_{\pm6} 7_{\pm7}^{2},
1_0^{2} 3_{\pm1} 4_{\pm2} 8_{\pm4} 2_{\pm5}^{2} 5_{\pm5} 7_{\pm7}^{2}, 
1_0 3_{\pm1} 4_{\pm2} 8_{\pm4} 2_{\pm5}^{2} 5_{\pm5} 7_{\pm7}, \\
& 1_0^{2} 3_{\pm1}^{2} 4_{\pm2}^{2} 8_{\pm4} 2_{\pm5}^{3} 5_{\pm5} 6_{\pm6}^{2} 7_{\pm7}^{2},
1_0^{2} 3_{\pm1}^{2} 4_{\pm2} 8_{\pm4} 2_{\pm5}^{2} 5_{\pm5} 6_{\pm6} 7_{\pm7}^{2},
1_0 3_{\pm1}^{2} 4_{\pm2} 8_{\pm4} 2_{\pm5}^{2} 5_{\pm5} 6_{\pm6} 7_{\pm7}, 
1_0 3_{\pm1} 4_{\pm2}^{2} 8_{\pm4} 2_{\pm5}^{2} 5_{\pm5} 6_{\pm6} 7_{\pm7}^{2}, \\
& 1_0^{2} 3_{\pm1}^{2} 4_{\pm2}^{2} 8_{\pm4} 2_{\pm5}^{3} 5_{\pm5} 6_{\pm6} 7_{\pm7}^{3},
1_0 3_{\pm1} 4_{\pm2} 8_{\pm4} 2_{\pm5} 5_{\pm5} 7_{\pm7}^{2},
1_0 3_{\pm1}^{2} 4_{\pm2}^{2} 8_{\pm4} 2_{\pm5}^{3} 5_{\pm5} 6_{\pm6} 7_{\pm7}^{2},
1_0 4_{\pm2} 8_{\pm4} 2_{\pm5} 6_{\pm6} 7_{\pm7}, \\
& 1_0^{2} 3_{\pm1} 4_{\pm2} 8_{\pm4} 2_{\pm5}^{2} 6_{\pm6} 7_{\pm7}^{2},
1_0^{2} 3_{\pm1}^{2} 4_{\pm2}^{2} 8_{\pm4}^{2} 2_{\pm5}^{3} 5_{\pm5} 6_{\pm6} 7_{\pm7}^{3},
1_0 3_{\pm1} 4_{\pm2} 8_{\pm4} 2_{\pm5}^{2} 6_{\pm6} 7_{\pm7}, 
1_0^{2} 3_{\pm1} 4_{\pm2}^{2} 8_{\pm4} 2_{\pm5}^{2} 5_{\pm5} 6_{\pm6} 7_{\pm7}^{2}, \\
& 1_0 4_{\pm2} 8_{\pm4} 2_{\pm5} 7_{\pm7}^{2},
1_0 3_{\pm1}^{2} 4_{\pm2} 8_{\pm4} 2_{\pm5}^{2} 5_{\pm5} 7_{\pm7}^{2}, 
1_0^{2} 3_{\pm1}^{2} 4_{\pm2}^{2} 8_{\pm4} 2_{\pm5}^{3} 5_{\pm5} 6_{\pm6} 7_{\pm7}^{2}, 
1_0 3_{\pm1}^{2} 4_{\pm2} 8_{\pm4} 2_{\pm5}^{2} 6_{\pm6} 7_{\pm7}^{2},
1_0 3_{\pm1} 4_{\pm2} 8_{\pm4} 2_{\pm5}^{2} 7_{\pm7}^{2}, \\
& 1_0 3_{\pm1} 4_{\pm2}^{2} 8_{\pm4} 2_{\pm5}^{2} 5_{\pm5} 6_{\pm6} 7_{\pm7}, 
1_0 3_{\pm1}^{2} 4_{\pm2}^{2} 8_{\pm4} 2_{\pm5}^{2} 5_{\pm5} 6_{\pm6} 7_{\pm7}^{2},
1_0 3_{\pm1} 4_{\pm2}^{2} 8_{\pm4} 2_{\pm5}^{2} 5_{\pm5} 7_{\pm7}^{2}, 
1_0 3_{\pm1} 4_{\pm2} 8_{\pm4} 2_{\pm5} 6_{\pm6} 7_{\pm7}^{2}, \\
& 1_0 3_{\pm1} 4_{\pm2} 8_{\pm4} 2_{\pm5} 5_{\pm5} 7_{\pm7}, 
1_0 3_{\pm1} 4_{\pm2}^{2} 8_{\pm4} 2_{\pm5}^{2} 6_{\pm6} 7_{\pm7}^{2},
1_0 3_{\pm1} 4_{\pm2} 8_{\pm4} 2_{\pm5} 6_{\pm6} 7_{\pm7},
1_0 3_{\pm1} 4_{\pm2} 8_{\pm4} 2_{\pm5} 7_{\pm7}^{2},
1_0 4_{\pm2} 8_{\pm4} 2_{\pm5} 7_{\pm7}. \\
%\end{align*}
%\end{gather}
& \quad \\
%%Let $\xi=(0,-3,-1,-2,-3,-4_{\pm3},-4)$. 
%%\begin{align*}
%%\xymatrix{
%%& & 2 &  &  \\
%%1  \ar[r] & 3 \ar[r]  & 4  \ar[u]  \ar[r] & 5  \ar[r] &  6 &  7  \ar[l]   \ar[r] & 8 }
%%\end{align*}
%%The highest $l$-weight monomials of Hernandez-Leclerc modules of type $E_8$ that are not of type $A$, $D$ or $E_7$ are 
%\begin{gather}
%\begin{align*}
& (3) 1_{0} 8_{\pm2}2_{\pm5} 6_{\pm6},
1_{0} 4_{\pm2} 8_{\pm2}2_{\pm5} 6_{\pm6},
1_{0}3_{\pm1} 8_{\pm2}2_{\pm5} 6_{\pm6},
1_{0} 3_{\pm1} 8_{\pm2}2_{\pm5} 5_{\pm5} 6_{\pm6},
1_{0} 3_{\pm1} 4_{\pm2} 8_{\pm2} 2_{\pm5}^{2} 5_{\pm5} 6_{\pm6}^{2}, 
1_{0} 4_{\pm2} 8_{\pm2}2_{\pm5} 5_{\pm5} 6_{\pm6}, \\
& 1_{0}3_{\pm1} 8_{\pm2}2_{\pm5} 6_{\pm6}^{2},
1_{0}^{2} 3_{\pm1} 4_{\pm2} 7_{\pm3} 8_{\pm2}2_{\pm5}^{2}5_{\pm5} 6_{\pm6}^{3},
1_{0}^{2} 3_{\pm1} 4_{\pm2} 8_{\pm2}2_{\pm5}^{2}5_{\pm5} 6_{\pm6}^{2}, 
1_{0} 3_{\pm1} 4_{\pm2} 8_{\pm2}2_{\pm5} 5_{\pm5} 6_{\pm6}^{2}, 
1_{0} 3_{\pm1} 4_{\pm2} 8_{\pm2}2_{\pm5} 6_{\pm6}^{2},  \\
&  1_{0} 4_{\pm2} 8_{\pm2}2_{\pm5} 6_{\pm6}^{2},
1_{0} 3_{\pm1} 4_{\pm2} 7_{\pm3} 8_{\pm2}2_{\pm5}^{2}5_{\pm5} 6_{\pm6}^{2},
1_{0}^{2} 3_{\pm1}^{2}4_{\pm2} 7_{\pm3} 8_{\pm2}2_{\pm5}^{3}5_{\pm5} 6_{\pm6}^{3},
1_{0} 3_{\pm1} 4_{\pm2} 8_{\pm2}2_{\pm5}^{2}5_{\pm5} 6_{\pm6},  
1_{0} 4_{\pm2} 7_{\pm3} 8_{\pm2}2_{\pm5} 6_{\pm6}^{2},  \\
& 1_{0} 3_{\pm1} 4_{\pm2} 8_{\pm2}2_{\pm5} 5_{\pm5} 6_{\pm6},
1_{0} 3_{\pm1}^{2}4_{\pm2} 7_{\pm3} 8_{\pm2}2_{\pm5}^{2}5_{\pm5} 6_{\pm6}^{3}, 
1_{0} 3_{\pm1} 4_{\pm2} 7_{\pm3} 8_{\pm2}2_{\pm5}^{2}6_{\pm6}^{3},
1_{0}^{2} 3_{\pm1}^{2}4_{\pm2}^{2}7_{\pm3} 8_{\pm2}2_{\pm5}^{3}5_{\pm5}^{2}6_{\pm6}^{3}, \\
& 1_{0}^{2} 3_{\pm1} 4_{\pm2}^{2}7_{\pm3} 8_{\pm2}2_{\pm5}^{3}5_{\pm5} 6_{\pm6}^{3},
1_{0} 3_{\pm1} 7_{\pm3} 8_{\pm2}2_{\pm5} 6_{\pm6}^{2},
1_{0}^{2} 3_{\pm1} 4_{\pm2} 7_{\pm3} 8_{\pm2}2_{\pm5}^{2}5_{\pm5} 6_{\pm6}^{2},
1_{0} 3_{\pm1}^{2}4_{\pm2} 8_{\pm2}2_{\pm5}^{2}5_{\pm5} 6_{\pm6}^{2}, \\
& 1_{0}^{2} 3_{\pm1}^{2}4_{\pm2}^{2}7_{\pm3} 8_{\pm2}2_{\pm5}^{3}5_{\pm5} 6_{\pm6}^{4},
1_{0} 3_{\pm1} 4_{\pm2} 8_{\pm2}2_{\pm5}^{2}6_{\pm6}^{2},
1_{0}^{2}3_{\pm1}^{2}4_{\pm2} 7_{\pm3} 8_{\pm2}2_{\pm5}^{2}5_{\pm5} 6_{\pm6}^{3},
1_{0} 3_{\pm1} 4_{\pm2}^{2}7_{\pm3} 8_{\pm2}2_{\pm5}^{2}5_{\pm5} 6_{\pm6}^{3}, \\
& 1_{0}^{2} 3_{\pm1}^{2}4_{\pm2}^{2}7_{\pm3} 8_{\pm2}^{2} 2_{\pm5}^{3} 5_{\pm5} 6_{\pm6}^{4},
1_{0} 3_{\pm1} 4_{\pm2} 7_{\pm3} 8_{\pm2}2_{\pm5} 5_{\pm5} 6_{\pm6}^{2},
1_{0} 3_{\pm1}^{2}4_{\pm2}^{2}7_{\pm3} 8_{\pm2}2_{\pm5}^{3}5_{\pm5} 6_{\pm6}^{3},
1_{0} 3_{\pm1} 4_{\pm2}^{2}8_{\pm2}2_{\pm5}^{2}5_{\pm5} 6_{\pm6}^{2}, \\
& 1_{0}^{2} 3_{\pm1} 4_{\pm2} 7_{\pm3} 8_{\pm2}2_{\pm5}^{2}6_{\pm6}^{3},
1_{0}^{2} 3_{\pm1}^{2}4_{\pm2}^{2}7_{\pm3}^{2}8_{\pm2}2_{\pm5}^{3}5_{\pm5} 6_{\pm6}^{4},
1_{0}^{2} 3_{\pm1} 4_{\pm2}^{2}7_{\pm3} 8_{\pm2}2_{\pm5}^{2}5_{\pm5} 6_{\pm6}^{3},
1_{0} 3_{\pm1} 4_{\pm2} 7_{\pm3} 8_{\pm2}2_{\pm5} 6_{\pm6}^{3}, \\
& 1_{0} 3_{\pm1}^{2}4_{\pm2} 7_{\pm3} 8_{\pm2}2_{\pm5}^{2}5_{\pm5} 6_{\pm6}^{2},
1_{0}^{2}3_{\pm1}^{2}4_{\pm2}^{2}7_{\pm3} 8_{\pm2}2_{\pm5}^{3}5_{\pm5} 6_{\pm6}^{3},
1_{0} 3_{\pm1}^{2}4_{\pm2} 7_{\pm3} 8_{\pm2}2_{\pm5}^{2}6_{\pm6}^{3}, 
1_{0} 3_{\pm1} 4_{\pm2} 7_{\pm3} 8_{\pm2}2_{\pm5}^{2} 6_{\pm6}^{2}, \\
& 1_{0} 3_{\pm1}^{2}4_{\pm2}^{2}7_{\pm3} 8_{\pm2}2_{\pm5}^{2}5_{\pm5} 6_{\pm6}^{3}, 
1_{0} 3_{\pm1} 4_{\pm2}^{2}7_{\pm3} 8_{\pm2}2_{\pm5}^{2}5_{\pm5} 6_{\pm6}^{2},
1_{0} 3_{\pm1} 4_{\pm2}^{2}7_{\pm3} 8_{\pm2}2_{\pm5}^{2}6_{\pm6}^{3},
1_{0} 3_{\pm1} 4_{\pm2} 7_{\pm3} 8_{\pm2}2_{\pm5} 6_{\pm6}^{2}.
\end{align*}
\end{gather}

%Let $\xi=(0,-3,-1,-2,-3,-4_{\pm3},-4)$. 
%\begin{align*}
%\xymatrix{
%& & 2 &  &  \\
%1  \ar[r] & 3 \ar[r]  & 4  \ar[u]  \ar[r] & 5  \ar[r] &  6 &  7  \ar[l] & 8   \ar[l] }
%\end{align*}
%The highest $l$-weight monomials of Hernandez-Leclerc modules of type $E_8$ that are not of type $A$, $D$ or $E_7$ are 

\begin{gather}
\begin{align*}
& (4)  1_0 7_{\pm3} 2_{\pm5} 6_{\pm6} 8_{\pm6},
1_0 3_{\pm1} 7_{\pm3} 2_{\pm5} 5_{\pm5} 6_{\pm6} 8_{\pm6},
1_0 3_{\pm1} 4_{\pm2} 7_{\pm3} 2_{\pm5}^{2}5_{\pm5} 6_{\pm6}^{2}8_{\pm6},
1_0 4_{\pm2} 7_{\pm3} 2_{\pm5} 5_{\pm5} 6_{\pm6} 8_{\pm6}, \\
& 1_0 3_{\pm1} 7_{\pm3} 2_{\pm5} 6_{\pm6}^{2}8_{\pm6},
1_0^{2}3_{\pm1} 4_{\pm2} 7_{\pm3}^{2}2_{\pm5}^{2}5_{\pm5} 6_{\pm6}^{3}8_{\pm6},
1_0^{2}3_{\pm1} 4_{\pm2} 7_{\pm3} 2_{\pm5}^{2}5_{\pm5} 6_{\pm6}^{2}8_{\pm6},
1_0 3_{\pm1} 4_{\pm2} 7_{\pm3} 2_{\pm5} 5_{\pm5} 6_{\pm6}^{2}8_{\pm6}, \\
& 1_0 3_{\pm1} 4_{\pm2} 7_{\pm3}^{2}2_{\pm5}^{2}5_{\pm5} 6_{\pm6}^{2}8_{\pm6},
1_0^{2}3_{\pm1}^{2}4_{\pm2} 7_{\pm3}^{2}2_{\pm5}^{3}5_{\pm5} 6_{\pm6}^{3}8_{\pm6},
1_0 3_{\pm1} 4_{\pm2} 7_{\pm3} 2_{\pm5}^{2}5_{\pm5} 6_{\pm6} 8_{\pm6},
1_0 3_{\pm1}^{2}4_{\pm2} 7_{\pm3}^{2}2_{\pm5}^{2}5_{\pm5} 6_{\pm6}^{3}8_{\pm6}, \\
& 1_0 4_{\pm2} 7_{\pm3} 2_{\pm5} 6_{\pm6}^{2}8_{\pm6},
1_0 3_{\pm1} 4_{\pm2} 7_{\pm3}^{2}2_{\pm5}^{2}6_{\pm6}^{3}8_{\pm6},
1_0^{2}3_{\pm1}^{2}4_{\pm2}^{2}7_{\pm3}^{2}2_{\pm5}^{3}5_{\pm5}^{2}6_{\pm6}^{3}8_{\pm6},
1_0^{2}3_{\pm1} 4_{\pm2}^{2}7_{\pm3}^{2}2_{\pm5}^{3}5_{\pm5} 6_{\pm6}^{3}8_{\pm6}, \\
& 1_0 3_{\pm1} 7_{\pm3}^{2}2_{\pm5} 6_{\pm6}^{2}8_{\pm6},
1_0^{2}3_{\pm1} 4_{\pm2} 7_{\pm3}^{2}2_{\pm5}^{2}5_{\pm5} 6_{\pm6}^{2}8_{\pm6},
1_0 3_{\pm1}^{2}4_{\pm2} 7_{\pm3} 2_{\pm5}^{2}5_{\pm5} 6_{\pm6}^{2}8_{\pm6},
1_0^{2}3_{\pm1}^{2}4_{\pm2}^{2}7_{\pm3}^{2}2_{\pm5}^{3}5_{\pm5} 6_{\pm6}^{4}8_{\pm6},\\
& 1_0 3_{\pm1} 4_{\pm2} 7_{\pm3} 2_{\pm5}^{2}6_{\pm6}^{2}8_{\pm6},
1_0^{2}3_{\pm1}^{2}4_{\pm2} 7_{\pm3}^{2}2_{\pm5}^{2}5_{\pm5} 6_{\pm6}^{3}8_{\pm6},
1_0 3_{\pm1} 4_{\pm2}^{2}7_{\pm3}^{2}2_{\pm5}^{2}5_{\pm5} 6_{\pm6}^{3}8_{\pm6},
1_0^{2}3_{\pm1}^{2}4_{\pm2}^{2}7_{\pm3}^{3}2_{\pm5}^{3}5_{\pm5} 6_{\pm6}^{4}8_{\pm6}^{2}, \\
& 1_0 3_{\pm1} 4_{\pm2} 7_{\pm3}^{2}2_{\pm5} 5_{\pm5} 6_{\pm6}^{2}8_{\pm6},
1_0 3_{\pm1}^{2}4_{\pm2}^{2}7_{\pm3}^{2}2_{\pm5}^{3}5_{\pm5} 6_{\pm6}^{3}8_{\pm6},
1_0 3_{\pm1} 4_{\pm2}^{2}7_{\pm3} 2_{\pm5}^{2}5_{\pm5} 6_{\pm6}^{2}8_{\pm6},
1_0^{2}3_{\pm1} 4_{\pm2} 7_{\pm3}^{2}2_{\pm5}^{2}6_{\pm6}^{3}8_{\pm6}, \\
& 1_0^{2}3_{\pm1}^{2}4_{\pm2}^{2}7_{\pm3}^{3}2_{\pm5}^{3}5_{\pm5} 6_{\pm6}^{4}8_{\pm6},
1_0^{2}3_{\pm1} 4_{\pm2}^{2}7_{\pm3}^{2}2_{\pm5}^{2}5_{\pm5} 6_{\pm6}^{3}8_{\pm6},
1_0 3_{\pm1} 4_{\pm2} 7_{\pm3}^{2}2_{\pm5} 6_{\pm6}^{3}8_{\pm6},
1_0 3_{\pm1} 7_{\pm3} 2_{\pm5} 6_{\pm6} 8_{\pm6}, \\
& 1_0 3_{\pm1}^{2}4_{\pm2} 7_{\pm3}^{2}2_{\pm5}^{2}5_{\pm5} 6_{\pm6}^{2}8_{\pm6},
1_0^{2}3_{\pm1}^{2}4_{\pm2}^{2}7_{\pm3}^{2}2_{\pm5}^{3}5_{\pm5} 6_{\pm6}^{3}8_{\pm6},
1_0 3_{\pm1} 4_{\pm2} 7_{\pm3} 2_{\pm5} 5_{\pm5} 6_{\pm6} 8_{\pm6},
1_0 3_{\pm1}^{2}4_{\pm2} 7_{\pm3}^{2}2_{\pm5}^{2}6_{\pm6}^{3}8_{\pm6}, \\
& 1_0 4_{\pm2} 7_{\pm3}^{2}2_{\pm5} 6_{\pm6}^{2}8_{\pm6},
1_0 3_{\pm1} 4_{\pm2} 7_{\pm3}^{2}2_{\pm5}^{2}6_{\pm6}^{2}8_{\pm6},
1_0 3_{\pm1}^{2}4_{\pm2}^{2}7_{\pm3}^{2}2_{\pm5}^{2}5_{\pm5} 6_{\pm6}^{3}8_{\pm6},
1_0 3_{\pm1} 4_{\pm2}^{2}7_{\pm3}^{2}2_{\pm5}^{2}5_{\pm5} 6_{\pm6}^{2}8_{\pm6}, \\
& 1_0 3_{\pm1} 4_{\pm2} 7_{\pm3} 2_{\pm5} 6_{\pm6}^{2}8_{\pm6},
1_0 3_{\pm1} 4_{\pm2}^{2}7_{\pm3}^{2}2_{\pm5}^{2}6_{\pm6}^{3}8_{\pm6},
1_0 4_{\pm2} 7_{\pm3} 2_{\pm5} 6_{\pm6} 8_{\pm6},
1_0 3_{\pm1} 4_{\pm2} 7_{\pm3}^{2}2_{\pm5} 6_{\pm6}^{2}8_{\pm6}. \\
%\end{align*}
%\end{gather}
%
& \quad \\
%Let $\xi=(0,-3,-1,-2,-3,-2,-3,-4)$. 
%\begin{align*}
%\xymatrix{
%& & 2 &  &  \\
%1  \ar[r] & 3 \ar[r]  & 4  \ar[u]  \ar[r] & 5 &  6  \ar[l]   \ar[r] & 7  \ar[r] & 8 }
%\end{align*}
%The highest $l$-weight monomials of Hernandez-Leclerc modules of type $E_8$ that are not of type $A$, $D$ or $E_7$ are 
%\begin{gather}
%\begin{align*}
& (5)  1_{0} 6_{\pm2} 2_{\pm5} 5_{\pm5} 8_{\pm6},  
1_{0} 4_{\pm2} 6_{\pm2} 2_{\pm5} 5_{\pm5} 8_{\pm6},  
1_{0} 3_{\pm1} 6_{\pm2} 2_{\pm5} 5_{\pm5}^{2} 8_{\pm6},  
1_{0} 3_{\pm1} 4_{\pm2} 6_{\pm2}^{2} 2_{\pm5}^{2} 5_{\pm5}^{3} 7_{\pm2} 8_{\pm6},   \\
& 1_{0} 4_{\pm2} 6_{\pm2} 2_{\pm5} 5_{\pm5}^{2} 8_{\pm6},  
1_{0} 3_{\pm1} 6_{\pm2}^{2} 2_{\pm5} 5_{\pm5}^{2} 7_{\pm2} 8_{\pm6},  
1_{0}^{2} 3_{\pm1} 4_{\pm2} 6_{\pm2}^{3} 2_{\pm5}^{2} 5_{\pm5}^{4} 7_{\pm2} 8_{\pm6},  
1_{0}^{2} 3_{\pm1} 4_{\pm2} 6_{\pm2}^{2} 2_{\pm5}^{2} 5_{\pm5}^{3} 7_{\pm2} 8_{\pm6},   \\
& 1_{0} 3_{\pm1} 4_{\pm2} 6_{\pm2}^{2} 2_{\pm5}^{2} 5_{\pm5}^{3} 8_{\pm6},  
1_{0}^{2} 3_{\pm1}^{2} 4_{\pm2} 6_{\pm2}^{3} 2_{\pm5}^{3} 5_{\pm5}^{4} 7_{\pm2} 8_{\pm6},  
1_{0} 3_{\pm1} 4_{\pm2} 6_{\pm2} 2_{\pm5}^{2} 5_{\pm5}^{2} 8_{\pm6},  
1_{0} 3_{\pm1} 4_{\pm2} 6_{\pm2}^{2} 2_{\pm5} 5_{\pm5}^{3} 7_{\pm2} 8_{\pm6},  \\
& 1_{0} 3_{\pm1}^{2} 4_{\pm2} 6_{\pm2}^{3} 2_{\pm5}^{2} 5_{\pm5}^{4} 7_{\pm2} 8_{\pm6},  
1_{0} 4_{\pm2} 6_{\pm2}^{2} 2_{\pm5} 5_{\pm5}^{2} 7_{\pm2} 8_{\pm6},  
1_{0}^{2} 3_{\pm1}^{2} 4_{\pm2}^{2} 6_{\pm2}^{3} 2_{\pm5}^{3} 5_{\pm5}^{5} 7_{\pm2} 8_{\pm6},  
1_{0} 3_{\pm1} 4_{\pm2} 6_{\pm2}^{3} 2_{\pm5}^{2} 5_{\pm5}^{3} 7_{\pm2} 8_{\pm6},  \\
& 1_{0}^{2} 3_{\pm1} 4_{\pm2}^{2} 6_{\pm2}^{3} 2_{\pm5}^{3} 5_{\pm5}^{4} 7_{\pm2} 8_{\pm6},  
1_{0} 3_{\pm1} 6_{\pm2}^{2} 2_{\pm5} 5_{\pm5}^{2} 8_{\pm6},  
1_{0} 3_{\pm1}^{2} 4_{\pm2} 6_{\pm2}^{2} 2_{\pm5}^{2} 5_{\pm5}^{3} 7_{\pm2} 8_{\pm6},  
1_{0}^{2} 3_{\pm1}^{2} 4_{\pm2}^{2} 6_{\pm2}^{4} 2_{\pm5}^{3} 5_{\pm5}^{5} 7_{\pm2}^{2} 8_{\pm6},  \\
& 1_{0} 3_{\pm1} 4_{\pm2} 6_{\pm2}^{2} 2_{\pm5}^{2} 5_{\pm5}^{2} 7_{\pm2} 8_{\pm6},  
1_{0}^{2} 3_{\pm1} 4_{\pm2} 6_{\pm2}^{2} 2_{\pm5}^{2} 5_{\pm5}^{3} 8_{\pm6},  
1_{0}^{2} 3_{\pm1}^{2} 4_{\pm2} 6_{\pm2}^{3} 2_{\pm5}^{2} 5_{\pm5}^{4} 7_{\pm2} 8_{\pm6},  
1_{0} 3_{\pm1} 4_{\pm2}^{2} 6_{\pm2}^{3} 2_{\pm5}^{2} 5_{\pm5}^{4} 7_{\pm2} 8_{\pm6},  \\
& 1_{0}^{2} 3_{\pm1}^{2} 4_{\pm2}^{2} 6_{\pm2}^{4} 2_{\pm5}^{3} 5_{\pm5}^{5} 7_{\pm2} 8_{\pm6}]^{2};
1_{0} 3_{\pm1} 4_{\pm2} 6_{\pm2}^{2} 2_{\pm5} 5_{\pm5}^{3} 8_{\pm6},  
1_{0} 3_{\pm1}^{2} 4_{\pm2}^{2} 6_{\pm2}^{3} 2_{\pm5}^{3} 5_{\pm5}^{4} 7_{\pm2} 8_{\pm6},  
1_{0}^{2} 3_{\pm1} 4_{\pm2} 6_{\pm2}^{3} 2_{\pm5}^{2} 5_{\pm5}^{3} 7_{\pm2} 8_{\pm6},   \\
& 1_{0}^{2} 3_{\pm1}^{2} 4_{\pm2}^{2} 6_{\pm2}^{4} 2_{\pm5}^{3} 5_{\pm5}^{5} 7_{\pm2} 8_{\pm6},  
1_{0} 3_{\pm1} 4_{\pm2}^{2} 6_{\pm2}^{2} 2_{\pm5}^{2} 5_{\pm5}^{3} 7_{\pm2} 8_{\pm6},  
1_{0}^{2} 3_{\pm1} 4_{\pm2}^{2} 6_{\pm2}^{3} 2_{\pm5}^{2} 5_{\pm5}^{4} 7_{\pm2} 8_{\pm6},  
1_{0} 3_{\pm1} 6_{\pm2} 2_{\pm5} 5_{\pm5} 8_{\pm6},   \\
& 1_{0} 3_{\pm1}^{2} 4_{\pm2} 6_{\pm2}^{2} 2_{\pm5}^{2} 5_{\pm5}^{3} 8_{\pm6},  
1_{0}^{2} 3_{\pm1}^{2} 4_{\pm2}^{2} 6_{\pm2}^{3} 2_{\pm5}^{3} 5_{\pm5}^{4} 7_{\pm2} 8_{\pm6},  
1_{0} 3_{\pm1} 4_{\pm2} 6_{\pm2} 2_{\pm5} 5_{\pm5}^{2} 8_{\pm6},  
1_{0} 3_{\pm1} 4_{\pm2} 6_{\pm2}^{3} 2_{\pm5} 5_{\pm5}^{3} 7_{\pm2} 8_{\pm6},  \\
& 1_{0} 3_{\pm1}^{2} 4_{\pm2} 6_{\pm2}^{3} 2_{\pm5}^{2} 5_{\pm5}^{3} 7_{\pm2} 8_{\pm6},  
1_{0} 4_{\pm2} 6_{\pm2}^{2} 2_{\pm5} 5_{\pm5}^{2} 8_{\pm6},  
1_{0} 3_{\pm1}^{2} 4_{\pm2}^{2} 6_{\pm2}^{3} 2_{\pm5}^{2} 5_{\pm5}^{4} 7_{\pm2} 8_{\pm6},  
1_{0} 3_{\pm1} 4_{\pm2} 6_{\pm2}^{2} 2_{\pm5}^{2} 5_{\pm5}^{2} 8_{\pm6},  \\
& 1_{0} 3_{\pm1} 4_{\pm2}^{2} 6_{\pm2}^{2} 2_{\pm5}^{2} 5_{\pm5}^{3} 8_{\pm6},  
1_{0} 3_{\pm1} 4_{\pm2} 6_{\pm2}^{2} 2_{\pm5} 5_{\pm5}^{2} 7_{\pm2} 8_{\pm6},  
1_{0} 3_{\pm1} 4_{\pm2}^{2} 6_{\pm2}^{3} 2_{\pm5}^{2} 5_{\pm5}^{3} 7_{\pm2} 8_{\pm6},  
1_{0} 3_{\pm1} 4_{\pm2} 6_{\pm2}^{2} 2_{\pm5} 5_{\pm5}^{2} 8_{\pm6}.  \\
%\end{align*}
%\end{gather}
%
& \quad \\
%Let $\xi=(0,-3,-1,-2,-3,-2,-3,-2)$. 
%\begin{align*}
%\xymatrix{
%& & 2 &  &  \\
%1  \ar[r] & 3 \ar[r]  & 4  \ar[u]  \ar[r] & 5 &  6  \ar[l]   \ar[r] & 7 & 8  \ar[l]}
%\end{align*}
%The highest $l$-weight monomials of Hernandez-Leclerc modules of type $E_8$ that are not of type $A$, $D$ or $E_7$ are 
%\begin{gather}
%\begin{align*}
& (6)  1_0 6_{\pm2} 8_{\pm2} 2_{\pm5} 5_{\pm5} 7_{\pm5},
1_0 3_{\pm1} 6_{\pm2} 8_{\pm2} 2_{\pm5} 5_{\pm5}^{2} 7_{\pm5},
1_0 3_{\pm1} 4_{\pm2} 6_{\pm2}^{2} 8_{\pm2} 2_{\pm5}^{2} 5_{\pm5}^{3}  7_{\pm5}^{2},
1_0 4_{\pm2} 6_{\pm2} 8_{\pm2} 2_{\pm5} 5_{\pm5}^{2} 7_{\pm5}, \\
& 1_0 3_{\pm1} 6_{\pm2}^{2} 8_{\pm2} 2_{\pm5} 5_{\pm5}^{2} 7_{\pm5}^{2},
1_0^{2} 3_{\pm1} 4_{\pm2} 6_{\pm2}^{3}  8_{\pm2} 2_{\pm5}^{2} 5_{\pm5}^{4} 7_{\pm5}^{2},
1_0^{2} 3_{\pm1} 4_{\pm2} 6_{\pm2}^{2} 8_{\pm2} 2_{\pm5}^{2} 5_{\pm5}^{3} 7_{\pm5}^{2},
1_0 3_{\pm1} 4_{\pm2} 6_{\pm2}^{2} 8_{\pm2} 2_{\pm5}^{2} 5_{\pm5}^{3} 7_{\pm5}, \\
& 1_0^{2} 3_{\pm1}^{2} 4_{\pm2} 6_{\pm2}^{3}  8_{\pm2} 2_{\pm5}^{3}  5_{\pm5}^{4}  7_{\pm5}^{2},
1_0 3_{\pm1} 4_{\pm2} 6_{\pm2} 8_{\pm2} 2_{\pm5}^{2} 5_{\pm5}^{2} 7_{\pm5},
1_0 3_{\pm1} 4_{\pm2} 6_{\pm2}^{2} 8_{\pm2} 2_{\pm5} 5_{\pm5}^{3}  7_{\pm5}^{2},
1_0 3_{\pm1}^{2} 4_{\pm2} 6_{\pm2}^{3}  8_{\pm2} 2_{\pm5}^{2} 5_{\pm5}^{4}  7_{\pm5}^{2}, \\
& 1_0 4_{\pm2} 6_{\pm2}^{2} 8_{\pm2} 2_{\pm5} 5_{\pm5}^{2} 7_{\pm5}^{2},
1_0^{2} 3_{\pm1}^{2} 4_{\pm2}^{2} 6_{\pm2}^{3}  8_{\pm2} 2_{\pm5}^{3}  5_{\pm5}^{5}  7_{\pm5}^{2},
1_0 3_{\pm1} 4_{\pm2} 6_{\pm2}^{3}  8_{\pm2} 2_{\pm5}^{2} 5_{\pm5}^{3}  7_{\pm5}^{2},
1_0^{2} 3_{\pm1} 4_{\pm2}^{2} 6_{\pm2}^{3}  8_{\pm2} 2_{\pm5}^{3}  5_{\pm5}^{4}  7_{\pm5}^{2}, \\
& 1_0 3_{\pm1} 6_{\pm2}^{2} 8_{\pm2} 2_{\pm5} 5_{\pm5}^{2} 7_{\pm5},
1_0 3_{\pm1}^{2} 4_{\pm2} 6_{\pm2}^{2} 8_{\pm2} 2_{\pm5}^{2} 5_{\pm5}^{3}  7_{\pm5}^{2},
1_0^{2} 3_{\pm1}^{2} 4_{\pm2}^{2} 6_{\pm2}^{4}  8_{\pm2} 2_{\pm5}^{3}  5_{\pm5}^{5}  7_{\pm5}]^{3},
1_0 3_{\pm1} 4_{\pm2} 6_{\pm2}^{2} 8_{\pm2} 2_{\pm5}^{2} 5_{\pm5}^{2} 7_{\pm5}^{2}, \\
& 1_0^{2} 3_{\pm1} 4_{\pm2} 6_{\pm2}^{2} 8_{\pm2} 2_{\pm5}^{2} 5_{\pm5}^{3}  7_{\pm5},
1_0^{2} 3_{\pm1}^{2} 4_{\pm2} 6_{\pm2}^{3}  8_{\pm2} 2_{\pm5}^{2} 5_{\pm5}^{4}  7_{\pm5}^{2},
1_0 3_{\pm1} 4_{\pm2}^{2} 6_{\pm2}^{3}  8_{\pm2} 2_{\pm5}^{2} 5_{\pm5}^{4}  7_{\pm5}^{2},
1_0^{2} 3_{\pm1}^{2} 4_{\pm2}^{2} 6_{\pm2}^{4}  8_{\pm2}^{2} 2_{\pm5}^{3}  5_{\pm5}^{5}  7_{\pm5}^{3}, \\
& 1_0 3_{\pm1} 4_{\pm2} 6_{\pm2}^{2} 8_{\pm2} 2_{\pm5} 5_{\pm5}^{3} 7_{\pm5},
1_0 3_{\pm1}^{2} 4_{\pm2}^{2} 6_{\pm2}^{3}  8_{\pm2} 2_{\pm5}^{3}  5_{\pm5}^{4}  7_{\pm5}^{2},
1_0^{2} 3_{\pm1} 4_{\pm2} 6_{\pm2}^{3}  8_{\pm2} 2_{\pm5}^{2} 5_{\pm5}^{3} 7_{\pm5}^{2},
1_0^{2} 3_{\pm1}^{2} 4_{\pm2}^{2} 6_{\pm2}^{4}  8_{\pm2} 2_{\pm5}^{3}  5_{\pm5}^{5}  7_{\pm5}^{2}, \\
& 1_0 3_{\pm1} 4_{\pm2}^{2} 6_{\pm2}^{2} 8_{\pm2} 2_{\pm5}^{2} 5_{\pm5}^{3}  7_{\pm5}^{2},
1_0^{2} 3_{\pm1} 4_{\pm2}^{2} 6_{\pm2}^{3}  8_{\pm2} 2_{\pm5}^{2} 5_{\pm5}^{4}  7_{\pm5}^{2},
1_0 3_{\pm1} 6_{\pm2} 8_{\pm2} 2_{\pm5} 5_{\pm5} 7_{\pm5},
1_0 3_{\pm1}^{2} 4_{\pm2} 6_{\pm2}^{2} 8_{\pm2} 2_{\pm5}^{2} 5_{\pm5}^{3} 7_{\pm5}, \\
& 1_0^{2} 3_{\pm1}^{2} 4_{\pm2}^{2} 6_{\pm2}^{3}  8_{\pm2} 2_{\pm5}^{3}  5_{\pm5}^{4}  7_{\pm5}^{2},
1_0 3_{\pm1} 4_{\pm2} 6_{\pm2} 8_{\pm2} 2_{\pm5} 5_{\pm5}^{2} 7_{\pm5},
1_0 3_{\pm1} 4_{\pm2} 6_{\pm2}^{3}  8_{\pm2} 2_{\pm5} 5_{\pm5}^{3}  7_{\pm5}^{2},
1_0 3_{\pm1}^{2} 4_{\pm2} 6_{\pm2}^{3}  8_{\pm2} 2_{\pm5}^{2} 5_{\pm5}^{3}  7_{\pm5}^{2}, \\
& 1_0 4_{\pm2} 6_{\pm2}^{2} 8_{\pm2} 2_{\pm5} 5_{\pm5}^{2} 7_{\pm5},
1_0 3_{\pm1}^{2} 4_{\pm2}^{2} 6_{\pm2}^{3}  8_{\pm2} 2_{\pm5}^{2} 5_{\pm5}^{4}  7_{\pm5}^{2},
1_0 3_{\pm1} 4_{\pm2} 6_{\pm2}^{2} 8_{\pm2} 2_{\pm5}^{2} 5_{\pm5}^{2} 7_{\pm5},
1_0 3_{\pm1} 4_{\pm2}^{2} 6_{\pm2}^{2} 8_{\pm2} 2_{\pm5}^{2} 5_{\pm5}^{3}  7_{\pm5}, \\
& 1_0 3_{\pm1} 4_{\pm2} 6_{\pm2}^{2} 8_{\pm2} 2_{\pm5} 5_{\pm5}^{2} 7_{\pm5}^{2},
1_0 3_{\pm1} 4_{\pm2}^{2} 6_{\pm2}^{3}  8_{\pm2} 2_{\pm5}^{2} 5_{\pm5}^{3}  7_{\pm5}^{2},
1_0 4_{\pm2} 6_{\pm2} 8_{\pm2} 2_{\pm5} 5_{\pm5} 7_{\pm5},
1_0 3_{\pm1} 4_{\pm2} 6_{\pm2}^{2} 8_{\pm2} 2_{\pm5} 5_{\pm5}^{2} 7_{\pm5}.
\end{align*}
\end{gather}

%Let $\xi=(0,-3,-1,-2,-3,-2,-1,-2)$. 
%\begin{align*}
%\xymatrix{
%& & 2 &  &  \\
%1  \ar[r] & 3 \ar[r]  & 4  \ar[u]  \ar[r] & 5 &  6  \ar[l]  & 7  \ar[l] \ar[r] & 8 }
%\end{align*}
%The highest $l$-weight monomials of Hernandez-Leclerc modules of type $E_8$ that are not of type $A$, $D$ or $E_7$ are 
\begin{gather}
\begin{align*}
& (7)  1_{0} 7_{\pm1} 2_{\pm5} 5_{\pm5} 8_{\pm4},
1_{0} 4_{\pm2} 7_{\pm1} 2_{\pm5} 5_{\pm5}^{2} 8_{\pm4},
1_{0} 3_{\pm1} 4_{\pm2} 7_{\pm1}^{2} 2_{\pm5}^{2} 5_{\pm5}^{3} 8_{\pm4}, 
1_{0} 3_{\pm1} 4_{\pm2} 6_{\pm2} 7_{\pm1} 2_{\pm5} 5_{\pm5}^{3} 8_{\pm4},\\
& 1_{0} 3_{\pm1} 4_{\pm2} 6_{\pm2} 7_{\pm1} 2_{\pm5}^{2} 5_{\pm5}^{3} 8_{\pm4},
1_{0} 3_{\pm1} 7_{\pm1}^{2} 2_{\pm5} 5_{\pm5}^{2} 8_{\pm4},
1_{0}^{2} 3_{\pm1} 4_{\pm2} 6_{\pm2} 7_{\pm1}^{2} 2_{\pm5}^{2} 5_{\pm5}^{4} 8_{\pm4},
1_{0} 3_{\pm1} 6_{\pm2} 7_{\pm1} 2_{\pm5} 5_{\pm5}^{2} 8_{\pm4}, \\
& 1_{0} 3_{\pm1} 4_{\pm2} 7_{\pm1} 2_{\pm5}^{2} 5_{\pm5}^{2} 8_{\pm4},
1_{0}^{2} 3_{\pm1} 4_{\pm2} 7_{\pm1}^{2} 2_{\pm5}^{2} 5_{\pm5}^{3} 8_{\pm4},
1_{0}^{2} 3_{\pm1}^{2} 4_{\pm2} 6_{\pm2} 7_{\pm1}^{2} 2_{\pm5}^{3} 5_{\pm5}^{4} 8_{\pm4},
1_{0} 3_{\pm1} 4_{\pm2} 7_{\pm1}^{2} 2_{\pm5} 5_{\pm5}^{3} 8_{\pm4},\\
& 1_{0} 3_{\pm1} 7_{\pm1} 2_{\pm5} 5_{\pm5}^{2} 8_{\pm4},
1_{0} 3_{\pm1}^{2} 4_{\pm2} 6_{\pm2} 7_{\pm1}^{2} 2_{\pm5}^{2} 5_{\pm5}^{4} 8_{\pm4},
1_{0}^{2} 3_{\pm1} 4_{\pm2} 6_{\pm2} 7_{\pm1} 2_{\pm5}^{2} 5_{\pm5}^{3} 8_{\pm4},
1_{0}^{2} 3_{\pm1}^{2} 4_{\pm2}^{2} 6_{\pm2} 7_{\pm1}^{2} 2_{\pm5}^{3} 5_{\pm5}^{5} 8_{\pm4}, \\
& 1_{0} 4_{\pm2} 7_{\pm1}^{2} 2_{\pm5} 5_{\pm5}^{2} 8_{\pm4},
1_{0} 3_{\pm1} 4_{\pm2} 6_{\pm2} 7_{\pm1}^{2} 2_{\pm5}^{2} 5_{\pm5}^{3} 8_{\pm4},
1_{0}^{2} 3_{\pm1} 4_{\pm2}^{2} 6_{\pm2} 7_{\pm1}^{2} 2_{\pm5}^{3} 5_{\pm5}^{4} 8_{\pm4},
1_{0} 3_{\pm1} 7_{\pm1} 2_{\pm5} 5_{\pm5} 8_{\pm4}, \\
& 1_{0} 3_{\pm1}^{2} 4_{\pm2} 7_{\pm1}^{2} 2_{\pm5}^{2} 5_{\pm5}^{3} 8_{\pm4},
1_{0}^{2} 3_{\pm1}^{2} 4_{\pm2}^{2} 6_{\pm2} 7_{\pm1}^{3} 2_{\pm5}^{3} 5_{\pm5}^{5} 8_{\pm4}]^{2};
1_{0} 3_{\pm1} 4_{\pm2} 7_{\pm1}^{2} 2_{\pm5}^{2} 5_{\pm5}^{2} 8_{\pm4},
1_{0}^{2} 3_{\pm1}^{2} 4_{\pm2} 6_{\pm2} 7_{\pm1}^{2} 2_{\pm5}^{2} 5_{\pm5}^{4} 8_{\pm4}, \\
& 1_{0} 3_{\pm1} 4_{\pm2}^{2} 6_{\pm2} 7_{\pm1}^{2} 2_{\pm5}^{2} 5_{\pm5}^{4} 8_{\pm4},
1_{0}^{2} 3_{\pm1}^{2} 4_{\pm2}^{2} 6_{\pm2} 7_{\pm1}^{3} 2_{\pm5}^{3} 5_{\pm5}^{5} 8_{\pm4},
1_{0} 3_{\pm1}^{2} 4_{\pm2} 6_{\pm2} 7_{\pm1} 2_{\pm5}^{2} 5_{\pm5}^{3} 8_{\pm4},
1_{0} 3_{\pm1}^{2} 4_{\pm2}^{2} 6_{\pm2} 7_{\pm1}^{2} 2_{\pm5}^{3} 5_{\pm5}^{4} 8_{\pm4}, \\
& 1_{0} 4_{\pm2} 6_{\pm2} 7_{\pm1} 2_{\pm5} 5_{\pm5}^{2} 8_{\pm4},
1_{0}^{2} 3_{\pm1} 4_{\pm2} 6_{\pm2} 7_{\pm1}^{2} 2_{\pm5}^{2} 5_{\pm5}^{3} 8_{\pm4},
1_{0}^{2} 3_{\pm1}^{2} 4_{\pm2}^{2} 6_{\pm2}^{2} 7_{\pm1}^{2} 2_{\pm5}^{3} 5_{\pm5}^{5} 8_{\pm4},
1_{0} 3_{\pm1} 4_{\pm2} 6_{\pm2} 7_{\pm1} 2_{\pm5}^{2} 5_{\pm5}^{2} 8_{\pm4}, \\
& 1_{0} 3_{\pm1} 4_{\pm2} 7_{\pm1} 2_{\pm5} 5_{\pm5}^{2} 8_{\pm4},
1_{0} 3_{\pm1} 4_{\pm2}^{2} 7_{\pm1}^{2} 2_{\pm5}^{2} 5_{\pm5}^{3} 8_{\pm4},
1_{0}^{2} 3_{\pm1} 4_{\pm2}^{2} 6_{\pm2} 7_{\pm1}^{2} 2_{\pm5}^{2} 5_{\pm5}^{4} 8_{\pm4},
1_{0}^{2} 3_{\pm1}^{2} 4_{\pm2}^{2} 6_{\pm2} 7_{\pm1}^{2} 2_{\pm5}^{3} 5_{\pm5}^{4} 8_{\pm4}, \\
& 1_{0} 3_{\pm1} 4_{\pm2} 6_{\pm2} 7_{\pm1}^{2} 2_{\pm5} 5_{\pm5}^{3} 8_{\pm4},
1_{0} 3_{\pm1}^{2} 4_{\pm2} 6_{\pm2} 7_{\pm1}^{2} 2_{\pm5}^{2} 5_{\pm5}^{3} 8_{\pm4},
1_{0} 3_{\pm1} 4_{\pm2}^{2} 6_{\pm2} 7_{\pm1} 2_{\pm5}^{2} 5_{\pm5}^{3} 8_{\pm4},
1_{0} 3_{\pm1}^{2} 4_{\pm2}^{2} 6_{\pm2} 7_{\pm1}^{2} 2_{\pm5}^{2} 5_{\pm5}^{4} 8_{\pm4}, \\
& 1_{0} 4_{\pm2} 7_{\pm1} 2_{\pm5} 5_{\pm5} 8_{\pm4},
1_{0} 3_{\pm1} 4_{\pm2} 7_{\pm1}^{2} 2_{\pm5} 5_{\pm5}^{2} 8_{\pm4},
1_{0} 3_{\pm1} 4_{\pm2}^{2} 6_{\pm2} 7_{\pm1}^{2} 2_{\pm5}^{2} 5_{\pm5}^{3} 8_{\pm4},
1_{0} 3_{\pm1} 4_{\pm2} 6_{\pm2} 7_{\pm1} 2_{\pm5} 5_{\pm5}^{2} 8_{\pm4}. \\
%\end{align*}
%\end{gather}
& \quad \\
%Let $\xi=(0,-3,-1,-2,-3,-2,-1,0)$. 
%\begin{align*}
%\xymatrix{
%& & 2 &  &  \\
%1  \ar[r] & 3 \ar[r]  & 4  \ar[u]  \ar[r] & 5 &  6  \ar[l]  & 7  \ar[l] & 8 \ar[l] }
%\end{align*}
%The highest $l$-weight monomials of Hernandez-Leclerc modules of type $E_8$ that are not of type $A$, $D$ or $E_7$ are 
%\begin{gather}
%\begin{align*}
& (8)  1_{0} 8_{0} 2_{\pm5} 5_{\pm5},
1_{0} 3_{\pm1} 8_{0} 2_{\pm5} 5_{\pm5}^{2},
1_{0} 4_{\pm2} 8_{0} 2_{\pm5} 5_{\pm5}^{2},
1_{0} 3_{\pm1} 4_{\pm2} 7_{\pm1} 8_{0} 2_{\pm5}^{2} 5_{\pm5}^{3}, 
1_{0} 4_{\pm2} 8_{0} 2_{\pm5} 5_{\pm5},
1_{0} 3_{\pm1} 4_{\pm2} 7_{\pm1} 8_{0} 2_{\pm5} 5_{\pm5}^{2}, \\
& 1_{0} 3_{\pm1} 8_{0} 2_{\pm5} 5_{\pm5},
1_{0} 3_{\pm1} 4_{\pm2}^{2} 6_{\pm2} 7_{\pm1} 8_{0} 2_{\pm5}^{2} 5_{\pm5}^{3},
1_{0} 3_{\pm1} 4_{\pm2} 6_{\pm2} 8_{0} 2_{\pm5} 5_{\pm5}^{2},
1_{0} 3_{\pm1} 4_{\pm2} 6_{\pm2} 8_{0} 2_{\pm5}^{2} 5_{\pm5}^{3},
1_{0} 3_{\pm1} 7_{\pm1} 8_{0} 2_{\pm5} 5_{\pm5}^{2}, \\
&  1_{0} 3_{\pm1} 4_{\pm2} 8_{0} 2_{\pm5}^{2} 5_{\pm5}^{2},
1_{0}^{2} 3_{\pm1} 4_{\pm2} 7_{\pm1} 8_{0} 2_{\pm5}^{2} 5_{\pm5}^{3},
1_{0}^{2} 3_{\pm1}^{2} 4_{\pm2} 6_{\pm2} 7_{\pm1} 8_{0} 2_{\pm5}^{3} 5_{\pm5}^{4},
1_{0} 3_{\pm1} 4_{\pm2} 7_{\pm1} 8_{0} 2_{\pm5} 5_{\pm5}^{3}, \\
& 1_{0} 3_{\pm1}^{2} 4_{\pm2} 6_{\pm2} 7_{\pm1} 8_{0} 2_{\pm5}^{2} 5_{\pm5}^{4},
1_{0}^{2} 3_{\pm1} 4_{\pm2} 6_{\pm2} 8_{0} 2_{\pm5}^{2} 5_{\pm5}^{3},
1_{0}^{2} 3_{\pm1}^{2} 4_{\pm2}^{2} 6_{\pm2} 7_{\pm1} 8_{0} 2_{\pm5}^{3} 5_{\pm5}^{5},
1_{0} 3_{\pm1} 4_{\pm2} 6_{\pm2} 8_{0} 2_{\pm5} 5_{\pm5}^{3}, \\
& 1_{0} 4_{\pm2} 7_{\pm1} 8_{0} 2_{\pm5} 5_{\pm5}^{2},
1_{0} 3_{\pm1} 4_{\pm2} 6_{\pm2} 7_{\pm1} 8_{0} 2_{\pm5}^{2} 5_{\pm5}^{3},
1_{0}^{2} 3_{\pm1} 4_{\pm2}^{2} 6_{\pm2} 7_{\pm1} 8_{0} 2_{\pm5}^{3} 5_{\pm5}^{4},
1_{0} 3_{\pm1} 6_{\pm2} 8_{0} 2_{\pm5} 5_{\pm5}^{2},  \\
& 1_{0} 3_{\pm1}^{2} 4_{\pm2} 7_{\pm1} 8_{0} 2_{\pm5}^{2} 5_{\pm5}^{3},
1_{0}^{2} 3_{\pm1}^{2} 4_{\pm2}^{2} 6_{\pm2} 7_{\pm1} 8_{0}^{2} 2_{\pm5}^{3} 5_{\pm5}^{5},
1_{0} 3_{\pm1} 4_{\pm2} 7_{\pm1} 8_{0} 2_{\pm5}^{2} 5_{\pm5}^{2},
1_{0}^{2} 3_{\pm1}^{2} 4_{\pm2} 6_{\pm2} 7_{\pm1} 8_{0} 2_{\pm5}^{2} 5_{\pm5}^{4}, \\
& 1_{0} 3_{\pm1} 4_{\pm2}^{2} 6_{\pm2} 7_{\pm1} 8_{0} 2_{\pm5}^{2} 5_{\pm5}^{4},
1_{0}^{2} 3_{\pm1}^{2} 4_{\pm2}^{2} 6_{\pm2} 7_{\pm1}^{2} 8_{0} 2_{\pm5}^{3} 5_{\pm5}^{5},
1_{0} 3_{\pm1}^{2} 4_{\pm2} 6_{\pm2} 8_{0} 2_{\pm5}^{2} 5_{\pm5}^{3},
1_{0} 3_{\pm1}^{2} 4_{\pm2}^{2} 6_{\pm2} 7_{\pm1} 8_{0} 2_{\pm5}^{3} 5_{\pm5}^{4}, \\
& 1_{0} 4_{\pm2} 6_{\pm2} 8_{0} 2_{\pm5} 5_{\pm5}^{2},
1_{0}^{2} 3_{\pm1} 4_{\pm2} 6_{\pm2} 7_{\pm1} 8_{0} 2_{\pm5}^{2} 5_{\pm5}^{3},
1_{0}^{2} 3_{\pm1}^{2} 4_{\pm2}^{2} 6_{\pm2}^{2} 7_{\pm1} 8_{0} 2_{\pm5}^{3} 5_{\pm5}^{5},
1_{0} 3_{\pm1} 4_{\pm2} 6_{\pm2} 8_{0} 2_{\pm5}^{2} 5_{\pm5}^{2}, \\
& 1_{0} 3_{\pm1} 4_{\pm2} 8_{0} 2_{\pm5} 5_{\pm5}^{2},
1_{0} 3_{\pm1} 4_{\pm2}^{2} 7_{\pm1} 8_{0} 2_{\pm5}^{2} 5_{\pm5}^{3},
1_{0}^{2} 3_{\pm1} 4_{\pm2}^{2} 6_{\pm2} 7_{\pm1} 8_{0} 2_{\pm5}^{2} 5_{\pm5}^{4},
1_{0}^{2} 3_{\pm1}^{2} 4_{\pm2}^{2} 6_{\pm2} 7_{\pm1} 8_{0} 2_{\pm5}^{3} 5_{\pm5}^{4},\\
& 1_{0} 3_{\pm1} 4_{\pm2} 6_{\pm2} 7_{\pm1} 8_{0} 2_{\pm5} 5_{\pm5}^{3},
1_{0} 3_{\pm1}^{2} 4_{\pm2} 6_{\pm2} 7_{\pm1} 8_{0} 2_{\pm5}^{2} 5_{\pm5}^{3},
1_{0} 3_{\pm1} 4_{\pm2}^{2} 6_{\pm2} 8_{0} 2_{\pm5}^{2} 5_{\pm5}^{3},
1_{0}^{2} 3_{\pm1} 4_{\pm2} 6_{\pm2} 7_{\pm1} 8_{0} 2_{\pm5}^{2} 5_{\pm5}^{4}, \\
& 1_{0} 3_{\pm1}^{2} 4_{\pm2}^{2} 6_{\pm2} 7_{\pm1} 8_{0} 2_{\pm5}^{2} 5_{\pm5}^{4}. \\
%\end{align*}
%\end{gather}
& \quad \\
%Let $\xi=(0,-3,-1,-2,-1,-2,-3,-4)$. 
%\begin{align*}
%\xymatrix{
%& & 2 &  &  \\
%1  \ar[r] & 3 \ar[r]  & 4  \ar[u]  & 5 \ar[l]  \ar[r] &  6 \ar[r] &  7  \ar[r]  & 8}
%\end{align*}
%The highest $l$-weight monomials of Hernandez-Leclerc modules of type $E_8$ that are not of type $A$, $D$ or $E_7$ are
%\begin{gather} 
%\begin{align*}
& (9)  1_{0} 5_{\pm1} 2_{\pm5} 4_{\pm4} 8_{\pm6},
1_{0} 3_{\pm1} 5_{\pm1}^{2} 2_{\pm5} 4_{\pm4}^{2} 6_{\pm4} 8_{\pm6},
1_{0} 3_{\pm1} 5_{\pm1}^{3} 2_{\pm5}^{2} 4_{\pm4}^{2} 6_{\pm4} 7_{\pm5} 8_{\pm6},
1_{0} 5_{\pm1}^{2} 2_{\pm5} 4_{\pm4} 6_{\pm4} 8_{\pm6}, \\
& 1_{0} 3_{\pm1} 5_{\pm1}^{2} 2_{\pm5} 4_{\pm4}^{2} 7_{\pm5} 8_{\pm6},
1_{0}^{2} 3_{\pm1} 5_{\pm1}^{4} 2_{\pm5}^{2} 4_{\pm4}^{3} 6_{\pm4} 7_{\pm5} 8_{\pm6},
1_{0} 3_{\pm1} 5_{\pm1}^{3} 2_{\pm5}^{2} 4_{\pm4}^{2} 6_{\pm4} 8_{\pm6},
1_{0}^{2} 3_{\pm1} 5_{\pm1}^{3} 2_{\pm5}^{2} 4_{\pm4}^{2} 6_{\pm4} 7_{\pm5} 8_{\pm6},\\
& 1_{0}^{2} 3_{\pm1}^{2} 5_{\pm1}^{4} 2_{\pm5}^{3} 4_{\pm4}^{3} 6_{\pm4} 7_{\pm5} 8_{\pm6},
1_{0} 3_{\pm1} 5_{\pm1}^{2} 2_{\pm5}^{2} 4_{\pm4} 6_{\pm4} 8_{\pm6},
1_{0} 3_{\pm1} 5_{\pm1}^{3} 2_{\pm5} 4_{\pm4}^{2} 6_{\pm4} 7_{\pm5} 8_{\pm6},
1_{0} 5_{\pm1}^{2} 2_{\pm5} 4_{\pm4} 7_{\pm5} 8_{\pm6}, \\
& 1_{0} 3_{\pm1}^{2} 5_{\pm1}^{4} 2_{\pm5}^{2} 4_{\pm4}^{3} 6_{\pm4} 7_{\pm5} 8_{\pm6},
1_{0}^{2} 3_{\pm1}^{2} 5_{\pm1}^{5} 2_{\pm5}^{3} 4_{\pm4}^{3} 6_{\pm4}^{2} 7_{\pm5} 8_{\pm6},
1_{0} 3_{\pm1} 5_{\pm1}^{3} 2_{\pm5}^{2} 4_{\pm4}^{2} 7_{\pm5} 8_{\pm6},
1_{0} 3_{\pm1} 5_{\pm1}^{2} 2_{\pm5} 4_{\pm4}^{2} 8_{\pm6}, \\
& 1_{0} 3_{\pm1}^{2} 5_{\pm1}^{3} 2_{\pm5}^{2} 4_{\pm4}^{2} 6_{\pm4} 7_{\pm5} 8_{\pm6},
1_{0}^{2} 3_{\pm1} 5_{\pm1}^{4} 2_{\pm5}^{3} 4_{\pm4}^{2} 6_{\pm4} 7_{\pm5} 8_{\pm6},
1_{0}^{2} 3_{\pm1}^{2} 5_{\pm1}^{5} 2_{\pm5}^{3} 4_{\pm4}^{3} 6_{\pm4} 7_{\pm5}^{2} 8_{\pm6},
1_{0} 3_{\pm1} 5_{\pm1}^{2} 2_{\pm5}^{2} 4_{\pm4} 7_{\pm5} 8_{\pm6}, \\
& 1_{0}^{2} 3_{\pm1} 5_{\pm1}^{3} 2_{\pm5}^{2} 4_{\pm4}^{2} 6_{\pm4} 8_{\pm6},
1_{0} 3_{\pm1} 5_{\pm1}^{4} 2_{\pm5}^{2} 4_{\pm4}^{2} 6_{\pm4} 7_{\pm5} 8_{\pm6},
1_{0}^{2} 3_{\pm1}^{2} 5_{\pm1}^{4} 2_{\pm5}^{2} 4_{\pm4}^{3} 6_{\pm4} 7_{\pm5} 8_{\pm6},
1_{0}^{2} 3_{\pm1}^{2} 5_{\pm1}^{5} 2_{\pm5}^{3} 4_{\pm4}^{3} 6_{\pm4} 7_{\pm5} 8_{\pm6}^{2}, \\
&1_{0} 3_{\pm1} 5_{\pm1}^{3} 2_{\pm5} 4_{\pm4}^{2} 6_{\pm4} 8_{\pm6},
1_{0}^{2} 3_{\pm1} 5_{\pm1}^{3} 2_{\pm5}^{2} 4_{\pm4}^{2} 7_{\pm5} 8_{\pm6},
1_{0} 3_{\pm1}^{2} 5_{\pm1}^{4} 2_{\pm5}^{3} 4_{\pm4}^{2} 6_{\pm4} 7_{\pm5} 8_{\pm6},
1_{0}^{2} 3_{\pm1}^{2} 5_{\pm1}^{5} 2_{\pm5}^{3} 4_{\pm4}^{3} 6_{\pm4} 7_{\pm5} 8_{\pm6}, \\
& 1_{0} 3_{\pm1} 5_{\pm1}^{3} 2_{\pm5}^{2} 4_{\pm4} 6_{\pm4} 7_{\pm5} 8_{\pm6},
1_{0} 3_{\pm1} 5_{\pm1} 2_{\pm5} 4_{\pm4} 8_{\pm6},
1_{0} 3_{\pm1}^{2} 5_{\pm1}^{3} 2_{\pm5}^{2} 4_{\pm4}^{2} 6_{\pm4} 8_{\pm6},
1_{0}^{2} 3_{\pm1} 5_{\pm1}^{4} 2_{\pm5}^{2} 4_{\pm4}^{2} 6_{\pm4} 7_{\pm5} 8_{\pm6}, \\
& 1_{0}^{2} 3_{\pm1}^{2} 5_{\pm1}^{4} 2_{\pm5}^{3} 4_{\pm4}^{2} 6_{\pm4} 7_{\pm5} 8_{\pm6},
1_{0} 3_{\pm1} 5_{\pm1}^{2} 2_{\pm5} 4_{\pm4} 6_{\pm4} 8_{\pm6},
1_{0} 3_{\pm1} 5_{\pm1}^{3} 2_{\pm5} 4_{\pm4}^{2} 7_{\pm5} 8_{\pm6},
1_{0} 5_{\pm1}^{2} 2_{\pm5} 4_{\pm4} 8_{\pm6}, \\
&  1_{0} 3_{\pm1}^{2} 5_{\pm1}^{3} 2_{\pm5}^{2} 4_{\pm4}^{2} 7_{\pm5} 8_{\pm6},
1_{0} 3_{\pm1}^{2} 5_{\pm1}^{4} 2_{\pm5}^{2} 4_{\pm4}^{2} 6_{\pm4} 7_{\pm5} 8_{\pm6},
1_{0} 3_{\pm1} 5_{\pm1}^{2} 2_{\pm5}^{2} 4_{\pm4} 8_{\pm6},
1_{0} 3_{\pm1} 5_{\pm1}^{2} 2_{\pm5} 4_{\pm4} 7_{\pm5} 8_{\pm6}, \\
& 1_{0} 3_{\pm1} 5_{\pm1}^{3} 2_{\pm5}^{2} 4_{\pm4} 6_{\pm4} 8_{\pm6},
1_{0} 3_{\pm1} 5_{\pm1}^{3} 2_{\pm5}^{2} 4_{\pm4} 7_{\pm5} 8_{\pm6},
1_{0} 5_{\pm1} 2_{\pm5} 8_{\pm6},
1_{0} 3_{\pm1} 5_{\pm1}^{2} 2_{\pm5} 4_{\pm4} 8_{\pm6}.
\end{align*}
\end{gather}

%Let $\xi=(0,-3,-1,-2,-1,-2,-3,-2)$. 
%\begin{align*}
%\xymatrix{
%& & 2 &  &  \\
%1  \ar[r] & 3 \ar[r]  & 4  \ar[u]  & 5 \ar[l]  \ar[r] &  6 \ar[r] &  7 & 8  \ar[l] }
%\end{align*}
%The highest $l$-weight monomials of Hernandez-Leclerc modules of type $E_8$ that are not of type $A$, $D$ or $E_7$ are 
\begin{gather}
\begin{align*}
& (10)  1_{0} 5_{\pm1} 8_{\pm2} 2_{\pm5} 4_{\pm4} 7_{\pm5},
1_{0} 3_{\pm1} 5_{\pm1}^{2} 8_{\pm2} 2_{\pm5} 4_{\pm4}^{2} 6_{\pm4} 7_{\pm5},
1_{0} 3_{\pm1} 5_{\pm1}^{3} 8_{\pm2} 2_{\pm5}^{2} 4_{\pm4}^{2} 6_{\pm4} 7_{\pm5}^{2},
1_{0} 5_{\pm1}^{2} 8_{\pm2} 2_{\pm5} 4_{\pm4} 6_{\pm4} 7_{\pm5}, \\
& 1_{0} 3_{\pm1} 5_{\pm1}^{2} 8_{\pm2} 2_{\pm5} 4_{\pm4}^{2} 7_{\pm5}^{2},
1_{0}^{2} 3_{\pm1} 5_{\pm1}^{4} 8_{\pm2} 2_{\pm5}^{2} 4_{\pm4}^{3} 6_{\pm4} 7_{\pm5}^{2},
1_{0} 3_{\pm1} 5_{\pm1}^{3} 8_{\pm2} 2_{\pm5}^{2} 4_{\pm4}^{2} 6_{\pm4} 7_{\pm5},
1_{0}^{2} 3_{\pm1} 5_{\pm1}^{3} 8_{\pm2} 2_{\pm5}^{2} 4_{\pm4}^{2} 6_{\pm4} 7_{\pm5}^{2}, \\
& 1_{0}^{2} 3_{\pm1}^{2} 5_{\pm1}^{4} 8_{\pm2} 2_{\pm5}^{3} 4_{\pm4}^{3} 6_{\pm4} 7_{\pm5}^{2},
1_{0} 3_{\pm1} 5_{\pm1}^{2} 8_{\pm2} 2_{\pm5}^{2} 4_{\pm4} 6_{\pm4} 7_{\pm5},
1_{0} 3_{\pm1} 5_{\pm1}^{3} 8_{\pm2} 2_{\pm5} 4_{\pm4}^{2} 6_{\pm4} 7_{\pm5}^{2},
1_{0} 5_{\pm1}^{2} 8_{\pm2} 2_{\pm5} 4_{\pm4} 7_{\pm5}^{2}, \\
& 1_{0} 3_{\pm1}^{2} 5_{\pm1}^{4} 8_{\pm2} 2_{\pm5}^{2} 4_{\pm4}^{3} 6_{\pm4} 7_{\pm5}^{2},
1_{0}^{2} 3_{\pm1}^{2} 5_{\pm1}^{5} 8_{\pm2} 2_{\pm5}^{3} 4_{\pm4}^{3} 6_{\pm4}^{2} 7_{\pm5}^{2},
1_{0} 3_{\pm1} 5_{\pm1}^{3} 8_{\pm2} 2_{\pm5}^{2} 4_{\pm4}^{2} 7_{\pm5}^{2},
1_{0} 3_{\pm1} 5_{\pm1}^{2} 8_{\pm2} 2_{\pm5} 4_{\pm4}^{2} 7_{\pm5}, \\
& 1_{0} 3_{\pm1}^{2} 5_{\pm1}^{3} 8_{\pm2} 2_{\pm5}^{2} 4_{\pm4}^{2} 6_{\pm4} 7_{\pm5}^{2},
1_{0}^{2} 3_{\pm1} 5_{\pm1}^{4} 8_{\pm2} 2_{\pm5}^{3} 4_{\pm4}^{2} 6_{\pm4} 7_{\pm5}^{2},
1_{0}^{2} 3_{\pm1}^{2} 5_{\pm1}^{5} 8_{\pm2} 2_{\pm5}^{3} 4_{\pm4}^{3} 6_{\pm4} 7_{\pm5}^{3},
1_{0} 3_{\pm1} 5_{\pm1}^{2} 8_{\pm2} 2_{\pm5}^{2} 4_{\pm4} 7_{\pm5}^{2}, \\
& 1_{0}^{2} 3_{\pm1} 5_{\pm1}^{3} 8_{\pm2} 2_{\pm5}^{2} 4_{\pm4}^{2} 6_{\pm4} 7_{\pm5},
1_{0} 3_{\pm1} 5_{\pm1}^{4} 8_{\pm2} 2_{\pm5}^{2} 4_{\pm4}^{2} 6_{\pm4} 7_{\pm5}^{2},
1_{0}^{2} 3_{\pm1}^{2} 5_{\pm1}^{4} 8_{\pm2} 2_{\pm5}^{2} 4_{\pm4}^{3} 6_{\pm4} 7_{\pm5}^{2},
1_{0}^{2} 3_{\pm1}^{2} 5_{\pm1}^{5} 8_{\pm2}^{2} 2_{\pm5}^{3} 4_{\pm4}^{3} 6_{\pm4} 7_{\pm5}^{3}, \\
& 1_{0} 3_{\pm1} 5_{\pm1}^{3} 8_{\pm2} 2_{\pm5} 4_{\pm4}^{2} 6_{\pm4} 7_{\pm5},
1_{0}^{2} 3_{\pm1} 5_{\pm1}^{3} 8_{\pm2} 2_{\pm5}^{2} 4_{\pm4}^{2} 7_{\pm5}^{2},
1_{0} 3_{\pm1}^{2} 5_{\pm1}^{4} 8_{\pm2} 2_{\pm5}^{3} 4_{\pm4}^{2} 6_{\pm4} 7_{\pm5}^{2},
1_{0}^{2} 3_{\pm1}^{2} 5_{\pm1}^{5} 8_{\pm2} 2_{\pm5}^{3} 4_{\pm4}^{3} 6_{\pm4} 7_{\pm5}^{2}, \\
& 1_{0} 3_{\pm1} 5_{\pm1}^{3} 8_{\pm2} 2_{\pm5}^{2} 4_{\pm4} 6_{\pm4} 7_{\pm5}^{2},
1_{0} 3_{\pm1} 5_{\pm1} 8_{\pm2} 2_{\pm5} 4_{\pm4} 7_{\pm5},
1_{0} 3_{\pm1}^{2} 5_{\pm1}^{3} 8_{\pm2} 2_{\pm5}^{2} 4_{\pm4}^{2} 6_{\pm4} 7_{\pm5},
1_{0}^{2} 3_{\pm1} 5_{\pm1}^{4} 8_{\pm2} 2_{\pm5}^{2} 4_{\pm4}^{2} 6_{\pm4} 7_{\pm5}^{2}, \\
& 1_{0}^{2} 3_{\pm1}^{2} 5_{\pm1}^{4} 8_{\pm2} 2_{\pm5}^{3} 4_{\pm4}^{2} 6_{\pm4} 7_{\pm5}^{2},
1_{0} 3_{\pm1} 5_{\pm1}^{2} 8_{\pm2} 2_{\pm5} 4_{\pm4} 6_{\pm4} 7_{\pm5},
1_{0} 3_{\pm1} 5_{\pm1}^{3} 8_{\pm2} 2_{\pm5} 4_{\pm4}^{2} 7_{\pm5}^{2},
1_{0} 5_{\pm1}^{2} 8_{\pm2} 2_{\pm5} 4_{\pm4} 7_{\pm5}, \\
& 1_{0} 3_{\pm1}^{2} 5_{\pm1}^{3} 8_{\pm2} 2_{\pm5}^{2} 4_{\pm4}^{2} 7_{\pm5}^{2},
1_{0} 3_{\pm1}^{2} 5_{\pm1}^{4} 8_{\pm2} 2_{\pm5}^{2} 4_{\pm4}^{2} 6_{\pm4} 7_{\pm5}^{2},
1_{0} 3_{\pm1} 5_{\pm1}^{2} 8_{\pm2} 2_{\pm5}^{2} 4_{\pm4} 7_{\pm5},
1_{0} 3_{\pm1} 5_{\pm1}^{2} 8_{\pm2} 2_{\pm5} 4_{\pm4} 7_{\pm5}^{2}, \\
& 1_{0} 3_{\pm1} 5_{\pm1}^{3} 8_{\pm2} 2_{\pm5}^{2} 4_{\pm4} 6_{\pm4} 7_{\pm5},
1_{0} 3_{\pm1} 5_{\pm1}^{3} 8_{\pm2} 2_{\pm5}^{2} 4_{\pm4} 7_{\pm5}^{2},
1_{0} 5_{\pm1} 8_{\pm2} 2_{\pm5} 7_{\pm5},
1_{0} 3_{\pm1} 5_{\pm1}^{2} 8_{\pm2} 2_{\pm5} 4_{\pm4} 7_{\pm5}. \\
& \qquad   \\
& (11)  1_{0} 5_{\pm1} 7_{\pm1} 2_{\pm5} 4_{\pm4} 6_{\pm4} 8_{\pm4},
1_{0} 3_{\pm1} 5_{\pm1}^{2} 7_{\pm1} 2_{\pm5} 4_{\pm4}^{2} 6_{\pm4}^{2} 8_{\pm4},
1_{0} 5_{\pm1}^{2} 7_{\pm1} 2_{\pm5} 4_{\pm4} 6_{\pm4}^{2} 8_{\pm4},
1_{0} 3_{\pm1} 5_{\pm1}^{3} 7_{\pm1}^{2} 2_{\pm5}^{2} 4_{\pm4}^{2} 6_{\pm4}^{3} 8_{\pm4},\\
& 1_{0} 3_{\pm1} 5_{\pm1}^{2} 7_{\pm1}^{2} 2_{\pm5} 4_{\pm4}^{2} 6_{\pm4}^{2} 8_{\pm4},
1_{0} 3_{\pm1} 5_{\pm1}^{3} 7_{\pm1} 2_{\pm5}^{2} 4_{\pm4}^{2} 6_{\pm4}^{2} 8_{\pm4},
1_{0}^{2} 3_{\pm1} 5_{\pm1}^{4} 7_{\pm1}^{2} 2_{\pm5}^{2} 4_{\pm4}^{3} 6_{\pm4}^{3} 8_{\pm4},
1_{0} 3_{\pm1} 5_{\pm1}^{2} 7_{\pm1} 2_{\pm5} 4_{\pm4}^{2} 6_{\pm4} 8_{\pm4}, \\
& 1_{0} 3_{\pm1} 5_{\pm1}^{2} 7_{\pm1} 2_{\pm5}^{2} 4_{\pm4} 6_{\pm4}^{2} 8_{\pm4},
1_{0}^{2} 3_{\pm1} 5_{\pm1}^{3} 7_{\pm1}^{2} 2_{\pm5}^{2} 4_{\pm4}^{2} 6_{\pm4}^{3} 8_{\pm4},
1_{0}^{2} 3_{\pm1}^{2} 5_{\pm1}^{4} 7_{\pm1}^{2} 2_{\pm5}^{3} 4_{\pm4}^{3} 6_{\pm4}^{3} 8_{\pm4},
1_{0} 3_{\pm1} 5_{\pm1}^{3} 7_{\pm1}^{2} 2_{\pm5} 4_{\pm4}^{2} 6_{\pm4}^{3} 8_{\pm4}, \\
& 1_{0}^{2} 3_{\pm1} 5_{\pm1}^{3} 7_{\pm1} 2_{\pm5}^{2} 4_{\pm4}^{2} 6_{\pm4}^{2} 8_{\pm4},
1_{0} 3_{\pm1}^{2} 5_{\pm1}^{4} 7_{\pm1}^{2} 2_{\pm5}^{2} 4_{\pm4}^{3} 6_{\pm4}^{3} 8_{\pm4},
1_{0}^{2} 3_{\pm1}^{2} 5_{\pm1}^{5} 7_{\pm1}^{2} 2_{\pm5}^{3} 4_{\pm4}^{3} 6_{\pm4}^{4} 8_{\pm4},
1_{0} 3_{\pm1} 5_{\pm1}^{3} 7_{\pm1} 2_{\pm5} 4_{\pm4}^{2} 6_{\pm4}^{2} 8_{\pm4}, \\
& 1_{0} 5_{\pm1}^{2} 7_{\pm1}^{2} 2_{\pm5} 4_{\pm4} 6_{\pm4}^{2} 8_{\pm4},
1_{0} 3_{\pm1} 5_{\pm1}^{3} 7_{\pm1}^{2} 2_{\pm5}^{2} 4_{\pm4}^{2} 6_{\pm4}^{2} 8_{\pm4},
1_{0} 3_{\pm1} 5_{\pm1} 7_{\pm1} 2_{\pm5} 4_{\pm4} 6_{\pm4} 8_{\pm4},
1_{0} 3_{\pm1}^{2} 5_{\pm1}^{3} 7_{\pm1}^{2} 2_{\pm5}^{2} 4_{\pm4}^{2} 6_{\pm4}^{3} 8_{\pm4}, \\
& 1_{0}^{2} 3_{\pm1} 5_{\pm1}^{4} 7_{\pm1}^{2} 2_{\pm5}^{3} 4_{\pm4}^{2} 6_{\pm4}^{3} 8_{\pm4},
1_{0}^{2} 3_{\pm1}^{2} 5_{\pm1}^{5} 7_{\pm1}^{3} 2_{\pm5}^{3} 4_{\pm4}^{3} 6_{\pm4}^{4} 8_{\pm4}^{2},
1_{0} 3_{\pm1} 5_{\pm1}^{2} 7_{\pm1}^{2} 2_{\pm5}^{2} 4_{\pm4} 6_{\pm4}^{2} 8_{\pm4},
1_{0} 3_{\pm1} 5_{\pm1}^{4} 7_{\pm1}^{2} 2_{\pm5}^{2} 4_{\pm4}^{2} 6_{\pm4}^{3} 8_{\pm4}, \\
&1_{0}^{2} 3_{\pm1}^{2} 5_{\pm1}^{4} 7_{\pm1}^{2} 2_{\pm5}^{2} 4_{\pm4}^{3} 6_{\pm4}^{3} 8_{\pm4},
1_{0}^{2} 3_{\pm1}^{2} 5_{\pm1}^{5} 7_{\pm1}^{3} 2_{\pm5}^{3} 4_{\pm4}^{3} 6_{\pm4}^{4} 8_{\pm4},
1_{0} 3_{\pm1}^{2} 5_{\pm1}^{3} 7_{\pm1} 2_{\pm5}^{2} 4_{\pm4}^{2} 6_{\pm4}^{2} 8_{\pm4},
1_{0} 5_{\pm1}^{2} 7_{\pm1} 2_{\pm5} 4_{\pm4} 6_{\pm4} 8_{\pm4}, \\
& 1_{0}^{2} 3_{\pm1} 5_{\pm1}^{3} 7_{\pm1}^{2} 2_{\pm5}^{2} 4_{\pm4}^{2} 6_{\pm4}^{2} 8_{\pm4},
1_{0} 3_{\pm1}^{2} 5_{\pm1}^{4} 7_{\pm1}^{2} 2_{\pm5}^{3} 4_{\pm4}^{2} 6_{\pm4}^{3} 8_{\pm4},
1_{0}^{2} 3_{\pm1}^{2} 5_{\pm1}^{5} 7_{\pm1}^{2} 2_{\pm5}^{3} 4_{\pm4}^{3} 6_{\pm4}^{3} 8_{\pm4},
1_{0} 3_{\pm1} 5_{\pm1}^{2} 7_{\pm1} 2_{\pm5}^{2} 4_{\pm4} 6_{\pm4} 8_{\pm4}, \\
& 1_{0} 3_{\pm1} 5_{\pm1}^{2} 7_{\pm1} 2_{\pm5} 4_{\pm4} 6_{\pm4}^{2} 8_{\pm4},
1_{0} 3_{\pm1} 5_{\pm1}^{3} 7_{\pm1}^{2} 2_{\pm5}^{2} 4_{\pm4} 6_{\pm4}^{3} 8_{\pm4},
1_{0}^{2} 3_{\pm1} 5_{\pm1}^{4} 7_{\pm1}^{2} 2_{\pm5}^{2} 4_{\pm4}^{2} 6_{\pm4}^{3} 8_{\pm4},
1_{0}^{2} 3_{\pm1}^{2} 5_{\pm1}^{4} 7_{\pm1}^{2} 2_{\pm5}^{3} 4_{\pm4}^{2} 6_{\pm4}^{3} 8_{\pm4}, \\
& 1_{0} 3_{\pm1} 5_{\pm1}^{3} 7_{\pm1}^{2} 2_{\pm5} 4_{\pm4}^{2} 6_{\pm4}^{2} 8_{\pm4},
1_{0} 3_{\pm1} 5_{\pm1}^{3} 7_{\pm1} 2_{\pm5}^{2} 4_{\pm4} 6_{\pm4}^{2} 8_{\pm4},
1_{0} 3_{\pm1}^{2} 5_{\pm1}^{3} 7_{\pm1}^{2} 2_{\pm5}^{2} 4_{\pm4}^{2} 6_{\pm4}^{2} 8_{\pm4},
1_{0} 3_{\pm1}^{2} 5_{\pm1}^{4} 7_{\pm1}^{2} 2_{\pm5}^{2} 4_{\pm4}^{2} 6_{\pm4}^{3} 8_{\pm4}, \\
& 1_{0} 5_{\pm1} 7_{\pm1} 2_{\pm5} 6_{\pm4} 8_{\pm4},
1_{0} 3_{\pm1} 5_{\pm1}^{2} 7_{\pm1}^{2} 2_{\pm5} 4_{\pm4} 6_{\pm4}^{2} 8_{\pm4},
1_{0} 3_{\pm1} 5_{\pm1}^{3} 7_{\pm1}^{2} 2_{\pm5}^{2} 4_{\pm4} 6_{\pm4}^{2} 8_{\pm4},
1_{0} 3_{\pm1} 5_{\pm1}^{2} 7_{\pm1} 2_{\pm5} 4_{\pm4} 6_{\pm4} 8_{\pm4}. \\
%\end{align*}
%\end{gather}
& \quad \\
%%
%%
%%
%%
%%
%Let $\xi=(0,-3,-1,-2,-1,-2,-1,0)$. 
%\begin{align*}
%\xymatrix{
%& & 2 &  &  \\
%1  \ar[r] & 3 \ar[r]  & 4  \ar[u]  & 5 \ar[l]  \ar[r] &  6 &  7 \ar[l] & 8  \ar[l]}
%\end{align*}
%The highest $l$-weight monomials of Hernandez-Leclerc modules of type $E_8$ that are not of type $A$, $D$ or $E_7$ are 
% \begin{gather}
%\begin{align*}
& (12)  1_{0} 5_{\pm1} 8_{0} 2_{\pm5} 4_{\pm4} 6_{\pm4},
1_{0} 3_{\pm1} 5_{\pm1}^{2} 8_{0} 2_{\pm5} 4_{\pm4}^{2} 6_{\pm4}^{2},
1_{0} 5_{\pm1}^{2} 8_{0} 2_{\pm5} 4_{\pm4} 6_{\pm4}^{2},
1_{0} 3_{\pm1} 5_{\pm1}^{3} 7_{\pm1} 8_{0} 2_{\pm5}^{2} 4_{\pm4}^{2} 6_{\pm4}^{3}, \\
& 1_{0} 3_{\pm1} 5_{\pm1}^{2} 7_{\pm1} 8_{0} 2_{\pm5} 4_{\pm4}^{2} 6_{\pm4}^{2},
1_{0} 3_{\pm1} 5_{\pm1}^{3} 8_{0} 2_{\pm5}^{2} 4_{\pm4}^{2} 6_{\pm4}^{2},
1_{0}^{2} 3_{\pm1} 5_{\pm1}^{4} 7_{\pm1} 8_{0} 2_{\pm5}^{2} 4_{\pm4}^{3} 6_{\pm4}^{3},
1_{0} 3_{\pm1} 5_{\pm1}^{2} 8_{0} 2_{\pm5} 4_{\pm4}^{2} 6_{\pm4}, \\
& 1_{0} 3_{\pm1} 5_{\pm1}^{2} 8_{0} 2_{\pm5}^{2} 4_{\pm4} 6_{\pm4}^{2},
1_{0}^{2} 3_{\pm1} 5_{\pm1}^{3} 7_{\pm1} 8_{0} 2_{\pm5}^{2} 4_{\pm4}^{2} 6_{\pm4}^{3},
1_{0}^{2} 3_{\pm1}^{2} 5_{\pm1}^{4} 7_{\pm1} 8_{0} 2_{\pm5}^{3} 4_{\pm4}^{3} 6_{\pm4}^{3},
1_{0} 3_{\pm1} 5_{\pm1}^{3} 7_{\pm1} 8_{0} 2_{\pm5} 4_{\pm4}^{2} 6_{\pm4}^{3}, \\
& 1_{0}^{2} 3_{\pm1} 5_{\pm1}^{3} 8_{0} 2_{\pm5}^{2} 4_{\pm4}^{2} 6_{\pm4}^{2},
1_{0} 3_{\pm1}^{2} 5_{\pm1}^{4} 7_{\pm1} 8_{0} 2_{\pm5}^{2} 4_{\pm4}^{3} 6_{\pm4}^{3},
1_{0}^{2} 3_{\pm1}^{2} 5_{\pm1}^{5} 7_{\pm1} 8_{0} 2_{\pm5}^{3} 4_{\pm4}^{3} 6_{\pm4}^{4},
1_{0} 3_{\pm1} 5_{\pm1}^{3} 8_{0} 2_{\pm5} 4_{\pm4}^{2} 6_{\pm4}^{2}, \\
& 1_{0} 5_{\pm1}^{2} 7_{\pm1} 8_{0} 2_{\pm5} 4_{\pm4} 6_{\pm4}^{2},
1_{0} 3_{\pm1} 5_{\pm1}^{3} 7_{\pm1} 8_{0} 2_{\pm5}^{2} 4_{\pm4}^{2} 6_{\pm4}^{2},
1_{0} 3_{\pm1} 5_{\pm1} 8_{0} 2_{\pm5} 4_{\pm4} 6_{\pm4},
1_{0} 3_{\pm1}^{2} 5_{\pm1}^{3} 7_{\pm1} 8_{0} 2_{\pm5}^{2} 4_{\pm4}^{2} 6_{\pm4}^{3}, \\
& 1_{0}^{2} 3_{\pm1} 5_{\pm1}^{4} 7_{\pm1} 8_{0} 2_{\pm5}^{3} 4_{\pm4}^{2} 6_{\pm4}^{3},
1_{0}^{2} 3_{\pm1}^{2} 5_{\pm1}^{5} 7_{\pm1} 8_{0}^{2} 2_{\pm5}^{3} 4_{\pm4}^{3} 6_{\pm4}^{4},
1_{0} 3_{\pm1} 5_{\pm1}^{2} 7_{\pm1} 8_{0} 2_{\pm5}^{2} 4_{\pm4} 6_{\pm4}^{2},
1_{0} 3_{\pm1} 5_{\pm1}^{4} 7_{\pm1} 8_{0} 2_{\pm5}^{2} 4_{\pm4}^{2} 6_{\pm4}^{3}, \\
& 1_{0}^{2} 3_{\pm1}^{2} 5_{\pm1}^{4} 7_{\pm1} 8_{0} 2_{\pm5}^{2} 4_{\pm4}^{3} 6_{\pm4}^{3},
1_{0}^{2} 3_{\pm1}^{2} 5_{\pm1}^{5} 7_{\pm1}^{2} 8_{0} 2_{\pm5}^{3} 4_{\pm4}^{3} 6_{\pm4}^{4},
1_{0} 3_{\pm1}^{2} 5_{\pm1}^{3} 8_{0} 2_{\pm5}^{2} 4_{\pm4}^{2} 6_{\pm4}^{2},
1_{0} 5_{\pm1}^{2} 8_{0} 2_{\pm5} 4_{\pm4} 6_{\pm4}, \\
& 1_{0}^{2} 3_{\pm1} 5_{\pm1}^{3} 7_{\pm1} 8_{0} 2_{\pm5}^{2} 4_{\pm4}^{2} 6_{\pm4}^{2},
1_{0} 3_{\pm1}^{2} 5_{\pm1}^{4} 7_{\pm1} 8_{0} 2_{\pm5}^{3} 4_{\pm4}^{2} 6_{\pm4}^{3},
1_{0}^{2} 3_{\pm1}^{2} 5_{\pm1}^{5} 7_{\pm1} 8_{0} 2_{\pm5}^{3} 4_{\pm4}^{3} 6_{\pm4}^{3},
1_{0} 3_{\pm1} 5_{\pm1}^{2} 8_{0} 2_{\pm5}^{2} 4_{\pm4} 6_{\pm4}, \\
& 1_{0} 3_{\pm1} 5_{\pm1}^{2} 8_{0} 2_{\pm5} 4_{\pm4} 6_{\pm4}^{2},
1_{0} 3_{\pm1} 5_{\pm1}^{3} 7_{\pm1} 8_{0} 2_{\pm5}^{2} 4_{\pm4} 6_{\pm4}^{3},
1_{0}^{2} 3_{\pm1} 5_{\pm1}^{4} 7_{\pm1} 8_{0} 2_{\pm5}^{2} 4_{\pm4}^{2} 6_{\pm4}^{3},
1_{0}^{2} 3_{\pm1}^{2} 5_{\pm1}^{4} 7_{\pm1} 8_{0} 2_{\pm5}^{3} 4_{\pm4}^{2} 6_{\pm4}^{3}, \\
& 1_{0} 3_{\pm1} 5_{\pm1}^{3} 7_{\pm1} 8_{0} 2_{\pm5} 4_{\pm4}^{2} 6_{\pm4}^{2},
1_{0} 3_{\pm1} 5_{\pm1}^{3} 8_{0} 2_{\pm5}^{2} 4_{\pm4} 6_{\pm4}^{2},
1_{0} 3_{\pm1}^{2} 5_{\pm1}^{3} 7_{\pm1} 8_{0} 2_{\pm5}^{2} 4_{\pm4}^{2} 6_{\pm4}^{2},
1_{0} 3_{\pm1}^{2} 5_{\pm1}^{4} 7_{\pm1} 8_{0} 2_{\pm5}^{2} 4_{\pm4}^{2} 6_{\pm4}^{3}, \\
& 1_{0} 5_{\pm1} 8_{0} 2_{\pm5} 6_{\pm4},
1_{0} 3_{\pm1} 5_{\pm1}^{2} 7_{\pm1} 8_{0} 2_{\pm5} 4_{\pm4} 6_{\pm4}^{2},
1_{0} 3_{\pm1} 5_{\pm1}^{3} 7_{\pm1} 8_{0} 2_{\pm5}^{2} 4_{\pm4} 6_{\pm4}^{2},
1_{0} 3_{\pm1} 5_{\pm1}^{2} 8_{0} 2_{\pm5} 4_{\pm4} 6_{\pm4}. \\
\end{align*}
\end{gather}
%%
%%
%%
%Let $\xi=(0,-3,-1,-2,-1,0,-1,0)$. 
%\begin{align*}
%\xymatrix{
%& & 2 &  &  \\
%1  \ar[r] & 3 \ar[r]  & 4  \ar[u]  & 5 \ar[l]  &  6 \ar[l] \ar[r] &  7  & 8 \ar[l]}
%\end{align*}
%The highest $l$-weight monomials of Hernandez-Leclerc modules of type $E_8$ that are not of type $A$, $D$ or $E_7$ are 
\begin{gather}
\begin{align*}
& (13)  1_{0} 6_{0} 8_{0} 2_{\pm5} 4_{\pm4} 7_{\pm3},
1_{0} 3_{\pm1} 6_{0}^{2} 8_{0} 2_{\pm5} 4_{\pm4}^{2} 7_{\pm3}^{2},
1_{0} 6_{0}^{2} 8_{0} 2_{\pm5} 4_{\pm4} 7_{\pm3}^{2},
1_{0} 3_{\pm1} 6_{0}^{3} 8_{0} 2_{\pm5}^{2} 4_{\pm4}^{2} 7_{\pm3}^{2}, \\
& 1_{0} 3_{\pm1} 6_{0}^{2} 8_{0} 2_{\pm5} 4_{\pm4}^{2} 7_{\pm3},
1_{0} 3_{\pm1} 5_{\pm1} 6_{0}^{2} 8_{0} 2_{\pm5}^{2} 4_{\pm4}^{2} 7_{\pm3}^{2},
1_{0}^{2} 3_{\pm1} 5_{\pm1} 6_{0}^{3} 8_{0} 2_{\pm5}^{2} 4_{\pm4}^{3} 7_{\pm3}^{2},
1_{0} 3_{\pm1} 5_{\pm1} 6_{0} 8_{0} 2_{\pm5} 4_{\pm4}^{2} 7_{\pm3}, \\
& 1_{0} 3_{\pm1} 6_{0}^{2} 8_{0} 2_{\pm5}^{2} 4_{\pm4} 7_{\pm3}^{2},
1_{0}^{2} 3_{\pm1} 6_{0}^{3} 8_{0} 2_{\pm5}^{2} 4_{\pm4}^{2} 7_{\pm3}^{2},
1_{0}^{2} 3_{\pm1}^{2} 5_{\pm1} 6_{0}^{3} 8_{0} 2_{\pm5}^{3} 4_{\pm4}^{3} 7_{\pm3}^{2},
1_{0}^{2} 3_{\pm1} 5_{\pm1} 6_{0}^{2} 8_{0} 2_{\pm5}^{2} 4_{\pm4}^{2} 7_{\pm3}^{2}, \\
& 1_{0} 3_{\pm1} 6_{0}^{3} 8_{0} 2_{\pm5} 4_{\pm4}^{2} 7_{\pm3}^{2},
1_{0} 3_{\pm1}^{2} 5_{\pm1} 6_{0}^{3} 8_{0} 2_{\pm5}^{2} 4_{\pm4}^{3} 7_{\pm3}^{2},
1_{0}^{2} 3_{\pm1}^{2} 5_{\pm1} 6_{0}^{4}  8_{0} 2_{\pm5}^{3} 4_{\pm4}^{3} 7_{\pm3}^{3},
1_{0} 3_{\pm1} 5_{\pm1} 6_{0}^{2} 8_{0} 2_{\pm5} 4_{\pm4}^{2} 7_{\pm3}^{2}, \\
& 1_{0} 6_{0}^{2} 8_{0} 2_{\pm5} 4_{\pm4} 7_{\pm3},
1_{0} 3_{\pm1} 6_{0} 8_{0} 2_{\pm5} 4_{\pm4} 7_{\pm3},
1_{0} 3_{\pm1}^{2} 6_{0}^{3} 8_{0} 2_{\pm5}^{2} 4_{\pm4}^{2} 7_{\pm3}^{2},
1_{0} 3_{\pm1} 5_{\pm1} 6_{0}^{2} 8_{0} 2_{\pm5}^{2} 4_{\pm4}^{2} 7_{\pm3}, \\
& 1_{0}^{2} 3_{\pm1} 5_{\pm1} 6_{0}^{3} 8_{0} 2_{\pm5}^{3} 4_{\pm4}^{2} 7_{\pm3}^{2},
1_{0}^{2} 3_{\pm1}^{2} 5_{\pm1} 6_{0}^{4}  8_{0}^{2} 2_{\pm5}^{3} 4_{\pm4}^{3} 7_{\pm3}^{3},
1_{0} 3_{\pm1} 6_{0}^{2} 8_{0} 2_{\pm5}^{2} 4_{\pm4} 7_{\pm3},
1_{0} 3_{\pm1} 5_{\pm1} 6_{0}^{3} 8_{0} 2_{\pm5}^{2} 4_{\pm4}^{2} 7_{\pm3}^{2}, \\
& 1_{0}^{2} 3_{\pm1}^{2} 5_{\pm1} 6_{0}^{3} 8_{0} 2_{\pm5}^{2} 4_{\pm4}^{3} 7_{\pm3}^{2},
1_{0}^{2} 3_{\pm1}^{2} 5_{\pm1} 6_{0}^{4}  8_{0} 2_{\pm5}^{3} 4_{\pm4}^{3} 7_{\pm3}^{2},
1_{0} 5_{\pm1} 6_{0} 8_{0} 2_{\pm5} 4_{\pm4} 7_{\pm3},
1_{0}^{2} 3_{\pm1} 5_{\pm1} 6_{0}^{2} 8_{0} 2_{\pm5}^{2} 4_{\pm4}^{2} 7_{\pm3}, \\
& 1_{0} 3_{\pm1}^{2} 5_{\pm1} 6_{0}^{2} 8_{0} 2_{\pm5}^{2} 4_{\pm4}^{2} 7_{\pm3}^{2},
1_{0} 3_{\pm1}^{2} 5_{\pm1} 6_{0}^{3} 8_{0} 2_{\pm5}^{3} 4_{\pm4}^{2} 7_{\pm3}^{2},
1_{0}^{2} 3_{\pm1}^{2} 5_{\pm1}^{2} 6_{0}^{3} 8_{0} 2_{\pm5}^{3} 4_{\pm4}^{3} 7_{\pm3}^{2},
1_{0} 3_{\pm1} 5_{\pm1} 6_{0} 8_{0} 2_{\pm5}^{2} 4_{\pm4} 7_{\pm3}, \\
& 1_{0} 3_{\pm1} 6_{0}^{2} 8_{0} 2_{\pm5} 4_{\pm4} 7_{\pm3}^{2},
1_{0} 3_{\pm1} 6_{0}^{3} 8_{0} 2_{\pm5}^{2} 4_{\pm4} 7_{\pm3}^{2},
1_{0}^{2} 3_{\pm1} 5_{\pm1} 6_{0}^{3} 8_{0} 2_{\pm5}^{2} 4_{\pm4}^{2} 7_{\pm3}^{2},
1_{0}^{2} 3_{\pm1}^{2} 5_{\pm1} 6_{0}^{3} 8_{0} 2_{\pm5}^{3} 4_{\pm4}^{2} 7_{\pm3}^{2}, \\
& 1_{0} 3_{\pm1} 5_{\pm1} 6_{0}^{2} 8_{0} 2_{\pm5}^{2} 4_{\pm4} 7_{\pm3}^{2},
1_{0} 3_{\pm1} 5_{\pm1} 6_{0}^{2} 8_{0} 2_{\pm5} 4_{\pm4}^{2} 7_{\pm3},
1_{0} 3_{\pm1}^{2} 5_{\pm1} 6_{0}^{2} 8_{0} 2_{\pm5}^{2} 4_{\pm4}^{2} 7_{\pm3},
1_{0} 3_{\pm1}^{2} 5_{\pm1} 6_{0}^{3} 8_{0} 2_{\pm5}^{2} 4_{\pm4}^{2} 7_{\pm3}^{2}, \\
& 1_{0} 6_{0} 8_{0} 2_{\pm5} 7_{\pm3},
1_{0} 3_{\pm1} 6_{0}^{2} 8_{0} 2_{\pm5} 4_{\pm4} 7_{\pm3},
1_{0} 3_{\pm1} 5_{\pm1} 6_{0}^{2} 8_{0} 2_{\pm5}^{2} 4_{\pm4} 7_{\pm3},
1_{0} 3_{\pm1} 5_{\pm1} 6_{0} 8_{0} 2_{\pm5} 4_{\pm4} 7_{\pm3}. \\
% \end{align*}
% \end{gather}
%%
%%
& \quad \\
%%
%%
%Let $\xi=(-1,-4,-2,-3,-2,-1,0,-1)$. 
%\begin{align*}
%\xymatrix{
%& & 2 &  &  \\
%1  \ar[r] & 3 \ar[r]  & 4  \ar[u]  & 5 \ar[l]  &  6 \ar[l]  & 7  \ar[l]  \ar[r] & 8}
%\end{align*}
%The highest $l$-weight monomials of Hernandez-Leclerc modules of type $E_8$ that are not of type $A$, $D$ or $E_7$ are 
%\begin{gather}
%\begin{align*}
& (13)  1_{0} 2_{\pm6} 7_{0} 8_{\pm3},
1_{0} 2_{\pm6} 4_{\pm5} 7_{0} 8_{\pm3},
1_{0} 3_{\pm2} 7_{0}^{2} 2_{\pm6} 4_{\pm5}^{2} 8_{\pm3},
1_{0} 7_{0}^{2} 2_{\pm6} 4_{\pm5} 8_{\pm3},
1_{0} 3_{\pm2} 6_{\pm1} 7_{0} 2_{\pm6} 4_{\pm5}^{2} 8_{\pm3}, \\
& 1_{0} 3_{\pm2} 6_{\pm1} 7_{0}^{2} 2_{\pm6}^{2} 4_{\pm5}^{2} 8_{\pm3}, 
1_{0} 6_{\pm1} 7_{0} 2_{\pm6} 4_{\pm5} 8_{\pm3},
1_{0} 3_{\pm2} 5_{\pm2} 7_{0} 2_{\pm6} 4_{\pm5}^{2} 8_{\pm3},
1_{0} 3_{\pm2} 5_{\pm2} 7_{0}^{2} 2_{\pm6}^{2} 4_{\pm5}^{2} 8_{\pm3}, \\
& 1_{0}^{2} 3_{\pm2} 5_{\pm2} 6_{\pm1} 7_{0}^{2} 2_{\pm6}^{2} 4_{\pm5}^{3} 8_{\pm3},
1_{0} 3_{\pm2} 5_{\pm2} 6_{\pm1} 7_{0} 2_{\pm6}^{2} 4_{\pm5}^{2} 8_{\pm3},
1_{0} 3_{\pm2} 7_{0} 2_{\pm6} 4_{\pm5} 8_{\pm3},
1_{0} 3_{\pm2} 7_{0}^{2} 2_{\pm6}^{2} 4_{\pm5} 8_{\pm3}, \\
& 1_{0}^{2} 3_{\pm2} 6_{\pm1} 7_{0}^{2} 2_{\pm6}^{2} 4_{\pm5}^{2} 8_{\pm3},
1_{0}^{2} 3_{\pm2}^{2} 5_{\pm2} 6_{\pm1} 7_{0}^{2} 2_{\pm6}^{3} 4_{\pm5}^{3} 8_{\pm3},
1_{0} 3_{\pm2} 6_{\pm1} 7_{0} 2_{\pm6}^{2} 4_{\pm5} 8_{\pm3},
1_{0} 5_{\pm2} 7_{0} 2_{\pm6} 4_{\pm5} 8_{\pm3}, \\
& 1_{0}^{2} 3_{\pm2} 5_{\pm2} 7_{0}^{2} 2_{\pm6}^{2} 4_{\pm5}^{2} 8_{\pm3},
1_{0} 3_{\pm2} 6_{\pm1} 7_{0}^{2} 2_{\pm6} 4_{\pm5}^{2} 8_{\pm3},
1_{0} 3_{\pm2}^{2} 5_{\pm2} 6_{\pm1} 7_{0}^{2} 2_{\pm6}^{2} 4_{\pm5}^{3} 8_{\pm3},
1_{0}^{2} 3_{\pm2}^{2} 5_{\pm2} 6_{\pm1} 7_{0}^{3} 2_{\pm6}^{3} 4_{\pm5}^{3} 8_{\pm3}^{2}, \\
& 1_{0} 3_{\pm2} 5_{\pm2} 7_{0}^{2} 2_{\pm6} 4_{\pm5}^{2} 8_{\pm3},
1_{0} 3_{\pm2}^{2} 6_{\pm1} 7_{0}^{2} 2_{\pm6}^{2} 4_{\pm5}^{2} 8_{\pm3},
1_{0}^{2} 3_{\pm2} 5_{\pm2} 6_{\pm1} 7_{0}^{2} 2_{\pm6}^{3} 4_{\pm5}^{2} 8_{\pm3},
1_{0}^{2} 3_{\pm2}^{2} 5_{\pm2} 6_{\pm1} 7_{0}^{3} 2_{\pm6}^{3} 4_{\pm5}^{3} 8_{\pm3}, \\
& 1_{0} 3_{\pm2} 5_{\pm2} 6_{\pm1} 7_{0}^{2} 2_{\pm6}^{2} 4_{\pm5}^{2} 8_{\pm3},
1_{0}^{2} 3_{\pm2} 5_{\pm2} 6_{\pm1} 7_{0} 2_{\pm6}^{2} 4_{\pm5}^{2} 8_{\pm3},
1_{0}^{2} 3_{\pm2}^{2} 5_{\pm2} 6_{\pm1} 7_{0}^{2} 2_{\pm6}^{2} 4_{\pm5}^{3} 8_{\pm3}, \\
& 1_{0} 3_{\pm2} 5_{\pm2} 6_{\pm1} 7_{0} 2_{\pm6} 4_{\pm5}^{2} 8_{\pm3},
1_{0} 3_{\pm2} 5_{\pm2} 7_{0} 2_{\pm6}^{2} 4_{\pm5} 8_{\pm3},
1_{0} 3_{\pm2}^{2} 5_{\pm2} 7_{0}^{2} 2_{\pm6}^{2} 4_{\pm5}^{2} 8_{\pm3},
1_{0} 3_{\pm2}^{2} 5_{\pm2} 6_{\pm1} 7_{0}^{2} 2_{\pm6}^{3} 4_{\pm5}^{2} 8_{\pm3}, \\
& 1_{0}^{2} 3_{\pm2}^{2} 5_{\pm2}^{2} 6_{\pm1} 7_{0}^{2} 2_{\pm6}^{3} 4_{\pm5}^{3} 8_{\pm3},
1_{0} 3_{\pm2}^{2} 5_{\pm2} 6_{\pm1} 7_{0} 2_{\pm6}^{2} 4_{\pm5}^{2} 8_{\pm3},
1_{0} 3_{\pm2} 7_{0}^{2} 2_{\pm6} 4_{\pm5} 8_{\pm3},
1_{0}^{2} 3_{\pm2}^{2} 5_{\pm2} 6_{\pm1}^{2} 7_{0}^{2} 2_{\pm6}^{3} 4_{\pm5}^{3} 8_{\pm3},  \\
& 1_{0} 3_{\pm2} 6_{\pm1} 7_{0}^{2} 2_{\pm6}^{2} 4_{\pm5} 8_{\pm3},
1_{0}^{2} 3_{\pm2} 5_{\pm2} 6_{\pm1} 7_{0}^{2} 2_{\pm6}^{2} 4_{\pm5}^{2} 8_{\pm3},
1_{0}^{2} 3_{\pm2}^{2} 5_{\pm2} 6_{\pm1} 7_{0}^{2} 2_{\pm6}^{3} 4_{\pm5}^{2} 8_{\pm3},
1_{0} 3_{\pm2} 6_{\pm1} 7_{0} 2_{\pm6} 4_{\pm5} 8_{\pm3}, \\
& 1_{0} 3_{\pm2} 5_{\pm2} 7_{0}^{2} 2_{\pm6}^{2} 4_{\pm5} 8_{\pm3},
1_{0} 3_{\pm2}^{2} 5_{\pm2} 6_{\pm1} 7_{0}^{2} 2_{\pm6}^{2} 4_{\pm5}^{2} 8_{\pm3},
1_{0} 3_{\pm2} 5_{\pm2} 6_{\pm1} 7_{0} 2_{\pm6}^{2} 4_{\pm5} 8_{\pm3},
1_{0} 3_{\pm2} 5_{\pm2} 7_{0} 2_{\pm6} 4_{\pm5} 8_{\pm3}. \\
%\end{align*}
% \end{gather}
%%
& \quad \\
%%
%%
%%
%Let $\xi=(-2,-5,-3,-4_{\pm3},-2,-1,0)$. 
%\begin{align*}
%\xymatrix{
%& & 2 &  &  \\
%1  \ar[r] & 3 \ar[r]  & 4  \ar[u]  & 5 \ar[l]  &  6 \ar[l]  & 7  \ar[l]  & 8 \ar[l] }
%\end{align*}
%The highest $l$-weight monomials of Hernandez-Leclerc modules of type $E_8$ that are not of type $A$, $D$ or $E_7$ are 
%\begin{gather}
%\begin{align*}
& (15)  1_{\pm2} 2_{\pm7}  8_{0},
1_{\pm2} 8_{0} 2_{\pm7} 4_{\pm6},
1_{\pm2} 3_{\pm3} 7_{\pm1} 8_{0} 2_{\pm7} 4_{\pm6}^{2},
1_{\pm2} 7_{\pm1} 8_{0} 2_{\pm7} 4_{\pm6},
1_{\pm2} 3_{\pm3} 6_{\pm2} 8_{0} 2_{\pm7} 4_{\pm6}^{2}, 
1_{\pm2} 3_{\pm3} 5_{\pm3} 8_{0} 2_{\pm7} 4_{\pm6}^{2}, \\
& 1_{\pm2} 6_{\pm2} 8_{0} 2_{\pm7} 4_{\pm6},
1_{\pm2} 3_{\pm3} 6_{\pm2} 7_{\pm1} 8_{0} 2_{\pm7}^{2}4_{\pm6}^{2},
1_{\pm2} 3_{\pm3} 5_{\pm3} 7_{\pm1} 8_{0} 2_{\pm7}^{2}4_{\pm6}^{2}, 
1_{\pm2} 3_{\pm3} 5_{\pm3} 6_{\pm2} 8_{0} 2_{\pm7}^{2}4_{\pm6}^{2},  \\
& 1_{\pm2} 3_{\pm3} 8_{0} 2_{\pm7} 4_{\pm6},
1_{\pm2} 5_{\pm3} 8_{0} 2_{\pm7} 4_{\pm6}, 
1_{\pm2} 3_{\pm3} 6_{\pm2} 8_{0} 2_{\pm7}^{2}4_{\pm6}, 
1_{\pm2}^{2}3_{\pm3} 6_{\pm2} 7_{\pm1} 8_{0} 2_{\pm7}^{2}4_{\pm6}^{2},
1_{\pm2}^{2}3_{\pm3}^{2}5_{\pm3} 6_{\pm2} 7_{\pm1} 8_{0} 2_{\pm7}^{3}  4_{\pm6}^{3}, \\
& 1_{\pm2}^{2}3_{\pm3} 5_{\pm3} 7_{\pm1} 8_{0} 2_{\pm7}^{2}4_{\pm6}^{2},
1_{\pm2} 3_{\pm3} 6_{\pm2} 7_{\pm1} 8_{0} 2_{\pm7} 4_{\pm6}^{2},
1_{\pm2} 3_{\pm3}^{2}5_{\pm3} 6_{\pm2} 7_{\pm1} 8_{0} 2_{\pm7}^{2}4_{\pm6}^{3},
1_{\pm2}^{2}3_{\pm3}^{2}5_{\pm3} 6_{\pm2} 7_{\pm1} 8_{0}^{2}2_{\pm7}^{3}  4_{\pm6}^{3}, \\
& 1_{\pm2} 3_{\pm3} 5_{\pm3} 7_{\pm1} 8_{0} 2_{\pm7} 4_{\pm6}^{2},
1_{\pm2} 3_{\pm3}^{2}6_{\pm2} 7_{\pm1} 8_{0} 2_{\pm7}^{2}4_{\pm6}^{2},
1_{\pm2}^{2}3_{\pm3} 5_{\pm3} 6_{\pm2} 7_{\pm1} 8_{0} 2_{\pm7}^{3}  4_{\pm6}^{2},
1_{\pm2}^{2}3_{\pm3}^{2}5_{\pm3} 6_{\pm2} 7_{\pm1}^{2}8_{0} 2_{\pm7}^{3}  4_{\pm6}^{3}, \\
& 1_{\pm2} 3_{\pm3} 5_{\pm3} 6_{\pm2} 7_{\pm1} 8_{0} 2_{\pm7}^{2}4_{\pm6}^{2},
1_{\pm2}^{2}3_{\pm3} 5_{\pm3} 6_{\pm2} 8_{0} 2_{\pm7}^{2}4_{\pm6}^{2},
1_{\pm2}^{2}3_{\pm3}^{2}5_{\pm3} 6_{\pm2} 7_{\pm1} 8_{0} 2_{\pm7}^{2}4_{\pm6}^{3},
1_{\pm2}^{2}3_{\pm3}^{2}5_{\pm3} 6_{\pm2}^{2}7_{\pm1} 8_{0} 2_{\pm7}^{3}  4_{\pm6}^{3}, \\
& 1_{\pm2} 3_{\pm3} 5_{\pm3} 6_{\pm2} 8_{0} 2_{\pm7} 4_{\pm6}^{2},
1_{\pm2} 3_{\pm3} 5_{\pm3} 8_{0} 2_{\pm7}^{2}4_{\pm6},
1_{\pm2} 3_{\pm3}^{2}5_{\pm3} 7_{\pm1} 8_{0} 2_{\pm7}^{2}4_{\pm6}^{2},
1_{\pm2} 3_{\pm3}^{2}5_{\pm3} 6_{\pm2} 7_{\pm1} 8_{0} 2_{\pm7}^{3}  4_{\pm6}^{2}, \\
& 1_{\pm2}^{2}3_{\pm3}^{2}5_{\pm3}^{2}6_{\pm2} 7_{\pm1} 8_{0} 2_{\pm7}^{3}  4_{\pm6}^{3},
1_{\pm2} 3_{\pm3}^{2}5_{\pm3} 6_{\pm2} 8_{0} 2_{\pm7}^{2}4_{\pm6}^{2},
1_{\pm2} 3_{\pm3} 7_{\pm1} 8_{0} 2_{\pm7} 4_{\pm6}, 
1_{\pm2} 3_{\pm3} 7_{\pm1} 8_{0} 2_{\pm7}^{2}4_{\pm6}, \\
& 1_{\pm2} 3_{\pm3} 6_{\pm2} 7_{\pm1} 8_{0} 2_{\pm7}^{2}4_{\pm6},
1_{\pm2}^{2}3_{\pm3} 5_{\pm3} 6_{\pm2} 7_{\pm1} 8_{0} 2_{\pm7}^{2}4_{\pm6}^{2},
1_{\pm2}^{2}3_{\pm3}^{2}5_{\pm3} 6_{\pm2} 7_{\pm1} 8_{0} 2_{\pm7}^{3}  4_{\pm6}^{2},
1_{\pm2} 3_{\pm3} 6_{\pm2} 8_{0} 2_{\pm7} 4_{\pm6}, \\
& 1_{\pm2} 3_{\pm3} 5_{\pm3} 7_{\pm1} 8_{0} 2_{\pm7}^{2}4_{\pm6},
1_{\pm2} 3_{\pm3}^{2}5_{\pm3} 6_{\pm2} 7_{\pm1} 8_{0} 2_{\pm7}^{2}4_{\pm6}^{2},
1_{\pm2} 3_{\pm3} 5_{\pm3} 6_{\pm2} 8_{0} 2_{\pm7}^{2}4_{\pm6},
1_{\pm2} 3_{\pm3} 5_{\pm3} 8_{0} 2_{\pm7} 4_{\pm6}, \\
& 1_{\pm2}^{2}3_{\pm3} 5_{\pm3} 6_{\pm2} 7_{\pm1} 8_{0} 2_{\pm7}^{2}4_{\pm6}^{3}.
\end{align*}
\end{gather}

%Let $\xi=(0,-1,-1,-2,-3,-4,-5,-6)$. 
%\begin{align*}
%\xymatrix{
%& & 2  \ar[d] &  &  \\
%1  \ar[r] & 3 \ar[r]  & 4  \ar[r] & 5  \ar[r] &  6 \ar[r] &  7  \ar[r]  & 8}
%\end{align*}
%The highest $l$-weight monomials of Hernandez-Leclerc modules of type $E_8$ that are not of type $A$, $D$ or $E_7$ are 
\begin{gather}
\begin{align*}
& (16)  1_{0} 2_{\pm1} 8_{\pm8}, 
1_{0} 2_{\pm1} 4_{\pm4} 8_{\pm8},
1_{0} 2_{\pm1} 3_{\pm1} 4_{\pm4} 5_{\pm5} 8_{\pm8},
1_{0} 2_{\pm1}^{2}3_{\pm1} 4_{\pm4} 5_{\pm5} 6_{\pm6} 8_{\pm8},
1_{0}^{2}2_{\pm1}^{2}3_{\pm1} 4_{\pm4} 5_{\pm5} 6_{\pm6} 7_{\pm7} 8_{\pm8}, \\
& 1_{0} 2_{\pm1} 3_{\pm1} 4_{\pm4} 6_{\pm6} 8_{\pm8},
1_{0} 2_{\pm1}^{2}3_{\pm1} 4_{\pm4} 5_{\pm5} 7_{\pm7} 8_{\pm8},
1_{0}^{2}2_{\pm1}^{3}3_{\pm1}^{2}4_{\pm4}^{2}5_{\pm5} 6_{\pm6} 7_{\pm7} 8_{\pm8},
1_{0} 2_{\pm1}^{2}3_{\pm1}^{2}4_{\pm4} 5_{\pm5} 6_{\pm6} 7_{\pm7} 8_{\pm8}, \\
& 1_{0} 2_{\pm1}^{2}3_{\pm1} 4_{\pm4} 6_{\pm6} 7_{\pm7} 8_{\pm8},
1_{0} 2_{\pm1} 3_{\pm1} 4_{\pm4} 7_{\pm7} 8_{\pm8},
1_{0} 2_{\pm1} 5_{\pm5} 8_{\pm8},
1_{0} 2_{\pm1} 3_{\pm1} 6_{\pm6} 8_{\pm8},
1_{0}^{2}2_{\pm1}^{2}3_{\pm1} 4_{\pm4} 5_{\pm5} 6_{\pm6} 8_{\pm8}, \\
& 1_{0}^{2}2_{\pm1}^{3}3_{\pm1}^{2}4_{\pm4} 5_{\pm5}^{2}6_{\pm6} 7_{\pm7} 8_{\pm8},
1_{0}^{2}2_{\pm1}^{3}3_{\pm1} 4_{\pm4} 5_{\pm5} 6_{\pm6} 7_{\pm7} 8_{\pm8},
1_{0}^{2}2_{\pm1}^{2}3_{\pm1} 4_{\pm4} 5_{\pm5} 7_{\pm7} 8_{\pm8},
1_{0} 2_{\pm1}^{2}3_{\pm1} 4_{\pm4} 5_{\pm5} 8_{\pm8}, \\
& 1_{0} 2_{\pm1} 3_{\pm1} 5_{\pm5} 6_{\pm6} 8_{\pm8},
1_{0}^{2}2_{\pm1}^{3}3_{\pm1}^{2}4_{\pm4} 5_{\pm5} 6_{\pm6}^{2}7_{\pm7} 8_{\pm8},
1_{0}^{2}2_{\pm1}^{2}3_{\pm1}^{2}4_{\pm4} 5_{\pm5} 6_{\pm6} 7_{\pm7} 8_{\pm8},
1_{0} 2_{\pm1}^{2}3_{\pm1}^{2}4_{\pm4} 5_{\pm5} 6_{\pm6} 8_{\pm8}, \\
& 1_{0} 2_{\pm1}^{2}3_{\pm1} 5_{\pm5} 6_{\pm6} 7_{\pm7} 8_{\pm8},
1_{0}^{2}2_{\pm1}^{3}3_{\pm1}^{2}4_{\pm4} 5_{\pm5} 6_{\pm6} 7_{\pm7}^{2}8_{\pm8},
1_{0} 2_{\pm1}^{3}3_{\pm1}^{2}4_{\pm4} 5_{\pm5} 6_{\pm6} 7_{\pm7} 8_{\pm8},
1_{0} 2_{\pm1} 6_{\pm6} 8_{\pm8}, \\
& 1_{0} 2_{\pm1} 3_{\pm1} 5_{\pm5} 7_{\pm7} 8_{\pm8},
1_{0}^{2}2_{\pm1}^{2}3_{\pm1} 4_{\pm4} 6_{\pm6} 7_{\pm7} 8_{\pm8},
1_{0}^{2}2_{\pm1}^{3}3_{\pm1}^{2}4_{\pm4} 5_{\pm5} 6_{\pm6} 7_{\pm7} 8_{\pm8}^{2};
1_{0}^{2}2_{\pm1}^{2}3_{\pm1} 5_{\pm5} 6_{\pm6} 7_{\pm7} 8_{\pm8}, \\
& 1_{0} 2_{\pm1} 7_{\pm7} 8_{\pm8}, 
1_{0} 2_{\pm1} 3_{\pm1} 6_{\pm6} 7_{\pm7} 8_{\pm8},
1_{0} 2_{\pm1}^{2}3_{\pm1} 6_{\pm6} 7_{\pm7} 8_{\pm8}, 
1_{0} 2_{\pm1}^{2}3_{\pm1} 4_{\pm4} 6_{\pm6} 8_{\pm8},
1_{0} 2_{\pm1}^{2}3_{\pm1}^{2}4_{\pm4} 5_{\pm5} 7_{\pm7} 8_{\pm8}, \\
& 1_{0}^{2}2_{\pm1}^{3}3_{\pm1}^{2}4_{\pm4} 5_{\pm5} 6_{\pm6} 7_{\pm7} 8_{\pm8},
1_{0} 2_{\pm1}^{2}3_{\pm1}^{2}4_{\pm4} 6_{\pm6} 7_{\pm7} 8_{\pm8},
1_{0} 2_{\pm1}^{2}3_{\pm1} 4_{\pm4} 7_{\pm7} 8_{\pm8},
1_{0} 2_{\pm1} 3_{\pm1} 4_{\pm4} 8_{\pm8}, \\
& 1_{0} 2_{\pm1}^{2}3_{\pm1} 5_{\pm5} 6_{\pm6} 8_{\pm8},
1_{0} 2_{\pm1}^{2}3_{\pm1}^{2}5_{\pm5} 6_{\pm6} 7_{\pm7} 8_{\pm8},
1_{0} 2_{\pm1}^{2}3_{\pm1} 5_{\pm5} 7_{\pm7} 8_{\pm8},
1_{0} 2_{\pm1} 3_{\pm1} 5_{\pm5} 8_{\pm8}, 
1_{0} 2_{\pm1} 3_{\pm1} 7_{\pm7} 8_{\pm8}.
\end{align*}
\end{gather}

%Let $\xi=(0,-1,-1,-2,-3,-4,-5,-4)$. 
%\begin{align*}
%\xymatrix{
%& & 2  \ar[d] &  &  \\
%1  \ar[r] & 3 \ar[r]  & 4  \ar[r] & 5  \ar[r] &  6 \ar[r] &  7& 8  \ar[l] }
%\end{align*}
%The highest $l$-weight monomials of Hernandez-Leclerc modules of type $E_8$ that are not of type $A$, $D$ or $E_7$ are 
\begin{gather}
\begin{align*}
& (17)  1_{0} 2_{\pm1} 8_{\pm4} 4_{\pm4} 7_{\pm7},
1_{0} 2_{\pm1} 3_{\pm1} 8_{\pm4} 4_{\pm4} 5_{\pm5} 7_{\pm7},
1_{0} 2_{\pm1}^{2} 3_{\pm1} 8_{\pm4} 4_{\pm4} 5_{\pm5} 6_{\pm6} 7_{\pm7},
1_{0}^{2} 2_{\pm1}^{2} 3_{\pm1} 8_{\pm4} 4_{\pm4} 5_{\pm5} 6_{\pm6} 7_{\pm7}^{2}, \\
& 1_{0} 2_{\pm1} 3_{\pm1} 8_{\pm4} 4_{\pm4} 6_{\pm6} 7_{\pm7},
1_{0} 2_{\pm1}^{2} 3_{\pm1} 8_{\pm4} 4_{\pm4} 5_{\pm5} 7_{\pm7}^{2},
1_{0}^{2} 2_{\pm1}^{3} 3_{\pm1}^{2} 8_{\pm4} 4_{\pm4}^{2} 5_{\pm5} 6_{\pm6} 7_{\pm7}^{2},
1_{0} 2_{\pm1}^{2} 3_{\pm1}^{2} 8_{\pm4} 4_{\pm4} 5_{\pm5} 6_{\pm6} 7_{\pm7}^{2}, \\
& 1_{0} 2_{\pm1}^{2} 3_{\pm1} 8_{\pm4} 4_{\pm4} 6_{\pm6} 7_{\pm7}^{2},
1_{0} 2_{\pm1} 3_{\pm1} 8_{\pm4} 4_{\pm4} 7_{\pm7}^{2},
1_{0} 2_{\pm1} 8_{\pm4} 5_{\pm5} 7_{\pm7},
1_{0} 2_{\pm1} 8_{\pm4} 7_{\pm7},
1_{0}^{2} 2_{\pm1}^{2} 3_{\pm1} 8_{\pm4} 4_{\pm4} 5_{\pm5} 6_{\pm6} 7_{\pm7}, \\
& 1_{0}^{2} 2_{\pm1}^{3} 3_{\pm1}^{2} 8_{\pm4} 4_{\pm4} 5_{\pm5}^{2} 6_{\pm6} 7_{\pm7}^{2},
1_{0}^{2} 2_{\pm1}^{3} 3_{\pm1} 8_{\pm4} 4_{\pm4} 5_{\pm5} 6_{\pm6} 7_{\pm7}^{2},
1_{0}^{2} 2_{\pm1}^{2} 3_{\pm1} 8_{\pm4} 4_{\pm4} 5_{\pm5} 7_{\pm7}^{2},
1_{0} 2_{\pm1}^{2} 3_{\pm1} 8_{\pm4} 4_{\pm4} 5_{\pm5} 7_{\pm7}, \\
& 1_{0} 2_{\pm1} 3_{\pm1} 8_{\pm4} 5_{\pm5} 6_{\pm6} 7_{\pm7},
1_{0}^{2} 2_{\pm1}^{3} 3_{\pm1}^{2} 8_{\pm4} 4_{\pm4} 5_{\pm5} 6_{\pm6}^{2} 7_{\pm7}^{2},
1_{0}^{2} 2_{\pm1}^{2} 3_{\pm1}^{2} 8_{\pm4} 4_{\pm4} 5_{\pm5} 6_{\pm6} 7_{\pm7}^{2},
1_{0} 2_{\pm1}^{2} 3_{\pm1}^{2} 8_{\pm4} 4_{\pm4} 5_{\pm5} 6_{\pm6} 7_{\pm7}, \\
& 1_{0} 2_{\pm1}^{2} 3_{\pm1} 8_{\pm4} 5_{\pm5} 6_{\pm6} 7_{\pm7}^{2},
1_{0}^{2} 2_{\pm1}^{3} 3_{\pm1}^{2} 8_{\pm4} 4_{\pm4} 5_{\pm5} 6_{\pm6} 7_{\pm7}^{3},
1_{0} 2_{\pm1}^{3} 3_{\pm1}^{2} 8_{\pm4} 4_{\pm4} 5_{\pm5} 6_{\pm6} 7_{\pm7}^{2},
1_{0} 2_{\pm1} 8_{\pm4} 6_{\pm6} 7_{\pm7}, \\
& 1_{0} 2_{\pm1} 3_{\pm1} 8_{\pm4} 5_{\pm5} 7_{\pm7}^{2},
1_{0}^{2} 2_{\pm1}^{2} 3_{\pm1} 8_{\pm4} 4_{\pm4} 6_{\pm6} 7_{\pm7}^{2},
1_{0}^{2} 2_{\pm1}^{3} 3_{\pm1}^{2} 8_{\pm4}^{2} 4_{\pm4} 5_{\pm5} 6_{\pm6} 7_{\pm7}^{3},
1_{0}^{2} 2_{\pm1}^{2} 3_{\pm1} 8_{\pm4} 5_{\pm5} 6_{\pm6} 7_{\pm7}^{2}, \\
& 1_{0} 2_{\pm1} 3_{\pm1} 8_{\pm4} 6_{\pm6} 7_{\pm7}^{2},
1_{0} 2_{\pm1} 8_{\pm4} 7_{\pm7}^{2},
1_{0} 2_{\pm1} 3_{\pm1} 8_{\pm4} 7_{\pm7}^{2},
1_{0} 2_{\pm1}^{2} 3_{\pm1} 8_{\pm4} 4_{\pm4} 6_{\pm6} 7_{\pm7},
1_{0} 2_{\pm1}^{2} 3_{\pm1}^{2} 8_{\pm4} 4_{\pm4} 5_{\pm5} 7_{\pm7}^{2}, \\
& 1_{0}^{2} 2_{\pm1}^{3} 3_{\pm1}^{2} 8_{\pm4} 4_{\pm4} 5_{\pm5} 6_{\pm6} 7_{\pm7}^{2},
1_{0} 2_{\pm1}^{2} 3_{\pm1}^{2} 8_{\pm4} 4_{\pm4} 6_{\pm6} 7_{\pm7}^{2},
1_{0} 2_{\pm1}^{2} 3_{\pm1} 8_{\pm4} 4_{\pm4} 7_{\pm7}^{2},
1_{0} 2_{\pm1} 3_{\pm1} 8_{\pm4} 4_{\pm4} 7_{\pm7}, 
1_{0} 2_{\pm1} 3_{\pm1} 8_{\pm4} 6_{\pm6} 7_{\pm7}, \\
& 1_{0} 2_{\pm1}^{2} 3_{\pm1} 8_{\pm4} 5_{\pm5} 6_{\pm6} 7_{\pm7},
1_{0} 2_{\pm1}^{2} 3_{\pm1}^{2} 8_{\pm4} 5_{\pm5} 6_{\pm6} 7_{\pm7}^{2},
1_{0} 2_{\pm1}^{2} 3_{\pm1} 8_{\pm4} 5_{\pm5} 7_{\pm7}^{2},
1_{0} 2_{\pm1} 3_{\pm1} 8_{\pm4} 5_{\pm5} 7_{\pm7}, 
1_{0} 2_{\pm1}^{2} 3_{\pm1} 8_{\pm4} 6_{\pm6} 7_{\pm7}^{2}.
\end{align*}
\end{gather}

%Let $\xi=(0,-1,-1,-2,-3,-4_{\pm3},-4)$. 
%\begin{align*}
%\xymatrix{
%& & 2  \ar[d] &  &  \\
%1  \ar[r] & 3 \ar[r]  & 4  \ar[r] & 5 \ar[r] &  6 &  7   \ar[l]  \ar[r] & 8 }
%\end{align*}
%The highest $l$-weight monomials of Hernandez-Leclerc modules of type $E_8$ that are not of type $A$, $D$ or $E_7$ are 
\begin{gather}
\begin{align*}
& (18)  1_{0} 2_{\pm1} 7_{\pm3} 4_{\pm4} 6_{\pm6} 8_{\pm6},
1_{0} 2_{\pm1} 3_{\pm1} 7_{\pm3} 4_{\pm4} 5_{\pm5} 6_{\pm6} 8_{\pm6},
1_{0} 2_{\pm1}^{2} 3_{\pm1} 7_{\pm3} 4_{\pm4} 5_{\pm5} 6_{\pm6}^{2} 8_{\pm6},
1_{0} 2_{\pm1} 7_{\pm3} 5_{\pm5} 6_{\pm6} 8_{\pm6}, \\
& 1_{0} 2_{\pm1} 3_{\pm1} 7_{\pm3} 4_{\pm4} 6_{\pm6}^{2} 8_{\pm6},
1_{0}^{2} 2_{\pm1}^{2} 3_{\pm1} 7_{\pm3}^{2} 4_{\pm4} 5_{\pm5} 6_{\pm6}^{3} 8_{\pm6},
1_{0}^{2} 2_{\pm1}^{2} 3_{\pm1} 7_{\pm3} 4_{\pm4} 5_{\pm5} 6_{\pm6}^{2} 8_{\pm6},
1_{0} 2_{\pm1} 3_{\pm1} 7_{\pm3} 5_{\pm5} 6_{\pm6}^{2} 8_{\pm6}, \\
& 1_{0} 2_{\pm1}^{2} 3_{\pm1} 7_{\pm3}^{2} 4_{\pm4} 5_{\pm5} 6_{\pm6}^{2} 8_{\pm6},
1_{0}^{2} 2_{\pm1}^{3} 3_{\pm1}^{2} 7_{\pm3}^{2} 4_{\pm4}^{2} 5_{\pm5} 6_{\pm6}^{3} 8_{\pm6},
1_{0} 2_{\pm1}^{2} 3_{\pm1}^{2} 7_{\pm3}^{2} 4_{\pm4} 5_{\pm5} 6_{\pm6}^{3} 8_{\pm6},
1_{0} 2_{\pm1} 7_{\pm3} 6_{\pm6}^{2} 8_{\pm6}, \\
& 1_{0} 2_{\pm1}^{2} 3_{\pm1} 7_{\pm3}^{2} 4_{\pm4} 6_{\pm6}^{3} 8_{\pm6},
1_{0} 2_{\pm1}^{2} 3_{\pm1} 7_{\pm3} 4_{\pm4} 5_{\pm5} 6_{\pm6} 8_{\pm6},
1_{0}^{2} 2_{\pm1}^{3} 3_{\pm1}^{2} 7_{\pm3}^{2} 4_{\pm4} 5_{\pm5}^{2} 6_{\pm6}^{3} 8_{\pm6},
1_{0}^{2} 2_{\pm1}^{3} 3_{\pm1} 7_{\pm3}^{2} 4_{\pm4} 5_{\pm5} 6_{\pm6}^{3} 8_{\pm6}, \\
& 1_{0} 2_{\pm1} 3_{\pm1} 7_{\pm3}^{2} 4_{\pm4} 6_{\pm6}^{2} 8_{\pm6},
1_{0}^{2} 2_{\pm1}^{2} 3_{\pm1} 7_{\pm3}^{2} 4_{\pm4} 5_{\pm5} 6_{\pm6}^{2} 8_{\pm6},
1_{0} 2_{\pm1}^{2} 3_{\pm1}^{2} 7_{\pm3} 4_{\pm4} 5_{\pm5} 6_{\pm6}^{2} 8_{\pm6},
1_{0}^{2} 2_{\pm1}^{3} 3_{\pm1}^{2} 7_{\pm3}^{2} 4_{\pm4} 5_{\pm5} 6_{\pm6}^{4} 8_{\pm6}, \\
& 1_{0}^{2} 2_{\pm1}^{2} 3_{\pm1}^{2} 7_{\pm3}^{2} 4_{\pm4} 5_{\pm5} 6_{\pm6}^{3} 8_{\pm6},
1_{0} 2_{\pm1}^{2} 3_{\pm1} 7_{\pm3} 4_{\pm4} 6_{\pm6}^{2} 8_{\pm6},
1_{0} 2_{\pm1}^{2} 3_{\pm1} 7_{\pm3}^{2} 5_{\pm5} 6_{\pm6}^{3} 8_{\pm6},
1_{0}^{2} 2_{\pm1}^{3} 3_{\pm1}^{2} 7_{\pm3}^{3} 4_{\pm4} 5_{\pm5} 6_{\pm6}^{4} 8_{\pm6}^{2}, \\
& 1_{0} 2_{\pm1}^{3} 3_{\pm1}^{2} 7_{\pm3}^{2} 4_{\pm4} 5_{\pm5} 6_{\pm6}^{3} 8_{\pm6},
1_{0} 2_{\pm1}^{2} 3_{\pm1} 7_{\pm3} 5_{\pm5} 6_{\pm6}^{2} 8_{\pm6},
1_{0} 2_{\pm1} 3_{\pm1} 7_{\pm3}^{2} 5_{\pm5} 6_{\pm6}^{2} 8_{\pm6},
1_{0}^{2} 2_{\pm1}^{2} 3_{\pm1} 7_{\pm3}^{2} 4_{\pm4} 6_{\pm6}^{3} 8_{\pm6}, \\
& 1_{0}^{2} 2_{\pm1}^{3} 3_{\pm1}^{2} 7_{\pm3}^{3} 4_{\pm4} 5_{\pm5} 6_{\pm6}^{4} 8_{\pm6},
1_{0}^{2} 2_{\pm1}^{2} 3_{\pm1} 7_{\pm3}^{2} 5_{\pm5} 6_{\pm6}^{3} 8_{\pm6},
1_{0} 2_{\pm1} 3_{\pm1} 7_{\pm3}^{2} 6_{\pm6}^{3} 8_{\pm6},
1_{0} 2_{\pm1} 3_{\pm1} 7_{\pm3} 4_{\pm4} 6_{\pm6} 8_{\pm6}, \\
& 1_{0} 2_{\pm1}^{2} 3_{\pm1}^{2} 7_{\pm3}^{2} 4_{\pm4} 5_{\pm5} 6_{\pm6}^{2} 8_{\pm6},
1_{0}^{2} 2_{\pm1}^{3} 3_{\pm1}^{2} 7_{\pm3}^{2} 4_{\pm4} 5_{\pm5} 6_{\pm6}^{3} 8_{\pm6},
1_{0} 2_{\pm1}^{2} 3_{\pm1}^{2} 7_{\pm3}^{2} 4_{\pm4} 6_{\pm6}^{3} 8_{\pm6},
1_{0} 2_{\pm1} 7_{\pm3}^{2} 6_{\pm6}^{2} 8_{\pm6}, \\
& 1_{0} 2_{\pm1}^{2} 3_{\pm1} 7_{\pm3}^{2} 4_{\pm4} 6_{\pm6}^{2} 8_{\pm6},
1_{0} 2_{\pm1} 3_{\pm1} 7_{\pm3} 5_{\pm5} 6_{\pm6} 8_{\pm6},
1_{0} 2_{\pm1}^{2} 3_{\pm1}^{2} 7_{\pm3}^{2} 5_{\pm5} 6_{\pm6}^{3} 8_{\pm6},
1_{0} 2_{\pm1}^{2} 3_{\pm1} 7_{\pm3}^{2} 5_{\pm5} 6_{\pm6}^{2} 8_{\pm6}, \\
& 1_{0} 2_{\pm1} 3_{\pm1} 7_{\pm3} 6_{\pm6}^{2} 8_{\pm6},
1_{0} 2_{\pm1}^{2} 3_{\pm1} 7_{\pm3}^{2} 6_{\pm6}^{3} 8_{\pm6},
1_{0} 2_{\pm1} 7_{\pm3} 6_{\pm6} 8_{\pm6},
1_{0} 2_{\pm1} 3_{\pm1} 7_{\pm3}^{2} 6_{\pm6}^{2} 8_{\pm6}.
\end{align*}
\end{gather}

%Let $\xi=(0,-1,-1,-2,-3,-4_{\pm3},-2)$. 
%\begin{align*}
%\xymatrix{
%& & 2  \ar[d] &  &  \\
%1  \ar[r] & 3 \ar[r]  & 4  \ar[r] & 5  \ar[r]  &  6  &  7  \ar[l] & 8  \ar[l] }
%\end{align*}
%The highest $l$-weight monomials of Hernandez-Leclerc modules of type $E_8$ that are not of type $A$, $D$ or $E_7$ are 
\begin{gather}
\begin{align*}
& (19)  1_{0} 2_{\pm1} 8_{\pm2} 4_{\pm4} 6_{\pm6},
1_{0} 2_{\pm1} 3_{\pm1} 8_{\pm2} 4_{\pm4} 5_{\pm5} 6_{\pm6},
1_{0} 2_{\pm1}^{2} 3_{\pm1} 8_{\pm2} 4_{\pm4} 5_{\pm5} 6_{\pm6}^{2},
1_{0} 2_{\pm1} 8_{\pm2} 5_{\pm5} 6_{\pm6}, 
1_{0} 2_{\pm1} 3_{\pm1} 8_{\pm2} 6_{\pm6}^{2}, \\
& 1_{0} 2_{\pm1} 3_{\pm1} 8_{\pm2} 4_{\pm4} 6_{\pm6}^{2},
1_{0}^{2} 2_{\pm1}^{2} 3_{\pm1} 7_{\pm3} 8_{\pm2} 4_{\pm4} 5_{\pm5} 6_{\pm6}^{3},
1_{0}^{2} 2_{\pm1}^{2} 3_{\pm1} 8_{\pm2} 4_{\pm4} 5_{\pm5} 6_{\pm6}^{2},
1_{0} 2_{\pm1} 3_{\pm1} 8_{\pm2} 5_{\pm5} 6_{\pm6}^{2},  \\
& 1_{0} 2_{\pm1}^{2} 3_{\pm1} 7_{\pm3} 8_{\pm2} 4_{\pm4} 5_{\pm5} 6_{\pm6}^{2},
1_{0} 2_{\pm1}^{2} 3_{\pm1}^{2} 7_{\pm3} 8_{\pm2} 4_{\pm4} 5_{\pm5} 6_{\pm6}^{3},
1_{0} 2_{\pm1} 3_{\pm1} 7_{\pm3} 8_{\pm2} 6_{\pm6}^{2},
1_{0} 2_{\pm1} 8_{\pm2} 6_{\pm6}^{2}, 
1_{0} 2_{\pm1} 8_{\pm2} 6_{\pm6}, \\
& 1_{0} 2_{\pm1}^{2} 3_{\pm1} 7_{\pm3} 8_{\pm2} 4_{\pm4} 6_{\pm6}^{3},
1_{0} 2_{\pm1}^{2} 3_{\pm1} 8_{\pm2} 4_{\pm4} 5_{\pm5} 6_{\pm6},
1_{0}^{2} 2_{\pm1}^{3}  3_{\pm1}^{2} 7_{\pm3} 8_{\pm2} 4_{\pm4} 5_{\pm5}^{2} 6_{\pm6}^{3},
1_{0}^{2} 2_{\pm1}^{3}  3_{\pm1} 7_{\pm3} 8_{\pm2} 4_{\pm4} 5_{\pm5} 6_{\pm6}^{3}, \\
& 1_{0} 2_{\pm1} 3_{\pm1} 7_{\pm3} 8_{\pm2} 4_{\pm4} 6_{\pm6}^{2},
1_{0}^{2} 2_{\pm1}^{2} 3_{\pm1} 7_{\pm3} 8_{\pm2} 4_{\pm4} 5_{\pm5} 6_{\pm6}^{2},
1_{0} 2_{\pm1}^{2} 3_{\pm1}^{2} 8_{\pm2} 4_{\pm4} 5_{\pm5} 6_{\pm6}^{2},
1_{0}^{2} 2_{\pm1}^{3}  3_{\pm1}^{2} 7_{\pm3} 8_{\pm2} 4_{\pm4} 5_{\pm5} 6_{\pm6}^{4}, \\
& 1_{0}^{2} 2_{\pm1}^{2} 3_{\pm1}^{2} 7_{\pm3} 8_{\pm2} 4_{\pm4} 5_{\pm5} 6_{\pm6}^{3},
1_{0} 2_{\pm1}^{2} 3_{\pm1} 8_{\pm2} 4_{\pm4} 6_{\pm6}^{2},
1_{0} 2_{\pm1}^{2} 3_{\pm1} 7_{\pm3} 8_{\pm2} 5_{\pm5} 6_{\pm6}^{3},
1_{0}^{2} 2_{\pm1}^{3}  3_{\pm1}^{2} 7_{\pm3} 8_{\pm2}^{2} 4_{\pm4} 5_{\pm5} 6_{\pm6}^{4}, \\
& 1_{0} 2_{\pm1}^{3}  3_{\pm1}^{2} 7_{\pm3} 8_{\pm2} 4_{\pm4} 5_{\pm5} 6_{\pm6}^{3},
1_{0} 2_{\pm1}^{2} 3_{\pm1} 8_{\pm2} 5_{\pm5} 6_{\pm6}^{2},
1_{0} 2_{\pm1} 3_{\pm1} 7_{\pm3} 8_{\pm2} 5_{\pm5} 6_{\pm6}^{2}, 
1_{0}^{2} 2_{\pm1}^{2} 3_{\pm1} 7_{\pm3} 8_{\pm2} 4_{\pm4} 6_{\pm6}^{3}, \\
& 1_{0}^{2} 2_{\pm1}^{2} 3_{\pm1} 7_{\pm3} 8_{\pm2} 5_{\pm5} 6_{\pm6}^{3},
1_{0} 2_{\pm1} 3_{\pm1} 7_{\pm3} 8_{\pm2} 6_{\pm6}^{3},
1_{0} 2_{\pm1} 3_{\pm1} 8_{\pm2} 4_{\pm4} 6_{\pm6}, 
1_{0} 2_{\pm1} 3_{\pm1} 8_{\pm2} 5_{\pm5} 6_{\pm6}, \\
& 1_{0} 2_{\pm1}^{2} 3_{\pm1}^{2} 7_{\pm3} 8_{\pm2} 4_{\pm4} 5_{\pm5} 6_{\pm6}^{2},
1_{0}^{2} 2_{\pm1}^{3}  3_{\pm1}^{2} 7_{\pm3} 8_{\pm2} 4_{\pm4} 5_{\pm5} 6_{\pm6}^{3},
1_{0} 2_{\pm1}^{2} 3_{\pm1}^{2} 7_{\pm3} 8_{\pm2} 4_{\pm4} 6_{\pm6}^{3},
1_{0} 2_{\pm1} 7_{\pm3} 8_{\pm2} 6_{\pm6}^{2}, \\
& 1_{0} 2_{\pm1}^{2} 3_{\pm1} 7_{\pm3} 8_{\pm2} 4_{\pm4} 6_{\pm6}^{2},
1_{0} 2_{\pm1}^{2} 3_{\pm1}^{2} 7_{\pm3} 8_{\pm2} 5_{\pm5} 6_{\pm6}^{3},
1_{0} 2_{\pm1}^{2} 3_{\pm1} 7_{\pm3} 8_{\pm2} 5_{\pm5} 6_{\pm6}^{2},
1_{0} 2_{\pm1}^{2} 3_{\pm1} 7_{\pm3} 8_{\pm2} 6_{\pm6}^{3}, \\
& 1_{0}^{2} 2_{\pm1}^{3}  3_{\pm1}^{2} 7_{\pm3} 8_{\pm2} 4_{\pm4}^{2} 5_{\pm5} 6_{\pm6}^{3},
1_{0}^{2} 2_{\pm1}^{3}  3_{\pm1}^{2} 7_{\pm3}^{2} 8_{\pm2} 4_{\pm4} 5_{\pm5} 6_{\pm6}^{4}. \\
%\end{align*}
%\end{gather}
%%
& \quad \\
%%
%%
%%
%Let $\xi=(0,-3,-1,-2,-1,0,-1,-2)$. 
%\begin{align*}
%\xymatrix{
%& & 2 &  &  \\
%1  \ar[r] & 3 \ar[r]  & 4  \ar[u]  & 5 \ar[l]  &  6 \ar[l] \ar[r] &  7  \ar[r]  & 8}
%\end{align*}
%The highest $l$-weight monomials of Hernandez-Leclerc modules of type $E_8$ that are not of type $A$, $D$ or $E_7$ are 
%%
%%
%\begin{gather}
%\begin{align*}
& (20)  1_{0} 2_{\pm1} 6_{\pm2} 4_{\pm4} 5_{\pm5} 8_{\pm6},
1_{0} 2_{\pm1} 3_{\pm1} 6_{\pm2} 4_{\pm4} 5_{\pm5}^{2}  8_{\pm6},
1_{0} 2_{\pm1}^{2}  3_{\pm1} 6_{\pm2}^{2}  4_{\pm4} 5_{\pm5}^{3}   7_{\pm5} 8_{\pm6},
1_{0} 2_{\pm1} 6_{\pm2} 5_{\pm5}^{2}  8_{\pm6}, \\
& 1_{0} 2_{\pm1} 3_{\pm1} 6_{\pm2}^{2}  4_{\pm4} 5_{\pm5}^{2}  7_{\pm5} 8_{\pm6},
1_{0}^{2}  2_{\pm1}^{2}  3_{\pm1} 6_{\pm2}^{3}   4_{\pm4} 5_{\pm5}^{4}   7_{\pm5} 8_{\pm6},
1_{0}^{2}  2_{\pm1}^{2}  3_{\pm1} 6_{\pm2}^{2}  4_{\pm4} 5_{\pm5}^{3}   7_{\pm5} 8_{\pm6},
1_{0} 2_{\pm1}^{2}  3_{\pm1} 6_{\pm2}^{2}  4_{\pm4} 5_{\pm5}^{3}   8_{\pm6}, \\
& 1_{0}^{2}  2_{\pm1}^{3}   3_{\pm1}^{2}  6_{\pm2}^{3}   4_{\pm4}^{2}  5_{\pm5}^{4}   7_{\pm5} 8_{\pm6},
1_{0} 2_{\pm1} 3_{\pm1} 6_{\pm2}^{2}  5_{\pm5}^{3}   7_{\pm5} 8_{\pm6},
1_{0} 2_{\pm1}^{2}  3_{\pm1}^{2}  6_{\pm2}^{3}   4_{\pm4} 5_{\pm5}^{4}   7_{\pm5} 8_{\pm6},
1_{0} 2_{\pm1} 6_{\pm2}^{2}  5_{\pm5}^{2}  7_{\pm5} 8_{\pm6}, \\
& 1_{0} 2_{\pm1}^{2}  3_{\pm1} 6_{\pm2} 4_{\pm4} 5_{\pm5}^{2}  8_{\pm6},
1_{0}^{2}  2_{\pm1}^{3}   3_{\pm1}^{2}  6_{\pm2}^{3}   4_{\pm4} 5_{\pm5}^{5}   7_{\pm5} 8_{\pm6},
1_{0} 2_{\pm1}^{2}  3_{\pm1} 6_{\pm2}^{3}   4_{\pm4} 5_{\pm5}^{3}   7_{\pm5} 8_{\pm6},
1_{0}^{2}  2_{\pm1}^{3}   3_{\pm1} 6_{\pm2}^{3}   4_{\pm4} 5_{\pm5}^{4}   7_{\pm5} 8_{\pm6}, \\
& 1_{0} 2_{\pm1} 3_{\pm1} 6_{\pm2}^{2}  4_{\pm4} 5_{\pm5}^{2}  8_{\pm6},
1_{0} 2_{\pm1}^{2}  3_{\pm1}^{2}  6_{\pm2}^{2}  4_{\pm4} 5_{\pm5}^{3}   7_{\pm5} 8_{\pm6},
1_{0}^{2}  2_{\pm1}^{3}   3_{\pm1}^{2}  6_{\pm2}^{4}   4_{\pm4} 5_{\pm5}^{5}   7_{\pm5}^{2}  8_{\pm6},
1_{0}^{2}  2_{\pm1}^{2}  3_{\pm1} 6_{\pm2}^{2}  4_{\pm4} 5_{\pm5}^{3}   8_{\pm6}, \\
& 1_{0}^{2}  2_{\pm1}^{2}  3_{\pm1}^{2}  6_{\pm2}^{3}   4_{\pm4} 5_{\pm5}^{4}   7_{\pm5} 8_{\pm6},
1_{0} 2_{\pm1}^{2}  3_{\pm1} 6_{\pm2}^{2}  4_{\pm4} 5_{\pm5}^{2}  7_{\pm5} 8_{\pm6},
1_{0} 2_{\pm1}^{2}  3_{\pm1} 6_{\pm2}^{3}   5_{\pm5}^{4}   7_{\pm5} 8_{\pm6},
1_{0}^{2}  2_{\pm1}^{3}   3_{\pm1}^{2}  6_{\pm2}^{4} 4_{\pm4} 5_{\pm5}^{5} 7_{\pm5} 8_{\pm6}^{2}, \\
& 1_{0} 2_{\pm1}^{3}   3_{\pm1}^{2}  6_{\pm2}^{3}   4_{\pm4} 5_{\pm5}^{4}   7_{\pm5} 8_{\pm6},
1_{0} 2_{\pm1} 3_{\pm1} 6_{\pm2}^{2}  5_{\pm5}^{3}   8_{\pm6},
1_{0}^{2}  2_{\pm1}^{2}  3_{\pm1} 6_{\pm2}^{3}   4_{\pm4} 5_{\pm5}^{3}   7_{\pm5} 8_{\pm6},
1_{0}^{2}  2_{\pm1}^{3}   3_{\pm1}^{2}  6_{\pm2}^{4}   4_{\pm4} 5_{\pm5}^{5}   7_{\pm5} 8_{\pm6}, \\
& 1_{0} 2_{\pm1}^{2}  3_{\pm1} 6_{\pm2}^{2}  5_{\pm5}^{3}   7_{\pm5} 8_{\pm6},
1_{0}^{2}  2_{\pm1}^{2}  3_{\pm1} 6_{\pm2}^{3}   5_{\pm5}^{4}   7_{\pm5} 8_{\pm6},
1_{0} 2_{\pm1} 3_{\pm1} 6_{\pm2} 4_{\pm4} 5_{\pm5} 8_{\pm6},
1_{0} 2_{\pm1}^{2}  3_{\pm1}^{2}  6_{\pm2}^{2}  4_{\pm4} 5_{\pm5}^{3} 8_{\pm6}, \\
& 1_{0}^{2}  2_{\pm1}^{3}   3_{\pm1}^{2}  6_{\pm2}^{3}   4_{\pm4} 5_{\pm5}^{4}   7_{\pm5} 8_{\pm6},
1_{0} 2_{\pm1} 3_{\pm1} 6_{\pm2}^{3}   5_{\pm5}^{3}   7_{\pm5} 8_{\pm6},
1_{0} 2_{\pm1}^{2}  3_{\pm1}^{2}  6_{\pm2}^{3}   4_{\pm4} 5_{\pm5}^{3}   7_{\pm5} 8_{\pm6},
1_{0} 2_{\pm1} 6_{\pm2}^{2}  5_{\pm5}^{2}  8_{\pm6}, \\
& 1_{0} 2_{\pm1} 3_{\pm1} 6_{\pm2} 5_{\pm5}^{2}  8_{\pm6},
1_{0} 2_{\pm1}^{2}  3_{\pm1}^{2}  6_{\pm2}^{3}   5_{\pm5}^{4}   7_{\pm5} 8_{\pm6},
1_{0} 2_{\pm1}^{2}  3_{\pm1} 6_{\pm2}^{2}  4_{\pm4} 5_{\pm5}^{2}  8_{\pm6},
1_{0} 2_{\pm1}^{2}  3_{\pm1} 6_{\pm2}^{2}  5_{\pm5}^{3}  8_{\pm6}, \\
& 1_{0} 2_{\pm1} 3_{\pm1} 6_{\pm2}^{2}  5_{\pm5}^{2}  7_{\pm5} 8_{\pm6},
1_{0} 2_{\pm1}^{2}  3_{\pm1} 6_{\pm2}^{3}   5_{\pm5}^{3}   7_{\pm5} 8_{\pm6},
1_{0} 2_{\pm1} 6_{\pm2} 5_{\pm5} 8_{\pm6},
1_{0} 2_{\pm1} 3_{\pm1} 6_{\pm2}^{2}  5_{\pm5}^{2}  8_{\pm6}. \\
%\end{align*}
%\end{gather}
%%
& \quad \\
%%
%%
%%
%Let $\xi=(0,-3,-1,-2,-1,0,-1,-2)$. 
%\begin{align*}
%\xymatrix{
%& & 2 &  &  \\
%1  \ar[r] & 3 \ar[r]  & 4  \ar[u]  & 5 \ar[l]  &  6 \ar[l] \ar[r] &  7  \ar[r]  & 8}
%\end{align*}
%The highest $l$-weight monomials of Hernandez-Leclerc modules of type $E_8$ that are not of type $A$, $D$ or $E_7$ are 
%%
%%
%\begin{gather}
%\begin{align*}
& (21)  1_{0} 6_{0} 2_{\pm5} 8_{\pm4},
1_{0} 3_{\pm1} 6_{0}^{2}  2_{\pm5} 4_{\pm4}^{2}   7_{\pm3} 8_{\pm4},
1_{0} 6_{0}^{2}   2_{\pm5} 4_{\pm4} 7_{\pm3} 8_{\pm4},
1_{0} 3_{\pm1} 6_{0}^{2}   2_{\pm5} 4_{\pm4} 8_{\pm4},
1_{0} 3_{\pm1} 6_{0}^{3}    2_{\pm5}^{2}   4_{\pm4}^{2}   7_{\pm3} 8_{\pm4}, \\
& 1_{0} 3_{\pm1} 6_{0}^{2}   2_{\pm5} 4_{\pm4}^{2}   8_{\pm4},
1_{0} 6_{0}^{2}   2_{\pm5} 4_{\pm4} 8_{\pm4},
1_{0} 3_{\pm1} 5_{\pm1} 6_{0}^{2}   2_{\pm5}^{2}   4_{\pm4}^{2}   7_{\pm3} 8_{\pm4},
1_{0}^{2}   3_{\pm1} 5_{\pm1} 6_{0}^{3}    2_{\pm5}^{2}   4_{\pm4}^{3}    7_{\pm3} 8_{\pm4},
1_{0} 3_{\pm1} 5_{\pm1} 6_{0} 2_{\pm5} 4_{\pm4}^{2}   8_{\pm4}, \\
& 1_{0} 6_{0} 2_{\pm5} 4_{\pm4} 8_{\pm4},
1_{0} 3_{\pm1} 6_{0}^{2}   2_{\pm5}^{2}   4_{\pm4} 7_{\pm3} 8_{\pm4},
1_{0}^{2}  3_{\pm1} 6_{0}^{3}    2_{\pm5}^{2}   4_{\pm4}^{2}   7_{\pm3} 8_{\pm4},
1_{0}^{2}  3_{\pm1}^{2}   5_{\pm1} 6_{0}^{3}    2_{\pm5}^{3}    4_{\pm4}^{3}    7_{\pm3} 8_{\pm4},
1_{0}^{2}   3_{\pm1} 5_{\pm1} 6_{0}^{2}   2_{\pm5}^{2}   4_{\pm4}^{2}   7_{\pm3} 8_{\pm4}, \\
& 1_{0} 3_{\pm1} 6_{0}^{3}    2_{\pm5} 4_{\pm4}^{2}   7_{\pm3} 8_{\pm4},
1_{0} 3_{\pm1}^{2}   5_{\pm1} 6_{0}^{3}    2_{\pm5}^{2}   4_{\pm4}^{3}    7_{\pm3} 8_{\pm4},
1_{0}^{2}   3_{\pm1}^{2}   5_{\pm1} 6_{0}^{4}     2_{\pm5}^{3}    4_{\pm4}^{3}    7_{\pm3}^{2}   8_{\pm4},
1_{0} 3_{\pm1} 5_{\pm1} 6_{0}^{2}   2_{\pm5} 4_{\pm4}^{2}   7_{\pm3} 8_{\pm4}, \\
& 1_{0} 3_{\pm1} 6_{0} 2_{\pm5} 4_{\pm4} 8_{\pm4},
1_{0} 3_{\pm1} 5_{\pm1} 6_{0} 2_{\pm5} 4_{\pm4} 8_{\pm4}, 
1_{0} 3_{\pm1} 5_{\pm1} 6_{0}^{2}   2_{\pm5}^{2}  4_{\pm4} 8_{\pm4},
1_{0} 3_{\pm1}^{2}   6_{0}^{3}    2_{\pm5}^{2}   4_{\pm4}^{2}   7_{\pm3} 8_{\pm4},
1_{0} 3_{\pm1} 5_{\pm1} 6_{0}^{2}   2_{\pm5}^{2}   4_{\pm4}^{2}   8_{\pm4}, \\
& 1_{0}^{2}   3_{\pm1} 5_{\pm1} 6_{0}^{3}    2_{\pm5}^{3}    4_{\pm4}^{2}   7_{\pm3} 8_{\pm4},
1_{0}^{2}   3_{\pm1}^{2}   5_{\pm1} 6_{0}^{4}     2_{\pm5}^{3}    4_{\pm4}^{3}    7_{\pm3} 8_{\pm4}]^{2},
1_{0} 3_{\pm1} 6_{0}^{2}   2_{\pm5}^{2}   4_{\pm4} 8_{\pm4},
1_{0} 3_{\pm1} 5_{\pm1} 6_{0}^{3}    2_{\pm5}^{2}   4_{\pm4}^{2}   7_{\pm3} 8_{\pm4}, \\
& 1_{0}^{2}   3_{\pm1}^{2}   5_{\pm1} 6_{0}^{3}    2_{\pm5}^{2}   4_{\pm4}^{3}    7_{\pm3} 8_{\pm4},
1_{0}^{2}   3_{\pm1}^{2}   5_{\pm1} 6_{0}^{4}     2_{\pm5}^{3}    4_{\pm4}^{3}    7_{\pm3} 8_{\pm4},
1_{0} 5_{\pm1} 6_{0} 2_{\pm5} 4_{\pm4} 8_{\pm4},
1_{0}^{2}   3_{\pm1} 5_{\pm1} 6_{0}^{2}   2_{\pm5}^{2}   4_{\pm4}^{2}   8_{\pm4}, \\
& 1_{0} 3_{\pm1}^{2}   5_{\pm1} 6_{0}^{2}   2_{\pm5}^{2}   4_{\pm4}^{2}   7_{\pm3} 8_{\pm4},
1_{0} 3_{\pm1}^{2}   5_{\pm1} 6_{0}^{3}    2_{\pm5}^{3}    4_{\pm4}^{2}   7_{\pm3} 8_{\pm4},
1_{0}^{2}   3_{\pm1}^{2}   5_{\pm1}^{2}   6_{0}^{3}    2_{\pm5}^{3}    4_{\pm4}^{3}    7_{\pm3} 8_{\pm4},
1_{0} 3_{\pm1} 5_{\pm1} 6_{0} 2_{\pm5}^{2}   4_{\pm4} 8_{\pm4}, \\
& 1_{0} 3_{\pm1} 6_{0}^{2}   2_{\pm5} 4_{\pm4} 7_{\pm3} 8_{\pm4},
1_{0} 3_{\pm1} 6_{0}^{3}    2_{\pm5}^{2}   4_{\pm4} 7_{\pm3} 8_{\pm4},
1_{0}^{2}   3_{\pm1} 5_{\pm1} 6_{0}^{3}    2_{\pm5}^{2}   4_{\pm4}^{2}   7_{\pm3} 8_{\pm4},
1_{0}^{2}   3_{\pm1}^{2}   5_{\pm1} 6_{0}^{3}    2_{\pm5}^{3}    4_{\pm4}^{2}   7_{\pm3} 8_{\pm4}, \\
& 1_{0} 3_{\pm1} 5_{\pm1} 6_{0}^{2}   2_{\pm5}^{2}   4_{\pm4} 7_{\pm3} 8_{\pm4},
1_{0} 3_{\pm1} 5_{\pm1} 6_{0}^{2}   2_{\pm5} 4_{\pm4}^{2}   8_{\pm4},
1_{0} 3_{\pm1}^{2}   5_{\pm1} 6_{0}^{2}   2_{\pm5}^{2}   4_{\pm4}^{2}   8_{\pm4},
1_{0} 3_{\pm1}^{2}   5_{\pm1} 6_{0}^{3}    2_{\pm5}^{2}   4_{\pm4}^{2}   7_{\pm3} 8_{\pm4}.
\end{align*}
\end{gather}

%Let $\xi=(0,-1,-1,-2,-3,-2,-3,-2)$. 
%\begin{align*}
%\xymatrix{
%& & 2  \ar[d] &  &  \\
%1  \ar[r] & 3 \ar[r]  & 4  \ar[r] & 5 &  6  \ar[l]  \ar[r]  &  7  & 8 \ar[l] }
%\end{align*}
%The highest $l$-weight monomials of Hernandez-Leclerc modules of type $E_8$ that are not of type $A$, $D$ or $E_7$ are 
\begin{gather}
\begin{align*}
& (22)  1_{0} 2_{\pm1} 6_{\pm2} 8_{\pm2} 4_{\pm4} 5_{\pm5} 7_{\pm5},
1_{0} 2_{\pm1} 3_{\pm1} 6_{\pm2} 8_{\pm2} 4_{\pm4} 5_{\pm5}^{2} 7_{\pm5},
1_{0} 2_{\pm1}^{2} 3_{\pm1} 6_{\pm2}^{2} 8_{\pm2} 4_{\pm4} 5_{\pm5}^{3} 7_{\pm5}^{2},
1_{0} 2_{\pm1} 6_{\pm2} 8_{\pm2} 5_{\pm5}^{2} 7_{\pm5}, \\
& 1_{0} 2_{\pm1} 3_{\pm1} 6_{\pm2}^{2} 8_{\pm2} 4_{\pm4} 5_{\pm5}^{2} 7_{\pm5}^{2},
1_{0}^{2} 2_{\pm1}^{2} 3_{\pm1} 6_{\pm2}^{3} 8_{\pm2} 4_{\pm4} 5_{\pm5}^{4} 7_{\pm5}^{2},
1_{0}^{2} 2_{\pm1}^{2} 3_{\pm1} 6_{\pm2}^{2} 8_{\pm2} 4_{\pm4} 5_{\pm5}^{3} 7_{\pm5}^{2},
1_{0} 2_{\pm1}^{2} 3_{\pm1} 6_{\pm2}^{2} 8_{\pm2} 4_{\pm4} 5_{\pm5}^{3} 7_{\pm5}, \\
& 1_{0}^{2} 2_{\pm1}^{3} 3_{\pm1}^{2} 6_{\pm2}^{3} 8_{\pm2} 4_{\pm4}^{2} 5_{\pm5}^{4} 7_{\pm5}^{2},
1_{0} 2_{\pm1} 3_{\pm1} 6_{\pm2}^{2} 8_{\pm2} 5_{\pm5}^{3} 7_{\pm5}^{2},
1_{0} 2_{\pm1}^{2} 3_{\pm1}^{2} 6_{\pm2}^{3} 8_{\pm2} 4_{\pm4} 5_{\pm5}^{4} 7_{\pm5}^{2},
1_{0} 2_{\pm1} 6_{\pm2}^{2} 8_{\pm2} 5_{\pm5}^{2} 7_{\pm5}^{2}, \\
& 1_{0} 2_{\pm1}^{2} 3_{\pm1} 6_{\pm2} 8_{\pm2} 4_{\pm4} 5_{\pm5}^{2} 7_{\pm5},
1_{0}^{2} 2_{\pm1}^{3} 3_{\pm1}^{2} 6_{\pm2}^{3} 8_{\pm2} 4_{\pm4} 5_{\pm5}^{5} 7_{\pm5}^{2},
1_{0} 2_{\pm1}^{2} 3_{\pm1} 6_{\pm2}^{3} 8_{\pm2} 4_{\pm4} 5_{\pm5}^{3} 7_{\pm5}^{2},
1_{0}^{2} 2_{\pm1}^{3} 3_{\pm1} 6_{\pm2}^{3} 8_{\pm2} 4_{\pm4} 5_{\pm5}^{4} 7_{\pm5}^{2}, \\
& 1_{0} 2_{\pm1} 3_{\pm1} 6_{\pm2}^{2} 8_{\pm2} 4_{\pm4} 5_{\pm5}^{2} 7_{\pm5},
1_{0} 2_{\pm1}^{2} 3_{\pm1}^{2} 6_{\pm2}^{2} 8_{\pm2} 4_{\pm4} 5_{\pm5}^{3} 7_{\pm5}^{2},
1_{0}^{2} 2_{\pm1}^{3} 3_{\pm1}^{2} 6_{\pm2}^{4} 8_{\pm2} 4_{\pm4} 5_{\pm5}^{5} 7_{\pm5}^{3},
1_{0}^{2} 2_{\pm1}^{2} 3_{\pm1} 6_{\pm2}^{2} 8_{\pm2} 4_{\pm4} 5_{\pm5}^{3} 7_{\pm5}, \\
& 1_{0}^{2} 2_{\pm1}^{2} 3_{\pm1}^{2} 6_{\pm2}^{3} 8_{\pm2} 4_{\pm4} 5_{\pm5}^{4} 7_{\pm5}^{2},
1_{0} 2_{\pm1}^{2} 3_{\pm1} 6_{\pm2}^{2} 8_{\pm2} 4_{\pm4} 5_{\pm5}^{2} 7_{\pm5}^{2},
1_{0} 2_{\pm1}^{2} 3_{\pm1} 6_{\pm2}^{3} 8_{\pm2} 5_{\pm5}^{4} 7_{\pm5}^{2},
1_{0}^{2} 2_{\pm1}^{3} 3_{\pm1}^{2} 6_{\pm2}^{4} 8_{\pm2}^{2} 4_{\pm4} 5_{\pm5}^{5} 7_{\pm5}^{3}, \\
& 1_{0} 2_{\pm1}^{3} 3_{\pm1}^{2} 6_{\pm2}^{3} 8_{\pm2} 4_{\pm4} 5_{\pm5}^{4} 7_{\pm5}^{2},
1_{0} 2_{\pm1} 3_{\pm1} 6_{\pm2}^{2} 8_{\pm2} 5_{\pm5}^{3} 7_{\pm5},
1_{0}^{2} 2_{\pm1}^{2} 3_{\pm1} 6_{\pm2}^{3} 8_{\pm2} 4_{\pm4} 5_{\pm5}^{3} 7_{\pm5}^{2},
1_{0}^{2} 2_{\pm1}^{3} 3_{\pm1}^{2} 6_{\pm2}^{4} 8_{\pm2} 4_{\pm4} 5_{\pm5}^{5} 7_{\pm5}^{2}, \\
& 1_{0} 2_{\pm1}^{2} 3_{\pm1} 6_{\pm2}^{2} 8_{\pm2} 5_{\pm5}^{3} 7_{\pm5}^{2},
1_{0}^{2} 2_{\pm1}^{2} 3_{\pm1} 6_{\pm2}^{3} 8_{\pm2} 5_{\pm5}^{4} 7_{\pm5}^{2},
1_{0} 2_{\pm1} 3_{\pm1} 6_{\pm2} 8_{\pm2} 4_{\pm4} 5_{\pm5} 7_{\pm5},
1_{0} 2_{\pm1}^{2} 3_{\pm1}^{2} 6_{\pm2}^{2} 8_{\pm2} 4_{\pm4} 5_{\pm5}^{3} 7_{\pm5}, \\
& 1_{0}^{2} 2_{\pm1}^{3} 3_{\pm1}^{2} 6_{\pm2}^{3} 8_{\pm2} 4_{\pm4} 5_{\pm5}^{4} 7_{\pm5}^{2},
1_{0} 2_{\pm1} 3_{\pm1} 6_{\pm2}^{3} 8_{\pm2} 5_{\pm5}^{3} 7_{\pm5}^{2},
1_{0} 2_{\pm1}^{2} 3_{\pm1}^{2} 6_{\pm2}^{3} 8_{\pm2} 4_{\pm4} 5_{\pm5}^{3} 7_{\pm5}^{2},
1_{0} 2_{\pm1} 6_{\pm2}^{2} 8_{\pm2} 5_{\pm5}^{2} 7_{\pm5}, \\
& 1_{0} 2_{\pm1} 3_{\pm1} 6_{\pm2} 8_{\pm2} 5_{\pm5}^{2} 7_{\pm5},
1_{0} 2_{\pm1}^{2} 3_{\pm1}^{2} 6_{\pm2}^{3} 8_{\pm2} 5_{\pm5}^{4} 7_{\pm5}^{2},
1_{0} 2_{\pm1}^{2} 3_{\pm1} 6_{\pm2}^{2} 8_{\pm2} 4_{\pm4} 5_{\pm5}^{2} 7_{\pm5},
1_{0} 2_{\pm1}^{2} 3_{\pm1} 6_{\pm2}^{2} 8_{\pm2} 5_{\pm5}^{3} 7_{\pm5}, \\
& 1_{0} 2_{\pm1} 3_{\pm1} 6_{\pm2}^{2} 8_{\pm2} 5_{\pm5}^{2} 7_{\pm5}^{2},
1_{0} 2_{\pm1}^{2} 3_{\pm1} 6_{\pm2}^{3} 8_{\pm2} 5_{\pm5}^{3} 7_{\pm5}^{2},
1_{0} 2_{\pm1} 6_{\pm2} 8_{\pm2} 5_{\pm5} 7_{\pm5},
1_{0} 2_{\pm1} 3_{\pm1} 6_{\pm2}^{2} 8_{\pm2} 5_{\pm5}^{2} 7_{\pm5}.
\end{align*}
\end{gather}

%Let $\xi=(0,-1,-1,-2,-3,-2,-1,-2)$. 
%\begin{align*}
%\xymatrix{
%& & 2  \ar[d] &  &  \\
%1  \ar[r] & 3 \ar[r]  & 4  \ar[r] & 5 &  6  \ar[l]  &  7  \ar[l]   \ar[r]  & 8 }
%\end{align*}
%The highest $l$-weight monomials of Hernandez-Leclerc modules of type $E_8$ that are not of type $A$, $D$ or $E_7$ are 
\begin{gather}
\begin{align*}
& (23)  1_{0} 2_{\pm1} 7_{\pm1} 4_{\pm4} 5_{\pm5} 8_{\pm4},
1_{0} 2_{\pm1} 3_{\pm1} 7_{\pm1} 4_{\pm4} 5_{\pm5}^{2} 8_{\pm4},
1_{0} 2_{\pm1} 7_{\pm1} 5_{\pm5}^{2} 8_{\pm4},
1_{0} 2_{\pm1}^{2} 3_{\pm1} 7_{\pm1}^{2} 4_{\pm4} 5_{\pm5}^{3} 8_{\pm4}, \\
& 1_{0} 2_{\pm1}^{2} 3_{\pm1} 6_{\pm2} 7_{\pm1} 4_{\pm4} 5_{\pm5}^{3} 8_{\pm4},
1_{0} 2_{\pm1} 3_{\pm1} 7_{\pm1}^{2} 4_{\pm4} 5_{\pm5}^{2} 8_{\pm4},
1_{0}^{2} 2_{\pm1}^{2} 3_{\pm1} 6_{\pm2} 7_{\pm1}^{2} 4_{\pm4} 5_{\pm5}^{4} 8_{\pm4},
1_{0} 2_{\pm1}^{2} 3_{\pm1} 7_{\pm1} 4_{\pm4} 5_{\pm5}^{2} 8_{\pm4}, \\
& 1_{0}^{2} 2_{\pm1}^{2} 3_{\pm1} 7_{\pm1}^{2} 4_{\pm4} 5_{\pm5}^{3} 8_{\pm4},
1_{0} 2_{\pm1} 3_{\pm1} 6_{\pm2} 7_{\pm1} 4_{\pm4} 5_{\pm5}^{2} 8_{\pm4},
1_{0}^{2} 2_{\pm1}^{3} 3_{\pm1}^{2} 6_{\pm2} 7_{\pm1}^{2} 4_{\pm4}^{2} 5_{\pm5}^{4} 8_{\pm4},
1_{0} 2_{\pm1} 3_{\pm1} 7_{\pm1}^{2} 5_{\pm5}^{3} 8_{\pm4}, \\
& 1_{0} 2_{\pm1}^{2} 3_{\pm1}^{2} 6_{\pm2} 7_{\pm1}^{2} 4_{\pm4} 5_{\pm5}^{4} 8_{\pm4},
1_{0}^{2} 2_{\pm1}^{2} 3_{\pm1} 6_{\pm2} 7_{\pm1} 4_{\pm4} 5_{\pm5}^{3} 8_{\pm4},
1_{0}^{2} 2_{\pm1}^{3} 3_{\pm1}^{2} 6_{\pm2} 7_{\pm1}^{2} 4_{\pm4} 5_{\pm5}^{5} 8_{\pm4},
1_{0} 2_{\pm1} 7_{\pm1}^{2} 5_{\pm5}^{2} 8_{\pm4}, \\
& 1_{0} 2_{\pm1}^{2} 3_{\pm1} 6_{\pm2} 7_{\pm1}^{2} 4_{\pm4} 5_{\pm5}^{3} 8_{\pm4},
1_{0}^{2} 2_{\pm1}^{3} 3_{\pm1} 6_{\pm2} 7_{\pm1}^{2} 4_{\pm4} 5_{\pm5}^{4} 8_{\pm4},
1_{0} 2_{\pm1} 3_{\pm1} 7_{\pm1} 4_{\pm4} 5_{\pm5} 8_{\pm4},
1_{0} 2_{\pm1} 3_{\pm1} 6_{\pm2} 7_{\pm1} 5_{\pm5}^{3} 8_{\pm4}, \\
& 1_{0} 2_{\pm1}^{2} 3_{\pm1}^{2} 7_{\pm1}^{2} 4_{\pm4} 5_{\pm5}^{3} 8_{\pm4},
1_{0}^{2} 2_{\pm1}^{3} 3_{\pm1}^{2} 6_{\pm2} 7_{\pm1}^{3} 4_{\pm4} 5_{\pm5}^{5} 8_{\pm4}^{2}, 
1_{0}^{2} 2_{\pm1}^{2} 3_{\pm1}^{2} 6_{\pm2} 7_{\pm1}^{2} 4_{\pm4} 5_{\pm5}^{4} 8_{\pm4},
1_{0} 2_{\pm1}^{2} 3_{\pm1} 7_{\pm1}^{2} 4_{\pm4} 5_{\pm5}^{2} 8_{\pm4}, \\
& 1_{0} 2_{\pm1}^{2} 3_{\pm1} 6_{\pm2} 7_{\pm1}^{2} 5_{\pm5}^{4} 8_{\pm4},
1_{0}^{2} 2_{\pm1}^{3} 3_{\pm1}^{2} 6_{\pm2} 7_{\pm1}^{3} 4_{\pm4} 5_{\pm5}^{5} 8_{\pm4},
1_{0} 2_{\pm1}^{2} 3_{\pm1}^{2} 6_{\pm2} 7_{\pm1} 4_{\pm4} 5_{\pm5}^{3} 8_{\pm4},
1_{0} 2_{\pm1}^{3} 3_{\pm1}^{2} 6_{\pm2} 7_{\pm1}^{2} 4_{\pm4} 5_{\pm5}^{4} 8_{\pm4}, \\
& 1_{0} 2_{\pm1} 6_{\pm2} 7_{\pm1} 5_{\pm5}^{2} 8_{\pm4},
1_{0}^{2} 2_{\pm1}^{2} 3_{\pm1} 6_{\pm2} 7_{\pm1}^{2} 4_{\pm4} 5_{\pm5}^{3} 8_{\pm4},
1_{0}^{2} 2_{\pm1}^{3} 3_{\pm1}^{2} 6_{\pm2}^{2} 7_{\pm1}^{2} 4_{\pm4} 5_{\pm5}^{5} 8_{\pm4},
1_{0} 2_{\pm1} 3_{\pm1} 7_{\pm1} 5_{\pm5}^{2} 8_{\pm4}, \\
& 1_{0} 2_{\pm1}^{2} 3_{\pm1} 7_{\pm1}^{2} 5_{\pm5}^{3} 8_{\pm4},
1_{0}^{2} 2_{\pm1}^{2} 3_{\pm1} 6_{\pm2} 7_{\pm1}^{2} 5_{\pm5}^{4} 8_{\pm4},
1_{0} 2_{\pm1}^{2} 3_{\pm1} 6_{\pm2} 7_{\pm1} 4_{\pm4} 5_{\pm5}^{2} 8_{\pm4},
1_{0}^{2} 2_{\pm1}^{3} 3_{\pm1}^{2} 6_{\pm2} 7_{\pm1}^{2} 4_{\pm4} 5_{\pm5}^{4} 8_{\pm4}, \\
& 1_{0} 2_{\pm1} 3_{\pm1} 6_{\pm2} 7_{\pm1}^{2} 5_{\pm5}^{3} 8_{\pm4},
1_{0} 2_{\pm1}^{2} 3_{\pm1}^{2} 6_{\pm2} 7_{\pm1}^{2} 4_{\pm4} 5_{\pm5}^{3} 8_{\pm4},
1_{0} 2_{\pm1}^{2} 3_{\pm1} 6_{\pm2} 7_{\pm1} 5_{\pm5}^{3} 8_{\pm4},
1_{0} 2_{\pm1}^{2} 3_{\pm1}^{2} 6_{\pm2} 7_{\pm1}^{2} 5_{\pm5}^{4} 8_{\pm4}, \\
& 1_{0} 2_{\pm1} 7_{\pm1} 5_{\pm5} 8_{\pm4},
1_{0} 2_{\pm1} 3_{\pm1} 7_{\pm1}^{2} 5_{\pm5}^{2} 8_{\pm4},
1_{0} 2_{\pm1}^{2} 3_{\pm1} 6_{\pm2} 7_{\pm1}^{2} 5_{\pm5}^{3} 8_{\pm4},
1_{0} 2_{\pm1} 3_{\pm1} 6_{\pm2} 7_{\pm1} 5_{\pm5}^{2} 8_{\pm4}.
\end{align*}
\end{gather}

%Let $\xi=(0,-1,-1,-2,-3,-2,-1,0)$. 
%\begin{align*}
%\xymatrix{
%& & 2  \ar[d] &  &  \\
%1  \ar[r] & 3 \ar[r]  & 4  \ar[r] & 5 &  6  \ar[l]  &  7  \ar[l]   & 8  \ar[l]}
%\end{align*}
%The highest $l$-weight monomials of Hernandez-Leclerc modules of type $E_8$ that are not of type $A$, $D$ or $E_7$ are 
\begin{gather}
\begin{align*}
& (24)  1_{0} 2_{\pm1} 8_{0} 5_{\pm5},
1_{0} 2_{\pm1} 8_{0} 4_{\pm4} 5_{\pm5},
1_{0} 2_{\pm1} 3_{\pm1} 8_{0} 4_{\pm4} 5_{\pm5}^{2},
1_{0} 2_{\pm1} 8_{0} 5_{\pm5}^{2},
1_{0} 2_{\pm1}^{2}3_{\pm1} 7_{\pm1} 8_{0} 4_{\pm4} 5_{\pm5}^{3}, 
1_{0} 2_{\pm1}^{2}3_{\pm1} 6_{\pm2} 8_{0} 4_{\pm4} 5_{\pm5}^{3}, \\
& 1_{0} 2_{\pm1} 3_{\pm1} 7_{\pm1} 8_{0} 5_{\pm5}^{2},
1_{0} 2_{\pm1} 3_{\pm1} 7_{\pm1} 8_{0} 4_{\pm4} 5_{\pm5}^{2},
1_{0}^{2}2_{\pm1}^{2}3_{\pm1} 6_{\pm2} 7_{\pm1} 8_{0} 4_{\pm4} 5_{\pm5}^{4},
1_{0} 2_{\pm1}^{2}3_{\pm1} 8_{0} 4_{\pm4} 5_{\pm5}^{2}, \\
& 1_{0} 2_{\pm1} 7_{\pm1} 8_{0} 5_{\pm5}^{2}, 
1_{0}^{2}2_{\pm1}^{2}3_{\pm1} 7_{\pm1} 8_{0} 4_{\pm4} 5_{\pm5}^{3},
1_{0} 2_{\pm1} 3_{\pm1} 6_{\pm2} 8_{0} 4_{\pm4} 5_{\pm5}^{2},
1_{0}^{2}2_{\pm1}^{3}3_{\pm1}^{2}6_{\pm2} 7_{\pm1} 8_{0} 4_{\pm4}^{2}5_{\pm5}^{4},
1_{0} 2_{\pm1} 3_{\pm1} 7_{\pm1} 8_{0} 5_{\pm5}^{3}, \\
& 1_{0} 2_{\pm1}^{2}3_{\pm1}^{2}6_{\pm2} 7_{\pm1} 8_{0} 4_{\pm4} 5_{\pm5}^{4},
1_{0}^{2}2_{\pm1}^{2}3_{\pm1} 6_{\pm2} 8_{0} 4_{\pm4} 5_{\pm5}^{3},
1_{0}^{2}2_{\pm1}^{3}3_{\pm1}^{2}6_{\pm2} 7_{\pm1} 8_{0} 4_{\pm4} 5_{\pm5}^{5},
1_{0} 2_{\pm1}^{2}3_{\pm1} 6_{\pm2} 7_{\pm1} 8_{0} 5_{\pm5}^{3}, \\
& 1_{0} 2_{\pm1}^{2}3_{\pm1} 6_{\pm2} 7_{\pm1} 8_{0} 4_{\pm4} 5_{\pm5}^{3},
1_{0}^{2}2_{\pm1}^{3}3_{\pm1} 6_{\pm2} 7_{\pm1} 8_{0} 4_{\pm4} 5_{\pm5}^{4},
1_{0} 2_{\pm1} 3_{\pm1} 8_{0} 4_{\pm4} 5_{\pm5},
1_{0} 2_{\pm1} 3_{\pm1} 6_{\pm2} 8_{0} 5_{\pm5}^{3}, \\
& 1_{0} 2_{\pm1}^{2}3_{\pm1}^{2}7_{\pm1} 8_{0} 4_{\pm4} 5_{\pm5}^{3},
1_{0}^{2}2_{\pm1}^{3}3_{\pm1}^{2}6_{\pm2} 7_{\pm1} 8_{0}^{2}4_{\pm4} 5_{\pm5}^{5},
1_{0}^{2}2_{\pm1}^{2}3_{\pm1}^{2}6_{\pm2} 7_{\pm1} 8_{0} 4_{\pm4} 5_{\pm5}^{4},
1_{0} 2_{\pm1}^{2}3_{\pm1} 7_{\pm1} 8_{0} 4_{\pm4} 5_{\pm5}^{2}, \\
& 1_{0} 2_{\pm1}^{2}3_{\pm1} 6_{\pm2} 7_{\pm1} 8_{0} 5_{\pm5}^{4},
1_{0}^{2}2_{\pm1}^{3}3_{\pm1}^{2}6_{\pm2} 7_{\pm1}^{2}8_{0} 4_{\pm4} 5_{\pm5}^{5},
1_{0} 2_{\pm1}^{2}3_{\pm1}^{2}6_{\pm2} 8_{0} 4_{\pm4} 5_{\pm5}^{3},
1_{0} 2_{\pm1}^{3}3_{\pm1}^{2}6_{\pm2} 7_{\pm1} 8_{0} 4_{\pm4} 5_{\pm5}^{4}, \\
& 1_{0} 2_{\pm1} 6_{\pm2} 8_{0} 5_{\pm5}^{2},
1_{0}^{2}2_{\pm1}^{2}3_{\pm1} 6_{\pm2} 7_{\pm1} 8_{0} 4_{\pm4} 5_{\pm5}^{3},
1_{0}^{2}2_{\pm1}^{3}3_{\pm1}^{2}6_{\pm2}^{2}7_{\pm1} 8_{0} 4_{\pm4} 5_{\pm5}^{5},
1_{0} 2_{\pm1} 3_{\pm1} 6_{\pm2} 8_{0} 5_{\pm5}^{2}
1_{0} 2_{\pm1} 3_{\pm1} 8_{0} 5_{\pm5}^{2}, \\
& 1_{0} 2_{\pm1}^{2}3_{\pm1} 7_{\pm1} 8_{0} 5_{\pm5}^{3},
1_{0}^{2}2_{\pm1}^{2}3_{\pm1} 6_{\pm2} 7_{\pm1} 8_{0} 5_{\pm5}^{4},
1_{0} 2_{\pm1}^{2}3_{\pm1} 6_{\pm2} 8_{0} 4_{\pm4} 5_{\pm5}^{2},
1_{0}^{2}2_{\pm1}^{3}3_{\pm1}^{2}6_{\pm2} 7_{\pm1} 8_{0} 4_{\pm4} 5_{\pm5}^{4}, \\
& 1_{0} 2_{\pm1} 3_{\pm1} 6_{\pm2} 7_{\pm1} 8_{0} 5_{\pm5}^{3},
1_{0} 2_{\pm1}^{2}3_{\pm1}^{2}6_{\pm2} 7_{\pm1} 8_{0} 4_{\pm4} 5_{\pm5}^{3},
1_{0} 2_{\pm1}^{2}3_{\pm1} 6_{\pm2} 8_{0} 5_{\pm5}^{3},
1_{0} 2_{\pm1}^{2}3_{\pm1}^{2}6_{\pm2} 7_{\pm1} 8_{0} 5_{\pm5}^{4}.
\end{align*}
\end{gather}

%Let $\xi=(0,-1,-1,-2,-1,-2,-3,-4)$. 
%\begin{align*}
%\xymatrix{
%& & 2  \ar[d] &  &  \\
%1  \ar[r] & 3 \ar[r]  & 4  & 5  \ar[l]  \ar[r] &  6   \ar[r] &  7  \ar[r] & 8}
%\end{align*}
%The highest $l$-weight monomials of Hernandez-Leclerc modules of type $E_8$ that are not of type $A$, $D$ or $E_7$ are 
\begin{gather}
\begin{align*}
& (25)  1_{0} 2_{\pm1} 5_{\pm1} 4_{\pm4} 8_{\pm6},
1_{0} 2_{\pm1} 5_{\pm1} 4_{\pm4}^{2} 8_{\pm6},
1_{0} 2_{\pm1} 3_{\pm1} 5_{\pm1}^{2} 4_{\pm4}^{3} 6_{\pm4} 8_{\pm6},
1_{0} 2_{\pm1}^{2} 3_{\pm1} 5_{\pm1}^{3} 4_{\pm4}^{4} 6_{\pm4} 7_{\pm5} 8_{\pm6},
1_{0} 2_{\pm1} 5_{\pm1}^{2} 4_{\pm4}^{2} 6_{\pm4} 8_{\pm6}, \\
& 1_{0} 2_{\pm1} 3_{\pm1} 5_{\pm1}^{2} 4_{\pm4}^{3} 7_{\pm5} 8_{\pm6},
1_{0}^{2} 2_{\pm1}^{2} 3_{\pm1} 5_{\pm1}^{4} 4_{\pm4}^{5}  6_{\pm4} 7_{\pm5} 8_{\pm6},
1_{0} 2_{\pm1}^{2} 3_{\pm1} 5_{\pm1}^{3} 4_{\pm4}^{4} 6_{\pm4} 8_{\pm6},
1_{0}^{2} 2_{\pm1}^{2} 3_{\pm1} 5_{\pm1}^{3} 4_{\pm4}^{4} 6_{\pm4} 7_{\pm5} 8_{\pm6}, \\
& 1_{0}^{2} 2_{\pm1}^{3} 3_{\pm1}^{2} 5_{\pm1}^{4} 4_{\pm4}^{6}  6_{\pm4} 7_{\pm5} 8_{\pm6},
1_{0} 2_{\pm1} 3_{\pm1} 5_{\pm1}^{3} 4_{\pm4}^{3} 6_{\pm4} 7_{\pm5} 8_{\pm6},
1_{0} 2_{\pm1} 5_{\pm1}^{2} 4_{\pm4}^{2} 7_{\pm5} 8_{\pm6},
1_{0} 2_{\pm1}^{2} 3_{\pm1} 5_{\pm1}^{2} 4_{\pm4}^{3} 6_{\pm4} 8_{\pm6}, \\
& 1_{0} 2_{\pm1}^{2} 3_{\pm1}^{2} 5_{\pm1}^{4} 4_{\pm4}^{5}  6_{\pm4} 7_{\pm5} 8_{\pm6},
1_{0}^{2} 2_{\pm1}^{3} 3_{\pm1}^{2} 5_{\pm1}^{5}  4_{\pm4}^{6}  6_{\pm4}^{2} 7_{\pm5} 8_{\pm6},
1_{0} 2_{\pm1}^{2} 3_{\pm1} 5_{\pm1}^{3} 4_{\pm4}^{4} 7_{\pm5} 8_{\pm6},
1_{0} 2_{\pm1} 3_{\pm1} 5_{\pm1}^{2} 4_{\pm4}^{3} 8_{\pm6}, \\
& 1_{0} 2_{\pm1}^{2} 3_{\pm1}^{2} 5_{\pm1}^{3} 4_{\pm4}^{4} 6_{\pm4} 7_{\pm5} 8_{\pm6},
1_{0}^{2} 2_{\pm1}^{3} 3_{\pm1} 5_{\pm1}^{4} 4_{\pm4}^{5}  6_{\pm4} 7_{\pm5} 8_{\pm6},
1_{0}^{2} 2_{\pm1}^{3} 3_{\pm1}^{2} 5_{\pm1}^{5}  4_{\pm4}^{6}  6_{\pm4} 7_{\pm5}^{2} 8_{\pm6},
1_{0}^{2} 2_{\pm1}^{2} 3_{\pm1} 5_{\pm1}^{3} 4_{\pm4}^{4} 6_{\pm4} 8_{\pm6}, \\
& 1_{0} 2_{\pm1}^{2} 3_{\pm1} 5_{\pm1}^{2} 4_{\pm4}^{3} 7_{\pm5} 8_{\pm6},
1_{0} 2_{\pm1}^{2} 3_{\pm1} 5_{\pm1}^{4} 4_{\pm4}^{4} 6_{\pm4} 7_{\pm5} 8_{\pm6},
1_{0}^{2} 2_{\pm1}^{2} 3_{\pm1}^{2} 5_{\pm1}^{4} 4_{\pm4}^{5}  6_{\pm4} 7_{\pm5} 8_{\pm6},
1_{0}^{2} 2_{\pm1}^{3} 3_{\pm1}^{2} 5_{\pm1}^{5}  4_{\pm4}^{6}  6_{\pm4} 7_{\pm5} 8_{\pm6}^{2}, \\
& 1_{0} 2_{\pm1} 3_{\pm1} 5_{\pm1}^{3} 4_{\pm4}^{3} 6_{\pm4} 8_{\pm6},
1_{0}^{2} 2_{\pm1}^{2} 3_{\pm1} 5_{\pm1}^{3} 4_{\pm4}^{4} 7_{\pm5} 8_{\pm6},
1_{0} 2_{\pm1}^{3} 3_{\pm1}^{2} 5_{\pm1}^{4} 4_{\pm4}^{5}  6_{\pm4} 7_{\pm5} 8_{\pm6},
1_{0}^{2} 2_{\pm1}^{3} 3_{\pm1}^{2} 5_{\pm1}^{5}  4_{\pm4}^{6}  6_{\pm4} 7_{\pm5} 8_{\pm6}, \\
& 1_{0} 2_{\pm1}^{2} 3_{\pm1} 5_{\pm1}^{3} 4_{\pm4}^{3} 6_{\pm4} 7_{\pm5} 8_{\pm6},
1_{0} 2_{\pm1} 3_{\pm1} 5_{\pm1} 4_{\pm4}^{2} 8_{\pm6},
1_{0} 2_{\pm1}^{2} 3_{\pm1}^{2} 5_{\pm1}^{3} 4_{\pm4}^{4} 6_{\pm4} 8_{\pm6},
1_{0}^{2} 2_{\pm1}^{2} 3_{\pm1} 5_{\pm1}^{4} 4_{\pm4}^{4} 6_{\pm4} 7_{\pm5} 8_{\pm6}, \\
& 1_{0}^{2} 2_{\pm1}^{3} 3_{\pm1}^{2} 5_{\pm1}^{4} 4_{\pm4}^{5}  6_{\pm4} 7_{\pm5} 8_{\pm6},
1_{0} 2_{\pm1} 3_{\pm1} 5_{\pm1}^{3} 4_{\pm4}^{3} 7_{\pm5} 8_{\pm6},
1_{0} 2_{\pm1} 5_{\pm1}^{2} 4_{\pm4}^{2} 8_{\pm6},
1_{0} 2_{\pm1} 3_{\pm1} 5_{\pm1}^{2} 4_{\pm4}^{2} 6_{\pm4} 8_{\pm6}, \\
& 1_{0} 2_{\pm1}^{2} 3_{\pm1}^{2} 5_{\pm1}^{3} 4_{\pm4}^{4} 7_{\pm5} 8_{\pm6},
1_{0} 2_{\pm1}^{2} 3_{\pm1}^{2} 5_{\pm1}^{4} 4_{\pm4}^{4} 6_{\pm4} 7_{\pm5} 8_{\pm6},
1_{0} 2_{\pm1}^{2} 3_{\pm1} 5_{\pm1}^{2} 4_{\pm4}^{3} 8_{\pm6},
1_{0} 2_{\pm1} 3_{\pm1} 5_{\pm1}^{2} 4_{\pm4}^{2} 7_{\pm5} 8_{\pm6}, \\
& 1_{0} 2_{\pm1}^{2} 3_{\pm1} 5_{\pm1}^{3} 4_{\pm4}^{3} 6_{\pm4} 8_{\pm6},
1_{0} 2_{\pm1}^{2} 3_{\pm1} 5_{\pm1}^{3} 4_{\pm4}^{3} 7_{\pm5} 8_{\pm6},
1_{0} 2_{\pm1} 3_{\pm1} 5_{\pm1}^{2} 4_{\pm4}^{2} 8_{\pm6}. \\
% \end{align*}
%\end{gather}
%%
%%
& \quad \\
%%
%%
%Let $\xi=(0,-1,-1,-2,-1,-2,-3,-2)$. 
%\begin{align*}
%\xymatrix{
%& & 2  \ar[d] &  &  \\
%1  \ar[r] & 3 \ar[r]  & 4  & 5  \ar[l]  \ar[r] &  6   \ar[r] &  7 & 8  \ar[l] }
%\end{align*}
%The highest $l$-weight monomials of Hernandez-Leclerc modules of type $E_8$ that are not of type $A$, $D$ or $E_7$ are 
%%
%%
%%
%%
%\begin{gather}
%\begin{align*}
& (26)  1_{0} 2_{\pm1} 5_{\pm1} 8_{\pm2} 4_{\pm4}^{2} 7_{\pm5},
1_{0} 2_{\pm1} 3_{\pm1} 5_{\pm1}^{2} 8_{\pm2} 4_{\pm4}^{3}  6_{\pm4} 7_{\pm5},
1_{0} 2_{\pm1}^{2} 3_{\pm1} 5_{\pm1}^{3}  8_{\pm2} 4_{\pm4}^{4}   6_{\pm4} 7_{\pm5}^{2},
1_{0} 2_{\pm1} 5_{\pm1}^{2} 8_{\pm2} 4_{\pm4}^{2} 6_{\pm4} 7_{\pm5}, \\
& 1_{0} 2_{\pm1} 3_{\pm1} 5_{\pm1}^{2} 8_{\pm2} 4_{\pm4}^{3}  7_{\pm5}^{2},
1_{0}^{2} 2_{\pm1}^{2} 3_{\pm1} 5_{\pm1}^{4}   8_{\pm2} 4_{\pm4}^{5}    6_{\pm4} 7_{\pm5}^{2},
1_{0} 2_{\pm1}^{2} 3_{\pm1} 5_{\pm1}^{3}  8_{\pm2} 4_{\pm4}^{4}   6_{\pm4} 7_{\pm5},
1_{0}^{2} 2_{\pm1}^{2} 3_{\pm1} 5_{\pm1}^{3}  8_{\pm2} 4_{\pm4}^{4}   6_{\pm4} 7_{\pm5}^{2}, \\
& 1_{0}^{2} 2_{\pm1}^{3}  3_{\pm1}^{2} 5_{\pm1}^{4}   8_{\pm2} 4_{\pm4}^{6}     6_{\pm4} 7_{\pm5}^{2},
1_{0} 2_{\pm1} 3_{\pm1} 5_{\pm1}^{3}  8_{\pm2} 4_{\pm4}^{3}  6_{\pm4} 7_{\pm5}^{2},
1_{0} 2_{\pm1} 5_{\pm1}^{2} 8_{\pm2} 4_{\pm4}^{2} 7_{\pm5}^{2},
1_{0} 2_{\pm1}^{2} 3_{\pm1} 5_{\pm1}^{2} 8_{\pm2} 4_{\pm4}^{3}  6_{\pm4} 7_{\pm5}, \\
& 1_{0} 2_{\pm1}^{2} 3_{\pm1}^{2} 5_{\pm1}^{4}   8_{\pm2} 4_{\pm4}^{5}    6_{\pm4} 7_{\pm5}^{2},
1_{0}^{2} 2_{\pm1}^{3}  3_{\pm1}^{2} 5_{\pm1}^{5}  8_{\pm2} 4_{\pm4}^{6}   6_{\pm4}^{2} 7_{\pm5}^{2},
1_{0} 2_{\pm1}^{2} 3_{\pm1} 5_{\pm1}^{3}  8_{\pm2} 4_{\pm4}^{4}   7_{\pm5}^{2},
1_{0} 2_{\pm1} 3_{\pm1} 5_{\pm1}^{2} 8_{\pm2} 4_{\pm4}^{3}  7_{\pm5}, \\
& 1_{0} 2_{\pm1}^{2} 3_{\pm1}^{2} 5_{\pm1}^{3}  8_{\pm2} 4_{\pm4}^{4}  6_{\pm4} 7_{\pm5}^{2},
1_{0}^{2} 2_{\pm1}^{3}  3_{\pm1} 5_{\pm1}^{4}   8_{\pm2} 4_{\pm4}^{5}  6_{\pm4} 7_{\pm5}^{2},
1_{0}^{2} 2_{\pm1}^{3}  3_{\pm1}^{2} 5_{\pm1}^{5}    8_{\pm2} 4_{\pm4}^{6}     6_{\pm4} 7_{\pm5}^{3},
1_{0}^{2} 2_{\pm1}^{2} 3_{\pm1} 5_{\pm1}^{3}  8_{\pm2} 4_{\pm4}^{4}   6_{\pm4} 7_{\pm5}, \\
& 1_{0} 2_{\pm1}^{2} 3_{\pm1} 5_{\pm1}^{2} 8_{\pm2} 4_{\pm4}^{3}  7_{\pm5}^{2},
1_{0} 2_{\pm1}^{2} 3_{\pm1} 5_{\pm1}^{4}   8_{\pm2} 4_{\pm4}^{4}   6_{\pm4} 7_{\pm5}^{2},
1_{0}^{2} 2_{\pm1}^{2} 3_{\pm1}^{2} 5_{\pm1}^{4} 8_{\pm2} 4_{\pm4}^{5}    6_{\pm4} 7_{\pm5}^{2},
1_{0}^{2} 2_{\pm1}^{3}  3_{\pm1}^{2} 5_{\pm1}^{5} 8_{\pm2}^{2} 4_{\pm4}^{6} 6_{\pm4} 7_{\pm5}^{3}, \\
& 1_{0} 2_{\pm1} 3_{\pm1} 5_{\pm1}^{3}  8_{\pm2} 4_{\pm4}^{3}  6_{\pm4} 7_{\pm5},
1_{0}^{2} 2_{\pm1}^{2} 3_{\pm1} 5_{\pm1}^{3}  8_{\pm2} 4_{\pm4}^{4} 7_{\pm5}^{2},
1_{0} 2_{\pm1}^{3}  3_{\pm1}^{2} 5_{\pm1}^{4}   8_{\pm2} 4_{\pm4}^{5}  6_{\pm4} 7_{\pm5}^{2},
1_{0}^{2} 2_{\pm1}^{3}  3_{\pm1}^{2} 5_{\pm1}^{5}  8_{\pm2} 4_{\pm4}^{6} 6_{\pm4} 7_{\pm5}^{2}, \\
& 1_{0} 2_{\pm1}^{2} 3_{\pm1} 5_{\pm1}^{3}  8_{\pm2} 4_{\pm4}^{3}  6_{\pm4} 7_{\pm5}^{2},
1_{0} 2_{\pm1} 3_{\pm1} 5_{\pm1} 8_{\pm2} 4_{\pm4}^{2} 7_{\pm5},
1_{0} 2_{\pm1}^{2} 3_{\pm1}^{2} 5_{\pm1}^{3} 8_{\pm2} 4_{\pm4}^{4}  6_{\pm4} 7_{\pm5},
1_{0}^{2} 2_{\pm1}^{2} 3_{\pm1} 5_{\pm1}^{4} 8_{\pm2} 4_{\pm4}^{4}  6_{\pm4} 7_{\pm5}^{2}, \\
& 1_{0}^{2} 2_{\pm1}^{3}  3_{\pm1}^{2} 5_{\pm1}^{4} 8_{\pm2} 4_{\pm4}^{5}  6_{\pm4} 7_{\pm5}^{2},
1_{0} 2_{\pm1} 3_{\pm1} 5_{\pm1}^{3}  8_{\pm2} 4_{\pm4}^{3}  7_{\pm5}^{2},
1_{0} 2_{\pm1} 5_{\pm1}^{2} 8_{\pm2} 4_{\pm4}^{2} 7_{\pm5},
1_{0} 2_{\pm1} 3_{\pm1} 5_{\pm1}^{2} 8_{\pm2} 4_{\pm4}^{2} 6_{\pm4} 7_{\pm5}, \\
& 1_{0} 2_{\pm1}^{2} 3_{\pm1}^{2} 5_{\pm1}^{3}  8_{\pm2} 4_{\pm4}^{4}  7_{\pm5}^{2},
1_{0} 2_{\pm1}^{2} 3_{\pm1}^{2} 5_{\pm1}^{4}   8_{\pm2} 4_{\pm4}^{4}  6_{\pm4} 7_{\pm5}^{2},
1_{0} 2_{\pm1}^{2} 3_{\pm1} 5_{\pm1}^{2} 8_{\pm2} 4_{\pm4}^{3}  7_{\pm5},
1_{0} 2_{\pm1} 3_{\pm1} 5_{\pm1}^{2} 8_{\pm2} 4_{\pm4}^{2} 7_{\pm5}^{2}, \\
& 1_{0} 2_{\pm1}^{2} 3_{\pm1} 5_{\pm1}^{3}  8_{\pm2} 4_{\pm4}^{3}  6_{\pm4} 7_{\pm5},
1_{0} 2_{\pm1}^{2} 3_{\pm1} 5_{\pm1}^{3}  8_{\pm2} 4_{\pm4}^{3}  7_{\pm5}^{2},
1_{0} 2_{\pm1} 5_{\pm1} 8_{\pm2} 4_{\pm4} 7_{\pm5},
1_{0} 2_{\pm1} 3_{\pm1} 5_{\pm1}^{2} 8_{\pm2} 4_{\pm4}^{2} 7_{\pm5}. \\
% \end{align*}
%\end{gather}
%%
%%
& \quad \\
%%
%%
%Let $\xi=(0,-1,-1,-2,-1,-2,-1,-2)$. 
%\begin{align*}
%\xymatrix{
%& & 2  \ar[d] &  &  \\
%1  \ar[r] & 3 \ar[r]  & 4  & 5  \ar[l]  \ar[r] &  6  &  7  \ar[l]   \ar[r]  & 8 }
%\end{align*}
%The highest $l$-weight monomials of Hernandez-Leclerc modules of type $E_8$ that are not of type $A$, $D$ or $E_7$ are 
%%
%%
%%
%%
%\begin{gather}
%\begin{align*}
& (27)  1_{0} 2_{\pm1} 5_{\pm1} 7_{\pm1} 4_{\pm4}^{2} 6_{\pm4} 8_{\pm4},
1_{0} 2_{\pm1} 3_{\pm1} 5_{\pm1}^{2} 7_{\pm1} 4_{\pm4}^{3}  6_{\pm4}^{2} 8_{\pm4},
1_{0} 2_{\pm1} 5_{\pm1}^{2} 7_{\pm1} 4_{\pm4}^{2} 6_{\pm4}^{2} 8_{\pm4},
1_{0} 2_{\pm1}^{2} 3_{\pm1} 5_{\pm1}^{3}  7_{\pm1}^{2} 4_{\pm4}^{4}   6_{\pm4}^{3}  8_{\pm4}, \\
& 1_{0} 2_{\pm1} 3_{\pm1} 5_{\pm1}^{2} 7_{\pm1}^{2} 4_{\pm4}^{3}  6_{\pm4}^{2} 8_{\pm4},
1_{0} 2_{\pm1}^{2} 3_{\pm1} 5_{\pm1}^{3}  7_{\pm1} 4_{\pm4}^{4}   6_{\pm4}^{2} 8_{\pm4},
1_{0}^{2} 2_{\pm1}^{2} 3_{\pm1} 5_{\pm1}^{4}   7_{\pm1}^{2} 4_{\pm4}^{5}   6_{\pm4}^{3}  8_{\pm4},
1_{0} 2_{\pm1}^{2} 3_{\pm1} 5_{\pm1}^{2} 7_{\pm1} 4_{\pm4}^{3}  6_{\pm4}^{2} 8_{\pm4}, \\
& 1_{0} 2_{\pm1} 3_{\pm1} 5_{\pm1}^{2} 7_{\pm1} 4_{\pm4}^{3}  6_{\pm4} 8_{\pm4},
1_{0}^{2} 2_{\pm1}^{2} 3_{\pm1} 5_{\pm1}^{3}  7_{\pm1}^{2} 4_{\pm4}^{4}  6_{\pm4}^{3}  8_{\pm4},
1_{0}^{2} 2_{\pm1}^{3}  3_{\pm1}^{2} 5_{\pm1}^{4}   7_{\pm1}^{2} 4_{\pm4}^{6}  6_{\pm4}^{3}  8_{\pm4},
1_{0} 2_{\pm1} 3_{\pm1} 5_{\pm1}^{3}  7_{\pm1}^{2} 4_{\pm4}^{3}  6_{\pm4}^{3}  8_{\pm4}, \\
& 1_{0}^{2} 2_{\pm1}^{2} 3_{\pm1} 5_{\pm1}^{3}  7_{\pm1} 4_{\pm4}^{4}   6_{\pm4}^{2} 8_{\pm4},
1_{0} 2_{\pm1}^{2} 3_{\pm1}^{2} 5_{\pm1}^{4}   7_{\pm1}^{2} 4_{\pm4}^{5}   6_{\pm4}^{3}  8_{\pm4},
1_{0}^{2} 2_{\pm1}^{3}  3_{\pm1}^{2} 5_{\pm1}^{5}   7_{\pm1}^{2} 4_{\pm4}^{6}    6_{\pm4}^{4}   8_{\pm4},
1_{0} 2_{\pm1} 5_{\pm1}^{2} 7_{\pm1}^{2} 4_{\pm4}^{2} 6_{\pm4}^{2} 8_{\pm4}, \\
& 1_{0} 2_{\pm1}^{2} 3_{\pm1} 5_{\pm1}^{3}  7_{\pm1}^{2} 4_{\pm4}^{4}   6_{\pm4}^{2} 8_{\pm4},
1_{0} 2_{\pm1} 3_{\pm1} 5_{\pm1} 7_{\pm1} 4_{\pm4}^{2} 6_{\pm4} 8_{\pm4},
1_{0} 2_{\pm1} 3_{\pm1} 5_{\pm1}^{3}  7_{\pm1} 4_{\pm4}^{3}  6_{\pm4}^{2} 8_{\pm4},
1_{0} 2_{\pm1}^{2} 3_{\pm1}^{2} 5_{\pm1}^{3}  7_{\pm1}^{2} 4_{\pm4}^{4}   6_{\pm4}^{3}  8_{\pm4}, \\
& 1_{0}^{2} 2_{\pm1}^{3}  3_{\pm1} 5_{\pm1}^{4}   7_{\pm1}^{2} 4_{\pm4}^{5}   6_{\pm4}^{3}  8_{\pm4},
1_{0}^{2} 2_{\pm1}^{3}  3_{\pm1}^{2} 5_{\pm1}^{5}   7_{\pm1}^{3}  4_{\pm4}^{6}    6_{\pm4}^{4}   8_{\pm4}^{2},
1_{0} 2_{\pm1}^{2} 3_{\pm1} 5_{\pm1}^{2} 7_{\pm1}^{2} 4_{\pm4}^{3}  6_{\pm4}^{2} 8_{\pm4},
1_{0} 2_{\pm1}^{2} 3_{\pm1} 5_{\pm1}^{4}   7_{\pm1}^{2} 4_{\pm4}^{4}   6_{\pm4}^{3}  8_{\pm4}, \\
& 1_{0}^{2} 2_{\pm1}^{2} 3_{\pm1}^{2} 5_{\pm1}^{4}   7_{\pm1}^{2} 4_{\pm4}^{5}   6_{\pm4}^{3}  8_{\pm4},
1_{0}^{2} 2_{\pm1}^{3}  3_{\pm1}^{2} 5_{\pm1}^{5}   7_{\pm1}^{3}  4_{\pm4}^{6}    6_{\pm4}^{4}   8_{\pm4},
1_{0} 2_{\pm1}^{2} 3_{\pm1}^{2} 5_{\pm1}^{3}  7_{\pm1} 4_{\pm4}^{4}   6_{\pm4}^{2} 8_{\pm4},
1_{0} 2_{\pm1} 5_{\pm1}^{2} 7_{\pm1} 4_{\pm4}^{2} 6_{\pm4} 8_{\pm4}, \\
& 1_{0}^{2} 2_{\pm1}^{2} 3_{\pm1} 5_{\pm1}^{3}  7_{\pm1}^{2} 4_{\pm4}^{4}   6_{\pm4}^{2} 8_{\pm4},
1_{0} 2_{\pm1}^{3}  3_{\pm1}^{2} 5_{\pm1}^{4}   7_{\pm1}^{2} 4_{\pm4}^{5}   6_{\pm4}^{3}  8_{\pm4},
1_{0}^{2} 2_{\pm1}^{3}  3_{\pm1}^{2} 5_{\pm1}^{5}   7_{\pm1}^{2} 4_{\pm4}^{6}    6_{\pm4}^{3}  8_{\pm4},
1_{0} 2_{\pm1} 3_{\pm1} 5_{\pm1}^{2} 7_{\pm1} 4_{\pm4}^{2} 6_{\pm4}^{2} 8_{\pm4}, \\
& 1_{0} 2_{\pm1}^{2} 3_{\pm1} 5_{\pm1}^{3}  7_{\pm1}^{2} 4_{\pm4}^{3}  6_{\pm4}^{3}  8_{\pm4},
1_{0} 2_{\pm1}^{2} 3_{\pm1} 5_{\pm1}^{2} 7_{\pm1} 4_{\pm4}^{3}  6_{\pm4} 8_{\pm4},
1_{0}^{2} 2_{\pm1}^{2} 3_{\pm1} 5_{\pm1}^{4}   7_{\pm1}^{2} 4_{\pm4}^{4}   6_{\pm4}^{3}  8_{\pm4},
1_{0}^{2} 2_{\pm1}^{3}  3_{\pm1}^{2} 5_{\pm1}^{4}   7_{\pm1}^{2} 4_{\pm4}^{5}   6_{\pm4}^{3}  8_{\pm4}, \\
& 1_{0} 2_{\pm1} 3_{\pm1} 5_{\pm1}^{3}  7_{\pm1}^{2} 4_{\pm4}^{3}  6_{\pm4}^{2} 8_{\pm4},
1_{0} 2_{\pm1}^{2} 3_{\pm1} 5_{\pm1}^{3}  7_{\pm1} 4_{\pm4}^{3}  6_{\pm4}^{2} 8_{\pm4},
1_{0} 2_{\pm1}^{2} 3_{\pm1}^{2} 5_{\pm1}^{3}  7_{\pm1}^{2} 4_{\pm4}^{4}   6_{\pm4}^{2} 8_{\pm4},
1_{0} 2_{\pm1}^{2} 3_{\pm1}^{2} 5_{\pm1}^{4}   7_{\pm1}^{2} 4_{\pm4}^{4}   6_{\pm4}^{3}  8_{\pm4}, \\
& 1_{0} 2_{\pm1} 5_{\pm1} 7_{\pm1} 4_{\pm4} 6_{\pm4} 8_{\pm4},
1_{0} 2_{\pm1} 3_{\pm1} 5_{\pm1}^{2} 7_{\pm1}^{2} 4_{\pm4}^{2} 6_{\pm4}^{2} 8_{\pm4},
1_{0} 2_{\pm1}^{2} 3_{\pm1} 5_{\pm1}^{3}  7_{\pm1}^{2} 4_{\pm4}^{3}  6_{\pm4}^{2} 8_{\pm4},
1_{0} 2_{\pm1} 3_{\pm1} 5_{\pm1}^{2} 7_{\pm1} 4_{\pm4}^{2} 6_{\pm4} 8_{\pm4}.
\end{align*}
\end{gather}

%%
%%
%Let $\xi=(0,-1,-1,-2,-1,-2,-1,0)$. 
%\begin{align*}
%\xymatrix{
%& & 2  \ar[d] &  &  \\
%1  \ar[r] & 3 \ar[r]  & 4  & 5  \ar[l]  \ar[r] &  6  &  7  \ar[l]  & 8 \ar[l]}
%\end{align*}
%The highest $l$-weight monomials of Hernandez-Leclerc modules of type $E_8$ that are not of type $A$, $D$ or $E_7$ are 
\begin{gather}
\begin{align*}
& (28)  1_{0} 2_{\pm1} 5_{\pm1} 8_{0} 4_{\pm4}^{2} 6_{\pm4},
1_{0} 2_{\pm1} 3_{\pm1} 5_{\pm1}^{2} 8_{0} 4_{\pm4}^{3}  6_{\pm4}^{2},
1_{0} 2_{\pm1} 5_{\pm1}^{2} 8_{0} 4_{\pm4}^{2} 6_{\pm4}^{2},
1_{0} 2_{\pm1}^{2} 3_{\pm1} 5_{\pm1}^{3}  7_{\pm1} 8_{0} 4_{\pm4}^{4}  6_{\pm4}^{3}, \\
& 1_{0} 2_{\pm1} 3_{\pm1} 5_{\pm1}^{2} 7_{\pm1} 8_{0} 4_{\pm4}^{3}  6_{\pm4}^{2},
1_{0} 2_{\pm1}^{2} 3_{\pm1} 5_{\pm1}^{3}  8_{0} 4_{\pm4}^{4}  6_{\pm4}^{2},
1_{0}^{2} 2_{\pm1}^{2} 3_{\pm1} 5_{\pm1}^{4}  7_{\pm1} 8_{0} 4_{\pm4}^{5}   6_{\pm4}^{3},
1_{0} 2_{\pm1}^{2} 3_{\pm1} 5_{\pm1}^{2} 8_{0} 4_{\pm4}^{3}  6_{\pm4}^{2},\\
& 1_{0} 2_{\pm1} 3_{\pm1} 5_{\pm1}^{2} 8_{0} 4_{\pm4}^{3}  6_{\pm4},
1_{0}^{2} 2_{\pm1}^{2} 3_{\pm1} 5_{\pm1}^{3}  7_{\pm1} 8_{0} 4_{\pm4}^{4}  6_{\pm4}^{3},
1_{0}^{2} 2_{\pm1}^{3}  3_{\pm1}^{2} 5_{\pm1}^{4}  7_{\pm1} 8_{0} 4_{\pm4}^{6}   6_{\pm4}^{3},
1_{0} 2_{\pm1} 3_{\pm1} 5_{\pm1}^{3}  7_{\pm1} 8_{0} 4_{\pm4}^{3}  6_{\pm4}^{3}, \\
& 1_{0}^{2} 2_{\pm1}^{2} 3_{\pm1} 5_{\pm1}^{3}  8_{0} 4_{\pm4}^{4}  6_{\pm4}^{2},
1_{0} 2_{\pm1}^{2} 3_{\pm1}^{2} 5_{\pm1}^{4}  7_{\pm1} 8_{0} 4_{\pm4}^{5}   6_{\pm4}^{3},
1_{0}^{2} 2_{\pm1}^{3}  3_{\pm1}^{2} 5_{\pm1}^{5}   7_{\pm1} 8_{0} 4_{\pm4}^{6}   6_{\pm4}^{4},
1_{0} 2_{\pm1} 5_{\pm1}^{2} 7_{\pm1} 8_{0} 4_{\pm4}^{2} 6_{\pm4}^{2}, \\
& 1_{0} 2_{\pm1}^{2} 3_{\pm1} 5_{\pm1}^{3}  7_{\pm1} 8_{0} 4_{\pm4}^{4}  6_{\pm4}^{2},
1_{0} 2_{\pm1} 3_{\pm1} 5_{\pm1} 8_{0} 4_{\pm4}^{2} 6_{\pm4},
1_{0} 2_{\pm1} 3_{\pm1} 5_{\pm1}^{3}  8_{0} 4_{\pm4}^{3}  6_{\pm4}^{2},
1_{0} 2_{\pm1}^{2} 3_{\pm1}^{2} 5_{\pm1}^{3}  7_{\pm1} 8_{0} 4_{\pm4}^{4}  6_{\pm4}^{3}, \\
& 1_{0}^{2} 2_{\pm1}^{3}  3_{\pm1} 5_{\pm1}^{4}  7_{\pm1} 8_{0} 4_{\pm4}^{5}   6_{\pm4}^{3},
1_{0}^{2} 2_{\pm1}^{3}  3_{\pm1}^{2} 5_{\pm1}^{5}   7_{\pm1} 8_{0}^{2} 4_{\pm4}^{6}   6_{\pm4}^{4},
1_{0} 2_{\pm1}^{2} 3_{\pm1} 5_{\pm1}^{2} 7_{\pm1} 8_{0} 4_{\pm4}^{3}  6_{\pm4}^{2},
1_{0} 2_{\pm1}^{2} 3_{\pm1} 5_{\pm1}^{4}  7_{\pm1} 8_{0} 4_{\pm4}^{4}  6_{\pm4}^{3}, \\
& 1_{0}^{2} 2_{\pm1}^{2} 3_{\pm1}^{2} 5_{\pm1}^{4}  7_{\pm1} 8_{0} 4_{\pm4}^{5}   6_{\pm4}^{3},
1_{0}^{2} 2_{\pm1}^{3}  3_{\pm1}^{2} 5_{\pm1}^{5}   7_{\pm1}^{2} 8_{0} 4_{\pm4}^{6}   6_{\pm4}^{4},
1_{0} 2_{\pm1}^{2} 3_{\pm1}^{2} 5_{\pm1}^{3}  8_{0} 4_{\pm4}^{4}  6_{\pm4}^{2},
1_{0} 2_{\pm1} 5_{\pm1}^{2} 8_{0} 4_{\pm4}^{2} 6_{\pm4}, \\
& 1_{0}^{2} 2_{\pm1}^{2} 3_{\pm1} 5_{\pm1}^{3}  7_{\pm1} 8_{0} 4_{\pm4}^{4}  6_{\pm4}^{2},
1_{0} 2_{\pm1}^{3}  3_{\pm1}^{2} 5_{\pm1}^{4}  7_{\pm1} 8_{0} 4_{\pm4}^{5}   6_{\pm4}^{3},
1_{0}^{2} 2_{\pm1}^{3}  3_{\pm1}^{2} 5_{\pm1}^{5}   7_{\pm1} 8_{0} 4_{\pm4}^{6}   6_{\pm4}^{3},
1_{0} 2_{\pm1} 3_{\pm1} 5_{\pm1}^{2} 8_{0} 4_{\pm4}^{2} 6_{\pm4}^{2}, \\
& 1_{0} 2_{\pm1}^{2} 3_{\pm1} 5_{\pm1}^{3}  7_{\pm1} 8_{0} 4_{\pm4}^{3}  6_{\pm4}^{3},
1_{0} 2_{\pm1}^{2} 3_{\pm1} 5_{\pm1}^{2} 8_{0} 4_{\pm4}^{3}  6_{\pm4},
1_{0}^{2} 2_{\pm1}^{2} 3_{\pm1} 5_{\pm1}^{4}  7_{\pm1} 8_{0} 4_{\pm4}^{4}  6_{\pm4}^{3},
1_{0}^{2} 2_{\pm1}^{3}  3_{\pm1}^{2} 5_{\pm1}^{4}  7_{\pm1} 8_{0} 4_{\pm4}^{5}   6_{\pm4}^{3}, \\
& 1_{0} 2_{\pm1} 3_{\pm1} 5_{\pm1}^{3}  7_{\pm1} 8_{0} 4_{\pm4}^{3}  6_{\pm4}^{2},
1_{0} 2_{\pm1}^{2} 3_{\pm1} 5_{\pm1}^{3}  8_{0} 4_{\pm4}^{3}  6_{\pm4}^{2},
1_{0} 2_{\pm1}^{2} 3_{\pm1}^{2} 5_{\pm1}^{3}  7_{\pm1} 8_{0} 4_{\pm4}^{4}  6_{\pm4}^{2},
1_{0} 2_{\pm1}^{2} 3_{\pm1}^{2} 5_{\pm1}^{4}  7_{\pm1} 8_{0} 4_{\pm4}^{4}  6_{\pm4}^{3}, \\
& 1_{0} 2_{\pm1} 5_{\pm1} 8_{0} 4_{\pm4} 6_{\pm4},
1_{0} 2_{\pm1} 3_{\pm1} 5_{\pm1}^{2} 7_{\pm1} 8_{0} 4_{\pm4}^{2} 6_{\pm4}^{2},
1_{0} 2_{\pm1}^{2} 3_{\pm1} 5_{\pm1}^{3}  7_{\pm1} 8_{0} 4_{\pm4}^{3}  6_{\pm4}^{2},
1_{0} 2_{\pm1} 3_{\pm1} 5_{\pm1}^{2} 8_{0} 4_{\pm4}^{2} 6_{\pm4}. \\
%\end{align*}
%\end{gather}
%%
%%
& \quad \\
%%
%%
%Let $\xi=(0,-1,-1,-2,-1,0,-1,-2)$. 
%\begin{align*}
%\xymatrix{
%& & 2  \ar[d] &  &  \\
%1  \ar[r] & 3 \ar[r]  & 4  & 5  \ar[l]  &  6  \ar[l]  \ar[r] &  7  \ar[r] & 8}
%\end{align*}
%The highest $l$-weight monomials of Hernandez-Leclerc modules of type $E_8$ that are not of type $A$, $D$ or $E_7$ are 
%%
%%
%%
%%
%\begin{gather}
%\begin{align*}
& (29)  1_{0} 2_{\pm1} 6_{0} 4_{\pm4}^{2} 8_{\pm4},
1_{0} 2_{\pm1} 3_{\pm1} 6_{0}^{2} 4_{\pm4}^{3} 7_{\pm3} 8_{\pm4},
1_{0} 2_{\pm1} 6_{0}^{2} 4_{\pm4}^{2} 7_{\pm3} 8_{\pm4},
1_{0} 2_{\pm1}^{2} 3_{\pm1} 6_{0}^{3} 4_{\pm4}^{4} 7_{\pm3} 8_{\pm4}, \\
& 1_{0} 2_{\pm1} 3_{\pm1} 6_{0}^{2} 4_{\pm4}^{3} 8_{\pm4},
1_{0} 2_{\pm1}^{2} 3_{\pm1} 5_{\pm1} 6_{0}^{2} 4_{\pm4}^{4} 7_{\pm3} 8_{\pm4},
1_{0}^{2} 2_{\pm1}^{2} 3_{\pm1} 5_{\pm1} 6_{0}^{3} 4_{\pm4}^{5} 7_{\pm3} 8_{\pm4},
1_{0} 2_{\pm1} 3_{\pm1} 5_{\pm1} 6_{0} 4_{\pm4}^{3} 8_{\pm4}, \\
& 1_{0} 2_{\pm1}^{2} 3_{\pm1} 6_{0}^{2} 4_{\pm4}^{3} 7_{\pm3} 8_{\pm4},
1_{0}^{2} 2_{\pm1}^{2} 3_{\pm1} 6_{0}^{3} 4_{\pm4}^{4} 7_{\pm3} 8_{\pm4},
1_{0}^{2} 2_{\pm1}^{3} 3_{\pm1}^{2} 5_{\pm1} 6_{0}^{3} 4_{\pm4}^{6} 7_{\pm3} 8_{\pm4},
1_{0}^{2} 2_{\pm1}^{2} 3_{\pm1} 5_{\pm1} 6_{0}^{2} 4_{\pm4}^{4} 7_{\pm3} 8_{\pm4}, \\
& 1_{0} 2_{\pm1} 3_{\pm1} 6_{0}^{3} 4_{\pm4}^{3} 7_{\pm3} 8_{\pm4},
1_{0} 2_{\pm1}^{2} 3_{\pm1}^{2} 5_{\pm1} 6_{0}^{3} 4_{\pm4}^{5} 7_{\pm3} 8_{\pm4},
1_{0}^{2} 2_{\pm1}^{3} 3_{\pm1}^{2} 5_{\pm1} 6_{0}^{4} 4_{\pm4}^{6} 7_{\pm3}^{2} 8_{\pm4},
1_{0} 2_{\pm1} 6_{0}^{2} 4_{\pm4}^{2} 8_{\pm4}, \\
& 1_{0} 2_{\pm1} 3_{\pm1} 6_{0} 4_{\pm4}^{2} 8_{\pm4},
1_{0} 2_{\pm1} 3_{\pm1} 5_{\pm1} 6_{0}^{2} 4_{\pm4}^{3} 7_{\pm3} 8_{\pm4},
1_{0} 2_{\pm1}^{2} 3_{\pm1}^{2} 6_{0}^{3} 4_{\pm4}^{4} 7_{\pm3} 8_{\pm4},
1_{0} 2_{\pm1}^{2} 3_{\pm1} 5_{\pm1} 6_{0}^{2} 4_{\pm4}^{4} 8_{\pm4}, \\
& 1_{0}^{2} 2_{\pm1}^{3} 3_{\pm1} 5_{\pm1} 6_{0}^{3} 4_{\pm4}^{5} 7_{\pm3} 8_{\pm4},
1_{0}^{2} 2_{\pm1}^{3} 3_{\pm1}^{2} 5_{\pm1} 6_{0}^{4} 4_{\pm4}^{6} 7_{\pm3} 8_{\pm4}^{2},
1_{0} 2_{\pm1}^{2} 3_{\pm1} 6_{0}^{2} 4_{\pm4}^{3} 8_{\pm4},
1_{0} 2_{\pm1}^{2} 3_{\pm1} 5_{\pm1} 6_{0}^{3} 4_{\pm4}^{4} 7_{\pm3} 8_{\pm4}, \\
& 1_{0}^{2} 2_{\pm1}^{2} 3_{\pm1}^{2} 5_{\pm1} 6_{0}^{3} 4_{\pm4}^{5} 7_{\pm3} 8_{\pm4},
1_{0}^{2} 2_{\pm1}^{3} 3_{\pm1}^{2} 5_{\pm1} 6_{0}^{4} 4_{\pm4}^{6} 7_{\pm3} 8_{\pm4},
1_{0} 2_{\pm1} 5_{\pm1} 6_{0} 4_{\pm4}^{2} 8_{\pm4},
1_{0}^{2} 2_{\pm1}^{2} 3_{\pm1} 5_{\pm1} 6_{0}^{2} 4_{\pm4}^{4} 8_{\pm4}, \\
& 1_{0} 2_{\pm1}^{2} 3_{\pm1}^{2} 5_{\pm1} 6_{0}^{2} 4_{\pm4}^{4} 7_{\pm3} 8_{\pm4},
1_{0} 2_{\pm1}^{3} 3_{\pm1}^{2} 5_{\pm1} 6_{0}^{3} 4_{\pm4}^{5} 7_{\pm3} 8_{\pm4},
1_{0}^{2} 2_{\pm1}^{3} 3_{\pm1}^{2} 5_{\pm1}^{2} 6_{0}^{3} 4_{\pm4}^{6} 7_{\pm3} 8_{\pm4},
1_{0} 2_{\pm1} 3_{\pm1} 6_{0}^{2} 4_{\pm4}^{2} 7_{\pm3} 8_{\pm4}, \\
& 1_{0} 2_{\pm1}^{2} 3_{\pm1} 5_{\pm1} 6_{0} 4_{\pm4}^{3} 8_{\pm4},
1_{0} 2_{\pm1}^{2} 3_{\pm1} 6_{0}^{3} 4_{\pm4}^{3} 7_{\pm3} 8_{\pm4},
1_{0}^{2} 2_{\pm1}^{2} 3_{\pm1} 5_{\pm1} 6_{0}^{3} 4_{\pm4}^{4} 7_{\pm3} 8_{\pm4},
1_{0}^{2} 2_{\pm1}^{3} 3_{\pm1}^{2} 5_{\pm1} 6_{0}^{3} 4_{\pm4}^{5} 7_{\pm3} 8_{\pm4}, \\
& 1_{0} 2_{\pm1}^{2} 3_{\pm1} 5_{\pm1} 6_{0}^{2} 4_{\pm4}^{3} 7_{\pm3} 8_{\pm4},
1_{0} 2_{\pm1} 3_{\pm1} 5_{\pm1} 6_{0}^{2} 4_{\pm4}^{3} 8_{\pm4},
1_{0} 2_{\pm1}^{2} 3_{\pm1}^{2} 5_{\pm1} 6_{0}^{2} 4_{\pm4}^{4} 8_{\pm4},
1_{0} 2_{\pm1}^{2} 3_{\pm1}^{2} 5_{\pm1} 6_{0}^{3} 4_{\pm4}^{4} 7_{\pm3} 8_{\pm4}, \\
& 1_{0} 2_{\pm1} 6_{0} 4_{\pm4} 8_{\pm4},
1_{0} 2_{\pm1} 3_{\pm1} 6_{0}^{2} 4_{\pm4}^{2} 8_{\pm4},
1_{0} 2_{\pm1}^{2} 3_{\pm1} 5_{\pm1} 6_{0}^{2} 4_{\pm4}^{3} 8_{\pm4},
1_{0} 2_{\pm1} 3_{\pm1} 5_{\pm1} 6_{0} 4_{\pm4}^{2} 8_{\pm4}. \\
%\end{align*}
%\end{gather}
%%
%%
& \quad \\
%%
%%
%Let $\xi=(0,-1,-1,-2,-1,0,-1,0)$. 
%\begin{align*}
%\xymatrix{
%& & 2  \ar[d] &  &  \\
%1  \ar[r] & 3 \ar[r]  & 4  & 5  \ar[l]  &  6  \ar[l]  \ar[r] &  7 & 8  \ar[l] }
%\end{align*}
%The highest $l$-weight monomials of Hernandez-Leclerc modules of type $E_8$ that are not of type $A$, $D$ or $E_7$ are 
%%
%%
%%
%%
%\begin{gather}
%\begin{align*}
& (30)  1_{0} 2_{\pm1} 6_{0} 8_{0} 4_{\pm4}^{2} 7_{\pm3},
1_{0} 2_{\pm1} 3_{\pm1} 6_{0}^{2} 8_{0} 4_{\pm4}^{3}  7_{\pm3}^{2},
1_{0} 2_{\pm1} 6_{0}^{2} 8_{0} 4_{\pm4}^{2} 7_{\pm3}^{2},
1_{0} 2_{\pm1}^{2} 3_{\pm1} 6_{0}^{3}  8_{0} 4_{\pm4}^{4}   7_{\pm3}^{2}, \\
& 1_{0} 2_{\pm1} 3_{\pm1} 6_{0}^{2} 8_{0} 4_{\pm4}^{3}  7_{\pm3},
1_{0} 2_{\pm1}^{2} 3_{\pm1} 5_{\pm1} 6_{0}^{2} 8_{0} 4_{\pm4}^{4}   7_{\pm3}^{2},
1_{0}^{2} 2_{\pm1}^{2} 3_{\pm1} 5_{\pm1} 6_{0}^{3}  8_{0} 4_{\pm4}^{5}    7_{\pm3}^{2},
1_{0} 2_{\pm1} 3_{\pm1} 5_{\pm1} 6_{0} 8_{0} 4_{\pm4}^{3}  7_{\pm3}, \\
& 1_{0} 2_{\pm1}^{2} 3_{\pm1} 6_{0}^{2} 8_{0} 4_{\pm4}^{3}  7_{\pm3}^{2},
1_{0}^{2} 2_{\pm1}^{2} 3_{\pm1} 6_{0}^{3}  8_{0} 4_{\pm4}^{4}   7_{\pm3}^{2},
1_{0}^{2} 2_{\pm1}^{3}  3_{\pm1}^{2} 5_{\pm1} 6_{0}^{3}  8_{0} 4_{\pm4}^{6}    7_{\pm3}^{2},
1_{0}^{2} 2_{\pm1}^{2} 3_{\pm1} 5_{\pm1} 6_{0}^{2} 8_{0} 4_{\pm4}^{4}   7_{\pm3}^{2}, \\
& 1_{0} 2_{\pm1} 3_{\pm1} 6_{0}^{3}  8_{0} 4_{\pm4}^{3}  7_{\pm3}^{2},
1_{0} 2_{\pm1}^{2} 3_{\pm1}^{2} 5_{\pm1} 6_{0}^{3}  8_{0} 4_{\pm4}^{5}    7_{\pm3}^{2},
1_{0}^{2} 2_{\pm1}^{3}  3_{\pm1}^{2} 5_{\pm1} 6_{0}^{4}   8_{0} 4_{\pm4}^{6}    7_{\pm3}^{3},
1_{0} 2_{\pm1} 6_{0}^{2} 8_{0} 4_{\pm4}^{2} 7_{\pm3}, \\
& 1_{0} 2_{\pm1} 3_{\pm1} 6_{0} 8_{0} 4_{\pm4}^{2} 7_{\pm3},
1_{0} 2_{\pm1} 3_{\pm1} 5_{\pm1} 6_{0}^{2} 8_{0} 4_{\pm4}^{3}  7_{\pm3}^{2},
1_{0} 2_{\pm1}^{2} 3_{\pm1}^{2} 6_{0}^{3}  8_{0} 4_{\pm4}^{4}   7_{\pm3}^{2},
1_{0} 2_{\pm1}^{2} 3_{\pm1} 5_{\pm1} 6_{0}^{2} 8_{0} 4_{\pm4}^{4}   7_{\pm3}, \\
& 1_{0}^{2} 2_{\pm1}^{3}  3_{\pm1} 5_{\pm1} 6_{0}^{3}  8_{0} 4_{\pm4}^{5}    7_{\pm3}^{2},
1_{0}^{2} 2_{\pm1}^{3}  3_{\pm1}^{2} 5_{\pm1} 6_{0}^{4}   8_{0}^{2} 4_{\pm4}^{6}    7_{\pm3}^{3},
1_{0} 2_{\pm1}^{2} 3_{\pm1} 6_{0}^{2} 8_{0} 4_{\pm4}^{3}  7_{\pm3},
1_{0} 2_{\pm1}^{2} 3_{\pm1} 5_{\pm1} 6_{0}^{3}  8_{0} 4_{\pm4}^{4}   7_{\pm3}^{2}, \\
& 1_{0}^{2} 2_{\pm1}^{2} 3_{\pm1}^{2} 5_{\pm1} 6_{0}^{3}  8_{0} 4_{\pm4}^{5}  7_{\pm3}^{2},
1_{0}^{2} 2_{\pm1}^{3}  3_{\pm1}^{2} 5_{\pm1} 6_{0}^{4}   8_{0} 4_{\pm4}^{6} 7_{\pm3}^{2},
1_{0} 2_{\pm1} 5_{\pm1} 6_{0} 8_{0} 4_{\pm4}^{2} 7_{\pm3},
1_{0}^{2} 2_{\pm1}^{2} 3_{\pm1} 5_{\pm1} 6_{0}^{2} 8_{0} 4_{\pm4}^{4}   7_{\pm3}, \\
& 1_{0} 2_{\pm1}^{2} 3_{\pm1}^{2} 5_{\pm1} 6_{0}^{2} 8_{0} 4_{\pm4}^{4}   7_{\pm3}^{2},
1_{0} 2_{\pm1}^{3}  3_{\pm1}^{2} 5_{\pm1} 6_{0}^{3}  8_{0} 4_{\pm4}^{5}    7_{\pm3}^{2},
1_{0}^{2} 2_{\pm1}^{3}  3_{\pm1}^{2} 5_{\pm1}^{2} 6_{0}^{3}  8_{0} 4_{\pm4}^{6}    7_{\pm3}^{2},
1_{0} 2_{\pm1} 3_{\pm1} 6_{0}^{2} 8_{0} 4_{\pm4}^{2} 7_{\pm3}^{2}, \\
& 1_{0} 2_{\pm1}^{2} 3_{\pm1} 5_{\pm1} 6_{0} 8_{0} 4_{\pm4}^{3}  7_{\pm3},
1_{0} 2_{\pm1}^{2} 3_{\pm1} 6_{0}^{3}  8_{0} 4_{\pm4}^{3}  7_{\pm3}^{2},
1_{0}^{2} 2_{\pm1}^{2} 3_{\pm1} 5_{\pm1} 6_{0}^{3}  8_{0} 4_{\pm4}^{4}   7_{\pm3}^{2},
1_{0}^{2} 2_{\pm1}^{3}  3_{\pm1}^{2} 5_{\pm1} 6_{0}^{3}  8_{0} 4_{\pm4}^{5}    7_{\pm3}^{2}, \\
& 1_{0} 2_{\pm1}^{2} 3_{\pm1} 5_{\pm1} 6_{0}^{2} 8_{0} 4_{\pm4}^{3}  7_{\pm3}^{2},
1_{0} 2_{\pm1} 3_{\pm1} 5_{\pm1} 6_{0}^{2} 8_{0} 4_{\pm4}^{3}  7_{\pm3},
1_{0} 2_{\pm1}^{2} 3_{\pm1}^{2} 5_{\pm1} 6_{0}^{2} 8_{0} 4_{\pm4}^{4}   7_{\pm3},
1_{0} 2_{\pm1}^{2} 3_{\pm1}^{2} 5_{\pm1} 6_{0}^{3}  8_{0} 4_{\pm4}^{4}   7_{\pm3}^{2}, \\
& 1_{0} 2_{\pm1} 6_{0} 8_{0} 4_{\pm4} 7_{\pm3},
1_{0} 2_{\pm1} 3_{\pm1} 6_{0}^{2} 8_{0} 4_{\pm4}^{2} 7_{\pm3},
1_{0} 2_{\pm1}^{2} 3_{\pm1} 5_{\pm1} 6_{0}^{2} 8_{0} 4_{\pm4}^{3}  7_{\pm3},
1_{0} 2_{\pm1} 3_{\pm1} 5_{\pm1} 6_{0} 8_{0} 4_{\pm4}^{2} 7_{\pm3}. \\
\end{align*}
\end{gather}

%Let $\xi=(-1,-2,-2,-3,-2,-1,0,-1)$. 
%\begin{align*}
%\xymatrix{
%& & 2  \ar[d] &  &  \\
%1  \ar[r] & 3 \ar[r]  & 4  & 5  \ar[l]  &  6  \ar[l] &  7  \ar[l] \ar[r] & 8}
%\end{align*}
%The highest $l$-weight monomials of Hernandez-Leclerc modules of type $E_8$ that are not of type $A$, $D$ or $E_7$ are 
%%
%%

\begin{gather}
\begin{align*}
& (31)  1_{\pm1} 2_{\pm2} 7_{0} 4_{\pm5}^{2} 8_{\pm3},
1_{\pm1} 2_{\pm2} 3_{\pm2} 7_{0}^{2} 4_{\pm5}^{3} 8_{\pm3},
1_{\pm1} 2_{\pm2} 7_{0}^{2} 4_{\pm5}^{2} 8_{\pm3},
1_{\pm1} 2_{\pm2} 3_{\pm2} 6_{\pm1} 7_{0} 4_{\pm5}^{3} 8_{\pm3}, \\
& 1_{\pm1} 2_{\pm2}^{2} 3_{\pm2} 6_{\pm1} 7_{0}^{2} 4_{\pm5}^{4} 8_{\pm3},
1_{\pm1} 2_{\pm2} 6_{\pm1} 7_{0} 4_{\pm5}^{2} 8_{\pm3},
1_{\pm1} 2_{\pm2} 3_{\pm2} 5_{\pm2} 7_{0} 4_{\pm5}^{3} 8_{\pm3},
1_{\pm1} 2_{\pm2}^{2} 3_{\pm2} 5_{\pm2} 7_{0}^{2} 4_{\pm5}^{4} 8_{\pm3}, \\
& 1_{\pm1}^{2} 2_{\pm2}^{2} 3_{\pm2} 5_{\pm2} 6_{\pm1} 7_{0}^{2} 4_{\pm5}^{5} 8_{\pm3},
1_{\pm1} 2_{\pm2}^{2} 3_{\pm2} 5_{\pm2} 6_{\pm1} 7_{0} 4_{\pm5}^{4} 8_{\pm3},
1_{\pm1} 2_{\pm2} 3_{\pm2} 7_{0} 4_{\pm5}^{2} 8_{\pm3},
1_{\pm1} 2_{\pm2}^{2} 3_{\pm2} 7_{0}^{2} 4_{\pm5}^{3} 8_{\pm3}, \\
& 1_{\pm1}^{2} 2_{\pm2}^{2} 3_{\pm2} 6_{\pm1} 7_{0}^{2} 4_{\pm5}^{4} 8_{\pm3},
1_{\pm1}^{2} 2_{\pm2}^{3} 3_{\pm2}^{2} 5_{\pm2} 6_{\pm1} 7_{0}^{2} 4_{\pm5}^{6} 8_{\pm3},
1_{\pm1} 2_{\pm2} 5_{\pm2} 7_{0} 4_{\pm5}^{2} 8_{\pm3},
1_{\pm1} 2_{\pm2}^{2} 3_{\pm2} 6_{\pm1} 7_{0} 4_{\pm5}^{3} 8_{\pm3}, \\
& 1_{\pm1}^{2} 2_{\pm2}^{2} 3_{\pm2} 5_{\pm2} 7_{0}^{2} 4_{\pm5}^{4} 8_{\pm3},
1_{\pm1} 2_{\pm2} 3_{\pm2} 6_{\pm1} 7_{0}^{2} 4_{\pm5}^{3} 8_{\pm3},
1_{\pm1} 2_{\pm2}^{2} 3_{\pm2}^{2} 5_{\pm2} 6_{\pm1} 7_{0}^{2} 4_{\pm5}^{5} 8_{\pm3},
1_{\pm1}^{2} 2_{\pm2}^{3} 3_{\pm2}^{2} 5_{\pm2} 6_{\pm1} 7_{0}^{3} 4_{\pm5}^{6} 8_{\pm3}^{2}, \\
& 1_{\pm1} 2_{\pm2} 3_{\pm2} 5_{\pm2} 7_{0}^{2} 4_{\pm5}^{3} 8_{\pm3},
1_{\pm1} 2_{\pm2}^{2} 3_{\pm2}^{2} 6_{\pm1} 7_{0}^{2} 4_{\pm5}^{4} 8_{\pm3},
1_{\pm1}^{2} 2_{\pm2}^{3} 3_{\pm2} 5_{\pm2} 6_{\pm1} 7_{0}^{2} 4_{\pm5}^{5} 8_{\pm3},
1_{\pm1}^{2} 2_{\pm2}^{3} 3_{\pm2}^{2} 5_{\pm2} 6_{\pm1} 7_{0}^{3} 4_{\pm5}^{6} 8_{\pm3}, \\
& 1_{\pm1} 2_{\pm2}^{2} 3_{\pm2} 5_{\pm2} 6_{\pm1} 7_{0}^{2} 4_{\pm5}^{4} 8_{\pm3},
1_{\pm1}^{2} 2_{\pm2}^{2} 3_{\pm2} 5_{\pm2} 6_{\pm1} 7_{0} 4_{\pm5}^{4} 8_{\pm3},
1_{\pm1}^{2} 2_{\pm2}^{2} 3_{\pm2}^{2} 5_{\pm2} 6_{\pm1} 7_{0}^{2} 4_{\pm5}^{5} 8_{\pm3},
1_{\pm1}^{2} 2_{\pm2}^{3} 3_{\pm2}^{2} 5_{\pm2} 6_{\pm1}^{2} 7_{0}^{2} 4_{\pm5}^{6} 8_{\pm3}, \\
& 1_{\pm1} 2_{\pm2} 3_{\pm2} 5_{\pm2} 6_{\pm1} 7_{0} 4_{\pm5}^{3} 8_{\pm3},
1_{\pm1} 2_{\pm2}^{2} 3_{\pm2} 5_{\pm2} 7_{0} 4_{\pm5}^{3} 8_{\pm3},
1_{\pm1} 2_{\pm2}^{2} 3_{\pm2}^{2} 5_{\pm2} 7_{0}^{2} 4_{\pm5}^{4} 8_{\pm3},
1_{\pm1} 2_{\pm2}^{3} 3_{\pm2}^{2} 5_{\pm2} 6_{\pm1} 7_{0}^{2} 4_{\pm5}^{5} 8_{\pm3}, \\
& 1_{\pm1}^{2} 2_{\pm2}^{3} 3_{\pm2}^{2} 5_{\pm2}^{2} 6_{\pm1} 7_{0}^{2} 4_{\pm5}^{6} 8_{\pm3},
1_{\pm1} 2_{\pm2} 7_{0} 4_{\pm5} 8_{\pm3},
1_{\pm1} 2_{\pm2}^{2} 3_{\pm2}^{2} 5_{\pm2} 6_{\pm1} 7_{0} 4_{\pm5}^{4} 8_{\pm3},
1_{\pm1} 2_{\pm2} 3_{\pm2} 7_{0}^{2} 4_{\pm5}^{2} 8_{\pm3}, \\
& 1_{\pm1} 2_{\pm2}^{2} 3_{\pm2} 6_{\pm1} 7_{0}^{2} 4_{\pm5}^{3} 8_{\pm3},
1_{\pm1}^{2} 2_{\pm2}^{2} 3_{\pm2} 5_{\pm2} 6_{\pm1} 7_{0}^{2} 4_{\pm5}^{4} 8_{\pm3},
1_{\pm1}^{2} 2_{\pm2}^{3} 3_{\pm2}^{2} 5_{\pm2} 6_{\pm1} 7_{0}^{2} 4_{\pm5}^{5} 8_{\pm3},
1_{\pm1} 2_{\pm2} 3_{\pm2} 6_{\pm1} 7_{0} 4_{\pm5}^{2} 8_{\pm3}, \\
& 1_{\pm1} 2_{\pm2}^{2} 3_{\pm2} 5_{\pm2} 7_{0}^{2} 4_{\pm5}^{3} 8_{\pm3},
1_{\pm1} 2_{\pm2}^{2} 3_{\pm2}^{2} 5_{\pm2} 6_{\pm1} 7_{0}^{2} 4_{\pm5}^{4} 8_{\pm3},
1_{\pm1} 2_{\pm2}^{2} 3_{\pm2} 5_{\pm2} 6_{\pm1} 7_{0} 4_{\pm5}^{3} 8_{\pm3},
1_{\pm1} 2_{\pm2} 3_{\pm2} 5_{\pm2} 7_{0} 4_{\pm5}^{2} 8_{\pm3}. \\
%\end{align*}
%\end{gather}
%%
%%
& \quad \\
%%
%%
%Let $\xi=(-2,-3,-3,-4_{\pm3},-2,-1,0)$. 
%\begin{align*}
%\xymatrix{
%& & 2  \ar[d] &  &  \\
%1  \ar[r] & 3 \ar[r]  & 4  & 5  \ar[l]  &  6  \ar[l] &  7  \ar[l] & 8 \ar[l]}
%\end{align*}
%The highest $l$-weight monomials of Hernandez-Leclerc modules of type $E_8$ that are not of type $A$, $D$ or $E_7$ are 
%\begin{gather}
%\begin{align*}
& (32)  1_{\pm2} 2_{\pm3} 8_{0} 4_{\pm6},
1_{\pm2} 2_{\pm3} 8_{0} 4_{\pm6}^{2},
1_{\pm2} 2_{\pm3} 3_{\pm3} 7_{\pm1} 8_{0} 4_{\pm6}^{3},
1_{\pm2} 2_{\pm3} 7_{\pm1} 8_{0} 4_{\pm6}^{2},
1_{\pm2} 2_{\pm3} 3_{\pm3} 6_{\pm2} 8_{0} 4_{\pm6}^{3},
1_{\pm2} 2_{\pm3} 6_{\pm2} 8_{0} 4_{\pm6}^{2}, \\
& 1_{\pm2} 2_{\pm3}^{2} 3_{\pm3} 6_{\pm2} 7_{\pm1} 8_{0} 4_{\pm6}^{4},
1_{\pm2} 2_{\pm3} 3_{\pm3} 8_{0} 4_{\pm6}^{2},
1_{\pm2} 2_{\pm3} 3_{\pm3} 5_{\pm3} 8_{0} 4_{\pm6}^{3},
1_{\pm2} 2_{\pm3}^{2} 3_{\pm3} 7_{\pm1} 8_{0} 4_{\pm6}^{3},
1_{\pm2} 2_{\pm3}^{2} 3_{\pm3} 5_{\pm3} 7_{\pm1} 8_{0} 4_{\pm6}^{4}, \\
& 1_{\pm2}^{2} 2_{\pm3}^{2} 3_{\pm3} 5_{\pm3} 6_{\pm2} 7_{\pm1} 8_{0} 4_{\pm6}^{5},
1_{\pm2} 2_{\pm3}^{2} 3_{\pm3} 5_{\pm3} 6_{\pm2} 8_{0} 4_{\pm6}^{4}, 
1_{\pm2}^{2} 2_{\pm3}^{2} 3_{\pm3} 6_{\pm2} 7_{\pm1} 8_{0} 4_{\pm6}^{4},
1_{\pm2}^{2} 2_{\pm3}^{3} 3_{\pm3}^{2} 5_{\pm3} 6_{\pm2} 7_{\pm1} 8_{0} 4_{\pm6}^{6}, \\
&  1_{\pm2}^{2} 2_{\pm3}^{2} 3_{\pm3} 5_{\pm3} 7_{\pm1} 8_{0} 4_{\pm6}^{4},
1_{\pm2} 2_{\pm3} 3_{\pm3} 6_{\pm2} 7_{\pm1} 8_{0} 4_{\pm6}^{3},
1_{\pm2} 2_{\pm3}^{2} 3_{\pm3}^{2} 5_{\pm3} 6_{\pm2} 7_{\pm1} 8_{0} 4_{\pm6}^{5},
1_{\pm2}^{2} 2_{\pm3}^{3} 3_{\pm3}^{2} 5_{\pm3} 6_{\pm2} 7_{\pm1} 8_{0}^{2} 4_{\pm6}^{6}, \\
& 1_{\pm2} 2_{\pm3} 3_{\pm3} 5_{\pm3} 7_{\pm1} 8_{0} 4_{\pm6}^{3},
1_{\pm2} 2_{\pm3}^{2} 3_{\pm3}^{2} 6_{\pm2} 7_{\pm1} 8_{0} 4_{\pm6}^{4},
1_{\pm2}^{2} 2_{\pm3}^{3} 3_{\pm3} 5_{\pm3} 6_{\pm2} 7_{\pm1} 8_{0} 4_{\pm6}^{5},
1_{\pm2}^{2} 2_{\pm3}^{3} 3_{\pm3}^{2} 5_{\pm3} 6_{\pm2} 7_{\pm1}^{2} 8_{0} 4_{\pm6}^{6}, \\
& 1_{\pm2} 2_{\pm3}^{2} 3_{\pm3} 5_{\pm3} 6_{\pm2} 7_{\pm1} 8_{0} 4_{\pm6}^{4},
1_{\pm2}^{2} 2_{\pm3}^{2} 3_{\pm3} 5_{\pm3} 6_{\pm2} 8_{0} 4_{\pm6}^{4},
1_{\pm2}^{2} 2_{\pm3}^{2} 3_{\pm3}^{2} 5_{\pm3} 6_{\pm2} 7_{\pm1} 8_{0} 4_{\pm6}^{5},
1_{\pm2}^{2} 2_{\pm3}^{3} 3_{\pm3}^{2} 5_{\pm3} 6_{\pm2}^{2} 7_{\pm1} 8_{0} 4_{\pm6}^{6}, \\
& 1_{\pm2} 2_{\pm3} 3_{\pm3} 5_{\pm3} 6_{\pm2} 8_{0} 4_{\pm6}^{3},
1_{\pm2} 2_{\pm3}^{2} 3_{\pm3} 5_{\pm3} 8_{0} 4_{\pm6}^{3},
1_{\pm2} 2_{\pm3}^{2} 3_{\pm3}^{2} 5_{\pm3} 7_{\pm1} 8_{0} 4_{\pm6}^{4},
1_{\pm2} 2_{\pm3}^{3} 3_{\pm3}^{2} 5_{\pm3} 6_{\pm2} 7_{\pm1} 8_{0} 4_{\pm6}^{5}, \\
& 1_{\pm2}^{2} 2_{\pm3}^{3} 3_{\pm3}^{2} 5_{\pm3}^{2} 6_{\pm2} 7_{\pm1} 8_{0} 4_{\pm6}^{6},
1_{\pm2} 2_{\pm3}^{2} 3_{\pm3}^{2} 5_{\pm3} 6_{\pm2} 8_{0} 4_{\pm6}^{4},
1_{\pm2} 2_{\pm3} 3_{\pm3} 7_{\pm1} 8_{0} 4_{\pm6}^{2}, 
1_{\pm2} 2_{\pm3} 5_{\pm3} 8_{0} 4_{\pm6}^{2}, 
1_{\pm2} 2_{\pm3}^{2} 3_{\pm3} 6_{\pm2} 8_{0} 4_{\pm6}^{3}, \\
& 1_{\pm2} 2_{\pm3}^{2} 3_{\pm3} 6_{\pm2} 7_{\pm1} 8_{0} 4_{\pm6}^{3},
1_{\pm2}^{2} 2_{\pm3}^{2} 3_{\pm3} 5_{\pm3} 6_{\pm2} 7_{\pm1} 8_{0} 4_{\pm6}^{4},
1_{\pm2}^{2} 2_{\pm3}^{3} 3_{\pm3}^{2} 5_{\pm3} 6_{\pm2} 7_{\pm1} 8_{0} 4_{\pm6}^{5},
1_{\pm2} 2_{\pm3} 3_{\pm3} 6_{\pm2} 8_{0} 4_{\pm6}^{2}, \\
& 1_{\pm2} 2_{\pm3}^{2} 3_{\pm3} 5_{\pm3} 7_{\pm1} 8_{0} 4_{\pm6}^{3},
1_{\pm2} 2_{\pm3}^{2} 3_{\pm3}^{2} 5_{\pm3} 6_{\pm2} 7_{\pm1} 8_{0} 4_{\pm6}^{4},
1_{\pm2} 2_{\pm3}^{2} 3_{\pm3} 5_{\pm3} 6_{\pm2} 8_{0} 4_{\pm6}^{3},
1_{\pm2} 2_{\pm3} 3_{\pm3} 5_{\pm3} 8_{0} 4_{\pm6}^{2}.  \\
%\end{align*}
%\end{gather}
%%
%%
& \quad \\
%%
%%
%Let $\xi=(0,-1,-1,0,-1,-2,-1,-2)$. 
%\begin{align*}
%\xymatrix{
%& & 2 &  &  \\
%1  \ar[r] & 3   & 4 \ar[l]  \ar[u]  \ar[r] & 5  \ar[r] &  6 &  7  \ar[l]   \ar[r] & 8 }
%\end{align*}
%The highest $l$-weight monomials of Hernandez-Leclerc modules of type $E_8$ that are not of type $A$, $D$ or $E_7$ are 
%%\begin{gather}
%%\begin{align*}
& (33)  1_{0} 4_{0}^{2} 7_{\pm1} 2_{\pm3} 3_{\pm3}^{2} 5_{\pm3} 6_{\pm4} 8_{\pm4},
1_{0} 4_{0}^{3} 7_{\pm1} 2_{\pm3}^{2} 3_{\pm3}^{2} 5_{\pm3} 6_{\pm4}^{2} 8_{\pm4},
1_{0} 4_{0}^{2} 7_{\pm1} 2_{\pm3} 3_{\pm3}^{2} 6_{\pm4}^{2} 8_{\pm4},
1_{0} 4_{0} 7_{\pm1} 2_{\pm3} 3_{\pm3} 6_{\pm4} 8_{\pm4}, \\
& 1_{0} 4_{0}^{4} 7_{\pm1}^{2} 2_{\pm3}^{2} 3_{\pm3}^{3} 5_{\pm3} 6_{\pm4}^{3} 8_{\pm4},
1_{0} 4_{0}^{3} 7_{\pm1}^{2} 2_{\pm3}^{2} 3_{\pm3}^{2} 5_{\pm3} 6_{\pm4}^{2} 8_{\pm4},
1_{0} 4_{0}^{4} 7_{\pm1} 2_{\pm3}^{2} 3_{\pm3}^{3} 5_{\pm3} 6_{\pm4}^{2} 8_{\pm4},
1_{0} 4_{0}^{3} 7_{\pm1} 2_{\pm3} 3_{\pm3}^{2} 5_{\pm3} 6_{\pm4}^{2} 8_{\pm4}, \\
& 1_{0}^{2} 4_{0}^{5} 7_{\pm1}^{2} 2_{\pm3}^{3} 3_{\pm3}^{4} 5_{\pm3} 6_{\pm4}^{3} 8_{\pm4},
1_{0} 4_{0}^{3} 7_{\pm1} 2_{\pm3}^{2} 3_{\pm3}^{2} 5_{\pm3} 6_{\pm4} 8_{\pm4},
1_{0}^{2} 4_{0}^{4} 7_{\pm1}^{2} 2_{\pm3}^{2} 3_{\pm3}^{3} 5_{\pm3} 6_{\pm4}^{3} 8_{\pm4},
1_{0} 4_{0}^{3} 7_{\pm1}^{2} 2_{\pm3}^{2} 3_{\pm3}^{2} 6_{\pm4}^{3} 8_{\pm4}, \\
& 1_{0}^{2} 4_{0}^{6} 7_{\pm1}^{2} 2_{\pm3}^{3} 3_{\pm3}^{4} 5_{\pm3}^{2} 6_{\pm4}^{3} 8_{\pm4},
1_{0}^{2} 4_{0}^{4} 7_{\pm1} 2_{\pm3}^{2} 3_{\pm3}^{3} 5_{\pm3} 6_{\pm4}^{2} 8_{\pm4},
1_{0} 4_{0}^{5} 7_{\pm1}^{2} 2_{\pm3}^{3} 3_{\pm3}^{3} 5_{\pm3} 6_{\pm4}^{3} 8_{\pm4},
1_{0} 4_{0}^{2} 7_{\pm1}^{2} 2_{\pm3} 3_{\pm3}^{2} 6_{\pm4}^{2} 8_{\pm4}, \\
& 1_{0} 4_{0}^{4} 7_{\pm1}^{2} 2_{\pm3}^{2} 3_{\pm3}^{3} 5_{\pm3} 6_{\pm4}^{2} 8_{\pm4},
1_{0} 4_{0}^{2} 7_{\pm1} 2_{\pm3} 3_{\pm3} 5_{\pm3} 6_{\pm4} 8_{\pm4},
1_{0}^{2} 4_{0}^{6} 7_{\pm1}^{2} 2_{\pm3}^{3} 3_{\pm3}^{4} 5_{\pm3} 6_{\pm4}^{4} 8_{\pm4}, 
1_{0} 4_{0}^{4} 7_{\pm1}^{2} 2_{\pm3}^{2} 3_{\pm3}^{2} 5_{\pm3} 6_{\pm4}^{3} 8_{\pm4}, \\
& 1_{0} 4_{0}^{3} 7_{\pm1} 2_{\pm3}^{2} 3_{\pm3}^{2} 6_{\pm4}^{2} 8_{\pm4},
1_{0}^{2} 4_{0}^{5} 7_{\pm1}^{2} 2_{\pm3}^{2} 3_{\pm3}^{4} 5_{\pm3} 6_{\pm4}^{3} 8_{\pm4},
1_{0}^{2} 4_{0}^{6} 7_{\pm1}^{3} 2_{\pm3}^{3} 3_{\pm3}^{4} 5_{\pm3} 6_{\pm4}^{4} 8_{\pm4}^{2},
1_{0} 4_{0}^{4} 7_{\pm1}^{2} 2_{\pm3}^{2} 3_{\pm3}^{3} 6_{\pm4}^{3} 8_{\pm4}, \\
& 1_{0} 4_{0}^{3} 7_{\pm1}^{2} 2_{\pm3} 3_{\pm3}^{2} 5_{\pm3} 6_{\pm4}^{2} 8_{\pm4},
1_{0}^{2} 4_{0}^{5} 7_{\pm1}^{2} 2_{\pm3}^{3} 3_{\pm3}^{3} 5_{\pm3} 6_{\pm4}^{3} 8_{\pm4},
1_{0} 4_{0}^{4} 7_{\pm1} 2_{\pm3}^{2} 3_{\pm3}^{2} 5_{\pm3} 6_{\pm4}^{2} 8_{\pm4},
1_{0} 4_{0}^{2} 7_{\pm1} 2_{\pm3} 3_{\pm3}^{2} 6_{\pm4} 8_{\pm4}, \\
& 1_{0}^{2} 4_{0}^{6} 7_{\pm1}^{3} 2_{\pm3}^{3} 3_{\pm3}^{4} 5_{\pm3} 6_{\pm4}^{4} 8_{\pm4},
1_{0}^{2} 4_{0}^{4} 7_{\pm1}^{2} 2_{\pm3}^{2} 3_{\pm3}^{3} 5_{\pm3} 6_{\pm4}^{2} 8_{\pm4},
1_{0} 4_{0}^{5} 7_{\pm1}^{2} 2_{\pm3}^{2} 3_{\pm3}^{3} 5_{\pm3} 6_{\pm4}^{3} 8_{\pm4},
1_{0} 4_{0}^{2} 7_{\pm1} 2_{\pm3} 3_{\pm3} 6_{\pm4}^{2} 8_{\pm4}, \\
& 1_{0} 4_{0}^{3} 7_{\pm1}^{2} 2_{\pm3} 3_{\pm3}^{2} 6_{\pm4}^{3} 8_{\pm4},
1_{0}^{2} 4_{0}^{6} 7_{\pm1}^{2} 2_{\pm3}^{3} 3_{\pm3}^{4} 5_{\pm3} 6_{\pm4}^{3} 8_{\pm4},
1_{0} 4_{0}^{3} 7_{\pm1} 2_{\pm3} 3_{\pm3}^{2} 5_{\pm3} 6_{\pm4} 8_{\pm4},
1_{0}^{2} 4_{0}^{4} 7_{\pm1}^{2} 2_{\pm3}^{2} 3_{\pm3}^{3} 6_{\pm4}^{3} 8_{\pm4}, \\
& 1_{0} 4_{0}^{3} 7_{\pm1}^{2} 2_{\pm3}^{2} 3_{\pm3}^{2} 6_{\pm4}^{2} 8_{\pm4},
1_{0}^{2} 4_{0}^{5} 7_{\pm1}^{2} 2_{\pm3}^{2} 3_{\pm3}^{3} 5_{\pm3} 6_{\pm4}^{3} 8_{\pm4},
1_{0} 4_{0}^{3} 7_{\pm1} 2_{\pm3} 3_{\pm3}^{2} 6_{\pm4}^{2} 8_{\pm4},
1_{0} 4_{0}^{4} 7_{\pm1}^{2} 2_{\pm3}^{2} 3_{\pm3}^{2} 5_{\pm3} 6_{\pm4}^{2} 8_{\pm4}, \\
& 1_{0} 4_{0}^{4} 7_{\pm1}^{2} 2_{\pm3}^{2} 3_{\pm3}^{2} 6_{\pm4}^{3} 8_{\pm4},
1_{0} 4_{0}^{2} 7_{\pm1}^{2} 2_{\pm3} 3_{\pm3} 6_{\pm4}^{2} 8_{\pm4},
1_{0} 4_{0}^{3} 7_{\pm1}^{2} 2_{\pm3} 3_{\pm3}^{2} 6_{\pm4}^{2} 8_{\pm4},
1_{0} 4_{0}^{2} 7_{\pm1} 2_{\pm3} 3_{\pm3} 6_{\pm4} 8_{\pm4}.
\end{align*}
\end{gather}

%Let $\xi=(0,-1,-1,0,-1,-2,-3,-4)$. 
%\begin{align*}
%\xymatrix{
%& & 2 &  &  \\
%1  \ar[r] & 3   & 4 \ar[l]  \ar[u]  \ar[r] & 5  \ar[r] &  6 \ar[r] &  7  \ar[r]  & 8}
%\end{align*}
%The highest $l$-weight monomials of Hernandez-Leclerc modules of type $E_8$ that are not of type $A$, $D$ or $E_7$ are 
\begin{gather}
\begin{align*}
& (34)  1_{0} 4_{0}^{2} 2_{\pm3} 3_{\pm3}^{2} 5_{\pm3} 8_{\pm6},
1_{0} 4_{0}^{3} 2_{\pm3}^{2} 3_{\pm3}^{2} 5_{\pm3} 6_{\pm4} 8_{\pm6},
1_{0} 4_{0}^{4} 2_{\pm3}^{2} 3_{\pm3}^{3} 5_{\pm3} 6_{\pm4} 7_{\pm5} 8_{\pm6},
1_{0} 4_{0}^{3} 2_{\pm3}^{2} 3_{\pm3}^{2} 5_{\pm3} 7_{\pm5} 8_{\pm6}, \\
& 1_{0} 4_{0}^{2} 2_{\pm3} 3_{\pm3}^{2} 6_{\pm4} 8_{\pm6},
1_{0}^{2} 4_{0}^{5} 2_{\pm3}^{3} 3_{\pm3}^{4} 5_{\pm3} 6_{\pm4} 7_{\pm5} 8_{\pm6},
1_{0} 4_{0}^{4} 2_{\pm3}^{2} 3_{\pm3}^{3} 5_{\pm3} 6_{\pm4} 8_{\pm6},
1_{0}^{2} 4_{0}^{4} 2_{\pm3}^{2} 3_{\pm3}^{3} 5_{\pm3} 6_{\pm4} 7_{\pm5} 8_{\pm6},\\
& 1_{0} 4_{0}^{3} 2_{\pm3}^{2} 3_{\pm3}^{2} 6_{\pm4} 7_{\pm5} 8_{\pm6},
1_{0} 4_{0}^{2} 2_{\pm3} 3_{\pm3}^{2} 7_{\pm5} 8_{\pm6},
1_{0} 4_{0} 2_{\pm3} 3_{\pm3} 8_{\pm6},
1_{0}^{2} 4_{0}^{6} 2_{\pm3}^{3} 3_{\pm3}^{4} 5_{\pm3}^{2} 6_{\pm4} 7_{\pm5} 8_{\pm6}, \\
& 1_{0} 4_{0}^{3} 2_{\pm3} 3_{\pm3}^{2} 5_{\pm3} 6_{\pm4} 8_{\pm6},
1_{0} 4_{0}^{5} 2_{\pm3}^{3} 3_{\pm3}^{3} 5_{\pm3} 6_{\pm4} 7_{\pm5} 8_{\pm6},
1_{0} 4_{0}^{4} 2_{\pm3}^{2} 3_{\pm3}^{3} 5_{\pm3} 7_{\pm5} 8_{\pm6},
1_{0} 4_{0}^{3} 2_{\pm3}^{2} 3_{\pm3}^{2} 5_{\pm3} 8_{\pm6}, \\
& 1_{0}^{2} 4_{0}^{6} 2_{\pm3}^{3} 3_{\pm3}^{4} 5_{\pm3} 6_{\pm4}^{2} 7_{\pm5} 8_{\pm6},
1_{0} 4_{0}^{4} 2_{\pm3}^{2} 3_{\pm3}^{2} 5_{\pm3} 6_{\pm4} 7_{\pm5} 8_{\pm6},
1_{0}^{2} 4_{0}^{5} 2_{\pm3}^{2} 3_{\pm3}^{4} 5_{\pm3} 6_{\pm4} 7_{\pm5} 8_{\pm6},
1_{0}^{2} 4_{0}^{4} 2_{\pm3}^{2} 3_{\pm3}^{3} 5_{\pm3} 6_{\pm4} 8_{\pm6}, \\
 & 1_{0}^{2} 4_{0}^{6} 2_{\pm3}^{3} 3_{\pm3}^{4} 5_{\pm3} 6_{\pm4} 7_{\pm5}^{2} 8_{\pm6},
1_{0} 4_{0}^{4} 2_{\pm3}^{2} 3_{\pm3}^{3} 6_{\pm4} 7_{\pm5} 8_{\pm6},
1_{0} 4_{0}^{3} 2_{\pm3} 3_{\pm3}^{2} 5_{\pm3} 7_{\pm5} 8_{\pm6},
1_{0}^{2} 4_{0}^{5} 2_{\pm3}^{3} 3_{\pm3}^{3} 5_{\pm3} 6_{\pm4} 7_{\pm5} 8_{\pm6}, \\
& 1_{0}^{2} 4_{0}^{6} 2_{\pm3}^{3} 3_{\pm3}^{4} 5_{\pm3} 6_{\pm4} 7_{\pm5} 8_{\pm6}^{2},
1_{0}^{2} 4_{0}^{4} 2_{\pm3}^{2} 3_{\pm3}^{3} 5_{\pm3} 7_{\pm5} 8_{\pm6},
1_{0} 4_{0}^{3} 2_{\pm3}^{2} 3_{\pm3}^{2} 6_{\pm4} 8_{\pm6},
1_{0} 4_{0}^{5} 2_{\pm3}^{2} 3_{\pm3}^{3} 5_{\pm3} 6_{\pm4} 7_{\pm5} 8_{\pm6}, \\
& 1_{0} 4_{0}^{3} 2_{\pm3} 3_{\pm3}^{2} 6_{\pm4} 7_{\pm5} 8_{\pm6},
1_{0} 4_{0}^{2} 2_{\pm3} 3_{\pm3} 5_{\pm3} 8_{\pm6},
1_{0}^{2} 4_{0}^{6} 2_{\pm3}^{3} 3_{\pm3}^{4} 5_{\pm3} 6_{\pm4} 7_{\pm5} 8_{\pm6},
1_{0} 4_{0}^{4} 2_{\pm3}^{2} 3_{\pm3}^{2} 5_{\pm3} 6_{\pm4} 8_{\pm6}, \\
& 1_{0}^{2} 4_{0}^{4} 2_{\pm3}^{2} 3_{\pm3}^{3} 6_{\pm4} 7_{\pm5} 8_{\pm6},
1_{0} 4_{0}^{3} 2_{\pm3}^{2} 3_{\pm3}^{2} 7_{\pm5} 8_{\pm6},
1_{0} 4_{0}^{2} 2_{\pm3} 3_{\pm3}^{2} 8_{\pm6},
1_{0}^{2} 4_{0}^{5} 2_{\pm3}^{2} 3_{\pm3}^{3} 5_{\pm3} 6_{\pm4} 7_{\pm5} 8_{\pm6}, \\
& 1_{0} 4_{0}^{2} 2_{\pm3} 3_{\pm3} 6_{\pm4} 8_{\pm6},
1_{0} 4_{0}^{4} 2_{\pm3}^{2} 3_{\pm3}^{2} 5_{\pm3} 7_{\pm5} 8_{\pm6},
1_{0} 4_{0}^{3} 2_{\pm3} 3_{\pm3}^{2} 5_{\pm3} 8_{\pm6},
1_{0} 4_{0}^{4} 2_{\pm3}^{2} 3_{\pm3}^{2} 6_{\pm4} 7_{\pm5} 8_{\pm6}, \\
& 1_{0} 4_{0}^{2} 2_{\pm3} 3_{\pm3} 7_{\pm5} 8_{\pm6},
1_{0} 4_{0}^{3} 2_{\pm3} 3_{\pm3}^{2} 6_{\pm4} 8_{\pm6},
1_{0} 4_{0}^{3} 2_{\pm3} 3_{\pm3}^{2} 7_{\pm5} 8_{\pm6},
1_{0} 4_{0}^{2} 2_{\pm3} 3_{\pm3} 8_{\pm6}. \\
%\end{align*}
%\end{gather}
%%
%%
& \quad \\
%%
%%
%Let $\xi=(0,-1,-1,0,-1,-2,-3,-2)$. 
%\begin{align*}
%\xymatrix{
%& & 2 &  &  \\
%1  \ar[r] & 3   & 4 \ar[l]  \ar[u]  \ar[r] & 5  \ar[r] &  6 \ar[r] &  7  & 8  \ar[l] }
%\end{align*}
%The highest $l$-weight monomials of Hernandez-Leclerc modules of type $E_8$ that are not of type $A$, $D$ or $E_7$ are 
%%
%%
%%
%%
%\begin{gather}
%\begin{align*}
& (35)  1_{0} 4_{0}^{2} 8_{\pm2} 2_{\pm3} 3_{\pm3}^{2} 5_{\pm3} 7_{\pm5},
1_{0} 4_{0}^{3} 8_{\pm2} 2_{\pm3}^{2} 3_{\pm3}^{2} 5_{\pm3} 6_{\pm4} 7_{\pm5},
1_{0} 4_{0}^{4} 8_{\pm2} 2_{\pm3}^{2} 3_{\pm3}^{3} 5_{\pm3} 6_{\pm4} 7_{\pm5}^{2},
1_{0} 4_{0}^{3} 8_{\pm2} 2_{\pm3}^{2} 3_{\pm3}^{2} 5_{\pm3} 7_{\pm5}^{2}, \\
& 1_{0} 4_{0}^{2} 8_{\pm2} 2_{\pm3} 3_{\pm3}^{2} 6_{\pm4} 7_{\pm5},
1_{0}^{2} 4_{0}^{5}  8_{\pm2} 2_{\pm3}^{3} 3_{\pm3}^{4} 5_{\pm3} 6_{\pm4} 7_{\pm5}^{2},
1_{0} 4_{0}^{4} 8_{\pm2} 2_{\pm3}^{2} 3_{\pm3}^{3} 5_{\pm3} 6_{\pm4} 7_{\pm5},
1_{0}^{2} 4_{0}^{4} 8_{\pm2} 2_{\pm3}^{2} 3_{\pm3}^{3} 5_{\pm3} 6_{\pm4} 7_{\pm5}^{2}, \\
& 1_{0} 4_{0}^{3} 8_{\pm2} 2_{\pm3}^{2} 3_{\pm3}^{2} 6_{\pm4} 7_{\pm5}^{2},
1_{0} 4_{0}^{2} 8_{\pm2} 2_{\pm3} 3_{\pm3}^{2} 7_{\pm5}^{2},
1_{0}^{2} 4_{0}^{6}  8_{\pm2} 2_{\pm3}^{3} 3_{\pm3}^{4} 5_{\pm3}^{2} 6_{\pm4} 7_{\pm5}^{2},
1_{0} 4_{0}^{3} 8_{\pm2} 2_{\pm3} 3_{\pm3}^{2} 5_{\pm3} 6_{\pm4} 7_{\pm5}, \\
& 1_{0} 4_{0}^{5}  8_{\pm2} 2_{\pm3}^{3} 3_{\pm3}^{3} 5_{\pm3} 6_{\pm4} 7_{\pm5}^{2},
1_{0} 4_{0}^{4} 8_{\pm2} 2_{\pm3}^{2} 3_{\pm3}^{3} 5_{\pm3} 7_{\pm5}^{2},
1_{0} 4_{0} 8_{\pm2} 2_{\pm3} 3_{\pm3} 7_{\pm5},
1_{0} 4_{0}^{3} 8_{\pm2} 2_{\pm3}^{2} 3_{\pm3}^{2} 5_{\pm3} 7_{\pm5}, \\
& 1_{0} 4_{0}^{2} 8_{\pm2} 2_{\pm3} 3_{\pm3} 7_{\pm5},
1_{0} 4_{0}^{4} 8_{\pm2} 2_{\pm3}^{2} 3_{\pm3}^{2} 5_{\pm3} 6_{\pm4} 7_{\pm5}^{2},
1_{0}^{2} 4_{0}^{5}  8_{\pm2} 2_{\pm3}^{2} 3_{\pm3}^{4} 5_{\pm3} 6_{\pm4} 7_{\pm5}^{2},
1_{0}^{2} 4_{0}^{4} 8_{\pm2} 2_{\pm3}^{2} 3_{\pm3}^{3} 5_{\pm3} 6_{\pm4} 7_{\pm5}, \\
& 1_{0}^{2} 4_{0}^{6}  8_{\pm2} 2_{\pm3}^{3} 3_{\pm3}^{4} 5_{\pm3} 6_{\pm4} 7_{\pm5}^{3},
1_{0} 4_{0}^{4} 8_{\pm2} 2_{\pm3}^{2} 3_{\pm3}^{3} 6_{\pm4} 7_{\pm5}^{2},
1_{0} 4_{0}^{3} 8_{\pm2} 2_{\pm3} 3_{\pm3}^{2} 5_{\pm3} 7_{\pm5}^{2},
1_{0}^{2} 4_{0}^{5}  8_{\pm2} 2_{\pm3}^{3} 3_{\pm3}^{3} 5_{\pm3} 6_{\pm4} 7_{\pm5}^{2}, \\
& 1_{0}^{2} 4_{0}^{6}  8_{\pm2}^{2} 2_{\pm3}^{3} 3_{\pm3}^{4} 5_{\pm3} 6_{\pm4} 7_{\pm5}^{3},
1_{0}^{2} 4_{0}^{4} 8_{\pm2} 2_{\pm3}^{2} 3_{\pm3}^{3} 5_{\pm3} 7_{\pm5}^{2},
1_{0} 4_{0}^{3} 8_{\pm2} 2_{\pm3}^{2} 3_{\pm3}^{2} 6_{\pm4} 7_{\pm5},
1_{0} 4_{0}^{5}  8_{\pm2} 2_{\pm3}^{2} 3_{\pm3}^{3} 5_{\pm3} 6_{\pm4} 7_{\pm5}^{2}, \\
& 1_{0} 4_{0}^{3} 8_{\pm2} 2_{\pm3} 3_{\pm3}^{2} 6_{\pm4} 7_{\pm5}^{2},
1_{0} 4_{0}^{2} 8_{\pm2} 2_{\pm3} 3_{\pm3} 5_{\pm3} 7_{\pm5},
1_{0}^{2} 4_{0}^{6}  8_{\pm2} 2_{\pm3}^{3} 3_{\pm3}^{4} 5_{\pm3} 6_{\pm4} 7_{\pm5}^{2},
1_{0} 4_{0}^{4} 8_{\pm2} 2_{\pm3}^{2} 3_{\pm3}^{2} 5_{\pm3} 6_{\pm4} 7_{\pm5}, \\
& 1_{0}^{2} 4_{0}^{4} 8_{\pm2} 2_{\pm3}^{2} 3_{\pm3}^{3} 6_{\pm4} 7_{\pm5}^{2},
1_{0} 4_{0}^{3} 8_{\pm2} 2_{\pm3}^{2} 3_{\pm3}^{2} 7_{\pm5}^{2},
1_{0} 4_{0}^{2} 8_{\pm2} 2_{\pm3} 3_{\pm3}^{2} 7_{\pm5},
1_{0}^{2} 4_{0}^{5}  8_{\pm2} 2_{\pm3}^{2} 3_{\pm3}^{3} 5_{\pm3} 6_{\pm4} 7_{\pm5}^{2}, \\
& 1_{0} 4_{0}^{2} 8_{\pm2} 2_{\pm3} 3_{\pm3} 6_{\pm4} 7_{\pm5},
1_{0} 4_{0}^{4} 8_{\pm2} 2_{\pm3}^{2} 3_{\pm3}^{2} 5_{\pm3} 7_{\pm5}^{2},
1_{0} 4_{0}^{3} 8_{\pm2} 2_{\pm3} 3_{\pm3}^{2} 5_{\pm3} 7_{\pm5},
1_{0} 4_{0}^{4} 8_{\pm2} 2_{\pm3}^{2} 3_{\pm3}^{2} 6_{\pm4} 7_{\pm5}^{2}, \\
& 1_{0} 4_{0}^{2} 8_{\pm2} 2_{\pm3} 3_{\pm3} 7_{\pm5}^{2},
1_{0} 4_{0}^{3} 8_{\pm2} 2_{\pm3} 3_{\pm3}^{2} 6_{\pm4} 7_{\pm5},
1_{0} 4_{0}^{3} 8_{\pm2} 2_{\pm3} 3_{\pm3}^{2} 7_{\pm5}^{2},
1_{0}^{2} 4_{0}^{6}  8_{\pm2} 2_{\pm3}^{3} 3_{\pm3}^{4} 5_{\pm3} 6_{\pm4}^{2} 7_{\pm5}^{2}. \\
%\end{align*}
%\end{gather}
%%
%%
& \quad \\
%%
%%
%Let $\xi=(0,-1,-1,0,-1,-2,-1,0)$. 
%\begin{align*}
%\xymatrix{
%& & 2 &  &  \\
%1  \ar[r] & 3   & 4 \ar[l]  \ar[u]  \ar[r] & 5  \ar[r] &  6 &  7  \ar[l]  & 8  \ar[l] }
%\end{align*}
%The highest $l$-weight monomials of Hernandez-Leclerc modules of type $E_8$ that are not of type $A$, $D$ or $E_7$ are 
%%\begin{gather}
%%\begin{align*}
& (36)  1_{0} 4_{0}^{2} 8_{0} 2_{\pm3} 3_{\pm3}^{2} 5_{\pm3} 6_{\pm4},
1_{0} 4_{0}^{3}  8_{0} 2_{\pm3}^{2} 3_{\pm3}^{2} 5_{\pm3} 6_{\pm4}^{2},
1_{0} 4_{0}^{2} 8_{0} 2_{\pm3} 3_{\pm3}^{2} 6_{\pm4}^{2},
1_{0} 4_{0} 8_{0} 2_{\pm3} 3_{\pm3} 6_{\pm4}, \\
& 1_{0} 4_{0}^{4}  7_{\pm1} 8_{0} 2_{\pm3}^{2} 3_{\pm3}^{3}  5_{\pm3} 6_{\pm4}^{3},
1_{0} 4_{0}^{3}  7_{\pm1} 8_{0} 2_{\pm3}^{2} 3_{\pm3}^{2} 5_{\pm3} 6_{\pm4}^{2},
1_{0} 4_{0}^{4}  8_{0} 2_{\pm3}^{2} 3_{\pm3}^{3}  5_{\pm3} 6_{\pm4}^{2},
1_{0} 4_{0}^{3}  8_{0} 2_{\pm3} 3_{\pm3}^{2} 5_{\pm3} 6_{\pm4}^{2}, \\
& 1_{0}^{2} 4_{0}^{5}  7_{\pm1} 8_{0} 2_{\pm3}^{3}  3_{\pm3}^{4}  5_{\pm3} 6_{\pm4}^{3},
1_{0} 4_{0}^{3}  8_{0} 2_{\pm3}^{2} 3_{\pm3}^{2} 5_{\pm3} 6_{\pm4},
1_{0}^{2} 4_{0}^{4}  7_{\pm1} 8_{0} 2_{\pm3}^{2} 3_{\pm3}^{3}  5_{\pm3} 6_{\pm4}^{3},
1_{0} 4_{0}^{3}  7_{\pm1} 8_{0} 2_{\pm3}^{2} 3_{\pm3}^{2} 6_{\pm4}^{3}, \\
& 1_{0}^{2} 4_{0}^{6}  7_{\pm1} 8_{0} 2_{\pm3}^{3}  3_{\pm3}^{4}  5_{\pm3}^{2} 6_{\pm4}^{3},
1_{0}^{2} 4_{0}^{4}  8_{0} 2_{\pm3}^{2} 3_{\pm3}^{3}  5_{\pm3} 6_{\pm4}^{2},
1_{0} 4_{0}^{5}  7_{\pm1} 8_{0} 2_{\pm3}^{3}  3_{\pm3}^{3}  5_{\pm3} 6_{\pm4}^{3},
1_{0} 4_{0}^{2} 7_{\pm1} 8_{0} 2_{\pm3} 3_{\pm3}^{2} 6_{\pm4}^{2}, \\
& 1_{0} 4_{0}^{4}  7_{\pm1} 8_{0} 2_{\pm3}^{2} 3_{\pm3}^{3}  5_{\pm3} 6_{\pm4}^{2},
1_{0} 4_{0}^{2} 8_{0} 2_{\pm3} 3_{\pm3} 5_{\pm3} 6_{\pm4},
1_{0}^{2} 4_{0}^{6}  7_{\pm1} 8_{0} 2_{\pm3}^{3}  3_{\pm3}^{4}  5_{\pm3} 6_{\pm4}^{4},
1_{0} 4_{0}^{4}  7_{\pm1} 8_{0} 2_{\pm3}^{2} 3_{\pm3}^{2} 5_{\pm3} 6_{\pm4}^{3}, \\
& 1_{0} 4_{0}^{3}  8_{0} 2_{\pm3}^{2} 3_{\pm3}^{2} 6_{\pm4}^{2},
1_{0}^{2} 4_{0}^{5}  7_{\pm1} 8_{0} 2_{\pm3}^{2} 3_{\pm3}^{4}  5_{\pm3} 6_{\pm4}^{3},
1_{0}^{2} 4_{0}^{6}  7_{\pm1} 8_{0}^{2} 2_{\pm3}^{3}  3_{\pm3}^{4}  5_{\pm3} 6_{\pm4}^{4},
1_{0} 4_{0}^{4}  7_{\pm1} 8_{0} 2_{\pm3}^{2} 3_{\pm3}^{3}  6_{\pm4}^{3}, \\
& 1_{0} 4_{0}^{3}  7_{\pm1} 8_{0} 2_{\pm3} 3_{\pm3}^{2} 5_{\pm3} 6_{\pm4}^{2},
1_{0}^{2} 4_{0}^{5}  7_{\pm1} 8_{0} 2_{\pm3}^{3}  3_{\pm3}^{3}  5_{\pm3} 6_{\pm4}^{3},
1_{0} 4_{0}^{4}  8_{0} 2_{\pm3}^{2} 3_{\pm3}^{2} 5_{\pm3} 6_{\pm4}^{2},
1_{0} 4_{0}^{2} 8_{0} 2_{\pm3} 3_{\pm3}^{2} 6_{\pm4}, \\
& 1_{0}^{2} 4_{0}^{6}  7_{\pm1}^{2} 8_{0} 2_{\pm3}^{3}  3_{\pm3}^{4}  5_{\pm3} 6_{\pm4}^{4},
1_{0}^{2} 4_{0}^{4}  7_{\pm1} 8_{0} 2_{\pm3}^{2} 3_{\pm3}^{3}  5_{\pm3} 6_{\pm4}^{2},
1_{0} 4_{0}^{5}  7_{\pm1} 8_{0} 2_{\pm3}^{2} 3_{\pm3}^{3}  5_{\pm3} 6_{\pm4}^{3},
1_{0} 4_{0}^{2} 8_{0} 2_{\pm3} 3_{\pm3} 6_{\pm4}^{2}, \\
& 1_{0} 4_{0}^{3}  7_{\pm1} 8_{0} 2_{\pm3} 3_{\pm3}^{2} 6_{\pm4}^{3},
1_{0}^{2} 4_{0}^{6}  7_{\pm1} 8_{0} 2_{\pm3}^{3}  3_{\pm3}^{4}  5_{\pm3} 6_{\pm4}^{3},
1_{0} 4_{0}^{3}  8_{0} 2_{\pm3} 3_{\pm3}^{2} 5_{\pm3} 6_{\pm4},
1_{0}^{2} 4_{0}^{4}  7_{\pm1} 8_{0} 2_{\pm3}^{2} 3_{\pm3}^{3}  6_{\pm4}^{3}, \\
& 1_{0} 4_{0}^{3}  7_{\pm1} 8_{0} 2_{\pm3}^{2} 3_{\pm3}^{2} 6_{\pm4}^{2},
1_{0}^{2} 4_{0}^{5}  7_{\pm1} 8_{0} 2_{\pm3}^{2} 3_{\pm3}^{3}  5_{\pm3} 6_{\pm4}^{3},
1_{0} 4_{0}^{3}  8_{0} 2_{\pm3} 3_{\pm3}^{2} 6_{\pm4}^{2},
1_{0} 4_{0}^{4}  7_{\pm1} 8_{0} 2_{\pm3}^{2} 3_{\pm3}^{2} 5_{\pm3} 6_{\pm4}^{2}, \\
& 1_{0} 4_{0}^{4}  7_{\pm1} 8_{0} 2_{\pm3}^{2} 3_{\pm3}^{2} 6_{\pm4}^{3},
1_{0} 4_{0}^{2} 7_{\pm1} 8_{0} 2_{\pm3} 3_{\pm3} 6_{\pm4}^{2},
1_{0} 4_{0}^{3}  7_{\pm1} 8_{0} 2_{\pm3} 3_{\pm3}^{2} 6_{\pm4}^{2},
1_{0} 4_{0}^{2} 8_{0} 2_{\pm3} 3_{\pm3} 6_{\pm4}. \\
\end{align*}
\end{gather}
%%
%%
%%
%%
%Let $\xi=(0,-1,-1,0,-1,0,-1,-2)$. 
%\begin{align*}
%\xymatrix{
%& & 2 &  &  \\
%1  \ar[r] & 3   & 4 \ar[l]  \ar[u]  \ar[r] & 5 &  6   \ar[l]    \ar[r] &  7  \ar[r]  & 8 }
%\end{align*}
%The highest $l$-weight monomials of Hernandez-Leclerc modules of type $E_8$ that are not of type $A$, $D$ or $E_7$ are 
\begin{gather}
\begin{align*}
& (37)  1_{0} 4_{0}^{2} 6_{0} 2_{\pm3} 3_{\pm3}^{2} 5_{\pm3}^{2} 8_{\pm4},
1_{0} 4_{0}^{3}  6_{0}^{2} 2_{\pm3}^{2} 3_{\pm3}^{2} 5_{\pm3}^{3}  7_{\pm3} 8_{\pm4},
1_{0} 4_{0}^{2} 6_{0}^{2} 2_{\pm3} 3_{\pm3}^{2} 5_{\pm3}^{2} 7_{\pm3} 8_{\pm4},
1_{0} 4_{0} 6_{0} 2_{\pm3} 3_{\pm3} 5_{\pm3} 8_{\pm4}, \\
& 1_{0} 4_{0}^{4}  6_{0}^{3}  2_{\pm3}^{2} 3_{\pm3}^{3}  5_{\pm3}^{4}  7_{\pm3} 8_{\pm4},
1_{0} 4_{0}^{3}  6_{0}^{2} 2_{\pm3}^{2} 3_{\pm3}^{2} 5_{\pm3}^{3}  8_{\pm4},
1_{0} 4_{0}^{4}  6_{0}^{2} 2_{\pm3}^{2} 3_{\pm3}^{3}  5_{\pm3}^{3}  7_{\pm3} 8_{\pm4},
1_{0}^{2} 4_{0}^{5}  6_{0}^{3}  2_{\pm3}^{3}  3_{\pm3}^{4}  5_{\pm3}^{4}  7_{\pm3} 8_{\pm4}, \\
& 1_{0} 4_{0}^{3}  6_{0} 2_{\pm3}^{2} 3_{\pm3}^{2} 5_{\pm3}^{2} 8_{\pm4},
1_{0} 4_{0}^{3}  6_{0}^{2} 2_{\pm3} 3_{\pm3}^{2} 5_{\pm3}^{3}  7_{\pm3} 8_{\pm4},
1_{0}^{2} 4_{0}^{4}  6_{0}^{3}  2_{\pm3}^{2} 3_{\pm3}^{3}  5_{\pm3}^{4}  7_{\pm3} 8_{\pm4},
1_{0}^{2} 4_{0}^{6}  6_{0}^{3}  2_{\pm3}^{3}  3_{\pm3}^{4}  5_{\pm3}^{5}  7_{\pm3} 8_{\pm4}, \\
& 1_{0}^{2} 4_{0}^{4}  6_{0}^{2} 2_{\pm3}^{2} 3_{\pm3}^{3}  5_{\pm3}^{3}  7_{\pm3} 8_{\pm4},
1_{0} 4_{0}^{3}  6_{0}^{3}  2_{\pm3}^{2} 3_{\pm3}^{2} 5_{\pm3}^{3}  7_{\pm3} 8_{\pm4},
1_{0} 4_{0}^{5}  6_{0}^{3}  2_{\pm3}^{3}  3_{\pm3}^{3}  5_{\pm3}^{4}  7_{\pm3} 8_{\pm4},
1_{0} 4_{0}^{2} 6_{0}^{2} 2_{\pm3} 3_{\pm3}^{2} 5_{\pm3}^{2} 8_{\pm4}, \\
& 1_{0} 4_{0}^{2} 6_{0} 2_{\pm3} 3_{\pm3} 5_{\pm3}^{2} 8_{\pm4},
1_{0}^{2} 4_{0}^{6}  6_{0}^{4}  2_{\pm3}^{3}  3_{\pm3}^{4}  5_{\pm3}^{5}  7_{\pm3}^{2} 8_{\pm4},
1_{0} 4_{0}^{4}  6_{0}^{3}  2_{\pm3}^{2} 3_{\pm3}^{2} 5_{\pm3}^{4}  7_{\pm3} 8_{\pm4},
1_{0} 4_{0}^{3}  6_{0}^{2} 2_{\pm3}^{2} 3_{\pm3}^{2} 5_{\pm3}^{2} 7_{\pm3} 8_{\pm4}, \\
& 1_{0} 4_{0}^{4}  6_{0}^{2} 2_{\pm3}^{2} 3_{\pm3}^{3}  5_{\pm3}^{3}  8_{\pm4},
1_{0}^{2} 4_{0}^{5}  6_{0}^{3}  2_{\pm3}^{2} 3_{\pm3}^{4}  5_{\pm3}^{4}  7_{\pm3} 8_{\pm4},
1_{0}^{2} 4_{0}^{6}  6_{0}^{4}  2_{\pm3}^{3}  3_{\pm3}^{4}  5_{\pm3}^{5}  7_{\pm3} 8_{\pm4}^{2},
1_{0} 4_{0}^{4}  6_{0}^{3}  2_{\pm3}^{2} 3_{\pm3}^{3}  5_{\pm3}^{3}  7_{\pm3} 8_{\pm4}, \\
& 1_{0} 4_{0}^{3}  6_{0}^{2} 2_{\pm3} 3_{\pm3}^{2} 5_{\pm3}^{3}  8_{\pm4},
1_{0}^{2} 4_{0}^{5}  6_{0}^{3}  2_{\pm3}^{3}  3_{\pm3}^{3}  5_{\pm3}^{4}  7_{\pm3} 8_{\pm4},
1_{0} 4_{0}^{2} 6_{0} 2_{\pm3} 3_{\pm3}^{2} 5_{\pm3} 8_{\pm4},
1_{0}^{2} 4_{0}^{6}  6_{0}^{4}  2_{\pm3}^{3}  3_{\pm3}^{4}  5_{\pm3}^{5}  7_{\pm3} 8_{\pm4}, \\
& 1_{0}^{2} 4_{0}^{4}  6_{0}^{2} 2_{\pm3}^{2} 3_{\pm3}^{3}  5_{\pm3}^{3}  8_{\pm4},
1_{0} 4_{0}^{4}  6_{0}^{2} 2_{\pm3}^{2} 3_{\pm3}^{2} 5_{\pm3}^{3}  7_{\pm3} 8_{\pm4},
1_{0} 4_{0}^{5}  6_{0}^{3}  2_{\pm3}^{2} 3_{\pm3}^{3}  5_{\pm3}^{4}  7_{\pm3} 8_{\pm4},
1_{0} 4_{0}^{2} 6_{0}^{2} 2_{\pm3} 3_{\pm3} 5_{\pm3}^{2} 7_{\pm3} 8_{\pm4}, \\
& 1_{0}^{2} 4_{0}^{6}  6_{0}^{3}  2_{\pm3}^{3}  3_{\pm3}^{4}  5_{\pm3}^{4}  7_{\pm3} 8_{\pm4},
1_{0} 4_{0}^{3}  6_{0} 2_{\pm3} 3_{\pm3}^{2} 5_{\pm3}^{2} 8_{\pm4},
1_{0} 4_{0}^{3}  6_{0}^{3}  2_{\pm3} 3_{\pm3}^{2} 5_{\pm3}^{3}  7_{\pm3} 8_{\pm4},
1_{0}^{2} 4_{0}^{4}  6_{0}^{3}  2_{\pm3}^{2} 3_{\pm3}^{3}  5_{\pm3}^{3}  7_{\pm3} 8_{\pm4}, \\
& 1_{0}^{2} 4_{0}^{5}  6_{0}^{3}  2_{\pm3}^{2} 3_{\pm3}^{3}  5_{\pm3}^{4}  7_{\pm3} 8_{\pm4},
1_{0} 4_{0}^{3}  6_{0}^{2} 2_{\pm3} 3_{\pm3}^{2} 5_{\pm3}^{2} 7_{\pm3} 8_{\pm4},
1_{0} 4_{0}^{3}  6_{0}^{2} 2_{\pm3}^{2} 3_{\pm3}^{2} 5_{\pm3}^{2} 8_{\pm4},
1_{0} 4_{0}^{4}  6_{0}^{2} 2_{\pm3}^{2} 3_{\pm3}^{2} 5_{\pm3}^{3}  8_{\pm4}, \\
& 1_{0} 4_{0}^{4}  6_{0}^{3}  2_{\pm3}^{2} 3_{\pm3}^{2} 5_{\pm3}^{3}  7_{\pm3} 8_{\pm4},
1_{0} 4_{0}^{2} 6_{0}^{2} 2_{\pm3} 3_{\pm3} 5_{\pm3}^{2} 8_{\pm4},
1_{0} 4_{0}^{3}  6_{0}^{2} 2_{\pm3} 3_{\pm3}^{2} 5_{\pm3}^{2} 8_{\pm4},
1_{0} 4_{0}^{2} 6_{0} 2_{\pm3} 3_{\pm3} 5_{\pm3} 8_{\pm4}. \\
%\end{align*}
%\end{gather}
%%
%%
& \quad \\
%%
%%
%Let $\xi=(-1,-2,-2,-1,-2,-1,0,-1)$. 
%\begin{align*}
%\xymatrix{
%& & 2 &  &  \\
%1  \ar[r] & 3   & 4 \ar[l]  \ar[u]  \ar[r] & 5 &  6   \ar[l]  &  7    \ar[l]  \ar[r]  & 8 }
%\end{align*}
%The highest $l$-weight monomials of Hernandez-Leclerc modules of type $E_8$ that are not of type $A$, $D$ or $E_7$ are 
%\begin{gather}
%\begin{align*}
& (38)  1_{\pm1} 4_{\pm1}^{2} 7_{0} 2_{\pm4} 3_{\pm4}^{2} 5_{\pm4}^{2} 8_{\pm3},
1_{\pm1} 4_{\pm1} 7_{0} 2_{\pm4} 3_{\pm4} 5_{\pm4} 8_{\pm3},
1_{\pm1} 4_{\pm1}^{3} 7_{0}^{2} 2_{\pm4}^{2} 3_{\pm4}^{2} 5_{\pm4}^{3} 8_{\pm3},
1_{\pm1} 4_{\pm1}^{2} 7_{0}^{2} 2_{\pm4} 3_{\pm4}^{2} 5_{\pm4}^{2} 8_{\pm3}, \\
& 1_{\pm1} 4_{\pm1}^{3} 6_{\pm1} 7_{0} 2_{\pm4}^{2} 3_{\pm4}^{2} 5_{\pm4}^{3} 8_{\pm3},
1_{\pm1} 4_{\pm1}^{4} 6_{\pm1} 7_{0}^{2} 2_{\pm4}^{2} 3_{\pm4}^{3} 5_{\pm4}^{4} 8_{\pm3},
1_{\pm1} 4_{\pm1}^{2} 6_{\pm1} 7_{0} 2_{\pm4} 3_{\pm4}^{2} 5_{\pm4}^{2} 8_{\pm3},
1_{\pm1} 4_{\pm1}^{3} 7_{0} 2_{\pm4}^{2} 3_{\pm4}^{2} 5_{\pm4}^{2} 8_{\pm3}, \\
& 1_{\pm1} 4_{\pm1}^{4} 7_{0}^{2} 2_{\pm4}^{2} 3_{\pm4}^{3} 5_{\pm4}^{3} 8_{\pm3},
1_{\pm1}^{2} 4_{\pm1}^{5} 6_{\pm1} 7_{0}^{2} 2_{\pm4}^{3} 3_{\pm4}^{4} 5_{\pm4}^{4} 8_{\pm3},
1_{\pm1} 4_{\pm1}^{4} 6_{\pm1} 7_{0} 2_{\pm4}^{2} 3_{\pm4}^{3} 5_{\pm4}^{3} 8_{\pm3},
1_{\pm1} 4_{\pm1}^{2} 7_{0} 2_{\pm4} 3_{\pm4} 5_{\pm4}^{2} 8_{\pm3}, \\
& 1_{\pm1} 4_{\pm1}^{3} 7_{0}^{2} 2_{\pm4} 3_{\pm4}^{2} 5_{\pm4}^{3} 8_{\pm3},
1_{\pm1}^{2} 4_{\pm1}^{4} 6_{\pm1} 7_{0}^{2} 2_{\pm4}^{2} 3_{\pm4}^{3} 5_{\pm4}^{4} 8_{\pm3},
1_{\pm1} 4_{\pm1}^{2} 7_{0} 2_{\pm4} 3_{\pm4}^{2} 5_{\pm4} 8_{\pm3},
1_{\pm1}^{2} 4_{\pm1}^{6} 6_{\pm1} 7_{0}^{2} 2_{\pm4}^{3} 3_{\pm4}^{4} 5_{\pm4}^{5} 8_{\pm3}, \\
& 1_{\pm1}^{2} 4_{\pm1}^{4} 7_{0}^{2} 2_{\pm4}^{2} 3_{\pm4}^{3} 5_{\pm4}^{3} 8_{\pm3},
1_{\pm1} 4_{\pm1}^{3} 6_{\pm1} 7_{0} 2_{\pm4} 3_{\pm4}^{2} 5_{\pm4}^{3} 8_{\pm3},
1_{\pm1} 4_{\pm1}^{3} 6_{\pm1} 7_{0}^{2} 2_{\pm4}^{2} 3_{\pm4}^{2} 5_{\pm4}^{3} 8_{\pm3},
1_{\pm1} 4_{\pm1}^{5} 6_{\pm1} 7_{0}^{2} 2_{\pm4}^{3} 3_{\pm4}^{3} 5_{\pm4}^{4} 8_{\pm3}, \\
& 1_{\pm1}^{2} 4_{\pm1}^{6} 6_{\pm1} 7_{0}^{3} 2_{\pm4}^{3} 3_{\pm4}^{4} 5_{\pm4}^{5} 8_{\pm3}^{2},
1_{\pm1} 4_{\pm1}^{4} 6_{\pm1} 7_{0}^{2} 2_{\pm4}^{2} 3_{\pm4}^{2} 5_{\pm4}^{4} 8_{\pm3},
1_{\pm1} 4_{\pm1}^{3} 7_{0}^{2} 2_{\pm4}^{2} 3_{\pm4}^{2} 5_{\pm4}^{2} 8_{\pm3},
1_{\pm1}^{2} 4_{\pm1}^{5} 6_{\pm1} 7_{0}^{2} 2_{\pm4}^{2} 3_{\pm4}^{4} 5_{\pm4}^{4} 8_{\pm3}, \\
& 1_{\pm1}^{2} 4_{\pm1}^{6} 6_{\pm1} 7_{0}^{3} 2_{\pm4}^{3} 3_{\pm4}^{4} 5_{\pm4}^{5} 8_{\pm3},
1_{\pm1} 4_{\pm1}^{4} 6_{\pm1} 7_{0}^{2} 2_{\pm4}^{2} 3_{\pm4}^{3} 5_{\pm4}^{3} 8_{\pm3},
1_{\pm1}^{2} 4_{\pm1}^{4} 6_{\pm1} 7_{0} 2_{\pm4}^{2} 3_{\pm4}^{3} 5_{\pm4}^{3} 8_{\pm3}, 
1_{\pm1}^{2} 4_{\pm1}^{5} 6_{\pm1} 7_{0}^{2} 2_{\pm4}^{3} 3_{\pm4}^{3} 5_{\pm4}^{4} 8_{\pm3}, \\
& 1_{\pm1}^{2} 4_{\pm1}^{6} 6_{\pm1}^{2} 7_{0}^{2} 2_{\pm4}^{3} 3_{\pm4}^{4} 5_{\pm4}^{5} 8_{\pm3},
1_{\pm1} 4_{\pm1}^{3} 6_{\pm1} 7_{0} 2_{\pm4}^{2} 3_{\pm4}^{2} 5_{\pm4}^{2} 8_{\pm3},
1_{\pm1} 4_{\pm1}^{3} 7_{0} 2_{\pm4} 3_{\pm4}^{2} 5_{\pm4}^{2} 8_{\pm3},
1_{\pm1} 4_{\pm1}^{4} 7_{0}^{2} 2_{\pm4}^{2} 3_{\pm4}^{2} 5_{\pm4}^{3} 8_{\pm3}, \\
& 1_{\pm1} 4_{\pm1}^{5} 6_{\pm1} 7_{0}^{2} 2_{\pm4}^{2} 3_{\pm4}^{3} 5_{\pm4}^{4} 8_{\pm3},
1_{\pm1}^{2} 4_{\pm1}^{6} 6_{\pm1} 7_{0}^{2} 2_{\pm4}^{3} 3_{\pm4}^{4} 5_{\pm4}^{4} 8_{\pm3},
1_{\pm1} 4_{\pm1}^{4} 6_{\pm1} 7_{0} 2_{\pm4}^{2} 3_{\pm4}^{2} 5_{\pm4}^{3} 8_{\pm3},
1_{\pm1} 4_{\pm1}^{2} 7_{0}^{2} 2_{\pm4} 3_{\pm4} 5_{\pm4}^{2} 8_{\pm3}, \\
& 1_{\pm1} 4_{\pm1}^{3} 6_{\pm1} 7_{0}^{2} 2_{\pm4} 3_{\pm4}^{2} 5_{\pm4}^{3} 8_{\pm3},
1_{\pm1}^{2} 4_{\pm1}^{4} 6_{\pm1} 7_{0}^{2} 2_{\pm4}^{2} 3_{\pm4}^{3} 5_{\pm4}^{3} 8_{\pm3},
1_{\pm1}^{2} 4_{\pm1}^{5} 6_{\pm1} 7_{0}^{2} 2_{\pm4}^{2} 3_{\pm4}^{3} 5_{\pm4}^{4} 8_{\pm3},
1_{\pm1} 4_{\pm1}^{3} 7_{0}^{2} 2_{\pm4} 3_{\pm4}^{2} 5_{\pm4}^{2} 8_{\pm3}, \\
& 1_{\pm1} 4_{\pm1}^{2} 6_{\pm1} 7_{0} 2_{\pm4} 3_{\pm4} 5_{\pm4}^{2} 8_{\pm3},
1_{\pm1} 4_{\pm1}^{4} 6_{\pm1} 7_{0}^{2} 2_{\pm4}^{2} 3_{\pm4}^{2} 5_{\pm4}^{3} 8_{\pm3},
1_{\pm1} 4_{\pm1}^{3} 6_{\pm1} 7_{0} 2_{\pm4} 3_{\pm4}^{2} 5_{\pm4}^{2} 8_{\pm3},
1_{\pm1} 4_{\pm1}^{2} 7_{0} 2_{\pm4} 3_{\pm4} 5_{\pm4} 8_{\pm3}. \\
%\end{align*}
%\end{gather}
%%
%%
& \quad \\
%%
%%
%Let $\xi=(-2,-3,-3,-2,-3,-2,-1,0)$. 
%\begin{align*}
%\xymatrix{
%& & 2 &  &  \\
%1  \ar[r] & 3   & 4 \ar[l]  \ar[u]  \ar[r] & 5  &  6   \ar[l]  &  7    \ar[l]  & 8 \ar[l] }
%\end{align*}
%The highest $l$-weight monomials of Hernandez-Leclerc modules of type $E_8$ that are not of type $A$, $D$ or $E_7$ are 
%\begin{gather}
%\begin{align*}
& (39)  1_{\pm2} 4_{\pm2}^{2} 8_{0} 2_{\pm5} 3_{\pm5}^{2} 5_{\pm5}^{2},
1_{\pm2} 4_{\pm2} 8_{0} 2_{\pm5} 3_{\pm5} 5_{\pm5},
1_{\pm2} 4_{\pm2}^{3} 7_{\pm1} 8_{0} 2_{\pm5}^{2} 3_{\pm5}^{2} 5_{\pm5}^{3},
1_{\pm2} 4_{\pm2}^{2} 7_{\pm1} 8_{0} 2_{\pm5} 3_{\pm5}^{2} 5_{\pm5}^{2}, \\
& 1_{\pm2} 4_{\pm2}^{3} 6_{\pm2} 8_{0} 2_{\pm5}^{2} 3_{\pm5}^{2} 5_{\pm5}^{3},
1_{\pm2} 4_{\pm2}^{4}  6_{\pm2} 7_{\pm1} 8_{0} 2_{\pm5}^{2} 3_{\pm5}^{3} 5_{\pm5}^{4},
1_{\pm2} 4_{\pm2}^{2} 6_{\pm2} 8_{0} 2_{\pm5} 3_{\pm5}^{2} 5_{\pm5}^{2},
1_{\pm2} 4_{\pm2}^{3} 8_{0} 2_{\pm5}^{2} 3_{\pm5}^{2} 5_{\pm5}^{2}, \\
& 1_{\pm2} 4_{\pm2}^{4}  7_{\pm1} 8_{0} 2_{\pm5}^{2} 3_{\pm5}^{3} 5_{\pm5}^{3},
1_{\pm2}^{2} 4_{\pm2}^{5}   6_{\pm2} 7_{\pm1} 8_{0} 2_{\pm5}^{3} 3_{\pm5}^{4}  5_{\pm5}^{4},
1_{\pm2} 4_{\pm2}^{4}  6_{\pm2} 8_{0} 2_{\pm5}^{2} 3_{\pm5}^{3} 5_{\pm5}^{3},
1_{\pm2} 4_{\pm2}^{2} 8_{0} 2_{\pm5} 3_{\pm5} 5_{\pm5}^{2}, \\
& 1_{\pm2} 4_{\pm2}^{3} 7_{\pm1} 8_{0} 2_{\pm5} 3_{\pm5}^{2} 5_{\pm5}^{3},
1_{\pm2}^{2} 4_{\pm2}^{4}  6_{\pm2} 7_{\pm1} 8_{0} 2_{\pm5}^{2} 3_{\pm5}^{3} 5_{\pm5}^{4},
1_{\pm2} 4_{\pm2}^{2} 8_{0} 2_{\pm5} 3_{\pm5}^{2} 5_{\pm5},
1_{\pm2}^{2} 4_{\pm2}^{6}    6_{\pm2} 7_{\pm1} 8_{0} 2_{\pm5}^{3} 3_{\pm5}^{4}  5_{\pm5}^{5}, \\
& 1_{\pm2}^{2} 4_{\pm2}^{4}  7_{\pm1} 8_{0} 2_{\pm5}^{2} 3_{\pm5}^{3} 5_{\pm5}^{3},
1_{\pm2} 4_{\pm2}^{3} 6_{\pm2} 8_{0} 2_{\pm5} 3_{\pm5}^{2} 5_{\pm5}^{3},
1_{\pm2} 4_{\pm2}^{3} 6_{\pm2} 7_{\pm1} 8_{0} 2_{\pm5}^{2} 3_{\pm5}^{2} 5_{\pm5}^{3},
1_{\pm2} 4_{\pm2}^{5}   6_{\pm2} 7_{\pm1} 8_{0} 2_{\pm5}^{3} 3_{\pm5}^{3} 5_{\pm5}^{4}, \\
& 1_{\pm2}^{2} 4_{\pm2}^{6}    6_{\pm2} 7_{\pm1} 8_{0}^{2} 2_{\pm5}^{3} 3_{\pm5}^{4}  5_{\pm5}^{5},
1_{\pm2} 4_{\pm2}^{4}  6_{\pm2} 7_{\pm1} 8_{0} 2_{\pm5}^{2} 3_{\pm5}^{2} 5_{\pm5}^{4},
1_{\pm2} 4_{\pm2}^{3} 7_{\pm1} 8_{0} 2_{\pm5}^{2} 3_{\pm5}^{2} 5_{\pm5}^{2},
1_{\pm2}^{2} 4_{\pm2}^{5}   6_{\pm2} 7_{\pm1} 8_{0} 2_{\pm5}^{2} 3_{\pm5}^{4}  5_{\pm5}^{4}, \\
& 1_{\pm2}^{2} 4_{\pm2}^{6}    6_{\pm2} 7_{\pm1}^{2} 8_{0} 2_{\pm5}^{3} 3_{\pm5}^{4}  5_{\pm5}^{5},
1_{\pm2} 4_{\pm2}^{4}  6_{\pm2} 7_{\pm1} 8_{0} 2_{\pm5}^{2} 3_{\pm5}^{3} 5_{\pm5}^{3},
1_{\pm2}^{2} 4_{\pm2}^{4}  6_{\pm2} 8_{0} 2_{\pm5}^{2} 3_{\pm5}^{3} 5_{\pm5}^{3},
1_{\pm2}^{2} 4_{\pm2}^{5}   6_{\pm2} 7_{\pm1} 8_{0} 2_{\pm5}^{3} 3_{\pm5}^{3} 5_{\pm5}^{4}, \\
& 1_{\pm2}^{2} 4_{\pm2}^{6}    6_{\pm2}^{2} 7_{\pm1} 8_{0} 2_{\pm5}^{3} 3_{\pm5}^{4}  5_{\pm5}^{5},
1_{\pm2} 4_{\pm2}^{3} 6_{\pm2} 8_{0} 2_{\pm5}^{2} 3_{\pm5}^{2} 5_{\pm5}^{2},
1_{\pm2} 4_{\pm2}^{3} 8_{0} 2_{\pm5} 3_{\pm5}^{2} 5_{\pm5}^{2},
1_{\pm2} 4_{\pm2}^{4}  7_{\pm1} 8_{0} 2_{\pm5}^{2} 3_{\pm5}^{2} 5_{\pm5}^{3}, \\
& 1_{\pm2} 4_{\pm2}^{5}   6_{\pm2} 7_{\pm1} 8_{0} 2_{\pm5}^{2} 3_{\pm5}^{3} 5_{\pm5}^{4},
1_{\pm2}^{2} 4_{\pm2}^{6}    6_{\pm2} 7_{\pm1} 8_{0} 2_{\pm5}^{3} 3_{\pm5}^{4}  5_{\pm5}^{4},
1_{\pm2} 4_{\pm2}^{4}  6_{\pm2} 8_{0} 2_{\pm5}^{2} 3_{\pm5}^{2} 5_{\pm5}^{3},
1_{\pm2} 4_{\pm2}^{2} 7_{\pm1} 8_{0} 2_{\pm5} 3_{\pm5} 5_{\pm5}^{2}, \\
& 1_{\pm2} 4_{\pm2}^{3} 6_{\pm2} 7_{\pm1} 8_{0} 2_{\pm5} 3_{\pm5}^{2} 5_{\pm5}^{3},
1_{\pm2}^{2} 4_{\pm2}^{4}  6_{\pm2} 7_{\pm1} 8_{0} 2_{\pm5}^{2} 3_{\pm5}^{3} 5_{\pm5}^{3},
1_{\pm2}^{2} 4_{\pm2}^{5}   6_{\pm2} 7_{\pm1} 8_{0} 2_{\pm5}^{2} 3_{\pm5}^{3} 5_{\pm5}^{4},
1_{\pm2} 4_{\pm2}^{3} 7_{\pm1} 8_{0} 2_{\pm5} 3_{\pm5}^{2} 5_{\pm5}^{2}, \\
& 1_{\pm2} 4_{\pm2}^{2} 6_{\pm2} 8_{0} 2_{\pm5} 3_{\pm5} 5_{\pm5}^{2},
1_{\pm2} 4_{\pm2}^{4}  6_{\pm2} 7_{\pm1} 8_{0} 2_{\pm5}^{2} 3_{\pm5}^{2} 5_{\pm5}^{3},
1_{\pm2} 4_{\pm2}^{3} 6_{\pm2} 8_{0} 2_{\pm5} 3_{\pm5}^{2} 5_{\pm5}^{2},
1_{\pm2} 4_{\pm2}^{2} 8_{0} 2_{\pm5} 3_{\pm5} 5_{\pm5}.
\end{align*}
\end{gather}

%Let $\xi=(-1,-2,-2,-1,0,-1,-2,-3)$. 
%\begin{align*}
%\xymatrix{
%& & 2 &  &  \\
%1  \ar[r] & 3   & 4 \ar[l]  \ar[u] & 5  \ar[l]   \ar[r]  &  6  \ar[r]  &  7  \ar[r]  & 8 }
%\end{align*}
%The highest $l$-weight monomials of Hernandez-Leclerc modules of type $E_8$ that are not of type $A$, $D$ or $E_7$ are 
\begin{gather}
\begin{align*}
& (40)  1_{\pm1} 5_{0}^{2} 2_{\pm4} 3_{\pm4}^{2} 6_{\pm3} 8_{\pm5},
1_{\pm1} 5_{0}^{3} 2_{\pm4}^{2} 3_{\pm4}^{2} 6_{\pm3} 7_{\pm4} 8_{\pm5},
1_{\pm1} 5_{0}^{2} 2_{\pm4} 3_{\pm4}^{2} 7_{\pm4} 8_{\pm5},
1_{\pm1} 5_{0} 2_{\pm4} 3_{\pm4} 8_{\pm5}, \\
& 1_{\pm1} 5_{0}^{4}  2_{\pm4}^{2} 3_{\pm4}^{3} 6_{\pm3} 7_{\pm4} 8_{\pm5},
1_{\pm1} 5_{0}^{3} 2_{\pm4}^{2} 3_{\pm4}^{2} 6_{\pm3} 8_{\pm5},
1_{\pm1} 4_{\pm1} 5_{0}^{3} 2_{\pm4}^{2} 3_{\pm4}^{3} 6_{\pm3} 7_{\pm4} 8_{\pm5},
1_{\pm1}^{2} 4_{\pm1} 5_{0}^{4}  2_{\pm4}^{3} 3_{\pm4}^{4}  6_{\pm3} 7_{\pm4} 8_{\pm5}, \\
& 1_{\pm1} 4_{\pm1} 5_{0}^{2} 2_{\pm4}^{2} 3_{\pm4}^{2} 6_{\pm3} 8_{\pm5},
1_{\pm1} 5_{0}^{3} 2_{\pm4} 3_{\pm4}^{2} 6_{\pm3} 7_{\pm4} 8_{\pm5},
1_{\pm1}^{2} 5_{0}^{4}  2_{\pm4}^{2} 3_{\pm4}^{3} 6_{\pm3} 7_{\pm4} 8_{\pm5},
1_{\pm1}^{2} 4_{\pm1} 5_{0}^{5}  2_{\pm4}^{3} 3_{\pm4}^{4}  6_{\pm3}^{2} 7_{\pm4} 8_{\pm5}, \\
& 1_{\pm1}^{2} 4_{\pm1} 5_{0}^{3} 2_{\pm4}^{2} 3_{\pm4}^{3} 6_{\pm3} 7_{\pm4} 8_{\pm5},
1_{\pm1} 5_{0}^{3} 2_{\pm4}^{2} 3_{\pm4}^{2} 7_{\pm4} 8_{\pm5},
1_{\pm1} 5_{0}^{2} 2_{\pm4} 3_{\pm4}^{2} 8_{\pm5},
1_{\pm1} 5_{0}^{2} 2_{\pm4} 3_{\pm4} 6_{\pm3} 8_{\pm5}, \\
& 1_{\pm1} 4_{\pm1} 5_{0}^{4}  2_{\pm4}^{3} 3_{\pm4}^{3} 6_{\pm3} 7_{\pm4} 8_{\pm5},
1_{\pm1}^{2} 4_{\pm1} 5_{0}^{5}  2_{\pm4}^{3} 3_{\pm4}^{4}  6_{\pm3} 7_{\pm4}^{2} 8_{\pm5},
1_{\pm1} 5_{0}^{4}  2_{\pm4}^{2} 3_{\pm4}^{2} 6_{\pm3} 7_{\pm4} 8_{\pm5},
1_{\pm1} 4_{\pm1} 5_{0}^{2} 2_{\pm4}^{2} 3_{\pm4}^{2} 7_{\pm4} 8_{\pm5}, \\
& 1_{\pm1} 4_{\pm1} 5_{0}^{3} 2_{\pm4}^{2} 3_{\pm4}^{3} 6_{\pm3} 8_{\pm5},
1_{\pm1}^{2} 4_{\pm1} 5_{0}^{4}  2_{\pm4}^{2} 3_{\pm4}^{4}  6_{\pm3} 7_{\pm4} 8_{\pm5},
1_{\pm1}^{2} 4_{\pm1} 5_{0}^{5}  2_{\pm4}^{3} 3_{\pm4}^{4}  6_{\pm3} 7_{\pm4} 8_{\pm5}^{2},
1_{\pm1} 4_{\pm1} 5_{0}^{3} 2_{\pm4}^{2} 3_{\pm4}^{3} 7_{\pm4} 8_{\pm5}, \\
& 1_{\pm1} 5_{0}^{3} 2_{\pm4} 3_{\pm4}^{2} 6_{\pm3} 8_{\pm5},
1_{\pm1} 4_{\pm1} 5_{0} 2_{\pm4} 3_{\pm4}^{2} 8_{\pm5},
1_{\pm1}^{2} 4_{\pm1} 5_{0}^{4}  2_{\pm4}^{3} 3_{\pm4}^{3} 6_{\pm3} 7_{\pm4} 8_{\pm5},
1_{\pm1}^{2} 4_{\pm1} 5_{0}^{5}  2_{\pm4}^{3} 3_{\pm4}^{4}  6_{\pm3} 7_{\pm4} 8_{\pm5}, \\
& 1_{\pm1}^{2} 4_{\pm1} 5_{0}^{3} 2_{\pm4}^{2} 3_{\pm4}^{3} 6_{\pm3} 8_{\pm5},
1_{\pm1} 4_{\pm1} 5_{0}^{3} 2_{\pm4}^{2} 3_{\pm4}^{2} 6_{\pm3} 7_{\pm4} 8_{\pm5},
1_{\pm1} 5_{0}^{2} 2_{\pm4} 3_{\pm4} 7_{\pm4} 8_{\pm5},
1_{\pm1} 4_{\pm1} 5_{0}^{4}  2_{\pm4}^{2} 3_{\pm4}^{3} 6_{\pm3} 7_{\pm4} 8_{\pm5}, \\
& 1_{\pm1}^{2} 4_{\pm1}^{2} 5_{0}^{4}  2_{\pm4}^{3} 3_{\pm4}^{4}  6_{\pm3} 7_{\pm4} 8_{\pm5},
1_{\pm1} 4_{\pm1} 5_{0}^{2} 2_{\pm4} 3_{\pm4}^{2} 6_{\pm3} 8_{\pm5},
1_{\pm1} 5_{0}^{3} 2_{\pm4} 3_{\pm4}^{2} 7_{\pm4} 8_{\pm5},
1_{\pm1}^{2} 4_{\pm1} 5_{0}^{3} 2_{\pm4}^{2} 3_{\pm4}^{3} 7_{\pm4} 8_{\pm5}, \\
& 1_{\pm1}^{2} 4_{\pm1} 5_{0}^{4}  2_{\pm4}^{2} 3_{\pm4}^{3} 6_{\pm3} 7_{\pm4} 8_{\pm5},
1_{\pm1} 4_{\pm1} 5_{0}^{2} 2_{\pm4} 3_{\pm4}^{2} 7_{\pm4} 8_{\pm5},
1_{\pm1} 4_{\pm1} 5_{0}^{2} 2_{\pm4}^{2} 3_{\pm4}^{2} 8_{\pm5},
1_{\pm1} 4_{\pm1} 5_{0}^{3} 2_{\pm4}^{2} 3_{\pm4}^{2} 6_{\pm3} 8_{\pm5}, \\
& 1_{\pm1} 4_{\pm1} 5_{0}^{3} 2_{\pm4}^{2} 3_{\pm4}^{2} 7_{\pm4} 8_{\pm5},
1_{\pm1} 5_{0}^{2} 2_{\pm4} 3_{\pm4} 8_{\pm5},
1_{\pm1} 4_{\pm1} 5_{0}^{2} 2_{\pm4} 3_{\pm4}^{2} 8_{\pm5},
1_{\pm1} 4_{\pm1} 5_{0} 2_{\pm4} 3_{\pm4} 8_{\pm5}. \\
%\end{align*}
%\end{gather}
%%
%%
& \quad \\
%%
%%
%Let $\xi=(-1,-2,-2,-1,0,-1,-2,-1)$. 
%\begin{align*}
%\xymatrix{
%& & 2 &  &  \\
%1  \ar[r] & 3   & 4 \ar[l]  \ar[u] & 5  \ar[l]   \ar[r]  &  6  \ar[r]  &  7 & 8  \ar[l]  }
%\end{align*}
%The highest $l$-weight monomials of Hernandez-Leclerc modules of type $E_8$ that are not of type $A$, $D$ or $E_7$ are 
%\begin{gather} 
%\begin{align*}
& (41)  1_{\pm1} 5_{0}^{2} 8_{\pm1} 2_{\pm4} 3_{\pm4}^{2} 6_{\pm3} 7_{\pm4},
1_{\pm1} 5_{0}^{3}  8_{\pm1} 2_{\pm4}^{2} 3_{\pm4}^{2} 6_{\pm3} 7_{\pm4}^{2},
1_{\pm1} 5_{0}^{2} 8_{\pm1} 2_{\pm4} 3_{\pm4}^{2} 7_{\pm4}^{2},
1_{\pm1} 5_{0} 8_{\pm1} 2_{\pm4} 3_{\pm4} 7_{\pm4}, \\
& 1_{\pm1} 5_{0}^{4}  8_{\pm1} 2_{\pm4}^{2} 3_{\pm4}^{3}  6_{\pm3} 7_{\pm4}^{2},
1_{\pm1} 5_{0}^{3}  8_{\pm1} 2_{\pm4}^{2} 3_{\pm4}^{2} 6_{\pm3} 7_{\pm4},
1_{\pm1} 4_{\pm1} 5_{0}^{3}  8_{\pm1} 2_{\pm4}^{2} 3_{\pm4}^{3}  6_{\pm3} 7_{\pm4}^{2},
1_{\pm1}^{2} 4_{\pm1} 5_{0}^{4}  8_{\pm1} 2_{\pm4}^{3}  3_{\pm4}^{4}  6_{\pm3} 7_{\pm4}^{2}, \\
& 1_{\pm1} 4_{\pm1} 5_{0}^{2} 8_{\pm1} 2_{\pm4}^{2} 3_{\pm4}^{2} 6_{\pm3} 7_{\pm4},
1_{\pm1} 5_{0}^{3}  8_{\pm1} 2_{\pm4} 3_{\pm4}^{2} 6_{\pm3} 7_{\pm4}^{2},
1_{\pm1}^{2} 5_{0}^{4}  8_{\pm1} 2_{\pm4}^{2} 3_{\pm4}^{3}  6_{\pm3} 7_{\pm4}^{2},
1_{\pm1}^{2} 4_{\pm1} 5_{0}^{5}  8_{\pm1} 2_{\pm4}^{3}  3_{\pm4}^{4}  6_{\pm3}^{2} 7_{\pm4}^{2}, \\
& 1_{\pm1}^{2} 4_{\pm1} 5_{0}^{3}  8_{\pm1} 2_{\pm4}^{2} 3_{\pm4}^{3}  6_{\pm3} 7_{\pm4}^{2},
1_{\pm1} 5_{0}^{3}  8_{\pm1} 2_{\pm4}^{2} 3_{\pm4}^{2} 7_{\pm4}^{2},
1_{\pm1} 5_{0}^{2} 8_{\pm1} 2_{\pm4} 3_{\pm4}^{2} 7_{\pm4},
1_{\pm1} 5_{0}^{2} 8_{\pm1} 2_{\pm4} 3_{\pm4} 6_{\pm3} 7_{\pm4}, \\
& 1_{\pm1} 4_{\pm1} 5_{0}^{4}  8_{\pm1} 2_{\pm4}^{3}  3_{\pm4}^{3}  6_{\pm3} 7_{\pm4}^{2},
1_{\pm1}^{2} 4_{\pm1} 5_{0}^{5}  8_{\pm1} 2_{\pm4}^{3}  3_{\pm4}^{4}  6_{\pm3} 7_{\pm4}^{3},
1_{\pm1} 5_{0}^{4}  8_{\pm1} 2_{\pm4}^{2} 3_{\pm4}^{2} 6_{\pm3} 7_{\pm4}^{2},
1_{\pm1} 4_{\pm1} 5_{0}^{2} 8_{\pm1} 2_{\pm4}^{2} 3_{\pm4}^{2} 7_{\pm4}^{2}, \\
& 1_{\pm1} 4_{\pm1} 5_{0}^{3}  8_{\pm1} 2_{\pm4}^{2} 3_{\pm4}^{3}  6_{\pm3} 7_{\pm4},
1_{\pm1}^{2} 4_{\pm1} 5_{0}^{4}  8_{\pm1} 2_{\pm4}^{2} 3_{\pm4}^{4}  6_{\pm3} 7_{\pm4}^{2},
1_{\pm1}^{2} 4_{\pm1} 5_{0}^{5}  8_{\pm1}^{2} 2_{\pm4}^{3}  3_{\pm4}^{4}  6_{\pm3} 7_{\pm4}^{3},
1_{\pm1} 4_{\pm1} 5_{0}^{3}  8_{\pm1} 2_{\pm4}^{2} 3_{\pm4}^{3}  7_{\pm4}^{2}, \\
& 1_{\pm1} 5_{0}^{3}  8_{\pm1} 2_{\pm4} 3_{\pm4}^{2} 6_{\pm3} 7_{\pm4},
1_{\pm1} 4_{\pm1} 5_{0} 8_{\pm1} 2_{\pm4} 3_{\pm4}^{2} 7_{\pm4},
1_{\pm1}^{2} 4_{\pm1} 5_{0}^{4}  8_{\pm1} 2_{\pm4}^{3}  3_{\pm4}^{3}  6_{\pm3} 7_{\pm4}^{2},
1_{\pm1}^{2} 4_{\pm1} 5_{0}^{5}  8_{\pm1} 2_{\pm4}^{3}  3_{\pm4}^{4}  6_{\pm3} 7_{\pm4}^{2}, \\
& 1_{\pm1}^{2} 4_{\pm1} 5_{0}^{3}  8_{\pm1} 2_{\pm4}^{2} 3_{\pm4}^{3}  6_{\pm3} 7_{\pm4},
1_{\pm1} 4_{\pm1} 5_{0}^{3}  8_{\pm1} 2_{\pm4}^{2} 3_{\pm4}^{2} 6_{\pm3} 7_{\pm4}^{2},
1_{\pm1} 5_{0}^{2} 8_{\pm1} 2_{\pm4} 3_{\pm4} 7_{\pm4}^{2},
1_{\pm1} 4_{\pm1} 5_{0}^{4}  8_{\pm1} 2_{\pm4}^{2} 3_{\pm4}^{3}  6_{\pm3} 7_{\pm4}^{2}, \\
& 1_{\pm1}^{2} 4_{\pm1}^{2} 5_{0}^{4}  8_{\pm1} 2_{\pm4}^{3}  3_{\pm4}^{4}  6_{\pm3} 7_{\pm4}^{2},
1_{\pm1} 4_{\pm1} 5_{0}^{2} 8_{\pm1} 2_{\pm4} 3_{\pm4}^{2} 6_{\pm3} 7_{\pm4},
1_{\pm1} 5_{0}^{3}  8_{\pm1} 2_{\pm4} 3_{\pm4}^{2} 7_{\pm4}^{2},
1_{\pm1}^{2} 4_{\pm1} 5_{0}^{3}  8_{\pm1} 2_{\pm4}^{2} 3_{\pm4}^{3}  7_{\pm4}^{2}, \\
& 1_{\pm1}^{2} 4_{\pm1} 5_{0}^{4}  8_{\pm1} 2_{\pm4}^{2} 3_{\pm4}^{3}  6_{\pm3} 7_{\pm4}^{2},
1_{\pm1} 4_{\pm1} 5_{0}^{2} 8_{\pm1} 2_{\pm4} 3_{\pm4}^{2} 7_{\pm4}^{2},
1_{\pm1} 4_{\pm1} 5_{0}^{2} 8_{\pm1} 2_{\pm4}^{2} 3_{\pm4}^{2} 7_{\pm4},
1_{\pm1} 4_{\pm1} 5_{0}^{3}  8_{\pm1} 2_{\pm4}^{2} 3_{\pm4}^{2} 6_{\pm3} 7_{\pm4}, \\
& 1_{\pm1} 4_{\pm1} 5_{0}^{3}  8_{\pm1} 2_{\pm4}^{2} 3_{\pm4}^{2} 7_{\pm4}^{2},
1_{\pm1} 5_{0}^{2} 8_{\pm1} 2_{\pm4} 3_{\pm4} 7_{\pm4},
1_{\pm1} 4_{\pm1} 5_{0}^{2} 8_{\pm1} 2_{\pm4} 3_{\pm4}^{2} 7_{\pm4},
1_{\pm1} 4_{\pm1} 5_{0} 8_{\pm1} 2_{\pm4} 3_{\pm4} 7_{\pm4}. \\
%\end{align*}
%\end{gather}
%%
%%
& \quad \\
%%
%%
%Let $\xi=(-1,-2,-2,-1,0,-1,0,-1)$. 
%\begin{align*}
%\xymatrix{
%& & 2 &  &  \\
%1  \ar[r] & 3   & 4 \ar[l]  \ar[u] & 5  \ar[l]   \ar[r]  &  6   &  7  \ar[l]  \ar[r] & 8 }
%\end{align*}
%The highest $l$-weight monomials of Hernandez-Leclerc modules of type $E_8$ that are not of type $A$, $D$ or $E_7$ are 
%%\begin{gather}
%%\begin{align*}
& (42)  1_{\pm1} 5_{0}^{2} 7_{0} 2_{\pm4} 3_{\pm4}^{2} 6_{\pm3}^{2} 8_{\pm3},
1_{\pm1} 5_{0} 7_{0} 2_{\pm4} 3_{\pm4} 6_{\pm3} 8_{\pm3},
1_{\pm1} 5_{0}^{3}  7_{0}^{2} 2_{\pm4}^{2} 3_{\pm4}^{2} 6_{\pm3}^{3}  8_{\pm3},
1_{\pm1} 5_{0}^{2} 7_{0}^{2} 2_{\pm4} 3_{\pm4}^{2} 6_{\pm3}^{2} 8_{\pm3}, \\
& 1_{\pm1} 5_{0}^{3}  7_{0} 2_{\pm4}^{2} 3_{\pm4}^{2} 6_{\pm3}^{2} 8_{\pm3},
1_{\pm1} 5_{0}^{4}  7_{0}^{2} 2_{\pm4}^{2} 3_{\pm4}^{3}  6_{\pm3}^{3}  8_{\pm3},
1_{\pm1} 5_{0}^{2} 7_{0} 2_{\pm4} 3_{\pm4}^{2} 6_{\pm3} 8_{\pm3},
1_{\pm1} 4_{\pm1} 5_{0}^{2} 7_{0} 2_{\pm4}^{2} 3_{\pm4}^{2} 6_{\pm3}^{2} 8_{\pm3}, \\
& 1_{\pm1} 4_{\pm1} 5_{0}^{3}  7_{0}^{2} 2_{\pm4}^{2} 3_{\pm4}^{3}  6_{\pm3}^{3}  8_{\pm3},
1_{\pm1}^{2} 4_{\pm1} 5_{0}^{4}  7_{0}^{2} 2_{\pm4}^{3}  3_{\pm4}^{4}  6_{\pm3}^{3}  8_{\pm3},
1_{\pm1} 4_{\pm1} 5_{0}^{3}  7_{0} 2_{\pm4}^{2} 3_{\pm4}^{3}  6_{\pm3}^{2} 8_{\pm3},
1_{\pm1} 5_{0}^{2} 7_{0} 2_{\pm4} 3_{\pm4} 6_{\pm3}^{2} 8_{\pm3}, \\
& 1_{\pm1} 5_{0}^{3}  7_{0}^{2} 2_{\pm4} 3_{\pm4}^{2} 6_{\pm3}^{3}  8_{\pm3},
1_{\pm1} 4_{\pm1} 5_{0} 7_{0} 2_{\pm4} 3_{\pm4}^{2} 6_{\pm3} 8_{\pm3},
1_{\pm1}^{2} 5_{0}^{4}  7_{0}^{2} 2_{\pm4}^{2} 3_{\pm4}^{3}  6_{\pm3}^{3}  8_{\pm3},
1_{\pm1}^{2} 4_{\pm1} 5_{0}^{5}  7_{0}^{2} 2_{\pm4}^{3}  3_{\pm4}^{4}  6_{\pm3}^{4}  8_{\pm3}, \\
& 1_{\pm1}^{2} 4_{\pm1} 5_{0}^{3}  7_{0}^{2} 2_{\pm4}^{2} 3_{\pm4}^{3}  6_{\pm3}^{3}  8_{\pm3},
1_{\pm1} 5_{0}^{3}  7_{0} 2_{\pm4} 3_{\pm4}^{2} 6_{\pm3}^{2} 8_{\pm3},
1_{\pm1} 5_{0}^{3}  7_{0}^{2} 2_{\pm4}^{2} 3_{\pm4}^{2} 6_{\pm3}^{2} 8_{\pm3},
1_{\pm1} 4_{\pm1} 5_{0}^{4}  7_{0}^{2} 2_{\pm4}^{3}  3_{\pm4}^{3}  6_{\pm3}^{3}  8_{\pm3}, \\
& 1_{\pm1}^{2} 4_{\pm1} 5_{0}^{5}  7_{0}^{3}  2_{\pm4}^{3}  3_{\pm4}^{4}  6_{\pm3}^{4}  8_{\pm3}^{2},
1_{\pm1} 5_{0}^{4}  7_{0}^{2} 2_{\pm4}^{2} 3_{\pm4}^{2} 6_{\pm3}^{3}  8_{\pm3},
1_{\pm1} 4_{\pm1} 5_{0}^{2} 7_{0}^{2} 2_{\pm4}^{2} 3_{\pm4}^{2} 6_{\pm3}^{2} 8_{\pm3},
1_{\pm1}^{2} 4_{\pm1} 5_{0}^{4}  7_{0}^{2} 2_{\pm4}^{2} 3_{\pm4}^{4}  6_{\pm3}^{3}  8_{\pm3}, \\
& 1_{\pm1}^{2} 4_{\pm1} 5_{0}^{5}  7_{0}^{3}  2_{\pm4}^{3}  3_{\pm4}^{4}  6_{\pm3}^{4}  8_{\pm3},
1_{\pm1} 4_{\pm1} 5_{0}^{3}  7_{0}^{2} 2_{\pm4}^{2} 3_{\pm4}^{3}  6_{\pm3}^{2} 8_{\pm3},
1_{\pm1}^{2} 4_{\pm1} 5_{0}^{3}  7_{0} 2_{\pm4}^{2} 3_{\pm4}^{3}  6_{\pm3}^{2} 8_{\pm3},
1_{\pm1}^{2} 4_{\pm1} 5_{0}^{4}  7_{0}^{2} 2_{\pm4}^{3}  3_{\pm4}^{3}  6_{\pm3}^{3}  8_{\pm3}, \\
& 1_{\pm1}^{2} 4_{\pm1} 5_{0}^{5}  7_{0}^{2} 2_{\pm4}^{3}  3_{\pm4}^{4}  6_{\pm3}^{3}  8_{\pm3},
1_{\pm1} 4_{\pm1} 5_{0}^{2} 7_{0} 2_{\pm4}^{2} 3_{\pm4}^{2} 6_{\pm3} 8_{\pm3},
1_{\pm1} 4_{\pm1} 5_{0}^{2} 7_{0} 2_{\pm4} 3_{\pm4}^{2} 6_{\pm3}^{2} 8_{\pm3},
1_{\pm1} 4_{\pm1} 5_{0}^{3}  7_{0}^{2} 2_{\pm4}^{2} 3_{\pm4}^{2} 6_{\pm3}^{3}  8_{\pm3}, \\
& 1_{\pm1} 4_{\pm1} 5_{0}^{4}  7_{0}^{2} 2_{\pm4}^{2} 3_{\pm4}^{3}  6_{\pm3}^{3}  8_{\pm3},
1_{\pm1}^{2} 4_{\pm1}^{2} 5_{0}^{4}  7_{0}^{2} 2_{\pm4}^{3}  3_{\pm4}^{4}  6_{\pm3}^{3}  8_{\pm3},
1_{\pm1} 4_{\pm1} 5_{0}^{3}  7_{0} 2_{\pm4}^{2} 3_{\pm4}^{2} 6_{\pm3}^{2} 8_{\pm3},
1_{\pm1} 5_{0}^{2} 7_{0}^{2} 2_{\pm4} 3_{\pm4} 6_{\pm3}^{2} 8_{\pm3}, \\
& 1_{\pm1} 5_{0}^{3}  7_{0}^{2} 2_{\pm4} 3_{\pm4}^{2} 6_{\pm3}^{2} 8_{\pm3},
1_{\pm1}^{2} 4_{\pm1} 5_{0}^{3}  7_{0}^{2} 2_{\pm4}^{2} 3_{\pm4}^{3}  6_{\pm3}^{2} 8_{\pm3},
1_{\pm1}^{2} 4_{\pm1} 5_{0}^{4}  7_{0}^{2} 2_{\pm4}^{2} 3_{\pm4}^{3}  6_{\pm3}^{3}  8_{\pm3},
1_{\pm1} 4_{\pm1} 5_{0}^{2} 7_{0}^{2} 2_{\pm4} 3_{\pm4}^{2} 6_{\pm3}^{2} 8_{\pm3}, \\
& 1_{\pm1} 5_{0}^{2} 7_{0} 2_{\pm4} 3_{\pm4} 6_{\pm3} 8_{\pm3},
1_{\pm1} 4_{\pm1} 5_{0}^{3}  7_{0}^{2} 2_{\pm4}^{2} 3_{\pm4}^{2} 6_{\pm3}^{2} 8_{\pm3},
1_{\pm1} 4_{\pm1} 5_{0}^{2} 7_{0} 2_{\pm4} 3_{\pm4}^{2} 6_{\pm3} 8_{\pm3},
1_{\pm1} 4_{\pm1} 5_{0} 7_{0} 2_{\pm4} 3_{\pm4} 6_{\pm3} 8_{\pm3}.
\end{align*}
\end{gather}

%Let $\xi=(-2,-3,-3,-2,-1,-2,-1,0)$. 
%\begin{align*}
%\xymatrix{
%& & 2 &  &  \\
%1  \ar[r] & 3   & 4 \ar[l]  \ar[u] & 5  \ar[l]   \ar[r]  &  6   &  7  \ar[l] & 8  \ar[l] }
%\end{align*}
%The highest $l$-weight monomials of Hernandez-Leclerc modules of type $E_8$ that are not of type $A$, $D$ or $E_7$ are 
\begin{gather}
\begin{align*}
& (43)  1_{\pm2} 5_{\pm1}^{2} 8_{0} 2_{\pm5} 3_{\pm5}^{2} 6_{\pm4}^{2},
1_{\pm2} 5_{\pm1} 8_{0} 2_{\pm5} 3_{\pm5} 6_{\pm4},
1_{\pm2} 5_{\pm1}^{3}  7_{\pm1} 8_{0} 2_{\pm5}^{2} 3_{\pm5}^{2} 6_{\pm4}^{3},
1_{\pm2} 5_{\pm1}^{2} 7_{\pm1} 8_{0} 2_{\pm5} 3_{\pm5}^{2} 6_{\pm4}^{2}, \\
& 1_{\pm2} 5_{\pm1}^{3}  8_{0} 2_{\pm5}^{2} 3_{\pm5}^{2} 6_{\pm4}^{2},
1_{\pm2} 5_{\pm1}^{4}  7_{\pm1} 8_{0} 2_{\pm5}^{2} 3_{\pm5}^{3}  6_{\pm4}^{3},
1_{\pm2} 5_{\pm1}^{2} 8_{0} 2_{\pm5} 3_{\pm5}^{2} 6_{\pm4},
1_{\pm2} 4_{\pm2} 5_{\pm1}^{2} 8_{0} 2_{\pm5}^{2} 3_{\pm5}^{2} 6_{\pm4}^{2}, \\
& 1_{\pm2} 4_{\pm2} 5_{\pm1}^{3}  7_{\pm1} 8_{0} 2_{\pm5}^{2} 3_{\pm5}^{3}  6_{\pm4}^{3},
1_{\pm2}^{2} 4_{\pm2} 5_{\pm1}^{4}  7_{\pm1} 8_{0} 2_{\pm5}^{3}  3_{\pm5}^{4}  6_{\pm4}^{3},
1_{\pm2} 4_{\pm2} 5_{\pm1}^{3}  8_{0} 2_{\pm5}^{2} 3_{\pm5}^{3}  6_{\pm4}^{2},
1_{\pm2} 5_{\pm1}^{2} 8_{0} 2_{\pm5} 3_{\pm5} 6_{\pm4}^{2}, \\
& 1_{\pm2} 5_{\pm1}^{3}  7_{\pm1} 8_{0} 2_{\pm5} 3_{\pm5}^{2} 6_{\pm4}^{3},
1_{\pm2} 4_{\pm2} 5_{\pm1} 8_{0} 2_{\pm5} 3_{\pm5}^{2} 6_{\pm4},
1_{\pm2}^{2} 5_{\pm1}^{4}  7_{\pm1} 8_{0} 2_{\pm5}^{2} 3_{\pm5}^{3}  6_{\pm4}^{3},
1_{\pm2}^{2} 4_{\pm2} 5_{\pm1}^{5}  7_{\pm1} 8_{0} 2_{\pm5}^{3}  3_{\pm5}^{4}  6_{\pm4}^{4}, \\
& 1_{\pm2}^{2} 4_{\pm2} 5_{\pm1}^{3}  7_{\pm1} 8_{0} 2_{\pm5}^{2} 3_{\pm5}^{3}  6_{\pm4}^{3},
1_{\pm2} 5_{\pm1}^{3}  8_{0} 2_{\pm5} 3_{\pm5}^{2} 6_{\pm4}^{2},
1_{\pm2} 5_{\pm1}^{3}  7_{\pm1} 8_{0} 2_{\pm5}^{2} 3_{\pm5}^{2} 6_{\pm4}^{2},
1_{\pm2} 4_{\pm2} 5_{\pm1}^{4}  7_{\pm1} 8_{0} 2_{\pm5}^{3}  3_{\pm5}^{3}  6_{\pm4}^{3}, \\
& 1_{\pm2}^{2} 4_{\pm2} 5_{\pm1}^{5}  7_{\pm1} 8_{0}^{2} 2_{\pm5}^{3}  3_{\pm5}^{4}  6_{\pm4}^{4},
1_{\pm2} 5_{\pm1}^{4}  7_{\pm1} 8_{0} 2_{\pm5}^{2} 3_{\pm5}^{2} 6_{\pm4}^{3},
1_{\pm2} 4_{\pm2} 5_{\pm1}^{2} 7_{\pm1} 8_{0} 2_{\pm5}^{2} 3_{\pm5}^{2} 6_{\pm4}^{2},
1_{\pm2}^{2} 4_{\pm2} 5_{\pm1}^{4}  7_{\pm1} 8_{0} 2_{\pm5}^{2} 3_{\pm5}^{4}  6_{\pm4}^{3}, \\
& 1_{\pm2}^{2} 4_{\pm2} 5_{\pm1}^{5}  7_{\pm1}^{2} 8_{0} 2_{\pm5}^{3}  3_{\pm5}^{4}  6_{\pm4}^{4},
1_{\pm2} 4_{\pm2} 5_{\pm1}^{3}  7_{\pm1} 8_{0} 2_{\pm5}^{2} 3_{\pm5}^{3}  6_{\pm4}^{2},
1_{\pm2}^{2} 4_{\pm2} 5_{\pm1}^{3}  8_{0} 2_{\pm5}^{2} 3_{\pm5}^{3}  6_{\pm4}^{2},
1_{\pm2}^{2} 4_{\pm2} 5_{\pm1}^{4}  7_{\pm1} 8_{0} 2_{\pm5}^{3}  3_{\pm5}^{3}  6_{\pm4}^{3}, \\
& 1_{\pm2}^{2} 4_{\pm2} 5_{\pm1}^{5}  7_{\pm1} 8_{0} 2_{\pm5}^{3}  3_{\pm5}^{4}  6_{\pm4}^{3},
1_{\pm2} 4_{\pm2} 5_{\pm1}^{2} 8_{0} 2_{\pm5}^{2} 3_{\pm5}^{2} 6_{\pm4},
1_{\pm2} 4_{\pm2} 5_{\pm1}^{2} 8_{0} 2_{\pm5} 3_{\pm5}^{2} 6_{\pm4}^{2},
1_{\pm2} 4_{\pm2} 5_{\pm1}^{3}  7_{\pm1} 8_{0} 2_{\pm5}^{2} 3_{\pm5}^{2} 6_{\pm4}^{3}, \\
& 1_{\pm2} 4_{\pm2} 5_{\pm1}^{4}  7_{\pm1} 8_{0} 2_{\pm5}^{2} 3_{\pm5}^{3}  6_{\pm4}^{3},
1_{\pm2}^{2} 4_{\pm2}^{2} 5_{\pm1}^{4}  7_{\pm1} 8_{0} 2_{\pm5}^{3}  3_{\pm5}^{4}  6_{\pm4}^{3},
1_{\pm2} 4_{\pm2} 5_{\pm1}^{3}  8_{0} 2_{\pm5}^{2} 3_{\pm5}^{2} 6_{\pm4}^{2},
1_{\pm2} 5_{\pm1}^{2} 7_{\pm1} 8_{0} 2_{\pm5} 3_{\pm5} 6_{\pm4}^{2}, \\
& 1_{\pm2} 5_{\pm1}^{3}  7_{\pm1} 8_{0} 2_{\pm5} 3_{\pm5}^{2} 6_{\pm4}^{2},
1_{\pm2}^{2} 4_{\pm2} 5_{\pm1}^{3}  7_{\pm1} 8_{0} 2_{\pm5}^{2} 3_{\pm5}^{3}  6_{\pm4}^{2},
1_{\pm2}^{2} 4_{\pm2} 5_{\pm1}^{4}  7_{\pm1} 8_{0} 2_{\pm5}^{2} 3_{\pm5}^{3}  6_{\pm4}^{3},
1_{\pm2} 4_{\pm2} 5_{\pm1}^{2} 7_{\pm1} 8_{0} 2_{\pm5} 3_{\pm5}^{2} 6_{\pm4}^{2}, \\
& 1_{\pm2} 5_{\pm1}^{2} 8_{0} 2_{\pm5} 3_{\pm5} 6_{\pm4},
1_{\pm2} 4_{\pm2} 5_{\pm1}^{3}  7_{\pm1} 8_{0} 2_{\pm5}^{2} 3_{\pm5}^{2} 6_{\pm4}^{2},
1_{\pm2} 4_{\pm2} 5_{\pm1}^{2} 8_{0} 2_{\pm5} 3_{\pm5}^{2} 6_{\pm4},
1_{\pm2} 4_{\pm2} 5_{\pm1} 8_{0} 2_{\pm5} 3_{\pm5} 6_{\pm4}. \\
%\end{align*}
%\end{gather}
%%
%%
& \quad \\
%%
%%
%%
%Let $\xi=(-2,-3,-3,-2,-1,0,-1,-2)$. 
%\begin{align*}
%\xymatrix{
%& & 2 &  &  \\
%1  \ar[r] & 3   & 4 \ar[l]  \ar[u] & 5  \ar[l]  &  6   \ar[l]  \ar[r]  &  7   \ar[r]  & 8 }
%\end{align*}
%The highest $l$-weight monomials of Hernandez-Leclerc modules of type $E_8$ that are not of type $A$, $D$ or $E_7$ are 
%\begin{gather}
%\begin{align*}
& (44)  1_{\pm2} 6_{0}^{2} 2_{\pm5} 3_{\pm5}^{2} 7_{\pm3} 8_{\pm4},
1_{\pm2} 6_{0} 2_{\pm5} 3_{\pm5} 8_{\pm4},
1_{\pm2} 6_{0}^{3} 2_{\pm5}^{2} 3_{\pm5}^{2} 7_{\pm3} 8_{\pm4},
1_{\pm2} 6_{0}^{2} 2_{\pm5} 3_{\pm5}^{2} 8_{\pm4}, \\
& 1_{\pm2} 5_{\pm1} 6_{0}^{2} 2_{\pm5}^{2} 3_{\pm5}^{2} 7_{\pm3} 8_{\pm4},
1_{\pm2} 5_{\pm1} 6_{0}^{3} 2_{\pm5}^{2} 3_{\pm5}^{3} 7_{\pm3} 8_{\pm4},
1_{\pm2} 5_{\pm1} 6_{0} 2_{\pm5} 3_{\pm5}^{2} 8_{\pm4},
1_{\pm2} 4_{\pm2} 6_{0}^{2} 2_{\pm5}^{2} 3_{\pm5}^{2} 7_{\pm3} 8_{\pm4}, \\
& 1_{\pm2} 4_{\pm2} 6_{0}^{3} 2_{\pm5}^{2} 3_{\pm5}^{3} 7_{\pm3} 8_{\pm4},
1_{\pm2}^{2} 4_{\pm2} 5_{\pm1} 6_{0}^{3} 2_{\pm5}^{3} 3_{\pm5}^{4} 7_{\pm3} 8_{\pm4},
1_{\pm2} 4_{\pm2} 5_{\pm1} 6_{0}^{2} 2_{\pm5}^{2} 3_{\pm5}^{3} 7_{\pm3} 8_{\pm4},
1_{\pm2} 6_{0}^{2} 2_{\pm5} 3_{\pm5} 7_{\pm3} 8_{\pm4}, \\
& 1_{\pm2} 4_{\pm2} 6_{0} 2_{\pm5} 3_{\pm5}^{2} 8_{\pm4},
1_{\pm2} 6_{0}^{3} 2_{\pm5} 3_{\pm5}^{2} 7_{\pm3} 8_{\pm4},
1_{\pm2}^{2} 5_{\pm1} 6_{0}^{3} 2_{\pm5}^{2} 3_{\pm5}^{3} 7_{\pm3} 8_{\pm4},
1_{\pm2}^{2} 4_{\pm2} 5_{\pm1} 6_{0}^{4} 2_{\pm5}^{3} 3_{\pm5}^{4} 7_{\pm3}^{2} 8_{\pm4}, \\
& 1_{\pm2}^{2} 4_{\pm2} 6_{0}^{3} 2_{\pm5}^{2} 3_{\pm5}^{3} 7_{\pm3} 8_{\pm4},
1_{\pm2} 5_{\pm1} 6_{0}^{2} 2_{\pm5} 3_{\pm5}^{2} 7_{\pm3} 8_{\pm4},
1_{\pm2} 5_{\pm1} 6_{0}^{2} 2_{\pm5}^{2} 3_{\pm5}^{2} 8_{\pm4},
1_{\pm2} 4_{\pm2} 5_{\pm1} 6_{0}^{3} 2_{\pm5}^{3} 3_{\pm5}^{3} 7_{\pm3} 8_{\pm4}, \\
& 1_{\pm2}^{2} 4_{\pm2} 5_{\pm1} 6_{0}^{4} 2_{\pm5}^{3} 3_{\pm5}^{4} 7_{\pm3} 8_{\pm4}^{2},
1_{\pm2} 5_{\pm1} 6_{0}^{3} 2_{\pm5}^{2} 3_{\pm5}^{2} 7_{\pm3} 8_{\pm4},
1_{\pm2} 4_{\pm2} 6_{0}^{2} 2_{\pm5}^{2} 3_{\pm5}^{2} 8_{\pm4},
1_{\pm2}^{2} 4_{\pm2} 5_{\pm1} 6_{0}^{3} 2_{\pm5}^{2} 3_{\pm5}^{4} 7_{\pm3} 8_{\pm4}, \\
& 1_{\pm2}^{2} 4_{\pm2} 5_{\pm1} 6_{0}^{4} 2_{\pm5}^{3} 3_{\pm5}^{4} 7_{\pm3} 8_{\pm4},
1_{\pm2} 4_{\pm2} 5_{\pm1} 6_{0}^{2} 2_{\pm5}^{2} 3_{\pm5}^{3} 8_{\pm4},
1_{\pm2}^{2} 4_{\pm2} 5_{\pm1} 6_{0}^{2} 2_{\pm5}^{2} 3_{\pm5}^{3} 7_{\pm3} 8_{\pm4},
1_{\pm2}^{2} 4_{\pm2} 5_{\pm1} 6_{0}^{3} 2_{\pm5}^{3} 3_{\pm5}^{3} 7_{\pm3} 8_{\pm4}, \\
& 1_{\pm2}^{2} 4_{\pm2} 5_{\pm1}^{2} 6_{0}^{3} 2_{\pm5}^{3} 3_{\pm5}^{4} 7_{\pm3} 8_{\pm4},
1_{\pm2} 4_{\pm2} 5_{\pm1} 6_{0} 2_{\pm5}^{2} 3_{\pm5}^{2} 8_{\pm4},
1_{\pm2} 4_{\pm2} 6_{0}^{2} 2_{\pm5} 3_{\pm5}^{2} 7_{\pm3} 8_{\pm4},
1_{\pm2} 4_{\pm2} 6_{0}^{3} 2_{\pm5}^{2} 3_{\pm5}^{2} 7_{\pm3} 8_{\pm4}, \\
& 1_{\pm2} 4_{\pm2} 5_{\pm1} 6_{0}^{3} 2_{\pm5}^{2} 3_{\pm5}^{3} 7_{\pm3} 8_{\pm4},
1_{\pm2}^{2} 4_{\pm2}^{2} 5_{\pm1} 6_{0}^{3} 2_{\pm5}^{3} 3_{\pm5}^{4} 7_{\pm3} 8_{\pm4},
1_{\pm2} 4_{\pm2} 5_{\pm1} 6_{0}^{2} 2_{\pm5}^{2} 3_{\pm5}^{2} 7_{\pm3} 8_{\pm4},
1_{\pm2} 6_{0}^{2} 2_{\pm5} 3_{\pm5} 8_{\pm4}, \\
& 1_{\pm2} 5_{\pm1} 6_{0}^{2} 2_{\pm5} 3_{\pm5}^{2} 8_{\pm4},
1_{\pm2}^{2} 4_{\pm2} 5_{\pm1} 6_{0}^{2} 2_{\pm5}^{2} 3_{\pm5}^{3} 8_{\pm4},
1_{\pm2}^{2} 4_{\pm2} 5_{\pm1} 6_{0}^{3} 2_{\pm5}^{2} 3_{\pm5}^{3} 7_{\pm3} 8_{\pm4},
1_{\pm2} 4_{\pm2} 6_{0}^{2} 2_{\pm5} 3_{\pm5}^{2} 8_{\pm4}, \\
& 1_{\pm2} 5_{\pm1} 6_{0} 2_{\pm5} 3_{\pm5} 8_{\pm4},
1_{\pm2} 4_{\pm2} 5_{\pm1} 6_{0}^{2} 2_{\pm5}^{2} 3_{\pm5}^{2} 8_{\pm4},
1_{\pm2} 4_{\pm2} 5_{\pm1} 6_{0} 2_{\pm5} 3_{\pm5}^{2} 8_{\pm4},
1_{\pm2} 4_{\pm2} 6_{0} 2_{\pm5} 3_{\pm5} 8_{\pm4}. \\
%\end{align*}
%\end{gather}
%%
%%
& \quad \\
%%
%%
%Let $\xi=(-2,-3,-3,-2,-1,0,-1,0)$. 
%\begin{align*}
%\xymatrix{
%& & 2 &  &  \\
%1  \ar[r] & 3   & 4 \ar[l]  \ar[u] & 5  \ar[l]  &  6   \ar[l]  \ar[r]  &  7 & 8   \ar[l]  }
%\end{align*}
%The highest $l$-weight monomials of Hernandez-Leclerc modules of type $E_8$ that are not of type $A$, $D$ or $E_7$ are 
%\begin{gather}
%\begin{align*}
& (45)  1_{\pm2} 6_{0}^{2} 8_{0} 2_{\pm5} 3_{\pm5}^{2} 7_{\pm3}^{2},
1_{\pm2} 6_{0} 8_{0} 2_{\pm5} 3_{\pm5} 7_{\pm3},
1_{\pm2} 6_{0}^{3} 8_{0} 2_{\pm5}^{2} 3_{\pm5}^{2} 7_{\pm3}^{2},
1_{\pm2} 6_{0}^{2} 8_{0} 2_{\pm5} 3_{\pm5}^{2} 7_{\pm3}, \\
& 1_{\pm2} 5_{\pm1} 6_{0}^{2} 8_{0} 2_{\pm5}^{2} 3_{\pm5}^{2} 7_{\pm3}^{2},
1_{\pm2} 5_{\pm1} 6_{0}^{3} 8_{0} 2_{\pm5}^{2} 3_{\pm5}^{3} 7_{\pm3}^{2},
1_{\pm2} 5_{\pm1} 6_{0} 8_{0} 2_{\pm5} 3_{\pm5}^{2} 7_{\pm3},
1_{\pm2} 4_{\pm2} 6_{0}^{2} 8_{0} 2_{\pm5}^{2} 3_{\pm5}^{2} 7_{\pm3}^{2}, \\
& 1_{\pm2} 4_{\pm2} 6_{0}^{3} 8_{0} 2_{\pm5}^{2} 3_{\pm5}^{3} 7_{\pm3}^{2},
1_{\pm2}^{2} 4_{\pm2} 5_{\pm1} 6_{0}^{3} 8_{0} 2_{\pm5}^{3} 3_{\pm5}^{4} 7_{\pm3}^{2},
1_{\pm2} 4_{\pm2} 5_{\pm1} 6_{0}^{2} 8_{0} 2_{\pm5}^{2} 3_{\pm5}^{3} 7_{\pm3}^{2},
1_{\pm2} 6_{0}^{2} 8_{0} 2_{\pm5} 3_{\pm5} 7_{\pm3}^{2}, \\
& 1_{\pm2} 4_{\pm2} 6_{0} 8_{0} 2_{\pm5} 3_{\pm5}^{2} 7_{\pm3},
1_{\pm2} 6_{0}^{3} 8_{0} 2_{\pm5} 3_{\pm5}^{2} 7_{\pm3}^{2},
1_{\pm2}^{2} 5_{\pm1} 6_{0}^{3} 8_{0} 2_{\pm5}^{2} 3_{\pm5}^{3} 7_{\pm3}^{2},
1_{\pm2}^{2} 4_{\pm2} 5_{\pm1} 6_{0}^{4} 8_{0} 2_{\pm5}^{3} 3_{\pm5}^{4} 7_{\pm3}^{3}, \\
& 1_{\pm2}^{2} 4_{\pm2} 6_{0}^{3} 8_{0} 2_{\pm5}^{2} 3_{\pm5}^{3} 7_{\pm3}^{2},
1_{\pm2} 5_{\pm1} 6_{0}^{2} 8_{0} 2_{\pm5} 3_{\pm5}^{2} 7_{\pm3}^{2},
1_{\pm2} 5_{\pm1} 6_{0}^{2} 8_{0} 2_{\pm5}^{2} 3_{\pm5}^{2} 7_{\pm3}, 
1_{\pm2} 4_{\pm2} 5_{\pm1} 6_{0}^{3} 8_{0} 2_{\pm5}^{3} 3_{\pm5}^{3} 7_{\pm3}^{2}, \\
& 1_{\pm2}^{2} 4_{\pm2} 5_{\pm1} 6_{0}^{4} 8_{0}^{2} 2_{\pm5}^{3} 3_{\pm5}^{4} 7_{\pm3}^{3},
1_{\pm2} 5_{\pm1} 6_{0}^{3} 8_{0} 2_{\pm5}^{2} 3_{\pm5}^{2} 7_{\pm3}^{2},
1_{\pm2} 4_{\pm2} 6_{0}^{2} 8_{0} 2_{\pm5}^{2} 3_{\pm5}^{2} 7_{\pm3},
1_{\pm2}^{2} 4_{\pm2} 5_{\pm1} 6_{0}^{3} 8_{0} 2_{\pm5}^{2} 3_{\pm5}^{4} 7_{\pm3}^{2}, \\
& 1_{\pm2}^{2} 4_{\pm2} 5_{\pm1} 6_{0}^{4} 8_{0} 2_{\pm5}^{3} 3_{\pm5}^{4} 7_{\pm3}^{2},
1_{\pm2} 4_{\pm2} 5_{\pm1} 6_{0}^{2} 8_{0} 2_{\pm5}^{2} 3_{\pm5}^{3} 7_{\pm3},
1_{\pm2}^{2} 4_{\pm2} 5_{\pm1} 6_{0}^{2} 8_{0} 2_{\pm5}^{2} 3_{\pm5}^{3} 7_{\pm3}^{2},
1_{\pm2}^{2} 4_{\pm2} 5_{\pm1} 6_{0}^{3} 8_{0} 2_{\pm5}^{3} 3_{\pm5}^{3} 7_{\pm3}^{2}, \\
& 1_{\pm2}^{2} 4_{\pm2} 5_{\pm1}^{2} 6_{0}^{3} 8_{0} 2_{\pm5}^{3} 3_{\pm5}^{4} 7_{\pm3}^{2},
1_{\pm2} 4_{\pm2} 5_{\pm1} 6_{0} 8_{0} 2_{\pm5}^{2} 3_{\pm5}^{2} 7_{\pm3},
1_{\pm2} 4_{\pm2} 6_{0}^{2} 8_{0} 2_{\pm5} 3_{\pm5}^{2} 7_{\pm3}^{2},
1_{\pm2} 4_{\pm2} 6_{0}^{3} 8_{0} 2_{\pm5}^{2} 3_{\pm5}^{2} 7_{\pm3}^{2}, \\
& 1_{\pm2} 4_{\pm2} 5_{\pm1} 6_{0}^{3} 8_{0} 2_{\pm5}^{2} 3_{\pm5}^{3} 7_{\pm3}^{2},
1_{\pm2}^{2} 4_{\pm2}^{2} 5_{\pm1} 6_{0}^{3} 8_{0} 2_{\pm5}^{3} 3_{\pm5}^{4} 7_{\pm3}^{2},
1_{\pm2} 4_{\pm2} 5_{\pm1} 6_{0}^{2} 8_{0} 2_{\pm5}^{2} 3_{\pm5}^{2} 7_{\pm3}^{2},
1_{\pm2} 6_{0}^{2} 8_{0} 2_{\pm5} 3_{\pm5} 7_{\pm3}, \\
& 1_{\pm2} 5_{\pm1} 6_{0}^{2} 8_{0} 2_{\pm5} 3_{\pm5}^{2} 7_{\pm3},
1_{\pm2}^{2} 4_{\pm2} 5_{\pm1} 6_{0}^{2} 8_{0} 2_{\pm5}^{2} 3_{\pm5}^{3} 7_{\pm3},
1_{\pm2}^{2} 4_{\pm2} 5_{\pm1} 6_{0}^{3} 8_{0} 2_{\pm5}^{2} 3_{\pm5}^{3} 7_{\pm3}^{2},
1_{\pm2} 4_{\pm2} 6_{0}^{2} 8_{0} 2_{\pm5} 3_{\pm5}^{2} 7_{\pm3}, \\
& 1_{\pm2} 5_{\pm1} 6_{0} 8_{0} 2_{\pm5} 3_{\pm5} 7_{\pm3},
1_{\pm2} 4_{\pm2} 5_{\pm1} 6_{0}^{2} 8_{0} 2_{\pm5}^{2} 3_{\pm5}^{2} 7_{\pm3},
1_{\pm2} 4_{\pm2} 5_{\pm1} 6_{0} 8_{0} 2_{\pm5} 3_{\pm5}^{2} 7_{\pm3},
1_{\pm2} 4_{\pm2} 6_{0} 8_{0} 2_{\pm5} 3_{\pm5} 7_{\pm3}.
\end{align*}
\end{gather}

%Let $\xi=(-3,-4,-4_{\pm3},-2,-1,0,-1)$. 
%\begin{align*}
%\xymatrix{
%& & 2 &  &  \\
%1  \ar[r] & 3   & 4 \ar[l]  \ar[u] & 5  \ar[l]  &  6   \ar[l] &  7  \ar[l]    \ar[r]   & 8 }
%\end{align*}
%The highest $l$-weight monomials of Hernandez-Leclerc modules of type $E_8$ that are not of type $A$, $D$ or $E_7$ are 
\begin{gather}
\begin{align*}
& (46)  1_{\pm3} 7_{0} 2_{\pm6} 3_{\pm6} 8_{\pm3},
1_{\pm3} 7_{0}^{2} 2_{\pm6} 3_{\pm6}^{2} 8_{\pm3},
1_{\pm3} 6_{\pm1} 7_{0} 2_{\pm6} 3_{\pm6}^{2} 8_{\pm3},
1_{\pm3} 6_{\pm1} 7_{0}^{2} 2_{\pm6}^{2} 3_{\pm6}^{2} 8_{\pm3},  \\
& 1_{\pm3} 5_{\pm2} 7_{0} 2_{\pm6} 3_{\pm6}^{2} 8_{\pm3},
1_{\pm3} 5_{\pm2} 7_{0}^{2} 2_{\pm6}^{2} 3_{\pm6}^{2} 8_{\pm3},
1_{\pm3} 5_{\pm2} 6_{\pm1} 7_{0}^{2} 2_{\pm6}^{2} 3_{\pm6}^{3} 8_{\pm3},
1 _{\pm3} 5_{\pm2} 6_{\pm1} 7_{0} 2_{\pm6}^{2} 3_{\pm6}^{2} 8_{\pm3}, \\
& 1_{\pm3} 4_{\pm3} 7_{0} 2_{\pm6} 3_{\pm6}^{2} 8_{\pm3},
 1_{\pm3} 4_{\pm3} 7_{0}^{2} 2_{\pm6}^{2} 3_{\pm6}^{2} 8_{\pm3},
 1_{\pm3} 4_{\pm3} 6_{\pm1} 7_{0}^{2} 2_{\pm6}^{2} 3_{\pm6}^{3} 8_{\pm3},
 1_{\pm3}^{2} 4_{\pm3} 5_{\pm2} 6_{\pm1} 7_{0}^{2} 2_{\pm6}^{3} 3_{\pm6}^{4} 8_{\pm3}, \\
& 1_{\pm3} 4_{\pm3} 5_{\pm2} 7_{0}^{2} 2_{\pm6}^{2} 3_{\pm6}^{3} 8_{\pm3},
 1_{\pm3} 4_{\pm3} 6_{\pm1} 7_{0} 2_{\pm6}^{2} 3_{\pm6}^{2} 8_{\pm3},
 1_{\pm3} 7_{0}^{2} 2_{\pm6} 3_{\pm6} 8_{\pm3},
 1_{\pm3} 6_{\pm1} 7_{0}^{2} 2_{\pm6} 3_{\pm6}^{2} 8_{\pm3}, \\
& 1_{\pm3}^{2} 5_{\pm2} 6_{\pm1} 7_{0}^{2} 2_{\pm6}^{2} 3_{\pm6}^{3} 8_{\pm3},
 1_{\pm3}^{2} 4_{\pm3} 5_{\pm2} 6_{\pm1} 7_{0}^{3} 2_{\pm6}^{3} 3_{\pm6}^{4} 8_{\pm3}^{2},
 1_{\pm3}^{2} 4_{\pm3} 6_{\pm1} 7_{0}^{2} 2_{\pm6}^{2} 3_{\pm6}^{3} 8_{\pm3},
 1_{\pm3} 5_{\pm2} 7_{0}^{2} 2_{\pm6} 3_{\pm6}^{2} 8_{\pm3}, \\
& 1_{\pm3} 6_{\pm1} 7_{0} 2_{\pm6} 3_{\pm6} 8_{\pm3},
 1_{\pm3} 4_{\pm3} 5_{\pm2} 6_{\pm1} 7_{0}^{2} 2_{\pm6}^{3} 3_{\pm6}^{3} 8_{\pm3},
 1_{\pm3}^{2} 4_{\pm3} 5_{\pm2} 6_{\pm1} 7_{0}^{3} 2_{\pm6}^{3} 3_{\pm6}^{4} 8_{\pm3},
 1_{\pm3} 5_{\pm2} 6_{\pm1} 7_{0}^{2} 2_{\pm6}^{2} 3_{\pm6}^{2} 8_{\pm3}, \\
& 1_{\pm3} 4_{\pm3} 5_{\pm2} 6_{\pm1} 7_{0} 2_{\pm6}^{2} 3_{\pm6}^{3} 8_{\pm3},
 1_{\pm3}^{2} 4_{\pm3} 5_{\pm2} 6_{\pm1} 7_{0}^{2} 2_{\pm6}^{2} 3_{\pm6}^{4} 8_{\pm3},
 1_{\pm3}^{2} 4_{\pm3} 5_{\pm2} 6_{\pm1}^{2} 7_{0}^{2} 2_{\pm6}^{3} 3_{\pm6}^{4} 8_{\pm3},
 1_{\pm3} 5_{\pm2} 6_{\pm1} 7_{0} 2_{\pm6} 3_{\pm6}^{2} 8_{\pm3}, \\
& 1_{\pm3} 4_{\pm3} 5_{\pm2} 7_{0} 2_{\pm6}^{2} 3_{\pm6}^{2} 8_{\pm3},
 1_{\pm3}^{2} 4_{\pm3} 5_{\pm2} 7_{0}^{2} 2_{\pm6}^{2} 3_{\pm6}^{3} 8_{\pm3},
 1_{\pm3}^{2} 4_{\pm3} 5_{\pm2} 6_{\pm1} 7_{0}^{2} 2_{\pm6}^{3} 3_{\pm6}^{3} 8_{\pm3},
 1_{\pm3}^{2} 4_{\pm3} 5_{\pm2}^{2} 6_{\pm1} 7_{0}^{2} 2_{\pm6}^{3} 3_{\pm6}^{4} 8_{\pm3}, \\
& 1_{\pm3}^{2} 4_{\pm3} 5_{\pm2} 6_{\pm1} 7_{0} 2_{\pm6}^{2} 3_{\pm6}^{3} 8_{\pm3},
 1_{\pm3} 5_{\pm2} 7_{0} 2_{\pm6} 3_{\pm6} 8_{\pm3},
 1_{\pm3} 4_{\pm3} 7_{0}^{2} 2_{\pm6} 3_{\pm6}^{2} 8_{\pm3},
 1_{\pm3} 4_{\pm3} 6_{\pm1} 7_{0}^{2} 2_{\pm6}^{2} 3_{\pm6}^{2} 8_{\pm3}, \\
& 1_{\pm3} 4_{\pm3} 5_{\pm2} 6_{\pm1} 7_{0}^{2} 2_{\pm6}^{2} 3_{\pm6}^{3} 8_{\pm3},
 1_{\pm3}^{2} 4_{\pm3}^{2} 5_{\pm2} 6_{\pm1} 7_{0}^{2} 2_{\pm6}^{3} 3_{\pm6}^{4} 8_{\pm3},
 1_{\pm3} 4_{\pm3} 5_{\pm2} 7_{0}^{2} 2_{\pm6}^{2} 3_{\pm6}^{2} 8_{\pm3},
 1_{\pm3} 4_{\pm3} 6_{\pm1} 7_{0} 2_{\pm6} 3_{\pm6}^{2} 8_{\pm3}, \\
& 1_{\pm3}^{2} 4_{\pm3} 5_{\pm2} 6_{\pm1} 7_{0}^{2} 2_{\pm6}^{2} 3_{\pm6}^{3} 8_{\pm3},
 1_{\pm3} 4_{\pm3} 5_{\pm2} 6_{\pm1} 7_{0} 2_{\pm6}^{2} 3_{\pm6}^{2} 8_{\pm3},
 1_{\pm3} 4_{\pm3} 5_{\pm2} 7_{0} 2_{\pm6} 3_{\pm6}^{2} 8_{\pm3},
 1_{\pm3} 4_{\pm3} 7_{0} 2_{\pm6} 3_{\pm6} 8_{\pm3}. \\
%\end{align*}
%\end{gather}
%%
%%
& \quad \\
%%
%%
%Let $\xi=(-4,-5,-5,-4_{\pm3},-2,-1,0)$. 
%\begin{align*}
%\xymatrix{
%& & 2 &  &  \\
%1  \ar[r] & 3   & 4 \ar[l]  \ar[u] & 5  \ar[l]  &  6  \ar[l] &  7  \ar[l]  & 8   \ar[l] }
%\end{align*}
%The highest $l$-weight monomials of Hernandez-Leclerc modules of type $E_8$ that are not of type $A$, $D$ or $E_7$ are 
%%\begin{gather}
%%\begin{align*}
& (47)  1_{\pm4} 8_{0} 2_{\pm7} 3_{\pm7},
1_{\pm4} 7_{\pm1} 8_{0} 2_{\pm7} 3_{\pm7}^{2},
1_{\pm4} 6_{\pm2} 8_{0} 2_{\pm7} 3_{\pm7}^{2},
1_{\pm4} 6_{\pm2} 7_{\pm1} 8_{0} 2_{\pm7}^{2} 3_{\pm7}^{2}, \\
& 1_{\pm4} 5_{\pm3} 8_{0} 2_{\pm7} 3_{\pm7}^{2},
1_{\pm4} 5_{\pm3} 7_{\pm1} 8_{0} 2_{\pm7}^{2} 3_{\pm7}^{2},
1_{\pm4} 5_{\pm3} 6_{\pm2} 7_{\pm1} 8_{0} 2_{\pm7}^{2} 3_{\pm7}^{3},
1_{\pm4} 5_{\pm3} 6_{\pm2} 8_{0} 2_{\pm7}^{2} 3_{\pm7}^{2}, \\
& 1_{\pm4} 4_{\pm4} 8_{0} 2_{\pm7} 3_{\pm7}^{2},
1_{\pm4} 4_{\pm4} 7_{\pm1} 8_{0} 2_{\pm7}^{2} 3_{\pm7}^{2},
1_{\pm4} 4_{\pm4} 6_{\pm2} 7_{\pm1} 8_{0} 2_{\pm7}^{2} 3_{\pm7}^{3},
1_{\pm4}^{2} 4_{\pm4} 5_{\pm3} 6_{\pm2} 7_{\pm1} 8_{0} 2_{\pm7}^{3} 3_{\pm7}^{4}, \\
& 1_{\pm4} 4_{\pm4} 5_{\pm3} 7_{\pm1} 8_{0} 2_{\pm7}^{2} 3_{\pm7}^{3},
1_{\pm4} 4_{\pm4} 6_{\pm2} 8_{0} 2_{\pm7}^{2} 3_{\pm7}^{2},
1_{\pm4} 7_{\pm1} 8_{0} 2_{\pm7} 3_{\pm7},
1_{\pm4} 6_{\pm2} 7_{\pm1} 8_{0} 2_{\pm7} 3_{\pm7}^{2}, \\
& 1_{\pm4}^{2} 5_{\pm3} 6_{\pm2} 7_{\pm1} 8_{0} 2_{\pm7}^{2} 3_{\pm7}^{3},
1_{\pm4}^{2} 4_{\pm4} 5_{\pm3} 6_{\pm2} 7_{\pm1} 8_{0}^{2} 2_{\pm7}^{3} 3_{\pm7}^{4},
1_{\pm4}^{2} 4_{\pm4} 6_{\pm2} 7_{\pm1} 8_{0} 2_{\pm7}^{2} 3_{\pm7}^{3},
1_{\pm4} 5_{\pm3} 7_{\pm1} 8_{0} 2_{\pm7} 3_{\pm7}^{2}, \\
& 1_{\pm4} 6_{\pm2} 8_{0} 2_{\pm7} 3_{\pm7},
1_{\pm4} 4_{\pm4} 5_{\pm3} 6_{\pm2} 7_{\pm1} 8_{0} 2_{\pm7}^{3} 3_{\pm7}^{3},
1_{\pm4}^{2} 4_{\pm4} 5_{\pm3} 6_{\pm2} 7_{\pm1}^{2} 8_{0} 2_{\pm7}^{3} 3_{\pm7}^{4},
1_{\pm4} 5_{\pm3} 6_{\pm2} 7_{\pm1} 8_{0} 2_{\pm7}^{2} 3_{\pm7}^{2}, \\
& 1_{\pm4} 4_{\pm4} 5_{\pm3} 6_{\pm2} 8_{0} 2_{\pm7}^{2} 3_{\pm7}^{3},
1_{\pm4}^{2} 4_{\pm4} 5_{\pm3} 6_{\pm2} 7_{\pm1} 8_{0} 2_{\pm7}^{2} 3_{\pm7}^{4},
1_{\pm4}^{2} 4_{\pm4} 5_{\pm3} 6_{\pm2}^{2} 7_{\pm1} 8_{0} 2_{\pm7}^{3} 3_{\pm7}^{4},
1_{\pm4} 5_{\pm3} 6_{\pm2} 8_{0} 2_{\pm7} 3_{\pm7}^{2}, \\
& 1_{\pm4} 4_{\pm4} 5_{\pm3} 8_{0} 2_{\pm7}^{2} 3_{\pm7}^{2},
1_{\pm4}^{2} 4_{\pm4} 5_{\pm3} 7_{\pm1} 8_{0} 2_{\pm7}^{2} 3_{\pm7}^{3},
1_{\pm4}^{2} 4_{\pm4} 5_{\pm3} 6_{\pm2} 7_{\pm1} 8_{0} 2_{\pm7}^{3} 3_{\pm7}^{3},
1_{\pm4}^{2} 4_{\pm4} 5_{\pm3}^{2} 6_{\pm2} 7_{\pm1} 8_{0} 2_{\pm7}^{3} 3_{\pm7}^{4}, \\
& 1_{\pm4}^{2} 4_{\pm4} 5_{\pm3} 6_{\pm2} 8_{0} 2_{\pm7}^{2} 3_{\pm7}^{3},
1_{\pm4} 5_{\pm3} 8_{0} 2_{\pm7} 3_{\pm7},
1_{\pm4} 4_{\pm4} 7_{\pm1} 8_{0} 2_{\pm7} 3_{\pm7}^{2},
1_{\pm4} 4_{\pm4} 6_{\pm2} 7_{\pm1} 8_{0} 2_{\pm7}^{2} 3_{\pm7}^{2}, \\
& 1_{\pm4} 4_{\pm4} 5_{\pm3} 6_{\pm2} 7_{\pm1} 8_{0} 2_{\pm7}^{2} 3_{\pm7}^{3},
1_{\pm4}^{2} 4_{\pm4}^{2} 5_{\pm3} 6_{\pm2} 7_{\pm1} 8_{0} 2_{\pm7}^{3} 3_{\pm7}^{4},
1_{\pm4} 4_{\pm4} 5_{\pm3} 7_{\pm1} 8_{0} 2_{\pm7}^{2} 3_{\pm7}^{2},
1_{\pm4} 4_{\pm4} 6_{\pm2} 8_{0} 2_{\pm7} 3_{\pm7}^{2}, \\
& 1_{\pm4}^{2} 4_{\pm4} 5_{\pm3} 6_{\pm2} 7_{\pm1} 8_{0} 2_{\pm7}^{2} 3_{\pm7}^{3},
1_{\pm4} 4_{\pm4} 5_{\pm3} 6_{\pm2} 8_{0} 2_{\pm7}^{2} 3_{\pm7}^{2},
1_{\pm4} 4_{\pm4} 5_{\pm3} 8_{0} 2_{\pm7} 3_{\pm7}^{2},
1_{\pm4} 4_{\pm4} 8_{0} 2_{\pm7} 3_{\pm7}. \\
%\end{align*}
%\end{gather}
%%
%%
& \quad \\
%%
%%
%Let $\xi=(-1,0,-2,-1,-2,-3,-4,-5)$. 
%\begin{align*}
%\xymatrix{
%& & 2 \ar[d]  &  &  \\
%1  \ar[r] & 3   & 4 \ar[l]   \ar[r]  & 5  \ar[r]  &  6  \ar[r]   &  7 \ar[r]  & 8 }
%\end{align*}
%The highest $l$-weight monomials of Hernandez-Leclerc modules of type $E_8$ that are not of type $A$, $D$ or $E_7$ are 
%\begin{gather}
%\begin{align*}
& (48)  1_{\pm1}2_{0}4_{\pm1}3_{\pm4}^{2} 5_{\pm4}8_{\pm7},
1_{\pm1}2_{0}^{2} 4_{\pm1}3_{\pm4}^{2} 5_{\pm4}6_{\pm5}8_{\pm7},
1_{\pm1}2_{0}^{2} 4_{\pm1}^{2} 3_{\pm4}^{3}  5_{\pm4}6_{\pm5}7_{\pm6}8_{\pm7},
1_{\pm1}2_{0}^{2} 4_{\pm1}3_{\pm4}^{2} 5_{\pm4}7_{\pm6}8_{\pm7}, \\
& 1_{\pm1}2_{0}4_{\pm1}3_{\pm4}^{2} 6_{\pm5}8_{\pm7},
1_{\pm1}^{2} 2_{0}^{3}  4_{\pm1}^{2} 3_{\pm4}^{4}  5_{\pm4}6_{\pm5}7_{\pm6}8_{\pm7},
1_{\pm1}2_{0}^{2} 4_{\pm1}^{2} 3_{\pm4}^{3}  5_{\pm4}6_{\pm5}8_{\pm7},
1_{\pm1}^{2} 2_{0}^{2} 4_{\pm1}^{2} 3_{\pm4}^{3}  5_{\pm4}6_{\pm5}7_{\pm6}8_{\pm7}, \\
& 1_{\pm1}2_{0}^{2} 4_{\pm1}3_{\pm4}^{2} 6_{\pm5}7_{\pm6}8_{\pm7},
1_{\pm1}2_{0}4_{\pm1}3_{\pm4}^{2} 7_{\pm6}8_{\pm7},
1_{\pm1}2_{0}3_{\pm4}8_{\pm7},
1_{\pm1}^{2} 2_{0}^{3}  4_{\pm1}^{3}  3_{\pm4}^{4}  5_{\pm4}^{2} 6_{\pm5}7_{\pm6}8_{\pm7}, \\
& 1_{\pm1}2_{0}^{3}  4_{\pm1}^{2} 3_{\pm4}^{3}  5_{\pm4}6_{\pm5}7_{\pm6}8_{\pm7},
1_{\pm1}2_{0}^{2} 4_{\pm1}^{2} 3_{\pm4}^{3}  5_{\pm4}7_{\pm6}8_{\pm7},
1_{\pm1}2_{0}^{2} 4_{\pm1}3_{\pm4}^{2} 5_{\pm4}8_{\pm7},
1_{\pm1}2_{0}4_{\pm1}^{2} 3_{\pm4}^{2} 5_{\pm4}6_{\pm5}8_{\pm7}, \\
& 1_{\pm1}^{2} 2_{0}^{3}  4_{\pm1}^{3}  3_{\pm4}^{4}  5_{\pm4}6_{\pm5}^{2} 7_{\pm6}8_{\pm7},
1_{\pm1}2_{0}^{2} 4_{\pm1}^{2} 3_{\pm4}^{2} 5_{\pm4}6_{\pm5}7_{\pm6}8_{\pm7},
1_{\pm1}^{2} 2_{0}^{2} 4_{\pm1}^{3}  3_{\pm4}^{4}  5_{\pm4}6_{\pm5}7_{\pm6}8_{\pm7},
1_{\pm1}^{2} 2_{0}^{2} 4_{\pm1}^{2} 3_{\pm4}^{3}  5_{\pm4}6_{\pm5}8_{\pm7}, \\
& 1_{\pm1}^{2} 2_{0}^{3}  4_{\pm1}^{3}  3_{\pm4}^{4}  5_{\pm4}6_{\pm5}7_{\pm6}^{2} 8_{\pm7},
1_{\pm1}2_{0}^{2} 4_{\pm1}^{2} 3_{\pm4}^{3}  6_{\pm5}7_{\pm6}8_{\pm7},
1_{\pm1}^{2} 2_{0}^{3}  4_{\pm1}^{2} 3_{\pm4}^{3}  5_{\pm4}6_{\pm5}7_{\pm6}8_{\pm7},
1_{\pm1}2_{0}4_{\pm1}^{2} 3_{\pm4}^{2} 5_{\pm4}7_{\pm6}8_{\pm7}, \\
& 1_{\pm1}^{2} 2_{0}^{3}  4_{\pm1}^{3}  3_{\pm4}^{4}  5_{\pm4}6_{\pm5}7_{\pm6}8_{\pm7}^{2},
1_{\pm1}^{2} 2_{0}^{2} 4_{\pm1}^{2} 3_{\pm4}^{3}  5_{\pm4}7_{\pm6}8_{\pm7},
1_{\pm1}2_{0}^{2} 4_{\pm1}^{3}  3_{\pm4}^{3}  5_{\pm4}6_{\pm5}7_{\pm6}8_{\pm7},
1_{\pm1}2_{0}4_{\pm1}^{2} 3_{\pm4}^{2} 6_{\pm5}7_{\pm6}8_{\pm7}, \\
& 1_{\pm1}2_{0}4_{\pm1}3_{\pm4}5_{\pm4}8_{\pm7},
1_{\pm1}2_{0}^{2} 4_{\pm1}3_{\pm4}^{2} 6_{\pm5}8_{\pm7},
1_{\pm1}^{2} 2_{0}^{3}  4_{\pm1}^{3}  3_{\pm4}^{4}  5_{\pm4}6_{\pm5}7_{\pm6}8_{\pm7},
1_{\pm1}2_{0}^{2} 4_{\pm1}^{2} 3_{\pm4}^{2} 5_{\pm4}6_{\pm5}8_{\pm7}, \\
& 1_{\pm1}^{2} 2_{0}^{2} 4_{\pm1}^{2} 3_{\pm4}^{3}  6_{\pm5}7_{\pm6} 8_{\pm7},
1_{\pm1}2_{0}^{2} 4_{\pm1}3_{\pm4}^{2} 7_{\pm6}8_{\pm7},
1_{\pm1}2_{0}4_{\pm1}3_{\pm4}^{2} 8_{\pm7},
1_{\pm1}^{2} 2_{0}^{2} 4_{\pm1}^{3}  3_{\pm4}^{3}  5_{\pm4}6_{\pm5}7_{\pm6} 8_{\pm7}, \\
& 1_{\pm1}2_{0}^{2} 4_{\pm1}^{2} 3_{\pm4}^{2} 5_{\pm4}7_{\pm6}8_{\pm7},
1_{\pm1}2_{0}4_{\pm1}^{2} 3_{\pm4}^{2} 5_{\pm4}8_{\pm7},
1_{\pm1}2_{0}4_{\pm1}3_{\pm4}6_{\pm5}8_{\pm7},
1_{\pm1}2_{0}^{2} 4_{\pm1}^{2} 3_{\pm4}^{2} 6_{\pm5}7_{\pm6} 8_{\pm7}, \\
& 1_{\pm1}2_{0}4_{\pm1}3_{\pm4}7_{\pm6}8_{\pm7},
1_{\pm1}2_{0}4_{\pm1}^{2} 3_{\pm4}^{2} 6_{\pm5}8_{\pm7},
1_{\pm1}2_{0}4_{\pm1}^{2} 3_{\pm4}^{2} 7_{\pm6}8_{\pm7},
1_{\pm1}2_{0}4_{\pm1}3_{\pm4}8_{\pm7}.
\end{align*}
\end{gather}

%Let $\xi=(-1,0,-2,-1,-2,-3,-4_{\pm3})$. 
%\begin{align*}
%\xymatrix{
%& & 2 \ar[d]  &  &  \\
%1  \ar[r] & 3   & 4 \ar[l]   \ar[r]  & 5  \ar[r]  &  6  \ar[r]   &  7 & 8  \ar[l] }
%\end{align*}
%The highest $l$-weight monomials of Hernandez-Leclerc modules of type $E_8$ that are not of type $A$, $D$ or $E_7$ are 
\begin{gather}
\begin{align*}
& (49)  1_{\pm1} 2_{0} 4_{\pm1} 8_{\pm3} 3_{\pm4}^{2} 5_{\pm4} 7_{\pm6},
1_{\pm1} 2_{0}^{2} 4_{\pm1} 8_{\pm3} 3_{\pm4}^{2} 5_{\pm4} 6_{\pm5} 7_{\pm6},
1_{\pm1} 2_{0}^{2} 4_{\pm1}^{2} 8_{\pm3} 3_{\pm4}^{3} 5_{\pm4} 6_{\pm5} 7_{\pm6}^{2},
1_{\pm1} 2_{0}^{2} 4_{\pm1} 8_{\pm3} 3_{\pm4}^{2} 5_{\pm4} 7_{\pm6}^{2}, \\
& 1_{\pm1} 2_{0} 4_{\pm1} 8_{\pm3} 3_{\pm4}^{2} 6_{\pm5} 7_{\pm6},
1_{\pm1}^{2} 2_{0}^{3} 4_{\pm1}^{2} 8_{\pm3} 3_{\pm4}^{4} 5_{\pm4} 6_{\pm5} 7_{\pm6}^{2},
1_{\pm1} 2_{0}^{2} 4_{\pm1}^{2} 8_{\pm3} 3_{\pm4}^{3} 5_{\pm4} 6_{\pm5} 7_{\pm6},
1_{\pm1}^{2} 2_{0}^{2} 4_{\pm1}^{2} 8_{\pm3} 3_{\pm4}^{3} 5_{\pm4} 6_{\pm5} 7_{\pm6}^{2}, \\
& 1_{\pm1} 2_{0}^{2} 4_{\pm1} 8_{\pm3} 3_{\pm4}^{2} 6_{\pm5} 7_{\pm6}^{2},
1_{\pm1} 2_{0} 4_{\pm1} 8_{\pm3} 3_{\pm4}^{2} 7_{\pm6}^{2},
1_{\pm1}^{2} 2_{0}^{3} 4_{\pm1}^{3} 8_{\pm3} 3_{\pm4}^{4} 5_{\pm4}^{2} 6_{\pm5} 7_{\pm6}^{2},
1_{\pm1} 2_{0}^{3} 4_{\pm1}^{2} 8_{\pm3} 3_{\pm4}^{3} 5_{\pm4} 6_{\pm5} 7_{\pm6}^{2}, \\
& 1_{\pm1} 2_{0}^{2} 4_{\pm1}^{2} 8_{\pm3} 3_{\pm4}^{3} 5_{\pm4} 7_{\pm6}^{2},
1_{\pm1} 2_{0} 8_{\pm3} 3_{\pm4} 7_{\pm6},
1_{\pm1} 2_{0}^{2} 4_{\pm1} 8_{\pm3} 3_{\pm4}^{2} 5_{\pm4} 7_{\pm6},
1_{\pm1} 2_{0} 4_{\pm1}^{2} 8_{\pm3} 3_{\pm4}^{2} 5_{\pm4} 6_{\pm5} 7_{\pm6}, \\
& 1_{\pm1}^{2} 2_{0}^{3} 4_{\pm1}^{3} 8_{\pm3} 3_{\pm4}^{4} 5_{\pm4} 6_{\pm5}^{2} 7_{\pm6}^{2},
1_{\pm1} 2_{0}^{2} 4_{\pm1}^{2} 8_{\pm3} 3_{\pm4}^{2} 5_{\pm4} 6_{\pm5} 7_{\pm6}^{2},
1_{\pm1}^{2} 2_{0}^{2} 4_{\pm1}^{3} 8_{\pm3} 3_{\pm4}^{4} 5_{\pm4} 6_{\pm5} 7_{\pm6}^{2},
1_{\pm1}^{2} 2_{0}^{2} 4_{\pm1}^{2} 8_{\pm3} 3_{\pm4}^{3} 5_{\pm4} 6_{\pm5} 7_{\pm6}, \\
& 1_{\pm1}^{2} 2_{0}^{3} 4_{\pm1}^{3} 8_{\pm3} 3_{\pm4}^{4} 5_{\pm4} 6_{\pm5} 7_{\pm6}^{3},
1_{\pm1} 2_{0}^{2} 4_{\pm1}^{2} 8_{\pm3} 3_{\pm4}^{3} 6_{\pm5} 7_{\pm6}^{2},
1_{\pm1}^{2} 2_{0}^{3} 4_{\pm1}^{2} 8_{\pm3} 3_{\pm4}^{3} 5_{\pm4} 6_{\pm5} 7_{\pm6}^{2},
1_{\pm1} 2_{0} 4_{\pm1}^{2} 8_{\pm3} 3_{\pm4}^{2} 5_{\pm4} 7_{\pm6}^{2}, \\
& 1_{\pm1}^{2} 2_{0}^{3} 4_{\pm1}^{3} 8_{\pm3}^{2} 3_{\pm4}^{4} 5_{\pm4} 6_{\pm5} 7_{\pm6}^{3},
1_{\pm1}^{2} 2_{0}^{2} 4_{\pm1}^{2} 8_{\pm3} 3_{\pm4}^{3} 5_{\pm4} 7_{\pm6}^{2},
1_{\pm1} 2_{0}^{2} 4_{\pm1}^{3} 8_{\pm3} 3_{\pm4}^{3} 5_{\pm4} 6_{\pm5} 7_{\pm6}^{2},
1_{\pm1} 2_{0} 4_{\pm1}^{2} 8_{\pm3} 3_{\pm4}^{2} 6_{\pm5} 7_{\pm6}^{2}, \\
& 1_{\pm1} 2_{0} 4_{\pm1} 8_{\pm3} 3_{\pm4} 5_{\pm4} 7_{\pm6},
1_{\pm1} 2_{0}^{2} 4_{\pm1} 8_{\pm3} 3_{\pm4}^{2} 6_{\pm5} 7_{\pm6},
1_{\pm1}^{2} 2_{0}^{3} 4_{\pm1}^{3} 8_{\pm3} 3_{\pm4}^{4} 5_{\pm4} 6_{\pm5} 7_{\pm6}^{2},
1_{\pm1} 2_{0}^{2} 4_{\pm1}^{2} 8_{\pm3} 3_{\pm4}^{2} 5_{\pm4} 6_{\pm5} 7_{\pm6}, \\
& 1_{\pm1}^{2} 2_{0}^{2} 4_{\pm1}^{2} 8_{\pm3} 3_{\pm4}^{3} 6_{\pm5} 7_{\pm6}^{2},
1_{\pm1} 2_{0}^{2} 4_{\pm1} 8_{\pm3} 3_{\pm4}^{2} 7_{\pm6}^{2},
1_{\pm1} 2_{0} 4_{\pm1} 8_{\pm3} 3_{\pm4}^{2} 7_{\pm6},
1_{\pm1}^{2} 2_{0}^{2} 4_{\pm1}^{3} 8_{\pm3} 3_{\pm4}^{3} 5_{\pm4} 6_{\pm5} 7_{\pm6}^{2}, \\
& 1_{\pm1} 2_{0}^{2} 4_{\pm1}^{2} 8_{\pm3} 3_{\pm4}^{2} 5_{\pm4} 7_{\pm6}^{2},
1_{\pm1} 2_{0} 4_{\pm1}^{2} 8_{\pm3} 3_{\pm4}^{2} 5_{\pm4} 7_{\pm6},
1_{\pm1} 2_{0} 4_{\pm1} 8_{\pm3} 3_{\pm4} 6_{\pm5} 7_{\pm6},
1_{\pm1} 2_{0}^{2} 4_{\pm1}^{2} 8_{\pm3} 3_{\pm4}^{2} 6_{\pm5} 7_{\pm6}^{2}, \\
& 1_{\pm1} 2_{0} 4_{\pm1} 8_{\pm3} 3_{\pm4} 7_{\pm6}^{2},
1_{\pm1} 2_{0} 4_{\pm1}^{2} 8_{\pm3} 3_{\pm4}^{2} 6_{\pm5} 7_{\pm6},
1_{\pm1} 2_{0} 4_{\pm1}^{2} 8_{\pm3} 3_{\pm4}^{2} 7_{\pm6}^{2},
1_{\pm1} 2_{0} 4_{\pm1} 8_{\pm3} 3_{\pm4} 7_{\pm6}. \\
%%\end{align*}
%%\end{gather}
%%
%%
& \quad \\
%%
%%
%Let $\xi=(-1,0,-2,-1,-2,-3,-2,-3)$. 
%\begin{align*}
%\xymatrix{
%& & 2 \ar[d]  &  &  \\
%1  \ar[r] & 3   & 4 \ar[l]   \ar[r]  & 5  \ar[r]  &  6   &  7 \ar[l]  \ar[r]   & 8 }
%\end{align*}
%The highest $l$-weight monomials of Hernandez-Leclerc modules of type $E_8$ that are not of type $A$, $D$ or $E_7$ are 
%%\begin{gather} 
%%\begin{align*}
& (50)  1_{\pm1} 2_{0} 4_{\pm1} 7_{\pm2} 3_{\pm4}^{2} 5_{\pm4} 6_{\pm5} 8_{\pm5},
1_{\pm1} 2_{0}^{2} 4_{\pm1} 7_{\pm2} 3_{\pm4}^{2} 5_{\pm4} 6_{\pm5}^{2} 8_{\pm5},
1_{\pm1} 2_{0} 4_{\pm1} 7_{\pm2} 3_{\pm4}^{2} 6_{\pm5}^{2} 8_{\pm5},
1_{\pm1} 2_{0} 7_{\pm2} 3_{\pm4} 6_{\pm5} 8_{\pm5}, \\
& 1_{\pm1} 2_{0}^{2} 4_{\pm1}^{2} 7_{\pm2}^{2} 3_{\pm4}^{3} 5_{\pm4} 6_{\pm5}^{3} 8_{\pm5},
1_{\pm1} 2_{0}^{2} 4_{\pm1} 7_{\pm2}^{2} 3_{\pm4}^{2} 5_{\pm4} 6_{\pm5}^{2} 8_{\pm5},
1_{\pm1} 2_{0}^{2} 4_{\pm1}^{2} 7_{\pm2} 3_{\pm4}^{3} 5_{\pm4} 6_{\pm5}^{2} 8_{\pm5},
1_{\pm1} 2_{0} 4_{\pm1}^{2} 7_{\pm2} 3_{\pm4}^{2} 5_{\pm4} 6_{\pm5}^{2} 8_{\pm5}, \\
& 1_{\pm1}^{2} 2_{0}^{3} 4_{\pm1}^{2} 7_{\pm2}^{2} 3_{\pm4}^{4}  5_{\pm4} 6_{\pm5}^{3} 8_{\pm5},
1_{\pm1}^{2} 2_{0}^{2} 4_{\pm1}^{2} 7_{\pm2}^{2} 3_{\pm4}^{3} 5_{\pm4} 6_{\pm5}^{3} 8_{\pm5},
1_{\pm1} 2_{0}^{2} 4_{\pm1} 7_{\pm2}^{2} 3_{\pm4}^{2} 6_{\pm5}^{3} 8_{\pm5},
1_{\pm1} 2_{0}^{2} 4_{\pm1} 7_{\pm2} 3_{\pm4}^{2} 5_{\pm4} 6_{\pm5} 8_{\pm5}, \\
& 1_{\pm1}^{2} 2_{0}^{3} 4_{\pm1}^{3} 7_{\pm2}^{2} 3_{\pm4}^{4}  5_{\pm4}^{2} 6_{\pm5}^{3} 8_{\pm5},
1_{\pm1}^{2} 2_{0}^{2} 4_{\pm1}^{2} 7_{\pm2} 3_{\pm4}^{3} 5_{\pm4} 6_{\pm5}^{2} 8_{\pm5},
1_{\pm1} 2_{0}^{3} 4_{\pm1}^{2} 7_{\pm2}^{2} 3_{\pm4}^{3} 5_{\pm4} 6_{\pm5}^{3} 8_{\pm5},
1_{\pm1} 2_{0} 4_{\pm1} 7_{\pm2}^{2} 3_{\pm4}^{2} 6_{\pm5}^{2} 8_{\pm5}, \\
& 1_{\pm1} 2_{0}^{2} 4_{\pm1}^{2} 7_{\pm2}^{2} 3_{\pm4}^{3} 5_{\pm4} 6_{\pm5}^{2} 8_{\pm5},
1_{\pm1} 2_{0} 4_{\pm1} 7_{\pm2} 3_{\pm4} 5_{\pm4} 6_{\pm5} 8_{\pm5},
1_{\pm1}^{2} 2_{0}^{3} 4_{\pm1}^{3} 7_{\pm2}^{2} 3_{\pm4}^{4}  5_{\pm4} 6_{\pm5}^{4}  8_{\pm5},
1_{\pm1} 2_{0}^{2} 4_{\pm1}^{2} 7_{\pm2}^{2} 3_{\pm4}^{2} 5_{\pm4} 6_{\pm5}^{3} 8_{\pm5}, \\
& 1_{\pm1}^{2} 2_{0}^{2} 4_{\pm1}^{3} 7_{\pm2}^{2} 3_{\pm4}^{4}  5_{\pm4} 6_{\pm5}^{3} 8_{\pm5},
1_{\pm1} 2_{0}^{2} 4_{\pm1} 7_{\pm2} 3_{\pm4}^{2} 6_{\pm5}^{2} 8_{\pm5},
1_{\pm1}^{2} 2_{0}^{3} 4_{\pm1}^{3} 7_{\pm2}^{3} 3_{\pm4}^{4}  5_{\pm4} 6_{\pm5}^{4}  8_{\pm5}^{2},
1_{\pm1} 2_{0}^{2} 4_{\pm1}^{2} 7_{\pm2}^{2} 3_{\pm4}^{3} 6_{\pm5}^{3} 8_{\pm5}, \\
& 1_{\pm1}^{2} 2_{0}^{3} 4_{\pm1}^{2} 7_{\pm2}^{2} 3_{\pm4}^{3} 5_{\pm4} 6_{\pm5}^{3} 8_{\pm5},
1_{\pm1} 2_{0}^{2} 4_{\pm1}^{2} 7_{\pm2} 3_{\pm4}^{2} 5_{\pm4} 6_{\pm5}^{2} 8_{\pm5},
1_{\pm1} 2_{0} 4_{\pm1} 7_{\pm2} 3_{\pm4}^{2} 6_{\pm5} 8_{\pm5},
1_{\pm1} 2_{0} 4_{\pm1}^{2} 7_{\pm2}^{2} 3_{\pm4}^{2} 5_{\pm4} 6_{\pm5}^{2} 8_{\pm5}, \\
& 1_{\pm1}^{2} 2_{0}^{3} 4_{\pm1}^{3} 7_{\pm2}^{3} 3_{\pm4}^{4}  5_{\pm4} 6_{\pm5}^{4}  8_{\pm5},
1_{\pm1}^{2} 2_{0}^{2} 4_{\pm1}^{2} 7_{\pm2}^{2} 3_{\pm4}^{3} 5_{\pm4} 6_{\pm5}^{2} 8_{\pm5},
1_{\pm1} 2_{0}^{2} 4_{\pm1}^{3} 7_{\pm2}^{2} 3_{\pm4}^{3} 5_{\pm4} 6_{\pm5}^{3} 8_{\pm5},
1_{\pm1} 2_{0} 4_{\pm1} 7_{\pm2} 3_{\pm4} 6_{\pm5}^{2} 8_{\pm5}, \\
& 1_{\pm1} 2_{0} 4_{\pm1}^{2} 7_{\pm2}^{2} 3_{\pm4}^{2} 6_{\pm5}^{3} 8_{\pm5},
1_{\pm1}^{2} 2_{0}^{3} 4_{\pm1}^{3} 7_{\pm2}^{2} 3_{\pm4}^{4}  5_{\pm4} 6_{\pm5}^{3} 8_{\pm5},
1_{\pm1}^{2} 2_{0}^{2} 4_{\pm1}^{2} 7_{\pm2}^{2} 3_{\pm4}^{3} 6_{\pm5}^{3} 8_{\pm5},
1_{\pm1} 2_{0}^{2} 4_{\pm1} 7_{\pm2}^{2} 3_{\pm4}^{2} 6_{\pm5}^{2} 8_{\pm5}, \\
& 1_{\pm1} 2_{0} 4_{\pm1}^{2} 7_{\pm2} 3_{\pm4}^{2} 5_{\pm4} 6_{\pm5} 8_{\pm5},
1_{\pm1}^{2} 2_{0}^{2} 4_{\pm1}^{3} 7_{\pm2}^{2} 3_{\pm4}^{3} 5_{\pm4} 6_{\pm5}^{3} 8_{\pm5},
1_{\pm1} 2_{0} 4_{\pm1}^{2} 7_{\pm2} 3_{\pm4}^{2} 6_{\pm5}^{2} 8_{\pm5},
1_{\pm1} 2_{0}^{2} 4_{\pm1}^{2} 7_{\pm2}^{2} 3_{\pm4}^{2} 5_{\pm4} 6_{\pm5}^{2} 8_{\pm5}, \\
& 1_{\pm1} 2_{0}^{2} 4_{\pm1}^{2} 7_{\pm2}^{2} 3_{\pm4}^{2} 6_{\pm5}^{3} 8_{\pm5},
1_{\pm1} 2_{0} 4_{\pm1} 7_{\pm2}^{2} 3_{\pm4} 6_{\pm5}^{2} 8_{\pm5},
1_{\pm1} 2_{0} 4_{\pm1}^{2} 7_{\pm2}^{2} 3_{\pm4}^{2} 6_{\pm5}^{2} 8_{\pm5},
1_{\pm1} 2_{0} 4_{\pm1} 7_{\pm2} 3_{\pm4} 6_{\pm5} 8_{\pm5}. \\
%%\end{align*}
%%\end{gather}
%%
%%
& \quad \\
%%
%%
%Let $\xi=(-1,0,-2,-1,-2,-3,-2,-1)$. 
%\begin{align*}
%\xymatrix{
%& & 2 \ar[d]  &  &  \\
%1  \ar[r] & 3   & 4 \ar[l]   \ar[r]  & 5  \ar[r]  &  6   &  7 \ar[l]  & 8 \ar[l] }
%\end{align*}
%The highest $l$-weight monomials of Hernandez-Leclerc modules of type $E_8$ that are not of type $A$, $D$ or $E_7$ are 
%%\begin{gather}
%%\begin{align*}
& (51)  1_{\pm1} 2_{0} 4_{\pm1} 8_{\pm1} 3_{\pm4}^{2} 5_{\pm4} 6_{\pm5},
1_{\pm1} 2_{0}^{2} 4_{\pm1} 8_{\pm1} 3_{\pm4}^{2} 5_{\pm4} 6_{\pm5}^{2},
1_{\pm1} 2_{0} 4_{\pm1} 8_{\pm1} 3_{\pm4}^{2} 6_{\pm5}^{2},
1_{\pm1} 2_{0} 8_{\pm1} 3_{\pm4} 6_{\pm5}, \\
& 1_{\pm1} 2_{0}^{2} 4_{\pm1}^{2} 7_{\pm2} 8_{\pm1} 3_{\pm4}^{3} 5_{\pm4} 6_{\pm5}^{3},
1_{\pm1} 2_{0}^{2} 4_{\pm1} 7_{\pm2} 8_{\pm1} 3_{\pm4}^{2} 5_{\pm4} 6_{\pm5}^{2},
1_{\pm1} 2_{0}^{2} 4_{\pm1}^{2} 8_{\pm1} 3_{\pm4}^{3} 5_{\pm4} 6_{\pm5}^{2},
1_{\pm1} 2_{0} 4_{\pm1}^{2} 8_{\pm1} 3_{\pm4}^{2} 5_{\pm4} 6_{\pm5}^{2}, \\
& 1_{\pm1}^{2} 2_{0}^{3} 4_{\pm1}^{2} 7_{\pm2} 8_{\pm1} 3_{\pm4}^{4} 5_{\pm4} 6_{\pm5}^{3},
1_{\pm1}^{2} 2_{0}^{2} 4_{\pm1}^{2} 7_{\pm2} 8_{\pm1} 3_{\pm4}^{3} 5_{\pm4} 6_{\pm5}^{3},
1_{\pm1} 2_{0}^{2} 4_{\pm1} 7_{\pm2} 8_{\pm1} 3_{\pm4}^{2} 6_{\pm5}^{3},
1_{\pm1} 2_{0}^{2} 4_{\pm1} 8_{\pm1} 3_{\pm4}^{2} 5_{\pm4} 6_{\pm5}, \\
& 1_{\pm1}^{2} 2_{0}^{3} 4_{\pm1}^{3} 7_{\pm2} 8_{\pm1} 3_{\pm4}^{4} 5_{\pm4}^{2} 6_{\pm5}^{3},
1_{\pm1}^{2} 2_{0}^{2} 4_{\pm1}^{2} 8_{\pm1} 3_{\pm4}^{3} 5_{\pm4} 6_{\pm5}^{2},
1_{\pm1} 2_{0}^{3} 4_{\pm1}^{2} 7_{\pm2} 8_{\pm1} 3_{\pm4}^{3} 5_{\pm4} 6_{\pm5}^{3},
1_{\pm1} 2_{0} 4_{\pm1} 7_{\pm2} 8_{\pm1} 3_{\pm4}^{2} 6_{\pm5}^{2}, \\
& 1_{\pm1} 2_{0}^{2} 4_{\pm1}^{2} 7_{\pm2} 8_{\pm1} 3_{\pm4}^{3} 5_{\pm4} 6_{\pm5}^{2},
1_{\pm1} 2_{0} 4_{\pm1} 8_{\pm1} 3_{\pm4} 5_{\pm4} 6_{\pm5},
1_{\pm1}^{2} 2_{0}^{3} 4_{\pm1}^{3} 7_{\pm2} 8_{\pm1} 3_{\pm4}^{4} 5_{\pm4} 6_{\pm5}^{4},
1_{\pm1} 2_{0}^{2} 4_{\pm1}^{2} 7_{\pm2} 8_{\pm1} 3_{\pm4}^{2} 5_{\pm4} 6_{\pm5}^{3}, \\
& 1_{\pm1}^{2} 2_{0}^{2} 4_{\pm1}^{3} 7_{\pm2} 8_{\pm1} 3_{\pm4}^{4} 5_{\pm4} 6_{\pm5}^{3},
1_{\pm1} 2_{0}^{2} 4_{\pm1} 8_{\pm1} 3_{\pm4}^{2} 6_{\pm5}^{2},
1_{\pm1}^{2} 2_{0}^{3} 4_{\pm1}^{3} 7_{\pm2} 8_{\pm1}^{2} 3_{\pm4}^{4} 5_{\pm4} 6_{\pm5}^{4},
1_{\pm1} 2_{0}^{2} 4_{\pm1}^{2} 7_{\pm2} 8_{\pm1} 3_{\pm4}^{3} 6_{\pm5}^{3}, \\
& 1_{\pm1}^{2} 2_{0}^{3} 4_{\pm1}^{2} 7_{\pm2} 8_{\pm1} 3_{\pm4}^{3} 5_{\pm4} 6_{\pm5}^{3},
1_{\pm1} 2_{0}^{2} 4_{\pm1}^{2} 8_{\pm1} 3_{\pm4}^{2} 5_{\pm4} 6_{\pm5}^{2},
1_{\pm1} 2_{0} 4_{\pm1} 8_{\pm1} 3_{\pm4}^{2} 6_{\pm5},
1_{\pm1} 2_{0} 4_{\pm1}^{2} 7_{\pm2} 8_{\pm1} 3_{\pm4}^{2} 5_{\pm4} 6_{\pm5}^{2}, \\
& 1_{\pm1}^{2} 2_{0}^{3} 4_{\pm1}^{3} 7_{\pm2}^{2} 8_{\pm1} 3_{\pm4}^{4} 5_{\pm4} 6_{\pm5}^{4},
1_{\pm1}^{2} 2_{0}^{2} 4_{\pm1}^{2} 7_{\pm2} 8_{\pm1} 3_{\pm4}^{3} 5_{\pm4} 6_{\pm5}^{2},
1_{\pm1} 2_{0}^{2} 4_{\pm1}^{3} 7_{\pm2} 8_{\pm1} 3_{\pm4}^{3} 5_{\pm4} 6_{\pm5}^{3},
1_{\pm1} 2_{0} 4_{\pm1} 8_{\pm1} 3_{\pm4} 6_{\pm5}^{2}, \\
& 1_{\pm1} 2_{0} 4_{\pm1}^{2} 7_{\pm2} 8_{\pm1} 3_{\pm4}^{2} 6_{\pm5}^{3},
1_{\pm1}^{2} 2_{0}^{3} 4_{\pm1}^{3} 7_{\pm2} 8_{\pm1} 3_{\pm4}^{4} 5_{\pm4} 6_{\pm5}^{3},
1_{\pm1}^{2} 2_{0}^{2} 4_{\pm1}^{2} 7_{\pm2} 8_{\pm1} 3_{\pm4}^{3} 6_{\pm5}^{3},
1_{\pm1} 2_{0}^{2} 4_{\pm1} 7_{\pm2} 8_{\pm1} 3_{\pm4}^{2} 6_{\pm5}^{2}, \\
& 1_{\pm1} 2_{0} 4_{\pm1}^{2} 8_{\pm1} 3_{\pm4}^{2} 5_{\pm4} 6_{\pm5},
1_{\pm1}^{2} 2_{0}^{2} 4_{\pm1}^{3} 7_{\pm2} 8_{\pm1} 3_{\pm4}^{3} 5_{\pm4} 6_{\pm5}^{3},
1_{\pm1} 2_{0} 4_{\pm1}^{2} 8_{\pm1} 3_{\pm4}^{2} 6_{\pm5}^{2},
1_{\pm1} 2_{0}^{2} 4_{\pm1}^{2} 7_{\pm2} 8_{\pm1} 3_{\pm4}^{2} 5_{\pm4} 6_{\pm5}^{2}, \\
& 1_{\pm1} 2_{0}^{2} 4_{\pm1}^{2} 7_{\pm2} 8_{\pm1} 3_{\pm4}^{2} 6_{\pm5}^{3},
1_{\pm1} 2_{0} 4_{\pm1} 7_{\pm2} 8_{\pm1} 3_{\pm4} 6_{\pm5}^{2},
1_{\pm1} 2_{0} 4_{\pm1}^{2} 7_{\pm2} 8_{\pm1} 3_{\pm4}^{2} 6_{\pm5}^{2},
1_{\pm1} 2_{0} 4_{\pm1} 8_{\pm1} 3_{\pm4} 6_{\pm5}.
\end{align*}
\end{gather}

%Let $\xi=(-1,0,-2,-1,-2,-1,-2,-3)$. 
%\begin{align*}
%\xymatrix{
%& & 2 \ar[d]  &  &  \\
%1  \ar[r] & 3   & 4 \ar[l]   \ar[r]  & 5   &  6 \ar[l]  \ar[r]  &  7 \ar[r]  & 8 }
%\end{align*}
%The highest $l$-weight monomials of Hernandez-Leclerc modules of type $E_8$ that are not of type $A$, $D$ or $E_7$ are 
\begin{gather}
\begin{align*}
& (52)  1_{\pm1}  2_{0}  4_{\pm1}  6_{\pm1}  3_{\pm4}^{2} 5_{\pm4}^{2} 8_{\pm5},
1_{\pm1}  2_{0}^{2} 4_{\pm1}  6_{\pm1}^{2} 3_{\pm4}^{2} 5_{\pm4}^{3} 7_{\pm4}  8_{\pm5},
1_{\pm1}  2_{0}  4_{\pm1}  6_{\pm1}^{2} 3_{\pm4}^{2} 5_{\pm4}^{2} 7_{\pm4}  8_{\pm5},
1_{\pm1}  2_{0}  6_{\pm1}  3_{\pm4}  5_{\pm4}  8_{\pm5}, \\
& 1_{\pm1}  2_{0}^{2} 4_{\pm1}^{2} 6_{\pm1}^{3} 3_{\pm4}^{3} 5_{\pm4}^{4} 7_{\pm4}  8_{\pm5},
1_{\pm1}  2_{0}^{2} 4_{\pm1}  6_{\pm1}^{2} 3_{\pm4}^{2} 5_{\pm4}^{3} 8_{\pm5},
1_{\pm1}  2_{0}^{2} 4_{\pm1}^{2} 6_{\pm1}^{2} 3_{\pm4}^{3} 5_{\pm4}^{3} 7_{\pm4}  8_{\pm5},
1_{\pm1}^{2} 2_{0}^{3} 4_{\pm1}^{2} 6_{\pm1}^{3} 3_{\pm4}^{4} 5_{\pm4}^{4} 7_{\pm4}  8_{\pm5}, \\
& 1_{\pm1}  2_{0}  4_{\pm1}^{2} 6_{\pm1}^{2} 3_{\pm4}^{2} 5_{\pm4}^{3} 7_{\pm4}  8_{\pm5},
1_{\pm1}^{2} 2_{0}^{2} 4_{\pm1}^{2} 6_{\pm1}^{3} 3_{\pm4}^{3} 5_{\pm4}^{4} 7_{\pm4}  8_{\pm5},
1_{\pm1}  2_{0}^{2} 4_{\pm1}  6_{\pm1}  3_{\pm4}^{2} 5_{\pm4}^{2} 8_{\pm5},
1_{\pm1}^{2} 2_{0}^{3} 4_{\pm1}^{3} 6_{\pm1}^{3} 3_{\pm4}^{4} 5_{\pm4}^{5} 7_{\pm4}  8_{\pm5}, \\
& 1_{\pm1}^{2} 2_{0}^{2} 4_{\pm1}^{2} 6_{\pm1}^{2} 3_{\pm4}^{3} 5_{\pm4}^{3} 7_{\pm4}  8_{\pm5},
1_{\pm1}  2_{0}^{2} 4_{\pm1}  6_{\pm1}^{3} 3_{\pm4}^{2} 5_{\pm4}^{3} 7_{\pm4}  8_{\pm5},
1_{\pm1}  2_{0}^{3} 4_{\pm1}^{2} 6_{\pm1}^{3} 3_{\pm4}^{3} 5_{\pm4}^{4} 7_{\pm4}  8_{\pm5},
1_{\pm1}  2_{0}  4_{\pm1}  6_{\pm1}^{2} 3_{\pm4}^{2} 5_{\pm4}^{2} 8_{\pm5}, \\
& 1_{\pm1}  2_{0}  4_{\pm1}  6_{\pm1}  3_{\pm4}  5_{\pm4}^{2} 8_{\pm5},
1_{\pm1}^{2} 2_{0}^{3} 4_{\pm1}^{3} 6_{\pm1}^{4} 3_{\pm4}^{4} 5_{\pm4}^{5} 7_{\pm4}^{2} 8_{\pm5},
1_{\pm1}  2_{0}^{2} 4_{\pm1}^{2} 6_{\pm1}^{3} 3_{\pm4}^{2} 5_{\pm4}^{4} 7_{\pm4}  8_{\pm5},
1_{\pm1}  2_{0}^{2} 4_{\pm1}^{2} 6_{\pm1}^{2} 3_{\pm4}^{3} 5_{\pm4}^{3} 8_{\pm5}, \\
& 1_{\pm1}^{2} 2_{0}^{2} 4_{\pm1}^{3} 6_{\pm1}^{3} 3_{\pm4}^{4} 5_{\pm4}^{4} 7_{\pm4}  8_{\pm5},
1_{\pm1}  2_{0}^{2} 4_{\pm1}  6_{\pm1}^{2} 3_{\pm4}^{2} 5_{\pm4}^{2} 7_{\pm4}  8_{\pm5},
1_{\pm1}^{2} 2_{0}^{3} 4_{\pm1}^{3} 6_{\pm1}^{4} 3_{\pm4}^{4} 5_{\pm4}^{5} 7_{\pm4}  8_{\pm5}^{2},
1_{\pm1}  2_{0}^{2} 4_{\pm1}^{2} 6_{\pm1}^{3} 3_{\pm4}^{3} 5_{\pm4}^{3} 7_{\pm4}  8_{\pm5}, \\
& 1_{\pm1}^{2} 2_{0}^{3} 4_{\pm1}^{2} 6_{\pm1}^{3} 3_{\pm4}^{3} 5_{\pm4}^{4} 7_{\pm4}  8_{\pm5},
1_{\pm1}  2_{0}  4_{\pm1}  6_{\pm1}  3_{\pm4}^{2} 5_{\pm4}  8_{\pm5},
1_{\pm1}  2_{0}  4_{\pm1}^{2} 6_{\pm1}^{2} 3_{\pm4}^{2} 5_{\pm4}^{3} 8_{\pm5},
1_{\pm1}^{2} 2_{0}^{3} 4_{\pm1}^{3} 6_{\pm1}^{4} 3_{\pm4}^{4} 5_{\pm4}^{5} 7_{\pm4}  8_{\pm5}, \\
& 1_{\pm1}^{2} 2_{0}^{2} 4_{\pm1}^{2} 6_{\pm1}^{2} 3_{\pm4}^{3} 5_{\pm4}^{3} 8_{\pm5},
1_{\pm1}  2_{0}^{2} 4_{\pm1}^{2} 6_{\pm1}^{2} 3_{\pm4}^{2} 5_{\pm4}^{3} 7_{\pm4}  8_{\pm5},
1_{\pm1}  2_{0}^{2} 4_{\pm1}^{3} 6_{\pm1}^{3} 3_{\pm4}^{3} 5_{\pm4}^{4} 7_{\pm4}  8_{\pm5},
1_{\pm1}  2_{0}  4_{\pm1}  6_{\pm1}^{2} 3_{\pm4}  5_{\pm4}^{2} 7_{\pm4}  8_{\pm5}, \\
& 1_{\pm1}^{2} 2_{0}^{3} 4_{\pm1}^{3} 6_{\pm1}^{3} 3_{\pm4}^{4} 5_{\pm4}^{4} 7_{\pm4}  8_{\pm5},
1_{\pm1}  2_{0}  4_{\pm1}^{2} 6_{\pm1}^{3} 3_{\pm4}^{2} 5_{\pm4}^{3} 7_{\pm4}  8_{\pm5},
1_{\pm1}^{2} 2_{0}^{2} 4_{\pm1}^{2} 6_{\pm1}^{3} 3_{\pm4}^{3} 5_{\pm4}^{3} 7_{\pm4}  8_{\pm5},
1_{\pm1}  2_{0}  4_{\pm1}^{2} 6_{\pm1}  3_{\pm4}^{2} 5_{\pm4}^{2} 8_{\pm5}, \\
& 1_{\pm1}^{2} 2_{0}^{2} 4_{\pm1}^{3} 6_{\pm1}^{3} 3_{\pm4}^{3} 5_{\pm4}^{4} 7_{\pm4}  8_{\pm5},
1_{\pm1}  2_{0}  4_{\pm1}^{2} 6_{\pm1}^{2} 3_{\pm4}^{2} 5_{\pm4}^{2} 7_{\pm4}  8_{\pm5},
1_{\pm1}  2_{0}^{2} 4_{\pm1}  6_{\pm1}^{2} 3_{\pm4}^{2} 5_{\pm4}^{2} 8_{\pm5},
1_{\pm1}  2_{0}^{2} 4_{\pm1}^{2} 6_{\pm1}^{2} 3_{\pm4}^{2} 5_{\pm4}^{3} 8_{\pm5}, \\
& 1_{\pm1}  2_{0}^{2} 4_{\pm1}^{2} 6_{\pm1}^{3} 3_{\pm4}^{2} 5_{\pm4}^{3} 7_{\pm4}  8_{\pm5},
1_{\pm1}  2_{0}  4_{\pm1}  6_{\pm1}^{2} 3_{\pm4}  5_{\pm4}^{2} 8_{\pm5},
1_{\pm1}  2_{0}  4_{\pm1}^{2} 6_{\pm1}^{2} 3_{\pm4}^{2} 5_{\pm4}^{2} 8_{\pm5},
1_{\pm1}  2_{0}  4_{\pm1}  6_{\pm1}  3_{\pm4}  5_{\pm4}  8_{\pm5}. \\
%%\end{align*}
%%\end{gather}
%%
%%
& \quad \\
%%
%%
%Let $\xi=(-1,0,-2,-1,-2,-1,-2,-1)$. 
%\begin{align*}
%\xymatrix{
%& & 2 \ar[d]  &  &  \\
%1  \ar[r] & 3   & 4 \ar[l]   \ar[r]  & 5   &  6 \ar[l]  \ar[r]  &  7 & 8 \ar[l]  }
%\end{align*}
%The highest $l$-weight monomials of Hernandez-Leclerc modules of type $E_8$ that are not of type $A$, $D$ or $E_7$ are 
%%\begin{gather}
%%\begin{align*}
& (53)  1_{\pm1} 2_{0} 4_{\pm1} 6_{\pm1} 8_{\pm1} 3_{\pm4}^{2} 5_{\pm4}^{2} 7_{\pm4},
1_{\pm1} 2_{0}^{2} 4_{\pm1} 6_{\pm1}^{2} 8_{\pm1} 3_{\pm4}^{2} 5_{\pm4}^{3}  7_{\pm4}^{2},
1_{\pm1} 2_{0} 4_{\pm1} 6_{\pm1}^{2} 8_{\pm1} 3_{\pm4}^{2} 5_{\pm4}^{2} 7_{\pm4}^{2},
1_{\pm1} 2_{0} 6_{\pm1} 8_{\pm1} 3_{\pm4} 5_{\pm4} 7_{\pm4}, \\
& 1_{\pm1} 2_{0}^{2} 4_{\pm1}^{2} 6_{\pm1}^{3}  8_{\pm1} 3_{\pm4}^{3}  5_{\pm4}^{4}  7_{\pm4}^{2},
1_{\pm1} 2_{0}^{2} 4_{\pm1} 6_{\pm1}^{2} 8_{\pm1} 3_{\pm4}^{2} 5_{\pm4}^{3}  7_{\pm4},
1_{\pm1} 2_{0}^{2} 4_{\pm1}^{2} 6_{\pm1}^{2} 8_{\pm1} 3_{\pm4}^{3}  5_{\pm4}^{3}  7_{\pm4}^{2},
1_{\pm1}^{2} 2_{0}^{3}  4_{\pm1}^{2} 6_{\pm1}^{3}  8_{\pm1} 3_{\pm4}^{4}  5_{\pm4}^{4}  7_{\pm4}^{2}, \\
& 1_{\pm1} 2_{0} 4_{\pm1}^{2} 6_{\pm1}^{2} 8_{\pm1} 3_{\pm4}^{2} 5_{\pm4}^{3}  7_{\pm4}^{2},
1_{\pm1}^{2} 2_{0}^{2} 4_{\pm1}^{2} 6_{\pm1}^{3}  8_{\pm1} 3_{\pm4}^{3}  5_{\pm4}^{4}  7_{\pm4}^{2},
1_{\pm1} 2_{0}^{2} 4_{\pm1} 6_{\pm1} 8_{\pm1} 3_{\pm4}^{2} 5_{\pm4}^{2} 7_{\pm4},
1_{\pm1}^{2} 2_{0}^{3}  4_{\pm1}^{3}  6_{\pm1}^{3}  8_{\pm1} 3_{\pm4}^{4}  5_{\pm4}^{5}  7_{\pm4}^{2}, \\
& 1_{\pm1}^{2} 2_{0}^{2} 4_{\pm1}^{2} 6_{\pm1}^{2} 8_{\pm1} 3_{\pm4}^{3}  5_{\pm4}^{3}  7_{\pm4}^{2},
1_{\pm1} 2_{0}^{2} 4_{\pm1} 6_{\pm1}^{3}  8_{\pm1} 3_{\pm4}^{2} 5_{\pm4}^{3}  7_{\pm4}^{2},
1_{\pm1} 2_{0}^{3}  4_{\pm1}^{2} 6_{\pm1}^{3}  8_{\pm1} 3_{\pm4}^{3}  5_{\pm4}^{4}  7_{\pm4}^{2},
1_{\pm1} 2_{0} 4_{\pm1} 6_{\pm1}^{2} 8_{\pm1} 3_{\pm4}^{2} 5_{\pm4}^{2} 7_{\pm4}, \\
& 1_{\pm1} 2_{0} 4_{\pm1} 6_{\pm1} 8_{\pm1} 3_{\pm4} 5_{\pm4}^{2} 7_{\pm4},
1_{\pm1}^{2} 2_{0}^{3}  4_{\pm1}^{3}  6_{\pm1}^{4}  8_{\pm1} 3_{\pm4}^{4}  5_{\pm4}^{5}  7_{\pm4}^{3},
1_{\pm1} 2_{0}^{2} 4_{\pm1}^{2} 6_{\pm1}^{3}  8_{\pm1} 3_{\pm4}^{2} 5_{\pm4}^{4}  7_{\pm4}^{2},
1_{\pm1} 2_{0}^{2} 4_{\pm1}^{2} 6_{\pm1}^{2} 8_{\pm1} 3_{\pm4}^{3}  5_{\pm4}^{3}  7_{\pm4}, \\
& 1_{\pm1}^{2} 2_{0}^{2} 4_{\pm1}^{3}  6_{\pm1}^{3}  8_{\pm1} 3_{\pm4}^{4}  5_{\pm4}^{4}  7_{\pm4}^{2},
1_{\pm1} 2_{0}^{2} 4_{\pm1} 6_{\pm1}^{2} 8_{\pm1} 3_{\pm4}^{2} 5_{\pm4}^{2} 7_{\pm4}^{2},
1_{\pm1}^{2} 2_{0}^{3}  4_{\pm1}^{3}  6_{\pm1}^{4}  8_{\pm1}^{2} 3_{\pm4}^{4}  5_{\pm4}^{5}  7_{\pm4}^{3},
1_{\pm1} 2_{0}^{2} 4_{\pm1}^{2} 6_{\pm1}^{3}  8_{\pm1} 3_{\pm4}^{3}  5_{\pm4}^{3}  7_{\pm4}^{2}, \\
& 1_{\pm1}^{2} 2_{0}^{3}  4_{\pm1}^{2} 6_{\pm1}^{3}  8_{\pm1} 3_{\pm4}^{3}  5_{\pm4}^{4}  7_{\pm4}^{2},
1_{\pm1} 2_{0} 4_{\pm1} 6_{\pm1} 8_{\pm1} 3_{\pm4}^{2} 5_{\pm4} 7_{\pm4},
1_{\pm1} 2_{0} 4_{\pm1}^{2} 6_{\pm1}^{2} 8_{\pm1} 3_{\pm4}^{2} 5_{\pm4}^{3}  7_{\pm4},
1_{\pm1}^{2} 2_{0}^{3}  4_{\pm1}^{3}  6_{\pm1}^{4}  8_{\pm1} 3_{\pm4}^{4}  5_{\pm4}^{5}  7_{\pm4}^{2}, \\
& 1_{\pm1}^{2} 2_{0}^{2} 4_{\pm1}^{2} 6_{\pm1}^{2} 8_{\pm1} 3_{\pm4}^{3}  5_{\pm4}^{3}  7_{\pm4},
1_{\pm1} 2_{0}^{2} 4_{\pm1}^{2} 6_{\pm1}^{2} 8_{\pm1} 3_{\pm4}^{2} 5_{\pm4}^{3}  7_{\pm4}^{2},
1_{\pm1} 2_{0}^{2} 4_{\pm1}^{3}  6_{\pm1}^{3}  8_{\pm1} 3_{\pm4}^{3}  5_{\pm4}^{4}  7_{\pm4}^{2},
1_{\pm1} 2_{0} 4_{\pm1} 6_{\pm1}^{2} 8_{\pm1} 3_{\pm4} 5_{\pm4}^{2} 7_{\pm4}^{2}, \\
& 1_{\pm1}^{2} 2_{0}^{3}  4_{\pm1}^{3}  6_{\pm1}^{3}  8_{\pm1} 3_{\pm4}^{4}  5_{\pm4}^{4}  7_{\pm4}^{2},
1_{\pm1} 2_{0} 4_{\pm1}^{2} 6_{\pm1}^{3}  8_{\pm1} 3_{\pm4}^{2} 5_{\pm4}^{3}  7_{\pm4}^{2},
1_{\pm1}^{2} 2_{0}^{2} 4_{\pm1}^{2} 6_{\pm1}^{3}  8_{\pm1} 3_{\pm4}^{3}  5_{\pm4}^{3}  7_{\pm4}^{2},
1_{\pm1} 2_{0} 4_{\pm1}^{2} 6_{\pm1} 8_{\pm1} 3_{\pm4}^{2} 5_{\pm4}^{2} 7_{\pm4}, \\
& 1_{\pm1}^{2} 2_{0}^{2} 4_{\pm1}^{3}  6_{\pm1}^{3}  8_{\pm1} 3_{\pm4}^{3}  5_{\pm4}^{4}  7_{\pm4}^{2},
1_{\pm1} 2_{0} 4_{\pm1}^{2} 6_{\pm1}^{2} 8_{\pm1} 3_{\pm4}^{2} 5_{\pm4}^{2} 7_{\pm4}^{2},
1_{\pm1} 2_{0}^{2} 4_{\pm1} 6_{\pm1}^{2} 8_{\pm1} 3_{\pm4}^{2} 5_{\pm4}^{2} 7_{\pm4},
1_{\pm1} 2_{0}^{2} 4_{\pm1}^{2} 6_{\pm1}^{2} 8_{\pm1} 3_{\pm4}^{2} 5_{\pm4}^{3}  7_{\pm4}, \\
& 1_{\pm1} 2_{0}^{2} 4_{\pm1}^{2} 6_{\pm1}^{3}  8_{\pm1} 3_{\pm4}^{2} 5_{\pm4}^{3}  7_{\pm4}^{2},
1_{\pm1} 2_{0} 4_{\pm1} 6_{\pm1}^{2} 8_{\pm1} 3_{\pm4} 5_{\pm4}^{2} 7_{\pm4},
1_{\pm1} 2_{0} 4_{\pm1}^{2} 6_{\pm1}^{2} 8_{\pm1} 3_{\pm4}^{2} 5_{\pm4}^{2} 7_{\pm4},
1_{\pm1} 2_{0} 4_{\pm1} 6_{\pm1} 8_{\pm1} 3_{\pm4} 5_{\pm4} 7_{\pm4}. \\
%%\end{align*}
%%\end{gather}
%%
%%
& \quad \\
%%
%%
%Let $\xi=(-1,0,-2,-1,-2,-1,0,-1)$. 
%\begin{align*}
%\xymatrix{
%& & 2 \ar[d]  &  &  \\
%1  \ar[r] & 3   & 4 \ar[l]   \ar[r]  & 5   &  6 \ar[l]  &  7  \ar[l]   \ar[r]  & 8 }
%\end{align*}
%The highest $l$-weight monomials of Hernandez-Leclerc modules of type $E_8$ that are not of type $A$, $D$ or $E_7$ are 
%%\begin{gather}
%%\begin{align*}
& (54)  1_{\pm1} 2_{0} 4_{\pm1} 7_{0} 3_{\pm4}^{2} 5_{\pm4}^{2} 8_{\pm3},
1_{\pm1} 2_{0} 7_{0} 3_{\pm4} 5_{\pm4} 8_{\pm3},
1_{\pm1} 2_{0}^{2} 4_{\pm1} 7_{0}^{2} 3_{\pm4}^{2} 5_{\pm4}^{3} 8_{\pm3},
1_{\pm1} 2_{0} 4_{\pm1} 7_{0}^{2} 3_{\pm4}^{2} 5_{\pm4}^{2} 8_{\pm3}, \\
& 1_{\pm1} 2_{0}^{2} 4_{\pm1} 6_{\pm1} 7_{0} 3_{\pm4}^{2} 5_{\pm4}^{3} 8_{\pm3},
1_{\pm1} 2_{0}^{2} 4_{\pm1}^{2} 6_{\pm1} 7_{0}^{2} 3_{\pm4}^{3} 5_{\pm4}^{4} 8_{\pm3},
1_{\pm1} 2_{0}^{2} 4_{\pm1} 7_{0} 3_{\pm4}^{2} 5_{\pm4}^{2} 8_{\pm3},
1_{\pm1} 2_{0}^{2} 4_{\pm1}^{2} 7_{0}^{2} 3_{\pm4}^{3} 5_{\pm4}^{3} 8_{\pm3}, \\
& 1_{\pm1} 2_{0} 4_{\pm1} 6_{\pm1} 7_{0} 3_{\pm4}^{2} 5_{\pm4}^{2} 8_{\pm3},
1_{\pm1}^{2} 2_{0}^{3} 4_{\pm1}^{2} 6_{\pm1} 7_{0}^{2} 3_{\pm4}^{4} 5_{\pm4}^{4} 8_{\pm3},
1_{\pm1} 2_{0}^{2} 4_{\pm1}^{2} 6_{\pm1} 7_{0} 3_{\pm4}^{3} 5_{\pm4}^{3} 8_{\pm3},
1_{\pm1} 2_{0} 4_{\pm1} 7_{0} 3_{\pm4} 5_{\pm4}^{2} 8_{\pm3}, \\
& 1_{\pm1} 2_{0} 4_{\pm1}^{2} 7_{0}^{2} 3_{\pm4}^{2} 5_{\pm4}^{3} 8_{\pm3},
1_{\pm1}^{2} 2_{0}^{2} 4_{\pm1}^{2} 6_{\pm1} 7_{0}^{2} 3_{\pm4}^{3} 5_{\pm4}^{4} 8_{\pm3},
1_{\pm1} 2_{0} 4_{\pm1} 7_{0} 3_{\pm4}^{2} 5_{\pm4} 8_{\pm3},
1_{\pm1}^{2} 2_{0}^{3} 4_{\pm1}^{3} 6_{\pm1} 7_{0}^{2} 3_{\pm4}^{4} 5_{\pm4}^{5} 8_{\pm3}, \\
& 1_{\pm1}^{2} 2_{0}^{2} 4_{\pm1}^{2} 7_{0}^{2} 3_{\pm4}^{3} 5_{\pm4}^{3} 8_{\pm3},
1_{\pm1} 2_{0}^{2} 4_{\pm1} 6_{\pm1} 7_{0}^{2} 3_{\pm4}^{2} 5_{\pm4}^{3} 8_{\pm3},
1_{\pm1} 2_{0}^{3} 4_{\pm1}^{2} 6_{\pm1} 7_{0}^{2} 3_{\pm4}^{3} 5_{\pm4}^{4} 8_{\pm3},
1_{\pm1} 2_{0} 4_{\pm1}^{2} 6_{\pm1} 7_{0} 3_{\pm4}^{2} 5_{\pm4}^{3} 8_{\pm3}, \\
& 1_{\pm1}^{2} 2_{0}^{3} 4_{\pm1}^{3} 6_{\pm1} 7_{0}^{3} 3_{\pm4}^{4} 5_{\pm4}^{5} 8_{\pm3}^{2},
1_{\pm1} 2_{0}^{2} 4_{\pm1}^{2} 6_{\pm1} 7_{0}^{2} 3_{\pm4}^{2} 5_{\pm4}^{4} 8_{\pm3},
1_{\pm1}^{2} 2_{0}^{2} 4_{\pm1}^{3} 6_{\pm1} 7_{0}^{2} 3_{\pm4}^{4} 5_{\pm4}^{4} 8_{\pm3},
1_{\pm1} 2_{0}^{2} 4_{\pm1} 7_{0}^{2} 3_{\pm4}^{2} 5_{\pm4}^{2} 8_{\pm3}, \\
& 1_{\pm1}^{2} 2_{0}^{3} 4_{\pm1}^{3} 6_{\pm1} 7_{0}^{3} 3_{\pm4}^{4} 5_{\pm4}^{5} 8_{\pm3},
1_{\pm1} 2_{0}^{2} 4_{\pm1}^{2} 6_{\pm1} 7_{0}^{2} 3_{\pm4}^{3} 5_{\pm4}^{3} 8_{\pm3},
1_{\pm1}^{2} 2_{0}^{2} 4_{\pm1}^{2} 6_{\pm1} 7_{0} 3_{\pm4}^{3} 5_{\pm4}^{3} 8_{\pm3},
1_{\pm1}^{2} 2_{0}^{3} 4_{\pm1}^{2} 6_{\pm1} 7_{0}^{2} 3_{\pm4}^{3} 5_{\pm4}^{4} 8_{\pm3}, \\
& 1_{\pm1}^{2} 2_{0}^{3} 4_{\pm1}^{3} 6_{\pm1}^{2} 7_{0}^{2} 3_{\pm4}^{4} 5_{\pm4}^{5} 8_{\pm3},
1_{\pm1} 2_{0} 4_{\pm1}^{2} 7_{0} 3_{\pm4}^{2} 5_{\pm4}^{2} 8_{\pm3},
1_{\pm1} 2_{0}^{2} 4_{\pm1}^{2} 7_{0}^{2} 3_{\pm4}^{2} 5_{\pm4}^{3} 8_{\pm3},
1_{\pm1} 2_{0}^{2} 4_{\pm1}^{3} 6_{\pm1} 7_{0}^{2} 3_{\pm4}^{3} 5_{\pm4}^{4} 8_{\pm3}, \\
& 1_{\pm1} 2_{0}^{2} 4_{\pm1} 6_{\pm1} 7_{0} 3_{\pm4}^{2} 5_{\pm4}^{2} 8_{\pm3},
1_{\pm1}^{2} 2_{0}^{3} 4_{\pm1}^{3} 6_{\pm1} 7_{0}^{2} 3_{\pm4}^{4} 5_{\pm4}^{4} 8_{\pm3},
1_{\pm1} 2_{0}^{2} 4_{\pm1}^{2} 6_{\pm1} 7_{0} 3_{\pm4}^{2} 5_{\pm4}^{3} 8_{\pm3},
1_{\pm1} 2_{0} 4_{\pm1} 7_{0}^{2} 3_{\pm4} 5_{\pm4}^{2} 8_{\pm3}, \\
& 1_{\pm1} 2_{0} 4_{\pm1}^{2} 6_{\pm1} 7_{0}^{2} 3_{\pm4}^{2} 5_{\pm4}^{3} 8_{\pm3},
1_{\pm1}^{2} 2_{0}^{2} 4_{\pm1}^{2} 6_{\pm1} 7_{0}^{2} 3_{\pm4}^{3} 5_{\pm4}^{3} 8_{\pm3},
1_{\pm1}^{2} 2_{0}^{2} 4_{\pm1}^{3} 6_{\pm1} 7_{0}^{2} 3_{\pm4}^{3} 5_{\pm4}^{4} 8_{\pm3},
1_{\pm1} 2_{0} 4_{\pm1}^{2} 7_{0}^{2} 3_{\pm4}^{2} 5_{\pm4}^{2} 8_{\pm3}, \\
& 1_{\pm1} 2_{0} 4_{\pm1} 6_{\pm1} 7_{0} 3_{\pm4} 5_{\pm4}^{2} 8_{\pm3},
1_{\pm1} 2_{0}^{2} 4_{\pm1}^{2} 6_{\pm1} 7_{0}^{2} 3_{\pm4}^{2} 5_{\pm4}^{3} 8_{\pm3},
1_{\pm1} 2_{0} 4_{\pm1}^{2} 6_{\pm1} 7_{0} 3_{\pm4}^{2} 5_{\pm4}^{2} 8_{\pm3},
1_{\pm1} 2_{0} 4_{\pm1} 7_{0} 3_{\pm4} 5_{\pm4} 8_{\pm3}.
\end{align*}
\end{gather}

%Let $\xi=(-2,-1,-3,-2,-3,-2,-1,0)$. 
%\begin{align*}
%\xymatrix{
%& & 2 \ar[d]  &  &  \\
%1  \ar[r] & 3   & 4 \ar[l]   \ar[r]  & 5   &  6 \ar[l]  &  7  \ar[l]  & 8  \ar[l] } 
%\end{align*}
%The highest $l$-weight monomials of Hernandez-Leclerc modules of type $E_8$ that are not of type $A$, $D$ or $E_7$ are 
\begin{gather}
\begin{align*}
& (55)  1_{\pm2} 2_{\pm1} 4_{\pm2} 8_{0} 3_{\pm5}^{2} 5_{\pm5}^{2},
1_{\pm2} 2_{\pm1} 8_{0} 3_{\pm5} 5_{\pm5},
1_{\pm2} 2_{\pm1}^{2} 4_{\pm2} 7_{\pm1} 8_{0} 3_{\pm5}^{2} 5_{\pm5}^{3},
1_{\pm2} 2_{\pm1} 4_{\pm2} 7_{\pm1} 8_{0} 3_{\pm5}^{2} 5_{\pm5}^{2}, \\
& 1_{\pm2} 2_{\pm1}^{2} 4_{\pm2} 6_{\pm2} 8_{0} 3_{\pm5}^{2} 5_{\pm5}^{3},
1_{\pm2} 2_{\pm1}^{2} 4_{\pm2}^{2} 6_{\pm2} 7_{\pm1} 8_{0} 3_{\pm5}^{3}  5_{\pm5}^{4},
1_{\pm2} 2_{\pm1}^{2} 4_{\pm2} 8_{0} 3_{\pm5}^{2} 5_{\pm5}^{2},
1_{\pm2} 2_{\pm1}^{2} 4_{\pm2}^{2} 7_{\pm1} 8_{0} 3_{\pm5}^{3} 5_{\pm5}^{3}, \\
& 1_{\pm2} 2_{\pm1} 4_{\pm2} 6_{\pm2} 8_{0} 3_{\pm5}^{2} 5_{\pm5}^{2},
1_{\pm2}^{2} 2_{\pm1}^{3}  4_{\pm2}^{2} 6_{\pm2} 7_{\pm1} 8_{0} 3_{\pm5}^{4}  5_{\pm5}^{4},
1_{\pm2} 2_{\pm1}^{2} 4_{\pm2}^{2} 6_{\pm2} 8_{0} 3_{\pm5}^{3}  5_{\pm5}^{3},
1_{\pm2} 2_{\pm1} 4_{\pm2} 8_{0} 3_{\pm5} 5_{\pm5}^{2}, \\
& 1_{\pm2} 2_{\pm1} 4_{\pm2}^{2} 7_{\pm1} 8_{0} 3_{\pm5}^{2} 5_{\pm5}^{3},
1_{\pm2}^{2} 2_{\pm1}^{2} 4_{\pm2}^{2} 6_{\pm2} 7_{\pm1} 8_{0} 3_{\pm5}^{3}  5_{\pm5}^{4},
1_{\pm2} 2_{\pm1} 4_{\pm2} 8_{0} 3_{\pm5}^{2} 5_{\pm5},
1_{\pm2}^{2} 2_{\pm1}^{3}  4_{\pm2}^{3}  6_{\pm2} 7_{\pm1} 8_{0} 3_{\pm5}^{4}   5_{\pm5}^{5}, \\
& 1_{\pm2}^{2} 2_{\pm1}^{2} 4_{\pm2}^{2} 7_{\pm1} 8_{0} 3_{\pm5}^{3}  5_{\pm5}^{3},
1_{\pm2} 2_{\pm1}^{2} 4_{\pm2} 6_{\pm2} 7_{\pm1} 8_{0} 3_{\pm5}^{2} 5_{\pm5}^{3},
1_{\pm2} 2_{\pm1}^{3}  4_{\pm2}^{2} 6_{\pm2} 7_{\pm1} 8_{0} 3_{\pm5}^{3}  5_{\pm5}^{4},
1_{\pm2} 2_{\pm1} 4_{\pm2}^{2} 6_{\pm2} 8_{0} 3_{\pm5}^{2} 5_{\pm5}^{3}, \\
& 1_{\pm2}^{2} 2_{\pm1}^{3}  4_{\pm2}^{3}  6_{\pm2} 7_{\pm1} 8_{0}^{2} 3_{\pm5}^{4}   5_{\pm5}^{5},
1_{\pm2} 2_{\pm1}^{2} 4_{\pm2}^{2} 6_{\pm2} 7_{\pm1} 8_{0} 3_{\pm5}^{2} 5_{\pm5}^{4},
1_{\pm2}^{2} 2_{\pm1}^{2} 4_{\pm2}^{3}  6_{\pm2} 7_{\pm1} 8_{0} 3_{\pm5}^{4}   5_{\pm5}^{4},
1_{\pm2} 2_{\pm1}^{2} 4_{\pm2} 7_{\pm1} 8_{0} 3_{\pm5}^{2} 5_{\pm5}^{2}, \\
& 1_{\pm2}^{2} 2_{\pm1}^{3}  4_{\pm2}^{3}  6_{\pm2} 7_{\pm1}^{2} 8_{0} 3_{\pm5}^{4}   5_{\pm5}^{5},
1_{\pm2} 2_{\pm1}^{2} 4_{\pm2}^{2} 6_{\pm2} 7_{\pm1} 8_{0} 3_{\pm5}^{3}  5_{\pm5}^{3},
1_{\pm2}^{2} 2_{\pm1}^{2} 4_{\pm2}^{2} 6_{\pm2} 8_{0} 3_{\pm5}^{3}  5_{\pm5}^{3},
1_{\pm2}^{2} 2_{\pm1}^{3}  4_{\pm2}^{2} 6_{\pm2} 7_{\pm1} 8_{0} 3_{\pm5}^{3}  5_{\pm5}^{4}, \\
& 1_{\pm2}^{2} 2_{\pm1}^{3}  4_{\pm2}^{3}  6_{\pm2}^{2} 7_{\pm1} 8_{0} 3_{\pm5}^{4}   5_{\pm5}^{5},
1_{\pm2} 2_{\pm1} 4_{\pm2}^{2} 8_{0} 3_{\pm5}^{2} 5_{\pm5}^{2},
1_{\pm2} 2_{\pm1}^{2} 4_{\pm2}^{2} 7_{\pm1} 8_{0} 3_{\pm5}^{2} 5_{\pm5}^{3},
1_{\pm2} 2_{\pm1}^{2} 4_{\pm2}^{3}  6_{\pm2} 7_{\pm1} 8_{0} 3_{\pm5}^{3}  5_{\pm5}^{4}, \\
& 1_{\pm2} 2_{\pm1}^{2} 4_{\pm2} 6_{\pm2} 8_{0} 3_{\pm5}^{2} 5_{\pm5}^{2},
1_{\pm2}^{2} 2_{\pm1}^{3}  4_{\pm2}^{3}  6_{\pm2} 7_{\pm1} 8_{0} 3_{\pm5}^{4}   5_{\pm5}^{4},
1_{\pm2} 2_{\pm1}^{2} 4_{\pm2}^{2} 6_{\pm2} 8_{0} 3_{\pm5}^{2} 5_{\pm5}^{3},
1_{\pm2} 2_{\pm1} 4_{\pm2} 7_{\pm1} 8_{0} 3_{\pm5} 5_{\pm5}^{2}, \\
& 1_{\pm2} 2_{\pm1} 4_{\pm2}^{2} 6_{\pm2} 7_{\pm1} 8_{0} 3_{\pm5}^{2} 5_{\pm5}^{3},
1_{\pm2}^{2} 2_{\pm1}^{2} 4_{\pm2}^{2} 6_{\pm2} 7_{\pm1} 8_{0} 3_{\pm5}^{3} 5_{\pm5}^{3},
1_{\pm2}^{2} 2_{\pm1}^{2} 4_{\pm2}^{3}  6_{\pm2} 7_{\pm1} 8_{0} 3_{\pm5}^{3} 5_{\pm5}^{4},
1_{\pm2} 2_{\pm1} 4_{\pm2}^{2} 7_{\pm1} 8_{0} 3_{\pm5}^{2} 5_{\pm5}^{2}, \\
& 1_{\pm2} 2_{\pm1} 4_{\pm2} 6_{\pm2} 8_{0} 3_{\pm5} 5_{\pm5}^{2},
1_{\pm2} 2_{\pm1}^{2} 4_{\pm2}^{2} 6_{\pm2} 7_{\pm1} 8_{0} 3_{\pm5}^{2} 5_{\pm5}^{3},
1_{\pm2} 2_{\pm1} 4_{\pm2}^{2} 6_{\pm2} 8_{0} 3_{\pm5}^{2} 5_{\pm5}^{2},
1_{\pm2} 2_{\pm1} 4_{\pm2} 8_{0} 3_{\pm5} 5_{\pm5}. \\
%%\end{align*}
%%\end{gather}
%%
%%
& \quad \\
%%
%%
%Let $\xi=(-1,0,-2,-1,0,-1,-2,-3)$. 
%\begin{align*}
%\xymatrix{
%& & 2 \ar[d]  &  &  \\
%1  \ar[r] & 3   & 4 \ar[l]   & 5  \ar[l]  \ar[r]  &  6 \ar[r]  &  7 \ar[r] & 8 } 
%\end{align*}
%The highest $l$-weight monomials of Hernandez-Leclerc modules of type $E_8$ that are not of type $A$, $D$ or $E_7$ are 
%%\begin{gather}
%%\begin{align*}
& (56)  1_{\pm1} 2_{0} 5_{0}^{2} 3_{\pm4}^{2} 4_{\pm3} 6_{\pm3} 8_{\pm5},
1_{\pm1} 2_{0}^{2} 5_{0}^{3} 3_{\pm4}^{2} 4_{\pm3}^{2} 6_{\pm3} 7_{\pm4} 8_{\pm5},
1_{\pm1} 2_{0} 5_{0}^{2} 3_{\pm4}^{2} 4_{\pm3} 7_{\pm4} 8_{\pm5},
1_{\pm1} 2_{0} 5_{0} 3_{\pm4} 4_{\pm3} 8_{\pm5}, \\
& 1_{\pm1} 2_{0}^{2} 5_{0}^{4} 3_{\pm4}^{3} 4_{\pm3}^{2} 6_{\pm3} 7_{\pm4} 8_{\pm5},
1_{\pm1} 2_{0}^{2} 5_{0}^{3} 3_{\pm4}^{2} 4_{\pm3}^{2} 6_{\pm3} 8_{\pm5},
1_{\pm1} 2_{0}^{2} 5_{0}^{3} 3_{\pm4}^{3} 4_{\pm3} 6_{\pm3} 7_{\pm4} 8_{\pm5},
1_{\pm1}^{2} 2_{0}^{3} 5_{0}^{4} 3_{\pm4}^{4} 4_{\pm3}^{2} 6_{\pm3} 7_{\pm4} 8_{\pm5}, \\
& 1_{\pm1} 2_{0} 5_{0}^{3} 3_{\pm4}^{2} 4_{\pm3} 6_{\pm3} 7_{\pm4} 8_{\pm5},
1_{\pm1} 2_{0}^{2} 5_{0}^{2} 3_{\pm4}^{2} 4_{\pm3} 6_{\pm3} 8_{\pm5},
1_{\pm1}^{2} 2_{0}^{2} 5_{0}^{4} 3_{\pm4}^{3} 4_{\pm3}^{2} 6_{\pm3} 7_{\pm4} 8_{\pm5},
1_{\pm1}^{2} 2_{0}^{3} 5_{0}^{5}  3_{\pm4}^{4} 4_{\pm3}^{2} 6_{\pm3}^{2} 7_{\pm4} 8_{\pm5}, \\
& 1_{\pm1}^{2} 2_{0}^{2} 5_{0}^{3} 3_{\pm4}^{3} 4_{\pm3} 6_{\pm3} 7_{\pm4} 8_{\pm5},
1_{\pm1} 2_{0}^{2} 5_{0}^{3} 3_{\pm4}^{2} 4_{\pm3}^{2} 7_{\pm4} 8_{\pm5},
1_{\pm1} 2_{0} 5_{0}^{2} 3_{\pm4}^{2} 4_{\pm3} 8_{\pm5},
1_{\pm1} 2_{0} 5_{0}^{2} 3_{\pm4} 4_{\pm3} 6_{\pm3} 8_{\pm5}, \\
& 1_{\pm1} 2_{0}^{3} 5_{0}^{4} 3_{\pm4}^{3} 4_{\pm3}^{2} 6_{\pm3} 7_{\pm4} 8_{\pm5},
1_{\pm1}^{2} 2_{0}^{3} 5_{0}^{5}  3_{\pm4}^{4} 4_{\pm3}^{2} 6_{\pm3} 7_{\pm4}^{2} 8_{\pm5},
1_{\pm1} 2_{0}^{2} 5_{0}^{4} 3_{\pm4}^{2} 4_{\pm3}^{2} 6_{\pm3} 7_{\pm4} 8_{\pm5},
1_{\pm1} 2_{0}^{2} 5_{0}^{3} 3_{\pm4}^{3} 4_{\pm3} 6_{\pm3} 8_{\pm5}, \\
& 1_{\pm1} 2_{0}^{2} 5_{0}^{2} 3_{\pm4}^{2} 4_{\pm3} 7_{\pm4} 8_{\pm5},
1_{\pm1}^{2} 2_{0}^{2} 5_{0}^{4} 3_{\pm4}^{4} 4_{\pm3} 6_{\pm3} 7_{\pm4} 8_{\pm5},
1_{\pm1}^{2} 2_{0}^{3} 5_{0}^{5}  3_{\pm4}^{4} 4_{\pm3}^{2} 6_{\pm3} 7_{\pm4} 8_{\pm5}^{2},
1_{\pm1} 2_{0}^{2} 5_{0}^{3} 3_{\pm4}^{3} 4_{\pm3} 7_{\pm4} 8_{\pm5}, \\
& 1_{\pm1} 2_{0} 5_{0} 3_{\pm4}^{2} 8_{\pm5},
1_{\pm1} 2_{0} 5_{0}^{3} 3_{\pm4}^{2} 4_{\pm3} 6_{\pm3} 8_{\pm5},
1_{\pm1}^{2} 2_{0}^{3} 5_{0}^{4} 3_{\pm4}^{3} 4_{\pm3}^{2} 6_{\pm3} 7_{\pm4} 8_{\pm5},
1_{\pm1}^{2} 2_{0}^{3} 5_{0}^{5}  3_{\pm4}^{4} 4_{\pm3}^{2} 6_{\pm3} 7_{\pm4} 8_{\pm5}, \\
& 1_{\pm1}^{2} 2_{0}^{2} 5_{0}^{3} 3_{\pm4}^{3} 4_{\pm3} 6_{\pm3} 8_{\pm5},
1_{\pm1} 2_{0}^{2} 5_{0}^{3} 3_{\pm4}^{2} 4_{\pm3} 6_{\pm3} 7_{\pm4} 8_{\pm5},
1_{\pm1} 2_{0} 5_{0}^{2} 3_{\pm4} 4_{\pm3} 7_{\pm4} 8_{\pm5},
1_{\pm1} 2_{0}^{2} 5_{0}^{4} 3_{\pm4}^{3} 4_{\pm3} 6_{\pm3} 7_{\pm4} 8_{\pm5}, \\
& 1_{\pm1}^{2} 2_{0}^{3} 5_{0}^{4} 3_{\pm4}^{4} 4_{\pm3} 6_{\pm3} 7_{\pm4} 8_{\pm5},
1_{\pm1} 2_{0} 5_{0}^{3} 3_{\pm4}^{2} 4_{\pm3} 7_{\pm4} 8_{\pm5},
1_{\pm1} 2_{0} 5_{0}^{2} 3_{\pm4}^{2} 6_{\pm3} 8_{\pm5},
1_{\pm1}^{2} 2_{0}^{2} 5_{0}^{3} 3_{\pm4}^{3} 4_{\pm3} 7_{\pm4} 8_{\pm5}, \\
& 1_{\pm1}^{2} 2_{0}^{2} 5_{0}^{4} 3_{\pm4}^{3} 4_{\pm3} 6_{\pm3} 7_{\pm4} 8_{\pm5},
1_{\pm1} 2_{0} 5_{0}^{2} 3_{\pm4}^{2} 7_{\pm4} 8_{\pm5},
1_{\pm1} 2_{0}^{2} 5_{0}^{2} 3_{\pm4}^{2} 4_{\pm3} 8_{\pm5},
1_{\pm1} 2_{0}^{2} 5_{0}^{3} 3_{\pm4}^{2} 4_{\pm3} 6_{\pm3} 8_{\pm5}, \\
& 1_{\pm1} 2_{0}^{2} 5_{0}^{3} 3_{\pm4}^{2} 4_{\pm3} 7_{\pm4} 8_{\pm5},
1_{\pm1} 2_{0} 5_{0}^{2} 3_{\pm4} 4_{\pm3} 8_{\pm5},
1_{\pm1} 2_{0} 5_{0}^{2} 3_{\pm4}^{2} 8_{\pm5},
1_{\pm1} 2_{0} 5_{0} 3_{\pm4} 8_{\pm5}. \\
%%\end{align*}
%%\end{gather}
%%
%%
& \quad \\
%%
%%
%Let $\xi=(-1,0,-2,-1,0,-1,-2,-1)$. 
%\begin{align*}
%\xymatrix{
%& & 2 \ar[d]  &  &  \\
%1  \ar[r] & 3   & 4 \ar[l]   & 5  \ar[l]  \ar[r]  &  6 \ar[r]  &  7 & 8  \ar[l] } 
%\end{align*}
%The highest $l$-weight monomials of Hernandez-Leclerc modules of type $E_8$ that are not of type $A$, $D$ or $E_7$ are 
%%\begin{gather} 
%%\begin{align*}
& (57)  1_{\pm1}  2_{0} 5_{0}^{2} 8_{\pm1} 3_{\pm4}^{2} 4_{\pm3} 6_{\pm3} 7_{\pm4},
1_{\pm1}  2_{0}^{2} 5_{0}^{3} 8_{\pm1} 3_{\pm4}^{2} 4_{\pm3}^{2} 6_{\pm3} 7_{\pm4}^{2},
1_{\pm1}  2_{0} 5_{0}^{2} 8_{\pm1} 3_{\pm4}^{2} 4_{\pm3} 7_{\pm4}^{2},
1_{\pm1}  2_{0} 5_{0} 8_{\pm1} 3_{\pm4} 4_{\pm3} 7_{\pm4}, \\
& 1_{\pm1}  2_{0}^{2} 5_{0}^{4} 8_{\pm1} 3_{\pm4}^{3} 4_{\pm3}^{2} 6_{\pm3} 7_{\pm4}^{2},
1_{\pm1}  2_{0}^{2} 5_{0}^{3} 8_{\pm1} 3_{\pm4}^{2} 4_{\pm3}^{2} 6_{\pm3} 7_{\pm4},
1_{\pm1}  2_{0}^{2} 5_{0}^{3} 8_{\pm1} 3_{\pm4}^{3} 4_{\pm3} 6_{\pm3} 7_{\pm4}^{2},
1_{\pm1}^{2} 2_{0}^{3} 5_{0}^{4} 8_{\pm1} 3_{\pm4}^{4} 4_{\pm3}^{2} 6_{\pm3} 7_{\pm4}^{2}, \\
& 1_{\pm1}  2_{0} 5_{0}^{3} 8_{\pm1} 3_{\pm4}^{2} 4_{\pm3} 6_{\pm3} 7_{\pm4}^{2},
1_{\pm1}  2_{0}^{2} 5_{0}^{2} 8_{\pm1} 3_{\pm4}^{2} 4_{\pm3} 6_{\pm3} 7_{\pm4},
1_{\pm1}^{2} 2_{0}^{2} 5_{0}^{4} 8_{\pm1} 3_{\pm4}^{3} 4_{\pm3}^{2} 6_{\pm3} 7_{\pm4}^{2},
1_{\pm1}^{2} 2_{0}^{3} 5_{0}^{5} 8_{\pm1} 3_{\pm4}^{4} 4_{\pm3}^{2} 6_{\pm3}^{2} 7_{\pm4}^{2}, \\
& 1_{\pm1}^{2} 2_{0}^{2} 5_{0}^{3} 8_{\pm1} 3_{\pm4}^{3} 4_{\pm3} 6_{\pm3} 7_{\pm4}^{2},
1_{\pm1}  2_{0}^{2} 5_{0}^{3} 8_{\pm1} 3_{\pm4}^{2} 4_{\pm3}^{2} 7_{\pm4}^{2},
1_{\pm1}  2_{0} 5_{0}^{2} 8_{\pm1} 3_{\pm4}^{2} 4_{\pm3} 7_{\pm4},
1_{\pm1}  2_{0} 5_{0}^{2} 8_{\pm1} 3_{\pm4} 4_{\pm3} 6_{\pm3} 7_{\pm4}, \\
& 1_{\pm1}  2_{0}^{3} 5_{0}^{4} 8_{\pm1} 3_{\pm4}^{3} 4_{\pm3}^{2} 6_{\pm3} 7_{\pm4}^{2},
1_{\pm1}^{2} 2_{0}^{3} 5_{0}^{5} 8_{\pm1} 3_{\pm4}^{4} 4_{\pm3}^{2} 6_{\pm3} 7_{\pm4}^{3},
1_{\pm1}  2_{0}^{2} 5_{0}^{4} 8_{\pm1} 3_{\pm4}^{2} 4_{\pm3}^{2} 6_{\pm3} 7_{\pm4}^{2},
1_{\pm1}  2_{0}^{2} 5_{0}^{3} 8_{\pm1} 3_{\pm4}^{3} 4_{\pm3} 6_{\pm3} 7_{\pm4}, \\
& 1_{\pm1}  2_{0}^{2} 5_{0}^{2} 8_{\pm1} 3_{\pm4}^{2} 4_{\pm3} 7_{\pm4}^{2},
1_{\pm1}^{2} 2_{0}^{2} 5_{0}^{4} 8_{\pm1} 3_{\pm4}^{4} 4_{\pm3} 6_{\pm3} 7_{\pm4}^{2},
1_{\pm1}^{2} 2_{0}^{3} 5_{0}^{5} 8_{\pm1}^{2} 3_{\pm4}^{4} 4_{\pm3}^{2} 6_{\pm3} 7_{\pm4}^{3},
1_{\pm1}  2_{0}^{2} 5_{0}^{3} 8_{\pm1} 3_{\pm4}^{3} 4_{\pm3} 7_{\pm4}^{2}, \\
& 1_{\pm1}  2_{0} 5_{0} 8_{\pm1} 3_{\pm4}^{2} 7_{\pm4},
1_{\pm1}  2_{0} 5_{0}^{3} 8_{\pm1} 3_{\pm4}^{2} 4_{\pm3} 6_{\pm3} 7_{\pm4},
1_{\pm1}^{2} 2_{0}^{3} 5_{0}^{4} 8_{\pm1} 3_{\pm4}^{3} 4_{\pm3}^{2} 6_{\pm3} 7_{\pm4}^{2},
1_{\pm1}^{2} 2_{0}^{3} 5_{0}^{5} 8_{\pm1} 3_{\pm4}^{4} 4_{\pm3}^{2} 6_{\pm3} 7_{\pm4}^{2}, \\
& 1_{\pm1}^{2} 2_{0}^{2} 5_{0}^{3} 8_{\pm1} 3_{\pm4}^{3} 4_{\pm3} 6_{\pm3} 7_{\pm4},
1_{\pm1}  2_{0}^{2} 5_{0}^{3} 8_{\pm1} 3_{\pm4}^{2} 4_{\pm3} 6_{\pm3} 7_{\pm4}^{2},
1_{\pm1}  2_{0} 5_{0}^{2} 8_{\pm1} 3_{\pm4} 4_{\pm3} 7_{\pm4}^{2},
1_{\pm1}  2_{0}^{2} 5_{0}^{4} 8_{\pm1} 3_{\pm4}^{3} 4_{\pm3} 6_{\pm3} 7_{\pm4}^{2}, \\
& 1_{\pm1}^{2} 2_{0}^{3} 5_{0}^{4} 8_{\pm1} 3_{\pm4}^{4} 4_{\pm3} 6_{\pm3} 7_{\pm4}^{2},
1_{\pm1}  2_{0} 5_{0}^{3} 8_{\pm1} 3_{\pm4}^{2} 4_{\pm3} 7_{\pm4}^{2},
1_{\pm1}  2_{0} 5_{0}^{2} 8_{\pm1} 3_{\pm4}^{2} 6_{\pm3} 7_{\pm4},
1_{\pm1}^{2} 2_{0}^{2} 5_{0}^{3} 8_{\pm1} 3_{\pm4}^{3} 4_{\pm3} 7_{\pm4}^{2}, \\
& 1_{\pm1}^{2} 2_{0}^{2} 5_{0}^{4} 8_{\pm1} 3_{\pm4}^{3} 4_{\pm3} 6_{\pm3} 7_{\pm4}^{2},
1_{\pm1}  2_{0} 5_{0}^{2} 8_{\pm1} 3_{\pm4}^{2} 7_{\pm4}^{2},
1_{\pm1}  2_{0}^{2} 5_{0}^{2} 8_{\pm1} 3_{\pm4}^{2} 4_{\pm3} 7_{\pm4},
1_{\pm1}  2_{0}^{2} 5_{0}^{3} 8_{\pm1} 3_{\pm4}^{2} 4_{\pm3} 6_{\pm3} 7_{\pm4}, \\
& 1_{\pm1}  2_{0}^{2} 5_{0}^{3} 8_{\pm1} 3_{\pm4}^{2} 4_{\pm3} 7_{\pm4}^{2},
1_{\pm1}  2_{0} 5_{0}^{2} 8_{\pm1} 3_{\pm4} 4_{\pm3} 7_{\pm4},
1_{\pm1}  2_{0} 5_{0}^{2} 8_{\pm1} 3_{\pm4}^{2} 7_{\pm4},
1_{\pm1}  2_{0} 5_{0} 8_{\pm1} 3_{\pm4} 7_{\pm4}.
\end{align*}
\end{gather}

%Let $\xi=(-1,0,-2,-1,0,-1,0,-1)$. 
%\begin{align*}
%\xymatrix{
%& & 2 \ar[d]  &  &  \\
%1  \ar[r] & 3   & 4 \ar[l]   & 5  \ar[l]  \ar[r]  &  6  &  7 \ar[l]  \ar[r]  & 8 } 
%\end{align*}
%The highest $l$-weight monomials of Hernandez-Leclerc modules of type $E_8$ that are not of type $A$, $D$ or $E_7$ are 
\begin{gather}
\begin{align*}
& (58)  1_{\pm1} 2_{0} 5_{0}^{2} 7_{0} 3_{\pm4}^{2} 4_{\pm3} 6_{\pm3}^{2} 8_{\pm3},
1_{\pm1} 2_{0} 5_{0} 7_{0} 3_{\pm4} 4_{\pm3} 6_{\pm3} 8_{\pm3},
1_{\pm1} 2_{0}^{2} 5_{0}^{3} 7_{0}^{2} 3_{\pm4}^{2} 4_{\pm3}^{2} 6_{\pm3}^{3} 8_{\pm3},
1_{\pm1} 2_{0} 5_{0}^{2} 7_{0}^{2} 3_{\pm4}^{2} 4_{\pm3} 6_{\pm3}^{2} 8_{\pm3}, \\
& 1_{\pm1} 2_{0}^{2} 5_{0}^{3} 7_{0} 3_{\pm4}^{2} 4_{\pm3}^{2} 6_{\pm3}^{2} 8_{\pm3},
1_{\pm1} 2_{0}^{2} 5_{0}^{4} 7_{0}^{2} 3_{\pm4}^{3} 4_{\pm3}^{2} 6_{\pm3}^{3} 8_{\pm3},
1_{\pm1} 2_{0}^{2} 5_{0}^{2} 7_{0} 3_{\pm4}^{2} 4_{\pm3} 6_{\pm3}^{2} 8_{\pm3},
1_{\pm1} 2_{0} 5_{0}^{2} 7_{0} 3_{\pm4}^{2} 4_{\pm3} 6_{\pm3} 8_{\pm3}, \\
 & 1_{\pm1} 2_{0}^{2} 5_{0}^{3} 7_{0}^{2} 3_{\pm4}^{3} 4_{\pm3} 6_{\pm3}^{3} 8_{\pm3},
1_{\pm1}^{2} 2_{0}^{3} 5_{0}^{4} 7_{0}^{2} 3_{\pm4}^{4} 4_{\pm3}^{2} 6_{\pm3}^{3} 8_{\pm3},
1_{\pm1} 2_{0}^{2} 5_{0}^{3} 7_{0} 3_{\pm4}^{3} 4_{\pm3} 6_{\pm3}^{2} 8_{\pm3},
1_{\pm1} 2_{0} 5_{0}^{2} 7_{0} 3_{\pm4} 4_{\pm3} 6_{\pm3}^{2} 8_{\pm3}, \\
& 1_{\pm1} 2_{0} 5_{0}^{3} 7_{0}^{2} 3_{\pm4}^{2} 4_{\pm3} 6_{\pm3}^{3} 8_{\pm3},
1_{\pm1} 2_{0} 5_{0} 7_{0} 3_{\pm4}^{2} 6_{\pm3} 8_{\pm3},
1_{\pm1}^{2} 2_{0}^{2} 5_{0}^{4} 7_{0}^{2} 3_{\pm4}^{3} 4_{\pm3}^{2} 6_{\pm3}^{3} 8_{\pm3},
1_{\pm1}^{2} 2_{0}^{3} 5_{0}^{5} 7_{0}^{2} 3_{\pm4}^{4} 4_{\pm3}^{2} 6_{\pm3}^{4} 8_{\pm3}, \\
& 1_{\pm1}^{2} 2_{0}^{2} 5_{0}^{3} 7_{0}^{2} 3_{\pm4}^{3} 4_{\pm3} 6_{\pm3}^{3} 8_{\pm3},
1_{\pm1} 2_{0}^{2} 5_{0}^{3} 7_{0}^{2} 3_{\pm4}^{2} 4_{\pm3}^{2} 6_{\pm3}^{2} 8_{\pm3},
1_{\pm1} 2_{0} 5_{0}^{3} 7_{0} 3_{\pm4}^{2} 4_{\pm3} 6_{\pm3}^{2} 8_{\pm3},
1_{\pm1} 2_{0}^{3} 5_{0}^{4} 7_{0}^{2} 3_{\pm4}^{3} 4_{\pm3}^{2} 6_{\pm3}^{3} 8_{\pm3}, \\
& 1_{\pm1}^{2} 2_{0}^{3} 5_{0}^{5} 7_{0}^{3} 3_{\pm4}^{4} 4_{\pm3}^{2} 6_{\pm3}^{4} 8_{\pm3}^{2},
1_{\pm1} 2_{0}^{2} 5_{0}^{4} 7_{0}^{2} 3_{\pm4}^{2} 4_{\pm3}^{2} 6_{\pm3}^{3} 8_{\pm3},
1_{\pm1} 2_{0}^{2} 5_{0}^{2} 7_{0}^{2} 3_{\pm4}^{2} 4_{\pm3} 6_{\pm3}^{2} 8_{\pm3},
1_{\pm1}^{2} 2_{0}^{2} 5_{0}^{4} 7_{0}^{2} 3_{\pm4}^{4} 4_{\pm3} 6_{\pm3}^{3} 8_{\pm3}, \\
& 1_{\pm1}^{2} 2_{0}^{3} 5_{0}^{5} 7_{0}^{3} 3_{\pm4}^{4} 4_{\pm3}^{2} 6_{\pm3}^{4} 8_{\pm3},
1_{\pm1} 2_{0}^{2} 5_{0}^{3} 7_{0}^{2} 3_{\pm4}^{3} 4_{\pm3} 6_{\pm3}^{2} 8_{\pm3},
1_{\pm1}^{2} 2_{0}^{2} 5_{0}^{3} 7_{0} 3_{\pm4}^{3} 4_{\pm3} 6_{\pm3}^{2} 8_{\pm3},
1_{\pm1}^{2} 2_{0}^{3} 5_{0}^{4} 7_{0}^{2} 3_{\pm4}^{3} 4_{\pm3}^{2} 6_{\pm3}^{3} 8_{\pm3}, \\
& 1_{\pm1}^{2} 2_{0}^{3} 5_{0}^{5} 7_{0}^{2} 3_{\pm4}^{4} 4_{\pm3}^{2} 6_{\pm3}^{3} 8_{\pm3},
1_{\pm1} 2_{0} 5_{0}^{2} 7_{0} 3_{\pm4}^{2} 6_{\pm3}^{2} 8_{\pm3},
1_{\pm1} 2_{0}^{2} 5_{0}^{3} 7_{0}^{2} 3_{\pm4}^{2} 4_{\pm3} 6_{\pm3}^{3} 8_{\pm3},
1_{\pm1} 2_{0}^{2} 5_{0}^{2} 7_{0} 3_{\pm4}^{2} 4_{\pm3} 6_{\pm3} 8_{\pm3}, \\
& 1_{\pm1} 2_{0}^{2} 5_{0}^{4} 7_{0}^{2} 3_{\pm4}^{3} 4_{\pm3} 6_{\pm3}^{3} 8_{\pm3},
1_{\pm1}^{2} 2_{0}^{3} 5_{0}^{4} 7_{0}^{2} 3_{\pm4}^{4} 4_{\pm3} 6_{\pm3}^{3} 8_{\pm3},
1_{\pm1} 2_{0}^{2} 5_{0}^{3} 7_{0} 3_{\pm4}^{2} 4_{\pm3} 6_{\pm3}^{2} 8_{\pm3},
1_{\pm1} 2_{0} 5_{0}^{2} 7_{0}^{2} 3_{\pm4} 4_{\pm3} 6_{\pm3}^{2} 8_{\pm3}, \\
& 1_{\pm1} 2_{0} 5_{0}^{3} 7_{0}^{2} 3_{\pm4}^{2} 4_{\pm3} 6_{\pm3}^{2} 8_{\pm3},
1_{\pm1}^{2} 2_{0}^{2} 5_{0}^{3} 7_{0}^{2} 3_{\pm4}^{3} 4_{\pm3} 6_{\pm3}^{2} 8_{\pm3},
1_{\pm1}^{2} 2_{0}^{2} 5_{0}^{4} 7_{0}^{2} 3_{\pm4}^{3} 4_{\pm3} 6_{\pm3}^{3} 8_{\pm3},
1_{\pm1} 2_{0} 5_{0}^{2} 7_{0}^{2} 3_{\pm4}^{2} 6_{\pm3}^{2} 8_{\pm3}, \\
& 1_{\pm1} 2_{0} 5_{0}^{2} 7_{0} 3_{\pm4} 4_{\pm3} 6_{\pm3} 8_{\pm3},
1_{\pm1} 2_{0}^{2} 5_{0}^{3} 7_{0}^{2} 3_{\pm4}^{2} 4_{\pm3} 6_{\pm3}^{2} 8_{\pm3},
1_{\pm1} 2_{0} 5_{0}^{2} 7_{0} 3_{\pm4}^{2} 6_{\pm3} 8_{\pm3},
1_{\pm1} 2_{0} 5_{0} 7_{0} 3_{\pm4} 6_{\pm3} 8_{\pm3}. \\
%%\end{align*}
%%\end{gather}
%%
%%
& \quad \\
%%
%%
%Let $\xi=(-2,-1,-3,-2,-1,-2,-1,0)$. 
%\begin{align*}
%\xymatrix{
%& & 2 \ar[d]  &  &  \\
%1  \ar[r] & 3   & 4 \ar[l]   & 5  \ar[l]  \ar[r]  &  6  &  7 \ar[l]  & 8  \ar[l] } 
%\end{align*}
%The highest $l$-weight monomials of Hernandez-Leclerc modules of type $E_8$ that are not of type $A$, $D$ or $E_7$ are 
%%\begin{gather}
%%\begin{align*}
& (59)  1_{\pm2} 2_{\pm1} 5_{\pm1}^{2} 8_{0} 3_{\pm5}^{2} 4_{\pm4} 6_{\pm4}^{2},
1_{\pm2} 2_{\pm1} 5_{\pm1} 8_{0} 3_{\pm5} 4_{\pm4} 6_{\pm4},
1_{\pm2} 2_{\pm1}^{2} 5_{\pm1}^{3} 7_{\pm1} 8_{0} 3_{\pm5}^{2} 4_{\pm4}^{2} 6_{\pm4}^{3},
1_{\pm2} 2_{\pm1} 5_{\pm1}^{2} 7_{\pm1} 8_{0} 3_{\pm5}^{2} 4_{\pm4} 6_{\pm4}^{2}, \\
& 1_{\pm2} 2_{\pm1}^{2} 5_{\pm1}^{3} 8_{0} 3_{\pm5}^{2} 4_{\pm4}^{2} 6_{\pm4}^{2},
1_{\pm2} 2_{\pm1}^{2} 5_{\pm1}^{4} 7_{\pm1} 8_{0} 3_{\pm5}^{3} 4_{\pm4}^{2} 6_{\pm4}^{3},
1_{\pm2} 2_{\pm1}^{2} 5_{\pm1}^{2} 8_{0} 3_{\pm5}^{2} 4_{\pm4} 6_{\pm4}^{2},
1_{\pm2} 2_{\pm1} 5_{\pm1}^{2} 8_{0} 3_{\pm5}^{2} 4_{\pm4} 6_{\pm4}, \\
& 1_{\pm2} 2_{\pm1}^{2} 5_{\pm1}^{3} 7_{\pm1} 8_{0} 3_{\pm5}^{3} 4_{\pm4} 6_{\pm4}^{3},
1_{\pm2}^{2} 2_{\pm1}^{3} 5_{\pm1}^{4} 7_{\pm1} 8_{0} 3_{\pm5}^{4} 4_{\pm4}^{2} 6_{\pm4}^{3},
1_{\pm2} 2_{\pm1}^{2} 5_{\pm1}^{3} 8_{0} 3_{\pm5}^{3} 4_{\pm4} 6_{\pm4}^{2},
1_{\pm2} 2_{\pm1} 5_{\pm1}^{2} 8_{0} 3_{\pm5} 4_{\pm4} 6_{\pm4}^{2}, \\
& 1_{\pm2} 2_{\pm1} 5_{\pm1}^{3} 7_{\pm1} 8_{0} 3_{\pm5}^{2} 4_{\pm4} 6_{\pm4}^{3},
1_{\pm2} 2_{\pm1} 5_{\pm1} 8_{0} 3_{\pm5}^{2} 6_{\pm4},
1_{\pm2}^{2} 2_{\pm1}^{2} 5_{\pm1}^{4} 7_{\pm1} 8_{0} 3_{\pm5}^{3} 4_{\pm4}^{2} 6_{\pm4}^{3},
1_{\pm2}^{2} 2_{\pm1}^{3} 5_{\pm1}^{5} 7_{\pm1} 8_{0} 3_{\pm5}^{4} 4_{\pm4}^{2} 6_{\pm4}^{4}, \\
& 1_{\pm2}^{2} 2_{\pm1}^{2} 5_{\pm1}^{3} 7_{\pm1} 8_{0} 3_{\pm5}^{3} 4_{\pm4} 6_{\pm4}^{3},
1_{\pm2} 2_{\pm1}^{2} 5_{\pm1}^{3} 7_{\pm1} 8_{0} 3_{\pm5}^{2} 4_{\pm4}^{2} 6_{\pm4}^{2},
1_{\pm2} 2_{\pm1} 5_{\pm1}^{3} 8_{0} 3_{\pm5}^{2} 4_{\pm4} 6_{\pm4}^{2},
1_{\pm2} 2_{\pm1}^{3} 5_{\pm1}^{4} 7_{\pm1} 8_{0} 3_{\pm5}^{3} 4_{\pm4}^{2} 6_{\pm4}^{3}, \\
& 1_{\pm2}^{2} 2_{\pm1}^{3} 5_{\pm1}^{5} 7_{\pm1} 8_{0}^{2} 3_{\pm5}^{4} 4_{\pm4}^{2} 6_{\pm4}^{4},
1_{\pm2} 2_{\pm1}^{2} 5_{\pm1}^{4} 7_{\pm1} 8_{0} 3_{\pm5}^{2} 4_{\pm4}^{2} 6_{\pm4}^{3},
1_{\pm2} 2_{\pm1}^{2} 5_{\pm1}^{2} 7_{\pm1} 8_{0} 3_{\pm5}^{2} 4_{\pm4} 6_{\pm4}^{2},
1_{\pm2}^{2} 2_{\pm1}^{2} 5_{\pm1}^{4} 7_{\pm1} 8_{0} 3_{\pm5}^{4} 4_{\pm4} 6_{\pm4}^{3}, \\
& 1_{\pm2}^{2} 2_{\pm1}^{3} 5_{\pm1}^{5} 7_{\pm1}^{2} 8_{0} 3_{\pm5}^{4} 4_{\pm4}^{2} 6_{\pm4}^{4},
1_{\pm2} 2_{\pm1}^{2} 5_{\pm1}^{3} 7_{\pm1} 8_{0} 3_{\pm5}^{3} 4_{\pm4} 6_{\pm4}^{2},
1_{\pm2}^{2} 2_{\pm1}^{2} 5_{\pm1}^{3} 8_{0} 3_{\pm5}^{3} 4_{\pm4} 6_{\pm4}^{2},
1_{\pm2}^{2} 2_{\pm1}^{3} 5_{\pm1}^{4} 7_{\pm1} 8_{0} 3_{\pm5}^{3} 4_{\pm4}^{2} 6_{\pm4}^{3}, \\
& 1_{\pm2}^{2} 2_{\pm1}^{3} 5_{\pm1}^{5} 7_{\pm1} 8_{0} 3_{\pm5}^{4} 4_{\pm4}^{2} 6_{\pm4}^{3},
1_{\pm2} 2_{\pm1} 5_{\pm1}^{2} 8_{0} 3_{\pm5}^{2} 6_{\pm4}^{2},
1_{\pm2} 2_{\pm1}^{2} 5_{\pm1}^{3} 7_{\pm1} 8_{0} 3_{\pm5}^{2} 4_{\pm4} 6_{\pm4}^{3},
1_{\pm2} 2_{\pm1}^{2} 5_{\pm1}^{2} 8_{0} 3_{\pm5}^{2} 4_{\pm4} 6_{\pm4}, \\
& 1_{\pm2} 2_{\pm1}^{2} 5_{\pm1}^{4} 7_{\pm1} 8_{0} 3_{\pm5}^{3} 4_{\pm4} 6_{\pm4}^{3},
1_{\pm2}^{2} 2_{\pm1}^{3} 5_{\pm1}^{4} 7_{\pm1} 8_{0} 3_{\pm5}^{4} 4_{\pm4} 6_{\pm4}^{3},
1_{\pm2} 2_{\pm1}^{2} 5_{\pm1}^{3} 8_{0} 3_{\pm5}^{2} 4_{\pm4} 6_{\pm4}^{2},
1_{\pm2} 2_{\pm1} 5_{\pm1}^{2} 7_{\pm1} 8_{0} 3_{\pm5} 4_{\pm4} 6_{\pm4}^{2}, \\
& 1_{\pm2} 2_{\pm1} 5_{\pm1}^{3} 7_{\pm1} 8_{0} 3_{\pm5}^{2} 4_{\pm4} 6_{\pm4}^{2},
1_{\pm2}^{2} 2_{\pm1}^{2} 5_{\pm1}^{3} 7_{\pm1} 8_{0} 3_{\pm5}^{3} 4_{\pm4} 6_{\pm4}^{2},
1_{\pm2}^{2} 2_{\pm1}^{2} 5_{\pm1}^{4} 7_{\pm1} 8_{0} 3_{\pm5}^{3} 4_{\pm4} 6_{\pm4}^{3},
1_{\pm2} 2_{\pm1} 5_{\pm1}^{2} 7_{\pm1} 8_{0} 3_{\pm5}^{2} 6_{\pm4}^{2}, \\
& 1_{\pm2} 2_{\pm1} 5_{\pm1}^{2} 8_{0} 3_{\pm5} 4_{\pm4} 6_{\pm4},
1_{\pm2} 2_{\pm1}^{2} 5_{\pm1}^{3} 7_{\pm1} 8_{0} 3_{\pm5}^{2} 4_{\pm4} 6_{\pm4}^{2},
1_{\pm2} 2_{\pm1} 5_{\pm1}^{2} 8_{0} 3_{\pm5}^{2} 6_{\pm4},
1_{\pm2} 2_{\pm1} 5_{\pm1} 8_{0} 3_{\pm5} 6_{\pm4}. \\
%%\end{align*}
%%\end{gather}
%%
%%
& \quad \\
%%
%%
%Let $\xi=(-2,-1,-3,-2,-1,0,-1,-2)$. 
%\begin{align*}
%\xymatrix{
%& & 2 \ar[d]  &  &  \\
%1  \ar[r] & 3   & 4 \ar[l]   & 5  \ar[l] &  6  \ar[l]   \ar[r]  &  7 \ar[r]  & 8 } 
%\end{align*}
%The highest $l$-weight monomials of Hernandez-Leclerc modules of type $E_8$ that are not of type $A$, $D$ or $E_7$ are 
%%\begin{gather}
%%\begin{align*}
& (60)  1_{\pm2} 2_{\pm1} 6_{0}^{2} 3_{\pm5}^{2} 4_{\pm4} 7_{\pm3} 8_{\pm4},
1_{\pm2} 2_{\pm1} 6_{0} 3_{\pm5} 4_{\pm4} 8_{\pm4},
1_{\pm2} 2_{\pm1}^{2} 6_{0}^{3} 3_{\pm5}^{2} 4_{\pm4}^{2} 7_{\pm3} 8_{\pm4},
1_{\pm2} 2_{\pm1} 6_{0}^{2} 3_{\pm5}^{2} 4_{\pm4} 8_{\pm4}, \\
& 1_{\pm2} 2_{\pm1}^{2} 5_{\pm1} 6_{0}^{2} 3_{\pm5}^{2} 4_{\pm4}^{2} 7_{\pm3} 8_{\pm4},
1_{\pm2} 2_{\pm1}^{2} 5_{\pm1} 6_{0}^{3} 3_{\pm5}^{3} 4_{\pm4}^{2} 7_{\pm3} 8_{\pm4},
1_{\pm2} 2_{\pm1} 5_{\pm1} 6_{0} 3_{\pm5}^{2} 4_{\pm4} 8_{\pm4},
1_{\pm2} 2_{\pm1}^{2} 6_{0}^{2} 3_{\pm5}^{2} 4_{\pm4} 7_{\pm3} 8_{\pm4}, \\
& 1_{\pm2} 2_{\pm1}^{2} 6_{0}^{3} 3_{\pm5}^{3} 4_{\pm4} 7_{\pm3} 8_{\pm4},
1_{\pm2}^{2} 2_{\pm1}^{3} 5_{\pm1} 6_{0}^{3} 3_{\pm5}^{4} 4_{\pm4}^{2} 7_{\pm3} 8_{\pm4},
1_{\pm2} 2_{\pm1}^{2} 5_{\pm1} 6_{0}^{2} 3_{\pm5}^{3} 4_{\pm4} 7_{\pm3} 8_{\pm4},
1_{\pm2} 2_{\pm1} 6_{0}^{2} 3_{\pm5} 4_{\pm4} 7_{\pm3} 8_{\pm4}, \\
& 1_{\pm2} 2_{\pm1} 6_{0} 3_{\pm5}^{2} 8_{\pm4},
1_{\pm2} 2_{\pm1} 6_{0}^{3} 3_{\pm5}^{2} 4_{\pm4} 7_{\pm3} 8_{\pm4},
1_{\pm2}^{2} 2_{\pm1}^{2} 5_{\pm1} 6_{0}^{3} 3_{\pm5}^{3} 4_{\pm4}^{2} 7_{\pm3} 8_{\pm4},
1_{\pm2}^{2} 2_{\pm1}^{3} 5_{\pm1} 6_{0}^{4} 3_{\pm5}^{4} 4_{\pm4}^{2} 7_{\pm3}^{2} 8_{\pm4}, \\
& 1_{\pm2}^{2} 2_{\pm1}^{2} 6_{0}^{3} 3_{\pm5}^{3} 4_{\pm4} 7_{\pm3} 8_{\pm4},
1_{\pm2} 2_{\pm1} 5_{\pm1} 6_{0}^{2} 3_{\pm5}^{2} 4_{\pm4} 7_{\pm3} 8_{\pm4},
1_{\pm2} 2_{\pm1}^{2} 5_{\pm1} 6_{0}^{2} 3_{\pm5}^{2} 4_{\pm4}^{2} 8_{\pm4},
1_{\pm2} 2_{\pm1}^{3} 5_{\pm1} 6_{0}^{3} 3_{\pm5}^{3} 4_{\pm4}^{2} 7_{\pm3} 8_{\pm4}, \\
& 1_{\pm2}^{2} 2_{\pm1}^{3} 5_{\pm1} 6_{0}^{4} 3_{\pm5}^{4} 4_{\pm4}^{2} 7_{\pm3} 8_{\pm4}^{2},
1_{\pm2} 2_{\pm1}^{2} 5_{\pm1} 6_{0}^{3} 3_{\pm5}^{2} 4_{\pm4}^{2} 7_{\pm3} 8_{\pm4},
1_{\pm2} 2_{\pm1}^{2} 6_{0}^{2} 3_{\pm5}^{2} 4_{\pm4} 8_{\pm4},
1_{\pm2}^{2} 2_{\pm1}^{2} 5_{\pm1} 6_{0}^{3} 3_{\pm5}^{4} 4_{\pm4} 7_{\pm3} 8_{\pm4}, \\
& 1_{\pm2}^{2} 2_{\pm1}^{3} 5_{\pm1} 6_{0}^{4} 3_{\pm5}^{4} 4_{\pm4}^{2} 7_{\pm3} 8_{\pm4},
1_{\pm2} 2_{\pm1}^{2} 5_{\pm1} 6_{0}^{2} 3_{\pm5}^{3} 4_{\pm4} 8_{\pm4},
1_{\pm2}^{2} 2_{\pm1}^{2} 5_{\pm1} 6_{0}^{2} 3_{\pm5}^{3} 4_{\pm4} 7_{\pm3} 8_{\pm4},
1_{\pm2}^{2} 2_{\pm1}^{3} 5_{\pm1} 6_{0}^{3} 3_{\pm5}^{3} 4_{\pm4}^{2} 7_{\pm3} 8_{\pm4}, \\
& 1_{\pm2}^{2} 2_{\pm1}^{3} 5_{\pm1}^{2} 6_{0}^{3} 3_{\pm5}^{4} 4_{\pm4}^{2} 7_{\pm3} 8_{\pm4},
1_{\pm2} 2_{\pm1} 6_{0}^{2} 3_{\pm5}^{2} 7_{\pm3} 8_{\pm4},
1_{\pm2} 2_{\pm1}^{2} 5_{\pm1} 6_{0} 3_{\pm5}^{2} 4_{\pm4} 8_{\pm4},
1_{\pm2} 2_{\pm1}^{2} 6_{0}^{3} 3_{\pm5}^{2} 4_{\pm4} 7_{\pm3} 8_{\pm4}, \\
& 1_{\pm2} 2_{\pm1}^{2} 5_{\pm1} 6_{0}^{3} 3_{\pm5}^{3} 4_{\pm4} 7_{\pm3} 8_{\pm4},
1_{\pm2}^{2} 2_{\pm1}^{3} 5_{\pm1} 6_{0}^{3} 3_{\pm5}^{4} 4_{\pm4} 7_{\pm3} 8_{\pm4},
1_{\pm2} 2_{\pm1}^{2} 5_{\pm1} 6_{0}^{2} 3_{\pm5}^{2} 4_{\pm4} 7_{\pm3} 8_{\pm4},
1_{\pm2} 2_{\pm1} 6_{0}^{2} 3_{\pm5} 4_{\pm4} 8_{\pm4}, \\
& 1_{\pm2} 2_{\pm1} 5_{\pm1} 6_{0}^{2} 3_{\pm5}^{2} 4_{\pm4} 8_{\pm4},
1_{\pm2}^{2} 2_{\pm1}^{2} 5_{\pm1} 6_{0}^{2} 3_{\pm5}^{3} 4_{\pm4} 8_{\pm4},
1_{\pm2}^{2} 2_{\pm1}^{2} 5_{\pm1} 6_{0}^{3} 3_{\pm5}^{3} 4_{\pm4} 7_{\pm3} 8_{\pm4},
1_{\pm2} 2_{\pm1} 6_{0}^{2} 3_{\pm5}^{2} 8_{\pm4}, \\
& 1_{\pm2} 2_{\pm1} 5_{\pm1} 6_{0} 3_{\pm5} 4_{\pm4} 8_{\pm4},
1_{\pm2} 2_{\pm1}^{2} 5_{\pm1} 6_{0}^{2} 3_{\pm5}^{2} 4_{\pm4} 8_{\pm4},
1_{\pm2} 2_{\pm1} 5_{\pm1} 6_{0} 3_{\pm5}^{2} 8_{\pm4},
1_{\pm2} 2_{\pm1} 6_{0} 3_{\pm5} 8_{\pm4}.
\end{align*}
\end{gather}

%Let $\xi=(-2,-1,-3,-2,-1,0,-1,0)$. 
%\begin{align*}
%\xymatrix{
%& & 2 \ar[d]  &  &  \\
%1  \ar[r] & 3   & 4 \ar[l]   & 5  \ar[l] &  6  \ar[l]   \ar[r]  &  7 & 8 \ar[l]  } 
%\end{align*}
%The highest $l$-weight monomials of Hernandez-Leclerc modules of type $E_8$ that are not of type $A$, $D$ or $E_7$ are 
\begin{gather}
\begin{align*}
& (61)  1_{\pm2} 2_{\pm1} 6_{0}^{2} 8_{0} 3_{\pm5}^{2} 4_{\pm4} 7_{\pm3}^{2},
1_{\pm2} 2_{\pm1} 6_{0} 8_{0} 3_{\pm5} 4_{\pm4} 7_{\pm3},
1_{\pm2} 2_{\pm1}^{2} 6_{0}^{3} 8_{0} 3_{\pm5}^{2} 4_{\pm4}^{2} 7_{\pm3}^{2},
1_{\pm2} 2_{\pm1} 6_{0}^{2} 8_{0} 3_{\pm5}^{2} 4_{\pm4} 7_{\pm3}, \\
& 1_{\pm2} 2_{\pm1}^{2} 5_{\pm1} 6_{0}^{2} 8_{0} 3_{\pm5}^{2} 4_{\pm4}^{2} 7_{\pm3}^{2},
1_{\pm2} 2_{\pm1}^{2} 5_{\pm1} 6_{0}^{3} 8_{0} 3_{\pm5}^{3} 4_{\pm4}^{2} 7_{\pm3}^{2},
1_{\pm2} 2_{\pm1} 5_{\pm1} 6_{0} 8_{0} 3_{\pm5}^{2} 4_{\pm4} 7_{\pm3},
1_{\pm2} 2_{\pm1}^{2} 6_{0}^{2} 8_{0} 3_{\pm5}^{2} 4_{\pm4} 7_{\pm3}^{2}, \\
& 1_{\pm2} 2_{\pm1}^{2} 6_{0}^{3} 8_{0} 3_{\pm5}^{3} 4_{\pm4} 7_{\pm3}^{2},
1_{\pm2}^{2} 2_{\pm1}^{3} 5_{\pm1} 6_{0}^{3} 8_{0} 3_{\pm5}^{4} 4_{\pm4}^{2} 7_{\pm3}^{2},
1_{\pm2} 2_{\pm1}^{2} 5_{\pm1} 6_{0}^{2} 8_{0} 3_{\pm5}^{3} 4_{\pm4} 7_{\pm3}^{2},
1_{\pm2} 2_{\pm1} 6_{0}^{2} 8_{0} 3_{\pm5} 4_{\pm4} 7_{\pm3}^{2}, \\
& 1_{\pm2} 2_{\pm1} 6_{0} 8_{0} 3_{\pm5}^{2} 7_{\pm3},
1_{\pm2} 2_{\pm1} 6_{0}^{3} 8_{0} 3_{\pm5}^{2} 4_{\pm4} 7_{\pm3}^{2},
1_{\pm2}^{2} 2_{\pm1}^{2} 5_{\pm1} 6_{0}^{3} 8_{0} 3_{\pm5}^{3} 4_{\pm4}^{2} 7_{\pm3}^{2},
1_{\pm2}^{2} 2_{\pm1}^{3} 5_{\pm1} 6_{0}^{4} 8_{0} 3_{\pm5}^{4} 4_{\pm4}^{2} 7_{\pm3}^{3}, \\
& 1_{\pm2}^{2} 2_{\pm1}^{2} 6_{0}^{3} 8_{0} 3_{\pm5}^{3} 4_{\pm4} 7_{\pm3}^{2},
1_{\pm2} 2_{\pm1} 5_{\pm1} 6_{0}^{2} 8_{0} 3_{\pm5}^{2} 4_{\pm4} 7_{\pm3}^{2},
1_{\pm2} 2_{\pm1}^{2} 5_{\pm1} 6_{0}^{2} 8_{0} 3_{\pm5}^{2} 4_{\pm4}^{2} 7_{\pm3},
1_{\pm2} 2_{\pm1}^{3} 5_{\pm1} 6_{0}^{3} 8_{0} 3_{\pm5}^{3} 4_{\pm4}^{2} 7_{\pm3}^{2}, \\
& 1_{\pm2}^{2} 2_{\pm1}^{3} 5_{\pm1} 6_{0}^{4} 8_{0}^{2} 3_{\pm5}^{4} 4_{\pm4}^{2} 7_{\pm3}^{3},
1_{\pm2} 2_{\pm1}^{2} 5_{\pm1} 6_{0}^{3} 8_{0} 3_{\pm5}^{2} 4_{\pm4}^{2} 7_{\pm3}^{2},
1_{\pm2} 2_{\pm1}^{2} 6_{0}^{2} 8_{0} 3_{\pm5}^{2} 4_{\pm4} 7_{\pm3},
1_{\pm2}^{2} 2_{\pm1}^{2} 5_{\pm1} 6_{0}^{3} 8_{0} 3_{\pm5}^{4} 4_{\pm4} 7_{\pm3}^{2}, \\
& 1_{\pm2}^{2} 2_{\pm1}^{3} 5_{\pm1} 6_{0}^{4} 8_{0} 3_{\pm5}^{4} 4_{\pm4}^{2} 7_{\pm3}^{2},
1_{\pm2} 2_{\pm1}^{2} 5_{\pm1} 6_{0}^{2} 8_{0} 3_{\pm5}^{3} 4_{\pm4} 7_{\pm3},
1_{\pm2}^{2} 2_{\pm1}^{2} 5_{\pm1} 6_{0}^{2} 8_{0} 3_{\pm5}^{3} 4_{\pm4} 7_{\pm3}^{2},
1_{\pm2}^{2} 2_{\pm1}^{3} 5_{\pm1} 6_{0}^{3} 8_{0} 3_{\pm5}^{3} 4_{\pm4}^{2} 7_{\pm3}^{2}, \\
& 1_{\pm2}^{2} 2_{\pm1}^{3} 5_{\pm1}^{2} 6_{0}^{3} 8_{0} 3_{\pm5}^{4} 4_{\pm4}^{2} 7_{\pm3}^{2},
1_{\pm2} 2_{\pm1} 6_{0}^{2} 8_{0} 3_{\pm5}^{2} 7_{\pm3}^{2},
1_{\pm2} 2_{\pm1}^{2} 5_{\pm1} 6_{0} 8_{0} 3_{\pm5}^{2} 4_{\pm4} 7_{\pm3},
1_{\pm2} 2_{\pm1}^{2} 6_{0}^{3} 8_{0} 3_{\pm5}^{2} 4_{\pm4} 7_{\pm3}^{2}, \\
& 1_{\pm2} 2_{\pm1}^{2} 5_{\pm1} 6_{0}^{3} 8_{0} 3_{\pm5}^{3} 4_{\pm4} 7_{\pm3}^{2},
1_{\pm2}^{2} 2_{\pm1}^{3} 5_{\pm1} 6_{0}^{3} 8_{0} 3_{\pm5}^{4} 4_{\pm4} 7_{\pm3}^{2},
1_{\pm2} 2_{\pm1}^{2} 5_{\pm1} 6_{0}^{2} 8_{0} 3_{\pm5}^{2} 4_{\pm4} 7_{\pm3}^{2},
1_{\pm2} 2_{\pm1} 6_{0}^{2} 8_{0} 3_{\pm5} 4_{\pm4} 7_{\pm3}, \\
& 1_{\pm2} 2_{\pm1} 5_{\pm1} 6_{0}^{2} 8_{0} 3_{\pm5}^{2} 4_{\pm4} 7_{\pm3},
1_{\pm2}^{2} 2_{\pm1}^{2} 5_{\pm1} 6_{0}^{2} 8_{0} 3_{\pm5}^{3} 4_{\pm4} 7_{\pm3},
1_{\pm2}^{2} 2_{\pm1}^{2} 5_{\pm1} 6_{0}^{3} 8_{0} 3_{\pm5}^{3} 4_{\pm4} 7_{\pm3}^{2},
1_{\pm2} 2_{\pm1} 6_{0}^{2} 8_{0} 3_{\pm5}^{2} 7_{\pm3}, \\
& 1_{\pm2} 2_{\pm1} 5_{\pm1} 6_{0} 8_{0} 3_{\pm5} 4_{\pm4} 7_{\pm3},
1_{\pm2} 2_{\pm1}^{2} 5_{\pm1} 6_{0}^{2} 8_{0} 3_{\pm5}^{2} 4_{\pm4} 7_{\pm3},
1_{\pm2} 2_{\pm1} 5_{\pm1} 6_{0} 8_{0} 3_{\pm5}^{2} 7_{\pm3},
1_{\pm2} 2_{\pm1} 6_{0} 8_{0} 3_{\pm5} 7_{\pm3}. \\
%%\end{align*}
%%\end{gather}
%%
%%
& \quad \\
%%
%%
%Let $\xi=(-3,-2,-4_{\pm3},-2,-1,0,-1)$. 
%\begin{align*}
%\xymatrix{
%& & 2 \ar[d]  &  &  \\
%1  \ar[r] & 3   & 4 \ar[l]   & 5  \ar[l] &  6  \ar[l]  &  7  \ar[l]   \ar[r]  & 8 } 
%\end{align*}
%The highest $l$-weight monomials of Hernandez-Leclerc modules of type $E_8$ that are not of type $A$, $D$ or $E_7$ are 
%%\begin{gather}
%%\begin{align*}
& (62)  1_{\pm3} 2_{\pm2} 7_{0} 3_{\pm6} 4_{\pm5} 8_{\pm3},
1_{\pm3} 2_{\pm2} 7_{0}^{2} 3_{\pm6}^{2} 4_{\pm5} 8_{\pm3},
1_{\pm3} 2_{\pm2} 6_{\pm1} 7_{0} 3_{\pm6}^{2} 4_{\pm5} 8_{\pm3},
1_{\pm3} 2_{\pm2}^{2} 6_{\pm1} 7_{0}^{2} 3_{\pm6}^{2} 4_{\pm5}^{2} 8_{\pm3}, \\
& 1_{\pm3} 2_{\pm2} 5_{\pm2} 7_{0} 3_{\pm6}^{2} 4_{\pm5} 8_{\pm3},
1_{\pm3} 2_{\pm2}^{2} 5_{\pm2} 7_{0}^{2} 3_{\pm6}^{2} 4_{\pm5}^{2} 8_{\pm3},
1_{\pm3} 2_{\pm2}^{2} 5_{\pm2} 6_{\pm1} 7_{0}^{2} 3_{\pm6}^{3} 4_{\pm5}^{2} 8_{\pm3},
1_{\pm3} 2_{\pm2}^{2} 5_{\pm2} 6_{\pm1} 7_{0} 3_{\pm6}^{2} 4_{\pm5}^{2} 8_{\pm3}, \\
& 1_{\pm3} 2_{\pm2} 7_{0} 3_{\pm6}^{2} 8_{\pm3},
1_{\pm3} 2_{\pm2}^{2} 7_{0}^{2} 3_{\pm6}^{2} 4_{\pm5} 8_{\pm3},
1_{\pm3} 2_{\pm2}^{2} 6_{\pm1} 7_{0}^{2} 3_{\pm6}^{3} 4_{\pm5} 8_{\pm3},
1_{\pm3}^{2} 2_{\pm2}^{3} 5_{\pm2} 6_{\pm1} 7_{0}^{2} 3_{\pm6}^{4} 4_{\pm5}^{2} 8_{\pm3}, \\
& 1_{\pm3} 2_{\pm2}^{2} 5_{\pm2} 7_{0}^{2} 3_{\pm6}^{3} 4_{\pm5} 8_{\pm3},
1_{\pm3} 2_{\pm2}^{2} 6_{\pm1} 7_{0} 3_{\pm6}^{2} 4_{\pm5} 8_{\pm3},
1_{\pm3} 2_{\pm2} 7_{0}^{2} 3_{\pm6} 4_{\pm5} 8_{\pm3},
1_{\pm3} 2_{\pm2} 6_{\pm1} 7_{0}^{2} 3_{\pm6}^{2} 4_{\pm5} 8_{\pm3}, \\
& 1_{\pm3}^{2} 2_{\pm2}^{2} 5_{\pm2} 6_{\pm1} 7_{0}^{2} 3_{\pm6}^{3} 4_{\pm5}^{2} 8_{\pm3},
1_{\pm3}^{2} 2_{\pm2}^{3} 5_{\pm2} 6_{\pm1} 7_{0}^{3} 3_{\pm6}^{4} 4_{\pm5}^{2} 8_{\pm3}^{2},
1_{\pm3}^{2} 2_{\pm2}^{2} 6_{\pm1} 7_{0}^{2} 3_{\pm6}^{3} 4_{\pm5} 8_{\pm3},
1_{\pm3} 2_{\pm2} 6_{\pm1} 7_{0} 3_{\pm6} 4_{\pm5} 8_{\pm3}, \\
& 1_{\pm3} 2_{\pm2} 5_{\pm2} 7_{0}^{2} 3_{\pm6}^{2} 4_{\pm5} 8_{\pm3},
1_{\pm3} 2_{\pm2}^{3} 5_{\pm2} 6_{\pm1} 7_{0}^{2} 3_{\pm6}^{3} 4_{\pm5}^{2} 8_{\pm3},
1_{\pm3}^{2} 2_{\pm2}^{3} 5_{\pm2} 6_{\pm1} 7_{0}^{3} 3_{\pm6}^{4} 4_{\pm5}^{2} 8_{\pm3},
1_{\pm3} 2_{\pm2}^{2} 5_{\pm2} 6_{\pm1} 7_{0}^{2} 3_{\pm6}^{2} 4_{\pm5}^{2} 8_{\pm3}, \\
& 1_{\pm3} 2_{\pm2}^{2} 5_{\pm2} 6_{\pm1} 7_{0} 3_{\pm6}^{3} 4_{\pm5} 8_{\pm3},
1_{\pm3}^{2} 2_{\pm2}^{2} 5_{\pm2} 6_{\pm1} 7_{0}^{2} 3_{\pm6}^{4} 4_{\pm5} 8_{\pm3},
1_{\pm3}^{2} 2_{\pm2}^{3} 5_{\pm2} 6_{\pm1}^{2} 7_{0}^{2} 3_{\pm6}^{4} 4_{\pm5}^{2} 8_{\pm3},
1_{\pm3} 2_{\pm2} 5_{\pm2} 6_{\pm1} 7_{0} 3_{\pm6}^{2} 4_{\pm5} 8_{\pm3}, \\
& 1_{\pm3} 2_{\pm2}^{2} 5_{\pm2} 7_{0} 3_{\pm6}^{2} 4_{\pm5} 8_{\pm3},
1_{\pm3}^{2} 2_{\pm2}^{2} 5_{\pm2} 7_{0}^{2} 3_{\pm6}^{3} 4_{\pm5} 8_{\pm3},
1_{\pm3}^{2} 2_{\pm2}^{3} 5_{\pm2} 6_{\pm1} 7_{0}^{2} 3_{\pm6}^{3} 4_{\pm5}^{2} 8_{\pm3},
1_{\pm3}^{2} 2_{\pm2}^{3} 5_{\pm2}^{2} 6_{\pm1} 7_{0}^{2} 3_{\pm6}^{4} 4_{\pm5}^{2} 8_{\pm3}, \\
& 1_{\pm3}^{2} 2_{\pm2}^{2} 5_{\pm2} 6_{\pm1} 7_{0} 3_{\pm6}^{3} 4_{\pm5} 8_{\pm3},
1_{\pm3} 2_{\pm2} 5_{\pm2} 7_{0} 3_{\pm6} 4_{\pm5} 8_{\pm3},
1_{\pm3} 2_{\pm2} 7_{0}^{2} 3_{\pm6}^{2} 8_{\pm3},
1_{\pm3} 2_{\pm2}^{2} 6_{\pm1} 7_{0}^{2} 3_{\pm6}^{2} 4_{\pm5} 8_{\pm3}, \\
& 1_{\pm3} 2_{\pm2}^{2} 5_{\pm2} 6_{\pm1} 7_{0}^{2} 3_{\pm6}^{3} 4_{\pm5} 8_{\pm3},
1_{\pm3}^{2} 2_{\pm2}^{3} 5_{\pm2} 6_{\pm1} 7_{0}^{2} 3_{\pm6}^{4} 4_{\pm5} 8_{\pm3},
1_{\pm3} 2_{\pm2}^{2} 5_{\pm2} 7_{0}^{2} 3_{\pm6}^{2} 4_{\pm5} 8_{\pm3},
1_{\pm3} 2_{\pm2} 6_{\pm1} 7_{0} 3_{\pm6}^{2} 8_{\pm3}, \\
& 1_{\pm3}^{2} 2_{\pm2}^{2} 5_{\pm2} 6_{\pm1} 7_{0}^{2} 3_{\pm6}^{3} 4_{\pm5} 8_{\pm3},
1_{\pm3} 2_{\pm2}^{2} 5_{\pm2} 6_{\pm1} 7_{0} 3_{\pm6}^{2} 4_{\pm5} 8_{\pm3},
1_{\pm3} 2_{\pm2} 5_{\pm2} 7_{0} 3_{\pm6}^{2} 8_{\pm3},
1_{\pm3} 2_{\pm2} 7_{0} 3_{\pm6} 8_{\pm3}. \\
%%\end{align*}
%%\end{gather}
%%
%%
& \quad \\
%%
%%
%Let $\xi=(-4, -3,-5,-4,-3,-2,-1,0)$. 
%\begin{align*}
%\xymatrix{
%& & 2 \ar[d]  &  &  \\
%1  \ar[r] & 3   & 4 \ar[l]   & 5  \ar[l] &  6  \ar[l]  &  7  \ar[l]   & 8  \ar[l] } 
%\end{align*}
%The highest $l$-weight monomials of Hernandez-Leclerc modules of type $E_8$ that are not of type $A$, $D$ or $E_7$ are 
%%\begin{gather}
%%\begin{align*}
& (63)  1_{\pm4} 2_{\pm3} 8_{0} 3_{\pm7}, 
1_{\pm4} 2_{\pm3} 8_{0} 3_{\pm7} 4_{\pm6},
1_{\pm4} 2_{\pm3} 7_{\pm1} 8_{0} 3_{\pm7}^{2} 4_{\pm6},
1_{\pm4} 2_{\pm3} 6_{\pm2} 8_{0} 3_{\pm7}^{2} 4_{\pm6},
1_{\pm4} 2_{\pm3}^{2} 6_{\pm2} 7_{\pm1} 8_{0} 3_{\pm7}^{2} 4_{\pm6}^{2}, \\
& 1_{\pm4} 2_{\pm3} 5_{\pm3} 8_{0} 3_{\pm7}^{2} 4_{\pm6},
1_{\pm4} 2_{\pm3}^{2} 5_{\pm3} 7_{\pm1} 8_{0} 3_{\pm7}^{2} 4_{\pm6}^{2},
1_{\pm4} 2_{\pm3}^{2} 5_{\pm3} 6_{\pm2} 7_{\pm1} 8_{0} 3_{\pm7}^{3} 4_{\pm6}^{2},
1_{\pm4} 2_{\pm3}^{2} 5_{\pm3} 6_{\pm2} 8_{0} 3_{\pm7}^{2} 4_{\pm6}^{2}, \\
& 1_{\pm4} 2_{\pm3} 8_{0} 3_{\pm7}^{2},
1_{\pm4} 2_{\pm3}^{2} 7_{\pm1} 8_{0} 3_{\pm7}^{2} 4_{\pm6},
1_{\pm4} 2_{\pm3}^{2} 6_{\pm2} 7_{\pm1} 8_{0} 3_{\pm7}^{3} 4_{\pm6},
1_{\pm4}^{2} 2_{\pm3}^{3} 5_{\pm3} 6_{\pm2} 7_{\pm1} 8_{0} 3_{\pm7}^{4} 4_{\pm6}^{2}, \\
& 1_{\pm4} 2_{\pm3}^{2} 5_{\pm3} 7_{\pm1} 8_{0} 3_{\pm7}^{3} 4_{\pm6},
1_{\pm4} 2_{\pm3}^{2} 6_{\pm2} 8_{0} 3_{\pm7}^{2} 4_{\pm6},
1_{\pm4} 2_{\pm3} 7_{\pm1} 8_{0} 3_{\pm7} 4_{\pm6},
1_{\pm4} 2_{\pm3} 6_{\pm2} 7_{\pm1} 8_{0} 3_{\pm7}^{2} 4_{\pm6}, \\
& 1_{\pm4}^{2} 2_{\pm3}^{2} 5_{\pm3} 6_{\pm2} 7_{\pm1} 8_{0} 3_{\pm7}^{3} 4_{\pm6}^{2},
1_{\pm4}^{2} 2_{\pm3}^{3} 5_{\pm3} 6_{\pm2} 7_{\pm1} 8_{0}^{2} 3_{\pm7}^{4} 4_{\pm6}^{2},
1_{\pm4}^{2} 2_{\pm3}^{2} 6_{\pm2} 7_{\pm1} 8_{0} 3_{\pm7}^{3} 4_{\pm6},
1_{\pm4} 2_{\pm3} 6_{\pm2} 8_{0} 3_{\pm7} 4_{\pm6}, \\
& 1_{\pm4} 2_{\pm3} 5_{\pm3} 7_{\pm1} 8_{0} 3_{\pm7}^{2} 4_{\pm6},
1_{\pm4} 2_{\pm3}^{3} 5_{\pm3} 6_{\pm2} 7_{\pm1} 8_{0} 3_{\pm7}^{3} 4_{\pm6}^{2},
1_{\pm4}^{2} 2_{\pm3}^{3} 5_{\pm3} 6_{\pm2} 7_{\pm1}^{2} 8_{0} 3_{\pm7}^{4} 4_{\pm6}^{2},
1_{\pm4} 2_{\pm3}^{2} 5_{\pm3} 6_{\pm2} 7_{\pm1} 8_{0} 3_{\pm7}^{2} 4_{\pm6}^{2}, \\
& 1_{\pm4} 2_{\pm3}^{2} 5_{\pm3} 6_{\pm2} 8_{0} 3_{\pm7}^{3} 4_{\pm6},
1_{\pm4}^{2} 2_{\pm3}^{2} 5_{\pm3} 6_{\pm2} 7_{\pm1} 8_{0} 3_{\pm7}^{4} 4_{\pm6},
1_{\pm4}^{2} 2_{\pm3}^{3} 5_{\pm3} 6_{\pm2}^{2} 7_{\pm1} 8_{0} 3_{\pm7}^{4} 4_{\pm6}^{2},
1_{\pm4} 2_{\pm3} 5_{\pm3} 6_{\pm2} 8_{0} 3_{\pm7}^{2} 4_{\pm6}, \\
& 1_{\pm4} 2_{\pm3}^{2} 5_{\pm3} 8_{0} 3_{\pm7}^{2} 4_{\pm6},
1_{\pm4}^{2} 2_{\pm3}^{2} 5_{\pm3} 7_{\pm1} 8_{0} 3_{\pm7}^{3} 4_{\pm6},
1_{\pm4}^{2} 2_{\pm3}^{3} 5_{\pm3} 6_{\pm2} 7_{\pm1} 8_{0} 3_{\pm7}^{3} 4_{\pm6}^{2},
1_{\pm4}^{2} 2_{\pm3}^{3} 5_{\pm3}^{2} 6_{\pm2} 7_{\pm1} 8_{0} 3_{\pm7}^{4} 4_{\pm6}^{2}, \\
& 1_{\pm4}^{2} 2_{\pm3}^{2} 5_{\pm3} 6_{\pm2} 8_{0} 3_{\pm7}^{3} 4_{\pm6},
1_{\pm4} 2_{\pm3} 5_{\pm3} 8_{0} 3_{\pm7} 4_{\pm6},
1_{\pm4} 2_{\pm3} 7_{\pm1} 8_{0} 3_{\pm7}^{2},
1_{\pm4} 2_{\pm3}^{2} 6_{\pm2} 7_{\pm1} 8_{0} 3_{\pm7}^{2} 4_{\pm6}, \\
& 1_{\pm4} 2_{\pm3}^{2} 5_{\pm3} 6_{\pm2} 7_{\pm1} 8_{0} 3_{\pm7}^{3} 4_{\pm6},
1_{\pm4}^{2} 2_{\pm3}^{3} 5_{\pm3} 6_{\pm2} 7_{\pm1} 8_{0} 3_{\pm7}^{4} 4_{\pm6},
1_{\pm4} 2_{\pm3}^{2} 5_{\pm3} 7_{\pm1} 8_{0} 3_{\pm7}^{2} 4_{\pm6},
1_{\pm4} 2_{\pm3} 6_{\pm2} 8_{0} 3_{\pm7}^{2}, \\
& 1_{\pm4}^{2} 2_{\pm3}^{2} 5_{\pm3} 6_{\pm2} 7_{\pm1} 8_{0} 3_{\pm7}^{3} 4_{\pm6},
1_{\pm4} 2_{\pm3}^{2} 5_{\pm3} 6_{\pm2} 8_{0} 3_{\pm7}^{2} 4_{\pm6},
1_{\pm4} 2_{\pm3} 5_{\pm3} 8_{0} 3_{\pm7}^{2}.
\end{align*}
\end{gather}

%Let $\xi=(0,-1,-1,0,-1,0,-1,0)$. 
%\begin{align*}
%\xymatrix{
%& & 2 &  &  \\
%1  \ar[r] & 3   & 4 \ar[l]  \ar[u]  \ar[r] & 5 &  6   \ar[l]    \ar[r] &  7  & 8 \ar[l]  }
%\end{align*}
%The highest $l$-weight monomials of Hernandez-Leclerc modules of type $E_8$ that are not of type $A$, $D$ or $E_7$ are 
\begin{gather}
\begin{align*}
& (64)  1_{0} 4_{0} 6_{0} 8_{0} 2_{\pm3} 3_{\pm3} 5_{\pm3} 7_{\pm3}, 
1_{0} 4_{0}^{2} 6_{0} 8_{0} 2_{\pm3} 3_{\pm3}^{2} 5_{\pm3}^{2} 7_{\pm3},
1_{0} 4_{0}^{3} 6_{0}^{2} 8_{0} 2_{\pm3}^{2} 3_{\pm3}^{2} 5_{\pm3}^{3} 7_{\pm3}^{2},
1_{0} 4_{0}^{2} 6_{0}^{2} 8_{0} 2_{\pm3} 3_{\pm3}^{2} 5_{\pm3}^{2} 7_{\pm3}^{2}, \\
& 1_{0} 4_{0}^{4} 6_{0}^{3} 8_{0} 2_{\pm3}^{2} 3_{\pm3}^{3} 5_{\pm3}^{4} 7_{\pm3}^{2},
1_{0} 4_{0}^{3} 6_{0}^{2} 8_{0} 2_{\pm3}^{2} 3_{\pm3}^{2} 5_{\pm3}^{3} 7_{\pm3},
1_{0} 4_{0}^{4} 6_{0}^{2} 8_{0} 2_{\pm3}^{2} 3_{\pm3}^{3} 5_{\pm3}^{3} 7_{\pm3}^{2},
1_{0}^{2} 4_{0}^{5} 6_{0}^{3} 8_{0} 2_{\pm3}^{3} 3_{\pm3}^{4} 5_{\pm3}^{4} 7_{\pm3}^{2}, \\
& 1_{0} 4_{0}^{3} 6_{0} 8_{0} 2_{\pm3}^{2} 3_{\pm3}^{2} 5_{\pm3}^{2} 7_{\pm3},
1_{0} 4_{0}^{3} 6_{0}^{2} 8_{0} 2_{\pm3} 3_{\pm3}^{2} 5_{\pm3}^{3} 7_{\pm3}^{2},
1_{0}^{2} 4_{0}^{4} 6_{0}^{3} 8_{0} 2_{\pm3}^{2} 3_{\pm3}^{3} 5_{\pm3}^{4} 7_{\pm3}^{2},
1_{0}^{2} 4_{0}^{6}  6_{0}^{3} 8_{0} 2_{\pm3}^{3} 3_{\pm3}^{4} 5_{\pm3}^{5} 7_{\pm3}^{2}, \\
& 1_{0}^{2} 4_{0}^{4} 6_{0}^{2} 8_{0} 2_{\pm3}^{2} 3_{\pm3}^{3} 5_{\pm3}^{3} 7_{\pm3}^{2},
1_{0} 4_{0}^{3} 6_{0}^{3} 8_{0} 2_{\pm3}^{2} 3_{\pm3}^{2} 5_{\pm3}^{3} 7_{\pm3}^{2},
1_{0} 4_{0}^{5} 6_{0}^{3} 8_{0} 2_{\pm3}^{3} 3_{\pm3}^{3} 5_{\pm3}^{4} 7_{\pm3}^{2},
1_{0} 4_{0}^{2} 6_{0}^{2} 8_{0} 2_{\pm3} 3_{\pm3}^{2} 5_{\pm3}^{2} 7_{\pm3}, \\
& 1_{0} 4_{0}^{2} 6_{0} 8_{0} 2_{\pm3} 3_{\pm3} 5_{\pm3}^{2} 7_{\pm3},
1_{0}^{2} 4_{0}^{6}  6_{0}^{4} 8_{0} 2_{\pm3}^{3} 3_{\pm3}^{4} 5_{\pm3}^{5} 7_{\pm3}^{3},
1_{0} 4_{0}^{4} 6_{0}^{3} 8_{0} 2_{\pm3}^{2} 3_{\pm3}^{2} 5_{\pm3}^{4} 7_{\pm3}^{2},
1_{0} 4_{0}^{3} 6_{0}^{2} 8_{0} 2_{\pm3}^{2} 3_{\pm3}^{2} 5_{\pm3}^{2} 7_{\pm3}^{2}, \\
& 1_{0} 4_{0}^{4} 6_{0}^{2} 8_{0} 2_{\pm3}^{2} 3_{\pm3}^{3} 5_{\pm3}^{3} 7_{\pm3},
1_{0}^{2} 4_{0}^{5} 6_{0}^{3} 8_{0} 2_{\pm3}^{2} 3_{\pm3}^{4} 5_{\pm3}^{4} 7_{\pm3}^{2},
1_{0}^{2} 4_{0}^{6}  6_{0}^{4} 8_{0}^{2} 2_{\pm3}^{3} 3_{\pm3}^{4} 5_{\pm3}^{5} 7_{\pm3}^{3},
1_{0} 4_{0}^{4} 6_{0}^{3} 8_{0} 2_{\pm3}^{2} 3_{\pm3}^{3} 5_{\pm3}^{3} 7_{\pm3}^{2}, \\
& 1_{0} 4_{0}^{3} 6_{0}^{2} 8_{0} 2_{\pm3} 3_{\pm3}^{2} 5_{\pm3}^{3} 7_{\pm3},
1_{0}^{2} 4_{0}^{5} 6_{0}^{3} 8_{0} 2_{\pm3}^{3} 3_{\pm3}^{3} 5_{\pm3}^{4} 7_{\pm3}^{2},
1_{0} 4_{0}^{2} 6_{0} 8_{0} 2_{\pm3} 3_{\pm3}^{2} 5_{\pm3} 7_{\pm3},
1_{0}^{2} 4_{0}^{6}  6_{0}^{4} 8_{0} 2_{\pm3}^{3} 3_{\pm3}^{4} 5_{\pm3}^{5} 7_{\pm3}^{2}, \\
& 1_{0}^{2} 4_{0}^{4} 6_{0}^{2} 8_{0} 2_{\pm3}^{2} 3_{\pm3}^{3} 5_{\pm3}^{3} 7_{\pm3},
1_{0} 4_{0}^{4} 6_{0}^{2} 8_{0} 2_{\pm3}^{2} 3_{\pm3}^{2} 5_{\pm3}^{3} 7_{\pm3}^{2},
1_{0} 4_{0}^{5} 6_{0}^{3} 8_{0} 2_{\pm3}^{2} 3_{\pm3}^{3} 5_{\pm3}^{4} 7_{\pm3}^{2},
1_{0} 4_{0}^{2} 6_{0}^{2} 8_{0} 2_{\pm3} 3_{\pm3} 5_{\pm3}^{2} 7_{\pm3}^{2}, \\
& 1_{0}^{2} 4_{0}^{6}  6_{0}^{3} 8_{0} 2_{\pm3}^{3} 3_{\pm3}^{4} 5_{\pm3}^{4} 7_{\pm3}^{2},
1_{0} 4_{0}^{3} 6_{0} 8_{0} 2_{\pm3} 3_{\pm3}^{2} 5_{\pm3}^{2} 7_{\pm3},
1_{0} 4_{0}^{3} 6_{0}^{3} 8_{0} 2_{\pm3} 3_{\pm3}^{2} 5_{\pm3}^{3} 7_{\pm3}^{2},
1_{0}^{2} 4_{0}^{4} 6_{0}^{3} 8_{0} 2_{\pm3}^{2} 3_{\pm3}^{3} 5_{\pm3}^{3} 7_{\pm3}^{2}, \\
& 1_{0}^{2} 4_{0}^{5} 6_{0}^{3} 8_{0} 2_{\pm3}^{2} 3_{\pm3}^{3} 5_{\pm3}^{4} 7_{\pm3}^{2},
1_{0} 4_{0}^{3} 6_{0}^{2} 8_{0} 2_{\pm3} 3_{\pm3}^{2} 5_{\pm3}^{2} 7_{\pm3}^{2},
1_{0} 4_{0}^{3} 6_{0}^{2} 8_{0} 2_{\pm3}^{2} 3_{\pm3}^{2} 5_{\pm3}^{2} 7_{\pm3},
1_{0} 4_{0}^{4} 6_{0}^{2} 8_{0} 2_{\pm3}^{2} 3_{\pm3}^{2} 5_{\pm3}^{3} 7_{\pm3}, \\
& 1_{0} 4_{0}^{4} 6_{0}^{3} 8_{0} 2_{\pm3}^{2} 3_{\pm3}^{2} 5_{\pm3}^{3} 7_{\pm3}^{2},
1_{0} 4_{0}^{2} 6_{0}^{2} 8_{0} 2_{\pm3} 3_{\pm3} 5_{\pm3}^{2} 7_{\pm3},
1_{0} 4_{0}^{3} 6_{0}^{2} 8_{0} 2_{\pm3} 3_{\pm3}^{2} 5_{\pm3}^{2} 7_{\pm3},
1_{0} 4_{0}^{2} 6_{0} 8_{0} 2_{\pm3} 3_{\pm3} 5_{\pm3} 7_{\pm3}.
\end{align*}
\end{gather}
\end{appendix}

\end{document}